\documentclass[11pt,reqno]{amsart}

\usepackage{amsmath,amsthm,amsfonts}
\usepackage{latexsym,mathabx,txfonts}
\usepackage{amssymb}
\usepackage{slashed}
\usepackage{tikz}
\usepackage{empheq}
\usepackage[breaklinks, linktocpage]{hyperref}

\let\OLDthebibliography\thebibliography
\renewcommand\thebibliography[1]{
	\OLDthebibliography{#1}
	\setlength{\parskip}{0pt}
	\setlength{\itemsep}{4pt}
}

\newcommand{\R}{\mathbb{R}}

\newcommand{\C}{\mathbb{C}}
\newcommand{\N}{\mathbb{N}}
\newcommand{\Z}{\mathbb{Z}}

\renewcommand{\Re}{\mathop{\mathrm{Re}}}
\renewcommand{\Im}{\mathop{\mathrm{Im}}}

\renewcommand{\bar}{\overline}
\renewcommand{\hat}{\widehat}

\numberwithin{equation}{section}

\newtheorem{thm}{Theorem}[section]
\newtheorem{cor}[thm]{Corollary}
\newtheorem{lem}[thm]{Lemma}
\newtheorem{prop}[thm]{Proposition}

\theoremstyle{remark}
\newtheorem{rem}{Remark}[section]

\newcommand{\Del}[1]{}

\def\sign{\mathrm{sign}}

\def\f{\frac}

\newcommand{\boxalign}[2][.87\textwidth]{
	\par\noindent\tikzstyle{mybox} = [draw=black,inner sep=6pt]
	\begin{center}\begin{tikzpicture}
			\node [mybox] (box){%
				\begin{minipage}{#1}{\vspace{-2mm}#2}\end{minipage}
			};
	\end{tikzpicture}\end{center}
}

\begin{document}
	
	\title[Blow up for the Zakharov system]{Finite time blow up for the energy critical\\ Zakharov system II: exact solutions}
	
	\author[J. Krieger]{Joachim Krieger}
	\author[T. Schmid]{Tobias Schmid}
	\address{EPFL SB MATH PDE,
		B\^{a}timent MA,
		Station 8,
		CH-1015 Lausanne}
	\email{joachim.krieger@epfl.ch}
	\email{tobias.schmid@epfl.ch}

	
	\begin{abstract} \hypersetup{hidelinks}
		Based on the companion paper \cite{KrSchm}, we show that the 4D energy critical Zakharov system admits finite time type II blow up solutions, similar to the ones constructed in \cite{Schm}. The main new difficulty this work deals with is the appearance of a term in the linearization around the approximate solution, which is non-local with respect to both space and time. In particular this cannot be handled by straightforward adaptation of the methods developed in \cite{KST2},  \cite{KST1},  \cite{KST}. The key new ingredients we use are a type of approximate modulation theory, taking advantage of frequency localisations, and the exploitation of an inhomogeneous wave equation with both a non-local,  as well as a local potential term. These terms arise for the main non-perturbative component of the ion density $n$ and can be solved via inversion of a certain Fredholm type operator, as well as by using distorted Fourier methods. Our result relies on a number of numerical non-degeneracy assumptions. 
	\end{abstract}
	
	\maketitle
	\begingroup
	\hypersetup{colorlinks=false,
		linkbordercolor=gray!19}
	\tableofcontents
	\endgroup
	\hypersetup{hidelinks}

	\section{Introduction}
	In the present article we consider the Zakharov system, i.e. we are interested in solutions of the system
		\begin{equation}\label{Zakharov} 
			\begin{cases}
				\;\; i\partial_t \psi+ \triangle \psi= -n\psi &\text{in}~  (0, t_0] \times \R^d \\[4pt]
				\;\; \square n = \triangle(|\psi|^2) &\text{in}~ (0, t_0] \times \R^d ,
			\end{cases}
		\end{equation}
		where $ \square  = -\partial_t^2+ \triangle $ is the wave operator.  The system \eqref{Zakharov} is Hamiltonian with (formally) conserved energy 
		\begin{align*}
			E_Z(\psi, n, \partial_t n)_{|_{t}} = \f12 \int_{\R^d} |\nabla \psi(t)|^2 + \f12 ||\nabla |^{-1} \partial_tn(t)|^2 + \f12 |n(t)|^2  + n(t) |\psi(t)|^2~dx,
		\end{align*}
		and mass $ \|\psi\|_{L^2}^2 $ along solutions $(\psi(t), n(t))$. This model was first introduced in \cite{zakharov} with $ d =3$ dimensions in order to describe rapid oscillations in the electric field of a weakly magnetized plasma, i.e. to explain for instance the phenomena of Langmuir waves. In particular the function $ n : \R^{d+1} \to \R$ in \eqref{Zakharov} models the ion density of the plasma and  $ \psi : \R^{d+1}\to \C $ is a complex envelope for the electric field.\\[3pt] 
		
		For results concerning local well-posedness of this system, as well as 'large' global solutions, see for instance \cite{Bej-Herr-Holmer-Tataru},  \cite{Borgain-Coll}, \cite{Candy-Herr-Nakanishi-Global},  \cite{Candy-Herr-Nakanishi-sharp}, \cite{Ginibre-Tsutsumi-Velo}, \cite{Guo-Nakanishi1}, \cite{Guo-Nakanishi2(Threshold)}, \cite{Kenig-Pon-veg-limit}, to name only few. We refer to our companion paper \cite{KrSchm} for more details and a discussion of some of the known results. 
		Although finite time blow up solutions are expected for the Zakharov system in various dimensions based on numerical evidence, see e. g. the discussion in \cite{Merle}, to the best of our knowledge, the only rigorously known finite time blow up solutions for \eqref{Zakharov} were constructed in $d = 2$ dimensions in the pioneering work~\cite{GM1} (see also \cite{GM2}). The strategy there is based on an ansatz analogous to the explicit pseudo-conformal blow up for the $L^2$-critical focusing nonlinear Schr\"odinger equation. 
	\\	
		
	In this paper we consider the Zakharov system on $\mathbb{R}^{4+1}$, and more specifically, we restrict to the class of radial solutions. 
	\begin{equation}\label{eq:ZakharovMain}\begin{split}
			&i\partial_t\psi + \triangle \psi = -n\psi\\
			&(-\partial_{tt} + \triangle)n = \triangle\big(|\psi\big|^2\big) 
	\end{split}\end{equation}
	Our main theorem is the following 
	\begin{thm}\label{thm:main} 
	Let $W(x) = \big(1+\frac{|x|^2}{8}\big)^{-1}$ and $\nu>1$ be a sufficiently large, irrational number. Also denote $W_{\lambda}(x) = \lambda\cdot W(\lambda x)$ and 
	further $\lambda(t) = t^{-\frac12-\nu}$. Then there is $t_0 = t_0(\nu)>0$ such that the system 
		\eqref{eq:ZakharovMain} admits a finite time radial blow solution $(\psi, n)$ on $(0, t_0]\times \R^4$ of the form 
		\begin{align*}
			&\psi(t, x) = W_{\lambda(t)}(x) +\tilde{\psi}(t, x),\\
			&n(t, x) = W_{\lambda(t)}^2(x) + \tilde{n}(t, x).
		\end{align*}
		We have $\tilde{\psi}(t, \cdot)\in H^2_{\mathbb{R}^4, loc}\cap L^\infty_{\mathbb{R}^4},\,\tilde{n}(t, \cdot)\in H^1_{\mathbb{R}^4,loc}\cap L^\infty_{\mathbb{R}^4}$ for each $t\in (0, t_0]$ and the asymptotic vanishing property 
		\begin{align*}
			&\lim_{t\rightarrow 0}\big\|\tilde{\psi}(t,\cdot)\big\|_{H^1_{\mathbb{R}^4}(|x|\leq t^{\frac12})} = 0,\\
			&\lim_{t\rightarrow 0 }\big\|\tilde{n}(t, \cdot)\big\|_{L^2_{\mathbb{R}^4}(|x|\leq t^{\frac12})} = 0. 
		\end{align*}  
		Moreover, the limits 
		\begin{align*}
			\tilde{\psi}_0: = \lim_{t\rightarrow 0}\tilde{\psi}(t,\cdot),\,\tilde{n}_0: =  \lim_{t\rightarrow 0}\tilde{n}(t,\cdot)
		\end{align*}
		exists in the $\dot{H}^1_{\mathbb{R}^4}$ and the $L^2_{\mathbb{R}^4}$-topology, respectively. 
	\end{thm}
	\begin{rem}\label{rem:thm:main} The regularity of the functions $\tilde{\psi}(t,\cdot), \tilde{n}(t,\cdot)$ is in fact better, and of class $H^{2\nu-1-}_{\mathbb{R}^4,loc}, H^{2\nu-2-}_{\mathbb{R}^4,loc} $
		for any $t\in (0, t_0]$. The restriction to irrational $\nu >1$ is only due to the corresponding approximation result in \cite{KrSchm}.
	\end{rem}
	
	\section{Outline of the strategy and the principal difficulties}
	
	The proof of the theorem will rely on the existence of an approximate finite time blow up solution $(\psi_*, n_*)$ constructed in the companion paper \cite{KrSchm}. The latter followed closely the method developed in \cite{Perelman1}. This approximate solution admits a description like $(\psi, n)$ in the statement of the theorem, and a more precise algebraic description relevant details of which are recalled in subsection~\ref{subsec:approxsolnbasics}. Our strategy for constructing $(\psi, n)$ will be to add correction terms $\tilde{z}, y$ to $\psi_*, n_*$, respectively, and which will be chosen to vanish at the blow up time $t = 0$; thus the {\it{radiation part}} $\tilde{\psi}_0, \tilde{n}_0$ in the statement of the theorem will come exclusively from the approximate solution. Passing to the {\it{Schr\"odinger time}}
	\[
	\tau: = \int_t^\infty \lambda^2(s)\,ds
	\]
	and the re-scaled variable $R = \lambda(t)\cdot r$, $r = |x|$, and assuming a general leading part $\lambda(t)\cdot e^{i\alpha(t)}W(\lambda(t)r)$ for $\psi_*$, where $\alpha(t) = \alpha_0\log t$ for some constant $\alpha_0$ (later on in the construction, we shall impose $\alpha_0 = 0$), we arrive upon using the notation $\tilde{z} = e^{i\alpha(t)}\lambda(t)z$, at an equation for $z$ whose leading order linear part, see \eqref{eq:zeqn2}, is given by  
	\begin{equation}\label{eq:mainlinearpart}
		-i(z_{\tau}+ \frac{\lambda_{\tau}}{\lambda}R\partial_R z)-  c\alpha_0\tau^{-1}z - i\frac{\lambda_{\tau}}{\lambda}z+ \big(\triangle_R + W^2(R)\big)z + \lambda^{-2}y_z\cdot W,
	\end{equation}
	where $y_z = \Box^{-1}\big(2\lambda^2\triangle\big(\Re(W\overline{z})\big)\big)$, with $\Box^{-1}$ the inhomogeneous Duhamel propagator vanishing at $t = 0$ (or equivalently at $\tau = +\infty$). The {\it{key novel difficulty}} compared to earlier works pursuing construction of finite time blow up solutions along the broad line implemented in this paper stems from the term $\lambda^{-2}y_z\cdot W$, which effectively constitutes a {\it{non-local (with respect to space and time) linear operator}} acting on $z$, and which cannot be treated perturbatively. We further make the following observations, which sometimes reference methods developed in the earlier papers \cite{KST2}, \cite{KST1}, \cite{KST}:
	\begin{itemize}
		\item After conjugation to a one dimensional operator on $\mathbb{R}_+$ and re-scaling the operator $-\mathcal{L}: = \triangle_R + W^2$ becomes $\frac{\partial^2}{\partial R^2}  -\frac{3}{4R^2} + \frac{8}{(1+R^2)^2}$, which is precisely the linearization occurring in \cite{KST2}. This operator has a resonance at frequency $0$, corresponding to the function $W(R)$. 
		\item The inhomogeneous propagator for the operator 
		\begin{equation}\label{eq:linearsimplified}
			-i(z_{\tau}+ \frac{\lambda_{\tau}}{\lambda}R\partial_R z)-  c\alpha_0\tau^{-1}z - i\frac{\lambda_{\tau}}{\lambda}z+ \big(\triangle_R + W^2(R)\big)z
		\end{equation}
		can be constructed as in \cite{KST2} by using the spectral representation associated to $\mathcal{L}$ via the distorted Fourier transformation $\mathcal{F}$ and replacing $R\partial_R$ by $\xi\partial_{\xi}$, with $\xi$ denoting the frequency variable. This generates errors which can be described in terms of the {\it{transference operator}} essentially given by $\mathcal{F}\circ\big(R\partial_R\big)\circ\mathcal{F}^{-1} - \xi\partial_{\xi}$. Albeit these errors are linear in $z$, the fact that they come with a temporally decaying weight $\frac{\lambda_{\tau}}{\lambda}\sim \tau^{-1}$ and we shall work with functions decaying rapidly with respect to $\tau$ means that these errors will be of perturbative nature and can be iterated away. We observe that the latter observation is responsible for the fact that the resonance at frequency $0$ plays no essential role in the perturbation theory developed in \cite{KST2}, \cite{KST}. 
		\item The additional term $\lambda^{-2}y_z\cdot W$ in \eqref{eq:mainlinearpart} does not come with a temporally decaying weight, and cannot be treated perturbatively. Still, our strategy shall be to think of $z$ as being obtained by applying the inhomogeneous Schr\"odinger propagator associated to \eqref{eq:linearsimplified}, to all the source terms, as well as the linear term $\lambda^{-2}y_z\cdot W$. It turns out that restricting to either the very small (distorted) frequency regime $\xi<\epsilon_1$, or the very large (distorted) frequency regime $\xi>\epsilon_1^{-1}$, for some $\epsilon_1\ll 1$, and taking advantage of a special cancellation for the resonant part by implementing a form of modulation theory, the term $\lambda^{-2}y_z\cdot W$ can be treated perturbatively. We give more details on this modulation part in the discussion below. 
		\item It then remains to deal with the most delicate case when $z$ is restricted to (distorted) frequency $\xi\in [\epsilon_1, \epsilon_1^{-1}]$. From the preceding discussion, it appears natural that the most delicate contribution to $P_{[\epsilon_1,\epsilon_1^{-1}]}z$ arises when applying the Schr\"odinger propagator associated to \eqref{eq:linearsimplified} to the troublesome source term $ P_{[\epsilon_1,\epsilon_1^{-1}]}\big(\lambda^{-2}y_z\cdot W\big)$. We call this contribution, when reduced to its non-resonant part\footnote{See discussion below} (in addition to a similarly troublesome term arising through the modulation procedure) $z_{nres}^{prin}$. Up until this stage, only the {\it{frequencies with respect to the spatial variable}} have been used, but at this juncture, it turns out to be important to also invoke {\it{frequencies with respect to time}}. In fact, there are two different time variables that play an essential role: we already mentioned the Schr\"odinger time above, but we shall also take advantage of the {\it{wave time}}, which we define as 
		\begin{align*}
			\tilde{\tau}: = \int_t^\infty \lambda(s)\,ds. 
		\end{align*}
		We note that the wave time changes much more slowly than the Schr\"odinger time (when we restrict to very large $\tau$ (and hence also $\tilde{\tau}$)), and so we can think of the term $\lambda^{-2}y_z\cdot W$, where the wave propagator can be expressed via the wave time $\tilde{\tau}$, see e.g. \eqref{eq:wavepropagator}, \eqref{eq:nflatfourierrepresent}, as having a kind of smoothing effect with respect to its dependence on Schr\"odinger time $\tau$. Performing integration by parts with respect to Schr\"odinger time in the Schr\"odinger propagator applied to $P_{[\epsilon_1,\epsilon_1^{-1}]}\big(\lambda^{-2}y_z\cdot W\big)$, one then reformulates the equation for $z_{nres}^{prin}$ in the form
		\begin{equation}\label{eq:znprinreformulated}
			\mathcal{L}z_{nres}^{prin} - 2n_{prin}\cdot W - \lambda^{-2}y_{\tilde{\lambda}}\cdot W = \tilde{E}_{main} + z_{nres, small}^{prin},
		\end{equation}
		see \eqref{eq:znreskeyreofrmulation1}, where the third term on the left is again due to modulating, and we neglect it for this discussion, while the term 
		\[
		n_{prin} = \lambda^{-2}\Box^{-1}\big(\lambda^2\triangle \Re\big(W\overline{z_{nres}^{prin}}\big)\big); 
		\]
		the remaining two terms on the right are perturbative in nature. Restricting the frequency with respect to wave time $\tilde{\tau}$ to size $<\gamma_1$ or $>\gamma_1^{-1}$ for some $\gamma_1\ll 1$, equation \eqref{eq:znprinreformulated} turns to leading order into an elliptic equation, which can be solved easily, 
		\item It remains to deal with the most delicate situation when $z_{nres}^{prin}$ is reduced to wave temporal frequency in $[\gamma_1, \gamma_1^{-1}]$. The strategy here is to {\it{change point of view}} and consider the variable $n_{prin}$, instead of $z_{nres}^{prin}$. It turns out that the function $\tilde{n}_{prin}: = \lambda^2\cdot n_{prin}$ solves a wave equation with a local as well as a non-local potential term, of the form 
		\[
		\Box \tilde{n}_{prin} +2W^2 \tilde{n}_{prin} -K \tilde{n}_{prin} = \text{error}, 
		\]
		where $K$ is a certain time independent integral operator. Restricting wave temporal frequencies to the interval $[\gamma_1,\gamma_1^{-1}]$ here, we shall be able to solve this equation by {\it{applying the Fourier transform with respect to $\tilde{\tau}$}}, and reducing things to the invertibility of a certain Fredholm type operator, see \eqref{eq:keyoperatortobeinverted}. The use of the temporal Fourier transform to translate the solution of the wave equation to an elliptic inversion appears reminiscent of procedures commonly applied in {\it{control theory}}. 
		\item Once control over $n_{prin}$ is obtained, one also obtains control over $z_{nres}^{prin}$ via \eqref{eq:znprinreformulated}, again ignoring here the term $\lambda^{-2}y_{\tilde{\lambda}}\cdot W$. 
	\end{itemize}
	
	Let us now discuss the modulation technique used to control the resonant part of the variable $z$. As we shall work with the Schr\"odinger propagator associated to \eqref{eq:linearsimplified}, we shall expand $z$ via the spectral representation associated to the operator $\mathcal{L}: = -\triangle - W^2$; this representation was already used in \cite{KST2}. Letting $\phi(R;\xi)$ the generalized Fourier basis, we write 
	\begin{equation}\label{eq:zFourierrep}
		z = \int_0^\infty \phi(R;\xi)\cdot \mathcal{F}(z)(\xi)\cdot\rho(\xi)\,d\xi,
	\end{equation}
	where $\rho(\xi)$ is the spectral measure. We shall then use the splitting 
	\begin{equation}\label{eq:znreszres}\begin{split}
			z &= \kappa\cdot\phi(R;0) + \int_0^\infty [\phi(R;\xi) - \phi(R;0)]\cdot \mathcal{F}(z)(\xi)\cdot\rho(\xi)\,d\xi =: z_{res} + z_{nres}\\
			&\kappa =  \int_0^\infty\mathcal{F}(z)(\xi)\cdot\rho(\xi)\,d\xi,
	\end{split}\end{equation}
	where $\mathcal{F}(z)$ is the (distorted) Fourier transform associated to $\mathcal{L}$. While the non-resonant part $ z_{nres}$ enjoys better bounds and in fact one does not lose temporal decay upon using suitable weighted norms and letting $\mathcal{F}(z)$ be the Schr\"odinger propagator of the source terms for the $z$-equation (expressed on the Fourier side), the resonant part appears to suffer a loss of one power of decay and destroy any iterative scheme. This situation is in principle quite analogous to the one encountered in \cite{KS}, where one uses a natural {\it{scaling invariance}} of the equation to enact a {\it{modulation step}} which essentially eliminates the resonant part entirely. 
	\\
	
	While the system \eqref{eq:ZakharovMain} is invariant under the phase shift $(\psi, n)\longrightarrow (e^{i\alpha}\psi, n)$, $\alpha\in \mathbb{R}$, it {\it{does not admit}} a natural scaling transformation. Nonetheless, the family of special approximate finite time blow up solutions in the companion paper \cite{KrSchm} can be embedded into a one parameter family $(\psi_{*}^{(\tilde{\lambda})}, n_*^{(\tilde{\lambda})})$, which {\it{to leading order}} corresponds to re-scaling $(\psi_*, n_*)$ according to the scaling $\psi_*(\cdot, R)\rightarrow \tilde{\lambda}\psi_*(\cdot,\tilde{\lambda}R)$, $n_*(\cdot, R)\rightarrow \tilde{\lambda}^2n_*(\cdot, \tilde{\lambda}R)$. For this see Remarks 2.32, 3.35, 4.4 in \cite{KrSchm}.\\
	The strategy of this paper is to replace $(\psi_*, n_*)$ by a {\it{modulated approximate solution}} $(\psi_*^{(\tilde{\lambda},\underline{\tilde{\alpha}})}, n_*^{(\tilde{\lambda},\underline{\tilde{\alpha}})})$ where the parameters $\tilde{\lambda}, \tilde{\alpha}$ shall be picked as {\it{time dependent functions}} in such fashion as to essentially eliminate the resonant part of $z$. A technical difficulty arising here has to do with the fact that the principal error generated by making $\tilde{\lambda}$ time dependent arises in the wave equation for $n$, and this in turn leads to the main contribution $-\lambda^{-2}y_{\tilde{\lambda}}\cdot W$ in the equation for $z$ (see \eqref{eq:znprinreformulated}), where $y_{\tilde{\lambda}}$ depends in non-local fashion on $\tilde{\lambda}$, being a forward wave propagator applied to certain (local) error terms. This leads to a kind of differential-integral equation for $\tilde{\lambda}$, see \eqref{eq:tildelambda}, whose solution relies on Fourier methods\footnote{Here by 'Fourier' we mean the standard Fourier transform, in one dimension.}. Furthermore, in order to cope with a certain {\it{degeneracy}} of this equation, we shall have to implement certain frequency truncations with respect to the (wave) temporal variable $\tilde{\tau}$, which is manifested by the operators $Q^{(\tilde{\tau})}_{<\tau^{\delta}}$ as well as further implicit such localizations in the equation. It appears that the use of microlocal methods of this type in a modulation theoretic context is a novel feature of this work. \\
	By comparison to $\tilde{\lambda}$, the equation for the second modulation parameter $\tilde{\alpha}$ is much simpler, and given by \eqref{eq:tildealpha1}. Again suitable temporal frequency localizations will be necessary for technical reasons. 
	\\
	
	The construction of the solution $(\psi, n)$ of the theorem will then be obtained by passing to the representation \eqref{eq:finalansatz}, and letting the perturbations $(z, y)$ solve \eqref{eq:zeqn2}\, \eqref{eq:y2def}, \eqref{eq:yzdfn}, while the modulation parameters $\tilde{\lambda}, \tilde{\alpha}$ satisfy
	\eqref{eq:tildelambda}, \eqref{eq:tildealpha1}. In order to infer bounds on $z$, we shall refer to the decomposition \eqref{eq:zdecompbasic} into the resonant and non-resonant parts, with the resonant part in turn solving \eqref{eq:kappa1eqn}, \eqref{eq:kappa2eqn}. The modulation equations in turn are designed in order to ensure good bounds for the solution of the latter pair of equations. 
	
	\section{Numerical non-degeneracy assumptions}
	
	We shall rely on a small number of explicit numerically verifiable non-degeneracy assumptions, see subsection~\ref{subsec:numerics}. These will be verified in a separate paper \cite{KSchmNum}. 
	
	\section{Organization of the paper}
	
	This paper is essentially divided into two parts: the main argument, comprising sections~\ref{section:variablesandeqns} - ~\ref{sec:solncompletion}, and the remaining but lengthy section~\ref{sec:appendix}, where technical tools and completions of certain proofs are given. 
	In section~\ref{section:variablesandeqns}, we introduce the modulated bulk parts $\psi_*^{(\tilde{\lambda}, \underline{\tilde{\alpha}})},\,n_*^{(\tilde{\lambda}, \underline{\tilde{\alpha}})}$ and we set up the equations governing the correction terms $(z, y)$. We further give some background on the distorted Fourier transform associated to the important operator $\mathcal{L}$, and we introduce the norms by means of which we control $z, y$. Finally, we record the evolution equations for the resonant part at the end of section~\ref{section:variablesandeqns}.
	In section~\ref{sec:linpropagators}, we  describe the propagator for the fundamental Schr\"odinger operator \eqref{eq:schrodingermodel} by passing to the distorted Fourier side. This involves in particular the {\it{transference operator}}, whose basic properties we recall. 
	We furthermore, derive a number of technical weighted energy type estimates for the {\it{wave propagator}}, and which shall be useful in controlling the delicate linear term $\lambda^{-2}y_z\cdot W$. 
	Section~\ref{sec:modlnparam} gives the precise modulation equations. The whole remainder of the paper is then concerned with establishing the existence of a solution for the combined system governing $(z, y)$ as well as $\tilde{\lambda}, \tilde{\alpha}$. This eventually leads to  
	Proposition~\ref{prop:existenceofsolution}, which then easily implies Theorem~\ref{thm:main}.
	
	\section{Remarks on notation}
	
	We shall often use expressions such as $\|\langle R\partial_R\rangle f\|_{X}$ for some norm $\|\cdot \|_{X}$. By this we shall mean 
	\[
	\|\langle R\partial_R\rangle f\|_{X}: = \big\|f\big\|_{X} + \big\|R\partial_R f\big\|_{X}. 
	\]
	We shall also use interpolation to define {\it{fractional operators}} such as $(R\partial_R)^{1+\delta}$, $0<\delta\ll 1$. We then define
	\[
	\|\langle R\partial_R\rangle^{1+\delta} f\|_{X}: = \big\|f\big\|_{X} + \big\|(R\partial_R)^{1+\delta} f\big\|_{X}. 
	\]
	We also frequently use norms of the form 
	\[
	\big\|f(R)\big\|_{g(R)\cdot L^p_{R^3\,dR}}: = \big\|g^{-1}\cdot f\big\|_{L^p_{R^3\,dR}},
	\]
	and similarly for other variables. 
	\\
	The notation $\tau^{-N+}$ shall mean $\tau^{-N+\delta_1}$ for some $0<\delta_1\ll 1$. Similarly, the notation $\tau^{-N-}$ shall mean $\tau^{-N-\delta_1}$ for some $0<\delta_1\ll 1$.
	\\
	We frequently encounter integrals of the schematic form $\int_0^\infty \frac{g(\xi)}{\hat{\tau}^2-\xi^2}\,d\xi$, where $\hat{\tau}\in \mathbb{R}_+$. Our tacit convention shall be that such integrals are to be interpreted in the {\it{principal value sense}}, i. e. as 
	\begin{align*}
	\lim_{\epsilon\downarrow 0}\int_0^\infty\chi_{|\xi-\hat{\tau}|\geq \epsilon} \frac{g(\xi)}{\hat{\tau}^2-\xi^2}\,d\xi
	\end{align*}

	The relation $A\lesssim B$ shall mean that there is an absolute constant $C$ such that $A\leq CB$, and similarly for $A\gtrsim B$. If $\epsilon>0$ is a small parameter, we write 
	\[
	A\ll_{\epsilon} B
	\]
	to mean that for any $\delta>0$ there is $\epsilon_*>0$ small enough such that for $0<\epsilon<\epsilon_*$ we have $A<\delta B$. If $0<X$ is a large parameter we say 
	\[
	A\ll_{X} B
	\]
	if the same relation as before holds with $\epsilon: = X^{-1}$. 
	\\
	
	Throughout, with $\nu\gg 1$ fixed, we shall let $\tau_*\gg 1$ a very large number, which will be the largest of a hierarchy of constants. We also use $N$ (coming from the temporal decay of the approximate solution), $0<\epsilon_1\ll 1$, and $0<\gamma_1\ll1 $, where we assume that 
	\[
	\tau_*^{-1}\ll N^{-1}\ll \epsilon_1\ll \gamma_1\ll\nu^{-1}. 
	\]
	Throughout we shall need to solve our equations on $[\tau_*,\infty)\times \mathbb{R}^4$, where the interval $[\tau_*,\infty)$ refers to {\it{Schr\"odinger time}}, or alternatively $[\tilde{\tau}_*,\infty)$ in terms of {\it{wave time}}. We can and shall assume that all functions that we use and sometimes apply non-local operators to are supported at Schr\"odinger time $\tau\geq \frac{\tau_*}{2}$, if necessary by applying suitable smooth cutoffs identically equal one on $[\tilde{\tau}_*,\infty)$. This being assumed, we don't always indicate the time interval $[\tau_*,\infty)$ in norms such as $\|\cdot\|_{\tau^{-N}L^2_{d\tau}}$, it being implicitly assumed that it is restricted to $[\tau_*,\infty)$.  
	\\
	
	Finally, we note that we shall use the notation $Q^{(\tilde{\tau})}_{<a}$ to denote frequency localization\footnote{In other contexts, the notation $Q_{<a}$ is sometimes used to refer to the {\it{modulation}} of waves, i. e. the distance of the Fourier support to the light cone, but our usage is different.} to frequencies $<a$ with respect to {\it{wave time}} $\tilde{\tau}$. 
	
	\section{Description of the system in terms of dynamical variables, including modulation parameters, and their equations}\label{section:variablesandeqns}
	
	\subsection{Initial setup and modulation ansatz}
	
	Our point of departure is an approximate solution $(\psi_*, n_*)$ which solves the system 
	\begin{equation}\label{eq:mainapprox}\begin{split}
			&i\partial_t\psi_* + \triangle \psi_* = -n_*\psi_* + E_1\\
			&(-\partial_{tt} + \triangle)n_* = \triangle\big(|\psi_*\big|^2\big) + E_2,
	\end{split}\end{equation}
	where the error terms vanish rapidly towards $t = 0$:
	\[
	\big\|E_1\big\|_{L^2_{R^3\,dR}}\leq t^N,\,\big\|E_2\big\|_{L^2_{R^3\,dR}}\leq t^N.
	\]
	For this see Theorem 1.1 in \cite{KrSchm}. 
	Throughout we write $R = \lambda(t)r$, $\lambda(t) = t^{-\frac12-\nu}$, and as before we define the {\it{Schr\"odinger time}} as 
	\[
	\tau: = \int_t^\infty \lambda^2(s)\,ds
	\]
	as well as the {\it{wave time}} as 
	\[
	\tilde{\tau}: =  \int_t^\infty \lambda(s)\,ds.
	\]
	In particular, we infer the algebraic relation 
	\[
	\tilde{\tau}\sim_{\nu}\tau^{\frac12 - \frac{1}{4\nu}}. 
	\]
	Setting 
	\[
	\psi_*(t, x) = e^{i\alpha(t)}\lambda(t)\cdot u_*(t, x), 
	\]
	In principle we intend to perturb $(u_*, n_*)$ to $(u_* + z, n_* + y)$. However, in order to control the growth of the resonant part of $z$, we shall have to modify the bulk part $(u_*, n_*)$. Precisely, we shall replace 
	\[
	(\psi_*, n_*)
	\]
	by {\it{partially modulated expressions}}. First, note that we have the phase modulation 
	\[
	\big(\psi, n\big)\longrightarrow \big(e^{i\alpha}\psi, n\big). 
	\]
	We intend to phase modulate $\psi$ but also in dependence on $r = |x|$. In order to avoid uncontrollable errors arising in the wave equation, we carefully choose 
	\begin{align*}
		\psi_*\longrightarrow e^{i\chi_1(r,t)\cdot \tilde{\alpha}}\psi_*. 
	\end{align*}
	Here the $C^\infty$-function $\chi_1(r, t)$ smoothly localizes to the region $r\leq\frac12 t^{\frac12+\epsilon}$, and specifically it equals $1$ for $r\in [0, \frac14 t^{\frac12+\tilde{\epsilon}}]$ for some $\tilde{\epsilon}>\epsilon$. We shall require the somewhat technical condition that 
	\begin{equation}\label{eq:chi1vanishingcond}
		\langle \chi_1W, W\rangle_{L^2_{R^3\,dR}} = 0,
	\end{equation}
	where we recall that $R = r\cdot\lambda$. This can be achieved for $\chi_1$ bounded in absolute value by a constant $C(\epsilon, \tilde{\epsilon},\nu)$. Furthermore, we can assume the bounds 
	\begin{equation}\label{eq:chi1bound}
		\big|\partial_R^l \chi_1\big|\lesssim_{\epsilon,\tilde{\epsilon},\nu} \sigma^{-l(\frac12-\frac{\tilde{\epsilon}}{2\nu})},\,l\geq 0. 
	\end{equation}
	Observe that we have 
	\begin{align*}
		e^{i\chi_1(r,t)\cdot \tilde{\alpha}}\psi_* - \psi_* &= i\chi_1(r,t)\cdot \tilde{\alpha}\psi_* + O\big(|\tilde{\alpha}|^2\big)\\
		& =\chi_1(r,t)e^{i\tilde{\alpha}}\psi_* + \chi_2(r, t)\psi_* -\psi_* +  O\big|\tilde{\alpha}|^2\big),
	\end{align*}
	where we define the second cutoff by $\chi_1 + \chi_2 = 1$. 
	Thus at the level of the Schr\"odinger equation the effect of including the spatial cutoffs into the phase is equivalent to truncating separately from the phase, up to quadratic and hence harmless errors. 
	On the other hand, we have that 
	\begin{align*}
		\triangle \Big|e^{i\chi_1(r,t)\cdot \tilde{\alpha}}\psi_*\Big|^2 = \triangle \Big|\psi_*\Big|^2, 
	\end{align*}
	whence we have not introduced new errors into the wave equation, provided we do not modify $n_*$. 
	\\
	
	We shall also have to modulate {\it{with respect to scaling}}. This is in principle unnatural, as the Zakharov system is not scaling invariant, but we interpret this to mean that we construct the approximate solution 
	\[
	(\psi_*, n_*)
	\]
	by using $\tilde{\lambda}\cdot \lambda(t)$, for $\tilde{\lambda}$ a positive constant, resulting in 
	\[
	(\psi_*^{(\tilde{\lambda})}, n_*^{(\tilde{\lambda})}), 
	\]
	and then letting $\tilde{\lambda}$ depend on time. 
	Finally, we arrive at the following precise modulation ansatz: 
	\begin{equation}\label{eq:psinmodulated}\begin{split}
			&(\psi_*, n_*)\longrightarrow\big(\psi_*^{(\tilde{\lambda}, \underline{\tilde{\alpha}})},\,n_*^{(\tilde{\lambda}, \underline{\tilde{\alpha}})}\big)\\
			&\psi_*^{(\tilde{\lambda}, \underline{\tilde{\alpha}})} = \chi_3\cdot e^{i\chi_1(r,t)\cdot \tilde{\alpha}}\cdot \psi_*^{(\tilde{\lambda})} + (1-\chi_3)\cdot \psi_*,\\
			& n_*^{(\tilde{\lambda}, \underline{\tilde{\alpha}})} = \chi_3\cdot n_*^{(\tilde{\lambda})} + (1-\chi_3)\cdot n_*. 
	\end{split}\end{equation}
	Here the cutoff $\chi_3(r, t)$ smoothly localizes to the inner Schr\"odinger zone $r\leq t^{\frac12+\epsilon}$ where $\epsilon\ll 1$ is as in the construction of the approximate solution. It is chosen to equal $1$ for $r<\frac12 t^{\frac12+\epsilon}$ and to equal zero for $r>2t^{\frac12+\epsilon}$. In particular, we have 
	\[
	\psi_*^{(\tilde{\lambda}, \underline{\tilde{\alpha}})} = e^{i\chi_1(r,t)\cdot \tilde{\alpha}}\cdot \psi_*^{(\tilde{\lambda})}
	\]
	on the support of $\chi_1$. 
	\subsection{The equation for the modulated ansatz}
	
	Recalling that we have 
	\begin{align*}
		\big(i\partial_t + \triangle\big)\psi_*^{(\tilde{\lambda})} = -n_*^{(\tilde{\lambda})}\cdot \psi_*^{(\tilde{\lambda})} + E_1,\,(-\partial_{tt} + \triangle)n_*^{(\tilde{\lambda})} = \triangle\big(|\psi_*^{(\tilde{\lambda})}\big|^2\big) + E_2,
	\end{align*}
	we infer the following modified equations for the modulated bulk term: 
	\begin{equation}\label{eq:Modbulkeqns}\begin{split}
			&\big(i\partial_t + \triangle\big)\psi_*^{(\tilde{\lambda}, \underline{\tilde{\alpha}})} = -n_*^{(\tilde{\lambda}, \underline{\tilde{\alpha}})}\cdot \psi_*^{(\tilde{\lambda}, \underline{\tilde{\alpha}})} + E_1 + E_1^{\text{mod}},\\
			&(-\partial_{tt} + \triangle)n_*^{(\tilde{\lambda},\underline{\tilde{\alpha}})} = \triangle\big(|\psi_*^{(\tilde{\lambda},\underline{\tilde{\alpha}})}\big|^2\big) + E_2 + E_2^{\text{mod}},
	\end{split}\end{equation}
	and where we set 
	\begin{equation}\label{eq:E1mod}\begin{split}
			E_1^{\text{mod}} &= \big(i\partial_t + \triangle\big)( \chi_3)\cdot \big(\psi_*^{(\tilde{\lambda})} - \psi_*\big) + 2\partial_r( \chi_3)\cdot (\partial_r\psi_*^{(\tilde{\lambda})} - \partial_r\psi_*)\\
			& +  i\big(i\partial_t + \triangle\big)(\chi_1)\tilde{\alpha}\cdot \psi_*^{(\tilde{\lambda})} + 2i\partial_r(\chi_1)\cdot\tilde{\alpha}\cdot \partial_r \psi_*^{(\tilde{\lambda})}\\
			& - \chi_1\cdot \partial_t\tilde{\alpha}\cdot \psi_*^{(\tilde{\lambda})} - E_{\text{nl}}^{\text{mod}} + i\tilde{\lambda}_t\cdot\partial_{\tilde{\lambda}}\big(\psi_*^{(\tilde{\lambda})}\big)+O\big(|\tilde{\alpha}|^2\big)
	\end{split}\end{equation}
	with 
	\begin{align*}
		E_{\text{nl}}^{\text{mod}} = \big(\chi_3^2 - \chi_3\big)\cdot \Big[n_*^{(\tilde{\lambda})}\cdot \psi_*^{(\tilde{\lambda})} - n_*\cdot\psi_*\Big],
	\end{align*}
	and further 
	\begin{equation}\label{eq:E2mod}\begin{split}
			E_2^{\text{mod}} &= -\tilde{\lambda}_{tt}\cdot \partial_{\tilde{\lambda}}n_*^{(\tilde{\lambda})} - \tilde{\lambda}_t\cdot \partial_t\big(\partial_{\tilde{\lambda}}n_*^{(\tilde{\lambda})}\big)\\
			& - 2\partial_t(\chi_3)\cdot  \tilde{\lambda}_t\cdot \partial_{\tilde{\lambda}}n_*^{(\tilde{\lambda})}+ \Box(\chi_3)\cdot \big(n_*^{(\tilde{\lambda})} - n_*\big)\\
			&- 2\partial_r(\chi_3)\cdot \partial_r\big(n_*^{(\tilde{\lambda})} - n_*\big) +\triangle_r(\chi_3)\cdot\big(|\psi_*^{(\tilde{\lambda})}|^2 - |\psi_*|^2\big)\\
			&+2\partial_r(\chi_3)\cdot \partial_r\big(|\psi_*^{(\tilde{\lambda})}|^2 - |\psi_*|^2\big)\\
	\end{split}\end{equation}
	
	\subsection{The equations for the final corrections $(z, y)$}
	
	We shall seek to construct a solution of the following form:
	\begin{equation}\label{eq:finalansatz}\begin{split}
			&\psi(t, x) = \psi_*^{(\tilde{\lambda}, \underline{\tilde{\alpha}})}  + e^{i\alpha(t)}\lambda(t)\cdot z\\
			&n(t, x) = n_*^{(\tilde{\lambda}, \underline{\tilde{\alpha}})} + y. 
	\end{split}\end{equation}
	Then we derive the following system of equations for $(z, y)$, where we observe that $\alpha(t) = \alpha_0\log t$, and we divide by $\lambda^3(t)$ in the equation for $z$ to account for the factor $\lambda$ in the preceding equation, as well as the change of coordinates. To simplify the notation, write
	\begin{align*}
		\psi_*^{(\tilde{\lambda}, \underline{\tilde{\alpha}})} = e^{i\alpha(t)}\lambda(t)\cdot \tilde{u}_*^{(\tilde{\lambda}, \underline{\tilde{\alpha}})}. 
	\end{align*}
	Observe that the new variable $\tilde{u}_*^{(\tilde{\lambda}, \underline{\tilde{\alpha}})}$ satisfies the following equation in terms of the partially re-scaled coordinates $(R, t)$:
	\begin{equation}\label{eq:ustartilde}\begin{split}
			&-\alpha'(t)\lambda^{-2}(t)\tilde{u}_*^{(\tilde{\lambda}, \underline{\tilde{\alpha}})}(t,R) + i\lambda'(t)\lambda^{-3}(t)\tilde{u}_*^{(\tilde{\lambda}, \underline{\tilde{\alpha}})}(t,R)\\&\hspace{2cm} + i\lambda^{-2}(t)(\tilde{u}^{(\tilde{\lambda}, \underline{\tilde{\alpha}})}_{*,t} + \frac{\lambda_t}{\lambda} R\partial_R \tilde{u}_*^{(\tilde{\lambda}, \underline{\tilde{\alpha}})}) + \triangle \tilde{u}_*^{(\tilde{\lambda}, \underline{\tilde{\alpha}})}\\
			&  = -\lambda^{-2}(t)n_*^{(\tilde{\lambda}, \underline{\tilde{\alpha}})} \tilde{u}_*^{(\tilde{\lambda}, \underline{\tilde{\alpha}})} + e_1 + e_1^{\text{mod}}, \\
	\end{split}\end{equation}
	where we set 
	\begin{equation}\label{eq:e1moddef}
		e_1^{\text{mod}} = \lambda^{-3}(t)\cdot  E_1^{\text{mod}},
	\end{equation}
	and $\triangle = \partial_R^2 + \frac{3}{R}\partial_R$. Moreover, we recall that $\alpha'(t) = \alpha_0 t^{-1}$. 
	\\
	We can now write down the combined system for $(z, y)$. We shall record the equation for $z$ in terms of $(\tau, R)$:
	\begin{equation}\label{eq:zeqn1}\begin{split}
			&-c\alpha_0\tau^{-1}z - i\frac{\lambda_{\tau}}{\lambda}z - i(z_{\tau} + \frac{\lambda_{\tau}}{\lambda}R\partial_R z) + \triangle_R z\\
			& = -\lambda^{-2}y \tilde{u}_*^{(\tilde{\lambda}, \underline{\tilde{\alpha}})} - \lambda^{-2}n_*^{(\tilde{\lambda}, \underline{\tilde{\alpha}})} z - \lambda^{-2}yz + e_1 + e_1^{\text{mod}}
	\end{split}\end{equation}
	where $c = c(\nu)>0$ satisfies 
	\[
	t^{-1}\lambda^{-2}(t) = c\cdot \tau^{-1}. 
	\]
	On the other hand, we formulate the equation for $y$ in terms of the $(t, r)$-coordinates for now:
	\begin{equation}\label{eq:yeqn1}\begin{split}
			\Box y = 2\lambda^2\triangle\big(\Re(\tilde{u}_*^{(\tilde{\lambda}, \underline{\tilde{\alpha}})} \overline{z})\big) + \lambda^2 \triangle\big(|z^2|\big) + E_2 +  E_2^{\text{mod}}
	\end{split}\end{equation}
	
	The equation for $z$ needs to be re-arranged to reveal the principal linear part, which is given by the expression 
	\begin{equation}\label{eq:zeqnprincipallinear}
		\triangle_R z + W^2(R)z + \lambda^{-2}y_z\cdot W, 
	\end{equation}
	where $y_z$ is defined implicitly as the solution vanishing at $\tilde{\tau} = +\infty$ of the following equation 
	\begin{equation}\label{eq:yzdfn}
		\Box y_z = 2\lambda^2\triangle\big(\Re(W\overline{z})\big).
	\end{equation}
	Instead of working with $y$, we shall then work with the 'better' difference term 
	\begin{equation}\label{eq:y2def}
		y_2: = y - y_z.
	\end{equation}
	We now arrive at the following system, which shall serve as the basis to derive the required estimates:
	\begin{equation}\label{eq:zeqn2}\begin{split}
			&-i(z_{\tau}+ \frac{\lambda_{\tau}}{\lambda}R\partial_R z)-  c\alpha_0\tau^{-1}z - i\frac{\lambda_{\tau}}{\lambda}z+ \big(\triangle_R + W^2(R)\big)z + \lambda^{-2}y_z\cdot W\\& =  -\lambda^{-2}\big(y \tilde{u}_*^{(\tilde{\lambda}, \underline{\tilde{\alpha}})} - y_z\cdot W\big)-\big(\lambda^{-2}n_*^{(\tilde{\lambda}, \underline{\tilde{\alpha}})} - W^2\big) z - \lambda^{-2}yz +  e_1 + e_1^{\text{mod}}\\
			&\Box y_2 = 2\lambda^2\triangle\big(\Re(\tilde{u}_*^{(\tilde{\lambda}, \underline{\tilde{\alpha}})}  \overline{z})\big) - 2\lambda^2\triangle\big(\Re(W\overline{z})\big) + \lambda^2 \triangle\big(|z^2|\big) + E_2 +  E_2^{\text{mod}}. 
	\end{split}\end{equation}
	Again the first equation is written in terms of the coordinates $(\tau, R)$, while the second equation is in terms of the original coordinates $(t, r)$. 
	At this point, we observe that we have not yet specified the evolution of the modulation parameters. In fact, this shall be done later on, with the goal of controlling the evolution of the resonant part of $z$. 
	\\
	
	For later reference, we already introduce a further important component of $y_2 = y - y_z$, which arises due to modulating in $\tilde{\lambda}$. In fact, the source term 
	\[
	-\tilde{\lambda}_{tt}\cdot \partial_{\tilde{\lambda}}n_*^{(\tilde{\lambda})} - \tilde{\lambda}_t\cdot \partial_{t}(\partial_{\tilde{\lambda}}n_*^{(\tilde{\lambda})})
		\]
	can be equated to leading order with 
	\begin{equation}\label{eq:tildelambdaleadcontrib}
		-2\tilde{\lambda}_{tt}\cdot\lambda^2 \Lambda W\cdot W - 2\frac{\tilde{\lambda}_t}{t}\cdot  (t\partial_t)\big(\lambda^2 \Lambda W\cdot W\big),\,\Lambda W = \partial_{\tilde{\lambda}}W_{\tilde{\lambda}}\big|_{\tilde{\lambda} = 1},\,W_{\tilde{\lambda}}(R) = \tilde{\lambda}\cdot W(\tilde{\lambda} R).
         \end{equation}
	Then let us set 
	\begin{equation}\label{eq:ytildelamba}
		y_{\tilde{\lambda}}: = -\Box^{-1}\big(2\tilde{\lambda}_{tt}\cdot\lambda^2 \Lambda W\cdot W + 2\frac{\tilde{\lambda}_t}{t}\cdot  (t\partial_t)\big(\lambda^2 \Lambda W\cdot W\big), 
	\end{equation}
	where as usual $\Box^{-1}$ is defined by imposing vanishing towards $\tilde{\tau} = +\infty$. Note that this term leads to a contribution to the {\it{first equation}} in \eqref{eq:zeqn2}, which to leading order is the source term 
	\begin{equation}\label{eq:ytildelambdaleadingcontr}
		-\lambda^{-2}\cdot y_{\tilde{\lambda}}\cdot W.
	\end{equation}
	
	\subsection{Basic spectral theory}\label{subsec:basicfourier}
	
	We use Fourier representations with respect to the reference operator $\mathcal{L} = -\triangle_R - W^2$, in light of \eqref{eq:zeqn2}. Thus a general function $f\in L^2_{R^3\,dR}$ admits an expansion 
	\begin{equation}\label{eq:Lrep}
		f(R) =  \int_0^\infty \phi(R;\xi)\mathcal{F}(f)(\xi)\cdot\rho(\xi)\,d\xi
	\end{equation}
	where the generalized Fourier basis $\phi(R;\xi)$ is essentially\footnote{Here we change the variable $\xi^{\frac12}$ from \cite{KST2} to $\xi$ for convenience.}  given in \cite{KST2}. In particular, we record the following:
	\begin{enumerate}
		\item We have $\phi(R; 0) = W(R)\sim R^{-2}$ as $R\rightarrow\infty$. Furthermore, in the non-oscillatory regime $R\xi\lesssim 1$ there is an expansion 
		\[
		\phi(R;\xi) = \phi(R; 0)\cdot [1+\sum_{j\geq 1}(R\xi)^{2j}\phi_j(R^2)]
		\]
		which converges absolutely, and where the functions $\phi_j(u)$ are smooth (in fact, analytic), and satisfy the bounds
		\[
		\big|\phi_j(u)\big|\leq \frac{C^j}{j!}
		\]
		and similar bounds for the derivatives, where $C$ is a suitable constant. 
		\item If $\xi\lesssim 1$ and we are in the oscillatory regime $R\xi\gtrsim 1$, we can write 
		\[
		\phi(R; \xi) = \sum_{\pm}\xi^{\frac12} \frac{e^{\pm iR\xi}}{R^{\frac32}}\cdot a_{\pm}(R;\xi),
		\]
		where the functions $a(R;\xi)$ are smooth on $[0,\infty)\times (0,1]$ and  satisfy the bounds 
		\[
		\big|a_{\pm}(R;\xi)\big|\lesssim \langle\log\xi\rangle ,
		\]
		as well as symbol bounds with respect to derivatives in terms of both variables. In the high frequency regime $\xi\gtrsim 1$, we can write 
		\[
		\phi(R;\xi) = \sum_{\pm}\frac{e^{\pm iR\xi}}{R^{\frac32}\xi^{\frac32}}\cdot a_{\pm}(R;\xi)
		\]
		where the functions $a(R;\xi)$ are bounded, smooth and satisfy symbol bounds with respect to both $R$ and $\xi$. 
		\item The spectral measure is smooth on $(0,\infty)$ and has the asymptotic behavior 
		\[
		\rho(\xi)\sim \frac{1}{\xi\log^2\xi}
		\]
		for $0<\xi\lesssim 1$, as well as 
		\[
		\rho(\xi)\sim \xi^3
		\]
		for $\xi\gtrsim 1$. It also satisfies symbol bounds. 
	\end{enumerate}
	
	The singular behavior of the spectral measure $\rho(\xi)$ when $\xi\rightarrow 0+$, reflecting the presence of the resonance $\phi(R; 0)$, is a key difficulty when controlling the evolution of $z$. To deal with it, we shall split functions into a 'resonant' and a 'non-resonant' part: using \eqref{eq:Lrep}, we split
	\begin{equation}\label{eq:resnonressplit}\begin{split}
			&f(R) = f_{res}(R) + f_{nres}(R)\\
			&f_{nres}(R): = \int_0^\infty [\phi(R;\xi) - \phi(R;0)]\mathcal{F}(f)(\xi)\cdot\rho(\xi)\,d\xi
	\end{split}\end{equation}
	Note that we can also formally write 
	\begin{equation}\label{eq:fnralternative}
		f_{nres} = \mathcal{L}^{-1}\big(\mathcal{L}f\big), 
	\end{equation}
	where the inverse $\mathcal{L}^{-1}$ is uniquely determined by imposing vanishing at the origin $R = 0$. It follows that we can write 
	\begin{align*}
		f_{res}(R) = \kappa\cdot \phi(R; 0), 
	\end{align*}
	and the function $f$ is completely determined in terms of the pair $\big(\kappa, f_{nres}\big)$. In order to describe $z$, we shall revert to this setting. 
	
	\subsection{Introducing the norm controlling the non-resonant part of $z$, and the norm for $y$}
	
	We shall control the non-resonant part $z_{nres}$ of $z$ in terms of the following norm:
	\begin{equation}\label{eq:Snormdefi}\begin{split}
			\big\|z_{nres}\big\|_{S}: &= \big\|\langle R\rangle^{-\delta_0}z_{nres}\big\|_{\tau^{-N}L^2_{d\tau}L^\infty_{dR}} + \big\|\langle R\rangle^{\frac12-\delta_0}\nabla_R z_{nres}\big\|_{\tau^{-N}L^2_{d\tau}L^\infty_{dR}}\\
			& + \big\|\mathcal{L}^2z_{nres}\big\|_{U} + \big\|z_{nres}\big\|_{\tau^{-N+1+}L^2_{d\tau}(L^{2+}_{R^3\,dR}+L^{\frac83+}_{R^3\,dR})}
	\end{split}\end{equation}
	where we define\footnote{The notation $f\in \mathcal{L}^{-\frac14}S'$ means that if $f(R) = \int_0^\infty \phi(R;\xi)\mathcal{F}(f)(\xi)\rho(\xi)\,d\xi$ then $\mathcal{L}^{\frac14}f = \int_0^\infty \phi(R;\xi)\xi^{\frac12}\mathcal{F}(f)(\xi)\rho(\xi)\,d\xi\in S'$.}
	\begin{align*}
		\big\|f\big\|_{U}: = \big\|f\big\|_{\tau^{-N}L^2_{d\tau}(L^{2+}_{R^3\,dR} + \tau^{-\frac12-\frac{1}{4\nu}+}[\langle R\rangle^{\delta_0}L^\infty_{R^3\,dR}\cap \mathcal{L}^{-\frac14}(\langle R\rangle^{-\frac12+\delta_0}L^\infty_{R^3\,dR})]}
	\end{align*}
	For $y$ we use the simpler norm (where $\nabla = \partial_R$)
	\begin{equation}\label{eq:ynorm}
		\big\|y\big\|_{Y}: = \big\|\lambda^{-2}\langle R\rangle^{\delta_0}\langle \nabla^2\rangle y\big\|_{\tau^{-N+1}L^2_{d\tau}L^{2+}_{R^3\,dR}} + \big\|\lambda^{-2}\langle \nabla^2\rangle y\big\|_{\tau^{-N+1}L^2_{d\tau}L^{2}_{R^3\,dR}}.
	\end{equation}

	\subsection{Equations for the resonant and non-resonant parts of $z$} We decompose (we write $\phi_0(R): = \phi(R; 0)$)
	\begin{equation}\label{eq:zdecompbasic}\begin{split}
			&z(\tau, R) = z_{res}(\tau, R) + z_{nres}(\tau, R),\,z_{res}(\tau, R) = \kappa(\tau)\cdot \phi_0(R),\\
			&\kappa(\tau) = \kappa_1(\tau) + i\kappa_2(\tau),\,\kappa_j(\tau)\in \R,\,j = 1, 2. 
	\end{split}\end{equation}
	The evolution equation of $\kappa(\tau)$, and more specifically, its real and imaginary parts, is then obtained by evaluating the first equation of \eqref{eq:zeqn2} at $R = 0$:
	\begin{equation}\label{eq:kappa1eqn}\begin{split}
			-\kappa_{1,\tau} - \frac{\lambda_{\tau}}{\lambda}\kappa_1(\tau) - c\alpha_0\tau^{-1}\kappa_2(\tau) &= \Im\big(\mathcal{L}z\big)|_{R = 0} - \Im\big(\lambda^{-2}(y \tilde{u}_*^{(\tilde{\lambda}, \underline{\tilde{\alpha}})})\big)|_{R = 0}\\
			&-\Im\big[\big(\lambda^{-2}n_*^{(\tilde{\lambda}, \underline{\tilde{\alpha}})} - W^2\big) z\big]|_{R = 0} - \Im\big(\lambda^{-2}yz\big)|_{R = 0}\\& +  \Im\big(e_1 + e_1^{\text{mod}}\big)|_{R = 0},  
	\end{split}\end{equation}
	\begin{equation}\label{eq:kappa2eqn}\begin{split}
			\kappa_{2,\tau} + \frac{\lambda_{\tau}}{\lambda}\kappa_2(\tau) - c\alpha_0\tau^{-1}\kappa_1(\tau) &=  \Re\big(\mathcal{L}z\big)|_{R = 0} - \Re\big(\lambda^{-2}(y \tilde{u}_*^{(\tilde{\lambda}, \underline{\tilde{\alpha}})})\big)|_{R = 0}\\
			&-\Re\big[\big(\lambda^{-2}n_*^{(\tilde{\lambda}, \underline{\tilde{\alpha}})} - W^2\big) z\big]|_{R = 0} - \Re\big(\lambda^{-2}yz\big)|_{R = 0}\\& +  \Re\big(e_1 + e_1^{\text{mod}}\big)|_{R = 0}.
	\end{split}\end{equation}
	As for the non-resonant part $z_{nres}(\tau, R)$, we can describe it in terms of the distorted Fourier representation as in \eqref{eq:resnonressplit}:
	\begin{equation}\label{eq:znres}
		z_{nres}(\tau, R): = \int_0^\infty [\phi(R;\xi) - \phi(R;0)]\mathcal{F}(z)(\tau, \xi)\cdot\rho(\xi)\,d\xi
	\end{equation}
	Here $\mathcal{F}(z)(\tau, \xi)$ shall be obtained by directly solving \eqref{eq:zeqn2} via translation to the Fourier side. 
	\\
	
	To close this subsection, we remark that the modulation parameters $\tilde{\lambda}, \tilde{\alpha}$ will be chosen in such a way as to essentially  eliminate $\kappa_{1,2}$. More precisely, the role of $\tilde{\alpha}$ shall be to control (but not completely eliminate) $\kappa_2$, while the role of the remaining modulation parameter shall be to control (but not completely eliminate) $\kappa_1$. 
	
	\section{Basic linear propagators for the Schr\"odinger and wave equations}\label{sec:linpropagators}
	
	\subsection{Translation of the Schr\"odinger equation to the Fourier side}
	
	Consider a model equation of the form 
	\begin{equation}\label{eq:schrodingermodel}
		-i(z_{\tau}+ \frac{\lambda_{\tau}}{\lambda}R\partial_R z)-  c\alpha_0\tau^{-1}z - i\frac{\lambda_{\tau}}{\lambda}z+ \big(\triangle_R + W^2(R)\big)z = E(\tau, R). 
	\end{equation}
	For us we shall mostly be concerned with the source function $E$ which is the difference of the right hand side of  the first equation in \eqref{eq:zeqn2} and $\lambda^{-2}y_z\cdot W$. 
	Our intention is to reformulate this equation in terms of the distorted Fourier transform of $z$, for which the presence of the scaling operator $R\partial_R $ poses a technical obstacle. This is dealt with by taking advantage of the  {\it{transference operator}} $\mathcal{K}$, which is defined as follows: 
	\begin{equation}\label{eq: transference1}
		\mathcal{F}\big((R\partial_R) f\big)(\xi) = -(\xi\partial_{\xi})\mathcal{F}(f)(\xi)- 3f(\xi) + \mathcal{K}\big(\mathcal{F}(f)\big)(\xi)
	\end{equation}
	From \cite{KST2} we note that the operator $\mathcal{K}$ is given in terms of a kernel $F(\xi,\eta)$ by means of the formula
	\begin{equation}\label{eq:Kformula}
		\mathcal{K}f(\xi) = \int_{0}^\infty \frac{F(\xi,\eta)\rho(\eta)}{\xi - \eta}f(\eta)\,d\eta. 
	\end{equation}
	We record the following basic result from \cite{KST2}, which implies the required boundedness properties of operator $\mathcal{K}$, see also \cite{KMS} for the 'trace' derivative bounds:
	\begin{lem}\label{lem:KboundsKST} The kernel $F(\xi,\eta)$ is of class $C^2\big((0,\infty)\times (0, \infty)\big)$, and satisfies the pointwise bounds 
		\begin{equation}\begin{split}
				\big|F(\xi,\eta)\big|\lesssim \Bigg\{\begin{array}{cc}\xi+\eta & \xi+\eta\leq 1\\ (\xi+\eta)^{-3}\cdot(1+|\xi - \eta|)^{-N} & \xi+\eta>1\end{array}
		\end{split}\end{equation}
		as well as the derivative bounds 
		\begin{equation}\begin{split}
				\sup_{j+k=2}\big|\partial_{\xi}^j \partial_{\eta}^k F(\xi,\eta)\big|\lesssim \Bigg\{\begin{array}{cc}|\log(\xi+\eta)|^3 & \xi+\eta\leq 1\\ (\xi+\eta)^{-5}\cdot(1+|\xi - \eta|)^{-N} & \xi+\eta>1\end{array}.
		\end{split}\end{equation}
		We also have the 'trace' derivative bounds (for $l\geq 1$ arbitrary)
		\begin{equation}\begin{split}
				\big|(\partial_{\xi} + \partial_{\eta})^lF(\xi,\eta)\big|\lesssim_l \Bigg\{\begin{array}{cc}1 & \xi+\eta\leq 1\\ (\xi+\eta)^{-3}\cdot(1+|\xi - \eta|)^{-N} & \xi+\eta>1\end{array}.
		\end{split}\end{equation}
		For any $p\in (1, \infty)$, the operator $\mathcal{K}$ acts in bounded fashion on $L^p_{\rho\,d\eta}$ as well as on $L^p_{d\eta}$. 
	\end{lem}
	Using the notation $\hat{z}: = \mathcal{F}(z)$ for simplicity, we can then reformulate \eqref{eq:schrodingermodel} in the following way:
	\begin{equation}\label{eq:zFourier1}
		-i\big(\partial_{\tau} - (2+i\tilde{c}\alpha_0)\frac{\lambda_{\tau}}{\lambda} - \frac{\lambda_{\tau}}{\lambda}\xi\partial_{\xi}+\frac{\lambda_{\tau}}{\lambda}\mathcal{K}\big)\hat{z}(\xi) - \xi^2\hat{z}(\xi) = \mathcal{F}\big(E\big), 
	\end{equation}
	for suitable $\tilde{c}\in \R$. 
	In fact, due to the rapid polynomial decay of the functions we are working with, we shall be able to move the transference operator part to the right hand side. Thus we shall instead consider 
	\begin{equation}\label{eq:zFourier2}
		-i\big(\partial_{\tau} - (2+i\tilde{c}\alpha_0)\frac{\lambda_{\tau}}{\lambda} - \frac{\lambda_{\tau}}{\lambda}\xi\partial_{\xi}\big)\hat{z}(\xi) - \xi^2\hat{z}(\xi) = -i\frac{\lambda_{\tau}}{\lambda}\mathcal{K}\hat{z}(\xi) + \mathcal{F}\big(E\big), 
	\end{equation}
	
	\subsection{The linear propagator for the Schr\"odinger operator on the distorted Fourier side}\label{subsec:Schrodingerpropagator}
	
	Here we study the precise solution of 
	\begin{equation}\label{eq:simplelinear}
		-i\big(\partial_{\tau} - (2+i\tilde{c}\alpha_0)\frac{\lambda_{\tau}}{\lambda} - \frac{\lambda_{\tau}}{\lambda}\xi\partial_{\xi}\big)\hat{z}(\tau, \xi) - \xi^2\hat{z}(\tau, \xi) = G(\tau, \xi). 
	\end{equation}
	In fact, we have 
	\begin{prop}\label{prop:linpropagator} The solution of \eqref{eq:simplelinear} vanishing at $\tau = +\infty$ is given by 
		\begin{equation}\label{eq:inhomprop}
			\hat{z}(\tau, \xi) =( -i)\int_{\tau}^\infty\big( \frac{\lambda(\tau)}{\lambda(\sigma)}\big)^{\theta}\cdot e^{i\lambda^2(\tau)\xi^2\cdot \int_{\sigma}^{\tau}\lambda^{-2}(s)\,ds}\cdot G\big(\sigma, \frac{\lambda(\tau)}{\lambda(\sigma)}\xi\big)\,d\sigma,\,\theta = 2+i\tilde{c}\alpha_0.
		\end{equation}
		We shall also write the right hand expression for simplicity as 
		\begin{align*}
			\int_\tau^\infty S(\tau, \sigma, \xi)\cdot G\big(\sigma, \frac{\lambda(\tau)}{\lambda(\sigma)}\xi\big)\,d\sigma =: S(G)(\tau,\xi).
		\end{align*}
		and use the notation $S_1 = \Re(iS), S_2 = \Im(iS)$ for later reference. 
	\end{prop}
	\begin{proof} Direct computation. 
	\end{proof}
	
	In order to derive basic bounds for the preceding propagator in terms of the source term, we state the following simple
	\begin{lem}\label{lem:basicL2} Assuming $N = N(\nu)$ is sufficiently large, we have the bound 
		\begin{align*}
			\Big\|\big\|\hat{z}\big\|_{L^2_{\rho\,d\xi}}\Big\|_{\tau^{1-N} L^2_{d\tau}}\lesssim\frac{1}{\sqrt{N}}\cdot \Big\|\big\|G\big\|_{L^2_{\rho\,d\xi}}\Big\|_{\tau^{-N}L^2_{d\tau}}. 
		\end{align*}
	\end{lem}
	\begin{proof} This is a consequence of Schur's criterion, since we have the bounds
		\begin{align*}
			\big\|\big(\frac{\sigma}{\tau}\big)^C\cdot \chi_{\tau\leq \sigma}\tau^{N-1}\sigma^{-N}\big\|_{L_{\sigma}^{\infty}L^1_{d\tau}}\lesssim \frac{1}{N},\,\big\|\big(\frac{\sigma}{\tau}\big)^C\cdot\chi_{\tau\leq \sigma}\tau^{N-1}\sigma^{-N}\big\|_{L_\tau^{\infty}L^1_{d\sigma}} \lesssim \frac{1}{N}. 
		\end{align*}
		for fixed $C$ and $N\geq N_*(C)$ sufficiently large. 
		
	\end{proof}

	\subsection{A formal expansion of the exact solution of \eqref{eq:zFourier2}} In the preceding subsection we have given the propagator corresponding to the left hand side of \eqref{eq:zFourier2}, neglecting the delicate term 
	\[
	-i\frac{\lambda_{\tau}}{\lambda}\mathcal{K}\hat{z}(\xi)
	\]
	on the right hand side. Observe that the latter is still linear in the variable $\hat{z}$, but comes with a decaying factor $\frac{\lambda_{\tau}}{\lambda}\sim \tau^{-1}$. The latter compensates for the integration in the propagator $S$ in Prop.~\ref{prop:linpropagator}. We can then formally write the solution of \eqref{eq:zFourier2} in the form
	\begin{equation}\label{eq:formalexpansion}
		\hat{z}(\tau,\xi) = \sum_{j=0}^\infty S\big(-i\frac{\lambda_{\tau}}{\lambda}\mathcal{K}\circ S\big)^j\big(\mathcal{F}(E)\big)
	\end{equation}
	The convergence of this sum in the sense of the norms of Lemma~\ref{lem:basicL2}  follows easily (recalling also Lemma~\ref{lem:KboundsKST} ) by picking $N$ sufficiently large (deppending on $\nu$). This simple estimate results in a loss of one power of $\tau$, however, and in the sequel we shall have to rely on more sophisticated estimates in order to avoid such a loss.
	In the future, it shall be important to split the right hand side of \eqref{eq:formalexpansion} into a principal part and an error part, namely 
	\begin{equation}\label{eq:ScalK1}
		\hat{z}(\tau,\xi) = S\big(\mathcal{F}(E)\big) + S_{\mathcal{K}}\big(\mathcal{F}(E)\big),
	\end{equation}
	where we set 
	\begin{equation}\label{eq:ScalK2}
		S_{\mathcal{K}}\big(\mathcal{F}(E)\big): =  \sum_{j=1}^\infty S\big(-i\frac{\lambda_{\tau}}{\lambda}\mathcal{K}\circ S\big)^j\big(\mathcal{F}(E)\big).
	\end{equation}

	\subsection{Some useful identities}
	
	Recalling \eqref{eq:kappa1eqn}, \eqref{eq:kappa2eqn}, it is useful to have explicit formulae for the source terms $\Im\big(\mathcal{L}z\big)|_{R = 0}, \Re\big(\mathcal{L}z\big)|_{R = 0}$. This we can easily accomplish by means of the Fourier propagator described in the preceding. In fact, assume that $E(\tau, R)$ is a {\it{real-valued}} source term in \eqref{eq:zeqn2}, such as $\lambda^{-2}y_z\cdot W$ (which here we interpret as a source term). Application of the propagator \eqref{eq:inhomprop} to the source term $G(\tau, \xi) = \mathcal{F}\big(E(\tau, \cdot)\big)(\xi)\in \R$ and reversing the distorted Fourier transform results in 
	\begin{align*}
		z(\tau, R) := \int_\tau^\infty \int_0^\infty \phi(R;\xi)\cdot S(\tau, \sigma, \xi)\cdot \mathcal{F}(E)\big(\sigma, \frac{\lambda(\tau)}{\lambda(\sigma)}\xi\big)\,\rho(\xi)\,d\xi d\sigma. 
	\end{align*}
	This implies that 
	\begin{equation}\label{eq:ImzR0}\begin{split}
			&\Im \mathcal{L}z|_{R = 0}\\& = -\int_\tau^\infty \int_0^\infty \big(\frac{\lambda(\tau)}{\lambda(\sigma)}\big)^2\xi^2\cdot \cos\big(\lambda^2(\tau)\xi^2\int_{\sigma}^{\tau}\lambda^{-2}(s)\,ds\big)\cdot \mathcal{F}(E)\big(\sigma, \frac{\lambda(\tau)}{\lambda(\sigma)}\xi\big)\,\rho(\xi)\,d\xi d\sigma
	\end{split}\end{equation}
	provided we set $\alpha_0 = 0$. \\
	Similarly, we infer 
	\begin{equation}\label{eq:ReR0}\begin{split}
			&\Re \mathcal{L}z|_{R = 0}\\& = \int_\tau^\infty \int_0^\infty \big(\frac{\lambda(\tau)}{\lambda(\sigma)}\big)^2\xi^2\cdot \sin\big(\lambda^2(\tau)\xi^2\int_{\sigma}^{\tau}\lambda^{-2}(s)\,ds\big)\cdot \mathcal{F}(E)\big(\sigma, \frac{\lambda(\tau)}{\lambda(\sigma)}\xi\big)\,\rho(\xi)\,d\xi d\sigma
	\end{split}\end{equation}
	
	\subsection{The linear propagator of the wave equation in the scaled coordinates}\label{subse:waveparametrix}\label{subsec:subsec:standardFourieronR4}
	
	Here we provide details as to how to control the propagator $\Box^{-1}F$, which by definition is the solution of the inhomogeneous wave equation $\Box u = F$ and which vanishes toward $\tau = +\infty$, it being given that the sources $F$ always vanish polynomially towards $\tau = +\infty$. In fact, rather than work with the Schr\"odinger time $\tau = \int_t^\infty \lambda^2(s)\,ds$, we work with the wave time $\tilde{\tau} = \int_t^\infty \lambda(s)\,ds$. As the elliptic part of $\Box $ is simply the standard $\triangle_{\R^4}$ in the radial setting, given (when re-scaled by $\lambda^{-2}$) by 
	\begin{align*}
		\partial_R^2 + \frac{3}{R}\partial_R, 
	\end{align*}
	we first exhibit the corresponding Fourier base $\phi_{\R^4}(R;\xi)$. From 
	\begin{align*}
		\big(\partial_R^2 + \frac{3}{R}\partial_R\big)\phi_{\R^4}(R;\xi) = -\xi^2\phi_{\R^4}(R;\xi), 
	\end{align*}
	we infer 
	\begin{align*}
		\big(-\partial_R^2 + \frac{3}{4R^2}\big)\big(R^{\frac32}\phi_{\R^4}\big)(R;\xi) = \xi^2R^{\frac32}\phi_{\R^4}(R;\xi). 
	\end{align*}
	From \cite{KST2} we can set 
	\[
	R^{\frac32}\phi_{\R^4}(R;\xi) = \xi^{-1}R^{\frac12}J_1(R\xi), 
	\]
	with associated spectral density $\rho_{\R^4}(\xi) = c\xi^3$ for suitable $c\in \mathbb{R}_+$. In particular, there is no transference operator since 
	\[
	(R\partial_R)\phi_{\R^4}(R;\xi) = (\xi\partial_{\xi})\phi_{\R^4}(R;\xi).
	\]
	In terms of the coordinates $(\tilde{\tau}, R)$, the equation $\Box n = F$ is transformed into 
	\begin{equation}\label{eq:nFintildetauR}
		-\big(\partial_{\tilde{\tau}}+ \frac{\lambda_{\tilde{\tau}}}{\lambda}R\partial_R\big)^2 n- \frac{\lambda_{\tilde{\tau}}}{\lambda}\big(\partial_{\tilde{\tau}}+ \frac{\lambda_{\tilde{\tau}}}{\lambda}R\partial_R\big)n + \big(\partial_R^2 + \frac{3}{R}\partial_R\big)n = \lambda^{-2}F. 
	\end{equation}
	Using the Fourier representation of $n$ with $\hat{n}(\tilde{\tau}, \xi) :=\mathcal{F}_{\R^4}(n):= x(\tilde{\tau},\xi)$, say, we obtain 
	\begin{align*}
		R\partial_{R}n &= c\int_0^\infty (R\partial_R)\phi(R;\xi) x(\xi)\xi^3\,d\xi\\
		& = c\int_0^\infty (\xi\partial_{\xi})\phi(R;\xi) x(\xi)\xi^3\,d\xi\\
		& = -c\int_0^\infty \phi(R;\xi) (\xi\partial_{\xi} + 4)x(\xi)\xi^3\,d\xi\\
	\end{align*}
	Thus in terms of the Fourier coefficients, introducing the dilation type operator 
	\begin{align*}
		\tilde{\mathcal{D}}_{\tau}: = \partial_{\tilde{\tau}} - \frac{\lambda_{\tilde{\tau}}}{\lambda}(\xi\partial_{\xi} + 4), 
	\end{align*}
	the preceding wave equation is transformed into (with $ \beta_{\tilde{\tau}} =  \frac{\lambda_{\tilde{\tau}}}{\lambda}$)
	\begin{align*}
		\tilde{\mathcal{D}}_{\tilde{\tau}}^2x(\tilde{\tau},\xi) + \beta_{\tilde{\tau}}\tilde{\mathcal{D}}_{\tilde{\tau}}x + \xi^2 x(\tilde{\tau}, \xi) = -\lambda^{-2}\mathcal{F}_{\R^4}(F)(\tilde{\tau},\xi)
	\end{align*}
	The inhomogeneous propagator for the operator on the left can be written down explicitly as follows: 
	\begin{equation}\label{eq:wavepropagator}\begin{split}
			x(\tilde{\tau},\xi) &= \int_{\tilde{\tau}}^\infty \frac{\lambda^3(\tilde{\tau})}{\lambda^3(\tilde{\sigma})}\cdot \frac{\sin\big[\lambda(\tilde{\tau})\xi\int_{\tilde{\tau}}^{\tilde{\sigma}}\lambda^{-1}(s)\,ds\big]}{\xi}\cdot \lambda^{-2}(\tilde{\sigma})\mathcal{F}_{\R^4}(F)(\tilde{\sigma},\frac{\lambda(\tilde{\tau})}{\lambda(\tilde{\sigma})}\xi)\,d\tilde{\sigma}\\
			&=:  \int_{\tilde{\tau}}^\infty U(\tilde{\tau}, \tilde{\sigma}, \xi)\cdot \lambda^{-2}(\tilde{\sigma})\mathcal{F}_{\R^4}(F)(\tilde{\sigma},\frac{\lambda(\tilde{\tau})}{\lambda(\tilde{\sigma})}\xi)\,d\tilde{\sigma},\\
	\end{split}\end{equation}
	corresponding to the 'physical function' 
	\begin{equation}\label{eq:nflatfourierrepresent}
		n(\tilde{\tau}, R) = \int_0^\infty \phi_{\R^4}(R;\xi)x(\tilde{\tau}, \xi)\rho_{\R^4}(\xi)\,d\xi. 
	\end{equation}
	To control the size of the propagator, we have 
	\begin{lem}\label{lem:wavebasicinhom} With the preceding notations, we have the bound 
		\begin{align*}
			\Big\|\big\|\langle R\rangle^{-1-\delta_0}n\big\|_{L^2_{R^3\,dR}}\Big\|_{L^2_{\tau^{-N}\,d\tau}}&\lesssim_{\delta_0}\Big\|\big\|\lambda^{-2}\xi^{-1-\delta_0}\langle\xi\rangle^{\delta_0}\mathcal{F}_{\R^4}(F)\big\|_{L^2_{\rho_{\R^4}\,d\xi}}\Big\|_{L^2_{\tau^{-N}\,d\tau}}\\
			& + \Big\|\big\|\lambda^{-2}\langle\xi\rangle^{-1}\langle\partial_{\xi}\rangle^{1+\delta_0}\mathcal{F}_{\R^4}(F)\big\|_{L^2_{\rho_{\R^4}\,d\xi}}\Big\|_{L^2_{\tau^{-N}\,d\tau}}.\\
		\end{align*}
		Further, denoting by $\tilde{\tau}$ the 'wave time', we have the estimate 
		\begin{align*}
			\Big\|\big\|\langle R\rangle^{-1-\delta_0}\partial_{\tilde{\tau}}n\big\|_{L^2_{R^3\,dR}}\Big\|_{L^2_{\tau^{-N}\,d\tau}}&\lesssim_{\delta_0}\Big\|\big\|\lambda^{-2}\xi^{-\delta_0}\langle\xi\rangle^{\delta_0}\mathcal{F}_{\R^4}(F)\big\|_{L^2_{\rho_{\R^4}\,d\xi}}\Big\|_{L^2_{\tau^{-N}\,d\tau}}\\
			& + \Big\|\big\|\lambda^{-2}\langle\partial_{\xi}\rangle^{1+\delta_0}\mathcal{F}_{\R^4}(F)\big\|_{L^2_{\rho_{\R^4}\,d\xi}}\Big\|_{L^2_{\tau^{-N}\,d\tau}}.\\
		\end{align*}
		In particular we have the estimate 
		\begin{align*}
			\Big\|\big\|\langle R\rangle^{-1-\delta_0}\partial_{\tau}n\big\|_{L^2_{R^3\,dR}}\Big\|_{L^2_{\tau^{-N-\frac12-\frac{1}{4\nu}}\,d\tau}}&\lesssim_{\delta_0}\Big\|\big\|\lambda^{-2}\xi^{-\delta_0}\langle\xi\rangle^{\delta_0}\mathcal{F}_{\R^4}(F)\big\|_{L^2_{\rho_{\R^4}\,d\xi}}\Big\|_{L^2_{\tau^{-N}\,d\tau}}\\
			& + \Big\|\big\|\lambda^{-2}\langle\partial_{\xi}\rangle^{1+\delta_0}\mathcal{F}_{\R^4}(F)\big\|_{L^2_{\rho_{\R^4}\,d\xi}}\Big\|_{L^2_{\tau^{-N}\,d\tau}}.\\
		\end{align*}
		We also have the estimates 
		\begin{align*}
			&\Big\|\big\|n\big\|_{L^2_{R^3\,dR}}\Big\|_{L^2_{\tau^{-N}\,d\tau}}\lesssim \Big\|\tilde{\tau}\cdot\big\|\lambda^{-2}\xi^{-1}\mathcal{F}_{\R^4}(F)\big\|_{L^2_{\rho_{\R^4}\,d\xi}}\Big\|_{L^2_{\tau^{-N}\,d\tau}}\\
			&\Big\|\big\|n\big\|_{\dot{H}^1_{R^3\,dR}}\Big\|_{L^2_{\tau^{-N}\,d\tau}}\lesssim \Big\|\tilde{\tau}\cdot\big\|\lambda^{-2}\mathcal{F}_{\R^4}(F)\big\|_{L^2_{\rho_{\R^4}\,d\xi}}\Big\|_{L^2_{\tau^{-N}\,d\tau}}\\
			&\Big\|\big\|n\big\|_{L^2_{R^3\,dR}}\Big\|_{L^2_{\tau^{-N}\,d\tau}}\lesssim \Big\|\tilde{\tau}^2\cdot\big\|\lambda^{-2}\mathcal{F}_{\R^4}(F)\big\|_{L^2_{\rho_{\R^4}\,d\xi}}\Big\|_{L^2_{\tau^{-N}\,d\tau}}.\\
		\end{align*}
		Recall that $\tau$ is the 'Schr\"odinger time while $\tilde{\tau}$ is the 'wave time'. 
	\end{lem}
	\begin{rem}\label{rem:wavepropagatorpartialtau} We stress the important temporal decay improving feature of the third estimate which is due to the difference between $\tau$ and $\tilde{\tau}$. 
	\end{rem}
	\begin{rem}\label{rem:wavepropagatormorerefined} The following proof in fact reveals that we can write $n = n_1 + n_2$ such that 
		\begin{align*}
			&\Big\|\big\|\langle R\rangle^{-1-\delta_0}n_1\big\|_{L^2_{R^3\,dR}}\Big\|_{L^2_{\tau^{-N}\,d\tau}}\lesssim_{\delta_0}\Big\|\big\|\lambda^{-2}\xi^{-1-\delta_0}\langle\xi\rangle^{\delta_0}\mathcal{F}_{\R^4}(F)\big\|_{L^2_{\rho_{\R^4}\,d\xi}}\Big\|_{L^2_{\tau^{-N}\,d\tau}},\\&
			\Big\|\big\|n_2\big\|_{L^2_{R^3\,dR}}\Big\|_{L^2_{\tau^{-N}\,d\tau}}\lesssim_{\delta_0} \Big\|\big\|\lambda^{-2}\langle\xi\rangle^{-1}\langle\partial_{\xi}\rangle^{1+\delta_0}\mathcal{F}_{\R^4}(F)\big\|_{L^2_{\rho_{\R^4}\,d\xi}}\Big\|_{L^2_{\tau^{-N}\,d\tau}}.
		\end{align*}
	\end{rem}
	\begin{proof} Observe that 
		\begin{align*}
			\lambda(\tilde{\tau}, \tilde{\sigma}): = \lambda(\tilde{\tau})\cdot \int_{\tilde{\tau}}^{\tilde{\sigma}}\lambda^{-1}(s)\,ds\sim_{\nu} \tilde{\sigma} - \tilde{\tau}
		\end{align*}
		for $\tilde{\tau}\leq \tilde{\sigma}\lesssim \tilde{\tau}$ and $\lambda(\tilde{\tau}, \tilde{\sigma})\sim \tilde{\tau}$ for $\tilde{\sigma}\gg \tilde{\tau}$. Decompose 
		\begin{align*}
			&\int_0^\infty \phi_{\R^4}(R;\xi)x(\tilde{\tau}, \xi)\rho_{\R^4}(\xi)\,d\xi\\&
			=  \int_0^\infty \int_{\tilde{\tau}}^\infty \phi_{\R^4}(R;\xi)U(\tilde{\tau}, \tilde{\sigma}, \xi)\cdot \lambda^{-2}(\tilde{\sigma})\mathcal{F}_{\R^4}(F)(\tilde{\sigma},\frac{\lambda(\tilde{\tau})}{\lambda(\tilde{\sigma})}\xi)\rho_{\R^4}(\xi)\,d\tilde{\sigma}d\xi\\
			& = \sum_{j=1}^3 n_j(\tilde{\tau}, R), 
		\end{align*}
		where we define 
		\begin{equation}\label{eq:n1first}\begin{split}
				&n_1(\tilde{\tau}, R) : \\&= \int_0^\infty \int_{\tilde{\tau}}^\infty \chi_{R\ll \tilde{\sigma}-\tilde{\tau}}\phi_{\R^4}(R;\xi)U(\tilde{\tau}, \tilde{\sigma}, \xi)\cdot \lambda^{-2}(\tilde{\sigma})\mathcal{F}_{\R^4}(F)(\tilde{\sigma},\frac{\lambda(\tilde{\tau})}{\lambda(\tilde{\sigma})}\xi)\rho_{\R^4}(\xi)\,d\tilde{\sigma}d\xi
		\end{split}\end{equation}
		and $n_j(\tilde{\tau}, R), j = 2, 3$ are defined similarly by inclusion of cutoffs $ \chi_{R\sim \tilde{\sigma}-\tilde{\tau}},\,  \chi_{R\gg\tilde{\sigma}-\tilde{\tau}}$, respectively. Consider first the case $j = 2$. By inspection, we get 
		\begin{align*}
			&\Big|U(\tilde{\tau}, \tilde{\sigma}, \xi)\cdot \lambda^{-2}(\tilde{\sigma})\mathcal{F}_{\R^4}(F)(\tilde{\sigma},\frac{\lambda(\tilde{\tau})}{\lambda(\tilde{\sigma})}\xi)\Big|\\
			&\lesssim \big(\frac{\lambda(\tilde{\tau})}{\lambda(\tilde{\sigma})}\big)^4\cdot \lambda^{-2}(\tilde{\sigma})\Big|\Big(\xi^{-1}\mathcal{F}_{\R^4}(F)\Big)(\tilde{\sigma},\frac{\lambda(\tilde{\tau})}{\lambda(\tilde{\sigma})}\xi)\Big|,
		\end{align*}
		and so 
		\begin{align*}
			\Big\|U(\tilde{\tau}, \tilde{\sigma}, \xi)\cdot \lambda^{-2}(\tilde{\sigma})\mathcal{F}_{\R^4}(F)(\tilde{\sigma},\frac{\lambda(\tilde{\tau})}{\lambda(\tilde{\sigma})}\xi)\Big\|_{L^2_{\rho_{\R^4}\,d\xi}}\lesssim \big(\frac{\lambda(\tilde{\tau})}{\lambda(\tilde{\sigma})}\big)^2\cdot \Big\|\xi^{-1}\mathcal{F}_{\R^4}(\lambda^{-2}F)(\sigma, \cdot)\Big\|_{L^2_{\rho_{\R^4}\,d\xi}}. 
		\end{align*}
		Using Plancherel's theorem, we infer that 
		\begin{align*}
			\Big\|\langle R\rangle^{-1-\delta_0}n_2(\tau, R)\Big\|_{L^2_{R^3\,dR}}\lesssim  \int_{\tilde{\tau}}^\infty\langle\tilde{\sigma} - \tilde{\tau}\rangle^{-1-\delta_0}\cdot \Big\|\xi^{-1}\mathcal{F}_{\R^4}(\lambda^{-2}F)(\sigma, \cdot)\Big\|_{L^2_{\rho_{\R^4}\,d\xi}}\,d\tilde{\sigma}. 
		\end{align*}
		Setting $K(\tilde{\sigma}, \tilde{\tau}): = \chi_{\tilde{\sigma}\geq \tilde{\tau}}\frac{\tilde{\tau}^N}{\tilde{\sigma}^N}\cdot \langle\tilde{\sigma} - \tilde{\tau}\rangle^{-1-\delta_0}$ and applying Schur's test, we deduce 
		\begin{align*}
			\Big\|\big\|\langle R\rangle^{-1-\delta_0}n_2(\tau, R)\big\|_{L^2_{R^3\,dR}}\Big\|_{\tau^{-N}d\tau}\lesssim \Big\|\big\|\xi^{-1}\mathcal{F}_{\R^4}(\lambda^{-2}F)(\sigma, \cdot)\big\|_{L^2_{\rho_{\R^4}\,d\xi}}\Big\|_{\sigma^{-N}L^2_{d\sigma}}, 
		\end{align*}
		as desired. 
		\\
		
		Next, to estimate $n_1$, we need to perform integration by parts with respect to $\xi$, letting $\delta_0 = 1$ for now. Since $R\ll \tilde{\sigma} - \tilde{\tau}$, this leads to the {\it{schematically written}} relation 
		\begin{align*}
			&n_1(\tilde{\tau}, R) : \\&= \int_0^\infty \int_{\tilde{\tau}}^\infty \chi_{R\ll \tilde{\sigma}-\tilde{\tau}}\phi_{\R^4}(R;\xi)U(\tilde{\tau}, \tilde{\sigma}, \xi)\cdot \lambda^{-2}(\tilde{\sigma})\mathcal{F}_{\R^4}(F)(\tilde{\sigma},\frac{\lambda(\tilde{\tau})}{\lambda(\tilde{\sigma})}\xi)\rho_{\R^4}(\xi)\,d\tilde{\sigma}d\xi\\
			& = n_{11} + n_{12},
		\end{align*}
		where we set 
		\begin{align*}
			&n_{11}: = \int_0^\infty \int_{\tilde{\tau}}^\infty \chi_{R\ll \tilde{\sigma}-\tilde{\tau}}\phi_{\R^4}(R;\xi)\frac{U(\tilde{\tau}, \tilde{\sigma}, \xi)}{\xi^{1+\delta_0}(\tilde{\sigma} - \tilde{\tau})^{1+\delta_0}}\\&\hspace{6cm}\cdot \lambda^{-2}(\tilde{\sigma})\mathcal{F}_{\R^4}(F)(\tilde{\sigma},\frac{\lambda(\tilde{\tau})}{\lambda(\tilde{\sigma})}\xi)\rho_{\R^4}(\xi)\,d\tilde{\sigma}d\xi\\
			& n_{12}: =  \int_0^\infty \int_{\tilde{\tau}}^\infty \chi_{R\ll \tilde{\sigma}-\tilde{\tau}}\phi_{\R^4}(R;\xi)\frac{U(\tilde{\tau}, \tilde{\sigma}, \xi)}{(\tilde{\sigma} - \tilde{\tau})^{1+\delta_0}}\\&\hspace{4.5cm}\cdot \lambda^{-2}(\tilde{\sigma})\partial_{\xi}^{1+\delta_0}\Big(\mathcal{F}_{\R^4}(F)(\tilde{\sigma},\frac{\lambda(\tilde{\tau})}{\lambda(\tilde{\sigma})}\xi)\Big)\rho_{\R^4}(\xi)\,d\tilde{\sigma}d\xi\\
		\end{align*}
		Then arguing precisely as for the term $n_2$, we easily infer the desired bound in the {\it{high-frequency regime}} $\xi\gtrsim 1$ first with $\delta_0 = 1$, and then for $0<\delta_0<1$ via interpolation. On the other hand, for the low frequency regime, we use 
		\begin{align*}
			\big\|\langle R\rangle^{-1-\delta_0}\phi_{\R^4}(R;\xi)\big\|_{L^2_{R^3\,dR}}\lesssim \xi^{-1+\delta_0}\langle\log\xi\rangle
		\end{align*}
		for $\delta_0 = 1$, 
		and so 
		\begin{align*}
			&\Big\|\langle R\rangle^{-1-\delta_0}\int_0^\infty \int_{\tilde{\tau}}^\infty \chi_{\xi\lesssim1}\chi_{R\ll \tilde{\sigma}-\tilde{\tau}}\phi_{\R^4}(R;\xi)\frac{U(\tilde{\tau}, \tilde{\sigma}, \xi)}{\xi^{1+\delta_0}(\tilde{\sigma} - \tilde{\tau})^{1+\delta_0}}\\&\hspace{5cm}\cdot \lambda^{-2}(\tilde{\sigma})\mathcal{F}_{\R^4}(F)(\tilde{\sigma},\frac{\lambda(\tilde{\tau})}{\lambda(\tilde{\sigma})}\xi)\rho_{\R^4}(\xi)\,d\tilde{\sigma}d\xi\Big\|_{L^2_{R^3\,dR}}\\
			&\lesssim \int_0^\infty \int_{\tilde{\tau}}^\infty  \chi_{\xi\lesssim1}\cdot \xi^{-2+\delta_0-}\cdot \xi^{-1-\delta_0}\Big|\frac{\lambda^{-2}(\tilde{\sigma})\mathcal{F}_{\R^4}(F)(\tilde{\sigma},\frac{\lambda(\tilde{\tau})}{\lambda(\tilde{\sigma})}\xi)}{(\tilde{\sigma} - \tilde{\tau})^{1+\delta_0}}\Big|\rho_{\R^4}(\xi)\,d\tilde{\sigma}d\xi
		\end{align*}
		Using the Cauchy-Schwarz inequality with respect to $\xi$ (using that $\rho_{\R^4}(\xi)\sim \xi^3$) and using the Schur's criterion as for $n_0$, we conclude that 
		\begin{align*}
			\Big\|\big\|\langle R\rangle^{-1-\delta_0}P_{\lesssim 1}n_{11}\big\|_{L^2_{R^3\,dR}}\Big\|_{\tau^{-N}d\tau}\lesssim \Big\|\big\|\xi^{-1-\delta_0}\lambda^{-2}(\tilde{\sigma})\mathcal{F}_{\R^4}(F)(\tilde{\sigma},\cdot)\big\|_{L^2_{\rho\,d\xi}}\Big\|_{\sigma^{-N}L^2_{d\sigma}},
		\end{align*}
		with the case $0<\delta_0<1$ again following by interpolation. 
		The argument for $n_{12}$ is analogous. Finally, for the term $n_3$ it suffices to use $|R \pm (\tilde{\sigma} - \tilde{\tau})|\gg (\tilde{\sigma} - \tilde{\tau})$, whence the factor $\langle R\rangle^{-1-\delta_0}$ ensures time integrability. The remaining estimates are proved similarly, or in the case of the last three estimates, by direct application of Plancherel's theorem. 
	\end{proof}
	
	We shall also require the following more detailed structural result:
	\begin{lem}\label{lem:wavebasicinhomstructure} Using the same notation as before, we can write 
		\begin{align*}
			n(\tau, R) = \sum_{\pm} \int_0^\infty e^{\pm iR\eta}\cdot N_{\pm}(R,\eta;\tau)\,d\eta + n_2(\tau, R), 
		\end{align*}
		where upon setting 
		\[
		\Phi_{\pm}(R,\xi,\eta): =  \chi_{R\xi\gtrsim 1}\cdot \tilde{\sigma}(R,\xi)\cdot \frac{e^{i(\pm\xi\pm\eta)R}}{R^{\frac12}\xi}
		\]
		for a bounded function $\tilde{\sigma}(R,\xi)\in C^\infty(\R_+\times \R_+)$, with symbol behavior with respect to both arguments, we have the bounds
		\begin{align*}
			&\Big\|\int_0^\infty\Phi_{\pm}(R,\xi,\eta)\cdot W(R)\cdot N_{\pm}(R,\eta;\tau) R^3\,dR d\eta\Big\|_{\tau^{-N+} L^2_{d\tau} L^2_{d\xi}}\\&\lesssim 
			\big\|\triangle^{-1}\big(\lambda^{-2}F\big)\big\|_{\tau^{-N} L^2_{d\tau}L^2_{R^3\,dR}}, \\
			&\Big\|\int_0^\infty\xi^{2+}\Phi_{\pm}(R,\xi,\eta)\cdot W(R)\cdot N_{\pm}(R,\eta;\tau) R^3\,dR d\eta\Big\|_{\tau^{-N+} L^2_{d\tau} L^2_{d\xi}}\\&\lesssim 
			\big\|\triangle^{0+}\big(\lambda^{-2}F\big)\big\|_{\tau^{-N} L^2_{d\tau}L^2_{R^3\,dR}} 
		\end{align*}
		as well as 
		\begin{align*}
			&\Big\|W \cdot n_2(\tau, R)\Big\|_{\tau^{-N+} L^2_{d\tau}L^1_{R^3\,dR}}\lesssim \big\|\triangle^{-1}\nabla\big(\langle R\rangle\cdot\lambda^{-2}F\big)\big\|_{\tau^{-N} L^2_{d\tau}L^2_{R^3\,dR}} +  \big\|\triangle^{-1}\big(\lambda^{-2}F\big)\big\|_{\tau^{-N} L^2_{d\tau}L^2_{R^3\,dR}}\\
			&\Big\|\nabla_R^{1+}\big(R^{\frac52+}W\cdot n_2\big)\Big\|_{\tau^{-N+} L^2_{d\tau}L^2_{dR}}\lesssim 
			\big\| \triangle^{0+}\big(\langle R\rangle\cdot\lambda^{-2}F\big)\big\|_{\tau^{-N} L^2_{d\tau}L^2_{R^3\,dR}} +  \big\|\triangle^{-1}\big(\lambda^{-2}F\big)\big\|_{\tau^{-N} L^2_{d\tau}L^2_{R^3\,dR}} \\&\hspace{5cm}
			+ \big\|\triangle^{-1}\nabla\big(\langle R\rangle\cdot\lambda^{-2}F\big)\big\|_{\tau^{-N} L^2_{d\tau}L^2_{R^3\,dR}}
		\end{align*}
		As a consequence of the preceding bounds, we infer that for $\kappa\in \{0, 2\}$, we have 
		\begin{align*}
			&\Big\|\langle \xi\partial_{\xi}\rangle^{1+\delta_1}\big(\chi_{R\xi\gtrsim 1}\big(n\cdot W\big), \xi^{-\kappa}\phi(R;\xi)\rangle_{L^2_{R^3\,dR}}\Big\|_{\tau^{-N+} L^2_{d\tau}L^2_{\rho(\xi)\,d\xi}}\\&\lesssim 
			\big\| \triangle^{0+}\big(\langle R\rangle\cdot\lambda^{-2}F\big)\big\|_{\tau^{-N} L^2_{d\tau}L^2_{R^3\,dR}} +  \big\|\triangle^{-1}\big(\lambda^{-2}F\big)\big\|_{\tau^{-N} L^2_{d\tau}L^2_{R^3\,dR}} \\&\hspace{5cm}
			+ \big\|\triangle^{-1}\nabla\big(\langle R\rangle\cdot\lambda^{-2}F\big)\big\|_{\tau^{-N} L^2_{d\tau}L^2_{R^3\,dR}}
		\end{align*}
		The terms $\langle R\rangle\cdot \lambda^{-2}F$ can be replaced by $\tilde{\tau}\cdot \lambda^{-2}F$ on the right hand side. 
	\end{lem}
	
	The technical proof is relegated to section~\ref{sec:appendix}. 
	\begin{rem}\label{rem:wavebasicinhomstructure} By splitting into the cases $\xi<\tau^{-M}, \xi>\tau^{-M}$ for some large $M$ and using interpolation as well as Lemma~\ref{lem:wavebasicinhom} in the latter regime, we can replace 
		\[
		L^2_{\rho(\xi)\,d\xi}
		\]
		by 
		\[
		L^{2+}_{\rho(\xi)\,d\xi}
		\]
		in the last inequality.
	\end{rem}
	We shall use the preceding lemma for a very specific source term $F$, namely the one figuring in the following 
	\begin{lem}\label{lem:specialF1} Let 
		\[
		F = \lambda^2\triangle\Re\big(W\cdot \bar{z}\big).
		\]
		Then defining $n = \lambda^{-2}\Box^{-1}F$ (as usual via the Duhamel parametrix), we can write
		\begin{align*}
			n(\tau, R) = \sum_{\pm} \int_0^\infty e^{\pm iR\eta}\cdot N_{\pm}(R,\eta;\tau)\,d\eta + n_2(\tau, R), 
		\end{align*}
		where $ \sum_{\pm} \int_0^\infty e^{\pm iR\eta}\cdot N_{\pm}(R,\eta;\tau)\,d\eta , n_2$ satisfy the same bounds as in the preceding lemma and remark but with $\tau^{-N+}$ replaced by $\tau^{-N+\frac12+}$ and the right hand side replaced by $\big\|z\big\|_{S}$.
	\end{lem}
	\begin{proof} We decompose 
		\[
		F = F_1 + F_2
		\]
		where we set 
		\begin{align*}
			F_1: =  \lambda^2\triangle\Re\big(\chi_{R\lesssim \tau}W\cdot \bar{z}\big),\,F_2: =  \lambda^2\triangle\Re\big(\chi_{R\gtrsim \tau}W\cdot \bar{z}\big)
		\end{align*}
		We first deal with the contribution of $F_2$. Write 
		\begin{align*}
			\nabla\Re\big(\chi_{R\gtrsim \tau}W\cdot \bar{z}\big) = \Re\big(\nabla(\chi_{R\gtrsim \tau}W)\cdot \bar{z}\big) + \Re\big(\chi_{R\gtrsim \tau}W\cdot \nabla\bar{z}\big).
		\end{align*}
		From the definition \eqref{eq:Snormdefi} we infer the inclusion
		\begin{align*}
			\chi_{R\gtrsim \tau}W\nabla\bar{z}\in \tau^{-N-\frac12+\delta_0+}L^2_{R^3\,dR}
		\end{align*}
		and so 
		\begin{align*}
			\Big\|\tilde{\tau}\Re\big(\chi_{R\gtrsim \tau}W\cdot \nabla\bar{z}\big)\Big\|_{\tau^{-N-\frac{1}{4\nu}+\delta_0+}L^2_{d\tau}L^2_{R^3\,dR}}\lesssim \big\|z\big\|_{S}. 
		\end{align*}
		An even better bound (without the $\delta_0$) obtains for the term 
		\[
		\Re\big(\nabla(\chi_{R\gtrsim \tau}W)\cdot \bar{z}\big).
		\]
		If we then spell out $\Box^{-1}F_2$ using the Fourier parametrix, write $\triangle = \nabla\cdot\nabla$ and use one operator $\nabla$ to counteract the inverse frequency in $U(\tilde{\tau}, \tilde{\sigma}, \eta)$ (recall the preceding proof), we easily check that (here $\Box^{-1}$ is in the sense of applying the Duhamel parametrix \eqref{eq:wavepropagator})
		\begin{align*}
			&\Big\|\chi_{R\lesssim \tau}W(R)\cdot \Box^{-1}F_2\Big\|_{\tau^{-N}L^2_{d\tau}L^1_{R^3\,dR}}\lesssim \big\|z\big\|_{S}, \\
			&\Big\|\nabla_R^{1+}\big(R^{\frac52+}\chi_{R\lesssim \tau}W(R)\cdot \Box^{-1}F_2\big)\Big\|_{\tau^{-N}L^2_{d\tau}L^2_{R^3\,dR}}\lesssim \big\|z\big\|_{S}
		\end{align*}
		whence we can place $\chi_{R\lesssim \tau}\lambda^{-2}\Box^{-1}F_2$ into $n_2$. For the remaining term $\chi_{R\gtrsim\tau}\Box^{-1}F_2$, assuming that we pass to the radial variable $R_1$ to describe $F_2$ and we have $R_1\gtrsim R$, we can modify the gain of $\tau^{-\frac12+\delta_0+}$ above slightly to also gain $R_1^{0-}\lesssim R^{0-}$, ensuring that $W(R)\cdot R^{0-}\in L^2_{R^3\,dR}$, and the above bounds again hold. In case $R_1\ll R$,   we can proceed as for the term $n_{IV}$ in the proof of the preceding lemma in section~\ref{sec:appendix}, since integration by parts with respect to the frequency variable gains $R^{-1}\lesssim R^{-\frac12+\frac{1}{4\nu}}\cdot\tilde{\tau}^{-1}$, and we can use that $W(R)\cdot  R^{-\frac12+\frac{1}{4\nu}}\in L^2_{R^3\,dR}$, while the extra $\tilde{\tau}^{-1}$ compensates for the time integral in the Duhamel propagator.\\
		As for the contribution of $F_1$, we shall use the preceding Lemma~\ref{lem:wavebasicinhomstructure}. In fact, we easily verify that 
		\begin{align*}
			&\big\|\triangle^{-1}\lambda^{-2}F_1\big\|_{\tau^{-N+}L^2_{d\tau}L^2_{R^3\,dR}}\lesssim \big\|z\big\|_{S},\,\big\|\triangle^{-1}\nabla\big(\langle R\rangle\cdot \lambda^{-2}F_1\big\|_{\tau^{-N+\frac12}L^2_{d\tau}L^2_{R^3\,dR}}\lesssim \big\|z\big\|_{S},\\
			&\big\|\triangle^{0+}\big(\langle R\rangle\cdot \lambda^{-2}F_1\big)\big\|_{\tau^{-N+\frac12}L^2_{d\tau}L^2_{R^3\,dR}}\lesssim \big\|z\big\|_{S}.
		\end{align*}
	\end{proof}
	
	A more basic estimate related to the term $F$ from the preceding lemma is the following, which we call a corollary due to its using the same proof ingredients:
	\begin{cor}\label{cor:yzW}  With $F$ as in the preceding lemma, we have the estimates
		\begin{align*}
			&\Big\|\langle \xi\partial_{\xi}\rangle^{1+\delta_0}\langle \lambda^{-2}\Box^{-1}(F)\cdot W, \phi(R;\xi)\rangle_{L^2_{R^3\,dR}}\Big\|_{\tau^{-N+}L^2_{d\tau}L^2_{\rho(\xi)\,d\xi}}\lesssim \big\|z\big\|_{S}\\
			&\Big\|\langle \xi\partial_{\xi}\rangle^{1+\delta_0}\langle \lambda^{-2}\Box^{-1}(F)\cdot W, \phi(R;\xi)\rangle_{L^2_{R^3\,dR}}\Big\|_{\tau^{-N+}L^2_{d\tau}L^\infty_{\rho(\xi)\,d\xi}}\lesssim \big\|z\big\|_{S}\\
		\end{align*}
	\end{cor}
	\begin{proof} As in the preceding proof we split $F = F_1 + F_2$. The contribution of $F_2$ is easy to handle since we already saw that $\lambda^{-2}\Box^{-1}(F_2)\in \tau^{-N}L^2_{d\tau}L^2_{R^3\,dR}$, and the operator $\langle \xi\partial_{\xi}\rangle^{1+\delta_0}$ 'costs'$ \langle R\rangle^{1+}$ which gets more than absorbed by the factor $W(R)$. For the contribution of $F_1$ we write 
		\begin{align*}
			\lambda^{-2}F_1 &= \triangle \Re\big(\chi_{R\lesssim\tau}W\cdot \bar{z}\big) = [\triangle(\chi_{R\lesssim\tau}W)\cdot \bar{z} + 2\nabla(\chi_{R\lesssim\tau}W)\cdot \nabla\bar{z}] + \chi_{R\lesssim\tau}W\cdot \triangle\bar{z}\\
			& =: F_{11} + F_{12}, 
		\end{align*}
		From Lemma~\ref{lem:wavebasicinhom}  and more precisely from Remark~\ref{rem:wavepropagatormorerefined} we infer that $\lambda^{-2}\Box^{-1}F_{11} = n_{11}+ n_{12}$ where using Sobolev's embedding we get 
		\begin{align*}
			&\big\|\langle R\rangle^{-1-\delta_0}n_{11}\big\|_{\tau^{-N}L^2_{d\tau}L^2_{R^3\,dR}}\lesssim \big\|F_{11}\big\|_{\tau^{-N}L^2_{d\tau}L^{\frac43-}_{R^3\,dR}}\lesssim \big\|z\big\|_{S},\\
			&\big\|n_{12}\big\|_{\tau^{-N}L^2_{d\tau}L^2_{R^3\,dR}}\lesssim \big\|\langle R\rangle^{1+\delta_0}F_{11}\big\|_{\tau^{-N}L^2_{d\tau}L^{2}_{R^3\,dR}}\lesssim \big\|z\big\|_{S}. 
		\end{align*}
		For the remaining term $F_{12}$ using Lemma~\ref{lem:wavebasicinhom} and the Sobolev embedding we have the estimate 
		\begin{align*}
			\big\|\lambda^{-2}\Box^{-1}F_{12}\big\|_{\tau^{-N-\frac{1}{2\nu}+}L^2_{d\tau}L^4_{R^3\,dR}}\lesssim \big\|\tilde{\tau}\cdot F_{12}\big\|_{\tau^{-N-\frac{1}{2\nu}+}L^2_{d\tau}L^2_{R^3\,dR}}\lesssim \big\|z\big\|_{S}. 
		\end{align*}
		Combining these bounds  and using the Plancherel's theorem we easily infer the estimate 
		\begin{align*}
			\Big\|\langle \xi\partial_{\xi}\rangle^{1+\delta_0}\langle \chi_{R\lesssim \tau^{100}}W(R)\cdot \lambda^{-2}\Box^{-1}F_1, \phi(R;\xi)\rangle_{L^2_{R^3\,dR}}\Big\|_{\tau^{-N+}L^2_{d\tau}L^2_{\rho(\xi)\,d\xi}}\lesssim \big\|z\big\|_{S}.
		\end{align*}
		On the other hand, using integration by parts with respect to the frequency in the Fourier representation of $\Box^{-1}F_1$ we get 
		\begin{align*}
			\big\|\chi_{R\gtrsim \tau^{100}}\lambda^{-2}\Box^{-1}F_1\big\|_{\tau^{-N}L^2_{d\tau}L^2_{R^3\,dR}}\lesssim \tau^{-99}\big\|\triangle^{-1}F_1\big\|_{\tau^{-N}L^2_{d\tau}L^2_{R^3\,dR}}\lesssim \tau^{-99+}\cdot\big\|z\big\|_{S},
		\end{align*}
		and from here 
		\begin{align*}
			\Big\|\langle \xi\partial_{\xi}\rangle^{1+\delta_0}\langle \chi_{R\gtrsim \tau^{100}}W(R)\cdot \lambda^{-2}\Box^{-1}F_1, \phi(R;\xi)\rangle_{L^2_{R^3\,dR}}\Big\|_{\tau^{-N-99+}L^2_{d\tau}L^2_{\rho(\xi)\,d\xi}}\lesssim \big\|z\big\|_{S}.
		\end{align*}
		This completes the proof of the first estimate. We relegate the proof of the second estimate to section~\ref{sec:appendix}. 
	\end{proof}
	\begin{rem}\label{rem:cor:yzW} The reason for the small loss of temporal decay expressed by the $\tau^{-N+}$ factor comes from the operator $\langle \xi\partial_{\xi}\rangle^{1+\delta_0}$. Replacing this by $\langle \xi\partial_{\xi}\rangle^{1-\delta_1}$, $\delta_1\gg\delta_0$, we obtain a better estimate with $\tau^{-N+}$ replaced by $\tau^{-N}$. There are also straightforward variations of the preceding corollary and its proof, such as 
		\[
		\big\|\lambda^{-2}\Box^{-1}F\cdot W\big\|_{\tau^{-N}L^2_{d\tau}L^{2-}_{R^3\,dR}}\lesssim \big\|z\big\|_{S}. 
		\]
	\end{rem}
	
	We complement the preceding with the following technical lemma, whose proof is also relegated to section~\ref{sec:appendix}:
	\begin{lem}\label{lem:yzWnonosc} Defining $F$ as in the preceding lemma, assume that $\psi(R;\xi)$ is a smooth and bounded function with symbol type bounds for its derivatives:
		\[
		\big\|\langle\xi\rangle^{2+}(\xi\partial_{\xi})^{l_1}(R\partial_R)^{\l_2}\psi\big\|_{L^\infty_{R,\xi}}\lesssim_{l_1,l_2} 1. 
		\]
		Then  we have the bound 
		\begin{align*}
			\big\|\langle \xi\partial_{\xi}\rangle^{1+\delta_0}\langle \int_0^\infty\Box^{-1}\big(F\big)\cdot W(R), \psi(R;\xi)\rangle_{L^2_{R^3\,dR}}\big\|_{\tau^{-N+}L^2_{d\tau}L^2_{\rho(\xi)\,d\xi}}\lesssim \big\|z\big\|_{S}. 
		\end{align*}
		One may also replace the norm $L^2_{\rho(\xi)\,d\xi}$ by $L^{2+}_{\rho(\xi)\,d\xi}$, under the weaker hypothesis $\big\|\langle\xi\rangle^{2}(\xi\partial_{\xi})^{l_1}(R\partial_R)^{\l_2}\psi\big\|_{L^\infty_{R,\xi}}\lesssim_{l_1,l_2} 1$. 
	\end{lem}
	By Schr\"odinger time differentiating the source term $F$, we obtain an improved bound analogous to the third estimate in Lemma~\ref{lem:wavebasicinhom}, with similar proof to the preceding:
	\begin{cor}\label{cor:yzWpartialtau} If $F$ is as in the preceding two lemmas, then we have the bound 
		\begin{align*}
			\big\|\langle \xi\partial_{\xi}\rangle^{1+\delta_1}\langle \partial_{\tau}\Box^{-1}(F)\cdot W,\,\phi(R;\xi)\rangle_{L^2_{R^3\,dR}}\big\|_{\tau^{-N-\frac12-\frac{1}{4\nu}+}L^2_{d\tau}L^2_{\rho(\xi)\,d\xi}}\lesssim \big\|z\big\|_{S}. 
		\end{align*}
	\end{cor}

	In the sequel we shall also need a slight variation of the preceding Lemma~\ref{lem:wavebasicinhomstructure} involving slightly different norms. The proof is entirely analogous:
	\begin{lem}\label{lem:wavebasicinhomstructure1} We have the bounds
		\begin{align*}
			&\big\|\langle\xi\partial_{\xi}\rangle^{1+\delta_0}\langle \Box^{-1}(F)\cdot W,\,\phi(R;\xi)\rangle_{L^2_{R^3\,dR}}\big\|_{\tau^{-N+}L^2_{d\tau}L^{\infty}_{d\xi}}\lesssim \big\|\lambda^{-2}F\big\|_{\tau^{-N}L^2_{d\tau}L^{1+}_{R^3\,dR}}, \\
			&\big\|\langle \Box^{-1}(F)\cdot W,\,\phi(R;\xi)\rangle_{L^2_{R^3\,dR}}\big\|_{\tau^{-N}L^2_{d\tau}L^{\infty}_{d\xi}}\lesssim \big\|\lambda^{-2}F\big\|_{\tau^{-N}L^2_{d\tau}L^{1+}_{R^3\,dR}}.
		\end{align*}
		Furthermore, we also have 
		\begin{align*}
			&\big\|\langle\xi\partial_{\xi}\rangle^{1+\delta_0}\partial_{\tilde{\tau}}^2\langle \Box^{-1}(F)\cdot W,\,\phi(R;\xi)\rangle_{L^2_{R^3\,dR}}\big\|_{\tau^{-N+}L^2_{d\tau}L^{\infty}_{d\xi}}\lesssim \big\|\lambda^{-2}\nabla_R F\big\|_{\tau^{-N}L^2_{d\tau}L^{1+}_{R^3\,dR}}, \\
			&\big\|\partial_{\tilde{\tau}}^2\langle \Box^{-1}(F)\cdot W,\,\phi(R;\xi)\rangle_{L^2_{R^3\,dR}}\big\|_{\tau^{-N}L^2_{d\tau}L^{\infty}_{d\xi}}\lesssim \big\|\lambda^{-2}\nabla_R F\big\|_{\tau^{-N}L^2_{d\tau}L^{1+}_{R^3\,dR}}.
		\end{align*}
	\end{lem}
	Note that application of $\partial_{\tilde{\tau}}^2$, where $\tilde{\tau}$ is the wave time, 'costs' one derivative due to the smoothing effect of $\Box^{-1}$.  
	\\
	We shall later also require further specializations, whose proofs are simple variations of the preceding: 
	\begin{lem}\label{lem:wavebasicinhomstructure2} Let $F$ be as in Lemma~\ref{lem:specialF1}. Then for $\delta\gg\delta_1>0$ we have the bounds 
		\begin{align*}
			&\big\|\chi_{\xi<\tau^{-\delta}}\langle \xi\partial_{\xi}\rangle^{1+\delta_0}\langle\Box^{-1}(F)\cdot W,\,\phi(R;\xi)\rangle_{L^2_{R^3\,dR}}\big\|_{\tau^{-N-}L^2_{d\tau}L^2_{\rho(\xi)\,d\xi}}\lesssim \big\|z\big\|_{S}\,\\
			&\big\|\chi_{\xi<\tau^{-\delta}}\langle \xi\partial_{\xi}\rangle^{1+\delta_0}\langle\Box^{-1}(F)\cdot W,\,\phi(R;\xi)\rangle_{L^2_{R^3\,dR}}\big\|_{\tau^{-N-}L^2_{d\tau}L^\infty_{\rho(\xi)\,d\xi}}\lesssim \big\|z\big\|_{S}
		\end{align*}
		Moreover, we also have the estimate 
		\begin{align*}
			\big\|\langle \xi\partial_{\xi}\rangle^{1+\delta_0}\langle\chi_{R>\tau^{\frac12-\frac{1}{4\nu}+\delta}}\cdot\Box^{-1}(F)\cdot W,\,\phi(R;\xi)\rangle_{L^2_{R^3\,dR}}\big\|_{\tau^{-N-}L^2_{d\tau}L^2_{\rho(\xi)\,d\xi}}\lesssim \big\|z\big\|_{S}. 
		\end{align*}
	\end{lem}

	\subsection{Technical interlude: estimates with temporal frequency localization}\label{subsec:tempfreqloc}
	
	We shall have to localize functions with respect to temporal frequency below. Here time will be the wave time $\tilde{\tau}$. Specifically, let $\chi$ be a smooth compactly supported function which equals $1$ on $[-1, 1]$, and define the operator 
	\begin{equation}\label{eq:temploc}
		Q^{(\tilde{\sigma})}_{<a}f(\tilde{\sigma}) = \int_{-\infty}^{\infty} a\check{\chi}\big(a(\tilde{\sigma} - \tilde{s})\big)\cdot f(\tilde{s})\,d\tilde{s}
	\end{equation}
	The superscript $\tilde{\sigma}$ is to emphasize that this operator acts with respect to the wave time. We shall be particularly interested in the action of this operator on functions of the form $\Box^{-1}F$. Observe that if $a>1$, say, then the composition
	\[
	\chi_{\tau>\tau*}\circ Q^{(\tilde{\tau})}_{<a}
	\]
	maps $\tau^{-N}L^2_{d\tau}$ into itself, with uniform mapping bound provided $\tau_*$ is sufficiently large. We also observe here that the parameter $a$ may be chosen as a function of $\tilde{\sigma}$. 
	
	For later use we record the following 
	\begin{lem}\label{lem:largemodgeneralboxinverse} 
		Let the equation 
		\[
		\Box n = F
		\]
		be reformulated as \eqref{eq:nFintildetauR} and solved on the Fourier side by \eqref{eq:wavepropagator}. Then we have the bound 
		\begin{align*}
			\Big\|\chi_{\tau>\tau*}\circ Q^{(\tilde{\tau})}_{\geq a}n\Big\|_{\tau^{-N} L^2_{d\tau} a^{-2}\langle R\rangle^{1+\delta_0}L^2_{R^3\,dR}}\lesssim \big\|\lambda^{-2}\langle R\rangle^2\langle\nabla\rangle F\big\|_{\tau^{-N} L^2_{d\tau} L^2_{R^3\,dR}}.
		\end{align*}
	\end{lem}
	\begin{proof} This follows by observing that 
		\begin{align*}
			\Big\|\chi_{\tau>\tau*}\circ Q^{(\tilde{\tau})}_{\geq a}n\Big\|_{\tau^{-N} L^2_{d\tau} S}\lesssim a^{-2}\big\|\partial_{\tilde{\tau}}^2n\big\|_{\tau^{-N} L^2_{d\tau} S}
		\end{align*}
		and directly computing the effect of $\partial_{\tilde{\tau}}^2$ on \eqref{eq:wavepropagator}, as well as using a simple modification of the proof of Lemma~\ref{lem:wavebasicinhom}.
	\end{proof}
	
	\section{The choice of the modulation parameters}\label{sec:modlnparam}
	
	\subsection{The equation for $\tilde{\alpha}$}
	We now determine the choice of the parameters $\tilde{\lambda}, \tilde{\alpha}$, in light of \eqref{eq:kappa1eqn}, \eqref{eq:kappa2eqn} as well as \eqref{eq:E1mod}, \eqref{eq:E2mod}. Observe that (recall \eqref{eq:e1moddef})
	\begin{align*}
		&e_1^{\text{mod}}|_{R = 0} \\&= -\lambda^{-2}\partial_t\tilde{\alpha}(\tau)\cdot\lambda^{-1}\big(\psi_*^{(\tilde{\lambda})}\big)|_{R = 0} + i\lambda^{-2}(\tau)\tilde{\lambda}_t\cdot\lambda^{-1}\partial_{\tilde{\lambda}}\big(\psi_*^{(\tilde{\lambda})}\big)|_{R = 0}+O\big(|\tilde{\alpha}|^2\big),
	\end{align*}
	and furthermore we have 
	\begin{align*}
		\lambda^{-1}\partial_{\tilde{\lambda}}\big(\psi_*^{(\tilde{\lambda})}\big)|_{R = 0}  = \Lambda W|_{R = 0} + O(\tilde{\lambda}),
	\end{align*}
	where we have used the notation 
	\begin{align*}
		\Lambda W = \partial_{\tilde{\lambda}}\Big(W_{\tilde{\lambda}}\Big)|_{\tilde{\lambda} = 1}. 
	\end{align*}
	In light of \eqref{eq:kappa2eqn} describing the evolution of the {\it{imaginary part of $\kappa$}}, we now impose the condition 
	\boxalign[12cm]{\begin{equation}\label{eq:tildealpha1}
		\partial_{\tau}\tilde{\alpha} = -Q_{<0}^{(\tilde{\tau})}\Big[\Re\big(\mathcal{L}z\big)|_{R = 0} - \Re\big(\lambda^{-2}(y \tilde{u}_*^{(\tilde{\lambda}, \underline{\tilde{\alpha}})})\big)|_{R = 0} + \Re( e_1)|_{R = 0}\Big].
	\end{equation}}
	The reason for including the multiplier $Q_{<0}^{(\tilde{\tau})}$ is somewhat subtle, and has to do with the fact that $\tilde{\alpha}$ in turn contributes source terms to the evolution of $\tilde{\lambda}$ via its effect on $E_1^{\text{mod}}$, see \eqref{eq:E1mod}. The equation governing the evolution of $\tilde{\lambda}$ introduced below becomes degenerate in a certain sense for very high temporal frequencies, which motivates limiting $\tilde{\alpha}$ to small temporal frequencies.  The price we have to pay is that $\kappa_2$ will in fact lose some temporal decay, but since this parameter describes the imaginary part of the resonant part of $z$, it turns out that the system accommodates such a loss. 
	
	\subsection{The equation for $\tilde{\lambda}$}\label{subsec:tildelambdaeqn}
	The choice of $\tilde{\lambda}$, which is used to control the {\it{real part of $\kappa$}}, is more subtle, and require a closer look at \eqref{eq:kappa1eqn}.  At first sight, keeping in mind the definition of $e_1^{\text{mod}}$ and its value at $R = 0$ given above, one might be tempted to treat the term 
	\[
	i\lambda^{-2}(\tau)\tilde{\lambda}_t\cdot\lambda^{-1}\partial_{\tilde{\lambda}}\big(\psi_*^{(\tilde{\lambda})}\big)|_{R = 0}
	\]
	as the main source term in \eqref{eq:kappa1eqn} to force the required cancellation required to control $\kappa_1(\tau)$.  However, it turns out that the main source term here really is 
	\[
	\Im\big(\mathcal{L}z\big)|_{R = 0} ,
	\]
	where we recall that $z$ (involving both the resonant as well as the non-resonant parts, with the former getting annihilated upon applying $\mathcal{L}$) solves \eqref{eq:zeqn2}. There the main contribution comes from the following source terms: 
	\begin{equation}\label{eq:kappa1mainsourceterms}
		-\lambda^{-2}y_z\cdot W  -\lambda^{-2}\big(y \tilde{u}_*^{(\tilde{\lambda}, \underline{\tilde{\alpha}})} - y_z\cdot W\big) + e_1^{\text{mod}}. 
	\end{equation}
	The modulation parameter $\tilde{\lambda}$ influences $y-y_z = y_2$ via the second equation in \eqref{eq:zeqn2}, see \eqref{eq:ytildelamba}, which in turn leads to the {\it{real valued contribution}} \eqref{eq:ytildelambdaleadingcontr} to the first equation in \eqref{eq:zeqn2}. 
	The main point will be to counteract the contribution of $-\lambda^{-2}y_z\cdot W$. To determine $\tilde{\lambda}$, we are guided again by the expression on the right in \eqref{eq:ImzR0}. Define 
	\begin{equation}\label{eq:Xdef}\begin{split}
			&\mathcal{F}\Big( -Q^{(\tilde{\tau})}_{<\tau^{\frac12+}}\big(\lambda^{-2}y_z\cdot W\big) - Q^{(\tilde{\tau})}_{<\tau^{\frac12+}}\big(\lambda^{-2}y_z\cdot (\tilde{u}_*^{(\tilde{\lambda}, \tilde{\alpha})}-W)\big)  \\
			& - Q^{(\tilde{\tau})}_{<\tau^{\frac12+}}\big(\lambda^{-2}(y - y_z)\cdot \tilde{u}_*^{(\tilde{\lambda}, \tilde{\alpha})} -  \lambda^{-2}y^{\text{mod}}_{\tilde{\lambda}}\cdot W\big)  - Q^{(\tilde{\tau})}_{<\tau^{\frac12+}}\big((\lambda^{-2}n_*^{(\tilde{\lambda}, \tilde{\alpha})} - W^2)z\big)\Big)(\tau, \xi)\\
			&= : X^{(\tilde{\lambda})}(\tau, \xi), 
	\end{split}\end{equation}
	where we have introduced the part of $y$ depending on terms in $E_2^{\text{mod}}$ contributed by modulating on $\tilde{\lambda}$
	\begin{equation}\label{eq:ylamndatildemod}
		y^{\text{mod}}_{\tilde{\lambda}}: = \Box^{-1}E_2^{\text{mod}}. 
	\end{equation}
	Furthermore, as far as the terms $y^{\text{mod}}_{\tilde{\lambda}}$ are concerned, we shall require below a modification of them for technical reasons as follows: we shall set 
	\begin{equation}\label{eq:ytildemodmodified}
		\tilde{y}^{\text{mod}}_{\tilde{\lambda}}: = \Box^{-1}\tilde{E}_2^{\text{mod}},
	\end{equation}
	where the term on the right is defined as follows, keeping in mind \eqref{eq:E2mod}:
	\begin{equation}\label{eq:tildeE2mod}\begin{split}
			\tilde{E}_2^{\text{mod}} &= -2\tilde{\lambda}_{tt}\cdot \lambda^2\Lambda W\cdot W -2 Q^{(\tilde{\tau})}_{<\tilde{\tau}^{\frac{10}{\nu}}}\big[\tilde{\lambda}_t\cdot \partial_t\big(\lambda^2\Lambda W\cdot W\big)\big]\\
			& - Q^{(\tilde{\tau})}_{<\tilde{\tau}^{\frac{10}{\nu}}}\big[\tilde{\lambda}_t\cdot \partial_t\big(\partial_{\tilde{\lambda}}n_*^{(\tilde{\lambda})} - 2\lambda^2\Lambda W\cdot W\big)\big]\\
			&-Q^{(\tilde{\tau})}_{<1}\Big[\tilde{\lambda}_{tt}\cdot \big(\partial_{\tilde{\lambda}}n_*^{(\tilde{\lambda})}-\lambda^2\Lambda W\cdot W\big)\\
			& + 2\partial_t(\chi_3)\cdot  \tilde{\lambda}_t\cdot \partial_{\tilde{\lambda}}n_*^{(\tilde{\lambda})}- \Box(\chi_3)\cdot \big(n_*^{(\tilde{\lambda})} - n_*\big)\\
			&+ 2\partial_r(\chi_3)\cdot \partial_r\big(n_*^{(\tilde{\lambda})} - n_*\big) -\triangle_r(\chi_3)\cdot\big(|\psi_*^{(\tilde{\lambda})}|^2 - |\psi_*|^2\big)\\
			&-2\partial_r(\chi_3)\cdot \partial_r\big(|\psi_*^{(\tilde{\lambda})}|^2 - |\psi_*|^2\big)\Big]\\
	\end{split}\end{equation}
	It is to be noted that the principal part consists of the first line on the right; it is the sum of these two terms, particularly their restriction to the low temporal frequency regime, which will mostly determine how to choose $\tilde{\lambda}$. 
	We shall use the decomposition \eqref{eq:resnonressplit} and we shall use a more refined representation of the real part of the coefficient of resonant part, $\kappa_1$, as follows:
	\begin{equation}\label{eq:kapparefined}
		\kappa_1 = c_*\cdot Q^{(\tilde{\tau})}_{\geq \gamma^{-1}}\tilde{\lambda} + \tilde{\kappa}_1,
	\end{equation}
	where $\gamma = \gamma(\tau_*)>0$ satisfies $\lim_{\tau_*\rightarrow\infty}\gamma(\tau_*) = 0$; we shall set $\gamma(\tau_*) = (\log\log\tau*)^{-1}$. 
	Here $c_*$ is a certain universal constant also made explicit in the sequel(see \eqref{eq:cstardef}) . Hence we shall use 
	\begin{equation}\label{eq:zedcomprefined}
		z = \big(c_*\cdot Q^{(\tilde{\tau})}_{\geq \gamma^{-1}}\tilde{\lambda} + \tilde{\kappa}_1 + i\kappa_2\big)\phi_0(R) + z_{nres}, 
	\end{equation}
	and in particular, this decomposition shall be used for $z$ in the source term \eqref{eq:Xdef}. 
	\\
	
	Ideally, we would choose $\tilde{\lambda}$ in such fashion that the contribution of the source term $\tilde{y}^{\text{mod}}_{\tilde{\lambda}}\cdot W$ to $\Im \mathcal{L}z|_{R = 0}$ via the Schr\"odinger propagation cancels the contribution of $X^{(\tilde{\lambda})}$, as well as the contribution from $e_1^{\text{mod}}$. However, technical issues linked to the degenercay of the resulting equation for $\tilde{\lambda}$, force a somewhat more delicate choice. 
	\\
	Let $\rho_1(\xi)$ be a function on $(0,\infty)$ which is $C^\infty$, agrees with $\rho(\xi)$ on $(0, 1]$ and equals $\xi^{-2}$ on $[2,\infty)$. Then we shall {\it{essentially}} require the relation (see Proposition~\ref{prop:linpropagator} for the definition of $S$, as well as \eqref{eq:ScalK1}, \eqref{eq:ScalK2} for the definition of $S_{\mathcal{K}}$)
	\boxalign[12cm]{\begin{align}
			\label{eq:tildelambda}
			&\Im \int_{\tau}^\infty\int_0^\infty \xi^2\cdot S(\tau, \sigma;\xi)\cdot X^{(\tilde{\lambda})}(\sigma, 0)\rho_1(\xi)\,d\xi d\sigma  + L^{(\tilde{\lambda})}_{\text{small}} + L^{(\tilde{\lambda})}_{\mathcal{K}}=\\[2pt]& \nonumber
			 - \Im\int_{\tau}^\infty\int_0^\infty \xi^2\cdot S(\tau, \sigma;\xi)\cdot\mathcal{F}\big(\lambda^{-2}\tilde{y}_{\tilde{\lambda}}^{\text{mod}}\cdot W\big)(\sigma, 0)\rho_1(\xi)\,d\xi d\sigma +  R^{(\tilde{\lambda})}_{\text{small}} + R^{(\tilde{\lambda})}_{\mathcal{K}} \\[2pt] \nonumber
			& - \Im \int_{\tau}^\infty\int_0^\infty \xi^2\cdot (S+S_{\mathcal{K}})(\tau, \sigma;\xi)\cdot \mathcal{F}\big(Q^{(\tilde{\tau})}_{<\tau^{\delta}}\big(e_1^{\text{mod}}\big)(\sigma, \frac{\lambda(\tau)}{\lambda(\sigma)}\xi)\rho(\xi)\,d\xi d\sigma\\[2pt] \nonumber
			& - \Im Q^{(\tilde{\tau})}_{<\tau^{\delta}}(e_1^{mod})|_{R = 0}, 
\end{align}}
	where the additional correction terms $L^{(\tilde{\lambda})}_{\text{small}}, R^{(\tilde{\lambda})}_{\text{small}}, R^{(\tilde{\lambda})}_{\mathcal{K}}, L^{(\tilde{\lambda})}_{\mathcal{K}}$ are chosen to play a perturbative role in the resulting equation for $\tilde{\lambda}$, but are still too large to lead to good contributions to $\kappa_1$. These functions are defined in detail as follows:
	\begin{equation}\label{eq:Ltildelambda}\begin{split}
			L^{(\tilde{\lambda})}_{\text{small}}: &= \Im \int_{\tau}^\infty\int_0^\infty \xi^2\cdot S(\tau, \sigma,\xi)\cdot Q^{(\tilde{\sigma})}_{<\sqrt{\gamma}^{-1}}X^{(\tilde{\lambda})}(\sigma, \frac{\lambda(\tau)}{\lambda(\sigma)}\xi)[\rho(\xi) - \rho_1(\xi)]\,d\xi d\sigma\\
			& + \Im \int_{\tau}^\infty\int_0^\infty \xi^2\cdot S(\tau, \sigma,\xi)\cdot Q^{(\tilde{\sigma})}_{<\sqrt{\gamma}^{-1}}(\triangle)X^{(\tilde{\lambda})}(\sigma, \frac{\lambda(\tau)}{\lambda(\sigma)}\xi)\rho_1(\xi)\,d\xi d\sigma
	\end{split}\end{equation}
	as well as 
	\begin{equation}\label{eq:Rtildelambda}\begin{split}
			&R^{(\tilde{\lambda})}_{\text{small}}: \\&= \Im \int_{\tau}^\infty\int_0^\infty \xi^2\cdot S(\tau, \sigma,\xi)\cdot\mathcal{F}\big(\lambda^{-2}Q^{(\tilde{\sigma})}_{<\gamma^{-1}}\tilde{y}_{\tilde{\lambda}}^{\text{mod}}\cdot W\big)(\sigma, \frac{\lambda(\tau)}{\lambda(\sigma)}\xi)[\rho(\xi) - \rho_1(\xi)]\,d\xi d\sigma\\
			& + \Im \int_{\tau}^\infty\int_0^\infty \xi^2\cdot S(\tau, \sigma,\xi)\cdot(\triangle)\mathcal{F}\big(\lambda^{-2}Q^{(\tilde{\sigma})}_{<\gamma^{-1}}\tilde{y}_{\tilde{\lambda}}^{\text{mod}}\cdot W\big)(\sigma, \frac{\lambda(\tau)}{\lambda(\sigma)}\xi)\rho_1(\xi)\,d\xi d\sigma\\
	\end{split}\end{equation}
	Here and in the sequel, we shall employ the notation 
	\begin{equation}\label{eq:triangleinparentheses}
		(\triangle)X(\sigma, \xi): = X(\sigma, \xi) - X(\sigma, 0). 
	\end{equation}
	These terms thus enjoy vanishing at $\xi = 0$, which makes them much better behaved in some sense. 
	\\
	Finally, we set 
	\begin{equation}\label{eq:LtildelambdacalK}\begin{split}
			&L^{(\tilde{\lambda})}_{\mathcal{K}}: = \Im \int_0^\infty \xi^2\cdot S_{\mathcal{K}}\big(X^{(\tilde{\lambda})}(\cdot, 0)\big)(\tau,\xi)\rho_1(\xi)\,d\xi + L^{(\tilde{\lambda})}_{\mathcal{K},\text{small}},\\
			&L^{(\tilde{\lambda})}_{\mathcal{K},\text{small}}: = \Im \int_0^\infty \xi^2\cdot S_{\mathcal{K}}\big(Q^{(\tilde{\sigma})}_{<\sqrt{\gamma}^{-1}}(\triangle)X^{(\tilde{\lambda})}(\cdot, \cdot)\big)(\tau,\xi)\rho_1(\xi)\,d\xi\\
			&\hspace{1.5cm} +  \Im \int_0^\infty \xi^2\cdot S_{\mathcal{K}}\big(Q^{(\tilde{\sigma})}_{<\sqrt{\gamma}^{-1}}X^{(\tilde{\lambda})}(\cdot, \cdot)\big)(\tau,\xi)[\rho(\xi) - \rho_1(\xi)]\,d\xi\\
	\end{split}\end{equation}
	and an analogous definition for $R^{(\tilde{\lambda})}_{\mathcal{K}}, R^{(\tilde{\lambda})}_{\mathcal{K},\text{small}}$. We included the term {\it{essentially}} before since we will only be able to solve \eqref{eq:tildelambda} {\it{up to an error term}} $\partial_{\tau}E$ for a 'nice' function $E(\tau)$. 
	
	\subsection{Remarks on the details of \eqref{eq:tildelambda}} The reason for including the multipliers $Q^{(\tilde{\tau})}_{<\tau^{\frac12+}}$ in front of the various terms in $X^{(\tilde{\lambda})}$ stems from the need to limit the (wave) temporal frequency of $\tilde{\lambda}$ to size $<\tau^{\frac12+}$ up to rapidly decaying tails, in turn required to control the contribution of the term $i\tilde{\lambda}_{\tau}\cdot \partial_{\tilde{\lambda}}\big(\psi_*^{(\tilde{\lambda})}\big)$ to $e_1^{\text{mod}}$; this term would become problematic in the regime of very large (wave) temporal frequency. 
	Similar reasonings motivate including the remaining multipliers of similar type. 
	
	\subsection{A comment on the ensuing sections~\ref{sec:modulneqnsbounds} - ~\ref{sec:nonresII}} In the following four sections, we shall use \eqref{eq:zeqn2}, \eqref{eq:kappa1eqn}, \eqref{eq:kappa2eqn}, \eqref{eq:tildealpha1}, \eqref{eq:tildelambda} to derive {\it{a priori bounds}} on the variables 
	\[
	z_{nres}, \tilde{\kappa}_1, \kappa_2, \tilde{\lambda}, \tilde{\alpha}.
	\] 
	In particular, we shall freely re-apply the equation for $z$ in these derivations. This somewhat formal procedure, which assumes that we already have a solution to this combined system of equations, in fact can be implemented in exactly the same manner for the iterative scheme in section~\ref{sec:solncompletion}, which then indeed establishes the existence of the solution constructively. We choose to present things in this way in order to prevent cluttering up the already complicated equations with further lists of indices. 
	
	\section{Solving the modulation equations}\label{sec:modulneqnsbounds}
	
	\subsection{Main statement on solution of \eqref{eq:tildelambda}} We shall formulate here the main result which, upon assuming that $z_{nres}, \tilde{\kappa}_1, \kappa_2$ as in \eqref{eq:zedcomprefined}
	given, allows us to then determine $\tilde{\lambda}$ solving \eqref{eq:tildelambda}. We still need to indicate the missing constant $c_*$, which we do here\footnote{The integral converges at $\xi_1 = 0$ since the Fourier coefficient vanishes there quadratically.}:
	\begin{equation}\label{eq:cstardef}
		c_*: = \int_0^\infty \sqrt{\xi_1}\rho(\sqrt{\xi_1})\cdot \xi_1^{-2}\mathcal{F}\big(\Lambda W\cdot W^2\big)(\sqrt{\xi}_1)\,d\xi_1
	\end{equation}
In order to measure the coefficients $\tilde{\kappa}_1, \kappa_2$, we shall use the following notation:
\begin{equation}\label{eq:tildekappa1kappa2norm}
 \big\|(\tilde{\kappa}_1, \kappa_2)\big\|_{\tau^{-N}L^2_{d\tau}} = \big\|\tilde{\kappa}_1\big\|_{\tau^{-N}L^2_{d\tau}} + \big\|\kappa_2\big\|_{\tau^{-N+\frac12+\frac{1}{4\nu}+}L^2_{d\tau}}.
\end{equation}
The weaker decay of $\kappa_2$, the imaginary part of the function $\kappa(\tau)$, see \eqref{eq:zdecompbasic}, is affordable since this term does not contribute to the delicate non-local term $y_z$ in light of \eqref{eq:yzdfn}.

	\begin{prop}\label{prop:solnoftildelambdaeqn}  Assume that $X^{(\tilde{\lambda})}(\sigma, \xi)$ is given by \eqref{eq:Xdef} with $y$ defined by \eqref{eq:yeqn1}, and $y_{\tilde{\lambda}}^{\text{mod}}, \tilde{y}_{\tilde{\lambda}}^{\text{mod}}$ defined as in subsection~\ref{subsec:tildelambdaeqn}, while we also recall \eqref{eq:yzdfn} for the definition of $y_z$. Write \eqref{eq:zedcomprefined} with $z_{nres}, \tilde{\kappa}_1, \kappa_2$ given functions, and also assume that $\tilde{\alpha}$ with $\tilde{\alpha}_{\tau}\in \log^{-1}\tau\cdot\tau^{-N}L^2_{d\tau}$ is given. 
		
		Then up to an error of the form $\partial_\tau E$, \eqref{eq:tildelambda} admits a solution $\tilde{\lambda}$ on $[\tau_*, \infty)$ and we have the bounds 
		\begin{align*}
			&\big\|\partial_{\tilde{\tau}}^2\langle \partial_{\tilde{\tau}}^2\rangle^{-1}\tilde{\lambda}\big\|_{\tau^{-N}L^2_{d\tau}}\lesssim \big\|z_{nres}\big\|_{S} + \big\|(\tilde{\kappa}_1, \kappa_2)\big\|_{\tau^{-N}L^2_{d\tau}} + \big\|\tilde{\alpha}_{\tau}\big\|_{\log^{-1}\tau\cdot\tau^{-N}L^2_{d\tau}}\\
			&\big\|E\big\|_{\tau^{-N-}L^2_{d\tau}}\lesssim \big\|z_{nres}\big\|_{S} + \big\|(\tilde{\kappa}_1, \kappa_2)\big\|_{\tau^{-N}L^2_{d\tau}} + \big\|\tilde{\alpha}_{\tau}\big\|_{\log^{-1}\tau\cdot\tau^{-N}L^2_{d\tau}}
		\end{align*}
		Furthermore, we can write 
		\begin{align*}
			&\tilde{\lambda}_{\tilde{\tau}\tilde{\tau}}(\tau) =  \tilde{\lambda}_{prin,\tilde{\tau}\tilde{\tau}}(\tau) + \delta\tilde{\lambda}(\tau),
		\end{align*}
		and denoting the (standard) Fourier transform with respect to $\tilde{\tau}$ by $\mathcal{F}_{\tilde{\tau}}$ we have 
		\begin{align*}
			&\mathcal{F}_{\tilde{\tau}}\big(\langle\partial_{\tilde{\tau}}^2\rangle^{-1}\big(\tilde{\lambda}_{prin,\tilde{\tau}\tilde{\tilde{\tau}}}\big)(\hat{\tilde{\tau}})+c_3(\hat{\tilde{\tau}}, \nu)\mathcal{F}_{\tilde{\tau}}\big(\frac{\tilde{\lambda}_{prin,\tilde{\tau}}}{\tilde{\tau}}\big)(\hat{\tilde{\tau}})\\& = \langle \hat{\tilde{\tau}}^2\rangle\beta_*(\hat{\tilde{\tau}})\cdot \mathcal{F}_{\tilde{\tau}}\big(\Pi^{(\tilde{\tau})}\int_0^\infty Q^{(\tilde{\tau})}_{<\tau^{\frac12+}}\lambda^{-2}\Box^{-1}\triangle\Re\big(\lambda^2 z_{nres}W\big)\cdot W^2 R^3\,dR\big)(\hat{\tilde{\tau}})
		\end{align*}
		where $\beta_*(\hat{\tilde{\tau}})$ is a smooth and bounded complex valued function with non-vanishing imaginary part for $\hat{\tilde{\tau}}\neq 0,\,\pm \hat{\tilde{\tau}}_*$ for suitable $\hat{\tilde{\tau}}_*\in \R_+$, and $\Pi^{(\tilde{\tau})}$ is a suitable projection operator defined further below. Also, the function $c_3 = c_3(\hat{\tilde{\tau}},\nu)$ is supported on $(-1,1)$, bounded and $C^\infty$ away from $\hat{\tilde{\tau}} = 0$, and satisfies the conjugation symmetry $c_3(-\hat{\tilde{\tau}}, \nu) = \overline{c_3(\hat{\tilde{\tau}}, \nu)}$, as well as symbol type bounds near $\hat{\tilde{\tau}} = 0$.
		Furthermore, we have 
		\begin{align*}
			\Big\|\langle\partial_{\tilde{\tau}\tilde{\tau}}\rangle^{-1} \delta\tilde{\lambda}(\tau)\Big\|_{\tau^{-N}L^2_{d\tau}}\leq c(\tau_*)(\big\|z_{nres}\big\|_{S} + \big\|\tilde{\alpha}_{\tau}\big\|_{\log^{-1}\tau\cdot\tau^{-N}L^2_{d\tau}}) +  \big\|(\tilde{\kappa}_1, \kappa_2)\big\|_{\tau^{-N}L^2_{d\tau}}
		\end{align*}
		where 
		\[
		\lim_{\tau_*\rightarrow\infty}c(\tau_*) = 0. 
		\]
		We also have the limiting relation 
		\[
		\lim_{\hat{\tilde{\tau}}\rightarrow 0}\beta_*(\hat{\tilde{\tau}}) = \Big(-\frac{1}{2}\cdot \int_0^\infty \triangle^{-1}\big(\Lambda W\cdot W)\cdot W^2 R^3\,dR\Big)^{-1}.
		\]
	\end{prop}
	\begin{rem} The point of our construction of $\tilde{\lambda}$ is that if we now determine $\kappa_{1}$ in accordance with \eqref{eq:kappa1eqn} and use the finer representation \eqref{eq:kapparefined}, we can infer the bound 
		\begin{align*}
			\big\|\tilde{\kappa}_1\big\|_{\tau^{-N}L^2_{d\tau}}\ll_{\tau_*}\big\|z_{nres}\big\|_{S} + \big\|(\tilde{\kappa}_1,\kappa_2)\big\|_{\tau^{-N}L^2_{d\tau}},
		\end{align*}
		as substantiated in section~\ref{sec:improvkappabounds}.
	\end{rem}
	\begin{proof} It proceeds in several steps:
		\\
		
		{\bf{Step 1}}: {\it{collecting the non-perturbative terms in \eqref{eq:tildelambda} on the right-hand side}}. Recalling \eqref{eq:Xdef} and its principal term 
		\[
		-Q^{(\tilde{\tau})}_{<\tau^{\frac12+}}\big(\lambda^{-2}y_z\cdot W\big),
		\]
		with $z$ represented by \eqref{eq:zedcomprefined}, we shall move the term involving from $\tilde{\lambda}$ up to smaller error to the right hand side. Precisely, we add  the term 
		\begin{align*}
			&c_*\cdot\Im \int_{\tau}^\infty \int_0^\infty \xi^2\cdot S(\tau, \sigma,\xi)\\&\hspace{2cm}\cdot \mathcal{F}\big(\big(\lambda^{-2}(\sigma)\partial_{\tilde{\sigma}}^{-2}\triangle\big(\lambda^2(\sigma)Q^{(\tilde{\sigma})}_{\gamma^{-1}<\cdot<\sigma^{\frac12+}}\tilde{\lambda}W^2\big)\cdot W\big)(\sigma, 0)\rho_1(\xi)\,d\xi d\sigma
		\end{align*}
		to both sides of \eqref{eq:tildelambda}, replacing the principal term in $X^{(\tilde{\lambda})}$ by 
		\[
		-Q^{(\tilde{\tau})}_{<\tau^{\frac12+}}\big(\lambda^{-2}y_{\tilde{z}}\cdot W\big) - \delta\tilde{z},\,\tilde{z}: = (\tilde{\kappa}_1 + i\kappa_2)\phi_0(R) + z_{nres},
		\]
		where we set 
		\begin{align*}
			\delta\tilde{z}: =& -c_*\lambda^{-2}\partial_{\tilde{\sigma}}^{-2}\triangle\big(\lambda^2Q^{(\tilde{\sigma})}_{\gamma^{-1}<\cdot<\sigma^{\frac12+}}\tilde{\lambda}W^2\big)\cdot W
			+c_*\lambda^{-2}\Box^{-1}\triangle\big(\lambda^2Q^{(\tilde{\sigma})}_{\gamma^{-1}<\cdot<\sigma^{\frac12+}}\tilde{\lambda}W^2\big)\cdot W
		\end{align*}
		
		Furthermore, we change the first term on the right hand side of \eqref{eq:tildelambda} to\footnote{Recall the definition of $c_*$ in \eqref{eq:cstardef}} 
		\begin{equation}\label{eq:Phitildelambdadef}\begin{split}
				&-\Im \int_{\tau}^\infty \int_0^\infty \xi^2\cdot S(\tau, \sigma,\xi)\cdot \Phi^{(\tilde{\lambda})}(\sigma,0)\rho_1(\xi)\,d\xi d\sigma,\\
				& \Phi^{(\tilde{\lambda})}(\sigma,0): = \lambda^{-2}\mathcal{F}\big(\tilde{y}_{\tilde{\lambda}}^{\text{mod}}\cdot W\big)(\sigma, 0) \\
				&\hspace{1.5cm}-c_*\lambda^{-2}\mathcal{F}\big(\partial_{\tilde{\sigma}}^{-2}\triangle\big(\lambda^2Q^{(\tilde{\sigma})}_{\gamma^{-1}<\cdot<\sigma^{\frac12+}}\tilde{\lambda}W^2\big)\cdot W\big)(\sigma, 0)  .
		\end{split}\end{equation}
		Let us call $\tilde{X}^{\tilde{\lambda}}(\sigma, 0)$ the modified left hand term in \eqref{eq:tildelambda} with principal term 
		\[
		-Q^{(\tilde{\tau})}_{<\tau^{\frac12+}}\big(\lambda^{-2}y_{\tilde{z}}\cdot W\big) - \delta\tilde{z}.  
		\]
		{\bf{Step 2}}: {\it{estimates for the perturbative source terms}}.\\
		Having thus reformulated \eqref{eq:tildelambda}, our task is now to solve it via a suitable fixed point argument for $\tilde{\lambda}$. The first order of the day is to analyse the source terms 
		\begin{align*}
			\Im \int_{\tau}^\infty\int_0^\infty \xi^2\cdot S(\tau, \sigma;\xi)\cdot \tilde{X}^{(\tilde{\lambda})}(\sigma, 0)\rho_1(\xi)\,d\xi d\sigma,\, L^{(\tilde{\lambda})}_{\text{small}}
		\end{align*}
		\begin{lem}\label{lem: tildelambdaeqnsourcebounds1} 
			We have a decomposition
			\begin{align*}
				&\Im \int_{\tau}^\infty\int_0^\infty \xi^2S(\tau, \sigma;\xi)\cdot \tilde{X}^{(\tilde{\lambda})}(\sigma, 0)\rho_1(\xi)\,d\xi d\sigma \\
				& = -\Im \int_{\tau}^\infty\int_0^\infty\xi^2 S(\tau, \sigma;\xi)\cdot \mathcal{F}\big(Q^{(\tilde{\tau})}_{<\tau^{\frac12+}}\big(\lambda^{-2}y_{z_{nres}}\cdot W\big)\big)(\sigma, 0)\rho_1(\xi)\,d\xi d\sigma\\
				& + Z(\tau), \\
			\end{align*}
			where we have the bounds 
			\begin{align*}
				&\big\|\langle \partial_{\tilde{\tau}}^2\rangle Z\big\|_{\log^{-2}\tau\cdot \tau^{-N}L^2_{d\tau}} \lesssim c(\tau_*)[\big\|z_{nres}\big\|_{S} + \big\|\tilde{\lambda}\big\|_{\tau^{-N}L^2_{d\tau}}]+ \big\|(\tilde{\kappa}_1,\kappa_2)\big\|_{\tau^{-N}L^2_{d\tau}}
			\end{align*}
			with $\lim_{\tau_*\rightarrow\infty} c(\tau_*) = 0$. 
		\end{lem}
		\begin{proof}(lemma) We need to control the various contributions from the terms constituting $ \tilde{X}^{(\tilde{\lambda})}$. Recall that the latter is given by \eqref{eq:Xdef} but with the first term replaced by 
			\[
			-Q^{(\tilde{\tau})}_{<\tau^{\frac12+}}\big(\lambda^{-2}y_{\tilde{z}}\cdot W\big) - \delta\tilde{z}. 
			\]
			In turn we can decompose
			\begin{align*}
				-Q^{(\tilde{\tau})}_{<\tau^{\frac12+}}\big(\lambda^{-2}y_{\tilde{z}}\cdot W\big)  = -Q^{(\tilde{\tau})}_{<\tau^{\frac12+}}\big(\lambda^{-2}y_{z_{nres}}\cdot W\big) - Q^{(\tilde{\tau})}_{<\tau^{\frac12+}}\big(\lambda^{-2}y_{(\tilde{\kappa}_1 + \kappa_2)\phi_0}\cdot W\big)
			\end{align*}
			
			{\it{(1): Contribution of $- Q^{(\tilde{\tau})}_{<\tau^{\frac12+}}\big(\lambda^{-2}y_{(\tilde{\kappa}_1 + i\kappa_2)\phi_0}\cdot W\big)$.}} This function being real valued, we reduce to estimating 
			\begin{align*}
				\int_{\tau}^\infty\int_0^\infty\xi^2 S_1(\tau, \sigma;\xi)\cdot \mathcal{F}\big(Q^{(\tilde{\sigma})}_{<\sigma^{\frac12+}}\big(\lambda^{-2}y_{(\tilde{\kappa}_1 + i\kappa_2)\phi_0}\cdot W\big)\big)(\sigma, 0)\rho_1(\xi)\,d\xi d\sigma
			\end{align*}
			Our main tool to accomplish this will be Lemma~\ref{lem:K_frefined}, in the simpler situation when the function $f(\sigma, \xi)$ is actually independent of $\xi$. Taking advantage of the simple estimate 
			\begin{align*}
				\big\|\partial_{\sigma}\big(Q^{(\tilde{\sigma})}_{<\sigma^{\frac12+}}f\big)\big\|_{\sigma^{-N-\delta_0}L^2_{d\sigma}}\lesssim \big\|f\big\|_{\sigma^{-N}L^2_{d\sigma}}
			\end{align*}
			provided $\delta_0\ll \nu^{-1}$, and further the bound
			\begin{align*}
				&\Big\|\langle\partial_{\tilde{\sigma}}^2\rangle\mathcal{F}\big(\lambda^{-2}y_{(\tilde{\kappa}_1 + i\kappa_2)\phi_0}\cdot W\big)(\sigma, \xi)\Big\|_{\sigma^{-N}L^2_{d\sigma} L^\infty_{d\xi}}\lesssim \big\|(\tilde{\kappa}_1, \kappa_2)\big\|_{\tau^{-N}L^2_{d\tau}},
			\end{align*}
			in turn a consequence of Lemma~\ref{lem:wavebasicinhomstructure1} , we deduce from Lemma~\ref{lem:K_frefined} the bound 
			\begin{align*}
				&\Big\|\langle\partial_{\tilde{\tau}}^2\rangle\int_{\tau}^\infty\int_0^\infty\xi^2 S_1(\tau, \sigma;\xi)\cdot \mathcal{F}\big(Q^{(\tilde{\sigma})}_{<\sigma^{\frac12+}}\big(\lambda^{-2}y_{(\tilde{\kappa}_1 + i\kappa_2)\phi_0}\cdot W\big)\big)(\sigma, 0)\rho_1(\xi)\,d\xi d\sigma\Big\|_{\log^{-2}\tau\cdot\tau^{-N}L^2_{d\tau}}\\
				&\lesssim \big\|(\tilde{\kappa}_1, \kappa_2)\big\|_{\tau^{-N}L^2_{d\tau}}, 
			\end{align*}
			and hence verifying the asserted bound for this contribution. 
			\\
			
			{\it{(2): Contribution of $-\delta\tilde{z}$.}} This term, which depends on $\tilde{\lambda}$, needs to be shown to enjoy a smallness gain, depending on $\gamma$. Consider the second term in the definition of $\delta\tilde{z}$, just before \eqref{eq:Phitildelambdadef}. 
			To begin with, using Lemma~\ref{lem:wavebasicinhom}, we have the estimate 
			\begin{align*}
				&\Big\|\mathcal{F}\Big(\lambda^{-2}\Box^{-1}\triangle\big(\lambda^2Q^{(\tilde{\sigma})}_{\gamma^{-1}<\cdot<\sigma^{\frac12+}}\tilde{\lambda}W^2\big)\cdot \chi_{R>\gamma^{-\frac{1}{100}}}W\Big)(\sigma, 0)\Big\|_{\sigma^{-N}L^2_{d\sigma}}\\&\ll_{\gamma}\big\|Q^{(\tilde{\sigma})}_{\gamma^{-1}<\cdot<\sigma^{\frac12+}}\tilde{\lambda}\big\|_{\sigma^{-N}L^2_{d\sigma}}, 
			\end{align*}
			and application of $\partial_{\sigma}$ to the left hand term inside the norm allows us to strengthen the norm on the left to $\big\|\cdot\big\|_{\sigma^{-N-\delta_0}L^2_{d\sigma}}$. 
			The same estimate applies to the first term in the definition of $-\delta\tilde{z}$, and the same bound still applies after applying $\partial_{\tilde{\sigma}}^2$ to the expressions inside the norms. 
			Using Lemma~\ref{lem:K_frefined} we see that these terms lead to contributions verifying the bound of the lemma. 
			\\
			We shall henceforth restrict the third factor $W$ in both terms constituting $-\delta\tilde{z}$ to the range $R<\gamma^{-\frac{1}{100}}$. 
			Write 
			\[
			u: = \Box^{-1}\triangle\big(\lambda^{2}Q^{(\tilde{\sigma})}_{\gamma^{-1}<\cdot<\sigma^{\frac12+}}\tilde{\lambda}W^2\big),
			\]
			We can easily restrict the frequency with respect to the standard Laplacian of this expression to range $<\gamma^{-\frac{1}{100}}$ up to errors of size $\gamma^N$. 
			Using Lemma~\ref{lem:freqlocalpropag1}, we can include here a localiser $Q^{(\tilde{\sigma})}_{\gamma^{-1+}<\cdot<\sigma^{\frac12+}}$ in front up to terms gaining smallness. Recalling \eqref{eq:nFintildetauR}, we have that 
			\begin{equation}\label{eq:FintildetauRanalogue}\begin{split}
					&-\partial_{\tilde{\tau}}^2u  = \big((\partial_{\tilde{\tau}}+\frac{\lambda_{\tilde{\tau}}}{\lambda}R\partial_R)^2 - \partial_{\tilde{\tau}}^2\big)u + \frac{\lambda_{\tilde{\tau}}}{\lambda}(\partial_{\tilde{\tau}}+\frac{\lambda_{\tilde{\tau}}}{\lambda}R\partial_R)u - (\partial_{RR} + \frac{3}{R}\partial_R)u + F,\\
					&F = \triangle_R\big(\lambda^{2}Q^{(\tilde{\tau})}_{\gamma^{-1}<\cdot<\tau^{\frac12+}}\tilde{\lambda}W^2\big)
			\end{split}\end{equation}
			Then we easily verify that (recall Lemma~\ref{lem:wavebasicinhom})
			\begin{align*}
				\Big\|\lambda^{-2}\chi_{R<\gamma^{-\frac{1}{100}}}\big((\partial_{\tilde{\tau}}+\frac{\lambda_{\tilde{\tau}}}{\lambda}R\partial_R)^2 - \partial_{\tilde{\tau}}^2\big)u\Big\|_{\tau^{-N-\frac12-\frac{1}{4\nu}}L^2_{d\tau}\langle R\rangle^{1+\delta_0}L^2_{R^3\,dR}}\lesssim_{\gamma}\big\|\tilde{\lambda}\big\|_{\tau^{-N}L^2_{d\tau}}.  
			\end{align*}
			and similar bound applies to the expression
			\[
			\frac{\lambda_{\tilde{\tau}}}{\lambda}(\partial_{\tilde{\tau}}+\frac{\lambda_{\tilde{\tau}}}{\lambda}R\partial_R)u.
			\]
			For the remaining expression, recalling Lemma~\ref{lem:largemodgeneralboxinverse}, we have the somewhat weaker bound 
			\begin{align*}
				\big\|\lambda^{-2}Q^{(\tilde{\tau})}_{\gamma^{-1+}<\cdot<\sigma^{\frac12+}}(\partial_{RR} + \frac{3}{R}\partial_R)u\big\|_{\tau^{-N}L^2_{d\tau}\langle R\rangle^{1+\delta_0}L^2_{R^3\,dR}}\ll_{\gamma} \big\|\tilde{\lambda}\big\|_{\tau^{-N}L^2_{d\tau}},
			\end{align*}
			and the norm on the left can again be slightly strengthened by inclusion of $\tau^{-\delta_0}$ upon application of $\partial_{\tau}$. 
			Applying $Q^{(\tilde{\sigma})}_{\gamma^{-1+}<\cdot<\sigma^{\frac12+}}$ to \eqref{eq:FintildetauRanalogue} and inverting $\partial_{\tilde{\tau}}^2$ via division by the Fourier symbol, we find the bound 
			\begin{align*}
				\Big\|\lambda^{-2}Q^{(\tilde{\tau})}_{\gamma^{-1+}<\cdot<\tau^{\frac12+}}u + \lambda^{-2}\partial_{\tilde{\tau}}^{-2}F\Big\|_{\tau^{-N}L^2_{d\tau}\langle R\rangle^{1+\delta_0}L^2_{R^3\,dR}}\ll_{\gamma, \tau_*}\big\|\tilde{\lambda}\big\|_{\tau^{-N}L^2_{d\tau}}, 
			\end{align*}
			and the bound can be improved as described before upon application of $\partial_{\tau}$ to the term on the left(namely replacing $\tau^{-N}$ by $\tau^{-N-}$). Furthermore, these estimates remain valid upon application of $\partial_{\tilde{\tau}}^2$. We conclude that the function 
			\begin{align*}
				\mathcal{F}\big(\big[\lambda^{-2}Q^{(\tilde{\sigma})}_{\gamma^{-1+}<\cdot<\sigma^{\frac12+}}u + \lambda^{-2}\partial_{\tilde{\sigma}}^{-2}F\big]\cdot W\big)(\sigma, 0)
			\end{align*} 
			satisfies the requirements for application of Lemma~\ref{lem:K_frefined}, which in conjunction with the preceding considerations implies the desired bound for the contribution of $\delta\tilde{z}$. 
			\\
			
			{\it{(3): contribution of the remaining terms in $\tilde{X}^{(\tilde{\lambda})}(\sigma, 0)$.}} This is done in section~\ref{sec:appendix}. 
		\end{proof}
		
		For the remaining error term $L_{\text{small}}^{(\tilde{\lambda})}$ on the left hand side of \eqref{eq:tildelambda} we have a stronger bound 
		\begin{lem}\label{lem:Lsmalltildelambda} We have the estimate 
			\begin{align*}
				\Big\|\langle\partial_{\tilde{\tau}}^2\rangle L_{\text{small}}^{(\tilde{\lambda})}\Big\|_{\tau^{-N-}L^2_{d\tau}}\lesssim_{\gamma} \big\|z_{nres}\big\|_{S} + \big\|\langle\partial_{\tilde{\tau}}^2\rangle^{-2}\partial_{\tilde{\tau}}^2\tilde{\lambda}\big\|_{\tau^{-N}L^2_{d\tau}} + \big\|(\tilde{\kappa}_1,\kappa_2)\big\|_{\tau^{-N}L^2_{d\tau}}
			\end{align*}
		\end{lem}
		\begin{proof} Recalling the definition \eqref{eq:Ltildelambda}, the main observation is that we have an 'extra factor' $\xi^2$ to take advantage of, which allows us to perform an integration by parts with respect to $\sigma$ to gain additional smallness. The inequality needs to again be verified for all the different terms constituting $X^{(\tilde{\lambda})}$, and we deal here with the most delicate principal term, relegating the other ones to section~\ref{sec:appendix}. Moreover, we deal here with the term on the second line in \eqref{eq:Ltildelambda}, since the term on the first line is supported in the regime $\xi\gtrsim 1$, while the chief difficulties arise in the low-frequency regime $\xi\ll 1$. Thus for the rest of this proof we shall set 
			\begin{equation}\label{eq:assumption1}
				X^{(\tilde{\lambda})}(\sigma, \xi) = \mathcal{F}\Big( -Q^{(\tilde{\sigma})}_{<\sigma^{\frac12+}}\big(\lambda^{-2}y_z\cdot W\big)\Big)(\sigma,\xi), 
			\end{equation}
			with $z$ in turn given by \eqref{eq:zedcomprefined}. Observe that since the preceding term is real-valued, we can write the second line in \eqref{eq:Ltildelambda} as 
			\begin{align*}
				\int_{\tau}^\infty\int_0^\infty \xi^2\cdot S_1(\tau, \sigma,\xi)\cdot Q^{(\tilde{\sigma})}_{<\sqrt{\gamma}^{-1}}(\triangle)X^{(\tilde{\lambda})}(\sigma, \frac{\lambda(\tau)}{\lambda(\sigma)}\xi)\rho_1(\xi)\,d\xi d\sigma
			\end{align*}
			where we recall the notation $(\triangle)X(\sigma,\xi) = X(\sigma,\xi) - X(\sigma, 0)$. Introduce the variable $\tilde{\xi}: = \frac{\lambda(\tau)}{\lambda(\sigma)}\xi$, and note the relation 
			\begin{align*}
				S_1(\tau,\sigma,\xi) = \frac{\lambda^2(\tau)}{\lambda^2(\sigma)}\cos\big(\lambda^2(\sigma)\tilde{\xi}^2\int_{\sigma}^{\tau}\lambda^{-2}(s)\,ds\big) =:\frac{\lambda^2(\tau)}{\lambda^2(\sigma)}\tilde{S}_1(\tau,\sigma,\tilde{\xi})
			\end{align*}
			and so 
			\begin{align*}
				\tilde{\xi}^2\tilde{S}_1(\tau,\sigma,\tilde{\xi}) = \zeta(\tau,\sigma)\partial_{\sigma}\tilde{S}_2(\tau,\sigma,\tilde{\xi}),
			\end{align*}
			where we put 
			\begin{align*}
				\tilde{S}_2(\tau,\sigma,\tilde{\xi} =\sin\big(\lambda^2(\sigma)\tilde{\xi}^2\int_{\sigma}^{\tau}\lambda^{-2}(s)\,ds\big),\, \zeta(\tau,\sigma) = \big[\partial_{\sigma}\big(\lambda^2(\sigma)\int_{\sigma}^{\tau}\lambda^{-2}(s)\,ds\big)\big]^{-1},
			\end{align*}
			whence $\zeta(\tau,\sigma)\sim \frac{\sigma\cdot\lambda^2(\tau)}{\tau\lambda^2(\sigma)}$. We can then reformulate the above double integral as 
			\begin{align*}
				& \int_{\tau}^\infty\int_0^\infty \xi^2\cdot S_1(\tau, \sigma,\xi)\cdot Q^{(\tilde{\sigma})}_{<\sqrt{\gamma}^{-1}}(\triangle)X^{(\tilde{\lambda})}(\sigma, \frac{\lambda(\tau)}{\lambda(\sigma)}\xi)\rho_1(\xi)\,d\xi d\sigma\\
				& = - \int_{\tau}^\infty\int_0^\infty \tilde{\xi}^2\cdot\tilde{S}_2(\tau, \sigma,\tilde{\xi})\\&\hspace{2cm}\cdot \partial_{\sigma}\big(\zeta(\tau,\sigma)\cdot\frac{\lambda(\sigma)}{\lambda(\tau)}\cdot \frac{Q^{(\tilde{\sigma})}_{<\sqrt{\gamma}^{-1}}(\triangle)X^{(\tilde{\lambda})}(\sigma,\tilde{\xi})}{\tilde{\xi}^2}\rho(\frac{\lambda(\sigma)}{\lambda(\tau)}\tilde{\xi})\big)\,d\tilde{\xi} d\sigma
			\end{align*}
			The required estimate then follows by applying Lemma~\ref{lem:tildeKfcontrol} as well as Lemma~\ref{lem:yzWbound1} and taking advantage of the symbol behavior of $\rho$ (see subsection~\ref{subsec:basicfourier}.
			)
		\end{proof}
		
		In a similar vein, we have the next lemma, which disposes of the perturbative term $L_{\mathcal{K},\text{small}}^{(\tilde{\lambda})}$: 
		\begin{lem}\label{lem:LsmallcalKtildelambda} We have the estimate
			\begin{align*}
				\Big\|\langle\partial_{\tilde{\tau}}^2\rangle L_{\mathcal{K},\text{small}}^{(\tilde{\lambda})}\Big\|_{\log^{-2}\tau\cdot\tau^{-N}L^2_{d\tau}}&\ll_{\tau_*}\big\|z_{nres}\big\|_{S} + \big\|\langle\partial_{\tilde{\tau}}^2\rangle^{-2}\partial_{\tilde{\tau}}^2\tilde{\lambda}\big\|_{\tau^{-N}L^2_{d\tau}} + \big\|(\tilde{\kappa}_1,\kappa_2)\big\|_{\tau^{-N}L^2_{d\tau}}\\
				& + \big\|\tilde{\alpha}_{\tau}\big\|_{\log^{-1}\tau\cdot \tau^{-N}L^2_{d\tau}}. 
			\end{align*}
		\end{lem}
		An outline of the technical proof is deferred to section~\ref{sec:appendix}. 
		\\
		
		The following lemma is entirely analogous to the preceding two, treating the error terms $R_{\text{small}}^{(\tilde{\lambda})},R_{\mathcal{K},\text{small}}^{(\tilde{\lambda})}$:
		\begin{lem}\label{lem:Rsmalltildelambda} We have the estimate 
			\begin{align*}
				&\Big\|\langle\partial_{\tilde{\tau}}^2\rangle R_{\text{small}}^{(\tilde{\lambda})}\Big\|_{\tau^{-N-\frac12-\frac{1}{4\nu}}L^2_{d\tau}}\lesssim_{\gamma}\big\|\langle\partial_{\tilde{\tau}}^2\rangle^{-1}\partial_{\tilde{\tau}}^2\tilde{\lambda}\big\|_{\tau^{-N}L^2_{d\tau}},\\
				&\Big\|\langle\partial_{\tilde{\tau}}^2\rangle R_{\mathcal{K},\text{small}}^{(\tilde{\lambda})}\Big\|_{\log^{-2}\tau\cdot\tau^{-N}L^2_{d\tau}}\ll_{\tau_*} \big\|\langle\partial_{\tilde{\tau}}^2\rangle^{-2}\partial_{\tilde{\tau}}^2\tilde{\lambda}\big\|_{\tau^{-N}L^2_{d\tau}}. 
			\end{align*}
		\end{lem}
		Finally, recalling \eqref{eq:tildelambda}, it remains to estimate the terms on the last two lines there, which are also of perturbative character:
		\begin{lem}\label{lem:e1modgooderrors} Setting 
			\begin{align*}
				&E(\tau):\\& = \Im \int_{\tau}^\infty\int_0^\infty \xi^2\cdot S(\tau, \sigma;\xi)\cdot \mathcal{F}\big(Q^{(\tilde{\tau})}_{<\tau^{\delta}}\big(e_1^{\text{mod}}\big)(\sigma, \frac{\lambda(\tau)}{\lambda(\sigma)}\xi)\rho(\xi)\,d\xi d\sigma\\& + \Im Q^{(\tilde{\tau})}_{<\tau^{\delta}}(e_1^{mod})|_{R = 0},
			\end{align*}
			we have the bound 
			\begin{align*}
				\Big\|\langle\partial_{\tilde{\tau}}^2\rangle E(\tau)\Big\|_{\log^{-2}\tau\cdot \tau^{-N}L^2_{d\tau}}&\ll_{\tau_*}\big\|z_{nres}\big\|_{S} + \big\|\langle\partial_{\tilde{\tau}}^2\rangle^{-2}\partial_{\tilde{\tau}}^2\tilde{\lambda}\big\|_{\tau^{-N}L^2_{d\tau}} + \big\|(\tilde{\kappa}_1,\kappa_2)\big\|_{\tau^{-N}L^2_{d\tau}}
				\\&\hspace{0.5cm}+\big\|\tilde{\alpha}_{\tau}\big\|_{\log^{-1}\tau\cdot\tau^{-N}L^2_{d\tau}}.
			\end{align*}
		\end{lem}
		\begin{proof} Recalling \eqref{eq:E1mod} as well as \eqref{eq:e1moddef}, we first consider the contribution of  $Q^{(\tilde{\tau})}_{<\tau^{\delta}}(\tilde{\lambda}_{\tau})$ coming from $ \Im Q^{(\tilde{\tau})}_{<\tau^{\delta}}(e_1^{mod})|_{R = 0}$. Observe that 
			\begin{align*}
				Q^{(\tilde{\tau})}_{<\tau^{\delta}}(\tilde{\lambda}_{\tau}) = Q^{(\tilde{\tau})}_{<\tau^{\delta}}\big(\frac{\tilde{\tau}}{\tau}\cdot \tilde{\lambda}_{\tilde{\tau}}\big). 
			\end{align*}
			Thus writing 
			\begin{align*}
				\tilde{\lambda}_{\tilde{\tau}} = Q_{0\leq \cdot<\tau^{\delta}}\tilde{\lambda}_{\tilde{\tau}} + Q^{(\tilde{\tau})}_{<0}\tilde{\lambda}_{\tilde{\tau}},
			\end{align*}
			we have that 
			\begin{align*}
				\big\|Q_{0\leq \cdot<\tau^{\delta}}\tilde{\lambda}_{\tilde{\tau}}\big\|_{\tau^{-N+\delta}L^2_{d\tau}}\lesssim \big\|\langle \partial_{\tilde{\tau}}^2\rangle^{-2}\partial_{\tilde{\tau}}^2\tilde{\lambda}\big\|_{\tau^{-N}L^2_{d\tau}}, 
			\end{align*}
			and further 
			\begin{align*}
				\big\|Q^{(\tilde{\tau})}_{<0}\tilde{\lambda}_{\tilde{\tau}}\big\|_{\tau^{-N+\frac12-\frac{1}{4\nu}}L^2_{d\tau}}\lesssim \big\|\langle \partial_{\tilde{\tau}}^2\rangle^{-2}\partial_{\tilde{\tau}}^2\tilde{\lambda}\big\|_{\tau^{-N}L^2_{d\tau}}.
			\end{align*}
			These estimates in turn imply that 
			\begin{align*}
				&\Big\|Q^{(\tilde{\tau})}_{<\tau^{\delta}}\big(\frac{\tilde{\tau}}{\tau}\cdot Q^{(\tilde{\tau})}_{0\leq \cdot<\tau^{\delta}}\tilde{\lambda}_{\tilde{\tau}}\big)\Big\|_{\tau^{-N+\delta - \frac12-\frac{1}{4\nu}}L^2_{d\tau}}\lesssim \big\|\langle \partial_{\tilde{\tau}}^2\rangle^{-2}\partial_{\tilde{\tau}}^2\tilde{\lambda}\big\|_{\tau^{-N}L^2_{d\tau}},\\ 
				&\Big\|Q^{(\tilde{\tau})}_{<\tau^{\delta}}\big(\frac{\tilde{\tau}}{\tau}\cdot Q^{(\tilde{\tau})}_{<0}\tilde{\lambda}_{\tilde{\tau}}\big)\Big\|_{\tau^{-N -\frac{1}{2\nu}}L^2_{d\tau}}\lesssim \big\|\langle \partial_{\tilde{\tau}}^2\rangle^{-2}\partial_{\tilde{\tau}}^2\tilde{\lambda}\big\|_{\tau^{-N}L^2_{d\tau}}.
			\end{align*}
			which implies the desired estimate as long as $\delta<\min\{\frac12, \frac{1}{2\nu}\}$. To conclude the contribution of the term at $R = 0$, it remains to deal with the term $O\big(|\tilde{\alpha}|^2\big)$, for which we have the bound 
			\begin{align*}
				\big\|O\big(|\tilde{\alpha}|^2\big)\big\|_{L^{2-2N+}L^2_{d\tau}}\lesssim \big\|\tilde{\alpha}_{\tau}\big\|_{\log^{-1}\tau\cdot\tau^{-N}L^2_{d\tau}}^2,
			\end{align*}
			which is of course much better than needed for $N\gg 1$.
			\\
			
			It remains to treat the integral term involving $e_1^{\text{mod}}$, for whose fine structure we recall \eqref{eq:E1mod}, as well as \eqref{eq:e1moddef}. We shall treat the contributions of the first, second and seventh term there, relegating the remaining ones to section~\ref{sec:appendix}. 
			\\
			
			{\it{(1): contribution of the term $\big(i\partial_t + \triangle\big)( \chi_3)\cdot \big(\psi_*^{(\tilde{\lambda})} - \psi_*\big)$ in $E_1^{\text{mod}}$.}} 
			\\
			
			This term leads to the re-scaled term $\big(i\partial_\tau + \triangle_R\big)( \chi_3)\cdot \lambda^{-1}\big(\psi_*^{(\tilde{\lambda})} - \psi_*\big)$. Then observe that 
			\begin{align*}
				\lambda^{-1}\big(\psi_*^{(\tilde{\lambda})} - \psi_*\big) = \tilde{\lambda}\cdot\big(W(R) + O(\log R\cdot\tau^{-1})\big)
			\end{align*}
			on the support of $\chi_3$, which we recall is confined to the region $r\lesssim t^{\frac12+\epsilon}$, or $R\lesssim \tau^{\frac12-\frac{\epsilon}{2\nu}}$, and the term $O(\ldots)$ has symbol behavior with respect to $R$. We further use 
			\begin{align*}
				\big|\partial_{\tau}\chi_3\big|\lesssim \tau^{-1},\, \big|\triangle_R\chi_3\big|\lesssim \tau^{-1+\frac{\epsilon}{\nu}}, 
			\end{align*}
			and observe the symbol behavior of all terms involved with respect to $R$, as well as the following bound resulting from the preceding ones:
			\begin{align*}
				&\Big\|\partial_{\tau}\mathcal{F}\Big(Q^{(\tilde{\tau})}_{<\tau^{\delta}}\big(i\partial_\tau + \triangle_R\big)( \chi_3)\cdot \lambda^{-1}\big(\psi_*^{(\tilde{\lambda})} - \psi_*\big)\Big)(\tau,\xi)\Big\|_{\tau^{-N-\delta_0}L^2_{d\tau}L_{d\xi}^\infty}\\&\hspace{8cm}\lesssim \big\|\langle \partial_{\tilde{\tau}}^2\rangle^{-2}\partial_{\tilde{\tau}}^2\tilde{\lambda}\big\|_{\tau^{-N}L^2_{d\tau}}
			\end{align*}
			provided $\delta_0\ll\nu^{-1}$, as well as the bound 
			\begin{align*}
				&\Big\|\langle\xi\partial_{\xi}\rangle^{1+\delta_0}\mathcal{F}\Big(Q^{(\tilde{\tau})}_{<\tau^{\delta}}\big(i\partial_\tau + \triangle_R\big)( \chi_3)\cdot \lambda^{-1}\big(\psi_*^{(\tilde{\lambda})} - \psi_*\big)\Big)(\tau,\xi)\Big\|_{\tau^{-N-\delta_0}L^2_{d\tau}L_{d\xi}^\infty}\\&\hspace{8cm}\lesssim \big\|\langle \partial_{\tilde{\tau}}^2\rangle^{-2}\partial_{\tilde{\tau}}^2\tilde{\lambda}\big\|_{\tau^{-N}L^2_{d\tau}},
			\end{align*}
			and in both estimates the norm $L^\infty_{d\xi}$ can be replaced by $L^2_{\rho(\xi)\,d\xi}$. It follows that the conditions for Lemma~\ref{lem:K_frefined}, Lemma~\ref{lem:tildeKfcontrol}, are satisfied, resulting in the desired bound for this contribution (where we apply one lemma or the other depending on the real and imaginary parts of the source term). 
			\\
			
			{\it{(2): contribution of the term $2\partial_r( \chi_3)\cdot (\partial_r\psi_*^{(\tilde{\lambda})} - \partial_r\psi_*)$ in $E_1^{\text{mod}}$.}} This term is analogous to the term $\triangle(\chi_3)\cdot \big(\psi_*^{(\tilde{\lambda})} - \psi_*\big)$ treated before. 
			\\
			
			{\it{(3): contribution of the term $Q^{(\tilde{\tau})}_{<\tau^{\delta}}\big(i\tilde{\lambda}_t\cdot\partial_{\tilde{\lambda}}\big(\psi_*^{(\tilde{\lambda})}\big)\big)$ in $E_1^{\text{mod}}$.}} After re-scaling (and disregarding negligible errors), we arrive at the source term 
			\[
			Q^{(\tilde{\tau})}_{<\tau^{\delta}}\big(i\tilde{\lambda}_{\tau}\cdot\partial_{\tilde{\lambda}}\big(\psi_*^{(\tilde{\lambda})}\big)\big)
			\]
			for $e_1^{\text{mod}}$. Then proceeding as for the source term $Q^{(\tilde{\tau})}_{<\tau^{\delta}}(\tilde{\lambda}_{\tau})$ treated at the beginning, we see that we have 
			\begin{align*}
				&\Big\|\partial_{\tau}\mathcal{F}\Big(Q^{(\tilde{\tau})}_{<\tau^{\delta}}\big(i\tilde{\lambda}_{\tau}\cdot\partial_{\tilde{\lambda}}\big(\psi_*^{(\tilde{\lambda})}\big)\big)\Big)(\tau,\xi)\Big\|_{\tau^{-N-\delta_0}L^2_{d\tau}L^\infty_{d\xi}}\lesssim \big\|\langle \partial_{\tilde{\tau}}^2\rangle^{-2}\partial_{\tilde{\tau}}^2\tilde{\lambda}\big\|_{\tau^{-N}L^2_{d\tau}}\\
				&\Big\|\langle\xi\partial_{\xi}\rangle^{1+\delta_0}\mathcal{F}\Big(Q^{(\tilde{\tau})}_{<\tau^{\delta}}\big(i\tilde{\lambda}_{\tau}\cdot\partial_{\tilde{\lambda}}\big(\psi_*^{(\tilde{\lambda})}\big)\big)\Big)(\tau,\xi)\Big\|_{\tau^{-N-\delta_0}L^2_{d\tau}L^\infty_{d\xi}}\lesssim \big\|\langle \partial_{\tilde{\tau}}^2\rangle^{-2}\partial_{\tilde{\tau}}^2\tilde{\lambda}\big\|_{\tau^{-N}L^2_{d\tau}}\\
			\end{align*}
			provided $\delta,\delta_0\ll \nu^{-1}$, and the space $L^\infty_{d\xi}$ can be replaced by $L^2_{\rho(\xi)\,d\xi}$. The desired conclusion follows again by applying Lemma~\ref{lem:K_frefined}, Lemma~\ref{lem:tildeKfcontrol}.
		\end{proof}
		
		The following lemma is quite similar; its proof uses the same kinds of estimates in addition to Lemma~\ref{lem:concatenation1}: 
		\begin{lem}\label{lem:e1modgooderrorsSK} Setting 
			\begin{align*}
				&E_{\mathcal{K}}(\tau):\\& = \Im \int_{\tau}^\infty\int_0^\infty \xi^2\cdot S_{\mathcal{K}}(\tau, \sigma;\xi)\cdot \mathcal{F}\big(Q^{(\tilde{\tau})}_{<\tau^{\delta}}\big(e_1^{\text{mod}}\big)(\sigma, \frac{\lambda(\tau)}{\lambda(\sigma)}\xi)\rho(\xi)\,d\xi d\sigma,
			\end{align*}
			we have the bound 
			\begin{align*}
				\Big\|\langle\partial_{\tilde{\tau}}^2\rangle E_{\mathcal{K}}(\tau)\Big\|_{\log^{-2}\tau\cdot \tau^{-N}L^2_{d\tau}}&\ll_{\tau_*}  \big\|z_{nres}\big\|_{S} + \big\|\langle\partial_{\tilde{\tau}}^2\rangle^{-2}\partial_{\tilde{\tau}}^2\tilde{\lambda}\big\|_{\tau^{-N}L^2_{d\tau}} + \big\|(\tilde{\kappa}_1,\kappa_2)\big\|_{\tau^{-N}L^2_{d\tau}}\\&\hspace{0.5cm}+\big\|\tilde{\alpha}_{\tau}\big\|_{\log^{-1}\tau\cdot\tau^{-N}L^2_{d\tau}}.
			\end{align*}
		\end{lem}

		{\bf{Step 3}}: {\it{Formulation of \eqref{eq:tildelambda} as perturbative problem, solution of model equation, and solving for $\Phi^{(\tilde{\lambda})}(\sigma, 0)$}}. Recalling the quantity $\Phi^{(\tilde{\lambda})}(\sigma, 0)$ from Step 1, see \eqref{eq:Phitildelambdadef}, and also recalling Lemma~\ref{lem: tildelambdaeqnsourcebounds1},  
		and recalling from Step 1 that we replace the first expression on the right hand side of \eqref{eq:tildelambda} by the first expression in \eqref{eq:Phitildelambdadef}, by simple comparison we see that it is natural to set 
		\begin{equation}\label{eq:Phitildelambdaansatz}
			\Phi^{(\tilde{\lambda})}(\sigma, 0) =  -\mathcal{F}\big(Q^{(\tilde{\tau})}_{<\tau^{\frac12+}}\big(\lambda^{-2}y_{z_{nres}}\cdot W\big)\big)(\sigma, 0) + \delta\Phi^{(\tilde{\lambda})}(\sigma, 0),
		\end{equation}
		where the last term $ \delta\Phi^{(\tilde{\lambda})}(\sigma, 0)$ is perturbative in nature. To show this, we first need the following auxiliary proposition, which will allow us to solve for $\delta\Phi^{(\tilde{\lambda})}(\sigma, 0)$:
		\begin{prop}\label{prop: solnforylambdaW} Assume that $f(\sigma)\in \sigma^{-N} L^2_{d\sigma}([\tau_*,\infty))$. Then there exists $z(\sigma)\in \log^{2}\sigma\cdot\sigma^{-N} L^2_{d\sigma}([\tau_*, \infty))$, depending linearly on $f$, with the property that 
			\begin{equation}\label{eq:zfrelation}
				\int_{\tau}^\infty\int_0^\infty \xi^2\cdot S_1(\tau, \sigma;\xi)\cdot z(\sigma)\rho_1(\xi)\,d\xi d\sigma = f(\tau),\,\forall \tau\in [\tau_*, \infty), 
			\end{equation}
			and we have the bound 
			\begin{align*}
				\big\|z\big\|_{\log^{2}\tau\cdot\tau^{-N} L^2_{d\tau}([\tau_*,\infty))}\lesssim_N \big\|f\big\|_{\tau^{-N} L^2_{d\tau}([\tau_*,\infty))}.
			\end{align*}
			If $f(\sigma)\in \log^{-2}(\sigma)\sigma^{-N} L^2_{d\sigma}([\tau_*,\infty))$, we get the better property 
			\[
			z(\sigma)\in \sigma^{-N} L^2_{d\sigma}([\tau_*, \infty)),
			\]
			with a corresponding estimate analogous to the preceding one. 
		\end{prop}
		We relegate the proof of this to section~\ref{sec:appendix}. Armed with it, we can now formulate the following 
		\begin{lem}\label{lem:deltaPhibound} There is a choice of $\delta\Phi^{(\tilde{\lambda})}(\sigma, 0)$ as in \eqref{eq:Phitildelambdaansatz} such that the equation \eqref{eq:tildelambda} is satisfied on $[\tau_*, \infty)$ and moreover we have the bound 
			\begin{align*}
				\Big\|\langle\partial_{\tilde{\sigma}}^2\rangle \delta\Phi^{(\tilde{\lambda})}(\sigma, 0)\Big\|_{\sigma^{-N}L^2_{d\sigma}}\leq c(\tau_*)[\big\|z_{nres}\big\|_{S} + \big\|\partial_{\tilde{\tau}}^2\langle\partial_{\tilde{\tau}}^2\rangle^{-1}\tilde{\lambda}\big\|_{\tau^{-N}L^2_{d\tau}}]+ \big\|(\tilde{\kappa}_1,\kappa_2)\big\|_{\tau^{-N}L^2_{d\tau}}
			\end{align*}
			with $\lim_{\tau_*\rightarrow\infty}c(\tau_*) = 0$. 
		\end{lem}
		\begin{proof} To begin with, we observe that the equation satisfied by $ \delta\Phi^{(\tilde{\lambda})}(\sigma, 0)$ is the following:
			\begin{equation}\label{eq:deltaPhitildelambda}\begin{split}
					&-\int_{\tau}^\infty\int_0^\infty \xi^2\cdot S_1(\tau, \sigma;\xi)\cdot \delta\Phi^{(\tilde{\lambda})}(\sigma, 0)\rho_1(\xi)\,d\xi d\sigma\\
					& = Z(\tau) +  L^{(\tilde{\lambda})}_{\text{small}} -  R^{(\tilde{\lambda})}_{\text{small}} + L^{(\tilde{\lambda})}_{\mathcal{K}} -  R^{(\tilde{\lambda})}_{\mathcal{K}}\\
					&+  \Im \int_{\tau}^\infty\int_0^\infty \xi^2\cdot (S+S_{\mathcal{K}})(\tau, \sigma;\xi)\cdot \mathcal{F}\big(Q^{(\tilde{\tau})}_{<\tau^{\delta}}\big(e_1^{\text{mod}}\big)(\sigma, \frac{\lambda(\tau)}{\lambda(\sigma)}\xi)\rho(\xi)\,d\xi d\sigma\\& +\Im Q^{(\tilde{\tau})}_{<\tau^{\delta}}(e_1^{mod}),
			\end{split}\end{equation}
			where we recall Lemma~\ref{lem: tildelambdaeqnsourcebounds1} for the definition of $Z(\tau)$. Considering the difference term $L^{(\tilde{\lambda})}_{\mathcal{K}} -  R^{(\tilde{\lambda})}_{\mathcal{K}}$, we state 
			\begin{lem}\label{Lem:LtildeK-RtildeKterm} We have the difference bound 
				\begin{align*}
					\Big\|\langle\partial_{\tilde{\sigma}}^2\rangle(L^{(\tilde{\lambda})}_{\mathcal{K}} -  R^{(\tilde{\lambda})}_{\mathcal{K}})\Big\|_{\log^{-2}(\tau)\cdot\tau^{-N}L^2_{d\tau}}&\ll_{\tau_*}\big\|z_{nres}\big\|_{S} + \big\|\partial_{\tilde{\tau}}^2\langle\partial_{\tilde{\tau}}^2\rangle^{-1}\tilde{\lambda}\big\|_{\tau^{-N}L^2_{d\tau}}]\\&\hspace{0.3cm} + \big\|(\tilde{\kappa}_1,\kappa_2)\big\|_{\tau^{-N}L^2_{d\tau}} + \big\|\langle\partial_{\tilde{\tau}}^2\rangle  \delta\Phi^{(\tilde{\lambda})}(\tau, 0)\Big\|_{\tau^{-N}L^2_{d\tau}}
				\end{align*}
			\end{lem}
			\begin{proof} Recalling \eqref{eq:LtildelambdacalK} and the analogues for $R^{(\tilde{\lambda})}_{\mathcal{K}}, R^{(\tilde{\lambda})}_{\mathcal{K},\text{small}}$, and also recalling \eqref{eq:Phitildelambdaansatz}, this is a consequence of Lemma~\ref{lem:SKdeltaPhi} together with Lemma~\ref{lem:LsmallcalKtildelambda}, ~\ref{lem:Rsmalltildelambda}.
				
			\end{proof}
			
			The desired conclusion then follows by combining the preceding Proposition~\ref{prop: solnforylambdaW} together with Lemma~\ref{lem: tildelambdaeqnsourcebounds1} , Lemma~\ref{lem:Lsmalltildelambda}, Lemma~\ref{lem:Rsmalltildelambda} as well as Lemma~\ref{lem:e1modgooderrors}, ~\ref{lem:e1modgooderrorsSK}, Lemma~\ref{Lem:LtildeK-RtildeKterm}.
		\end{proof}
		
		{\bf{Step 4}}: {\it{Solving for $\tilde{\lambda}$; completion of the proof of Proposition~\ref{prop:solnoftildelambdaeqn}.}} Thanks to the preceding Lemma and \eqref{eq:Phitildelambdaansatz}, it now suffices to solve the equation 
		\begin{equation}\label{eq:Phitildelambdaf}
			\Phi^{(\tilde{\lambda})}(\sigma, 0) = f(\sigma),
		\end{equation}
		for $\tilde{\lambda}$, where we keep in mind \eqref{eq:Phitildelambdadef}. In fact, it suffices to do so in approximate fashion, as follows: call $\tilde{\lambda}$ an $E$-approximate solution of \eqref{eq:Phitildelambdaf}, provided we have 
		\begin{equation}\label{eq:PhitildelambdafEapprox}
			\Phi^{(\tilde{\lambda})}(\sigma, 0) = f(\sigma) + \partial_{\sigma}E
		\end{equation}
		Then we use the following auxiliary 
		\begin{prop}\label{prop:Phitildelambdasolution} Assume that $\langle\partial_{\tilde{\sigma}}^2\rangle f\in \sigma^{-N}L^2_{\sigma}([\tau_*,\infty))$. Then \eqref{eq:Phitildelambdaf} is satisfied $E$-approximately by a pair functions $\tilde{\lambda}, E$ on $[\tau_*,\infty)$ satisfying the bounds 
			\begin{align*}
				&\big\|\langle\partial_{\tilde{\sigma}}^2\rangle^{-2}\partial_{\tilde{\sigma}}^2\tilde{\lambda}\big\|_{\sigma^{-N}L^2_{d\sigma}}\lesssim_N \big\|\langle\partial_{\tilde{\sigma}}^2\rangle f\big\|_{\sigma^{-N}L^2_{d\sigma}}\\
				&\big\|E\big\|_{\tau^{-N-}L^2_{d\tau}}\lesssim\big\|\langle\partial_{\tilde{\sigma}}^2\rangle f\big\|_{\sigma^{-N}L^2_{d\sigma}}.\\
			\end{align*}
			More precisely, there is a bounded complex valued function $\beta(\hat{\tilde{\tau}})\in C^\infty\big(\R\backslash\{0\}\big)$ and complex valued $ c_3(\hat{\tilde{\tau}},\nu)$, compactly supported with respect to $\hat{\tilde{\tau}}$, smooth away from $\hat{\tilde{\tau}}=0$ and satisfying symbol type bounds, such that\footnote{This formula gives a global definition of $\tilde{\lambda}$, but only the behavior of the function on $[\tau_*,\infty)$ matter to us.}
			\begin{align*}
				\mathcal{F}_{\tilde{\tau}}\big(\tilde{\lambda}_{\tilde{\tau}\tilde{\tau}}\big)(\hat{\tilde{\tau}}) + c_3(\hat{\tilde{\tau}},\nu)\cdot\mathcal{F}_{\tilde{\tau}}\big(\frac{\tilde{\lambda}_{\tilde{\tau}}}{\tilde{\tau}}\big)(\hat{\tilde{\tau}}) = \langle\hat{\tilde{\tau}}^4\rangle\cdot \beta(\hat{\tilde{\tau}})\cdot\mathcal{F}_{\tilde{\tau}}\big(\Pi^{(\tilde{\tau})}f\big)(\hat{\tilde{\tau}}) + \delta\zeta(\hat{\tilde{\tau}})
			\end{align*}
			where $\Pi^{(\tilde{\tau})}$ is a certain projection operator ensuring a finite number of vanishing conditions, the function $\beta(\hat{\tilde{\tau}})$ is smooth and bounded on $(0,\infty)$ with imaginary part non-vanishing on $(0,\infty)\backslash \{\tau_*\}$ for a $\tau_*\in (0,\infty)$, and the error term $\delta\zeta$ satisfies the bound 
			\begin{align*}
				\big\|\langle\partial_{\tilde{\tau}}^2\rangle^{-1}\mathcal{F}_{\tilde{\tau}}^{-1}(\delta\zeta)\big\|_{\tau^{-N}L^2_{d\tau}}\ll_{\tau_*}\big\|\langle\partial_{\tilde{\tau}}^2\rangle f\big\|_{\tau^{-N}L^2_{d\tau}}
			\end{align*}
			Both $\beta$ and $c_3$ satisfy the conjugation symmetry property, and $\beta$ satisfies symbol bounds. We have the limiting relation 
			\[
			\lim_{\hat{\tilde{\tau}}\rightarrow 0}\beta(\hat{\tilde{\tau}}) = \Big(-\frac{1}{2}\cdot \int_0^\infty \triangle^{-1}\big(\Lambda W\cdot W)\cdot W^2 R^3\,dR\Big)^{-1}.
			\]
			
		\end{prop}
		The technical proof of this proposition is deferred to section~\ref{sec:appendix}.
	\end{proof}
	
	\subsection{Main statement on solution of \eqref{eq:tildealpha1}} We next turn to solving the modulation equation \eqref{eq:tildealpha1}. We stress that this is in fact a delicate task, since the quantity 
	\[
	\Re\big(\mathcal{L}z\big)|_{R = 0}
	\]
	depends implicitly on $\tilde{\alpha}$ via its influence on \eqref{eq:E1mod}. Thus $\tilde{\alpha}$ will in any event be determined by application of a fixed point result. Let us denote by $E_{1}^{\text{mod},\tilde{\alpha}}$ the sum of the third, fourth and fifth term in \eqref{eq:E1mod}.
	\begin{lem}\label{lem:tildealphafixedpoint1} Given $f\in \tau^{-N}L^2_{d\tau}$ on $[\tau_*,\infty)$, the equation 
		\begin{align*}
			\tilde{\alpha}_{\tau} - \Re\int_0^\infty\xi^2\big(S+S_{\mathcal{K}}\big)\big(E_{1}^{\text{mod},\tilde{\alpha}}\big)\rho(\xi)\,d\xi = f(\tau)
		\end{align*}
		admits a solution $\alpha\in \tau^{1-N}L^2_{d\tau}$ with $\alpha_{\tau}\in \tau^{-N}L^2_{d\tau}$ and depending linearly on $f$. A similar conclusion applies if $f\in \log^{-j}(\tau)\tau^{-N}L^2_{d\tau}$, $j\geq 1$. 
	\end{lem}
	\begin{proof}  We will show that the expression 
		\[
		\Re\int_0^\infty\xi^2\big(S+S_{\mathcal{K}}\big)\big(E_{1}^{\text{mod},\tilde{\alpha}}\big)\rho(\xi)\,d\xi
		\]
		is perturbative with respect to $\tilde{\alpha}$, and so the assertion follows by means of Banach's principle. We commence with the main part, which is 
		\[
		\Re\int_0^\infty\xi^2 S\big(E_{1}^{\text{mod},\tilde{\alpha}}\big)\rho(\xi)\,d\xi
		\]
		Recalling the terms in  \eqref{eq:E1mod}, as well as \eqref{eq:ImzR0}, \eqref{eq:ReR0},  we can write the preceding expression as 
		\[
		\sum_{j=1}^4  \Phi_j^{(\tilde{\alpha})}(\tau),
		\]
		where we set 
		\begin{align*}
			&\Phi_1^{(\tilde{\alpha})}(\tau): = \int_{\tau}^\infty\int_0^\infty \xi^2S_2(\tau,\sigma,\xi)\tilde{\alpha}(\sigma)\mathcal{F}\big(\frac{\lambda_{\sigma}}{\lambda}W\cdot(R\partial_R)\chi_1\big)(\frac{\lambda(\tau)}{\lambda(\sigma)}\xi)\rho(\xi)\,d\xi d\sigma,
		\end{align*}
		and further 
		\begin{align*}
			&\Phi_2^{(\tilde{\alpha})}(\tau): =\int_{\tau}^\infty\int_0^\infty \xi^2 S_1(\tau,\sigma,\xi)\tilde{\alpha}(\sigma)\mathcal{F}\big(-\mathcal{L}(\chi_1W)\big)(\sigma, \frac{\lambda(\tau)}{\lambda(\sigma)}\xi)\rho(\xi)\,d\xi d\sigma.
		\end{align*}
		Finally, we set 
		\begin{align*}
			&\Phi_3^{(\tilde{\alpha})}(\tau): =\int_{\tau}^\infty\int_0^\infty \xi^2 S_2(\tau,\sigma,\xi)\tilde{\alpha}_{\sigma}(\sigma)\mathcal{F}\big(\chi_1W\big)(\sigma, \frac{\lambda(\tau)}{\lambda(\sigma)}\xi)\rho(\xi)\,d\xi d\sigma,
		\end{align*}
		and we let $\Phi_4^{(\tilde{\alpha})}(\tau)$ incorporate all the errors which arose when we replaced $\chi_1\psi_*^{(\tilde{\lambda})}$ by $\chi_1\lambda W$. 
		\\
		For the first term $\Phi_1^{(\tilde{\alpha})}(\tau)$, we note that the Fourier coefficient is localized at small frequencies $\xi\lesssim \sigma^{-\frac12+}$, since in the complementary situation, repeated integration by parts leads to the bound\footnote{Recall the support properties of $\chi_1$ stated before \eqref{eq:chi1vanishingcond}.}
		\begin{align*}
			\big|\mathcal{F}\big(\frac{\lambda_{\sigma}}{\lambda}W\cdot(R\partial_R)\chi_1\big)(\frac{\lambda(\tau)}{\lambda(\sigma)}\xi)\big|\lesssim \sigma^{-10},\,\xi\gtrsim \sigma^{-\frac12+},
		\end{align*}
		provided $N$ is sufficiently large. 
		We may then further restrict to $\sigma-\tau\gtrsim \tau^{\frac14}$, say, since else using the vanishing of $S_2(\tau,\sigma;\xi)$ at $\xi = 0$ we infer the bound\footnote{The added subscripts refer to the additional localizations in the double integral.}
		\begin{align*}
			\big\|\Phi_{1,\,\xi<\sigma^{-\frac12+},\sigma<\tau+\tau^{\frac14}}^{(\tilde{\alpha})}(\tau)\big\|_{\tau^{-N-\frac14+}L^2_{d\tau}}\lesssim \big\|\frac{\tilde{\alpha}}{\tau}\big\|_{\tau^{-N}L^2_{d\tau}}, 
		\end{align*}
		To deal with the case $\sigma - \tau>\tau^{\frac14}$, we take advantage of Remark~\ref{rem:tildeKfcontrol} as well as the bound $\big|\mathcal{F}\big(W\cdot R\partial_R(\chi_1)\big)\big|\lesssim \log\sigma$, following from the asymptotics in subsection~\ref{subsec:basicfourier}, resulting in the bound 
		\begin{align*}
			\big\|\Phi_{1,\,\xi<\sigma^{-\frac12+},\sigma>\tau+\tau^{\frac14}}^{(\tilde{\alpha})}(\tau)\big\|_{\tau^{-N}L^2_{d\tau}}\lesssim \big\|\frac{\tilde{\alpha}}{\tau}\big\|_{\tau^{-N}L^2_{d\tau}}\lesssim \frac{1}{N}\cdot \big\|\tilde{\alpha}_{\tau}\big\|_{\tau^{-N}L^2_{d\tau}}.
		\end{align*}
		In order to control the second term $\Phi_2^{(\tilde{\alpha})}(\tau)$, we observe that 
		\begin{align*}
			&\frac{\lambda^2(\sigma)}{\lambda^2(\tau)}S_1(\tau,\sigma,\xi)\cdot\mathcal{F}\big(-\mathcal{L}(\chi_1W)\big)(\sigma, \frac{\lambda(\tau)}{\lambda(\sigma)}\xi)\\
			& = \frac{\lambda^2(\sigma)}{\lambda^2(\tau)}\partial_{\sigma}\big(S_2(\tau,\sigma,\xi)\big)\cdot \mathcal{F}\big(\chi_1W\big)(\sigma, \frac{\lambda(\tau)}{\lambda(\sigma)}\xi).
		\end{align*}
		Performing integration by parts with respect to $\sigma$, we can then replace the second term by a linear combination of the schematically written integrals
		\begin{align*}
			&\int_{\tau}^\infty\int_0^\infty \xi^2 S_2(\tau,\sigma,\xi)\frac{\tilde{\alpha}(\sigma)}{\sigma}\cdot\mathcal{F}\big(\chi_1W\big)(\sigma, \frac{\lambda(\tau)}{\lambda(\sigma)}\xi)\rho(\xi)\,d\xi d\sigma,\\
			&\int_{\tau}^\infty\int_0^\infty \xi^2 S_2(\tau,\sigma,\xi)\frac{\tilde{\alpha}(\sigma)}{\sigma}\cdot(\xi\partial_{\xi})\mathcal{F}\big(\chi_1W\big)(\sigma, \frac{\lambda(\tau)}{\lambda(\sigma)}\xi)\rho(\xi)\,d\xi d\sigma,
		\end{align*}
		while the is a cancellation with the third term $\Phi_3^{(\tilde{\alpha})}(\tau)$. But then taking advantage of the cancellation condition \eqref{eq:chi1vanishingcond} and a further integration by parts with respect to $\sigma$ for either expression, we arrive at boundary terms $c\frac{\tilde{\alpha}(\tau)}{\tau}$ as well as double integrals which can be handled by means of Lemma~\ref{lem:K_frefined} by
		\begin{align*}
			\lesssim \big\|\frac{\tilde{\alpha}}{\tau}\big\|_{\tau^{-N}L^2_{d\tau}} + \big\|\tilde{\alpha}_{\tau}\big\|_{\tau^{-N-}L^2_{d\tau}}\lesssim N^{-1}\big\|\tilde{\alpha}_{\tau}\big\|_{\tau^{-N}L^2_{d\tau}},
		\end{align*}
		A similar argument applies if we work with the norm of  $\log^{-j}(\tau)\cdot \tau^{-N}L^2_{d\tau}$ instead. The estimate for the contribution of $S_{\mathcal{K}}$ as well as the errors due to replacing $\psi_*^{(\tilde{\lambda})}$ by $\lambda W$ are relegated to section~\ref{sec:appendix}. 
	\end{proof}
	
	We next need to derive somewhat improved bounds for the source terms $f$ appearing in the preceding lemma, as furnished by those parts of \eqref{eq:tildealpha1} which are independent of $\tilde{\alpha}$. Here we state the 
	\begin{lem}\label{lem:tildealphasourceterms} Let $f$ be the right hand side of \eqref{eq:tildealpha1}, modified by letting $z$ be defined as solution of \eqref{eq:zeqn2} but with $e_1^{\text{mod}}$ replaced by 
		\[
		e_1^{\text{mod}} - e_1^{\text{mod},\tilde{\alpha}}.
		\]
		Then we have the bound 
		\begin{align*}
			\big\|f\big\|_{\log^{-1}\tau\cdot\tau^{-N}L^2_{d\tau}}\lesssim \big\|z_{nres}\big\|_{S} + \big\|\partial_{\tilde{\tau}}^2\langle\partial_{\tilde{\tau}}^2\rangle^{-1}\tilde{\lambda}\big\|_{\tau^{-N}L^2_{d\tau}} + \big\|(\tilde{\kappa}_1,\kappa_2)\big\|_{\tau^{-N}L^2_{d\tau}}
		\end{align*}
	\end{lem}
	\begin{proof} As usual the main difficulty comes from the contribution of the delicate term $-\lambda^{-2}y_z\cdot W$ to both $\mathcal{L}z|_{R = 0}$ (expressed via the Schr\"odinger propagator) as well as to 
		\[
		- \Re\big(\lambda^{-2}(y \tilde{u}_*^{(\tilde{\lambda}, \underline{\tilde{\alpha}})})\big)|_{R = 0}
		\]
		In light of \eqref{eq:ReR0} and Lemma~\ref{lem:approxsolasymptotics1}, the sum of this term and $-\lambda^{-2}y_z\cdot W$ can almost be bounded by means of Lemma~\ref{lem:Lfdifferencebound}, but we have to make sure not to incur a temporal decay loss due to application of Corollary~\ref{cor:yzW} and the remark following it.  This requires a sharper analysis of the proof of Lemma~\ref{lem:Lfdifferencebound}: notice that due to Corollary~\ref{cor:yzWpartialtau}, and choosing $f = \langle \lambda^{-2}y_z\cdot W, \phi(R;\xi)\rangle$, we may pick $\delta = \frac12+\frac{1}{4\nu}-$ in the proof, while the norm $\big\|\langle \xi\partial_{\xi}\rangle^{1+\delta_1}f\big\|_{\tau^{-N}L^2_{d\tau}L^2_{\rho(\xi)\,d\xi}}$ is only required in the last part of the proof, concerning the double integral over $\sigma - \tau>\tau^{\delta}$.  Now Lemma~\ref{lem:wavebasicinhomstructure2} allows us to reduce the variable $R$ in the inner product defining $f$ to size $<\tau^{\frac12-\frac{1}{4\nu}+}$ (up to contributions which can be handled directly by Lemma~\ref{lem:Lfdifferencebound}), and then combining the oscillatory phases $e^{\pm i(\sigma - \tau)\xi^2}, e^{\pm iR\xi}$, the combined phase $\pm i(\sigma - \tau)\xi^2 \pm iR\xi$ is in the non-stationary regime except when $\xi< \tau^{-\frac{1}{2\nu}+}$. In the non-stationary regime we can then perform the integration by parts with respect to $\xi$ while avoiding the bad term $\langle\xi\partial_{\xi}\rangle^{1+\delta_1} f$ completely, while on the stationary regime, we can rely again on Lemma~\ref{lem:wavebasicinhomstructure2}.

		The remaining technical contributions are treated in section~\ref{sec:appendix}. 
	\end{proof}
	
	Combining the preceding two lemmas, we can now formulate our conclusion concerning the solution of \eqref{eq:tildealpha1}:
	\begin{prop}\label{prop:tildealphomodeqn} Assume that $(z,y)$ in \eqref{eq:tildealpha1} satisfy \eqref{eq:zeqn2}, with $e_1^{\text{mod}}$ implicitly depending on $\tilde{\alpha}$ via \eqref{eq:E1mod}. Then \eqref{eq:tildealpha1} admits a solution $\tilde{\alpha}$ satisfying the bound 
		\begin{align*}
			\Big\|\tilde{\alpha}_{\tau}\Big\|_{\log^{-1}\tau\cdot\tau^{-N}L^2_{d\tau}}\lesssim  \big\|z_{nres}\big\|_{S} + \big\|\partial_{\tilde{\tau}}^2\langle\partial_{\tilde{\tau}}^2\rangle^{-1}\tilde{\lambda}\big\|_{\tau^{-N}L^2_{d\tau}} + \big\|(\tilde{\kappa}_1,\kappa_2)\big\|_{\tau^{-N}L^2_{d\tau}}.
		\end{align*}
	\end{prop}
	
	\subsection{Final conclusion on modulation parameters $(\tilde{\alpha}, \tilde{\lambda})$.} In the preceding subsections we considered the modulation equations \eqref{eq:tildelambda}, \eqref{eq:tildealpha1} separately. However, in the sequel we shall need to solve them simultaneously, assuming the data $(z_{nres}, \tilde{\kappa}_1,\kappa_2)$ as given. Thus we formulate 
	\begin{prop}\label{prop:tilealphatildelambdasimultaneous} Assume that $(z, y)$, with $z$ given by \eqref{eq:zedcomprefined}, solve \eqref{eq:zeqn2}, \eqref{eq:y2def}, \eqref{eq:yzdfn} on $[\tau_*,\infty)$ (with solutions vanishing at $\tau = +\infty$), where $e_1^{\text{mod}}, E_2^{\text{mod}}$ depend on $\tilde{\alpha}, \tilde{\lambda}$ via \eqref{eq:E2mod}, \eqref{eq:E1mod}, \eqref{eq:e1moddef}. Then the combined system \eqref{eq:tildealpha1}, \eqref{eq:tildelambda} admits a solution $(\tilde{\alpha}, \tilde{\lambda})$ satisfying for $\tilde{\lambda}$ the conclusion of Proposition~\ref{prop:solnoftildelambdaeqn} with the corresponding inequalities without the $\tilde{\alpha}$-terms, and for $\tilde{\alpha}$ the conclusion of Proposition~\ref{prop:tildealphomodeqn} without the $\tilde{\lambda}$-dependent term. 
	\end{prop}
	\begin{proof} The final conclusion of Proposition~\ref{prop:solnoftildelambdaeqn}  shows that the $\tilde{\alpha}$-dependence of $\tilde{\lambda}$ is perturbative (due to the factor $c(\tau_*)$), and so the conclusion follows from a simple fixed point/iteration argument. 
		
	\end{proof}

	\section{Improved bounds for the resonant part; estimates for $(\tilde{\kappa}_1, \kappa_2)$}\label{sec:improvkappabounds}
	
	In this section we finally derive improved bounds on the parameters $(\tilde{\kappa}_1, \kappa_2)$, which, together with $\tilde{\lambda}$, describe the resonant part of $z$ by means of \eqref{eq:zedcomprefined}. The resulting proposition, in conjunction with Proposition~\ref{prop:tildealphomodeqn}  as well as Proposition~\ref{prop:solnoftildelambdaeqn} will then allow us to infer a priori bounds on the quadruple $\big(\tilde{\lambda}, \tilde{\alpha}, \tilde{\kappa}_1, \kappa_2\big)$ solely in terms of the remaining variable $z_{nres}(\tau, R)$, and reduce the problem to deriving a priori bounds on the latter. 
	We have the following 
	\begin{prop}\label{prop:tildekappa1kappa2apriori} Assume that $\tilde{\lambda}$ satisfies \eqref{eq:tildelambda} and $\tilde{\alpha}$ satisfies \eqref{eq:tildealpha1}, where it is understood that $(z, y)$ solve \eqref{eq:zeqn2}, \eqref{eq:y2def}, \eqref{eq:yzdfn}. Then the solutions of \eqref{eq:kappa1eqn}, \eqref{eq:kappa2eqn} satisfy the improved bound 
		\begin{align*}
			\Big\|(\tilde{\kappa}_1, \kappa_2)\Big\|_{\tau^{-N}L^2_{d\tau}}\leq c(\tau_*)\cdot\big[\big\|z_{nres}\big\|_{S} + \big\|(\tilde{\kappa}_1,\kappa_2)\big\|_{\tau^{-N}L^2_{d\tau}}\big]
		\end{align*}
		where we have the asymptotic relation $\lim_{\tau_*\rightarrow\infty}c(\tau_*) = 0$. In particular, we infer the bound 
		\begin{align*}
			\Big\|(\tilde{\kappa}_1, \kappa_2)\Big\|_{\tau^{-N}L^2_{d\tau}}\leq c(\tau_*)\cdot\big\|z_{nres}\big\|_{S} 
		\end{align*}
		for $\tau_*$ sufficiently large. 
	\end{prop}
	\begin{proof} We split things into the contribution to the real part of the resonant component, given by $\tilde{\kappa}_1$ (recall \eqref{eq:zedcomprefined}), as well as the imaginary part, given by $\kappa_2$.
		\\
		\subsubsection{Improved control over $\tilde{\kappa}_1$.}\label{subsubsec:tildekappaoneimprov} We need to analyze how the equation \eqref{eq:tildelambda} improves the right hand side of \eqref{eq:kappa1eqn}, keeping in mind \eqref{eq:ImzR0} which is valid for real-valued source terms $E$. Also recall the notation from Prop.~\ref{prop:linpropagator}. Now there are several sources which contribute to the fact of a non-trivial $\tilde{\kappa}_1$: 
		\begin{enumerate}
			\item Temporal frequency cutoffs introduced in \eqref{eq:Xdef}.
			\item Temporal frequency localizers $Q^{(\tilde{\sigma})}_{<\sqrt{\gamma}^{-1}}$ in \eqref{eq:Ltildelambda}, \eqref{eq:Rtildelambda}, and analogously for $L_{\mathcal{K},\text{small}}^{(\tilde{\lambda})}$, $R_{\mathcal{K},\text{small}}^{(\tilde{\lambda})}$.
			\item Temporal frequency cutoffs introduced in \eqref{eq:tildeE2mod}.
			\item The temporal frequency localization applied to $e_1^{\text{mod}}$  in  \eqref{eq:tildelambda}. 
			\item The fact that we solve \eqref{eq:tildelambda} only approximately, namely up to the term $\partial_{\tau}E$. 
		\end{enumerate}
		
		We start with the terms coming from (ii), which as far as $R_{\text{small}}^{(\tilde{\lambda})}$ is concerned in effect have the dominant effect and in are responsible for the term $c_*\cdot Q^{(\tilde{\tau})}_{\geq\gamma^{-1}}\tilde{\lambda}\cdot \phi_0(R)$ in the decomposition \eqref{eq:zedcomprefined}, up to a better error term which gets incorporated into $\tilde{\kappa}_1$:
		\begin{lem}\label{lem:tildekappaoneiicontrib1} Recalling \eqref{eq:Rtildelambda}, and defining 
			\[
			R^{(\tilde{\lambda})}_{\text{small},\geq\sqrt{\gamma}^{-1}}
			\]
			analogously but replacing $Q^{(\tilde{\sigma})}_{<\sqrt{\gamma}^{-1}}$ by $Q^{(\tilde{\sigma})}_{\geq \sqrt{\gamma}^{-1}}$, we have the relation 
			\begin{align*}
				R^{(\tilde{\lambda})}_{\text{small},\geq\sqrt{\gamma}^{-1}} = c_*\cdot \partial_{\tau}\big(Q^{(\tilde{\tau})}_{\geq \sqrt{\gamma}^{-1}}\tilde{\lambda}\big) + \partial_{\tau}\big(\delta R^{(\tilde{\lambda})}_{\text{small},\geq\sqrt{\gamma}^{-1}}\big)
			\end{align*}
			where the error satisfies the bound 
			\begin{align*}
				\Big\|\delta R^{(\tilde{\lambda})}_{\text{small},\geq\sqrt{\gamma}^{-1}}\Big\|_{\tau^{-N}L^2_{d\tau}}\leq c(\tau_*)\cdot \big\|\langle\partial_{\tilde{\tau}}^2\rangle^{-1}\partial_{\tilde{\tau}}^2\tilde{\lambda}\big\|_{\tau^{-N}L^2_{d\tau}}. 
			\end{align*}
			with $\lim_{\tau_*\rightarrow\infty}c(\tau_*) = 0$. 
		\end{lem}
		\begin{cor}\label{cor:Rtildelambdageqgamma-1maineffect} The solution of the following differential equation 
			\begin{align*}
				\kappa_{1,\tau}(\tau) + \frac{\lambda_{\tau}}{\lambda}\kappa_1(\tau) = R^{(\tilde{\lambda})}_{\text{small},\geq\sqrt{\gamma}^{-1}}
			\end{align*}
			and vanishing to order at least $\tau^{-(N-2)}$ at $\tau = +\infty$ satisfies 
			\[
			\kappa_1(\tau) =  c_*\cdot Q^{(\tilde{\tau})}_{\geq \sqrt{\gamma}^{-1}}\tilde{\lambda} + \delta\kappa_1
			\]
			where we have the bound 
			\[
			\big\| \delta\kappa_1\big\|_{\tau^{-N}L^2_{d\tau}}\leq c(\tau_*)\cdot \big\|\langle\partial_{\tilde{\tau}}^2\rangle^{-1}\partial_{\tilde{\tau}}^2\tilde{\lambda}\big\|_{\tau^{-N}L^2_{d\tau}}. 
			\]
		\end{cor}
		\begin{proof}(Cor.) We can write 
			\begin{align*}
				\kappa_1(\tau) &= \lambda^{-1}(\tau)\cdot\int_{\tau}^\infty \lambda(s)\cdot R^{(\tilde{\lambda})}_{\text{small},\geq\sqrt{\gamma}^{-1}}(s)\,ds\\
				& = c_*\cdot Q^{(\tilde{\tau})}_{\geq \sqrt{\gamma}^{-1}}\tilde{\lambda} + \delta\kappa_1,
			\end{align*}
			where we have 
			\begin{align*}
				\delta\kappa_1(\tau) &= -\lambda^{-1}(\tau)\cdot\int_{\tau}^\infty \lambda_s(s)\cdot Q^{(\tilde{\tau})}_{\geq \sqrt{\gamma}^{-1}}\tilde{\lambda}(s)\,ds\\
				& - \delta R^{(\tilde{\lambda})}_{\text{small},\geq\sqrt{\gamma}^{-1}}(\tau)  -\lambda^{-1}(\tau)\cdot\int_{\tau}^\infty \lambda_s(s)\cdot \delta R^{(\tilde{\lambda})}_{\text{small},\geq\sqrt{\gamma}^{-1}}(s)\,ds
			\end{align*}
			It is easy to see that the last two terms satisfy the bound stipulated for $\delta\kappa_1$ in the corollary. As for the first term on the right, observe that we can write 
			\begin{align*}
				Q^{(\tilde{\tau})}_{\geq \sqrt{\gamma}^{-1}}\tilde{\lambda} = \frac{\partial\tau}{\partial_{\tilde{\tau}}}\cdot \partial_{\tau}\big(\partial_{\tilde{\tau}}^{-1} Q^{(\tilde{\tau})}_{\geq \sqrt{\gamma}^{-1}}\tilde{\lambda}\big).
			\end{align*}
			Insert this into the integral (with $\tilde{\tau}, \tau$ replaced by $\tilde{s}, s$) and perform integration by parts with respect to $s$. This easily gives 
			\begin{align*}
				\Big\|\lambda^{-1}(\tau)\cdot\int_{\tau}^\infty \lambda_s(s)\cdot Q^{(\tilde{\tau})}_{\geq \sqrt{\gamma}^{-1}}\tilde{\lambda}(s)\,ds\Big\|_{\tau^{-N-\frac12+\frac{1}{4\nu}}L^2_{d\tau}}\lesssim \big\|Q^{(\tilde{\tau})}_{\geq \sqrt{\gamma}^{-1}}\tilde{\lambda}\big\|_{\tau^{-N}L^2_{d\tau}}, 
			\end{align*}
			which is easily seen to be compatible with the required bound for $\delta \kappa_1$. 
		\end{proof}
		\begin{proof}(Lemma~\ref{lem:tildekappaoneiicontrib1}) The key point is that we can replace $R^{(\tilde{\lambda})}_{\text{small},\geq\sqrt{\gamma}^{-1}}$ by 
			\begin{align*}
				&X(\tau): =\\
				& \int_{\tau}^\infty\int_0^\infty \xi^2\cdot \cos\big((\sigma - \tau)\xi^2\big)\cdot(\triangle)\mathcal{F}\big(\lambda^{-2}Q^{(\tilde{\sigma})}_{\geq\gamma^{-1}}\tilde{y}_{\tilde{\lambda}}^{\text{mod}}\cdot W\big)(\sigma, \frac{\lambda(\tau)}{\lambda(\sigma)}\xi)\rho(\xi)\,d\xi d\sigma
			\end{align*}
			and keeping \eqref{eq:tildeE2mod} in mind, we can in fact replace 
			\[
			\lambda^{-2}Q^{(\tilde{\sigma})}_{\geq\gamma^{-1}}\tilde{y}_{\tilde{\lambda}}^{\text{mod}}\cdot W
			\]
			by 
			\[
			Q^{(\tilde{\sigma})}_{\geq\gamma^{-1}}\tilde{\lambda}\cdot \Lambda W\cdot W^2, 
			\]
			and omit the scaling factor $\frac{\lambda(\tau)}{\lambda(\sigma)}$, 
			all up to errors which can be placed into 
			\[
			\partial_{\tau}\big(\delta R^{(\tilde{\lambda})}_{\text{small},\geq\sqrt{\gamma}^{-1}}\big)
			\]
			as verified in section~\ref{sec:appendix}. Furthermore, in light of Lemma~\ref{lem:wavetoSchrodfreqloc} and using the notation there, we can replace $Q^{(\tilde{\sigma})}_{\geq\gamma^{-1}}\tilde{\lambda}$ by $\tilde{\tilde{\lambda}}^{(\gamma^{-1})}$. But then proceeding as for the proof of Lemma~\ref{lem:Fouriertransform1} and applying the (Schr\"odinger) temporal Fourier transform, we arrive at  
			\begin{align*}
				\hat{X}(\tau) &= \Big(c_1\sqrt{|\hat{\tau}|}\rho_1(\sqrt{|\hat{\tau}|})\mathcal{F}\big( \Lambda W\cdot W^2\big)(\sqrt{|\hat{\tau}|})\\&\hspace{1cm}
				+ ic_2\int_0^\infty \frac{\hat{\tau}\sqrt{\xi}_1}{\hat{\tau}^2 - \xi_1^2}\rho_1(\sqrt{\xi}_1)\mathcal{F}\big( \Lambda W\cdot W^2\big)(\sqrt{\xi}_1)\,d\xi_1\Big)\cdot \mathcal{F}_{\tau}\big(\tilde{\tilde{\lambda}}^{(\gamma^{-1})}\big)(\hat{\tau}).
			\end{align*}
			Exploiting the fact that $\mathcal{F}_{\tau}\big(\tilde{\tilde{\lambda}}^{(\gamma^{-1})}\big)$ is supported at Schr\"odinger time frequency $\ll \tau^{0-}$, as well as the fact that 
			\begin{align*}
				\mathcal{F}\big( \Lambda W\cdot W^2\big)(0) = 0, 
			\end{align*}
			we write (recalling \eqref{eq:cstardef})
			\begin{align*}
				X(\tau)|_{\tau\geq \tau_*} = c_*\cdot \partial_{\tau}\big(\tilde{\tilde{\lambda}}^{(\gamma^{-1})}) + \partial_{\tau}Y(\tau), 
			\end{align*}
			where we set 
			\[
			Y(\tau) = \mathcal{F}_{\tau}^{-1}\big(\zeta(\hat{\tau})\cdot \mathcal{F}_{\tau}\big(\tilde{\tilde{\lambda}}^{(\gamma^{-1})}\big)(\hat{\tau})\big),
			\]
			and where we use the notation 
			\begin{align*}
				\zeta(\hat{\tau}): &= -ic_1\sqrt{|\hat{\tau}|}\rho_1(\sqrt{|\hat{\tau}|})\frac{\mathcal{F}\big( \Lambda W\cdot W^2\big)(\sqrt{|\hat{\tau}|})}{\hat{\tau}}\\
				& + c_2\int_0^\infty \frac{\sqrt{\xi}_1}{\hat{\tau}^2 - \xi_1^2}\rho_1(\sqrt{\xi}_1)\mathcal{F}\big( \Lambda W\cdot W^2\big)(\sqrt{\xi}_1)\,d\xi_1 - c_*
			\end{align*}
			Since $\big|\zeta(\hat{\tau})\big|\lesssim \log^{-2}(|\hat{\tau}|)$ as $|\hat{\tau}|\ll 1$, and enjoys symbol type bounds, we infer from Lemma~\ref{lem:wavetoSchrodfreqloc} the bound
			\begin{align*}
				\big\|Y(\tau)\big\|_{\log^{-2}(\tau)\cdot \tau^{-N}L^2_{d\tau}}\lesssim \big\|Q^{(\tilde{\tau})}_{\geq \gamma^{-1}}\tilde{\lambda}\big\|_{\tau^{-N}L^2_{d\tau}}\lesssim \big\|\langle\partial_{\tilde{\tau}}^2\rangle^{-2}\partial_{\tilde{\tau}}^2\tilde{\lambda}\big\|_{\tau^{-N}L^2_{d\tau}}. 
			\end{align*}
		\end{proof}
		To complete the contributions to the evolution of $\tilde{\kappa}_1$ contributed by (ii), we rely on 
		\begin{lem}\label{lem:tildekappaoneiicontrib2} Defining $L^{(\tilde{\lambda})}_{\text{small}, \geq \sqrt{\gamma}^{-1}}$ analogously to $R^{(\tilde{\lambda})}_{\text{small},\geq\sqrt{\gamma}^{-1}}$ (keeping in mind \eqref{eq:Ltildelambda}), we have 
			\begin{align*}
				L^{(\tilde{\lambda})}_{\text{small}, \geq \sqrt{\gamma}^{-1}} = \partial_{\tau}\big(\tilde{L}^{(\tilde{\lambda})}_{\text{small}, \geq \sqrt{\gamma}^{-1}}\big) + \frac{1}{\tau}\cdot \tilde{\tilde{L}}^{(\tilde{\lambda})}_{\text{small}, \geq \sqrt{\gamma}^{-1}}
			\end{align*}
			where we have the bound 
			\begin{align*}
				\Big\|\tilde{L}^{(\tilde{\lambda})}_{\text{small}, \geq \sqrt{\gamma}^{-1}}\Big\|_{\tau^{-N}L^2_{d\tau}}\leq c(\tau_*)\cdot \big[\big\|\langle\partial_{\tilde{\tau}}^2\rangle^{-1}\partial_{\tilde{\tau}}^2\tilde{\lambda}\big\|_{\tau^{-N}L^2_{d\tau}} + \big\|z_{nres}\big\|_{S}+\big\|(\tilde{\kappa}_1,\kappa_2)\big\|_{\tau^{-N}L^2_{d\tau}}\big]
			\end{align*}
			and similarly for $\tilde{\tilde{L}}^{(\tilde{\lambda})}_{\text{small}, \geq \sqrt{\gamma}^{-1}}$, where $c(\tau_*)\longrightarrow 0$ as $\tau_*\rightarrow\infty$. Analogous bounds apply to the quantities\footnote{These two terms are defined as in \eqref{eq:Rtildelambda} but with $Q^{(\tilde{\sigma})}_{<\sqrt{\gamma}^{-1}}$ replaced by $Q^{(\tilde{\sigma})}_{\geq \sqrt{\gamma}^{-1}}$.}
			\[
			L^{(\tilde{\lambda})}_{\mathcal{K},\text{small},\geq \sqrt{\gamma}^{-1}},\,R^{(\tilde{\lambda})}_{\mathcal{K},\text{small},\geq \sqrt{\gamma}^{-1}}.
			\]
			As a consequence, the differential equation 
			\[
			\kappa_1 + \frac{\lambda_{\tau}}{\lambda}\kappa_1 = \tilde{L}^{(\tilde{\lambda})}_{\text{small}, \geq \sqrt{\gamma}^{-1}}
			\]
			admits a solution satisfying the bound 
			\begin{align*}
				\big\|\kappa_1\big\|_{\tau^{-N}L^2_{d\tau}}\leq c(\tau_*)\cdot  \big[\big\|\langle\partial_{\tilde{\tau}}^2\rangle^{-1}\partial_{\tilde{\tau}}^2\tilde{\lambda}\big\|_{\tau^{-N}L^2_{d\tau}} + \big\|\tilde{\alpha}_{\tau}\big\|_{\log^{-1}\tau\cdot \tau^{-N}L^2_{d\tau}} + \big\|z_{nres}\big\|_{S}+\big\|(\tilde{\kappa}_1,\kappa_2)\big\|_{\tau^{-N}L^2_{d\tau}}\big],
			\end{align*}
			and similarly if we substitute $L^{(\tilde{\lambda})}_{\mathcal{K},\text{small},\geq \sqrt{\gamma}^{-1}}, R^{(\tilde{\lambda})}_{\mathcal{K},\text{small},\geq \sqrt{\gamma}^{-1}}$ on the right. 
		\end{lem}
		\begin{proof}(Lemma~\ref{lem:tildekappaoneiicontrib2}) Keeping in mind \eqref{eq:Xdef}, we shall treat here the contribution of the leading term $ -Q^{(\tilde{\tau})}_{<\tau^{\frac12+}}\big(\lambda^{-2}y_z\cdot W\big)$ to $L^{(\tilde{\lambda})}_{\text{small}, \geq \sqrt{\gamma}^{-1}}$, leaving the remaining terms to section~\ref{sec:appendix}. In light of \eqref{eq:Ltildelambda}, this consists of two terms, of which we consider 
			\begin{align*}
				\Im \int_{\tau}^\infty\int_0^\infty \xi^2\cdot S(\tau, \sigma,\xi)\cdot Q^{(\tilde{\sigma})}_{\geq\sqrt{\gamma}^{-1}}(\triangle)\mathcal{F}\Big(Q^{(\tilde{\sigma})}_{<\sigma^{\frac12+}}\big(\lambda^{-2}y_z\cdot W\big)\Big)(\sigma, \frac{\lambda(\tau)}{\lambda(\sigma)}\xi)\rho_1(\xi)\,d\xi d\sigma,
			\end{align*}
			the other term being treated similarly. Also recall \eqref{eq:triangleinparentheses}. Change the integration variable to $\tilde{\xi}: = \frac{\lambda(\tau)}{\lambda(\sigma)}\xi$, and use the relation 
			\begin{align*}
				\frac{\lambda^2(\sigma)}{\lambda^2(\tau)}\cdot\xi^2\cdot \Im S(\tau, \sigma,\xi) &= \frac{\lambda^2(\sigma)}{\lambda^2(\tau)}\cdot\tilde{\xi}^2\cdot \cos\big(\lambda^2(\sigma)\tilde{\xi}^2\int_{\sigma}^{\tau}\lambda^{-2}(s)\,ds\big)\\
				& = -\partial_{\tau}\Big(\sin\big(\lambda^2(\sigma)\tilde{\xi}^2\int_{\sigma}^{\tau}\lambda^{-2}(s)\,ds\big)\Big).
			\end{align*}
			We conclude that the preceding double integral can be written as 
			\begin{align*}
				&\partial_{\tau} \int_{\tau}^\infty\int_0^\infty \tilde{\xi}^2S_2(\sigma,\tau,\tilde{\xi})\cdot Z(\sigma,\tilde{\xi})\rho_1(\frac{\lambda(\sigma)}{\lambda(\tau)}\tilde{\xi})\,d\tilde{\xi} d\sigma\\
				& -  \int_{\tau}^\infty\int_0^\infty \tilde{\xi}^2S_2(\sigma,\tau,\tilde{\xi})\cdot Z(\sigma,\tilde{\xi})\partial_{\tau}\big(\rho_1(\frac{\lambda(\sigma)}{\lambda(\tau)}\tilde{\xi})\big)\,d\tilde{\xi} d\sigma\\
			\end{align*}
			where we put 
			\begin{align*}
				Z(\sigma,\tilde{\xi}) = \frac{Q^{(\tilde{\sigma})}_{\geq\sqrt{\gamma}^{-1}}(\triangle)\mathcal{F}\Big(Q^{(\tilde{\sigma})}_{<\sigma^{\frac12+}}\big(\lambda^{-2}y_z\cdot W\big)\Big)(\sigma,\tilde{\xi})}{\tilde{\xi}^2}. 
			\end{align*}
			The conclusion of the lemma is then a consequence of Lemma~\ref{lem:yzWbound3} in conjunction with Lemma~\ref{lem:tildeKfcontrol}. 
		\end{proof}
		
		We continue with the terms due to (i), again considering the most delicate term (see \eqref{eq:Xdef})
		\[
		Q^{(\tilde{\tau})}_{\geq\tau^{\frac12+}}\big(\lambda^{-2}y_z\cdot W\big)
		\]
		To treat its contribution to $\tilde{\kappa}_1$, in light of Lemma~\ref{lem:K_frefined}, it suffices to use that 
		\begin{align*}
			&\Big\|\langle \xi\partial_{\xi}\rangle^{1+\delta_0}\mathcal{F}\big(Q^{(\tilde{\tau})}_{\geq\tau^{\frac12+}}\big(\lambda^{-2}y_z\cdot W\big)\big)\Big\|_{\tau^{-N-1}L^2_{d\tau}L^2_{\rho(\xi)\,d\xi}}\\
			&\lesssim \big\|\langle\partial_{\tilde{\tau}}^2\rangle^{-1}\partial_{\tilde{\tau}}^2\tilde{\lambda}\big\|_{\tau^{-N}L^2_{d\tau}} + \big\|z_{nres}\big\|_{S}+\big\|(\tilde{\kappa}_1,\kappa_2)\big\|_{\tau^{-N}L^2_{d\tau}}.
		\end{align*}
		This in turn is a consequence of Lemma~\ref{lem:hightemprfreqyzdotW1}. We relegate the remaining terms from (i) to section~\ref{sec:appendix}. 
		\\
		
		Continuing with the terms contributed by (iii), consider the term 
		\begin{align*}
			\lambda^{-2}\Box^{-1}Q^{(\tilde{\tau})}_{\geq 1}\big(\tilde{\lambda}_{\tilde{\tau}\tilde{\tau}}\cdot\big(\partial_{\tilde{\lambda}}n_*^{(\tilde{\lambda})}-\lambda^2\Lambda W\cdot W\big)\big)\cdot W. 
		\end{align*}
		Here we use that 
		\begin{align*}
			&\Big\|\lambda^{-2}\big(\partial_{\tilde{\lambda}}n_*^{(\tilde{\lambda})}-\lambda^2\Lambda W\cdot W\big)\Big\|_{L^2_{R^3\,dR}}\lesssim\tau^{-1+O(\frac{1}{\nu})},\\
			&\Big\|\lambda^{-2}\partial_{\tilde{\tau}}\big(\partial_{\tilde{\lambda}}n_*^{(\tilde{\lambda})}-\lambda^2\Lambda W\cdot W\big)\Big\|_{L^2_{R^3\,dR}}\lesssim\tau^{-1-\frac12+O(\frac{1}{\nu})}.
		\end{align*}
		Further using that $\partial_{\tilde{\tau}} = \frac{\partial\tau}{\partial\tilde{\tau}}\cdot \partial_{\tau}$, we see that 
		\begin{align*}
			\lambda^{-2}\Box^{-1}Q^{(\tilde{\tau})}_{\geq 1}\big(\tilde{\lambda}_{\tilde{\tau}\tilde{\tau}}\cdot\big(\partial_{\tilde{\lambda}}n_*^{(\tilde{\lambda})}-\lambda^2\Lambda W\cdot W\big)\big)\cdot W = \partial_{\tau}Z(\tau, R), 
		\end{align*}
		where we have
		\begin{align*}
			\big\|Z\big\|_{\tau^{-N-\frac12+O(\frac{1}{\nu})}L^2_{d\tau}L^2_{R^3\,dR}}\lesssim \big\|\langle\partial_{\tilde{\tau}}^2\rangle^{-1}\partial_{\tilde{\tau}}^2\tilde{\lambda}\big\|_{\tau^{-N}L^2_{d\tau}}. 
		\end{align*}
		Recalling \eqref{eq:ImzR0}, and changing the integration variable to $\tilde{\xi} = \frac{\lambda(\tau)}{\lambda(\sigma)}\xi$, as well as substituting $\partial_{\sigma}Z$ for $E$ there, we infer after integration by parts with respect to $\sigma$ and application of Lemma~\ref{lem:K_frefined}, we can write 
		\begin{align*}
			\int_{\tau}^\infty\int_0^\infty \xi^2 S_1(\tau, \sigma,\xi)\cdot\mathcal{F}\big(\partial_{\sigma}Z\big)(\sigma, \frac{\lambda(\tau)}{\lambda(\sigma)}\xi)\rho(\xi)\,d\xi d\sigma = \partial_{\tau}\tilde{Z}(\tau),
		\end{align*}
		where we have the bound 
		\begin{align*}
			\Big\|Z\Big\|_{\tau^{-N-\frac12+O(\frac{1}{\nu})}L^2_{d\tau}}\lesssim  \big\|\langle\partial_{\tilde{\tau}}^2\rangle^{-1}\partial_{\tilde{\tau}}^2\tilde{\lambda}\big\|_{\tau^{-N}L^2_{d\tau}}. 
		\end{align*}
		Then the differential equation 
		\[
		\kappa_{1,\tau} + \frac{\lambda_{\tau}}{\lambda}\kappa_1(\tau) = \partial_{\tau}\tilde{Z}(\tau)
		\]
		admits a solution satisfying the bound 
		\[
		\big\|\kappa_1\big\|_{\tau^{-N-\frac12+O(\frac{1}{\nu})}L^2_{d\tau}}\lesssim  \big\|\langle\partial_{\tilde{\tau}}^2\rangle^{-1}\partial_{\tilde{\tau}}^2\tilde{\lambda}\big\|_{\tau^{-N}L^2_{d\tau}}. 
		\]
		Taking advantage of Prop.~\ref{prop:solnoftildelambdaeqn}  this is then consistent with the estimate in Prop.~\ref{prop:tildekappa1kappa2apriori}. 
		\\
		Note that the terms (v) immediately lead to the correct bound by replacing $\partial_{\tau}\tilde{Z}(\tau)$ by $\partial_{\tau}E$ in the immediately preceding, and we leave the remaining technical estimates for the terms in (iv) to section~\ref{sec:appendix}. 
	\end{proof}
	
	\subsubsection{Improved control over $\kappa_2$} Our point of departure is \eqref{eq:tildealpha1}, which was chosen in order to improve the right hand side of \eqref{eq:kappa2eqn}. The fact that $\kappa_2$ is non-vanishing is then due to 
	\begin{enumerate}
		\item The fact that we include the multiplier $Q^{(\tilde{\tau})}_{<0}$ in \eqref{eq:tildealpha1}. 
		\item The fact that we have omitted the terms 
		\[
		-\Re\big[\big(\lambda^{-2}n_*^{(\tilde{\lambda}, \underline{\tilde{\alpha}})} - W^2\big) z\big]|_{R = 0} - \Re\big(\lambda^{-2}yz\big)|_{R = 0}
		\]
		from \eqref{eq:kappa2eqn} in \eqref{eq:tildealpha1}.
		\item The fact that we have omitted the term $ \Re\big(e_1^{\text{mod}}\big)|_{R = 0}$ in  \eqref{eq:tildealpha1}. 
	\end{enumerate}
	
	Given the tools we have at our disposal, the contribution of each of these sources to the evolution of $\kappa_2$ will be straightforward to control. 
	\\
	To begin with, for the terms (i), we observe that the differential equation 
	\[
	\kappa_{2,\tau} + \frac{\lambda_{\tau}}{\lambda}\kappa_2(\tau) = Q^{(\tilde{\tau})}_{\geq 0}f,\,f\in \tau^{-N}L^2_{d\tau}
	\]
	and vanishing at $\tau = +\infty$ admits a solution satisfying 
	\begin{align*}
		\big\|\kappa_2\big\|_{\tau^{-N+\frac12+\frac{1}{4\nu}+}L^2_{d\tau}}\lesssim \big\|f\big\|_{\tau^{-N}L^2_{d\tau}}, 
	\end{align*}
	thus involving a power loss of decay. This is a straightforward consequence of the fact that
	\begin{align*}
		Q^{(\tilde{\tau})}_{\geq 0}f = \frac{\partial\tau}{\partial\tilde{\tau}}\cdot\partial_{\tau}\big(\partial_{\tilde{\tau}}^{-1}Q^{(\tilde{\tau})}_{\geq 0}f \big) = \partial_{\tau}\big(\frac{\partial\tau}{\partial\tilde{\tau}}\cdot\partial_{\tilde{\tau}}^{-1}Q^{(\tilde{\tau})}_{\geq 0}f \big) - \partial_{\tau}\big(\frac{\partial\tau}{\partial\tilde{\tau}} \big)\cdot\partial_{\tilde{\tau}}^{-1}Q^{(\tilde{\tau})}_{\geq 0}f
	\end{align*}
	in conjunction with the bounds 
	\begin{align*}
		\big\|\frac{\partial\tau}{\partial\tilde{\tau}}\cdot\partial_{\tilde{\tau}}^{-1}Q^{(\tilde{\tau})}_{\geq 0}f \big\|_{\tau^{-N+\frac12+\frac{1}{4\nu}+}L^2_{d\tau}} + \big\|\partial_{\tau}(\frac{\partial\tau}{\partial\tilde{\tau}})\cdot\partial_{\tilde{\tau}}^{-1}Q^{(\tilde{\tau})}_{\geq 0}f \big\|_{\tau^{-N-\frac12+\frac{1}{4\nu}+}L^2_{d\tau}}\lesssim \big\|f\big\|_{\tau^{-N}L^2_{d\tau}}.
	\end{align*}
	We then let 
	\[
	f = \Re\big(\mathcal{L}z\big)|_{R = 0} - \Re\big(\lambda^{-2}(y \tilde{u}_*^{(\tilde{\lambda}, \underline{\tilde{\alpha}})})\big)|_{R = 0} + \Re( e_1)|_{R = 0},
	\]
	where $z$ in the first term on the right is given via its distorted Fourier transform by \eqref{eq:formalexpansion}, with $E$ given in turn by \eqref{eq:recallE} (recalling the first equation of \eqref{eq:zeqn2}). The desired estimate on $\big\|f\big\|_{\tau^{-N}L^2_{d\tau}}$ for the first term on the right is then a consequence of Lemma~\ref{lem:tildeKfcontrol}, Lemma~\ref{lem:concatenation1} in conjunction with Lemma~\ref{lem:yzWbound1}, Lemma~\ref{lem:Xtildelambdaperturbterms}, Lemma~\ref{lem:basicboundsfore_1modandnonlinearterms}, Lemma~\ref{lem:Xtildelambdafinaltermcrudebound} and finally the estimate
	\begin{align*}
		\big\|\langle\xi\partial_{\xi}\rangle^{1+}\mathcal{F}\big(e_1^{\text{mod}}\big)\big\|_{\tau^{-N}L^2_{d\tau}L^2_{\rho(\xi)\,d\xi}}\lesssim \big\|\tilde{\alpha}_{\tau}\big\|_{\tau^{-N}L^2_{d\tau}} + \big\|\langle\partial_{\tilde{\tau}}^2\rangle^{-1}\tilde{\lambda}_{\tilde{\tau}\tilde{\tau}}\big\|_{\tau^{-N}L^2_{d\tau}},
	\end{align*}
	which follows from the symbol behavior of the terms forming $e_1^{\text{mod}}$ as well as Plancherel's theorem for the distorted Fourier transform. The required bound for the boundary terms $ \Re\big(\lambda^{-2}(y \tilde{u}_*^{(\tilde{\lambda}, \underline{\tilde{\alpha}})})\big)|_{R = 0}, \Re( e_1)|_{R = 0}$ is a straightforward consequence of Lemma~\ref{lem:wavebasicinhom}, recalling \eqref{eq:yzdfn}, \eqref{eq:y2def}, \eqref{eq:zeqn2}. The remaining terms (ii), (iii) are also straightforward to bound:
	\begin{align*}
		\big\|\Re\big[\big(\lambda^{-2}n_*^{(\tilde{\lambda}, \underline{\tilde{\alpha}})} - W^2\big) z\big]|_{R = 0}\big\|_{\tau^{-N}L^2_{d\tau}} + \big\| \Re\big(\lambda^{-2}yz\big)|_{R = 0}\big\|_{\tau^{-N}L^2_{d\tau}} + \big\|\Re\big(e_1^{\text{mod}}\big)|_{R = 0}\big\|
	\end{align*}

	\newpage 
	
	\section{The non-resonant part I; introducing the principal variables }\label{sec:nonres}
	
	\subsection{Preliminary reductions}\label{subsec:nonresprelimreduc}
	
	Keeping in mind the key decomposition \eqref{eq:zdecompbasic}, we now turn our attention to the {\it{non-resonant part}} $z_{nres}(\tau, R)$. Once we can improve the bounds for it, then the conclusion of Prop.~\ref{prop:tildekappa1kappa2apriori} together with Prop.~\ref{prop:solnoftildelambdaeqn}, Prop.~\ref{prop:tildealphomodeqn} give improved control over all parameters describing the solution. 
	\\
	
	Our point of departure is the first equation in \eqref{eq:zeqn2}, which we solve in terms of the distorted Fourier transform of $z_{nres}$ and then use \eqref{eq:znres}. To find the distorted Fourier transform, we take advantage of the propagator \eqref{eq:formalexpansion}, where we have 
	\begin{equation}\label{eq:recallE}\begin{split}
			&E: =\\& - \lambda^{-2}y_z\cdot W-\lambda^{-2}\big(y \tilde{u}_*^{(\tilde{\lambda}, \underline{\tilde{\alpha}})} -y_z\cdot W\big) -\big(\lambda^{-2}n_*^{(\tilde{\lambda}, \underline{\tilde{\alpha}})} - W^2\big) z - \lambda^{-2}yz +  e_1 + e_1^{\text{mod}}
	\end{split}\end{equation}
	Splitting the Schr\"odinger propagator as in \eqref{eq:ScalK1}, \eqref{eq:ScalK2}, the contribution of the part $S_{\mathcal{K}}$, which is of a more complicated algebraic structure, will turn out to be pertubative. Thus our strategy shall be to write 
	\begin{equation}\label{eq:hatzsplitting}
		\hat{z} = \hat{z}_* + \hat{z}_{**}, 
	\end{equation}
	where the main part $z_*$ solves 
	\begin{equation}\label{eq:hatzstareqn}
		-i\big(\partial_{\tau} - 2\frac{\lambda_{\tau}}{\lambda} - \frac{\lambda_{\tau}}{\lambda}\xi\partial_{\xi}\big)\hat{z}_*(\xi) - \xi^2\hat{z}_*(\xi) =  \mathcal{F}\big(E\big).
	\end{equation}
	This we can solve explicitly by means of Proposition~\ref{prop:linpropagator}. The main problem in improving bounds for $z_{nres}$, given in terms of $\hat{z}$ by \eqref{eq:znres}, comes from the linear term 
	\[
	-\lambda^{-2}y_z\cdot W.
	\]
	More specifically, in light of the decomposition \eqref{eq:zdecompbasic}, where the resonant part is already better, the main difficulty comes from the term 
	\[
	- \lambda^{-2}y_{z_{nres}}\cdot W.
	\]
	Our strategy shall be to 'peel off' further perturbative layers of this term to arrive at what we call the 'principal part' of $z_{nres}$, and associated with it the 'principal part' of the potential term $y$. The starting point is the relation 
	\begin{equation}\label{eq:snresdecomp}\begin{split}
			&z_{nres}(\tau, R) =\\&(-i)\int_{\tau}^\infty \int_0^\infty \big[\phi(R;\xi) - \phi(R;0)\big]\cdot S(\tau, \sigma,\xi)\cdot \mathcal{F}\big(E\big)(\sigma, \frac{\lambda(\tau)}{\lambda(\sigma)}\xi\big)\rho(\xi)\,d\xi d\sigma\\
			& + z_{nres,\mathcal{K}}(\tau, R),
	\end{split}\end{equation}
	where we define the last term $ z_{nres,\mathcal{K}}(\tau, R)$ by the first formula on the right but with $\mathcal{F}(E)$ replaced by 
	\[
	\sum_{j=1}^\infty \big(-i\frac{\lambda_{\tau}}{\lambda}\mathcal{K}\circ S\big)^j\big(\mathcal{F}(E)\big),
	\]
	and this term will again enjoy a smallness gain. We now consider the first term on the right hand side in \eqref{eq:snresdecomp}, which we call $z_{nres*}$. Fixing a small enough absolute constant $\epsilon_1>0$, we split it into 
	\begin{equation}\label{eq:znresstardecomp}
		z_{nres*}: = z_{nres*,<\epsilon_1}(\tau, R) + z_{nres*,\sim\epsilon_1}(\tau, R) + z_{nres*,>\epsilon_1^{-1}}(\tau, R),
	\end{equation}
	these terms being obtained by including smooth cutoffs $\chi_{\xi\lesssim \epsilon_1}$, $\chi_{\xi\sim \epsilon_1}$, $\chi_{\xi\gtrsim \epsilon_1^{-1}}$ into the integral expression on the right in \eqref{eq:snresdecomp}.
	The next lemma implies that the first and third of these terms are perturbative:
	\begin{lem}\label{lem:znresperturbative1} We have the bounds 
		\begin{align*}
			\Big\|z_{nres*,<\epsilon_1}\Big\|_{S}&\ll_{\epsilon_1}\big\|z\big\|_{S} + \big\|\langle\partial_{\tilde{\tau}}^2\rangle^{-1}\partial_{\tilde{\tau}}^2\tilde{\lambda}\big\|_{\tau^{-N}L^2_{d\tau}} + \big\|(\tilde{\kappa_1},\kappa_2)\big\|_{\tau^{-N}L^2_{d\tau}} + \big\|\tilde{\alpha}_{\tau}\big\|_{\tau^{-N}L^2_{d\tau}}\\& + \big
			\| e_1\big\|_{\tau^{-N-1}L^2_{d\tau}L^2_{R^3\,dR}},
		\end{align*}
		\begin{align*}
			\Big\|z_{nres*,>\epsilon_1^{-1}}\Big\|_{S}&\ll_{\epsilon_1}\big\|z\big\|_{S} + \big\|\langle\partial_{\tilde{\tau}}^2\rangle^{-1}\partial_{\tilde{\tau}}^2\tilde{\lambda}\big\|_{\tau^{-N}L^2_{d\tau}} + \big\|(\tilde{\kappa_1},\kappa_2)\big\|_{\tau^{-N}L^2_{d\tau}} + \big\|\tilde{\alpha}_{\tau}\big\|_{\tau^{-N}L^2_{d\tau}}\\& + \big
			\|\langle \nabla^4\rangle e_1\big\|_{\tau^{-N-1}L^2_{d\tau}L^2_{R^3\,dR}}.
		\end{align*}
	\end{lem}
	\begin{proof} We need to deal with the contributions of the various terms constituting $E$, the latter given by \eqref{eq:recallE}. 
		\\
		
		{\it{The estimate for the small frequency term $z_{nres*,<\epsilon_1}$.}}
		\\
		
		Our key technical tool to achieve the required estimate is Lemma~\ref{lem:nonresbasicsmallfreq1}, and establishing the present lemma reduces to bounding the various terms constituting $E$ in terms of the norm occuring in  Lemma~\ref{lem:nonresbasicsmallfreq1}. Start with the most delicate term 
		\[
		-\lambda^{-2}y_z\cdot W.
		\]
		where we intend to combine Lemma~\ref{lem:nonresbasicsmallfreq1} with Corollary~\ref{cor:yzW}, Corollary~\ref{cor:yzWpartialtau}. Observing the three terms on the right in Lemma~\ref{lem:nonresbasicsmallfreq1}, we note that Lemma~\ref{lem:nonresbasicsmallfreq1yzW} suffices to deal with the last two, but  Corollary~\ref{cor:yzW} allows us to deal with the first of the three terms only up to a small power loss in $\tau$, let alone a smallness gain. However, restricting the frequency $\xi$ to even smaller size $\xi<\tau^{-\delta_1}$, the desired bound follows from Lemma~\ref{lem:wavebasicinhomstructure2}, while if $\xi\geq \tau^{-\delta_1}$, we can use Remark~\ref{rem:nonresbasicsmallfreq1} in conjunction with the estimate $\big\|\mathcal{F}\big(\lambda^{-2}y_z\cdot W\big)\tau,\cdot)\big\|_{\tau^{-N}(L^2_{d\tau}L^2_{\rho\,d\xi}\cap L^\infty_{d\xi})}\lesssim \big\|z\big\|_{S}$, in turn a consequence of Corollary~\ref{cor:yzW} and Remark~\ref{rem:cor:yzW}.
		
		We next deal with the contribution of the first term on the right of the first equation in \eqref{eq:zeqn2}. This can be handled by means of Lemma~\ref{lem:nonresbasicsmallfreq1}, Remark~\ref{rem:nonresbasicsmallfreq1} together with Lemma~\ref{lem:Xtildelambdaperturbterms}, Lemma~\ref{lem:ytildelambdamodhightempfreq} .
		\\
		We continue with the contribution of the more delicate term $\big(\lambda^{-2}n_*^{(\tilde{\lambda}, \underline{\tilde{\alpha}})} - W^2\big) z$ in \eqref{eq:zeqn2}. We decompose this term into a 'good' part with plenty of decay, and a more difficult part with worse decay: 
		\begin{align*}
			\big(\lambda^{-2}n_*^{(\tilde{\lambda}, \underline{\tilde{\alpha}})} - W^2\big) z = \chi_{R\lesssim \tau^{\frac12 - \frac{1}{4\nu}}}\big(\lambda^{-2}n_*^{(\tilde{\lambda}, \underline{\tilde{\alpha}})} - W^2\big) z + \chi_{R\gtrsim\tau^{\frac12 - \frac{1}{4\nu}}}\big(\lambda^{-2}n_*^{(\tilde{\lambda}, \underline{\tilde{\alpha}})} - W^2\big) z
		\end{align*}
		The contribution of the first term on the right is handled by combining Lemma~\ref{lem:Xtildelambdafinaltermcrudeboundrefined} with Lemma~\ref{lem:nonresbasicsmallfreq1}. As for the second term on the right, we can use Lemma~\ref{lem:pseudotransferenceoperator2} to reduce this to the next Lemma~\ref{lem:nonresconnect1}.
		\\
		It remains to deal with the contribution of the last three terms in \eqref{eq:recallE}. For the term $-\lambda^{-2}yz$, the desired estimate is a consequence of the last part of Lemma~\ref{lem:nonresbasicsmallfreq1} in conjunction with Lemma~\ref{lem:basicboundsfore_1modandnonlinearterms} and Plancherel's theorem for the distorted Fourier transform. The contribution of $e_1$ is handled similarly as we suppose that $e_1\in \tau^{-N-2}L^2_{d\tau}L^2_{R^3\,dR}$. It remains to deal with the source term $e_1^{\text{mod}}$ (recall \eqref{eq:e1moddef}, \eqref{eq:E1mod}). This shall be done in section~\ref{sec:appendix}. 
		\\
		The high-frequency bound involving $z_{nres*,>\epsilon_1^{-1}}$ is obtained similarly by taking advantage of Lemma~\ref{lem:nonresbasiclargefreq1}. 
	\end{proof}
	
	In order to complete the first stage of the 'peeling process' for $z_{nres}$, we also need to control the term $z_{nres,\mathcal{K}}$, which we call of 'connecting type' since it arises by a number of applications of the operator $S\circ\big(\frac{\lambda_{\tau}}{\lambda}\mathcal{K}\big)$ to $E$. Here we have 
	\begin{lem}\label{lem:nonresconnect1} The following bound obtains: for $j\geq 1$
		\begin{align*}
			&\Big\|\int_{\tau}^\infty \int_0^\infty \big[\phi(R;\xi) - \phi(R;0)\big]\cdot S\big(-i\frac{\lambda_{\tau}}{\lambda}\mathcal{K}\circ S\big)^j\big(\mathcal{F}(E)\big)(\sigma, \frac{\lambda(\tau)}{\lambda(\sigma)}\xi\big)\rho(\xi)\,d\xi d\sigma\Big\|_{S}\\
			&\leq \delta^j(N,\tau_*)\cdot \Big[\big\|z\big\|_{S} + \big\|\langle\partial_{\tilde{\tau}}^2\rangle^{-1}\partial_{\tilde{\tau}}^2\tilde{\lambda}\big\|_{\tau^{-N}L^2_{d\tau}} + \big\|(\tilde{\kappa}_1,\kappa_2)\big\|_{\tau^{-N}L^2_{d\tau}} + \big\|\tilde{\alpha}_{\tau}\big\|_{\tau^{-N}L^2_{d\tau}}\\&\hspace{2cm} + \big
			\|\langle \nabla^4\rangle e_1\big\|_{\tau^{-N-1}L^2_{d\tau}L^2_{R^3\,dR}}\Big],
		\end{align*}
		where $\lim_{N,\,\tau_*\rightarrow\infty}\delta(N,\tau_*) = 0$. In particular, we conclude that 
		\begin{align*}
			&\big\|z_{nres,\mathcal{K}}\big\|_{S}\ll_{N,\tau_*}\big\|z\big\|_{S} + \big\|\langle\partial_{\tilde{\tau}}^2\rangle^{-1}\partial_{\tilde{\tau}}^2\tilde{\lambda}\big\|_{\tau^{-N}L^2_{d\tau}} + \big\|(\tilde{\kappa}_1,\kappa_2)\big\|_{\tau^{-N}L^2_{d\tau}} + \big\|\tilde{\alpha}_{\tau}\big\|_{\tau^{-N}L^2_{d\tau}}\\&\hspace{2cm} + \big
			\|\langle \nabla^4\rangle e_1\big\|_{\tau^{-N-1}L^2_{d\tau}L^2_{R^3\,dR}}
		\end{align*}
	\end{lem}
	The proof of this lemma can be carried out analogously to the preceding one, invoking the arguments for Lemma~\ref{lem:nonresbasicsmallfreq1} and Lemma~\ref{lem:nonresbasiclargefreq1} but also taking advantage of relation \eqref{eq:SKproductexpansion1} and combining the phases as in \eqref{eq:resprodxiphase1}. The smallness comes from an additional integration over time and the rapid temporal decay. 
	\\
	
	The preceding lemmas~\ref{lem:nonresconnect1}, ~\ref{lem:znresperturbative1} allow us to reformulate \eqref{eq:snresdecomp} in the following more concise form:
	\begin{equation}\label{eq:snresdecomp1}\begin{split}
			&z_{nres}(\tau, R) =\\&(-i)\int_{\tau}^\infty \int_0^\infty\chi_{\epsilon_1\lesssim\xi\lesssim \epsilon_1^{-1}}\big[\phi(R;\xi) - \phi(R;0)\big]\cdot S(\tau, \sigma,\xi)\cdot \mathcal{F}\big(E\big)(\sigma, \frac{\lambda(\tau)}{\lambda(\sigma)}\xi\big)\rho(\xi)\,d\xi d\sigma\\
			& + z_{nres,\text{small}}(\tau, R),
	\end{split}\end{equation}
	where the last term enjoys the improved bound 
	\begin{align*}
		\big\| z_{nres,\text{small}}\big\|_{S}&\ll_{N,\tau_*}\big\|z\big\|_{S} + \big\|\langle\partial_{\tilde{\tau}}^2\rangle^{-1}\partial_{\tilde{\tau}}^2\tilde{\lambda}\big\|_{\tau^{-N}L^2_{d\tau}} + \big\|(\tilde{\kappa}_1,\kappa_2)\big\|_{\tau^{-N}L^2_{d\tau}} + \big\|\tilde{\alpha}_{\tau}\big\|_{\tau^{-N}L^2_{d\tau}}\\&\hspace{0.7cm} + \big
		\|\langle \nabla^4\rangle e_1\big\|_{\tau^{-N-1}L^2_{d\tau}L^2_{R^3\,dR}}
	\end{align*}
	
	\subsection{Introducing the principal parts of $z_{nres}$ and of $y_z$.} According to \eqref{eq:snresdecomp1} the main contribution to $z_{nres}$ comes from the integral expression on the right hand side. In turn recalling \eqref{eq:recallE} for $E$, the main contribution comes on the one hand from the first term
	\[
	-\lambda^{-2}y_z\cdot W,
	\]
	but also the delicate interaction of the modulated potential term $y_{\tilde{\lambda}}$ with the bulk term $W$, given by 
	\[
	-\lambda^{-2}y_{\tilde{\lambda}}\cdot W.
	\]
	Indeed, from Prop.~\ref{prop:solnoftildelambdaeqn}, we see that $\tilde{\lambda}$, when measured in the right norm, cannot be expected to be perturbative compared to $z_{nres}$. 
	It is thus natural to single out the contribution of these terms, which we then label the {\it{principal part}} of $z_{nres}$: define
\boxalign[12cm]{ \begin{align}\label{eq:znresprin}
			&z_{nres}^{prin}: = i\int_{\tau}^\infty \int_0^\infty\chi_{\epsilon_1\lesssim\xi\lesssim \epsilon_1^{-1}}\big[\phi(R;\xi) - \phi(R;0)\big]\\ \nonumber &\hspace{3cm}\cdot S(\tau, \sigma,\xi)\cdot \mathcal{F}\big(\lambda^{-2}(y_z + y_{\tilde{\lambda}})\cdot W\big)(\sigma, \frac{\lambda(\tau)}{\lambda(\sigma)}\xi\big)\rho(\xi)\,d\xi d\sigma
	\end{align}
}
	Also, call $z_{nres}^{rest}$ the contribution of the remaining terms in $E$, so that we have (recall \eqref{eq:znresstardecomp})
	\begin{equation}\label{eq: znresstarsimepsilon1decomp}
		z_{nres*,\sim\epsilon_1} = z_{nres}^{prin} + z_{nres}^{rest}. 
	\end{equation}
	We note right away
	\begin{lem}\label{lem:znresrestbound} We have the improved bound 
		\begin{align*}
			\big\|z_{nres}^{rest}\big\|_{S}\ll_{\tau_*}&\big\|z\big\|_{S} + \big\|\langle\partial_{\tilde{\tau}}^2\rangle^{-1}\partial_{\tilde{\tau}}^2\tilde{\lambda}\big\|_{\tau^{-N}L^2_{d\tau}} + \big\|(\tilde{\kappa}_1,\kappa_2)\big\|_{\tau^{-N}L^2_{d\tau}} + \big\|\tilde{\alpha}_{\tau}\big\|_{\tau^{-N}L^2_{d\tau}}\\& + \big
			\|e_1\big\|_{\tau^{-N-1}L^2_{d\tau}},
		\end{align*}
	\end{lem}
	\begin{proof} One argues precisely as in the proof of Lemma~\ref{lem:znresperturbative1}, noting that one gains smallness depending on $\tau_*$ for all the source terms in $E$ (recall \eqref{eq:zeqn2}) except for $\lambda^{-2}y_z\cdot W, \lambda^{-2}y_{\tilde{\lambda}}\cdot W$. 
	\end{proof}
	The preceding lemma in conjunction with Lemma~\ref{lem:nonresconnect1}, Lemma~\ref{lem:znresperturbative1} imply that 
	\begin{equation}\label{eq:znresprinsimplebound}\begin{split}
			\big\|z_{nres}^{prin}\big\|_{S}\lesssim&\big\|z\big\|_{S} + \big\|\langle\partial_{\tilde{\tau}}^2\rangle^{-1}\partial_{\tilde{\tau}}^2\tilde{\lambda}\big\|_{\tau^{-N}L^2_{d\tau}} + \big\|(\tilde{\kappa}_1,\kappa_2)\big\|_{\tau^{-N}L^2_{d\tau}} + \big\|\tilde{\alpha}_{\tau}\big\|_{\tau^{-N}L^2_{d\tau}}\\& + \big
			\|e_1\big\|_{\tau^{-N-1}L^2_{d\tau}}
	\end{split}\end{equation}

	It shall next be our goal to derive the equation driving the evolution for $ z_{nres}^{prin}$, the term $ z_{nres}^{rest}$ playing a purely perturbative role. For this recall \eqref{eq:yzdfn}, whence formally we have 
	\begin{align*}
		\lambda^{-2}y_z\cdot W = 2\lambda^{-2}\Box^{-1}\big(\lambda^2\triangle \Re\big(W\bar{z}\big)\big).
	\end{align*}
	Denoting by $P_{<a}, P_{>a}$ etc localization operators with respect to the spectral parameter associated to $\mathcal{L} = -\triangle_{R} - W^2$, we now decompose the term 
	\[
	E_{main}: =  P_{\epsilon_1\lesssim\cdot\lesssim\epsilon_1^{-1}}\big(\lambda^{-2}y_z\cdot W\big)
	\]
	as follows: 
	\begin{equation}\label{eq:ndecompose}\begin{split}
			E_{main} &= 2\big[n_{prin} + n_{*,<\epsilon_1} + n_{*,>\epsilon_1^{-1}} + n_{rest} +n_{nres,\mathcal{K}} + n_{res}\big]\cdot W\\
			&-P_{<\epsilon_1}\big[\lambda^{-2}y_z\cdot W\big] - P_{>\epsilon_1^{-1}}\big[\lambda^{-2}y_z\cdot W\big].
	\end{split}\end{equation}
	In turn the individual terms occuring here are defined as follows: the principal contribution is 
	\begin{equation}\label{eq:nprindef}
		\boxed{n_{prin}: = \lambda^{-2}\tilde{n}_{prin}: = \lambda^{-2}\Box^{-1}\big(\lambda^2\triangle \Re\big(W\overline{z_{nres}^{prin}}\big)\big),}
	\end{equation}
	while those of perturbative character are 
	\begin{equation}\label{eq:ntermlist}\begin{split}
			&n_{rest}: = \lambda^{-2}\tilde{n}_{rest}: =  \lambda^{-2}\Box^{-1}\big(\lambda^2\triangle \Re\big(W\overline{z_{nres}^{rest}}\big)\big),\\
			&n_{nres,\mathcal{K}}: =  \lambda^{-2}\tilde{n}_{nres,\mathcal{K}}: = \lambda^{-2}\Box^{-1}\big(\lambda^2\triangle \Re\big(W\overline{z_{nres,\mathcal{K}}}\big)\big),\\
			&n_{*,<\epsilon_1}:= \lambda^{-2}\tilde{n}_{*,<\epsilon_1}:=  \lambda^{-2}\Box^{-1}\big(\lambda^2\triangle \Re\big(W\overline{z_{nres*,<\epsilon_1}}\big)\big),\\
			&n_{*,>\epsilon_1^{-1}}:= \lambda^{-2}\tilde{n}_{*,<\epsilon_1}:=  \lambda^{-2}\Box^{-1}\big(\lambda^2\triangle \Re\big(W\overline{z_{nres*,>\epsilon_1^{-1}}}\big)\big),\\
	\end{split}\end{equation}
	and finally, the remaining term $n_{res}$ reflects the contribution of the resonant part of $z$, recalling \eqref{eq:zdecompbasic}:
	\begin{align*}
		n_{res}: = \lambda^{-2}\tilde{n}_{res}:=  \lambda^{-2}\Box^{-1}\big(\lambda^2\triangle \Re\big(W\overline{z_{res}}\big)\big).\\
	\end{align*}
	It is to be expected that the term $n_{prin}$ is not perturbative compared to $z_{nres}^{prin}$, and correspondingly the left and right hand side of \eqref{eq:znresprin} are  both non-perturbative in terms of their dependence on $z_{nres}^{prin}$. This will require a reformulation of this equation, and in fact our strategy shall be to change the vantage point if needed and sometimes use the principal part of $y_z$, given by $n_{prin}$, as the primary dynamical variable. Once we have improved the bounds for it, it shall be fairly straightforward to improve control over all other components introduced above, and thence over $z_{nres}$.
	
	\subsection{Derivation of the 'effective equations' for $z_{nres}^{prin}, n_{prin}$}\label{subsec:znresprinnprineqns} Consider the relation \eqref{eq:znresprin}, and apply the operator $\mathcal{L}$ to it. This replaces $\phi(R;\xi)- \phi(R;0)$ by $\xi^2\phi(R;\xi)$. Recalling Prop.~\ref{prop:linpropagator}, we have that 
	\begin{align*}
		i\xi^2 S(\tau,\sigma,\xi) = -\partial_{\sigma}\big(e^{i\lambda^2(\tau)\xi^2\int_{\sigma}^{\tau}\lambda^{-2}(s)\,ds}\big). 
	\end{align*}
	This suggests applying integration by parts with respect to $\sigma$ to the expression, resulting in the following equation 
	\begin{equation}\label{eq:znreskeyreofrmulation1}
		\mathcal{L}z_{nres}^{prin} - 2n_{prin}\cdot W - \lambda^{-2}y_{\tilde{\lambda}}\cdot W = \tilde{E}_{main} + z_{nres, small}^{prin}
	\end{equation}
	where we use the notation $ \tilde{E}_{main} = E_{main}  - 2n_{prin}\cdot W$, as well as 
	\begin{equation}\label{eq:znressmall}\begin{split}
			&z_{nres,small}^{prin}: = \int_{\tau}^\infty \int_0^\infty\chi_{\epsilon_1\lesssim\xi\lesssim \epsilon_1^{-1}}\phi(R;\xi)\\&\hspace{3cm}\cdot S(\tau, \sigma,\xi)\cdot \partial_{\sigma}\Big(\mathcal{F}\big(\lambda^{-2}(y_z + y_{\tilde{\lambda}})\cdot W\big)(\sigma, \frac{\lambda(\tau)}{\lambda(\sigma)}\xi\big)\Big)\rho(\xi)\,d\xi d\sigma
	\end{split}\end{equation}
	The two terms on the right of \eqref{eq:znreskeyreofrmulation1} will in fact turn out to be perturbative, while we can interpret $n_{prin}, y_{\tilde{\lambda}}$ as 'functions' of $z_{nres,small}^{prin}$ in light of \eqref{eq:nprindef}, Prop.~\ref{prop:solnoftildelambdaeqn}. This point of view shall be useful in the {\it{small wave temporal frequency regime}} for $z_{nres,small}^{prin}$, where the operators $\Box^{-1}$ effectively become $\triangle^{-1}$ up to small errors. 
	\\
	
	Another view point is to think of $n_{prin}$ as the primary dynamical variable, and to reformulate \eqref{eq:znreskeyreofrmulation1} as a wave type equation for $n_{prin}$, a view point which shall turn out useful in the {\it{large wave temporal frequency regime}}.  Letting $\tilde{n}_{prin} = \lambda^2n_{prin}$, we find 
	\begin{equation}\label{eq:nprinequation1}
		\Box \tilde{n}_{prin}  = \lambda^2\triangle \Re\big(W\overline{z_{nres}^{prin}}\big) =  \lambda^2\triangle \Re\big(W\mathcal{L}^{-1}(\overline{\mathcal{L}z_{nres}^{prin})}\big),
	\end{equation}
	where $\mathcal{L}^{-1}$ is defined (via the usual variation of constants formula) by imposing vanishing\footnote{In fact, this follows from the definition of $z_{nres}^{prin}$.} at $R = 0$. To complete the equation, we then use \eqref{eq:znreskeyreofrmulation1} to replace $\mathcal{L}z_{nres}^{prin}$ by 
	\[
	2n_{prin}\cdot W + y_{\tilde{\lambda}}\cdot W + \tilde{E}_{main} + z_{nres, small}^{prin}.
	\]
	Observe that we can interpret $\tilde{\lambda}$ as a 'function' of $n_{prin}$ up to smaller error terms, in light of Prop.~\ref{prop:solnoftildelambdaeqn}. Then the main part of the wave equation \eqref{eq:nprinequation1} is 
	\begin{equation}\label{eq:nprinequation2}
		\Box \tilde{n}_{prin} - 2\triangle\big(W\cdot \mathcal{L}^{-1}(\tilde{n}_{prin}\cdot W)\big) - \triangle\big(W\cdot \mathcal{L}^{-1}( y_{\tilde{\lambda}}\cdot W)\big),
	\end{equation}
	which is in fact a wave operator with both local and non-local potential terms. 
	
	\section{The non-resonant part II; improved estimates for the principal variables }\label{sec:nonresII}
	
	As indicated in subsection~\ref{subsec:znresprinnprineqns}, we shall take advantage of either \eqref{eq:znreskeyreofrmulation1} or of \eqref{eq:nprinequation2}
	to improve the bounds for the non-resonant component $z_{nres}^{prin}$, depending on the wave temporal frequency. 
	
	\subsection{Improving the bound for small wave-temporal frequencies}\label{subsec:nonresIIsmallfreq} Here our starting point is \eqref{eq:znreskeyreofrmulation1}, to which we apply the frequency localizer $Q^{(\tilde{\tau})}_{<\gamma_1}$ for some $\gamma_1 = \gamma_1(\nu)\ll 1$.
	Then the main point is that the equation simplifies, upon using the following 
	\begin{lem}\label{lem:lowtempfreqznresprin1} We have the identity 
		\begin{align*}
			Q^{(\tilde{\tau})}_{<\gamma_1}\big(2n_{prin}\cdot W\big) = 2\big(W^2\cdot\Re(\overline{z_{nres}^{prin}})\big) + F_1, 
		\end{align*}
		where we have the error bound 
		\begin{align*}
			\big\|F_1\big\|_{\tau^{-N}L^2_{d\tau}L^2_{R^3\,dR}}\ll_{\gamma_1}\big\|z_{nres}^{prin}\big\|_{S}. 
		\end{align*}
	\end{lem}
	\begin{proof} We use the decomposition $Q^{(\tilde{\tau})}_{<\gamma_1}\big(2n_{prin}\cdot W\big) = P_{<\sqrt{\gamma_1}}Q^{(\tilde{\tau})}_{<\gamma_1}\big(2n_{prin}\cdot W\big) + P_{\geq \sqrt{\gamma_1}}Q^{(\tilde{\tau})}_{<\gamma_1}\big(2n_{prin}\cdot W\big)$, where the frequency localizers $P_{<\sqrt{\gamma_1}}, P_{\geq \sqrt{\gamma_1}}$ are standard Littlewood-Paley localizers. Then we use Corollary~\ref{cor:yzW} and Remark~\ref{rem:cor:yzW} together with Bernstein's inequality $\big\|P_{<a}f\big\|_{L^2_{R^3\,dR}}\ll_a \big\|f\big\|_{L^{2-}_{R^3\,dR}}$ to infer the bound 
		\begin{align*}
			\big\|P_{<\sqrt{\gamma_1}}Q^{(\tilde{\tau})}_{<\gamma_1}\big(2n_{prin}\cdot W\big)\big\|_{\tau^{-N}L^2_{d\tau}L^2_{R^3\,dR}}\ll_{\gamma_1}\big\|z_{nres}^{prin}\big\|_{S}.
		\end{align*}
		For the high frequency term, we use that 
		\begin{align*}
			\big( \Box^{-1} - I\big)\circ \triangle^{1-}P_{\geq \sqrt{\gamma_1}}Q^{(\tilde{\tau})}_{<\gamma_1} = \big(\sum_{k=1}^\infty \triangle^{-k-1}\partial_{\tilde{\tau}\tilde{\tau}}^k\big)\circ  \triangle^{1-}P_{\geq \sqrt{\gamma_1}}Q^{(\tilde{\tau})}_{<\gamma_1},
		\end{align*}
		as well as the fact that the operator on the right maps $\tau^{-N}L^2_{d\tau}L^2_{R^3\,dR}$ into itself with operator norm $\ll_{\gamma_1}1$. Since 
		\begin{align*}
			\big\|\triangle^{0+}\big(W\cdot \overline{z_{nres}^{prin}}\big)\big\|_{\tau^{-N}L^2_{d\tau}L^2_{R^3\,dR}}\lesssim \big\|z_{nres}^{prin}\big\|_S, 
		\end{align*}
		we infer 
		\begin{align*}
			P_{\geq \sqrt{\gamma}}Q^{(\tilde{\tau})}_{<\gamma_1}\big(2n_{prin}\cdot W\big) = 2W^2\cdot \Re(\overline{z_{nres}^{prin}}) + o_{\tau^{-N}L^2_{d\tau}L^2_{R^3\,dR}}(\gamma_1^{0+}),
		\end{align*}
		and the lemma follows. 
	\end{proof}
	In order to simplify the third term on the left hand side of  \eqref{eq:znreskeyreofrmulation1} in the low temporal frequency regime, we have to take advantage of the precise formulation of Proposition~\ref{prop:solnoftildelambdaeqn}. In the following lemma we denote by $\triangle^{-1}$ the operator which is given by division by $\xi^2$ on the Fourier side.
	\begin{lem}\label{lem:lowtempfreqznresprin2} Assume that $\tilde{\lambda}$ solves \eqref{eq:tildelambda} according to Prop.~\ref{prop:solnoftildelambdaeqn}. Then we have the identity (for $\tau\in [\tau_*,\infty)$)
		\begin{align*}
			Q^{(\tilde{\tau})}_{<\gamma_1}\big(-\lambda^{-2}y_{\tilde{\lambda}}\cdot W\big) = T( Q^{(\tilde{\tau})}_{<\gamma_1}\Re(z_{nres}^{prin})) + F_2, 
		\end{align*}
		where we define the linear operator $T$ with one-dimensional range as 
		\begin{align*}
			&T(z): = \alpha_*\cdot \big(\int_0^\infty zW^3 R^3\,dR\big)\cdot \triangle^{-1}\big(\Lambda W\cdot W\big)\cdot W,\\
			&\alpha_*: =  \big(\frac{1}{2}\cdot\int_0^\infty\triangle^{-1}\big(\Lambda W\cdot W\big)\cdot W^2 R^3\,dR\big)^{-1},
		\end{align*}
		and where the error term $F_2$ enjoys the bound  
		\begin{align*}
			\big\|F_2\big\|_{\tau^{-N}L^2_{d\tau}L^2_{R^3\,dR}}\ll_{\epsilon_1, \gamma_1, N,\tau_*} \big\|z_{nres}\big\|_{S} + \big\|(\tilde{\kappa}_1,\kappa_2)\big\|_{\tau^{-N}L^2_{d\tau}},
		\end{align*}
	\end{lem}
	\begin{proof} Using that $\big\|Q^{(\tilde{\tau})}_{<\gamma_1}\frac{\tilde{\lambda}_{\tilde{\tau}}}{\tilde{\tau}}\big\|_{\tau^{-N}L^2_{d\tau}} \ll_{N}\big\|Q^{(\tilde{\tau})}_{<\gamma_1}(\tilde{\lambda}_{\tilde{\tau}\tilde{\tau}})\big\|_{\tau^{-N}L^2_{d\tau}}$, and recalling \eqref{eq:ytildelamba}, we can use the argument in the proof of the preceding lemma to conclude that
		\begin{align*}
			Q^{(\tilde{\tau})}_{<\gamma_1}\big(-\lambda^{-2}y_{\tilde{\lambda}}\cdot W\big) = 2\tilde{\lambda}_{\tilde{\tau}\tilde{\tau}}\Lambda W\cdot W^2 + \tilde{F},\,
		\end{align*}
		where $ \tilde{F}$ can be included into $F_2$. The assertion of the lemma then follows from Proposition~\ref{prop:solnoftildelambdaeqn} and Remark~\ref{rem:lem:Phitildelambdamodeleqnsmallness} ; in fact, in light of Lemma~\ref{lem:fmodif}  and its proof we can move $Q^{(\tilde{\tau})}_{<\gamma_1}$ past $\Pi^{(\tilde{\tau})}$ modulo errors in $\tau^{-100N}L^2_{d\tau}$, and the argument for the preceding lemma as well as Lemma~\ref{lem:znresrestbound}, Lemma~\ref{lem:nonresconnect1}, Lemma~\ref{lem:znresperturbative1}  yield 
		\begin{align*}
			\int_0^\infty Q^{(\tilde{\tau})}_{<\gamma_1}\lambda^{-2}\Box^{-1}\triangle\Re\big(\lambda^2 z_{nres}W\big)\cdot W^2 R^3\,dR =  \int_0^\infty Q^{(\tilde{\tau})}_{<\gamma_1}z_{nres}^{prin}W^3R^3\,dR + \tilde{E},
		\end{align*}
		where we have the error bound $\big\|\tilde{E}\big\|_{\tau^{-N}L^2_{d\tau}}\ll_{\epsilon_1, N,\tau_*}\big\|z_{nres}\big\|_{S} + \big\|(\tilde{\kappa}_1,\kappa_2)\big\|_{\tau^{-N}L^2_{d\tau}}$, where as usual all functions are restricted to $[\tau_*,\infty)$ or a sufficiently small dilate thereof. 
	\end{proof}
	
	The two preceding lemmas allow us to reformulate the small-frequency portion of \eqref{eq:znreskeyreofrmulation1}  as follows: introduce the operator 
	\begin{equation}\label{eq:tildeL}
		\tilde{\mathcal{L}}: = -\triangle - 3W^2. 
	\end{equation} 
	Then we have 
	\begin{equation}\label{eq:smallfreqznresprineffective}\begin{split}
			& \tilde{\mathcal{L}}Q^{(\tilde{\tau})}_{<\gamma_1}\Re(z_{nres}^{prin}) + T\big(Q^{(\tilde{\tau})}_{<\gamma_1}\Re(z_{nres}^{prin})\big) = F_3\\
			&\mathcal{L}Q^{(\tilde{\tau})}_{<\gamma_1}\Im(z_{nres}^{prin}) = F_4, 
	\end{split}\end{equation}   
	where we set 
	\begin{align*}
		F_3 = F_1 + F_2 + \Re\big(\tilde{E}_{main} + z_{nres, small}^{prin}\big),\,F_4 = \Im\big(\tilde{E}_{main} + z_{nres, small}^{prin}\big).
	\end{align*}
	Here $F_{1,2}$ are as in Lemma~\ref{lem:lowtempfreqznresprin1}, Lemma~\ref{lem:lowtempfreqznresprin2}. Then the following lemma provides the needed improved bound for $Q^{(\tilde{\tau})}_{<\gamma_1}\Re(z_{nres}^{prin})$:
	\begin{lem}\label{lem:smalltempfreqznresprinimprovedbound} The solution 
		\[
		Q^{(\tilde{\tau})}_{<\gamma_1}(z_{nres}^{prin}) = Q^{(\tilde{\tau})}_{<\gamma_1}\Re(z_{nres}^{prin}) + iQ^{(\tilde{\tau})}_{<\gamma_1}\Im(z_{nres}^{prin}) 
		\]
		of \eqref{eq:smallfreqznresprineffective} satisfies
		\begin{align*}
			\big\| Q^{(\tilde{\tau})}_{<\gamma_1}(z_{nres}^{prin}) \big\|_{S}\lesssim c(\epsilon_1,\gamma_1, \tau_*, N)\big\|z_{nres}\big\|_{S} + \big\|(\tilde{\kappa}_1,\kappa_2)\big\|_{\tau^{-N}L^2_{d\tau}} + \big\|e_1\big\|_{\tau^{-N-1}L^2_{d\tau}L^2_{R^3\,dR}}. 
		\end{align*}
		where $\lim_{\epsilon_1^{-1},\gamma_1^{-1}, N,\tau_*\rightarrow+\infty}c(\epsilon_1, \gamma_1, \tau_*, N) = 0$. In particular, using Proposition~\ref{prop:tildekappa1kappa2apriori}, we obtain 
		\begin{align*}
			\big\| Q^{(\tilde{\tau})}_{<\gamma_1}(z_{nres}^{prin}) \big\|_{S}\lesssim c(\epsilon_1,\tau_*, N)\big\|z_{nres}\big\|_{S}+ \big\|e_1\big\|_{\tau^{-N-1}L^2_{d\tau}L^2_{R^3\,dR}}. 
		\end{align*}
	\end{lem}
	\begin{proof} We observe right away that in order to improve the bound for the fourth norm in  \eqref{eq:Snormdefi} for $ Q^{(\tilde{\tau})}_{<\gamma_1}\big(z_{nres}^{prin}\big)$, we can directly refer to \eqref{eq:znresprin} and take advantage of the bound $\big\|\phi(R;\xi) - \phi(R;0)]\big\|_{L^{\frac83+}_{R^3\,dR}}\lesssim 1$, in conjunction with Corollary~\ref{cor:yzWpartialtau}, Lemma~\ref{lem:ytildelambdamodhightempfreq}, Plancherel's theorem for the distorted Fourier transform and the Cauchy-Schwarz inequality and finally Proposition~\ref{prop:tilealphatildelambdasimultaneous} to bound the $\xi$-integral, to bound this component by $\ll_{\tau_*}\big\|z_{nres}\big\|_{S} + \big\|(\tilde{\kappa}_1,\kappa_2)\big\|_{\tau^{-N}L^2_{d\tau}}$. Henceforth we shall work to improve the bounds for the remaining norms constituting $\big\|\cdot\big\|_{S}$.
		The only complication comes from the first equation in \eqref{eq:smallfreqznresprineffective}, which contains the operator $T$. As a first step, we observe that 
		\[
		\int_0^\infty \Lambda W\cdot W^3 R^3\,dR = 0, 
		\]
		which implies that the equation 
		\begin{align*}
			\tilde{\mathcal{L}}(\phi) = W^3
		\end{align*}
		admits a unique solution\footnote{This solution is easily seen to be a linear combination of $W$ and $\Lambda W$: $\phi = -\frac{W}{2} - \frac{\Lambda W}{16}$.} in $L^2_{R^3\,dR}$. We can then write 
		\begin{align*}
			T(z) = \tilde{T}(\tilde{\mathcal{L}}z), 
		\end{align*}
		where we set 
		\begin{equation}\label{eq:tildeT}
			\tilde{T}(z): = \alpha_*\cdot \big(\int_0^\infty z\phi R^3\,dR\big)\cdot  \triangle^{-1}\big(\Lambda W\cdot W\big)\cdot W.
		\end{equation}
		provided $z\in L^2_{R^3\,dR}$. We can then reformulate the first equation in \eqref{eq:smallfreqznresprineffective} as 
		\begin{align*}
			\big(I + \tilde{T}\big)\big(\tilde{\mathcal{L}}Q^{(\tilde{\tau})}_{<\gamma_1}(z_{nres}^{prin}) \big) = F_3. 
		\end{align*}
		Now taking advantage of Lemma~\ref{lem:oneplustildetinverse} in section~\ref{sec:appendix}, we infer 
		\begin{align*}
			\Big\|\tilde{\mathcal{L}}Q^{(\tilde{\tau})}_{<\gamma_1}(z_{nres}^{prin})\Big\|_{\tau^{-N}L^2_{d\tau}L^2_{R^3\,dR}}\lesssim \big\|F_3\big\|_{\tau^{-N}L^2_{d\tau}L^2_{R^3\,dR}}.
		\end{align*}
		The conclusion of the lemma then follows by combining Lemma~\ref{lem:lowtempfreqznresprin2}, Lemma~\ref{lem:lowtempfreqznresprin1} and Lemma~\ref{lem:basicSfromtildeL} with the bound (here $\mathcal{L}_*$ equals $\mathcal{L}$ or $\tilde{\mathcal{L}}$)
		\begin{equation}\label{eq:tildeEmainsmallness}\begin{split}
				\big\|\tilde{E}_{main}\big\|_{\tau^{-N}L^2_{d\tau}L^{2+}_{R^3\,dR}\cap \langle R\rangle^{\frac{\delta_0}{2}}L^2_{R^3\,dR}} + \big\|\mathcal{L}_*^{-1}(z_{nres, small}^{prin})\big\|_{\tilde{S}}&\ll_{N,\tau_*} \big\|z_{nres}\big\|_{S} + \big\|(\tilde{\kappa}_1,\kappa_2)\big\|_{\tau^{-N}L^2_{d\tau}}\\&\hspace{1cm}+
				\big\|e_1\big\|_{\tau^{-N-1}L^2_{d\tau}L^2_{R^3\,dR}}. 
		\end{split}\end{equation}
		where $\big\|\cdot\big\|_{\tilde{S}}$ is the sum of the first three norms in \eqref{eq:Snormdefi}. We prove this in section~\ref{sec:appendix}. 
	\end{proof}
	
	\subsection{Improving the bound for large wave-temporal frequencies}
	
	We next aim to improve the bound for $Q^{(\tilde{\tau})}_{>\gamma^{-1}}(z_{nres}^{prin})$. Thus we consider the equation 
	\begin{equation}\label{eq:znreskeyreofrmulationHighFreq}
		\mathcal{L}Q^{(\tilde{\tau})}_{>\gamma_1^{-1}}z_{nres}^{prin} - Q^{(\tilde{\tau})}_{>\gamma_1^{-1}}\big(2n_{prin}\cdot W\big) - Q^{(\tilde{\tau})}_{>\gamma_1^{-1}}\big(y_{\tilde{\lambda}}\cdot W\big) = Q^{(\tilde{\tau})}_{>\gamma_1^{-1}}\big(\tilde{E}_{main} + z_{nres, small}^{prin}\big). 
	\end{equation}
	In fact, we shall be able to treat the second and third term on the left as perturbative terms, taking advantage of Prop.~\ref{prop:solnoftildelambdaeqn} for the third term:
	\begin{lem}\label{lem:znresprinhighfreqimprov} Recalling \eqref{eq:nprindef} for the definition of $n_{prin}$, we have the bound 
		\begin{align*}
			\Big\|Q^{(\tilde{\tau})}_{>\gamma_1^{-1}}\big(2n_{prin}\cdot W\big)\Big\|_{\tau^{-N}L^2_{d\tau}L^2_{R^3\,dR}}\ll_{\gamma_1}\big\|z_{nres}^{prin}\big\|_{S}. 
		\end{align*}
		Furthermore,  we can write 
		\begin{align*}
			Q^{(\tilde{\tau})}_{>\gamma_1^{-1}}\big(y_{\tilde{\lambda}}\cdot W\big) = 2Q^{(\tilde{\tau})}_{>\gamma_1^{-1}}\tilde{\lambda}\cdot \Lambda W\cdot W^2 + G_1, 
		\end{align*}
		where the last term on the right satisfies the bound 
		\begin{align*}
			\big\|G_1\big\|_{\tau^{-N}L^2_{d\tau}L^2_{R^3\,dR}}\ll_{\gamma_1}\big\|\langle\partial_{\tilde{\tau}}^2\rangle^{-1}\tilde{\lambda}_{\tilde{\tau}\tilde{\tau}}\big\|_{\tau^{-N}L^2_{d\tau}}\lesssim \big\|z_{nres}\big\|_{S} + \big\|(\tilde{\kappa}_1,\kappa_2)\big\|_{\tau^{-N}L^2_{d\tau}},
		\end{align*}
		where the last inequality follows from Proposition~\ref{prop:tilealphatildelambdasimultaneous} .
	\end{lem}
	The proof is similar to the one of Lemma~\ref{lem:lowtempfreqznresprin1}, we omit the simple details. Next, we take advantage of Proposition~\ref{prop:solnoftildelambdaeqn} in order to express $Q^{(\tilde{\tau})}_{>\gamma_1^{-1}}\tilde{\lambda}$ in terms of $Q^{(\tilde{\tau})}_{>\gamma_1^{-1}}z_{nres}$. The conclusion is the following 
	\begin{lem}\label{lem:znresprinhighfreqimprov1} We have the asymptotic formula (using \eqref{eq:alphastarstar})
		\begin{align*}
			2Q^{(\tilde{\tau})}_{>\gamma_1^{-1}}\tilde{\lambda}\cdot \Lambda W\cdot W^2 = 2\cdot \alpha_{**}^{-1}\cdot T_1\big(Q^{(\tilde{\tau})}_{>\gamma_1^{-1}}z_{nres}^{prin}\big)\cdot \Lambda W\cdot W^2 + G_2, 
		\end{align*}
		where $T_1(z): = \int_0^\infty z\cdot W\cdot\triangle(W^2)R^3\,dR$, and the error term $G_2$ satisfies the same bound as the term $F_2$ in Lemma~\ref{lem:lowtempfreqznresprin2}. 
	\end{lem}
	This lemma is proved in analogy to Lemma~\ref{lem:lowtempfreqznresprin2}, relying on the fact that the function $\beta_*(\hat{\tilde{\tau}})$ in Proposition~\ref{prop:solnoftildelambdaeqn} satisfies the limiting relation 
	\begin{align*}
		\lim_{\hat{\tilde{\tau}}\rightarrow +\infty}\beta_*(\hat{\tilde{\tau}}) = \alpha_{**}^{-1}, 
	\end{align*}
	in turn a consequence of the proof of Lemma~\ref{lem:Phitildelambdamodeleqn}.
	
	Combining the conclusions of the preceding two lemmas with \eqref{eq:tildeEmainsmallness}, we can infer the following 
	\begin{lem}\label{lem:znresprinlargetemprfreqsmallness} We have the following bound for the large temporal frequency component of $z_{nres}^{prin}$:
		\begin{equation}\label{eq:Q>gamma-1intermediate}\begin{split}
				\big\| Q^{(\tilde{\tau})}_{>\gamma_1^{-1}}(z_{nres}^{prin}) \big\|_{S}\lesssim c_1(\epsilon_1, \gamma_1,\tau_*,N)\cdot\big\|z_{nres}\big\|_{S} + \big\|(\tilde{\kappa}_1,\kappa_2)\big\|_{\tau^{-N}L^2_{d\tau}} + \big\|e_1\big\|_{\tau^{-N-1}L^2_{d\tau}}. 
		\end{split}\end{equation}
		where $\lim_{\epsilon_1^{-1}, \gamma_1^{-1}, \tau_*, N\rightarrow\infty}c_1(\epsilon_1, \gamma_1, \tau_*,N) = 0$. 
	\end{lem}
	\begin{proof} This is analogous to the proof of Lemma~\ref{lem:smalltempfreqznresprinimprovedbound}.  We first reformulate \eqref{eq:znreskeyreofrmulationHighFreq} as 
		\begin{align*}
			\mathcal{L}Q^{(\tilde{\tau})}_{>\gamma_1^{-1}}z_{nres}^{prin}  - 2\cdot \alpha_{**}^{-1}\cdot T_1\big(Q^{(\tilde{\tau})}_{>\gamma_1^{-1}}z_{nres}^{prin}\big)\cdot \Lambda W\cdot W^2 &= Q^{(\tilde{\tau})}_{>\gamma_1^{-1}}\big(\tilde{E}_{main} + z_{nres, small}^{prin}\big)\\
			& + H, 
		\end{align*}
		where $H$ satisfies the same bound as $F_2$ in Lemma~\ref{lem:lowtempfreqznresprin2}. Then observe that 
		\begin{equation}\label{eq:psidefn}
			\Lambda W\cdot W^2 = \mathcal{L}(\psi),\,\psi := 2\Lambda W + 16W\in L^2_{R^3\,dR}. 
		\end{equation}
		Then we can write the preceding equation for $Q^{(\tilde{\tau})}_{>\gamma_1^{-1}}z_{nres}^{prin}$ as 
		\begin{align*}
			\mathcal{L}\big(Q^{(\tilde{\tau})}_{>\gamma_1^{-1}}z_{nres}^{prin} - 2\cdot \alpha_{**}^{-1}\cdot T_1\big(Q^{(\tilde{\tau})}_{>\gamma_1^{-1}}z_{nres}^{prin}\big)\cdot\psi\big) = &Q^{(\tilde{\tau})}_{>\gamma_1^{-1}}\big(\tilde{E}_{main} + z_{nres, small}^{prin}\big)\\
			& + H. 
		\end{align*}
		The desired bound for $Q^{(\tilde{\tau})}_{>\gamma_1^{-1}}z_{nres}^{prin}$ then follows by combining Lemma~\ref{lem:T1inversion} with Lemma~\ref{lem:basicSfromtildeL}.
	\end{proof}

	\subsection{Improving the bound for intermediate wave-temporal frequencies I; introducing a key Fredholm operator}
	
	This is the most difficult situation to deal with, as the second and third term on the left hand side of \eqref{eq:znreskeyreofrmulation1} are no longer perturbative in this regime. 
	Here it shall be advantageous to use \eqref{eq:nprinequation1}, \eqref{eq:nprinequation2} instead. The resulting full equation for $\tilde{n}_{prin}$ is then given by 
	\begin{equation}\label{eq:tildenprinfull}\begin{split}
			& \Box \tilde{n}_{prin} - 2\triangle\big(W\cdot \mathcal{L}^{-1}(\tilde{n}_{prin}\cdot W)\big) - \triangle\big(W\cdot \mathcal{L}^{-1}( y_{\tilde{\lambda}}\cdot W)\big)\\
			& =  \lambda^2\triangle \Re\big(W\mathcal{L}^{-1}(\overline{\tilde{E}_{main} + z_{nres, small}^{prin}}\big)
	\end{split}\end{equation}
	Here we keep in mind that $y_{\tilde{\lambda}}$ is given by \eqref{eq:ytildelamba}, while $\tilde{\lambda}$ is described in terms of Proposition~\ref{prop:solnoftildelambdaeqn}. To clarify the structure of the wave operator, introduce the operator
	\begin{equation}\label{eq:Kdefinition}
		K: = 2\triangle\big(W\cdot \mathcal{L}^{-1}(\tilde{n}_{prin}\cdot W)\big)  + 2W^2. 
	\end{equation}
	Then we can write 
	\begin{align*}
		\Box \tilde{n}_{prin} - 2\triangle\big(W\cdot \mathcal{L}^{-1}(\tilde{n}_{prin}\cdot W)\big) = \Box \tilde{n}_{prin} +2W^2 \tilde{n}_{prin} -K \tilde{n}_{prin}, 
	\end{align*}
	and the key point shall be to understand the propagator associated to the wave operator $\Box  +2W^2 -K$, when restricted to intermediate (wave) temporal frequency. Inspired by methods common in control theory, we shall apply the (wave) temporal Fourier transform to this operator and reduce the wave equation \eqref{eq:tildenprinfull} to an elliptic equation. Overall the basic strategy shall be to reduce the preceding operator to the much simpler operator $\Box + 2W^2$, whose propagator is straightforward to describe, in analogy to subsection~\ref{subse:waveparametrix}. 
	To achieve this reduction, we formally factorize 
	\begin{align*}
		\Box + 2W^2 - K = \big(\Box + 2W^2\big)\circ\big(I - (\Box + 2W^2)^{-1}\circ K\big)
	\end{align*}
	Inverting this relation formally, we arrive at 
	\begin{equation}\label{eq:formalinverseofBox+K}
		\big(\Box + 2W^2 - K\big)^{-1} = \big(I - (\Box + 2W^2)^{-1}\circ K\big)^{-1}\circ \big(\Box + 2W^2\big)^{-1}. 
	\end{equation}
	The key difficulty now consists in controlling the first expression on the right, which turns out to be the inverse of a Fredholm operator when restricting the wave temporal frequencies and applying the temporal Fourier transform. More precisely, for technical reasons we shall split 
	\begin{align*}
		(\Box + 2W^2)^{-1}\circ K = (\Box + 2W^2)^{-1}\circ K_{main} + (\Box + 2W^2)^{-1}\circ K_{small},
	\end{align*}
	where $K_{small}$ enjoys smallness properties which shall allow us to treat the operator $(\Box + 2W^2)^{-1}\circ K_{small}$ perturbatively, while $K_{main}$, while not gaining smallness, has certain inherent localization properties, see the next subsection for the precise definition. 
	Denoting by $\hat{\tau}$ the Fourier variable corresponding to eave time $\tilde{\tau} = \int_t^\infty \lambda(s)\,ds$, consider the following model operator
	\begin{equation}\label{eq:waveoperatormodel1} 
		I - \big(\hat{\tau}^2 + \triangle +2W^2\big)^{-1}\circ K_{main}
	\end{equation} 
	which is an approximation of the operator $I - (\Box + 2W^2)^{-1}\circ K_{main}$ after application of the temporal Fourier transform. 
	A key point shall be to use the correct definition of $\big(\hat{\tau}^2 + \triangle +2W^2\big)^{-1}$ via the spectral representation associated to $\mathcal{L}_*: = -\triangle -2W^2$. This shall be done in section~\ref{sec:appendix}, see \eqref{eq:goodinverseoftriangle+hattausquare+2Wsquare}. To emphasize this point in the sequel, we shall write the preceding model operator as 
	\begin{equation}\label{eq:waveoperatormodel1} 
		I - \big(\hat{\tau}^2 + \triangle +2W^2\big)_{good}^{-1}\circ K_{main}
	\end{equation}  
	
	\subsection{Improving the bound for intermediate wave-temporal frequencies II; inverse of Fredholm operator via Carleman estimate}
	
	Let us fill in some details: first, for a sufficiently large constant $M = M(\gamma_1)>0$, where we recall $\gamma_1$ is the constant used in defining the different wave-temporal frequency regimes, we define the truncated operators
	\begin{equation}\label{eq:Kmainandsmall}
		K_{main}: = \chi_{R\lesssim M}K\chi_{R\lesssim M},\,K_{small}: = K - K_{main}. 
	\end{equation}
	Then the following lemma is a consequence of a simple analogue of Lemma~\ref{lem:wavebasicinhom} in conjunction with Lemma~\ref{lem:tildeboxwpropagatorapproximate}. Let $(\Box + 2W^2)^{-1}$ be the Duhamel propagator vanishing at $\tilde{\tau} = +\infty$, where $\Box + 2W^2$ is given by \eqref{eq:truewaveoperatorBoxtildeW}: 
	\begin{lem}\label{lem:Ksmallbound} We have the following bound for the contribution of $K_{small}$:
		\begin{align*}
			\Big\|(\Box + 2W^2)^{-1}\circ K_{small}u\Big\|_{\tau^{-N}L^2_{d\tau}\langle R\rangle R^{\delta_0}L^2_{R^3\,dR}}\ll_M \big\|u\big\|_{\tau^{-N}L^2_{d\tau}\langle R\rangle R^{\delta_0}L^2_{R^3\,dR}}
		\end{align*}
	\end{lem}
	Thanks to the preceding lemma and straightforward perturbation techniques, the invertibility of $I - \big(\Box +2W^2\big)^{-1}\circ K$ is now reduced to the invertibility of 
	\begin{equation}\label{eq:keyoperatortobeinverted}
		I - \big(\hat{\tau}^2 + \triangle +2W^2\big)_{good}^{-1}\circ K_{main}.
	\end{equation}
	after application of the wave temporal Fourier transform, where we refer to subsection~\ref{subsec:goodinverse} for the definition of the 'good inverse'. To proceed, we have the following 
	\begin{lem}\label{lem:Fredholmproperty} The operator \eqref{eq:keyoperatortobeinverted} maps $\langle R\rangle R^{\delta_0}L^2_{R^3\,dR}$ into itself and is a Fredholm operator there. 
	\end{lem}
	\begin{proof} This is a straightforward consequence of Rellich's theorem, using the smoothing property of $K_{main}$, and the fact that 
		\begin{align*}
			\Big\|\chi_{R\gtrsim L}\big(\hat{\tau}^2 + \triangle +2W^2\big)_{good}^{-1}\circ K_{main}\Big\|_{\langle R\rangle R^{\delta_0}L^2_{R^3\,dR}\rightarrow \langle R\rangle R^{\delta_0}L^2_{R^3\,dR}}\ll_{L}1. 
		\end{align*}
	\end{proof}
	
	Let $\phi_{\hat{\tau}} = \phi_*(R;\hat{\tau})\in \langle R\rangle R^{\delta_0}L^2_{R^3\,dR}$ be the unique function in the kernel of $\hat{\tau}^2 + \triangle +2W^2$ satisfying the boundary condition $\phi_{\hat{\tau}}(0) = 1$, see also Proposition~\ref{prop:triangle+2W2}. For later use, we shall also introduce the 
	function $\theta_{\hat{\tau}}\in C^\infty(\R_+)$ which satisfies 
	\begin{align*}
		\big(\hat{\tau}^2 + \triangle +2W^2\big)\theta_{\hat{\tau}} = 0,\,W(\phi_{\hat{\tau}}, \theta_{\hat{\tau}}): = \partial_R\phi_{\hat{\tau}}\cdot  \theta_{\hat{\tau}} - \partial_R \theta_{\hat{\tau}}\cdot \phi_{\hat{\tau}} = R^{-3},
	\end{align*}
	whence $\{\phi_{\hat{\tau}}, \theta_{\hat{\tau}}\}$ form a fundamental system for $\partial_R^2 + \frac{3}{R}\partial_R + \hat{\tau}^2 + 2W^2$ on $\R_+$. Introduce the projection operator 
	\begin{equation}\label{eq:Piprojection}
		\Pi(f) = f - \frac{\langle \frac{f}{\langle R\rangle R^{\delta_0}},\, \frac{\phi_{\hat{\tau}}}{\langle R\rangle R^{\delta_0}}\rangle_{L^2_{R^3\,dR}}}{\big\|\frac{\phi_{\hat{\tau}}}{\langle R\rangle R^{\delta_0}}\big\|_{L^2_{R^3\,dR}}^2}\cdot \phi_{\hat{\tau}},\,f\in \langle R\rangle R^{\delta_0}L^2_{R^3\,dR}.
	\end{equation}
	We shall first consider the composition of $\Pi$ and $I - \big(\hat{\tau}^2 + \triangle +2W^2\big)^{-1}_{good}\circ K_{main}$. 
	\begin{prop}\label{prop:PiTsurjective} If $\hat{\tau}>0$ Then the operator composition
		\[
		\Pi\circ \big(I - (\hat{\tau}^2 + \triangle +2W^2\big)^{-1}_{good}\circ K_{main}\big): \langle R\rangle R^{\delta_0}L^2_{R^3\,dR}\longrightarrow \Pi\big(\langle R\rangle R^{\delta_0}L^2_{R^3\,dR}\big)
		\]
		is surjective. 
	\end{prop}
	\begin{proof} This can be shown by means of a direct Volterra iteration argument (see next lemma), but we give here an alternative more conceptual proof relying on 
		the following standard Carleman estimate, valid for all radial $C^\infty(\R^4)$ functions $f$ compactly supported away from $0$ and which in addition to their first derivatives vanish rapidly toward $R = +\infty$:
		\begin{equation}\label{eq:Carleman}
			\boxed{2\hat{\tau}\sqrt{\lambda}\big\|R^{\lambda}f\big\|_{L^2_{R^3\,dR}}\leq \big\|R^{1+\lambda}(\triangle + \hat{\tau}^2)f\big\|_{L^2_{R^3\,dR}},}
		\end{equation}
		where $\lambda>0$ is arbitrary. 
		\\
		In order to prove the proposition, arguing by contradiction, we assume that there is $g\in \langle R\rangle^{-1} R^{-\delta_0}L^2_{R^3\,dR}$ with 
		\[
		\tilde{\Pi}(g)\neq 0,\,\langle \Pi\circ \big(I - \hat{\tau}^2 + \triangle +2W^2\big)^{-1}_{good}\circ K_{main}\big)(f),\,g\rangle_{L^2_{R^3\,dR}} = 0\,\forall\,f\in \langle R\rangle R^{\delta_0}L^2_{R^3\,dR}
		\]
		where $\tilde{\Pi}$ denotes the orthogonal projection in $ \langle R\rangle^{-1} R^{-\delta_0}L^2_{R^3\,dR}$ onto $\frac{\phi_{\hat{\tau}}}{\langle R\rangle^2 R^{2\delta_0}}$. It follows that 
		\begin{align*}
			\tilde{\Pi}(g) - K_{main}^*\circ(\hat{\tau}^2 + \triangle +2W^2)^{-1}_{good}(\tilde{\Pi}(g)) = 0.
		\end{align*}
		Letting $h: = (\hat{\tau}^2 + \triangle +2W^2)^{-1}_{good}(\tilde{\Pi}(g))\in \langle R\rangle R^{\delta_0}L^2_{R^3\,dR}$, we have 
		\begin{lem}\label{lem:hvanishing1} The function $h$ is in $L^2_{R^3\,dR}\cap C^\infty(\R_+)$ and satisfies the bounds 
			\begin{align*}
				\Big| D^{\alpha}h(R)\Big|\lesssim_N R^{-N},\,\alpha = 0, 1
			\end{align*}
			as $R\rightarrow +\infty$ for any $N>0$. 
		\end{lem}
		We relegate the proof to section~\ref{sec:appendix}.
		\\
		
		To complete the proof of the proposition, we follow an argument in \cite{Lerner}: introduce the function 
		\begin{align*}
			h_{\rho}: = \big(1-\chi(\frac{R}{\rho})\big)\cdot h(R), 
		\end{align*}
		where $\chi(x)$ is a smooth cutoff localizing to $|x|\leq 1$, while $\rho>0$ is a small parameter which we shall let converge toward zero. From \eqref{eq:Carleman} we infer 
		\begin{equation}\label{eq:carlemanapplied}
			2\hat{\tau}\sqrt{\lambda}\big\|R^{\lambda}h_{\rho}(R)\big\|_{L^2_{R^3\,dR}}\leq \Big\|R^{1+\lambda}(\triangle + \hat{\tau}^2)h_{\rho}\Big\|_{L^2_{R^3\,dR}}
		\end{equation}
		Taking advantage of the fact that
		\[
		(\hat{\tau}^2 + \triangle + 2W^2)h = K_{main}^*h,
		\]
		we can bound the right hand side of \eqref{eq:carlemanapplied} by 
		\begin{equation}\label{eq:delicateerrorsincarlemanmethod}
			2\big\|R^{1+\lambda}W^2h_{\rho}\big\|_{L^2_{R^3\,dR}} + \big\|R^{1+\lambda}\big(1-\chi(\frac{R}{\rho})\big)K_{main}^*h\big\|_{L^2_{R^3\,dR}} + C_1\rho^{\lambda - 1}, 
		\end{equation}
		where the constant $C_1$ depends implicitly on $h$, and we assume $\rho\ll 1$; observe that the last term control the errors due to differentiating the cutoff $\chi(\frac{R}{\rho})$. We then have the simple bound 
		\begin{align*}
			2\big\|R^{1+\lambda}W^2h_{\rho}\big\|_{L^2_{R^3\,dR}}\leq C_2\big\|R^{\lambda}h_{\rho}\big\|_{L^2_{R^3\,dR}}
		\end{align*}
		for an absolute constant $C_2$. As for the second term in \eqref{eq:delicateerrorsincarlemanmethod}, we note the fine structure of $K_{main}^*h$ given by \eqref{eq:Kmainstructure}. Since we may assume that $1-\chi(\frac{\cdot}{\rho})$ is increasing, we easily infer the bound 
		\begin{align*}
			\big\|R^{1+\lambda}\big(1-\chi(\frac{R}{\rho})\big)K_{main}^*h\big\|_{L^2_{R^3\,dR}} \leq C_3\big\|R^{\lambda}h_{\rho}\big\|_{L^2_{R^3\,dR}}, 
		\end{align*}
		where we observe the crucial feature that the operators $K_j^*$ in \eqref{eq:Kmainstructure} involve integrals between $R$ and $\infty$. Combining the preceding observations we infer the following consequence of the Carleman bound \eqref{eq:Carleman}:
		\begin{align*}
			2\hat{\tau}\sqrt{\lambda}\big\|R^{\lambda}h_{\rho}(R)\big\|_{L^2_{R^3\,dR}}\leq C_1\rho^{\lambda - 1} + C_4\big\|R^{\lambda}h_{\rho}\big\|_{L^2_{R^3\,dR}}, 
		\end{align*}
		where the crucial feature is that the constants $C_{1,4}$ are independent of $\rho,\lambda$. Letting $\rho\ll1$ and $\lambda_*\gg 1$ such that 
		\begin{align*}
			2\hat{\tau}\sqrt{\lambda_*} > 2C_4, 
		\end{align*}
		we deduce that 
		\begin{align*}
			\big\|R^{\lambda_*}h_{\rho}(R)\big\|_{L^2_{R^3\,dR}}\leq C_5\rho^{\lambda_* - 1}. 
		\end{align*}
		As $\big\|R^{\lambda_*}h_{\rho}(R)\big\|_{L^2_{R^3\,dR}}$ is a decreasing function of $\rho$, we conclude that
		\begin{align*}
			\big\|R^{\lambda_*}h_{\rho}(R)\big\|_{L^2_{R^3\,dR}}\leq \big\|R^{\lambda_*}h_{\rho'}(R)\big\|_{L^2_{R^3\,dR}}\leq C_5(\rho')^{\lambda_* - 1}
		\end{align*}
		for $\rho'<\rho$. Letting $\rho'\rightarrow 0$ we infer that 
		\begin{align*}
			\big\|R^{\lambda_*}h_{\rho}(R)\big\|_{L^2_{R^3\,dR}} = 0
		\end{align*}
		for any $\rho>0$, which in turn yields $h = 0$, whence $\tilde{\Pi}g = 0$, resulting in a contradiction as desired. 
	\end{proof}
	
	To ensure surjectivity of $I - \big(\hat{\tau}^2 + \triangle +2W^2\big)_{good}^{-1}\circ K_{main}$, it remains to show that $\phi_{\hat{\tau}}\in \langle R\rangle R^{\delta_0}L^2_{R^3\,dR}$ is in its image. The argument we have here is less elegant and requires finitely many non-degeneracy assumptions in principle amenable to numerical verification: 
	\begin{lem}\label{lem:phitauhatinimage} There is an absolute constant $\hat{\tau}_*>0$ such that assuming the non-degeneracy assumption {\bf{(A1)}} for $0\leq \hat{\tau}\leq \hat{\tau}_*$,, we have that 
		\[
		\phi_{\hat{\tau}}\in \text{range}\Big(I - \big(\hat{\tau}^2 + \triangle +2W^2\big)_{good}^{-1}\circ K_{main}\Big).
		\]
	\end{lem}
	\begin{proof} Introduce the operator (with $\{\phi_{\hat{\tau}}, \theta_{\hat{\tau}}\}$ the fundamental system introduced after Lemma~\ref{lem:Fredholmproperty})
		\[
		\big(\hat{\tau}^2 + \triangle + 2W^2\big)^{-1}_{0}f(R): = \phi_{\hat{\tau}}\cdot \int_0^R\theta_{\hat{\tau}}(s)f(s)s^3\,ds -  \theta_{\hat{\tau}}\cdot \int_0^R\phi_{\hat{\tau}}(s)f(s)s^3\,ds.
		\]
		Then set 
		\begin{equation}\label{eq:vtauhatiteration}
			v_{\hat{\tau}}: = \sum_{j=0}^\infty \Big(\big(\hat{\tau}^2 + \triangle + 2W^2\big)^{-1}_{0}\circ K_{main}\Big)^j(\phi_{\hat{\tau}}). 
		\end{equation}
		Recalling \eqref{eq:Kdefinition} and the definition of $\mathcal{L}^{-1}$ after \eqref{eq:nprinequation1}, we see that the preceding sum is a Volterra iteration, and hence converges rapidly (faster than exponentially) toward $v_{\hat{\tau}}\in H^2_{loc}(\R_+)\cap \langle R\rangle R^{\delta_0}L^2_{R^3\,dR}$. Furthermore, we have 
		\begin{align*}
			\big(I -  \big(\hat{\tau}^2 + \triangle +2W^2\big)_{0}^{-1}\circ K_{main}\big)(v_{\hat{\tau}}) = \phi_{\hat{\tau}}. 
		\end{align*}
		Observe that since $\big(\hat{\tau}^2 + \triangle +2W^2\big)_{0}^{-1}\circ K_{main}v_{\hat{\tau}})\in \langle R\rangle R^{\delta_0}L^2_{R^3\,dR},\,\big(\hat{\tau}^2 + \triangle +2W^2\big)_{good}^{-1}\circ K_{main}v_{\hat{\tau}})\in \langle R\rangle R^{\delta_0}L^2_{R^3\,dR}$, these expressions differ by a multiple of $\phi_{\hat{\tau}}$. It follows that 
		\begin{equation}\label{eq:goodinverseonvhattau}
			\big(I -  \big(\hat{\tau}^2 + \triangle +2W^2\big)_{good}^{-1}\circ K_{main}\big)(v_{\hat{\tau}}) = \alpha_{\hat{\tau}}\phi_{\hat{\tau}},\,\alpha_{\hat{\tau}}\in \C. 
		\end{equation}
		It remains to show that $\alpha_{\hat{\tau}}\neq 0$. This is the case provided $\hat{\tau}>\hat{\tau}_*$ where the latter is an absolute constant. In fact, the rapidly oscillatory character of $\phi_{\hat{\tau}}$ (see Proposition~\ref{prop:triangle+2W2} ) implies $v_{\hat{\tau}} = \phi_{\hat{\tau}} + O_{\langle R\rangle R^{\delta_0}L^2_{R^3\,dR}}(|\hat{\tau}|^{-2})$, $|\hat{\tau}|\gg 1$, and from here 
		\[
		\big(\hat{\tau}^2 + \triangle +2W^2\big)_{good}^{-1}\circ K_{main}(v_{\hat{\tau}}) = O_{\langle R\rangle R^{\delta_0}L^2_{R^3\,dR}}(|\hat{\tau}|^{-2}),\,|\hat{\tau}|\gg 1, 
		\]
		while we have $\big\|\phi_{\hat{\tau}}\big\|_{\langle R\rangle R^{\delta_0}L^2_{R^3\,dR}}\sim |\hat{\tau}|^{-\frac32}$, see Proposition~\ref{prop:triangle+2W2}. 
		The function
		\[
		\mathcal{F}_{\hat{\tau}}\big(K_{main}(v_{\hat{\tau}}))
		\]
		is analytic. By assumption {\bf{(A1)}} (see subsection~\ref{subsec:numerics}) for the finitely many values $\{\hat{\tau}_j\}_{j=1}^R\subset [0, \hat{\tau}_*]$ where $\mathcal{F}_{\hat{\tau}}\big(K_{main}(v_{\hat{\tau}}))$ may vanish, we have 
		\begin{align*}
			&v_{\hat{\tau}} = \phi_{\hat{\tau}} + \phi_{\hat{\tau}}\cdot \int_0^R\theta_{\hat{\tau}}(s)K_{main}v_{\hat{\tau}}(s)s^3\,ds -  \theta_{\hat{\tau}}\cdot \int_0^R\phi_{\hat{\tau}}(s)K_{main}v_{\hat{\tau}}(s)s^3\,ds\notin L^2_{R^3\,dR}\\&\tau\in \{\hat{\tau}_j\}_{j=1}^R.
		\end{align*}
		which then necessitates $\alpha_{\hat{\tau}}\neq 0$ due to Lemma~\ref{lem:goodinversewithK}. 
	\end{proof} 
	
	Thanks to the preceding proposition and lemma which imply the surjectivity of $I - \big(\hat{\tau}^2 + \triangle +2W^2\big)_{good}^{-1}\circ K_{main}$ on $\langle R\rangle R^{\delta_0}L^2_{R^3\,dR}$, we can now define the operator 
	\[
	\big(I - \big(\hat{\tau}^2 + \triangle +2W^2\big)_{good}^{-1}\circ K_{main}\big)^{-1}
	\]
	on this space by imposing orthogonality to the kernel of $\big(I - \big(\hat{\tau}^2 + \triangle +2W^2\big)_{good}^{-1}\circ K_{main}\big)$. Let us then say we have chosen the {\it{canonical inverse}}. There is still the somewhat technical issue as to whether the solution of 
	\begin{equation}\label{eq:hhattauuhatttau}
		\big(I - \big(\hat{\tau}^2 + \triangle +2W^2\big)_{good}^{-1}\circ K_{main}\big)u_{\hat{\tau}}(R) = f_{\hat{\tau}}(R), 
	\end{equation}
	where the subscripts in $u_{\hat{\tau}}, f_{\hat{\tau}}$ indicate a dependence on $\hat{\tau}$, which for $f_{\hat{\tau}}$ is of regularity $W^{N,2}$, is of the same regularity with respect to $\hat{\tau}$. To resolve this, we can invoke
	\begin{lem}\label{lem:smoothnesswithrespecttohattauforinverse} Given $f_{\hat{\tau}}(\cdot)\in W_{\hat{\tau}}^{M,2}\langle R\rangle R^{\delta_0}L^2_{R^3\,dR}$, $M = \frac{N}{\frac12-\frac{1}{4\nu}}$, with support in $[\gamma_1,\gamma_2]$, $0<\gamma_1<\gamma_2<\infty$, there exists a solution $u_{\hat{\tau}}$ satisfying
		\[
		\mathcal{F}_{\tilde{\tau}}^{-1}\big(u_{\hat{\tau}}\big)|_{[\tilde{\tau}_*,\infty)}\in \tau^{-N}L^2_{d\tau}\langle R\rangle R^{\delta_0}L^2_{R^3\,dR}
		\]
		satisfying the bound 
		\begin{align*}
			\big\|\mathcal{F}_{\tilde{\tau}}^{-1}\big(u_{\hat{\tau}}\big)|_{[\tilde{\tau}_*,\infty)}\big\|_{\tau^{-N}L^2_{d\tau}\langle R\rangle R^{\delta_0}L^2_{R^3\,dR}}\lesssim_{\gamma_1,\gamma_2}\big\|f_{\hat{\tau}}(\cdot)\big\|_{W_{\hat{\tau}}^{M,2}\langle R\rangle R^{\delta_0}L^2_{R^3\,dR}}.
		\end{align*}
	\end{lem}
	\begin{proof} For $\hat{\tau}_*\in [\gamma_1,\gamma_2]$, write for $\hat{\tau}$ close to $\hat{\tau}_*$ 
		\begin{align*}
			&I - \big(\hat{\tau}^2 + \triangle +2W^2\big)_{good}^{-1}\circ K_{main} = I - \big(\hat{\tau}_*^2 + \triangle +2W^2\big)_{good}^{-1}\circ K_{main} + T_{\hat{\tau},\hat{\tau}_*}\\
			&T_{\hat{\tau},\hat{\tau}_*}: = \Big[\big(\hat{\tau}_*^2 + \triangle +2W^2\big)_{good}^{-1}\circ K_{main} - \big(\hat{\tau}^2 + \triangle +2W^2\big)_{good}^{-1}\circ K_{main}\Big]
		\end{align*}
		Taking advantage of Lemma~\ref{lem:goodinverse2}, we see that for $|\hat{\tau} - \hat{\tau}_*|<\gamma$ we have 
		\begin{align*}
			\Big\|T_{\hat{\tau},\hat{\tau}_*}\Big\|_{\langle R\rangle R^{\delta_0}L^2_{R^3\,dR}\rightarrow \langle R\rangle R^{\delta_0}L^2_{R^3\,dR}}\ll_{\gamma} 1.
		\end{align*}
		Then letting 
		\begin{align*}
			\Phi_{\hat{\tau}_{*}}: = \big(I - \big(\hat{\tau}_*^2 + \triangle +2W^2\big)_{good}^{-1}\circ K_{main}\big)^{-1}
		\end{align*}
		be the canonical inverse, we can set 
		\begin{align*}
			\Psi_{\hat{\tau}}:= \Big(I - \big(\hat{\tau}^2 + \triangle +2W^2\big)_{good}^{-1}\circ K_{main}\Big)^{-1} = \Big(\sum_{j=0}^\infty \big(\Phi_{\hat{\tau}_{*}}\circ T_{\hat{\tau},\hat{\tau}_*}\big)^j\Big)\circ  \Phi_{\hat{\tau}_{*}},
		\end{align*}
		which is a bounded operator on $\langle R\rangle R^{\delta_0}L^2_{R^3\,dR}$ for $|\hat{\tau} - \hat{\tau}_*|<\gamma\ll 1$. Using Lemma~\ref{lem:keygoodinverse} the higher differentiability with respect to $\hat{\tau}$ follows, and more precisely we have the bound (with $C$ a suitable universal constant)
		\begin{align*}
			\Big\|\mathcal{F}_{\tilde{\tau}}^{-1}\circ\big(\Phi_{\hat{\tau}_{*}}\circ T_{\hat{\tau},\hat{\tau}_*}\big)^j(f)|_{[\tau_*,\infty)}\Big\|_{\tau^{-N}L^2_{d\tau}\langle R\rangle R^{\delta_0}L^2_{R^3\,dR}}\lesssim (C\gamma)^j\cdot \big\|f\big\|_{\tau^{-N}L^2_{d\tau}\langle R\rangle R^{\delta_0}L^2_{R^3\,dR}}.
		\end{align*}
		
		Now we cover $[\gamma_1,\gamma_2]$ with finitely many intervals of the form $I_j: = \big(\tau_{j} - \frac{\gamma}{2}, \tau_j + \frac{\gamma}{2}\big)$, $j = 1,\ldots, R$, and letting $\sum_{j=1}^R\chi_{I_j}(\hat{\tau}) = 1$ be partition of unity subordinate to $\{I_j\}_{j=1}^R$, we set 
		\begin{align*}
			u_{\hat{\tau}}: = \sum_{j=1}^R\chi_{I_j}(\hat{\tau})\cdot \Psi^{(I_j)}_{\hat{\tau}}(f_{\hat{\tau}}),
		\end{align*}
		where the $ \Psi^{(I_j)}_{\hat{\tau}}$ are constructed like $\Psi_{\hat{\tau}}$ on each $I_j$. 
	\end{proof}
	
	\subsection{Improving the bound for intermediate wave-temporal frequencies III; passing from the model \eqref{eq:waveoperatormodel1} to the true wave operator} Let us introduce the notation (recall \eqref{eq:nFintildetauR})
	\begin{equation}\label{eq:truewaveoperatorBoxtildeW}
		\tilde{\Box}_W: = -\big(\partial_{\tilde{\tau}} + \frac{\lambda_{\tilde{\tau}}}{\lambda}R\partial_R\big)^2 - \frac{\lambda_{\tilde{\tau}}}{\lambda}\big(\partial_{\tilde{\tau}} + \frac{\lambda_{\tilde{\tau}}}{\lambda}R\partial_R\big) + \triangle +2W^2
	\end{equation}
	for the 'true' wave operator, while we denote the 'model' wave operator by
	\begin{equation}\label{eq:modelwaveoperatortildebox}
		\tilde{\Box}: = -\partial_{\tilde{\tau}}^2 + \triangle +2W^2.
	\end{equation} 
	Then we define $Q^{(\tilde{\tau})}_{[\gamma_1, \gamma_1^{-1}]}\tilde{\Box}^{-1}$ by means of \eqref{eq:goodinverseoftriangle+hattausquare+2Wsquare} via the spectral representation and passage to the temporal Fourier transform, or alternatively from Remark~\ref{rem:lem:keygoodinverse}, via the inhomogeneous Duhamel propagator 
	\begin{equation}\label{eq:tildeboxduhamel}\begin{split}
			&\big( Q^{(\tilde{\tau})}_{[\gamma_1, \gamma_1^{-1}]}\tilde{\Box}^{-1}f \big)(\tilde{\tau}, R)\\&= \int_0^\infty\int_{\tilde{\tau}}^\infty U_*(\tilde{\tau}, \tilde{\sigma}, \xi)\cdot \phi_*(R;\xi)\cdot \mathcal{F}_*\big( Q^{(\tilde{\tau})}_{[\gamma_1, \gamma_1^{-1}]}f\big)(\tilde{\sigma}, \cdot)\rho_*(\xi)\,d\xi\\
			& - \frac{\xi_d\cdot\phi_d(R)}{2}\cdot \int_{\tilde{\tau}_*}^\infty e^{-\xi_d|\tilde{\tau} - \tilde{\sigma}|}\cdot \mathcal{F}_d\big( Q^{(\tilde{\tau})}_{[\gamma_1, \gamma_1^{-1}]}f(\tilde{\sigma}, \cdot)\big)\,d\tilde{\sigma} + a(\tilde{\tau})\cdot \phi_d(R)
	\end{split}\end{equation}
	where $|a(\tilde{\tau})|\lesssim e^{-\xi_d \tilde{\tau}}$ and hence can be neglected in the sequel, and we use the notation from Lemma~\ref{lem:reductionsteps2}. We need to understand the difference $[\tilde{\Box}_W^{-1} -  \tilde{\Box}^{-1}]\circ Q^{(\tilde{\tau})}_{[\gamma_1, \gamma_1^{-1}]}$, for which we need to get a sufficiently precise handle on the propagator $\tilde{\Box}_W^{-1}$, which as usual is defined to be vanishing at $\tilde{\tau} = +\infty$. The following is by now completely standard: 
	\begin{lem}\label{lem:tildeboxwpropagatorapproximate} We can write 
		\begin{align*}
			\tilde{\Box}_W^{-1}f &=  \int_0^\infty\int_{\tilde{\tau}}^\infty \tilde{U}_*(\tilde{\tau}, \tilde{\sigma}, \xi)\cdot \phi_*(R;\xi)\cdot \mathcal{F}_*\big(f\big)(\tilde{\sigma}, \frac{\lambda(\tilde{\tau})}{\lambda(\tilde{\sigma})}\xi)\rho_*(\xi)\,d\xi\\
			& - \frac{\xi_d\cdot\phi_d(R)}{2}\cdot \int_{\tilde{\tau}_*}^\infty e^{-\xi_d|\tilde{\tau} - \tilde{\sigma}|}\cdot \mathcal{F}_d\big(f(\tilde{\sigma}, \cdot)\big)\,d\tilde{\sigma} +  \tilde{\Box}_{W, small}^{-1}f , 
		\end{align*}
		where the propagator $\tilde{U}_*(\tilde{\tau}, \tilde{\sigma}, \xi)$ is given by $\frac{\rho_*^{\frac12}(\frac{\lambda(\tau)}{\lambda(\sigma)}\xi)}{\rho_*^{\frac12}(\xi)}\cdot U(\tilde{\tau}, \tilde{\sigma}, \xi)$ and $U$ is the propagator in \eqref{eq:wavepropagator}. For the error term at the end, we have the estimate 
		\begin{align*}
			\big\|\tilde{\Box}_{W, small}^{-1}f \big\|_{\tau^{-N}L^2_{d\tau}L^2_{R^3\,dR}}\ll_{N,\tau_*}\big\|f\big\|_{\tau^{-N}L^2_{d\tau}L^2_{R^3\,dR}}.
		\end{align*}
	\end{lem}
	The proof is indicated in section~\ref{sec:appendix}. 
	
	The following lemma gives control over the difference between $\tilde{\Box}^{-1}, \tilde{\Box}_W^{-1}$, composed with $K_{main}$:
	\begin{lem}\label{lem:diffboxinverses1} We have the operator bound 
		\begin{align*}
			\Big\|[\tilde{\Box}_W^{-1} -  \tilde{\Box}^{-1}]\circ Q^{(\tilde{\tau})}_{[\gamma_1, \gamma_1^{-1}]}\circ K_{main}\Big\|_{\tau^{-N}L^2_{d\tau}\langle R\rangle R^{\delta_0}L^2_{R^3\,dR}\rightarrow \tau^{-N}L^2_{d\tau}\langle R\rangle R^{\delta_0}L^2_{R^3\,dR}}\ll_{\gamma_1,\tau_*}1.
		\end{align*}
	\end{lem}
	\begin{proof} In light of \eqref{eq:tildeboxduhamel}, Lemma~\ref{lem:tildeboxwpropagatorapproximate} it suffices to consider the difference 
		\begin{align*}
			&X_{\triangle}: = \int_0^\infty\int_{\tilde{\tau}}^\infty U_*(\tilde{\tau}, \tilde{\sigma}, \xi)\cdot \phi_*(R;\xi)\cdot \mathcal{F}_*\big( K_{main}(Q^{(\tilde{\tau})}_{[\gamma_1, \gamma_1^{-1}]}f)\big)(\tilde{\sigma}, \xi)\rho_*(\xi)\,d\xi\\
			& - \int_0^\infty\int_{\tilde{\tau}}^\infty \tilde{U}_*(\tilde{\tau}, \tilde{\sigma}, \xi)\cdot \phi_*(R;\xi)\cdot \mathcal{F}_*\big( K_{main}(Q^{(\tilde{\tau})}_{[\gamma_1, \gamma_1^{-1}]}f)\big)(\tilde{\sigma}, \frac{\lambda(\tilde{\tau})}{\lambda(\tilde{\sigma})}\xi)\rho_*(\xi)\,d\xi\\
		\end{align*}
		Using a straightforward analogues of Lemma~\ref{lem:wavebasicinhom} and its proof, we obtain the bound 
		\begin{align*}
			\big\|X_{\triangle,\,\tilde{\sigma}-\tilde{\tau}\geq \log \tau_*}\big\|_{\tau^{-N}L^2_{d\tau}\langle R\rangle R^{\delta_0}L^2_{R^3\,dR}}\ll_{\tau_*}\big\|f\big\|_{\tau^{-N}L^2_{d\tau}\langle R\rangle R^{\delta_0}L^2_{R^3\,dR}},
		\end{align*}
		where the modifcation of the subscript on the left indicates inclusion of an extra cutoff $\chi_{\tilde{\sigma}-\tilde{\tau}\geq \log \tau_*}$ in either of the double integrals. In fact, this is a consequence of the bound
		\begin{align*}
			\big\|\xi^{-1-\delta_0}\mathcal{F}_*\big(K_{main}f\big)\big\|_{L^2_{\rho_*(\xi)\,d\xi}} +  \big\|\langle\partial_{\xi}\rangle^{1+\delta_0}\mathcal{F}_*\big(K_{main}f\big)\big\|_{L^2_{\rho_*(\xi)\,d\xi}}\lesssim \big\|f\big\|_{\langle R\rangle R^{\delta_0}L^2_{R^3\,dR}}.
		\end{align*}
		For the remaining contribution $X_{\triangle,\,\tilde{\sigma}-\tilde{\tau}<\log \tau_*}$, one uses that 
		\begin{align*}
			&\big\|\chi_{0\leq \tilde{\sigma} - \tilde{\tau}<\log\tau_*} (U_*(\tilde{\tau}, \tilde{\sigma}, \xi) - \tilde{U}_*(\tilde{\tau}, \tilde{\sigma}, \xi))\big\|_{L^\infty_{d\tilde{\tau}(\tilde{\tau}\geq \tilde{\tau}_*)}L^1_{d\tilde{\sigma}}}\ll_{\tau_*}1,\\
			&\big\|\chi_{0\leq \tilde{\sigma} - \tilde{\tau}<\log\tau_*} (U_*(\tilde{\tau}, \tilde{\sigma}, \xi) - \tilde{U}_*(\tilde{\tau}, \tilde{\sigma}, \xi))\big\|_{L^\infty_{d\tilde{\sigma}}L^1_{d\tilde{\tau}(\tilde{\tau}\geq \tilde{\tau}_*)}}\ll_{\tau_*}1,
		\end{align*}
		together with the difference bound
		\begin{align*}
			&\big\|\xi^{-1-\delta_0}(\triangle)\mathcal{F}_*\big(K_{main}f\big)\big\|_{L^2_{\rho_*(\xi)\,d\xi}} +  \big\|\langle\partial_{\xi}\rangle^{1+\delta_0}(\triangle)\mathcal{F}_*\big(K_{main}f\big)\big\|_{L^2_{\rho_*(\xi)\,d\xi}}\\&\ll_{\tau_*}\big\|f\big\|_{\langle R\rangle R^{\delta_0}L^2_{R^3\,dR}},
		\end{align*}
		where we use the notation $(\triangle)f(\xi): = f(\frac{\lambda(\tilde{\tau})}{\lambda(\tilde{\sigma})}\xi) - f(\xi)$, and we restrict to the regime $0\leq \tilde{\sigma} - \tilde{\tau}<\log\tau_*$. The desired bound 
		\begin{align*}
			\big\|X_{\triangle,\,\tilde{\sigma}-\tilde{\tau}<\log \tau_*}\big\|_{\tau^{-N}L^2_{d\tau}\langle R\rangle R^{\delta_0}L^2_{R^3\,dR}}\ll_{\tau_*}\big\|f\big\|_{\tau^{-N}L^2_{d\tau}\langle R\rangle R^{\delta_0}L^2_{R^3\,dR}},
		\end{align*}
		then follows by using Schur's test as well as the simple modification of Lemma~\ref{lem:wavebasicinhom}.
	\end{proof}
	
	We can now gather the preceding developments to infer the following 
	\begin{prop}\label{prop:inverseoftrueoperatorBoxinverseWK} The operators
		\begin{align*}
			& I - (\Box + 2W^2)^{-1}\circ Q^{(\tilde{\tau})}_{[\gamma_1, \gamma_1^{-1}]}K_{main},\\
			& I - (\Box + 2W^2)^{-1}\circ Q^{(\tilde{\tau})}_{[\gamma_1, \gamma_1^{-1}]}K
		\end{align*}
		admits a bounded inverses on $\tau^{-N}L^2_{d\tau}\langle R\rangle R^{\delta_0}L^2_{R^3\,dR}$ provided $\tau\geq \tau_*$ with $\tau_*$ sufficiently large(depending on $N$). The operator norms of the inverses only depends on $\gamma$. 
	\end{prop}
	\begin{proof} To begin with, we decompose
		\begin{align*}
			&I - (\Box + 2W^2)^{-1}\circ Q^{(\tilde{\tau})}_{[\gamma_1, \gamma_1^{-1}]}K\\& =  I - (\Box + 2W^2)^{-1}\circ Q^{(\tilde{\tau})}_{[\gamma_1, \gamma_1^{-1}]}K_{main} - (\Box + 2W^2)^{-1}\circ Q^{(\tilde{\tau})}_{[\gamma_1, \gamma_1^{-1}]}K_{small}
		\end{align*}
		The last term is perturbative on account of Lemma~\ref{lem:Ksmallbound}.
		Using a simple Neumann expansion, it then suffices to establish invertibility of 
		\begin{align*}
			I - (\Box + 2W^2)^{-1}\circ Q^{(\tilde{\tau})}_{[\gamma_1, \gamma_1^{-1}]}K_{main} &= I - \tilde{\Box}^{-1}\circ Q^{(\tilde{\tau})}_{[\gamma_1, \gamma_1^{-1}]}K_{main}\\
			& +  \big(\tilde{\Box}^{-1} - \tilde{\Box}_W^{-1}\big)\circ Q^{(\tilde{\tau})}_{[\gamma_1, \gamma_1^{-1}]}K_{main}.\\
		\end{align*}
		This in turn is a consequence of Lemma~\ref{lem:diffboxinverses1} in conjunction with Lemma~\ref{lem:smoothnesswithrespecttohattauforinverse} and another Neumann series expansion. 
		
	\end{proof}
	
	\subsection{Improving the bound for intermediate wave-temporal frequencies IV; return to \eqref{eq:tildenprinfull}} We shall now project \eqref{eq:tildenprinfull} to intermediate temporal frequencies belonging to the interval $[\gamma, \gamma^{-1}]$, and take advantage of the preceding subsections to control the left hand side. So far we have neglected the third term on the left, which arises due to modulating on $\tilde{\lambda}$. Keep in mind \eqref{eq:ytildelamba} for the definition of $y_{\tilde{\lambda}}$, where the function $\tilde{\lambda}$ in turn is described by Prop.~\ref{prop:solnoftildelambdaeqn}. To take advantage of the latter, we first have to deal with a slight technical complication, which comes from the fact that in \eqref{eq:ytildelamba} we encounter the product $\tilde{\lambda}_{tt}\cdot\lambda^2$, rather than just $\tilde{\lambda}_{tt}$. To deal with this discrepancy, we use
	\begin{lem}\label{lem:tildelambdatimeslambdasquared} We have the description 
		\begin{align*}
			&\mathcal{F}_{\tilde{\tau}}\big(\lambda^{-2}\langle\partial_{\tilde{\tau}}^2\rangle^{-1} Q^{(\tilde{\tau})}_{[\gamma_1, \gamma_1^{-1}]}\big(\tilde{\lambda}_{prin,\tilde{\tau}\tilde{\tilde{\tau}}}\cdot \lambda^2\big)\big)(\hat{\tilde{\tau}})\\& = \langle \hat{\tilde{\tau}}^2\rangle\cdot\beta_*(\hat{\tilde{\tau}})\cdot \mathcal{F}_{\tilde{\tau}}\big( Q^{(\tilde{\tau})}_{[\gamma_1, \gamma_1^{-1}]}\int_0^\infty \Box^{-1}\triangle\Re\big(\lambda^2 z_{nres}W\big)\cdot W^2 R^3\,dR\big)(\hat{\tilde{\tau}})\\
			& + \mathcal{F}_{\tilde{\tau}}\big(\delta\tilde{\lambda}_1\big), 
		\end{align*}
		where $ \beta_*(\hat{\tilde{\tau}})$ is as in Proposition~\ref{prop:solnoftildelambdaeqn}, and the error term $\delta\tilde{\lambda}_1$ satisfies the bound 
		\begin{align*}
			\Big\|\delta\tilde{\lambda}_1\Big\|_{\tau^{-N}L^2_{d\tau}}\leq c(\tau_*, \gamma_1)\big\|z_{nres}\big\|_{S} +  \big\|(\tilde{\kappa}_1, \kappa_2)\big\|_{\tau^{-N}L^2_{d\tau}}
		\end{align*}
		where $\lim_{\tau_*, \gamma_1^{-1}\rightarrow +\infty}c(\tau_*, \gamma_1) = 0$. As a consequence, we can write(recall \eqref{eq:nprinequation1}) the left hand side of the preceding equation in the form 
		\begin{align*}
			&\mathcal{F}_{\tilde{\tau}}\big(\lambda^{-2}\langle\partial_{\tilde{\tau}}^2\rangle^{-1} Q^{(\tilde{\tau})}_{[\gamma, \gamma^{-1}]}\big(\tilde{\lambda}_{prin,\tilde{\tau}\tilde{\tilde{\tau}}}\cdot \lambda^2\big)(\hat{\tilde{\tau}})\\& =  \langle \hat{\tilde{\tau}}^2\rangle\cdot\beta_*(\hat{\tilde{\tau}})\cdot\langle \mathcal{F}_{\tilde{\tau}}(Q^{(\tilde{\tau})}_{[\gamma, \gamma^{-1}]}\tilde{n}_{prin})((\hat{\tilde{\tau}}, \cdot), W^2\rangle_{L^2_{R^3\,dR}} + \mathcal{F}_{\tilde{\tau}}\big(\delta\tilde{\lambda}_2\big),
		\end{align*}
		where $\delta\tilde{\lambda}_2$ satisfies similar bounds as $\delta\tilde{\lambda}_1$, but with $c(\tau_*,\gamma_1)$ replaced by $c(\epsilon_1, \gamma_1, N,\tau_*)$. Finally, we have the bound 
		\begin{align*}
			\big\|Q^{(\tilde{\tau})}_{[\gamma_1, \gamma_1^{-1}]}\big(\frac{\tilde{\lambda}_{\tilde{\tau}}}{\tilde{\tau}}\big)|_{[\tau_*,\infty)}\big\|_{\tau^{-N}L^2_{d\tau}}\ll_{\tau_*,\gamma_1}\big\|\langle \partial_{\tilde{\tau}}^2\rangle^{-1}\tilde{\lambda}_{\tilde{\tau}\tilde{\tau}}\big\|_{\tau^{-N}L^2_{d\tau}} 
		\end{align*}
	\end{lem}
	This is similar to Lemma~\ref{lem:lowtempfreqznresprin2}, we omit the simple details. 
	We now write \eqref{eq:tildenprinfull} in the following frequency localized form:
	\begin{equation}\label{eq:tildenprinfull1}\begin{split}
			\tilde{\Box}_W\big(I - \tilde{\Box}^{-1}\circ K_{main} - \tilde{\Box}^{-1}Z\big)Q^{(\tilde{\tau})}_{[\gamma_1, \gamma_1^{-1}]}\tilde{n}_{prin} = \tilde{\Box}_W F_1 + F_2 + F_3 + F_4,
	\end{split}\end{equation}
	where $\tilde{\Box}^{-1}$ is defined via \eqref{eq:goodinverseoftriangle+hattausquare+2Wsquare}, while we let 
	\begin{align*}
		&Z(n): =  \triangle\big(W\cdot \mathcal{L}^{-1}( \chi_{R\lesssim M}W\cdot \Box^{-1}(\langle \zeta(n), W^2\rangle \chi_{R\lesssim M}\Lambda W\cdot W)\big),\\&\mathcal{F}_{\tilde{\tau}}\big(\zeta(n)\big)(\hat{\tilde{\tau}}) = \langle \hat{\tilde{\tau}}^2\rangle^2\cdot \beta_*(\hat{\tilde{\tau}})\cdot\mathcal{F}_{\tilde{\tau}}(n) \\&F_1 = \big(\tilde{\Box}_W^{-1} - \tilde{\Box}^{-1}\big)\circ K_{main}\big(Q^{(\tilde{\tau})}_{[\gamma_1, \gamma_1^{-1}]}\tilde{n}_{prin}\big) + \big(\tilde{\Box}_W^{-1} - \tilde{\Box}^{-1}\big)\circ Z\big(Q^{(\tilde{\tau})}_{[\gamma_1, \gamma_1^{-1}]}\tilde{n}_{prin}\big),\\
		&F_2 = K_{small}\big(Q^{(\tilde{\tau})}_{[\gamma_1, \gamma_1^{-1}]}\tilde{n}_{prin}\big),\\
		&F_3 = Q^{(\tilde{\tau})}_{[\gamma_1, \gamma_1^{-1}]}\Big(\lambda^2\triangle \Re\big(W\mathcal{L}^{-1}(\overline{\tilde{E}_{main} + z_{nres, small}^{prin}}\big)\Big),
	\end{align*}
	and finally, we also let 
	\begin{align*}
		F_4 = \triangle\big(W\cdot \mathcal{L}^{-1}( Q^{(\tilde{\tau})}_{[\gamma_1, \gamma_1^{-1}]}y_{\tilde{\lambda}}\cdot W)\big) - Z(Q^{(\tilde{\tau})}_{[\gamma_1, \gamma_1^{-1}]}\tilde{n}_{prin} ). 
	\end{align*}
	\begin{rem}\label{rem:F5forlater} For later reference, we also introduce the quantities
		\begin{align*}
			&\tilde{Z}(n): =  \chi_{R\lesssim M}W\cdot \Box^{-1}(\langle \zeta(n), W^2\rangle \chi_{R\lesssim M}\Lambda W\cdot W)\\
			&F_5: =  Q^{(\tilde{\tau})}_{[\gamma_1, \gamma_1^{-1}]}y_{\tilde{\lambda}}\cdot W - \tilde{Z}(Q^{(\tilde{\tau})}_{[\gamma_1, \gamma_1^{-1}]}\tilde{n}_{prin}).
		\end{align*}
		We already note the a priori estimate 
		\begin{align*}
			\big\|F_5\big\|_{\tau^{-N}L^2_{d\tau}L^2_{R^3\,dR}}\lesssim c(\tau_*,\gamma_1)\big\|z_{nres}\big\|_{S} + \big\|(\tilde{\kappa}_1,\kappa_2)\big\|_{\tau^{-N}L^2_{d\tau}},
		\end{align*}
		where $\lim_{\tau_*,\gamma_1^{-1}\rightarrow\infty}c(\tau_*,\gamma_1) = 0$. In fact, this is a consequence of the preceding lemma and simple estimates. 
	\end{rem}
	
	To solve this equation, it suffices alternatively to solve 
	\begin{equation}\label{eq:tildenprinfull2}\begin{split}
			\big(I - \tilde{\Box}^{-1}\circ K_{main} - \tilde{\Box}^{-1}Z\big)Q^{(\tilde{\tau})}_{[\gamma_1, \gamma_1^{-1}]}\tilde{n}_{prin} = F_1 +  \tilde{\Box}_W^{-1}\big(F_2 + F_3 + F_4\big). 
	\end{split}\end{equation}
	In principle we would like to use Prop.~\ref{prop:inverseoftrueoperatorBoxinverseWK} but we have to contend with the additional term $\tilde{\Box}^{-1}Z\big(Q^{(\tilde{\tau})}_{[\gamma_1, \gamma_1^{-1}]}\tilde{n}_{prin}\big)$ on the left. The structure of this operator gets simplified if we pass to the temporal Fourier transform: 
	\begin{lem}\label{lem:delicateonedimensionaloperator} We have the representation 
		\begin{align*}
			&\mathcal{F}_{\tilde{\tau}}\big( \tilde{\Box}^{-1}Z\big(Q^{(\tilde{\tau})}_{[\gamma_1, \gamma_1^{-1}]}\tilde{n}_{prin} \big)\big)(\hat{\tilde{\tau}}, R)\\& = g(\hat{\tilde{\tau}}, R)\cdot  \langle \hat{\tilde{\tau}}^2\rangle\cdot\beta_*(\hat{\tilde{\tau}})\cdot\langle \chi_{[\gamma_1, \gamma_1^{-1}]}\mathcal{F}_{\tilde{\tau}}(\tilde{n}_{prin} ), W^2\rangle 
		\end{align*}
		where we have $\mathcal{F}_{\tilde{\tau}}^{-1}g(\cdot, R)\in \tau^{-N}L^2_{d\tau}\langle R\rangle R^{\delta_0}L^2_{R^3\,dR}$ is independent of $\tilde{n}_{prin}$, and $\beta_*(\hat{\tilde{\tau}})$ is as in Proposition~\ref{prop:solnoftildelambdaeqn}.
	\end{lem}
	\begin{proof} It suffices to set 
		\begin{align*}
			&g(\hat{\tilde{\tau}}, R)\\& = \big(\hat{\tilde{\tau}}^2 + \triangle + 2W^2)\big)_{good}^{-1}\big[\triangle\big(W\cdot \mathcal{L}^{-1}\big[\chi_{R\lesssim M}W\cdot (\hat{\tilde{\tau}}^2 + \triangle)_{good}^{-1}\big( \chi_{R\lesssim M}\Lambda W\cdot W\big)\big]\big),
		\end{align*}
		where the operator $(\hat{\tilde{\tau}}^2 + \triangle)_{good}^{-1}$ is given by \eqref{eq:goodinverseoftriangle+hattausquare+2Wsquare}, but omitting the first term (involving $\phi_{*,d}$) and replacing $\phi_*(R;\xi), \mathcal{F}_*, \rho_*$ by $\phi_{\R^4}(R;\xi), \mathcal{F}_{\R^4}, \rho_{\R^4}$, see 
		subsection~\ref{subsec:subsec:standardFourieronR4}.
	\end{proof}
	If we now apply the temporal Fourier transform to \eqref{eq:tildenprinfull2}, we infer the following equation, writing $\mathcal{F}_{\tilde{\tau}}\big(Q^{(\tilde{\tau})}_{[\gamma_1, \gamma_1^{-1}]}\tilde{n}_{prin}\big) =: \hat{n}_{prin}^{(\gamma_1)}$:
	\begin{equation}\label{eq:hatnprinfinaleqn1}\begin{split}
			&\big(I - (\hat{\tilde{\tau}}^2 + \triangle + 2W^2)_{good}^{-1}\circ K_{main}\big)\hat{n}_{prin}^{(\gamma_1)} - g(\hat{\tilde{\tau}}, R)\cdot  \langle \hat{\tilde{\tau}}^2\rangle\cdot\beta_*(\hat{\tilde{\tau}})\cdot\langle \hat{n}_{prin}^{(\gamma_1)}, W^2\rangle\\
			& = \mathcal{F}_{\tilde{\tau}}\Big(F_1+\tilde{\Box}_W^{-1}\big(F_2 + F_3 + F_4\big)\Big).
	\end{split}\end{equation}
	Using Lemma~\ref{lem:smoothnesswithrespecttohattauforinverse}, which refers to \eqref{eq:hhattauuhatttau}, we can alternatively formulate the preceding equation in fixed point form as
	\begin{equation}\label{eq:hatnprinfinaleqn2}\begin{split}
			&\hat{n}_{prin}^{(\gamma_1)} - \tilde{g}(\hat{\tilde{\tau}}, R)\cdot  \langle \hat{\tilde{\tau}}^2\rangle\cdot\beta_*(\hat{\tilde{\tau}})\cdot\langle \hat{n}_{prin}^{(\gamma_1)}, W^2\rangle = G(\hat{\tilde{\tau}}, R),\\
			&G(\hat{\tilde{\tau}}, R) = \big(I - (\hat{\tilde{\tau}}^2 + \triangle + 2W^2)_{good}^{-1}\circ K_{main}\big)^{-1}\Big( \mathcal{F}_{\tilde{\tau}}\Big(F_1 + \tilde{\Box}_W^{-1}\big(\sum_{j=2}^4 F_j\big)\Big)\Big).
	\end{split}\end{equation}
	and we use the notation
	\begin{align*}
		\tilde{g}(\hat{\tilde{\tau}}, R) = \big(I - (\hat{\tilde{\tau}}^2 + \triangle + 2W^2)_{good}^{-1}\circ K_{main}\big)^{-1}g(\hat{\tilde{\tau}}, R)
	\end{align*}
	It remains to solve \eqref{eq:hatnprinfinaleqn2} for $\hat{n}_{prin}^{(\gamma_1)}$. Note that if 
	\[
	G(\hat{\tilde{\tau}}, R) =  \tilde{g}(\hat{\tilde{\tau}}, R),
	\]
	then we can set $\hat{n}_{prin}^{(\gamma)}(\hat{\tilde{\tau}}, R) = \kappa(\hat{\tilde{\tau}})\cdot \tilde{g}(\hat{\tilde{\tau}}, R)$, where\footnote{Note that here we have $\langle \tilde{g}(\hat{\tilde{\tau}}, \cdot), W^2(\cdot)\rangle = \int_0^\infty  \tilde{g}(\hat{\tilde{\tau}}, R)\cdot W^2(R)R^3\,dR$.} 
	\begin{equation}\label{eq:kappaofhattildetaudef}
		\kappa(\hat{\tilde{\tau}}) = \big(1 -  \langle \hat{\tilde{\tau}}^2\rangle\cdot\beta_*(\hat{\tilde{\tau}})\cdot\langle \tilde{g}(\hat{\tilde{\tau}}, \cdot), W^2(\cdot)\rangle\big)^{-1}
	\end{equation}
	which is well-defined by means of non-degeneracy assumption {\bf{(B3)}}. Furthermore, one checks, using Lemma~\ref{lem:goodinverse1}, that $\kappa\in C^\infty(\R\backslash \{0\})$, whence 
	\[
	\mathcal{F}_{\tilde{\tau}}^{-1}\big(\hat{n}_{prin}^{(\gamma)}\big)|_{[\tau_*,\infty)}\in \tau^{-N}L^2_{d\tau}\langle R\rangle R^{\delta_0}L^2_{R^3\,dR}. 
	\]
	
	For general $G(\hat{\tilde{\tau}}, R) $, we set $\hat{n}_{prin}^{(\gamma_1)} = G(\hat{\tilde{\tau}}, R)$, which leads to the error 
	\[
	- \tilde{g}(\hat{\tilde{\tau}}, R)\cdot  \langle \hat{\tilde{\tau}}^2\rangle\cdot\beta_*(\hat{\tilde{\tau}})\cdot\langle G(\hat{\tilde{\tau}}, \cdot), W^2(\cdot)\rangle,
	\]
	which reduces the problem to the first case considered. This reasoning leads to the following 
	\begin{prop}\label{prop:tildenprinfinalkeyprop} The equation \eqref{eq:tildenprinfull2} admits a solution $Q^{(\tilde{\tau})}_{[\gamma_1, \gamma_1^{-1}]}\tilde{n}_{prin}$ satisfying the bound 
		\begin{align*}
			&\big\|\lambda^{-2}Q^{(\tilde{\tau})}_{[\gamma_1, \gamma_1^{-1}]}\tilde{n}_{prin}\big\|_{\tau^{-N}L^2_{d\tau}\langle R\rangle R^{\delta_0}L^2_{R^3\,dR}}\\&\leq c(\epsilon_1, M, N, \tau_*)\cdot\big[\big\|\lambda^{-2}\tilde{n}_{prin}\big\|_{\tau^{-N}L^2_{d\tau}\langle R\rangle R^{\delta_0}L^2_{R^3\,dR}} + \big\|z_{nres}\big\|_{S}\big]+ \big\|(\tilde{\kappa}_1, \kappa_2)\big\|_{\tau^{-N}L^2_{d\tau}}\\& + \big\|e_1\big\|_{\tau^{-N-1}L^2_{d\tau}L^2_{R^3\,dR}},
		\end{align*}
		where $\lim_{\epsilon_1^{-1}, M, N, \tau_*\rightarrow\infty}c(\epsilon_1, N,  M,\tau_*) = 0$. As a consequence, we infer 
		\begin{align*}
			\big\|\lambda^{-2}\tilde{n}_{prin}\big\|_{\tau^{-N}L^2_{d\tau}\langle R\rangle R^{\delta_0}L^2_{R^3\,dR}}&\leq c(\epsilon_1, M, N, \tau_*)\cdot \big\|z_{nres}\big\|_{S} + \big\|(\tilde{\kappa}_1, \kappa_2)\big\|_{\tau^{-N}L^2_{d\tau}}\\&  + \big\|e_1\big\|_{\tau^{-N-1}L^2_{d\tau}L^2_{R^3\,dR}},
		\end{align*}
		provided $\epsilon_1^{-1}, N, M,  \tau_*$ are sufficiently large. 
	\end{prop} 
	\begin{proof} In light of the preceding considerations, \eqref{eq:hatnprinfinaleqn1},  as well as Proposition~\ref{prop:inverseoftrueoperatorBoxinverseWK}, the problem reduces to establishing the bound 
		\begin{align*}
			&\big\|\lambda^{-2}F_1\big\|_{\tau^{-N}L^2_{d\tau}\langle R\rangle R^{\delta_0}L^2_{R^3\,dR}} + \sum_{j=2}^4\big\|\lambda^{-2}\tilde{\Box}_W^{-1}F_j\big\|_{\tau^{-N}L^2_{d\tau}\langle R\rangle R^{\delta_0}L^2_{R^3\,dR}}\\&\leq c(\epsilon_1, M, N, \tau_*)\cdot\big[\big\|\lambda^{-2}\tilde{n}_{prin}\big\|_{\tau^{-N}L^2_{d\tau}} + \big\|z_{nres}\big\|_{S}\big]\\&\hspace{1.5cm} + \big\|(\tilde{\kappa}_1, \kappa_2)\big\|_{\tau^{-N}L^2_{d\tau}} + \big\|e_1\big\|_{\tau^{-N-1}L^2_{d\tau}L^2_{R^3\,dR}}. 
		\end{align*}
		Here the bound for $F_1$ is follows from Lemma~\ref{lem:diffboxinverses1} and an analogous bound for $ \big(\tilde{\Box}_W^{-1} - \tilde{\Box}^{-1}\big)\circ Z\big(Q^{(\tilde{\tau})}_{[\gamma_1, \gamma_1^{-1}]}\tilde{n}_{prin}\big)$. 
		Next, Lemma~\ref{lem:Ksmallbound} implies the desired bound for $\big\|\lambda^{-2}\tilde{\Box}_W^{-1}F_2\big\|_{\tau^{-N}L^2_{d\tau}\langle R\rangle R^{\delta_0}L^2_{R^3\,dR}}$. For $\lambda^{-2}\tilde{\Box}_W^{-1}F_3$, the desired bounds follow from the easily verified inequalities\footnote{As before we let $\big\|\cdot\big\|_{\tilde{S}}$ denote the sum of the first three norms in \eqref{eq:Snormdefi}.}, 
		\begin{align*}
			\big\|\triangle\Re \big(W\cdot u\big)\big\|_{\langle R\rangle^{-1-\delta_0}L^2_{R^3\,dR}}\lesssim \big\|u\big\|_{\tilde{S}},\,\big\|\triangle\Re \big(W\cdot \mathcal{L}^{-1}v\big)\big\|_{\langle R\rangle^{-1-\delta_0}L^2_{R^3\,dR}}\lesssim \big\|v\big\|_{L^{2+}_{R^3\,dR}}.
		\end{align*}
		in conjunction with \eqref{eq:tildeEmainsmallness}, Lemma~\ref{lem:tildeboxwpropagatorapproximate} and a straightforward analogue of Lemma~\ref{lem:wavebasicinhom} for its principal part. Finally, for the term $\lambda^{-2}\tilde{\Box}_W^{-1}F_4$, we use Lemma~\ref{lem:tildelambdatimeslambdasquared} and simple estimates to conclude that 
		\begin{align*}
			\big\|\lambda^{-2}F_4\big\|_{\tau^{-N}L^2_{d\tau}\langle R\rangle^{1+\delta_0}L^2_{R^3\,dR}}\leq c(\tau_*)\cdot \big\|z_{nres}\big\|_{S} + \big\|(\tilde{\kappa}_1,\kappa_2)\big\|_{\tau^{-N}L^2_{d\tau}},
		\end{align*}
		where $\lim_{\tau_*\rightarrow +\infty}c(\tau_*) = 0$. 
	\end{proof}
	
	The preceding proposition forms the crux for the improved estimates for the non-resonant part, as it allows to easily establish an analogue of Lemma~\ref{lem:znresprinlargetemprfreqsmallness}, ~\ref{lem:smalltempfreqznresprinimprovedbound}:
	\begin{lem}\label{lem:znresprinintermediatetemprfreqimprovbound} We have the estimate 
		\begin{align*}
			\big\| Q^{(\tilde{\tau})}_{[\gamma_1,\gamma_1^{-1}]}(z_{nres}^{prin}) \big\|_{S}\lesssim c(\epsilon_1,M,N,\tau_*)\big\|z_{nres}\big\|_{S} + \big\|(\tilde{\kappa}_1,\kappa_2)\big\|_{\tau^{-N}L^2_{d\tau}} + \big\|e_1\big\|_{\tau^{-N-1}L^2_{d\tau}L^2_{R^3\,dR}}. 
		\end{align*}
	\end{lem}
	\begin{proof} This is a consequence of applying the frequency localizer $Q^{(\tilde{\tau})}_{[\gamma_1,\gamma_1^{-1}]}$ to the relation \eqref{eq:znreskeyreofrmulation1} and taking advantage of Prop.~\ref{prop:tildenprinfinalkeyprop}, Lemma~\ref{lem:tildelambdatimeslambdasquared}, \eqref{eq:tildeEmainsmallness} and Lemma~\ref{lem:basicSfromtildeL}.
	\end{proof}
	
	Combining the preceding lemma with Lemma~\ref{lem:znresprinlargetemprfreqsmallness}, Lemma~\ref{lem:smalltempfreqznresprinimprovedbound}, Lemma~\ref{lem:znresrestbound}, Lemma~\ref{lem:nonresconnect1}, Lemma~\ref{lem:znresperturbative1} with Proposition~\ref{prop:tildekappa1kappa2apriori}, and picking the variables $\epsilon_1^{-1}, M, N, \tau_*$ sufficiently large, we finally infer 
	\begin{prop}\label{prop:keyapriori} Assume that $z$ is given by \eqref{eq:zdecompbasic} with $z_{nres}\in S$, $(\kappa_1,\kappa_2)\in \tau^{-N}L^2_{d\tau}$. Further assume $\tilde{\lambda}$ satisfies $\langle\partial_{\tilde{\tau}}^2\rangle^{-1}\tilde{\lambda}_{\tilde{\tau}\tilde{\tau}}\in \tau^{-N}L^2_{d\tau}$, $\tilde{\alpha}_{\tau}\in \log^{-1}\tau\cdot\tau^{-N}L^2_{d\tau}$, and $z$ solves \eqref{eq:zeqn2}, $(\kappa_1, \kappa_2)$ satisfy \eqref{eq:kappa1eqn}, \eqref{eq:kappa2eqn}, while $(\tilde{\lambda}, \tilde{\alpha}$ are chosen such that \eqref{eq:tildelambda}, \eqref{eq:tildealpha1}. Throughout we work on $[\tau_*,\infty)\times \R^4$. Finally, assume that $N, \tau_*$ are sufficiently large. Then we can infer the a priori bound
		\begin{align*}
			\big\|z_{nres}\big\|_{S} + \big\|(\kappa_1,\kappa_2)\big\|_{\tau^{-N}L^2_{d\tau}} + \big\|\langle\partial_{\tilde{\tau}}^2\rangle^{-1}\tilde{\lambda}_{\tilde{\tau}\tilde{\tau}}\big\|_{\tau^{-N}L^2_{d\tau}}  + \big\|\tilde{\alpha}_{\tau}\big\|_{\log^{-1}\tau\cdot\tau^{-N}L^2_{d\tau}}\lesssim \big\|e_1\big\|_{\tau^{-N-1}L^2_{d\tau}},
		\end{align*}
		where the implied constant is universal. 
	\end{prop}
	
	\section{Construction of the solution}\label{sec:solncompletion}
	
	Finally, we can complete the proof of Theorem~\ref{thm:main} by using an iterative scheme to construct the tuple $\big(z_{nres}, \kappa_1,\kappa_2, \tilde{\lambda}, \tilde{\alpha}\big)$.
	\begin{prop}\label{prop:existenceofsolution} The combined system \eqref{eq:zeqn2}, \eqref{eq:kappa1eqn}, \eqref{eq:kappa2eqn}, \eqref{eq:tildelambda}, \eqref{eq:tildealpha1} admits a solution on $[\tau_*,\infty)$ for $\tau_*$ sufficiently large, satisfying the bound 
		\begin{align*}
			&\big\|z_{nres}\big\|_{S} + \big\|(\kappa_1,\kappa_2)\big\|_{\tau^{-N}L^2_{d\tau}[\tau_*,\infty)} + \big\|\langle\partial_{\tilde{\tau}}^2\rangle^{-1}\tilde{\lambda}_{\tilde{\tau}\tilde{\tau}}\big\|_{\tau^{-N}L^2_{d\tau}[\tau_*,\infty)}  + \big\|\tilde{\alpha}_{\tau}\big\|_{\log^{-1}\tau\cdot\tau^{-N}L^2_{d\tau}[\tau_*,\infty)}\\&\lesssim \big\|\langle \nabla^4\rangle e_1\big\|_{\tau^{-N-1}L^2_{d\tau}L^2_{R^3\,dR}}.
		\end{align*}
		We similarly have the bound 
		\begin{align*}
			\big\|y\big\|_{Y}\lesssim \big\|\langle \nabla^4\rangle e_1\big\|_{\tau^{-N-1}L^2_{d\tau}L^2_{R^3\,dR}} + \big\|\langle \nabla^4\rangle e_2\big\|_{\tau^{-N-1}L^2_{d\tau}L^2_{R^3\,dR}}. 
		\end{align*}
		
	\end{prop}
	\begin{proof} This follows by setting up an iterative scheme, and using the a priori bounds from the preceding sections to conclude convergence of the scheme. In general, when a quantity in the preceding sections is affixed with a subscript $j$, this means all the functions $z, \tilde{\lambda}$ etc used to define it are replaced by their $j$-th iterate in the iteration scheme. To begin with, we set (recall \eqref{eq:recallE})
		\begin{align*}
			E_0 = e_1,\,z_{nres,0} = 0,\,\tilde{\alpha}_0 = 0,\,\tilde{\lambda}_0 = 1,\,\tilde{\kappa}_{1,0} = 0,\,\kappa_{2,0} = 0. 
		\end{align*}
		For $j\geq 1$, we write the $j$-th iterate $z_{nres, j}$ of the non-resonant part $z_{nres}$ as (recalling \eqref{eq:snresdecomp})
		\begin{align*}
			z_{nres, j} = z_{nres,\mathcal{K}, j} + z_{nres*, j},
		\end{align*}
		where we define $z_{nres,\mathcal{K}, j}$ as in \eqref{eq:snresdecomp} but with $E$ replaced by $E_{j-1}$. Next, write 
		\begin{align*}
			z_{nres*, j} =  z_{nres*, <\epsilon_1, j} + z_{nres*, \sim\epsilon_1, j} + z_{nres*, >\epsilon_1^{-1}, j},
		\end{align*}
		where the first and third term are given by the first expression on the right of \eqref{eq:snresdecomp} smoothly localized to $\xi<\epsilon_1, \xi>\epsilon_1^{-1}$, respectively, and with $E$ replaced by $E_{j-1}$. For the middle term at intermediate frequencies, we decompose it into 
		\begin{align*}
			z_{nres*, \sim\epsilon_1, j}  = z_{nres*, \sim\epsilon_1, j}^{prin} + z_{nres*, \sim\epsilon_1, j}^{rest}, 
		\end{align*}
		where the second term on the right is given by the first expression on the right of \eqref{eq:snresdecomp} , localized to $\xi\in [\epsilon_1, \epsilon_1^{-1}]$, and with $E$ replaced by $\big(E - \lambda^{-2}(y_z+y_{\tilde{\lambda}})\cdot W\big)_{j-1}$. Next, we write
		\begin{align*}
			z_{nres*, \sim\epsilon_1, j}^{prin} = Q_{<\gamma_1}\big(z_{nres*, \sim\epsilon_1, j}^{prin}\big) + Q_{[\gamma_1,\gamma_1^{-1}]}\big(z_{nres*, \sim\epsilon_1, j}^{prin}\big) + Q_{>\gamma_1^{-1}}\big(z_{nres*, \sim\epsilon_1, j}^{prin}\big),
		\end{align*}
		The first term on the right is defined as solution of \eqref{eq:smallfreqznresprineffective},  but with $F_{3,4}$ replaced by $F_{3,j-1}, F_{4,j-1}$, respectively. The third term on the right is defined as solution of the equation preceding \eqref{eq:psidefn}, but with all terms on the right at iteration stage $j-1$. At this stage, it only remains to define $z_{nres*, \sim\epsilon_1, j}^{prin} $ to completely determine $z_{nres,j}$, which we do via an auxiliary function $Q_{[\gamma,\gamma^{-1}]}\tilde{n}_{prin,j}$. Define the latter as solution of \eqref{eq:tildenprinfull2} with all terms on the right evaluated at stage $j-1$. Then, keeping in mind \eqref{eq:znreskeyreofrmulation1} as well as Remark~\ref{rem:F5forlater},
		\begin{align*}
			Q_{[\gamma_1,\gamma_1^{-1}]}\big(z_{nres}^{prin}\big) &=\mathcal{L}^{-1} \big(2\lambda^{-2}Q_{[\gamma_1,\gamma_1^{-1}]}\tilde{n}_{prin,j}\cdot W\big) + \lambda^{-2}\mathcal{L}^{-1}\tilde{Z}\big(Q_{[\gamma_1,\gamma_1^{-1}]}\tilde{n}_{prin,j}\big)\\
			&+\mathcal{L}^{-1}\big( F_{5,j-1}+ \tilde{E}_{main,j-1} + z_{nres, small,j-1}^{prin}\big).
		\end{align*}
		Here the last quantity is defined as in \eqref{eq:znressmall} but with $z, \tilde{\lambda}$ at iteration stage $j-1$. \\
		We next use \eqref{eq:tildelambda}, with the first, third and fourth lines, as well as the terms $R^{(\tilde{\lambda})}_{\text{small}}, R^{(\tilde{\lambda})}_{\mathcal{K}}$ evaluated at iterative stage $j-1$, to define $\tilde{\lambda}_j$, using Proposition~\ref{prop:solnoftildelambdaeqn}.  
		Further, use \eqref{eq:tildealpha1}, with right hand side evaluated at stage $j-1$, to define $\tilde{\alpha}_{j}$ via Proposition~\ref{prop:tildealphomodeqn}. Finally, we determine $\kappa_{1,j}, \kappa_{2,j}$ via \eqref{eq:kappa1eqn}, \eqref{eq:kappa2eqn}, while evaluating the right hand sides at stage $j-1$. Using \eqref{eq:kapparefined} with $\kappa_1,\,\tilde{\lambda}$ replaced by $\kappa_{1,j}, \tilde{\lambda}_j$, we also define $\tilde{\kappa}_{1,j}$, and now the iterates of all the dynamical variables have been defined. Now the same estimates as those leading to Proposition~\ref{prop:keyapriori}, in addition to Remark~\ref{rem:F5forlater} together with Lemma~\ref{lem:basicSfromtildeL}  imply that 
		\begin{align*}
			&\big\|z_{nres, j} - z_{nres, j-1}\big\|_{S} + \big\|(\kappa_{1,j} - \kappa_{1,j-1},\kappa_{2,j} - \kappa_{2,j-1})\big\|_{\tau^{-N}L^2_{d\tau}[\tau_*,\infty)}\\& + \big\|\langle\partial_{\tilde{\tau}}^2\rangle^{-1}(\tilde{\lambda}_{j,\tilde{\tau}\tilde{\tau}} - \tilde{\lambda}_{j-1,\tilde{\tau}\tilde{\tau}})\big\|_{\tau^{-N}L^2_{d\tau}[\tau_*,\infty)}  + \big\|\tilde{\alpha}_{j, \tau} - \tilde{\alpha}_{j-1, \tau} \big\|_{\log^{-1}\tau\cdot\tau^{-N}L^2_{d\tau}[\tau_*,\infty)}\\
			&\leq \delta^j(\gamma_1, \epsilon_1, N, \tau_*)\cdot \big\|\langle \nabla^4\rangle e_1\big\|_{\tau^{-N-1}L^2_{d\tau}L^2_{R^3\,dR}},
		\end{align*}
		where $\lim_{\gamma_1^{-1},\epsilon_1^{-1}, N, \tau_*\rightarrow +\infty}\delta(\gamma, \epsilon_1, N, \tau_*) = 0$. 
		We can then also infer the desired bound for $y$, which we recall is given by means of \eqref{eq:yzdfn}, \eqref{eq:y2def}, and the second equation of \eqref{eq:zeqn2}. The desired bound is then a consequence of Lemma~\ref{lem:wavebasicinhom} and the already established bounds on $z_{nres}, \tilde{\lambda}, \tilde{\alpha}, \tilde{\kappa}_1, \kappa_2$. This completes the proof of the proposition. 
	\end{proof}
	
	{\it{Proof of Theorem~\ref{thm:main}}}. In light of the preceding proposition, it suffices to show that the functions $(\psi, n)$ given by \eqref{eq:finalansatz}, where $z$ is given by \eqref{eq:zedcomprefined} while the functions $z_{nres}, \tilde{\kappa}_1, \kappa_2, \tilde{\lambda}$ are given by the preceding proposition, satisfy the conclusions of Theorem~\ref{thm:main}. From the statement of the latter and \eqref{eq:finalansatz}, we infer that 
	\begin{align*}
		&\tilde{\psi} = -W_{\lambda(t)} + \psi_*^{(\tilde{\lambda}, \underline{\tilde{\alpha}})}  + e^{i\alpha(t)}\lambda(t)\cdot z\\
		&\tilde{n} = -W_{\lambda(t)}^2 + n_*^{(\tilde{\lambda}, \underline{\tilde{\alpha}})} + y.
	\end{align*}
	Using Lemma~\ref{lem:approxsolasymptotics1} we infer for any $t\in (0, t_0]$ (recall $R = \lambda\cdot r$)
	\begin{align*}
		\big\| -W_{\lambda(t)} + \psi_*^{(\tilde{\lambda}, \underline{\tilde{\alpha}})} \big\|_{H^2_{r^3\,dr,\,loc}}\lesssim 1.
	\end{align*}
	As for the contribution involving $z$, we can use the first equation in \eqref{eq:zeqn2} together with Lemma~\ref{lem:Xtildelambdaperturbterms}, Lemma~\ref{lem:ytildelambdamodhightempfreq}, Lemma~\ref{lem:Xtildelambdafinaltermcrudebound}, and straightforward bounds applied to the terms in \eqref{eq:E1mod} after multiplication by $\lambda^{-3}$ (recall \eqref{eq:e1moddef}) that we have the crude bound (for some absolute constant $C$)
	\begin{align*}
		\big\|z_{\tau}\big\|_{\tau^{-N+C}L^2_{d\tau}H^2_{r^3\,dr,loc}}\lesssim 1. 
	\end{align*}
	Combined with $\big\|z\big\|_{\tau^{-N+C}L^2_{d\tau}H^2_{r^3\,dr,loc}}\lesssim 1$ which follows from the preceding proposition and a simple argument, we infer that we have 
	\[
	\big\|z(\tau,\cdot)\big\|_{H^2_{r^3\,dr,loc}}\lesssim 1
	\]
	for any $\tau\in [\tau_*,\infty)$, whence $\tilde{\psi}(t, \cdot)\in H^2_{\mathbb{R}^4,loc}$, $t\in (0, t_0]$. One shows similarly that $\tilde{n}(t,\cdot)\in H^1_{\mathbb{R}^4,loc}$, $t\in (0, t_0]$. The remaining assertions of the theorem are also straightforward consequences of the proposition.

	\section{Technical details}\label{sec:appendix}

	\subsection{Frequently used Technical tools}\label{sec:techtools}
	
	In order to control certain integrals appearing in the description of the source terms for the $\kappa_1$-evolution, \eqref{eq:kappa1eqn}, the following lemma shall be useful. Let us denote 
	\begin{align*}
		K_f(\tau): = \int_{\tau}^\infty\int_0^\infty \xi^2 S_1(\tau, \sigma,\xi)\cdot f(\sigma, \frac{\lambda(\tau)}{\lambda(\sigma)}\xi)\rho_1(\xi)\,d\xi d\sigma 
	\end{align*}
	
	\begin{lem}\label{lem:K_frefined} We have the bound 
		\begin{align*}
			\big\|K_f\big\|_{\log^{-2}(\tau)\cdot\tau^{-N} L^2_{d\tau}}&\lesssim_{\delta_0} \big\|\langle\xi^{(1-\delta_0)}\partial_{\xi}\rangle^{1+\delta_0}f\big\|_{\tau^{-N} L^2_{d\tau}L^\infty_{d\xi}} + \big\|\xi^{1-\delta_0}\partial_{\xi}f\big\|_{\tau^{-N+\frac{\delta_0^2}{4}} L^2_{d\tau}L^\infty_{d\xi}}\\&+\big\|\partial_{\tau}f\big\|_{\tau^{-N-\delta_0}L^2_{d\tau}L^\infty_{d\xi}} 
		\end{align*}
	\end{lem}
	\begin{proof} We decompose the expression into several pieces, by inclusion of cutoffs. For $0<\delta_1\ll \delta_0$ let $\chi_{\sigma-\tau<\tau^{\delta_1}}$ be a smooth cutoff localising to the indicated region: 
		\\
		
		{\it{(1): the regime $\sigma - \tau<\tau^{\delta_1}$.}} Including the cutoff $\chi_{\sigma-\tau<\tau^{\delta_1}}$ in the integral and performing integration by parts with respect to $\sigma$, we arrive at the term
		\begin{align*}
			\int_{\tau}^\infty\int_0^\infty \partial_{\sigma}\big(\chi_{\sigma-\tau<\tau^{\delta_1}}\big)\cdot S_2(\tau, \sigma,\xi)\cdot \frac{\lambda^2(\sigma)}{\lambda^2(\tau)}f(\sigma, \xi)\rho_1(\xi)\,d\xi d\sigma,
		\end{align*}
		as well as 
		\begin{align*}
			\int_{\tau}^\infty\int_0^\infty \chi_{\sigma-\tau<\tau^{\delta_1}}\cdot S_2(\tau, \sigma,\xi)\cdot \partial_{\sigma}\big(\frac{\lambda^2(\sigma)}{\lambda^2(\tau)}f(\sigma,  \frac{\lambda(\tau)}{\lambda(\sigma)}\xi)\big)\rho_1(\xi)\,d\xi d\sigma.
		\end{align*}
		The first integral is then supported in the region $\sigma - \tau\sim \tau^{\delta_1}$, and we have 
		\begin{align*}
			\big|\partial_{\sigma}\big(\chi_{\sigma-\tau<\tau^{\delta_1}}\big)\big|\lesssim \tau^{-\delta_1}. 
		\end{align*}
		To estimate this integral, we further distinguish between the regimes $\xi^2\lesssim \tau^{-\delta_1},\,\xi^2\gtrsim \tau^{-\delta_1}$. In the former regime, use the bound 
		\begin{align*}
			\Big|S_2(\tau, \sigma,\xi)\Big|\lesssim \xi^2\cdot\tau^{\delta_1}, 
		\end{align*}
		which in turn implies
		\begin{align*}
			\int_0^\infty \chi_{\xi^2<\tau^{-\delta_1}}\big|S_2(\tau, \sigma,\xi)\big|\cdot\big|f(\sigma,  \frac{\lambda(\tau)}{\lambda(\sigma)}\xi)\big|\rho_1(\xi)\,d\xi\lesssim_{\delta_1} \big\|f(\sigma, \cdot)\big\|_{L^\infty_{d\xi}}\cdot \log^{-2}(\tau).
		\end{align*}
		A similar bound is obtained in the region $\xi^2\gtrsim \tau^{-\delta_1}$ upon integrating by parts with respect to $\xi$ and recalling that $\sigma - \tau\sim \tau^{\delta_1}$. 
		Since we have the bounds 
		\begin{align*}
			\Big\|\chi_{\sigma-\tau>0}\partial_{\sigma}\big(\chi_{\sigma-\tau<\tau^{\delta_1}}\big)\Big\|_{L_{\tau}^\infty L_{\sigma}^1\cap L_{\sigma}^{\infty}L_{\tau}^1}\lesssim 1, 
		\end{align*}
		we easily infer from Schur's criterion the bound 
		\begin{align*}
			&\Big\|\int_{\tau}^\infty\int_0^\infty \partial_{\sigma}\big(\chi_{\sigma-\tau<\tau^{\delta_1}}\big)\cdot S_2(\tau, \sigma,\xi)\cdot \frac{\lambda^2(\sigma)}{\lambda^2(\tau)}f(\sigma, \xi)\rho_1(\xi)\,d\xi d\sigma\Big\|_{\log^{-2}\tau\cdot \tau^{-N}L^2_{d\tau}}\\
			&\lesssim \big\|f\big\|_{\tau^{-N}L^2_{d\tau}L^\infty_{d\xi}}, 
		\end{align*}
		verifying the claim of the lemma for this contribution. 
		\\
		As for the second integral expression generated by integration by parts with respect to $\sigma$, since we have 
		\begin{align*}
			\big\|\chi_{\sigma-\tau>0}\chi_{\sigma-\tau<\tau^{\delta_1}}\sigma^{-\delta_1}\big\|_{L_{\tau}^\infty L_{\sigma}^1\cap L_{\sigma}^{\infty}L_{\tau}^1}\lesssim 1, 
		\end{align*}
		we infer from Schur's criterion again the bound 
		\begin{align*}
			&\Big\|\int_{\tau}^\infty\int_0^\infty \chi_{\sigma-\tau<\tau^{\delta_1}}\cdot S_2(\tau, \sigma,\xi)\cdot \partial_{\sigma}\big(\frac{\lambda^2(\sigma)}{\lambda^2(\tau)}f(\sigma,  \frac{\lambda(\tau)}{\lambda(\sigma)}\xi)\big)\rho_1(\xi)\,d\xi d\sigma\Big\|_{\log^{-2}\tau\cdot \tau^{-N}L^2_{d\tau}}\\
			&\lesssim \big\|\sigma^{-1+\delta_1}f\big\|_{\sigma^{-N-}L^2_{d\sigma}L^\infty_{d\xi}} + \big\|\sigma^{\delta_1}\partial_{\sigma}f\big\|_{\sigma^{-N-}L^2_{d\sigma}L^\infty_{d\xi}} + \big\|\sigma^{-1+\delta_1}(\xi^{1-\delta_0}\partial_{\xi})f\big\|_{\sigma^{-N-}L^2_{d\sigma}L^\infty_{d\xi}}
		\end{align*}
		To arrive at this bound, we have used the more crude estimate 
		\begin{align*}
			\big\|S_2(\tau, \sigma,\xi)g(\sigma, \xi)\rho_1(\xi)\big\|_{L^1_{d\xi}}\lesssim \big\|g(\sigma, \cdot)\big\|_{L^\infty_{d\xi}}.
		\end{align*}
		The conclusion of the lemma is again implied provided $\delta_1<\min\{1,\delta_0\}$. 
		\\
		
		{\it{(2): the regime $\sigma - \tau\geq \tau^{\delta_1}$.}} Here we shall perform integration by parts with respect to $\xi$ instead, exploiting the simple identity 
		\begin{align*}
			2\xi\cdot S_1(\tau, \sigma, \xi) = \frac{\partial_{\xi} S_2(\tau, \sigma, \xi)}{\lambda^2(\tau)\cdot\int_{\sigma}^{\tau}\lambda^{-2}(s)\,ds}. 
		\end{align*}
		Carrying out the integration by parts, we arrive at the expression 
		\begin{align*}
			2\int_{\tau}^\infty\int_0^\infty \chi_{\sigma-\tau>\tau^{\delta_1}}\cdot \frac{S_2(\tau, \sigma,\xi)}{\lambda^2(\tau)\cdot\int_{\sigma}^{\tau}\lambda^{-2}(s)\,ds}\cdot \partial_{\xi}\big(f(\sigma, \frac{\lambda^2(\tau)}{\lambda^2(\sigma)}\xi)\xi\rho_1(\xi)\big)\,d\xi d\sigma
		\end{align*}
		Similarly to case 1, we split this into the regions $(\sigma - \tau)\xi^2\gtrsim 1, (\sigma-\tau)\xi^2\lesssim 1$. In the latter, we take advantage of the bound 
		\begin{align*}
			&\Big|\int_0^\infty \chi_{\sigma-\tau>\tau^{\delta_1}}\cdot \chi_{(\sigma-\tau)\xi^2\lesssim 1}\cdot \frac{S_2(\tau, \sigma,\xi)}{\lambda^2(\tau)\cdot\int_{\sigma}^{\tau}\lambda^{-2}(s)\,ds}\cdot \partial_{\xi}\big(f(\sigma, \frac{\lambda^2(\tau)}{\lambda^2(\sigma)}\xi)\xi\rho_1(\xi)\big)\,d\xi\Big|\\
			&\lesssim \big\|(\xi^{1-\delta_0}\partial_{\xi})f(\sigma, \cdot)\big\|_{L^\infty_{d\xi}}\cdot (\sigma-\tau)^{-1}\cdot\tau^{-\frac{\delta_1\cdot\delta_0}{2}}\log^{-2}\tau\\&\hspace{5cm}
			+ \big\|f(\sigma, \cdot)\big\|_{L^\infty_{d\xi}}\cdot (\sigma-\tau)^{-1}\cdot \log^{-3}(\tau).
		\end{align*}
		Then the desired bound follows for this contribution as in case 1 by taking advantage of Schur's criterion applied to the function 
		\[
		\chi_{\sigma-\tau>\tau^{\delta_1}}\big[(\sigma-\tau)^{-1}\cdot\tau^{-\frac{\delta_1\cdot\delta_0}{2}}\log^{-2}\tau + (\sigma-\tau)^{-1}\cdot \log^{-3}(\tau)\big].
		\]
		and setting $\delta_1 = \frac{2\delta_0}{3}$, say. 
		
		In the former case, we perform additional integration by parts with respect to $\xi$ to obtain a gain of the form $\big(\xi^2(\sigma - \tau)\big)^{-\delta_0}$, and the estimate is again easily concluded as in the preceding case.

	\end{proof}
	
	We further have the following slight variant of the preceding lemma: let 
	\begin{align*}
		&\tilde{K}_f: =  \int_{\tau}^\infty\int_0^\infty \xi^2\sin\big(\lambda^2(\tau)\xi^2\int_{\tau}^{\sigma}\lambda^{-2}(s)\,ds\big)\cdot f(\sigma, \frac{\lambda(\tau)}{\lambda(\sigma)}\xi)\rho_1(\xi)\,d\xi d\sigma,\\
		&\tilde{\tilde{K}}_f: =  \int_{\tau}^\infty\int_0^\infty \xi^2\cos\big(\lambda^2(\tau)\xi^2\int_{\tau}^{\sigma}\lambda^{-2}(s)\,ds\big)\cdot f(\sigma, \frac{\lambda(\tau)}{\lambda(\sigma)}\xi)\rho_1(\xi)\,d\xi d\sigma 
	\end{align*}
	\begin{lem}\label{lem:tildeKfcontrol} We have the following estimate: for any $\delta_1>0$,
		\begin{align*}
			\Big\|\tilde{K}_f\Big\|_{\tau^{-N} L^2_{d\tau}} + \Big\|\tilde{\tilde{K}}_f\Big\|_{\tau^{-N} L^2_{d\tau}} \lesssim_{\delta_1} \big\|\langle \xi\partial_{\xi}\rangle^{1+\delta_1} f\big\|_{\tau^{-N} L^2_{d\tau} L^{2+}_{\rho_1(\xi)\,d\xi}}.
		\end{align*}
		If $\gamma\ll 1$ and we include an extra smooth cutoff $\chi_{\xi<\gamma}$, we gain an extra smallness constant $c(\gamma)$ in this estimate where $\lim_{\gamma\rightarrow 0}c(\gamma) = 0$.  
	\end{lem}
	\begin{proof} We give the proof for $\tilde{K}_f$, the one for $\tilde{\tilde{K}}_f$ following similarly. We start by observing that for $\tau\leq \sigma\lesssim \tau$ we have 
		\begin{align*}
			\lambda^2(\tau)\xi^2\int_{\tau}^{\sigma}\lambda^{-2}(s)\,ds\sim \xi^2(\sigma - \tau). 
		\end{align*}
		Then decompose 
		\begin{align*}
			\tilde{K}_f &= \int_{\tau}^\infty\int_0^\infty \chi_1(\sigma,\tau;\xi)\xi^2\sin\big(\lambda^2(\tau)\xi^2\int_{\tau}^{\sigma}\lambda^{-2}(s)\,ds\big)\cdot f(\sigma, \xi)\rho_1(\xi)\,d\xi d\sigma\\
			& +  \int_{\tau}^\infty\int_0^\infty (1-\chi_1(\sigma,\tau;\xi))\xi^2\sin\big(\lambda^2(\tau)\xi^2\int_{\tau}^{\sigma}\lambda^{-2}(s)\,ds\big)\cdot f(\sigma, \xi)\rho_1(\xi)\,d\xi d\sigma\\
			& =: \tilde{K}_{f1} + \tilde{K}_{f2}, 
		\end{align*}
		where $ \chi_1(\sigma,\tau;\xi)$ smoothly localizes to $\lambda^2(\tau)\xi^2\int_{\tau}^{\sigma}\lambda^{-2}(s)\,ds\lesssim 1$. To estimate the first term on the right, use that 
		\begin{align*}
			\Big\|\chi_1(\sigma,\tau;\xi)\xi^2\sin\big(\lambda^2(\tau)\xi^2\int_{\tau}^{\sigma}\lambda^{-2}(s)\,ds\big)\Big\|_{L^{2-}_{\rho_1(\xi)\,d\xi}}\lesssim \frac{\big(\lambda^2(\tau)\int_{\tau}^{\sigma}\lambda^{-2}(s)\,ds\big)^{-1}}{\log^{1+} \big[\langle \lambda^2(\tau)\int_{\tau}^{\sigma}\lambda^{-2}(s)\,ds\rangle\big]},
		\end{align*}
		whence from Holder's inequality we infer 
		\begin{align*}
			&\Big|\int_0^\infty \chi_1(\sigma,\tau;\xi)\xi^2\sin\big(\lambda^2(\tau)\xi^2\int_{\tau}^{\sigma}\lambda^{-2}(s)\,ds\big)\cdot f(\sigma, \xi)\rho_1(\xi)\,d\xi d\sigma\big|\\
			&\lesssim \frac{\big(\lambda^2(\tau)\int_{\tau}^{\sigma}\lambda^{-2}(s)\,ds\big)^{-1}}{\log^{1+} \big[\langle \lambda^2(\tau)\int_{\tau}^{\sigma}\lambda^{-2}(s)\,ds\rangle\big]}\cdot \big\|f(\sigma, \cdot)\big\|_{L^{2+}_{\rho_1(\xi)\,d\xi}}. 
		\end{align*}
		Using Schur's criterion we then easily deduce 
		\begin{align*}
			\Big\|\tilde{K}_{f1}\Big\|_{\tau^{-N} L^2_{d\tau}}\lesssim \big\|f\big\|_{\tau^{-N} L^2_{d\tau}L^2_{\rho_1(\xi)\,d\xi}}
		\end{align*}
		As for $\tilde{K}_{f2}$, we use integration by parts with respect to $\xi$. Thus write 
		\begin{align*}
			\xi^2\sin\big(\lambda^2(\tau)\xi^2\int_{\tau}^{\sigma}\lambda^{-2}(s)\,ds\big) = \frac{-\frac12 \xi\cdot \partial_\xi\Big(\cos\big(\lambda^2(\tau)\xi^2\int_{\tau}^{\sigma}\lambda^{-2}(s)\,ds\big)\Big)}{\lambda^2(\tau)\int_{\tau}^{\sigma}\lambda^{-2}(s)\,ds},
		\end{align*}
		whence writing $\tilde{K}_{f2} = \int_{\tau}^\infty \tilde{k}_{f2}(\sigma)\,d\sigma$, we have 
		\begin{align*}
			\tilde{k}_{f2}(\sigma) = \int_0^\infty\frac{\cos\big(\lambda^2(\tau)\xi^2\int_{\tau}^{\sigma}\lambda^{-2}(s)\,ds\big)}{\lambda^2(\tau)\int_{\tau}^{\sigma}\lambda^{-2}(s)\,ds}\cdot  \partial_{\xi}\Big((1-\chi_1(\sigma,\tau;\xi))f(\sigma, \xi)\frac12\xi\rho_1(\xi)\Big)\,d\xi 
		\end{align*}
		Further fractional integration by parts and application of the Holder's inequality lead to the bound 
		\begin{align*}
			\big| \tilde{k}_{f2}(\sigma) \big|\lesssim   \frac{\big(\lambda^2(\tau)\int_{\tau}^{\sigma}\lambda^{-2}(s)\,ds\big)^{-1}}{\log^{1+} \big[\langle \lambda^2(\tau)\int_{\tau}^{\sigma}\lambda^{-2}(s)\,ds\rangle\big]}\cdot \big\|\langle \xi\partial_{\xi}\rangle^{1+\delta_1} f(\sigma, \cdot)\big\|_{ L^{2+}_{\rho_1(\xi)\,d\xi}},
		\end{align*}
		and from here on one concludes again by means of Schur's criterion. 
	\end{proof}
	\begin{rem}\label{rem:tildeKfcontrol} The preceding proof also implies that the same expressions $\tilde{K}_f, \tilde{\tilde{K}}_f$ but with an extra cutoff $\chi_{|\sigma - \tau|\gtrsim \tau^{\delta}}$ for some $\delta>0$ map from $\tau^{-N}L^2_{d\tau}$ into $\log^{-1}\tau\cdot \tau^{-N}L^2_{d\tau}$, provided we replace the norm $L^{2+}_{\rho(\xi)\,d\xi}$ by $L^{\infty}_{\rho(\xi)\,d\xi}$.
	\end{rem}
	
	The following lemma gives a kind of 'interpolate' between the preceding two: let
	\begin{align*}
		L_f: = & \int_{\tau}^\infty\int_0^\infty \xi^2\frac{\lambda^2(\tau)}{\lambda^2(\sigma)}\sin\big(\lambda^2(\tau)\xi^2\int_{\sigma}^{\tau}\lambda^{-2}(s)\,ds\big)\cdot f(\sigma, \frac{\lambda(\tau)}{\lambda(\sigma)}\xi)\rho(\xi)\,d\xi d\sigma\\& + \int_0^\infty f(\tau,\xi)\rho(\xi)\,d\xi. 
	\end{align*}
	\begin{lem}\label{lem:Lfdifferencebound} We have the bound 
		\begin{align*}
			\big\|L_f\big\|_{\log^{-1}\tau\cdot\tau^{-N}L^2_{d\tau}}\lesssim \big\|\langle\xi\partial_{\xi}\rangle^{1+\delta_0}f\big\|_{\tau^{-N}L^2_{d\tau}L^\infty_{d\xi}} + \big\|\partial_{\tau}f\big\|_{\tau^{-N-\delta_0}L^2_{d\tau}L^\infty_{d\xi}}
		\end{align*}
	\end{lem}
	\begin{proof} We split the double integral into two portions $L_f = L_f^1 + L_f^2$, where we set 
		\begin{align*}
			L_f^1(\tau): =  \int_{\tau}^{\tau+\tau^{\delta}}\int_0^\infty \xi^2\frac{\lambda^2(\tau)}{\lambda^2(\sigma)}\sin\big(\lambda^2(\tau)\xi^2\int_{\sigma}^{\tau}\lambda^{-2}(s)\,ds\big)\cdot f(\sigma, \frac{\lambda(\tau)}{\lambda(\sigma)}\xi)\rho(\xi)\,d\xi d\sigma
		\end{align*}
		for some $\delta>0$. Performing integration by parts with respect to $\sigma$, we replace this by the difference term 
		\begin{align*}
			\Big(\int_0^\infty \cos\big(\lambda^2(\tau)\xi^2\int_{\sigma}^{\tau}\lambda^{-2}(s)\,ds\big)\cdot f(\sigma, \frac{\lambda(\tau)}{\lambda(\sigma)}\xi)\rho(\xi)\,d\xi\Big)|_{\sigma = \tau}^{\sigma = \tau+\tau^{\delta}}
		\end{align*}
		as well as the double integral 
		\begin{align*}
			- \int_{\tau}^{\tau+\tau^{\delta}}\int_0^\infty\cos\big(\lambda^2(\tau)\xi^2\int_{\sigma}^{\tau}\lambda^{-2}(s)\,ds\big)\cdot \partial_{\sigma}\big(f(\sigma, \frac{\lambda(\tau)}{\lambda(\sigma)}\xi)\big)\rho(\xi)\,d\xi d\sigma.
		\end{align*}
		The lower difference term cancels against the single integral 
		\[
		\int_0^\infty f(\tau,\xi)\rho(\xi)\,d\xi, 
		\]
		and so the difference reduces to 
		\[
		\int_0^\infty \cos\big(\lambda^2(\tau)\xi^2\int_{\tau+\tau^{\delta}}^{\tau}\lambda^{-2}(s)\,ds\big)\cdot f(\tau+\tau^{\delta}, \frac{\lambda(\tau)}{\lambda(\tau+\tau^{\delta})}\xi)\rho(\xi)\,d\xi
		\]
		But the latter expression is easily bounded by 
		\begin{align*}
			\Big\|\cdot\Big\|_{\log^{-1}(\tau)\tau^{-N}L^2_{d\tau}}\lesssim A, 
		\end{align*}
		where we call $A$ the expression on the right in the statement of the lemma. In fact, one obtains this bound by dividing into the cases $\xi^2\tau^{\delta}\lesssim 1$, $\xi^2\tau^{\delta}\gtrsim 1$, and performing integration by parts in the latter region. 
		As for the remaining double integral, it is also easy to bound provided we choose $\delta<\delta_0$.  
		\\
		We are now left with the task of bounding the remaining double integral 
		\begin{align*}
			&L_f^2(\tau): =\\& \int_{\tau+\tau^{\delta}}^\infty\int_0^\infty \xi^2\frac{\lambda^2(\tau)}{\lambda^2(\sigma)}\sin\big(\lambda^2(\tau)\xi^2\int_{\sigma}^{\tau}\lambda^{-2}(s)\,ds\big)\cdot f(\sigma, \frac{\lambda(\tau)}{\lambda(\sigma)}\xi)\rho(\xi)\,d\xi d\sigma
		\end{align*}
		To deal with it, we divide into the regions $(\sigma-\tau)\xi^2\lesssim 1,\,(\sigma - \tau)\xi^2\gtrsim 1$. In the former, we use the bound 
		\begin{align*}
			&\Big|\int_0^{(\sigma-\tau)^{-1}} \xi^2\frac{\lambda^2(\tau)}{\lambda^2(\sigma)}\sin\big(\lambda^2(\tau)\xi^2\int_{\sigma}^{\tau}\lambda^{-2}(s)\,ds\big)\cdot f(\sigma, \frac{\lambda(\tau)}{\lambda(\sigma)}\xi)\rho(\xi)\,d\xi\Big|\\
			&\lesssim (\sigma-\tau)^{-1}\cdot\log^{-2}(\sigma-\tau)\cdot \big\|f(\sigma,\cdot)\big\|_{L^{\infty}_{d\xi}}
		\end{align*}
		and the desired bound follows for this contribution by means of Schur's criterion applied to the function 
		\[
		\chi_{\sigma-\tau>\tau^{\delta}}\cdot  (\sigma-\tau)^{-1}\cdot\log^{-2}(\sigma-\tau). 
		\]
		The case $(\sigma - \tau)\xi^2\gtrsim 1$ is handled by integration by parts with respect to $\xi$. 
	\end{proof}
	\begin{rem}\label{rem:Lfdifferencebound} The preceding proof easily produces the following sharpened version if we restrict to frequencies $\xi$ away from zero: there exists $0<\delta_2\ll 1$ as well as $\delta_3>0$ such that letting $L_f^{(>\tau^{-\delta_2})}$ be defined as $L_f$ but with an extra cutoff $\chi_{\xi>\tau^{-\delta_2}}$, we can bound 
		\begin{align*}
			\big\|L_f^{(>\tau^{-\delta_2})}\big\|_{\tau^{-N-\delta_3}L^2_{d\tau}}
		\end{align*}
		in terms of the right hand side of the preceding lemma.
	\end{rem}
	
	To control the 'left-over' resonant part of $z$, we shall require a variant of the previous lemmas with a source term $f$ of a more special structure
	\begin{lem}\label{lem:derivativemovesthrough1} Let us define 
		\begin{align*}
			M_f(\tau): =  \int_{\tau}^\infty\int_0^\infty \xi^2 S(\tau, \sigma,\xi)\cdot (\partial_{\sigma}f)(\sigma, \frac{\lambda(\tau)}{\lambda(\sigma)}\xi)\rho_1(\xi)\,d\xi d\sigma 
		\end{align*}
		Then we can write 
		\begin{align*}
			M_f(\tau) = \partial_{\tau}M_f^{(1)}(\tau) + \tau^{-1}M_f^{(2)}(\tau), 
		\end{align*}
		where we have the bounds 
		\begin{align*}
			\sum_j\big\|M_f^{(j)}\big\|_{\tau^{-N}L^2_{d\tau}}\lesssim \big\|\langle\xi\partial_{\xi}\rangle^{1+\delta_0}f\big\|_{\tau^{-N}L^2_{d\tau}}
		\end{align*}
	\end{lem}
	\begin{proof} We pass to the integration variable $\tilde{\xi} = \frac{\lambda(\tau)}{\lambda(\sigma)}\xi$ instead of $\xi$, and then perform integration by parts with respect to $\sigma$. This replaces $M_f(\tau)$ by 
		\begin{align*}
			-\int_0^\infty \xi^2 f(\tau,\xi)\rho_1(\xi)\,d\xi &-  \int_{\tau}^\infty\int_0^\infty \tilde{\xi}^2 \partial_{\sigma}\tilde{S}(\tau, \sigma,\tilde{\xi})\cdot f(\sigma, \tilde{\xi})\frac{\lambda(\sigma)}{\lambda(\tau)}\rho_1(\frac{\lambda(\sigma)}{\lambda(\tau)}\tilde{\xi})\,d\tilde{\xi} d\sigma\\
			& -  \int_{\tau}^\infty\int_0^\infty \tilde{\xi}^2 \tilde{S}(\tau, \sigma,\tilde{\xi})\cdot f(\sigma, \tilde{\xi})\partial_{\sigma}\big(\frac{\lambda(\sigma)}{\lambda(\tau)}\rho_1(\frac{\lambda(\sigma)}{\lambda(\tau)}\tilde{\xi})\big)\,d\tilde{\xi} d\sigma,
		\end{align*}
		where we have introduced the auxiliary function 
		\begin{align*}
			\tilde{S}(\tau, \sigma,\tilde{\xi}) = ie^{\lambda^2(\sigma)\tilde{\xi}^2\int_{\sigma}^{\tau}\lambda^{-2}(s)\,ds}
		\end{align*}
		Observing that 
		\begin{align*}
			\partial_{\sigma}\tilde{S}(\tau, \sigma,\tilde{\xi}) = \partial_{\tau}\big(\zeta(\tau,\sigma)\cdot\tilde{S}(\tau, \sigma,\tilde{\xi})\big) - \partial_{\tau}\big(\zeta(\tau,\sigma)\big)\cdot\tilde{S}(\tau, \sigma,\tilde{\xi}),
		\end{align*}
		where $\zeta(\tau,\sigma)$ satisfies $\zeta(\tau,\tau) = 1$, $\big|\partial_{\tau}\zeta(\tau,\sigma)\big|\lesssim \tau^{-1}$, we can write the sum of the first two terms in the previous sum of integrals as 
		\begin{align*}
			&- \partial_{\tau}\Big(\int_{\tau}^\infty\int_0^\infty \tilde{\xi}^2\zeta(\tau,\sigma)\tilde{S}(\tau, \sigma,\tilde{\xi})\cdot f(\sigma, \tilde{\xi})\frac{\lambda(\sigma)}{\lambda(\tau)}\rho_1(\frac{\lambda(\sigma)}{\lambda(\tau)}\tilde{\xi})\,d\tilde{\xi} d\sigma\Big)\\
			& + \int_{\tau}^\infty\int_0^\infty \tilde{\xi}^2\partial_{\tau}(\zeta(\tau,\sigma))\tilde{S}(\tau, \sigma,\tilde{\xi})\cdot f(\sigma, \tilde{\xi})\frac{\lambda(\sigma)}{\lambda(\tau)}\rho_1(\frac{\lambda(\sigma)}{\lambda(\tau)}\tilde{\xi})\,d\tilde{\xi} d\sigma.
		\end{align*}
		Denoting the first of these expressions, without the operator $\partial_{\tau}$, as $M_f^{(1)}(\tau)$, and the second expression, as well as the last expression in the above sum of three integrals as $\tilde{M}_f^{(2)}(\tau), \tilde{\tilde{M}}_f^{(2)}(\tau)$, respectively, we infer the bound 
		\begin{align*}
			\big\|M_f^{(1)}\big\|_{\tau^{-N}L^2_{d\tau}} + \big\|\tilde{M}_f^{(2)}\big\|_{\tau^{-N-1}L^2_{d\tau}} + \big\|\tilde{\tilde{M}}_f^{(2)}\big\|_{\tau^{-N-1}L^2_{d\tau}}\lesssim \big\|\langle\xi\partial_{\xi}\rangle^{1+\delta_0}f\big\|_{\tau^{-N}L^2_{d\tau}L^2_{\rho(\xi)\,d\xi}}
		\end{align*}
		by simple adaptation of the proofs of Lemma~\ref{lem:K_frefined}, ~\ref{lem:tildeKfcontrol}.
		
	\end{proof}
	
	In order to control the non-resonant part $z_{nres}$ of $z$, we shall have to bound the $\|\cdot\|_{S}$-norm of Schr\"odinger propagator terms. The following lemma helps bound the perturbative situations:
	\begin{lem}\label{lem:nonresbasicsmallfreq1} Define the term 
		\begin{align*}
			&z_{nres,<\epsilon_1}(\tau, R): =\\& (-i)\int_{\tau}^\infty \int_0^\infty \chi_{\xi\lesssim \epsilon_1}\big[\phi(R;\xi) - \phi(R;0)\big]\cdot S(\tau, \sigma,\xi)\cdot f(\sigma, \frac{\lambda(\tau)}{\lambda(\sigma)}\xi\big)\rho(\xi)\,d\xi d\sigma
		\end{align*}
		where the cutoff is smooth. Then we have the bound 
		\begin{equation}\label{eq:nonresbasicsmallfreq1}\begin{split}
				\Big\|z_{nres,<\epsilon_1}\Big\|_{S}\ll_{\epsilon_1}&\big\|\langle\xi\partial_{\xi}\rangle^{1+\delta_0} f\big\|_{\tau^{-N}L^2_{d\tau}(L^2_{\rho(\xi)\,d\xi}\cap L^\infty_{\rho(\xi)\,d\xi})}\\& + \big\|\langle\xi\partial_{\xi}\rangle^{1+\delta_0} \partial_{\tau}f\big\|_{\tau^{-N-\frac12-\frac{1}{4\nu}+}L^2_{d\tau}(L^2_{\rho(\xi)\,d\xi}\cap L^\infty_{\rho(\xi)\,d\xi})}\\
				& + \big\|\tau^{-1}\langle\xi\partial_{\xi}\rangle^{2+\delta_0} f\big\|_{\tau^{-N-\frac12-\frac{1}{4\nu}+}L^2_{d\tau}(L^2_{\rho(\xi)\,d\xi}\cap L^\infty_{\rho(\xi)\,d\xi})}\\
		\end{split}\end{equation}
		We also have the more crude estimate 
		\begin{align*}
			\Big\|z_{nres,<\epsilon_1}\Big\|_{S}\ll_{\tau_*}\big\|f\big\|_{\tau^{-N-2}L^2_{d\tau}L^2_{\rho(\xi)\,d\xi}}
		\end{align*}
	\end{lem}
	\begin{proof} The last part of the lemma is a straightforward consequence of the bound 
		\begin{align*}
			\Big\| \int_0^\infty \chi_{\xi\lesssim \epsilon_1}\big[\phi(R;\xi) - \phi(R;0)\big]\cdot f(\tau,\xi)\cdot \rho(\xi)\,d\xi\Big\|_{S}\lesssim \big\|f\big\|_{\tau^{-N-1}L^2_{d\tau} L^2_{\rho(\xi)\,d\xi}}. 
		\end{align*}
		We next turn to the more delicate second bound. We need to control the various parts of $\|\cdot\|_{S}$. 
		\\
		
		{\it{Control over $\big\|\langle R\rangle^{-\delta_0}z_{nres,<\epsilon_1}(\tau, R)\big\|_{\tau^{-N}L^2_{d\tau} L^\infty_{dR}}$.}} The main point is the basic estimate 
		\begin{align*}
			\Big|\phi(R;\xi) - \phi(R;0)\Big|\lesssim \langle\log \langle R\rangle\rangle\cdot \xi^2,\,\xi\lesssim 1. 
		\end{align*}
		Furthermore, we observe that (recall the definition of $S(\tau,\sigma,\xi)$ in Prop.~\ref{prop:linpropagator} )
		\begin{align*}
			\big|\lambda^2(\tau)\cdot\int_{\sigma}^{\tau}\lambda^{-2}(s)\,ds\big|\lesssim \frac{\tau}{\sigma}\cdot (\sigma - \tau)= :\zeta(\tau,\sigma).,\,\tau_*\leq \tau\leq \sigma  
		\end{align*}
		Then we decompose the integral in the lemma into a number of contributions: 
		\\
		
		{\it{(1): the region $\xi^2\lesssim \min\{\epsilon_1^2, \zeta^{-1}(\tau,\sigma)\}$.}} Observe that here we have 
		\begin{align*}
			\Big\|\langle R\rangle^{-\delta_0}\big[\phi(R;\xi) - \phi(R;0)\big]\cdot S(\tau, \sigma,\xi)\rho(\xi)\Big\|_{L^1_{d\xi}}\lesssim \zeta^{-1}\cdot\min\{\log^{-2}\epsilon_1, \log^{-2}\zeta\}. 
		\end{align*}
		Furthermore we have the bound 
		\begin{align*}
			\Big\|\chi_{\sigma\geq \tau}\frac{\tau^N}{\sigma^N}\cdot\zeta^{-1}\cdot\min\{\log^{-2}\epsilon_1, \log^{-2}\zeta\}\Big\|_{L_\tau^\infty L_{\sigma}^1\cap L_{\sigma}^\infty L_{\tau}^1}\lesssim \log^{-1}\epsilon_1. 
		\end{align*}
		We conclude from Holder's inequality and Schur's criterion that the contribution from this region to the integral is bounded by 
		\begin{align*}
			\big\|\langle R\rangle^{-\delta_0}\cdot \big\|_{\tau^{-N}L^2_{d\tau} L^\infty_{dR}}\lesssim \log^{-\frac12}\epsilon_1\cdot \Big\|f\Big\|_{\tau^{-N}L^2_{d\tau}L^\infty_{d\xi}}. 
		\end{align*}
		{\it{(2): the region $\epsilon_1^2\gtrsim\xi^2\gg\zeta^{-1}(\tau,\sigma)$.}} We intend to take advantage of integration by parts here, but have to carefully take into account the two oscillatory phases involving $\xi$. This requires us to distinguish between further situations:
		\\
		{\it{(2.a): resonant case $R\sim \xi\cdot\zeta(\tau,\sigma)$.}} We observe first that in this case $R\xi\sim \xi^2\zeta(\tau,\sigma)\gg 1$, whence $\phi(R;\xi)$ is in the oscillatory regime, and $R^{-1}\lesssim \xi$. It follows that 
		\begin{align*}
			R^{-\delta_0}\cdot\zeta&\lesssim \xi^{\frac{2\delta_0}{3}}\cdot (\xi\cdot\zeta)^{-\frac{\delta_0}{3}}\cdot \xi^{-1}\cdot (\xi\zeta)\\
			&\lesssim  \xi^{\frac{\delta_0}{3}}\cdot \zeta^{-\frac{\delta_0}{3}}\cdot R\xi^{-1},
		\end{align*}
		and so in this regime we have 
		\begin{align*}
			\Big|R^{-\delta_0}\cdot \big[\phi(R;\xi) - \phi(R;0)\big]\Big|\lesssim  \xi^{\frac{\delta_0}{3}}\cdot\zeta^{-1-\frac{\delta_0}{3}+}.
		\end{align*}
		on account of the bound $\big|\phi(R;\xi) - \phi(R;0)\big|\lesssim |\log \xi|\cdot R^{-\frac32}\xi^{\frac12}+R^{-2}$. We further have the bound
		\begin{align*}
			\Big\|\chi_{\sigma\geq \tau}\frac{\tau^N}{\sigma^N}\langle\zeta\rangle^{-1-\frac{\delta_0}{3}+}\Big\|_{L_{\tau}^\infty L_{\sigma}^1\cap L_{\sigma}^\infty L_{\tau}^1}\lesssim_{\delta_0} 1, 
		\end{align*}
		and so we infer from another application of Holder's inequality with respect to the $\xi$-integral and Schur's criterion, keeping in mind the restriction $\xi\lesssim \epsilon_1$, the following bound for this contribution:  
		\begin{align*}
			\big\|\langle R\rangle^{-\delta_0}\cdot \big\|_{\tau^{-N}L^2_{d\tau} L^\infty_{dR}}\lesssim \epsilon^{\frac{\delta_0}{3}}\cdot\big\|f\big\|_{\tau^{-N}L^2_{d\tau}L^\infty_{\rho(\xi)\,d\xi}}.
		\end{align*}
		{\it{(2.b): non-resonant case $R\ll \xi\cdot\zeta(\tau,\sigma)$.}} Here we further distinguish between the oscillatory region $R\xi\gtrsim 1$ and the non-oscillatory complement. In either case we shall perform integration by parts with respect to $\xi$, but in the former case we treat the contributions of $\phi(R;\xi)$ and $\phi(R;0)$ separately. In fact, assuming inclusion of a cutoff $\chi_{R\xi\gtrsim 1}$ and performing integration by parts with respect to $\xi$, using 
		\begin{align*}
			S(\tau,\sigma,\xi) = (2i\xi\cdot \lambda^2(\tau)\cdot\int_{\sigma}^{\tau}\lambda^{-2}(s)\,ds)^{-1}\cdot\partial_{\xi}S(\tau,\sigma,\xi), 
		\end{align*}
		the contribution of $\phi(R;0)$ can be reformulated as 
		\begin{align*}
			\int_{\tau}^\infty \int_0^\infty\phi(R;0)\cdot \frac{S(\tau, \sigma,\xi)}{\zeta_1(\tau,\sigma)}\cdot  \partial_{\xi}\Big(\frac{\chi_{R^{-1}\lesssim\xi\lesssim \epsilon_1}}{2\xi}f(\sigma, \frac{\lambda(\tau)}{\lambda(\sigma)}\xi\big)\rho(\xi)\Big)\,d\xi d\sigma,
		\end{align*}
		where we recall 
		\[
		-\zeta_1(\tau,\sigma) : = \lambda^2(\tau)\cdot\int_{\sigma}^{\tau}\lambda^{-2}(s)\,ds\sim \zeta(\tau,\sigma). 
		\]
		Additional integration by parts and use of the fact that $\chi_{R\xi\gtrsim 1}\phi(R;0)\lesssim \xi^2$ leads to the estimate 
		\begin{align*}
			&\Big|\langle R\rangle^{-\delta_0}\int_0^\infty\phi(R;0)\cdot \frac{S(\tau, \sigma,\xi)}{\zeta_1(\tau,\sigma)}\cdot  \partial_{\xi}\Big(\frac{\chi_{R^{-1}\lesssim\xi\lesssim \epsilon_1}}{2\xi}f(\sigma, \frac{\lambda(\tau)}{\lambda(\sigma)}\xi\big)\rho(\xi)\Big)\,d\xi\Big|\\
			&\lesssim \epsilon_0^{\frac{\delta_0}{2}}\zeta^{-1-\frac{\delta_0}{2}}(\tau, \sigma)\cdot \big\|\langle\xi\partial_{\xi}\rangle^{1+\delta_0}f(\sigma,\cdot)\big\|_{L^2_{\rho(\xi)\,d\xi}}. 
		\end{align*}
		The desired bound for this contribution follows again from Schur's criterion, as in the preceding situations. 
		\\
		It remains to consider the contribution of $\phi(R;\xi)$, which is of oscillatory character in the region $R\xi\gtrsim 1$. Hence we need to combine its phase with $S(\tau,\sigma,\xi)$ to perform integration by parts, and the procedure is otherwise completely analogous to the contribution of $\phi(R;0)$, taking advantage of the bound 
		\begin{align*}
			\big|\chi_{R\xi\gtrsim 1}\phi(R;\xi)\big|\lesssim \log\xi\cdot \xi^2. 
		\end{align*}
		We still need to deal with the non-oscillatory region $R\xi\lesssim 1$, but there we again proceed as for the contribution of $\phi(R;0)$ in the oscillatory regime, now taking advantage of the bound 
		\begin{align*}
			\Big|\chi_{R\xi\lesssim1}[\phi(R;\xi) - \phi(R;0)]\Big|\lesssim \langle\log \langle R\rangle\rangle\cdot \xi^2
		\end{align*}
		{\it{(2.c): non-resonant case $R\gg\xi\cdot\zeta(\tau,\sigma)$.}} Recalling that $\xi^2\zeta\gg 1$ in situation (2), we have $R\xi\gg 1$ and we are automatically in the oscillatory regime. Then proceed as in the preceding case {\it{(2.a)}}, distinguishing between the contributions of $\phi(R;\xi)$ and $\phi(R;0)$. 
		\\
		
		{\it{Control over $\big\|\langle R\rangle^{1-\delta_0}\nabla_R z_{nres,<\epsilon_1}(\tau, R)\big\|_{\tau^{-N}L^2_{d\tau} L^\infty_{dR}}$.}} Observe that $(R\partial_R)[\phi(R;\xi) - \phi(R;0)]$ obeys the same asymptotics as $\phi(R;\xi) - \phi(R;0)$ in the region $R\xi\lesssim 1$. On the other hand, we schematically have\footnote{Recall subsection~\ref{subsec:basicfourier}
		} 
		\begin{align*}
			(R\partial_R)\phi(R;\xi)\sim \sum_{\pm}\xi^{2}\cdot (R\xi)^{-\frac12}\cdot e^{\pm iR\xi}
		\end{align*}
		in the oscillatory region $R\xi\gtrsim 1$, and we can replicate the preceding argument to infer the desired bound. 
		\\
		
		{\it{Control over $\big\|\mathcal{L} z_{nres,<\epsilon_1}(\tau, R)\big\|_{U}$.}} Distinguish between the non-oscillatory case $R\xi\lesssim 1$ and the oscillatory regime $R\xi\gtrsim 1$. In the former we use that 
		\begin{align*}
			\big|\mathcal{L}\phi(R;\xi)\big|\lesssim \langle \log \langle R\rangle\rangle\cdot \xi^2\langle R\rangle^{-2},
		\end{align*}
		and replicate the preceding estimates. In the latter case we use the extra factor $\xi^2$ (from applying $\mathcal{L}$) to perform integration by parts with respect to $\sigma$, using 
		\[
		\xi^2\cdot e^{i\lambda^2(\tau)\xi^2\int_{\sigma}^{\tau}\lambda^{-2}(s)\,ds} =(-i)\frac{\lambda^2(\sigma)}{\lambda^2(\tau)}\cdot \partial_{\sigma}\big(e^{i\lambda^2(\tau)\xi^2\int_{\sigma}^{\tau}\lambda^{-2}(s)\,ds}\big). 
		\]
		This generates a boundary term 
		\[
		b(\tau, R): = \int_0^\infty \chi_{\xi\lesssim \epsilon_1}\chi_{R\xi\gtrsim 1}\phi(R;\xi)\cdot f(\tau,\xi)\rho(\xi)\,d\xi,
		\]
		for which using integration by parts and the asymptotics from subsection~\ref{subsec:basicfourier} we infer the bound 
		\begin{align*}
			\big|\langle R\rangle^2 b(\tau, R)\big|\lesssim |\log\epsilon_1|^{-\frac12}\cdot \big\|\langle \xi\partial_{\xi}\rangle f(\tau,\cdot)\big\|_{L^2_{\rho(\xi)\,d\xi}}
		\end{align*}
		The remaining terms arising when $\partial_{\sigma}$ hits $f(\sigma, \frac{\lambda(\tau)}{\lambda(\sigma)}\xi)$ are handled again as in the first part of the proof, leading to functions in 
		\[
		\tau^{-N-\frac12-\frac{1}{4\nu}}L^2_{d\tau}[L^\infty_{R^3\,dR}\cap \mathcal{L}^{-\frac34}(\langle R\rangle^{-\frac32}L^\infty_{R^3\,dR})]
		\]
	\end{proof}
	
	\begin{rem}\label{rem:nonresbasicsmallfreq1} If we first perform integration by parts with respect to $\sigma$ and then repeat the preceding proof, we can replace the right hand side by the following slightly modified norm: 
		\begin{align*}
			&\big\|f\big\|_{\tau^{-N}L^2_{d\tau}(L^2_{\rho(\xi)\,d\xi}\cap L^\infty_{\rho(\xi)\,d\xi})}+\big\|\xi^{-2}\langle\xi\partial_{\xi}\rangle^{1+\delta_0}\partial_{\tau}^2f\big\|_{\tau^{-N}L^2_{d\tau}(L^2_{\rho(\xi)\,d\xi}\cap L^\infty_{\rho(\xi)\,d\xi})}\\&  + \big\|\tau^{-1}\xi^{-2}\langle\xi\partial_{\xi}\rangle^{2+\delta_0}\partial_{\tau}f\big\|_{\tau^{-N-\frac12-\frac{1}{4\nu}+}L^2_{d\tau}(L^2_{\rho(\xi)\,d\xi}\cap L^\infty_{\rho(\xi)\,d\xi})}
		\end{align*}
		Finally, we also remark that we can avoid all temporal derivatives on the right hand side of \eqref{eq:nonresbasicsmallfreq1}; in effect, these only occur due to the improved bound for $\mathcal{L}z_{nres,<\epsilon_1}$, and using repeated integration by parts with respect to $\xi$ as well as Plancherel's theorem for the distorted Fourier transform we can use $\big\|\langle\xi\partial_{\xi}\rangle^2 f\big\|_{L^2_{\rho(\xi)\,d\xi}}$ instead of the last two terms in \eqref{eq:nonresbasicsmallfreq1}. 
	\end{rem}
	
	We shall apply the preceding lemma in particular in the context of the source term (recall \eqref{eq:yzdfn})
	\[
	f(\tau,\xi) = \mathcal{F}\big(y_z\cdot W)(\tau,\xi),
	\]
	whence the following lemma shall be useful:
	\begin{lem}\label{lem:nonresbasicsmallfreq1yzW} Let $f =  \mathcal{F}\big(y_z\cdot W)(\tau,\xi)$. Then we have the bounds
		\begin{align*}
			&\big\|\langle\xi\partial_{\xi}\rangle^{1+\delta_0} \partial_{\tau}f\big\|_{\tau^{-N-\frac12-\frac{1}{4\nu}+}L^2_{d\tau}(L^2_{\rho(\xi)\,d\xi}\cap L^\infty_{\rho(\xi)\,d\xi})}\ll_{\tau_*} \big\|z\big\|_{S},\\
			&\big\|\tau^{-1}\langle\xi\partial_{\xi}\rangle^{2+\delta_0} f\big\|_{\tau^{-N-\frac12-\frac{1}{4\nu}+}L^2_{d\tau}(L^2_{\rho(\xi)\,d\xi}\cap L^\infty_{\rho(\xi)\,d\xi})} \ll_{\tau_*} \big\|z\big\|_{S}.
		\end{align*}
	\end{lem}
	\begin{proof} The first estimate follows from Corollary~\ref{cor:yzWpartialtau}. For the second estimate, denoting by $R$ the spatial variable in $y_z\cdot W$, restricting to $R\xi\lesssim 1$ (i. e. non-oscillatory regime), the operator $\xi\partial_{\xi}$ has no effect, and the desired estimate follows from 
		Corollary~\ref{cor:yzW}. Restricting to the oscillatory regime $R\xi\gtrsim 1$, we can replace $\xi\partial_{\xi}$ by $R\partial_R$ up to error terms which fall under the purview of Corollary~\ref{cor:yzW}. Integrating by parts with respect to $R$ in the inner product defining $f$, when $R\partial_R$ falls on $W$ we can again conclude via Corollary~\ref{cor:yzW}. If $R\partial_R$ falls on $y_z$, we use the Fourier representation \eqref{eq:wavepropagator}, \eqref{eq:nflatfourierrepresent} and further integration by parts with respect to the frequency variable, which leads at most to a loss of $\lesssim \tilde{\tau}\sim \tau^{\frac12-\frac{1}{4\nu}}$. Having reduced the number of operators $\xi\partial_{\xi}$ by one, one concludes again via Corollary~\ref{cor:yzW}. 
	\end{proof}
	
	The high-frequency analogue of Lemma~\ref{lem:nonresbasicsmallfreq1} is the following: 
	\begin{lem}\label{lem:nonresbasiclargefreq1} Define the term 
		\begin{align*}
			&z_{nres,>\epsilon_1^{-1}}(\tau, R): =\\& (-i)\int_{\tau}^\infty \int_0^\infty \chi_{\xi\gtrsim \epsilon_1^{-1}}\big[\phi(R;\xi) - \phi(R;0)\big]\cdot S(\tau, \sigma,\xi)\cdot f(\sigma, \frac{\lambda(\tau)}{\lambda(\sigma)}\xi\big)\rho(\xi)\,d\xi d\sigma
		\end{align*}
		where the cutoff is smooth. Then we have the bound 
		\begin{equation}\label{eq:nonresbasicsmallfreq1}\begin{split}
				\Big\|z_{nres,>\epsilon_1^{-1}}\Big\|_{S}&\ll_{\epsilon_1,\tau_*}\big\|\mathcal{F}\big(\mathcal{L}^{1+}f\big)\big\|_{\tau^{-N}L^2_{d\tau}L^2_{\rho(\xi)\,d\xi}} + \big\|\mathcal{F}\big(\mathcal{L}^{1+}\partial_{\tau}f\big)\big\|_{\tau^{-N-}L^2_{d\tau}L^2_{\rho(\xi)\,d\xi}}\\
				& + \big\|\langle \xi\partial_{\xi}\rangle\mathcal{F}\big(\mathcal{L}^{1+}f\big)\big\|_{\tau^{-N}L^2_{d\tau}L^2_{\rho(\xi)\,d\xi}}
		\end{split}\end{equation}
		We also have the simpler bound 
		\begin{align*}
			\Big\|z_{nres,>\epsilon_1^{-1}}\Big\|_{S}&\ll_{\epsilon_1,\tau_*}\big\|\mathcal{F}\big(\mathcal{L}^2f\big)\big\|_{\tau^{-N-1-}L^2_{d\tau}L^2_{\rho(\xi)\,d\xi}}. 
		\end{align*}
	\end{lem}
	\begin{proof} We give details for the control of the third norm on the right in \eqref{eq:Snormdefi}. By means of Sobolev's embedding, it suffices to control 
		\begin{align*}
			\Big\|\int_{\tau}^\infty \int_0^\infty \xi^4\chi_{\xi\gtrsim \epsilon_1^{-1}}\phi(R;\xi)\cdot S(\tau, \sigma,\xi)\cdot f(\sigma, \frac{\lambda(\tau)}{\lambda(\sigma)}\xi\big)\rho(\xi)\,d\xi d\sigma\Big\|_{\tau^{-N}L^2_{d\tau}L^2_{R^3\,dR}}. 
		\end{align*}
		Noting from Proposition~\ref{prop:linpropagator} (with $\alpha_0 = 0$) that 
		\[
		(-i)\xi^2\cdot S(\tau, \sigma,\xi) = \partial_{\sigma}\big[\frac{\lambda^2(\sigma)}{\lambda^2(\tau)}S(\tau, \sigma,\xi)\big], 
		\]
		and performing integration by parts with respect to $\sigma$, and using the triangle inequality for the outer norm we arrive at the boundary contribution
		\begin{align*}
			B_1: = \Big\|\int_0^\infty \xi^2\chi_{\xi\gtrsim \epsilon_1^{-1}}\phi(R;\xi)\cdot  f(\tau,\xi\big)\rho(\xi)\,d\xi\Big\|_{\tau^{-N}L^2_{d\tau}L^2_{R^3\,dR}}
		\end{align*}
		as well as the term $B_2$ given by
		\begin{align*}
			\Big\|\int_{\tau}^\infty \int_0^\infty \xi^2\chi_{\xi\gtrsim \epsilon_1^{-1}}\phi(R;\xi)\cdot S(\tau, \sigma,\xi)\cdot \frac{\lambda^2(\sigma)}{\lambda^2(\tau)}\partial_{\sigma}\big(f(\sigma, \frac{\lambda(\tau)}{\lambda(\sigma)}\xi\big)\big)\rho(\xi)\,d\xi d\sigma\Big\|_{\tau^{-N}L^2_{d\tau}L^2_{R^3\,dR}}.
		\end{align*}
		Using the Plancherel's theorem for the distorted Fourier transform, we can estimate $B_1$ by means of 
		\begin{align*}
			\big\|B_1\big\|_{\tau^{-N}L^2_{d\tau}L^2_{R^3\,dR}}\ll_{\epsilon_1}\big\|\mathcal{F}\big(\mathcal{L}^{1+}f\big)\big\|_{\tau^{-N}L^2_{d\tau}L^2_{\rho(\xi)\,d\xi}}.  
		\end{align*}
		The term $B_2$ is bounded similarly upon observing that 
		\begin{align*}
			\partial_{\sigma}\big(f(\sigma, \frac{\lambda(\tau)}{\lambda(\sigma)}\xi\big)\big) = \partial_{\sigma}f\big(\sigma, \frac{\lambda(\tau)}{\lambda(\sigma)}\xi\big)  - \frac{\lambda_{\sigma}}{\lambda}\cdot \big((\xi\partial_{\xi})f\big)\big(\sigma, \frac{\lambda(\tau)}{\lambda(\sigma)}\xi\big).
		\end{align*}
		The second estimate of the lemma is proved by observing that 
		\begin{align*}
			\big\|\xi^4\cdot f(\sigma, \frac{\lambda(\tau)}{\lambda(\sigma)}\xi\big)\big\|_{\sigma^{-N}L^2_{d\tau}L^2_{\rho(\xi)\,d\xi}}\lesssim \big(\frac{\sigma}{\tau}\big)^C\cdot \big\|\xi^4 f(\sigma, \cdot)\big\|_{\sigma^{-N}L^2_{d\tau}L^2_{\rho(\xi)\,d\xi}}
		\end{align*}
		for $C = C(\nu)\ll N$ and using Schur's criterion as usual to infer the desired estimate (where the smallness gain comes choosing $\tau_*$ sufficiently large).  
	\end{proof}
	
	We shall also require a basic lemma which recovers control over the $S$-norm in an 'elliptic situation':
	\begin{lem}\label{lem:basicSfromtildeL} Let $\tilde{\mathcal{L}}$ be as in \eqref{eq:tildeL}, and denote by $\tilde{\phi}_0, \tilde{\theta}_0$ a fundamental system for $\tilde{\mathcal{L}}u = 0$ with $\tilde{\phi}_0(0) = 1$, satisfying the normalization condition
		\[
		W(\tilde{\phi}_0, \tilde{\theta}_0)(R) = \partial_R\tilde{\phi}_0(R)\cdot \tilde{\theta}_0(R) -  \tilde{\phi}_0(R)\cdot \partial_R\tilde{\theta}_0(R) = R^{-3}. 
		\]
		Then defining 
		\[
		\tilde{\mathcal{L}}^{-1}f = \tilde{\phi}_0(R)\cdot \int_0^R \tilde{\theta}_0(s)\cdot f(s) s^3\,ds - \tilde{\theta}_0(R)\cdot \int_0^R \tilde{\phi}_0(s)\cdot f(s) s^3\,ds,
		\]
		we have the estimate 
		\begin{align*}
			\big\|\tilde{\mathcal{L}}^{-1}f \big\|_{S}\lesssim \big\|f\big\|_{\tau^{-N}L^2_{d\tau}(L^{2+}_{R^3\,dR}\cap \langle R\rangle^{\frac{\delta_0}{2}}L^2_{R^3\,dR})}. 
		\end{align*}
		An analogous statement applies to the operator $\mathcal{L}$. The space $L^{2+}_{R^3\,dR}\cap \langle R\rangle^{\frac{\delta_0}{2}}L^2_{R^3\,dR})$ on the right may be replaced by $L^2_{R^3\,dR}$. 
	\end{lem}
	\begin{proof} In the following we omit the temporal part $\tau^{-N}L^2_{d\tau}$ of the norms for simplicity. We note that $\tilde{\phi}_0 = \Lambda W = \partial_{\lambda}\big(\lambda W(\lambda R)\big)|_{\lambda = 1}$. Further we have the asymptotic relations 
		\[
		\tilde{\theta}(R)\sim R^{-2},\,R\ll 1,\, \tilde{\theta}(R)\sim 1,\,R\gg 1.
		\]
		Using the Cauchy-Schwarz and Hoelder's inequality, we infer 
		\begin{align*}
			\Big|\tilde{\phi}_0(R)\cdot \int_0^R \tilde{\theta}_0(s)\cdot f(s) s^3\,ds\Big|\lesssim \big\|f\big\|_{L^{2+}_{R^3\,dR}\cap L^2_{R^3\,dR}},
		\end{align*}
		since $\big|\tilde{\phi}_0(R)\big|\cdot \big\|\chi_{s<R}\tilde{\theta}_0\big\|_{L^{2-}_{s^3\,ds} + L^2_{s^3\,ds}}\lesssim \big|\tilde{\phi}_0(R)\big|\cdot \langle R\rangle^2\lesssim 1$. Similarly we have 
		\begin{align*}
			\Big|\tilde{\theta}_0(R)\cdot \int_0^R \tilde{\phi}_0(s)\cdot f(s) s^3\,ds\Big|\lesssim \langle \log \langle R\rangle\rangle^{\frac12}\cdot \big\|f\big\|_{L^2_{R^3\,dR}}
		\end{align*}
		since $\big|\tilde{\theta}_0(R)\cdot\big\|\chi_{s<R}\tilde{\phi}_0\big\|_{L^2_{s^3\,ds}}\lesssim  \langle \log \langle R\rangle\rangle^{\frac12}$. Since we integrate over the region $s\leq R$, we have $\langle R\rangle^{-\delta_0}\leq \langle s\rangle^{-\delta_0}$. We conclude that 
		\[
		\big\|\langle R\rangle^{-\delta_0}\tilde{\mathcal{L}}^{-1}f\big\|_{L^\infty_{dR}}\lesssim \big\|\langle R\rangle^{-\frac{\delta_0}{2}}f\big\|_{L^{2+}_{R^3\,dR}\cap L^2_{R^3\,dR}}. 
		\]
		For the first derivative of $\tilde{\mathcal{L}}^{-1}f$, we use that 
		\begin{align*}
			\partial_R(\tilde{\mathcal{L}}^{-1}f) = \partial_R\tilde{\phi}_0(R)\cdot \int_0^R \tilde{\theta}_0(s)\cdot f(s) s^3\,ds - \partial_R\tilde{\theta}_0(R)\cdot \int_0^R \tilde{\phi}_0(s)\cdot f(s) s^3\,ds
		\end{align*}
		and the symbolic behavior of $\tilde{\phi}_0, \tilde{\theta}_0$ for $R\gg 1$ to conclude the bound as before: 
		\begin{align*}
			\big|\langle R\rangle^{1-\delta_0}\partial_R(\tilde{\mathcal{L}}^{-1}f)\big|\lesssim \big\|\langle R\rangle^{-\frac{\delta_0}{2}}f\big\|_{\tau^{-N}L^2_{d\tau}(L^{2+}_{R^3\,dR}\cap L^2_{R^3\,dR})}. 
		\end{align*}
		Finally, we also have $\big\|\mathcal{L}(\tilde{\mathcal{L}}^{-1}f)\|_{L^{2+}_{R^3\,dR}}\lesssim \big\|f\big\|_{\tau^{-N}L^2_{d\tau}(L^{2+}_{R^3\,dR}\cap L^2_{R^3\,dR})}$, which concludes the required bounds. The proof for $\mathcal{L}$ is analogous. 
	\end{proof}

	\subsection{Miscellaneous lemmas}
	
	\begin{lem}\label{lem:FourierNV1} Denoting\footnote{Recall subsection~\ref{subsec:subsec:standardFourieronR4}.} $\mathcal{F}_{\R^4}(f)(
		\xi) = \langle f, \phi_{\R^4}(R;\xi)\rangle_{L^2_{R^3\,dR}}$ for radial $f$ on $\R^4$ and $\xi>0$, we have 
		\begin{align*}
			\mathcal{F}_{\R^4}\big(W^2\big)(\hat{\tau})\neq 0\,\forall \hat{\tau}\neq 0.  
		\end{align*}
	\end{lem}
	\begin{proof} By simple re-scaling we may assume that $W(R) = \frac{1}{1+R^2}$. Then consider 
		\begin{align*}
			\int_{\R^4}e^{-a|x|^2 -a}\cdot e^{-ix\cdot\xi}\,dx = ce^{-a}\cdot a^{-2}\cdot e^{-\frac{|\xi|^2}{a}}. 
		\end{align*}
		We have 
		\begin{align*}
			\int_{\R^4}\frac{e^{-a|x|^2 -a}}{(1+|x|^2)^2}\cdot e^{-ix\cdot\xi}\,dx &= c\int_a^\infty \int_{a_1}^\infty e^{-a_2}\cdot a_2^{-2}\cdot e^{-\frac{|\xi|^2}{a_2}}\,da_2 da_1,\\
		\end{align*}
		whence 
		\begin{align*}
			\mathcal{F}_{\R^4}\big(W^2\big)(\hat{\tau}) =  c\int_0^\infty \int_{a_1}^\infty e^{-a_2}\cdot a_2^{-2}\cdot e^{-\frac{\hat{\tau}^2}{a_2}}\,da_2 da_1,\,c\neq 0,
		\end{align*}
		whence a non-vanishing function. 
	\end{proof}
	
	\begin{lem}\label{lem:FourierNV2} There is one $\hat{\tau}_*\in \R_+$ such that we have 
		\begin{align*}
			\mathcal{F}_{\R^4}\big(\Lambda W\cdot W\big)(\hat{\tau}_*) = 0,\,\mathcal{F}_{\R^4}\big(\Lambda W\cdot W\big)(\hat{\tau})\neq 0,\,\hat{\tau}\in \R_+\backslash\{\hat{\tau}_*\}.  
		\end{align*}
	\end{lem}
	\begin{proof} Recall that 
		\[
		2\Lambda W\cdot W = \partial_{\lambda}\big(\lambda^2\cdot W^2(\lambda R)\big)|_{\lambda = 1}. 
		\]
		Hence we have 
		\begin{align*}
			\frac{2}{c}\mathcal{F}_{\R^4}\big(\Lambda W\cdot W\big)(\hat{\tau}) = \partial_{\lambda}\Big(\lambda^{-2}\int_0^\infty \int_{a_1}^\infty e^{-a_2}\cdot a_2^{-2}\cdot e^{-\frac{\hat{\tau}^2}{\lambda^2 a_2}}\,da_2 da_1\Big)|_{\lambda = 1}
		\end{align*}
		It follows that we can write 
		\begin{align*}
			\frac{1}{c}\mathcal{F}_{\R^4}\big(\Lambda W\cdot W\big)(\hat{\tau}) &=-\int_0^\infty \int_{a_1}^\infty e^{-a_2}\cdot a_2^{-2}\cdot e^{-\frac{\hat{\tau}^2}{a_2}}\,da_2 da_1\\
			& + \int_0^\infty \int_{a_1}^\infty e^{-a_2}\cdot a_2^{-2}\cdot \frac{\hat{\tau}^2}{a_2}\cdot e^{-\frac{\hat{\tau}^2}{a_2}}\,da_2 da_1
		\end{align*}
		The last integral on the right can also be written as
		\begin{align*}
			&\int_0^\infty \int_{a_1}^\infty e^{-a_2}\cdot a_2^{-2}\cdot \frac{\hat{\tau}^2}{a_2}\cdot e^{-\frac{\hat{\tau}^2}{a_2}}\,da_2 da_1 = \int_0^\infty \int_{a_1}^\infty e^{-a_2}\cdot a_2^{-1}\cdot \partial_{a_2}\big(e^{-\frac{\hat{\tau}^2}{a_2}}\big)\,da_2 da_1\\
			& = - \int_0^\infty e^{-a_1}\cdot a_1^{-1}\cdot e^{-\frac{\hat{\tau}^2}{a_1}}\,da_1 + \int_0^\infty \int_{a_1}^\infty e^{-a_2}\cdot a_2^{-2}\cdot e^{-\frac{\hat{\tau}^2}{a_2}}\,da_2 da_1\\
			& + \int_0^\infty \int_{a_1}^\infty e^{-a_2}\cdot a_2^{-1}\cdot e^{-\frac{\hat{\tau}^2}{a_2}}\,da_2 da_1\\
		\end{align*}
		and we can apply another integration by parts to write the last integral as 
		\begin{align*}
			\int_0^\infty \int_{a_1}^\infty e^{-a_2}\cdot a_2^{-1}\cdot e^{-\frac{\hat{\tau}^2}{a_2}}\,da_2 da_1 = \int_0^\infty  e^{-a_1}\cdot e^{-\frac{\hat{\tau}^2}{a_1}}\, da_1.
		\end{align*}
		We conclude that in fact we have the relation 
		\begin{align*}
			\frac{1}{c}\mathcal{F}_{\R^4}\big(\Lambda W\cdot W\big)(\hat{\tau}) = \int_0^\infty  e^{-a_1}\cdot e^{-\frac{\hat{\tau}^2}{a_1}}\, da_1 -  \int_0^\infty e^{-a_1}\cdot a_1^{-1}\cdot e^{-\frac{\hat{\tau}^2}{a_1}}\,da_1.
		\end{align*}
		The second integral on the right is larger than the first for $0<\hat{\tau}\ll 1$ but less than the first one for $\hat{\tau}\gg 1$. It follows that there is at least one $\hat{\tau}_*>0$ for which 
		\[
		\mathcal{F}_{\R^4}\big(\Lambda W\cdot W\big)(\hat{\tau}_*) = 0.
		\]
		To see that there is exactly one such $\hat{\tau}_*>0$, we write the preceding relation as 
		\[
		\frac{1}{c}\mathcal{F}_{\R^4}\big(\Lambda W\cdot W\big)(\hat{\tau}) = \phi(\hat{\tau}^2) + \phi'(\hat{\tau}^2),
		\]
		where we set $\phi(\eta) = \int_0^\infty  e^{-a_1}\cdot e^{-\frac{\eta}{a_1}}\, da_1$, $\eta>0$, whence a positive function which is in $C^\infty(\R_+)$. Then note that (for arguments on $\R_+$)
		\begin{align*}
			\big(\frac{\phi'}{\phi}\big)' = \frac{\phi''\cdot\phi - (\phi')^2}{\phi^2}
		\end{align*}
		and we have 
		\begin{align*}
			&(\phi'(\eta))^2 = \big(\int_0^\infty e^{-a_1}\cdot a_1^{-1}\cdot e^{-\frac{\eta}{a_1}}\,da_1\big)^2\\&<\big(\int_0^\infty e^{-a_1}\cdot a_1^{-2}\cdot e^{-\frac{\eta}{a_1}}\,da_1\big)\cdot \big(\int_0^\infty e^{-a_1}\cdot e^{-\frac{\eta}{a_1}}\,da_1\big) = \phi(\eta)\cdot\phi''(\eta), 
		\end{align*}
		due to the Cauchy-Schwarz inequality, whence the function $\frac{\phi'}{\phi}$ is strictly monotonic on $\R_+$. Hence there is at most one $\hat{\tau}_*$ such that$
		\frac{\phi'(\hat{\tau}_*}{\phi(\hat{\tau}_*)} = -1$, and  the lemma follows.

	\end{proof}

	\begin{lem}\label{lem:freqlocalpropag1} The wave propagator \eqref{eq:wavepropagator} is approximately compatible with wave temporal frequency localisation in the following sense: letting $n(\tilde{\tau}, R)$ defined as in  \eqref{eq:wavepropagator} but with $F$ replaced by 
		\[
		Q_{>\gamma^{-1}}^{(\tilde{\sigma})}F, 
		\]
		we have 
		\[
		n = Q_{>\gamma^{-1+}}^{(\tilde{\tau})}n + n_e
		\]
		with the bound 
		\begin{align*}
			\Big\|\big\|\langle R\rangle^{-1-\delta_0}n_e\big\|_{L^2_{R^3\,dR}}\Big\|_{\tau^{-N}L^2_{d\tau}}\ll_{\tau_*,\gamma}\Big\|F\Big\|,
		\end{align*}
		where the generic norm on the right hand side denotes the expression on the right in the first inequality of Lemma~\ref{lem:wavebasicinhom}.  
	\end{lem}
	
	In a similar vein, we have 
	\begin{lem}\label{lem:freqlocalpropag2} Wave temporal frequency localization is compatible with the Schr\"odinger propagator in the following sense: we have for $j = 1,2$
		\begin{align*}
			&Q_{<0}^{(\tilde{\tau})}\int_{\tau}^\infty \int_0^\infty \xi^2 S_j(\tau,\sigma,\xi)\cdot \mathcal{F}(E)(\sigma,\frac{\lambda(\tau)}{\lambda(\sigma)}\xi)\rho(\xi)\,d\xi d\sigma\\
			& = Q_{<0}^{(\tilde{\tau})}\int_{\tau}^\infty \int_0^\infty \xi^2 S_j(\tau,\sigma,\xi)\cdot \mathcal{F}( Q^{(\tilde{\sigma})}_{<\sigma^{\delta}}E)(\sigma,\frac{\lambda(\tau)}{\lambda(\sigma)}\xi)\rho(\xi)\,d\xi d\sigma + F,\\
		\end{align*}
		for any $\delta>0$, where we have 
		\begin{align*}
			\Big\|F\big\|_{\sigma^{-N(1+\delta)}L^2_{d\sigma}}\lesssim_{N,\delta}\big\|E\big\|_{\sigma^{-N}L^2_{d\sigma}L^2_{R^3\,dR}}, 
		\end{align*}
	\end{lem}
	\begin{proof} We need to show that the function 
		\begin{align*}
			F = Q_{<0}^{(\tilde{\tau})}\int_{\tau}^\infty \int_0^\infty \xi^2 S_j(\tau,\sigma,\xi)\cdot \mathcal{F}( Q^{(\tilde{\sigma})}_{\geq \sigma^{\delta}}E)(\sigma,\frac{\lambda(\tau)}{\lambda(\sigma)}\xi)\rho(\xi)\,d\xi d\sigma
		\end{align*}
		enjoys better decay with respect to $\tau$. First pass to the new integration variable $\tilde{\xi}: = \frac{\lambda(\tau)}{\lambda(\sigma)}\xi$, and write 
		\begin{align*}
			\mathcal{F}( Q^{(\tilde{\sigma})}_{\geq \sigma^{\delta}}E)(\sigma,\tilde{\xi}) &= \partial_{\tilde{\sigma}}\big(\partial_{\tilde{\sigma}}^{-1} \mathcal{F}( Q^{(\tilde{\sigma})}_{\geq \sigma^{\delta}}E)(\sigma,\tilde{\xi})\big)\\
			& = \frac{\partial\sigma}{\partial\tilde{\sigma}}\cdot\partial_{\sigma}\big(\partial_{\tilde{\sigma}}^{-1} \mathcal{F}( Q^{(\tilde{\sigma})}_{\geq \sigma^{\delta}}E)(\sigma,\tilde{\xi})\big),
		\end{align*}
		where we observe the bound 
		\begin{align*}
			\big\|\big(\partial_{\tilde{\sigma}}^{-1} \mathcal{F}( Q^{(\tilde{\sigma})}_{\geq \sigma^{\delta}}E)(\sigma,\tilde{\xi})\big)\big\|_{\sigma^{-N-\delta}L^2_{d\sigma}L^2_{\rho(\tilde{\xi})}}\lesssim \big\|E\big\|_{\sigma^{-N}L^2_{d\sigma}L^2_{R^3\,dR}}. 
		\end{align*}
		We insert the preceding identity for 
		\[
		\mathcal{F}( Q^{(\tilde{\sigma})}_{\geq \sigma^{\delta}}E)(\sigma,\tilde{\xi})
		\]
		in the double integral and perform integration by parts, which in particular results in the factor (writing $S_j(\tau,\sigma,\xi) = \tilde{S}_j(\tau,\sigma, \tilde{\xi})$)
		\begin{align*}
			&\partial_{\sigma}\tilde{S}_j(\tau,\sigma, \tilde{\xi}) = \zeta(\tau,\sigma)\cdot \partial_{\tau}\tilde{S}_j(\tau,\sigma, \tilde{\xi}) + O(\sigma^{-1}),\\
			&\zeta(\tau,\sigma) = \frac{\lambda^2(\tau)}{\lambda^2(\sigma)}\cdot\big(1 + 2\lambda_{\sigma}\lambda(\sigma)\cdot\int_{\tau}^{\sigma}\lambda^{-2}(s)\,ds\big). 
		\end{align*}
		In turn we can write 
		\begin{align*}
			&Q_{<0}^{(\tilde{\tau})}\int_{\tau}^\infty \int_0^\infty \tilde{\xi}^2  \zeta(\tau,\sigma)\cdot \partial_{\tau}\tilde{S}_j(\tau,\sigma, \tilde{\xi}) \cdot \mathcal{F}( Q^{(\tilde{\sigma})}_{\geq \sigma^{\delta}}E)(\sigma,\tilde{\xi})\rho(\frac{\lambda(\sigma)}{\lambda(\tau)}\tilde{\xi})\,d\tilde{\xi} d\sigma\\
			& = \partial_{\tau}Q_{<0}^{(\tilde{\tau})}\int_{\tau}^\infty \int_0^\infty \tilde{\xi}^2  \zeta(\tau,\sigma)\cdot\tilde{S}_j(\tau,\sigma, \tilde{\xi}) \cdot \mathcal{F}( Q^{(\tilde{\sigma})}_{\geq \sigma^{\delta}}E)(\sigma,\tilde{\xi})\rho(\frac{\lambda(\sigma)}{\lambda(\tau)}\tilde{\xi})\,d\tilde{\xi} d\sigma +O(\tau^{-1})\\
		\end{align*}
		where we have 
		\[
		\partial_{\tau}Q_{<0}^{(\tilde{\tau})} = \frac{\partial\tilde{\tau}}{\partial\tau}\cdot \partial_{\tilde{\tau}}Q_{<0}^{(\tilde{\tau})}
		\]
		and the factor $\frac{\partial\tilde{\tau}}{\partial\tau}$ compensates for the factor $\frac{\partial\sigma}{\partial\tilde{\sigma}}$, taking advantage of the weights $\tau^{-N}$ in our norms. The lemma follows by reiterating this integration by parts sufficiently many times. 
		
	\end{proof}
	
	The next lemma deals with converting wave temporal frequency localization to Schr\"odinger temporal frequency localization
	\begin{lem}\label{lem:wavetoSchrodfreqloc} We can write 
		\begin{align*}
			Q^{(\tilde{\tau})}_{\geq a}f|_{[\tau_*,\infty)} &= Q^{(\tau)}_{\geq a\cdot\tau^{-\frac12-\frac{1}{4\nu}-}}\circ Q^{(\tilde{\tau})}_{\geq a}f|_{[\tau_*,\infty)} + O_{\tau^{-N-3}L^2_{d\tau}}(\big\|f\big\|_{\tau^{-N}L^2_{d\tau}})\\
			& =: \tilde{f}^{(a)} + O_{\tau^{-N-3}L^2_{d\tau}}(\big\|f\big\|_{\tau^{-N}L^2_{d\tau}})\\
		\end{align*}
		for a function $f$ supported on $[\frac{\tau_*}{2},\infty)$. Moreover, if $\zeta(\hat{\tau})\in C^\infty(\R\backslash\{0\})$ is globally bounded with symbol type bounds, then for $a\geq 0$
		\begin{align*}
			\Big\|\mathcal{F}_{\tau}^{-1}\big(\zeta(\hat{\tau})\cdot \mathcal{F}_{\tau}\big(\tilde{f}^{(a)}\big)(\hat{\tau})\Big\|_{\tau^{-N}L^2_{d\tau}[(\tau_*,\infty)}\lesssim_a \big\|f\big\|_{\tau^{-N}L^2_{d\tau}}.
		\end{align*}
		If $\zeta$ moreover satisfies the bound $\big|\zeta(\hat{\tau})\big|\leq |\log\hat{\tau}|^{-2}$, $0<|\hat{\tau}|\ll1$, then we have 
		\begin{align*}
			\Big\|\mathcal{F}_{\tau}^{-1}\big(\zeta(\hat{\tau})\cdot \mathcal{F}_{\tau}\big(\tilde{f}^{(a)}\big)(\hat{\tau})\Big\|_{\log^{-2}\tau\cdot \tau^{-N}L^2_{d\tau}[(\tau_*,\infty)}\lesssim \big\|f\big\|_{\tau^{-N}L^2_{d\tau}}.
		\end{align*}
	\end{lem}
	\begin{proof} The first statement of the lemma follows by using that 
		\[
		\Big\|Q^{(\tau)}_{<a\cdot\tau^{-\frac12-\frac{1}{4\nu}-}}\circ Q^{(\tilde{\tau})}_{\geq a}f|_{[\tau_*,\infty)}\Big\|_{\tau^{-N-3}L^2_{d\tau}} \lesssim_a   \big\|f\big\|_{\tau^{-N}L^2_{d\tau}},
		\]
		in turn a consequence of repeated application of the relations
		\[
		\big\|\partial_{\tilde{\tau}}^{-1} Q^{(\tilde{\tau})}_{\geq a}f\big\|_{\tau^{-N}L^2_{d\tau}}\lesssim _a  \big\|f\big\|_{\tau^{-N}L^2_{d\tau}}, \Big\|\big(Q^{(\tau)}_{<a\cdot\tau^{-\frac12-\frac{1}{4\nu}-}}\circ \big(\frac{\partial\tau}{\partial_{\tilde{\tau}}}\partial_{\tau}\big)f\big)|_{[\tau_*,\infty)}\Big\|_{\tau^{-N-}L^2_{d\tau}}\lesssim _a  \big\|f\big\|_{\tau^{-N}L^2_{d\tau}}.
		\]
		For the second inequality, we write  (here $\chi_j(\tau)$ localizes smoothly to $\tau\sim 2^j$)
		\begin{align*}
			\tilde{f}^{(a)}|_{[\tau_*,\infty)} = \sum_{j\gtrsim \log\tau_*}Q^{(\tau)}_{\geq a\cdot 2^{-(\frac12+\frac{1}{4\nu}+)j}}\big(\chi_j(\tau)Q^{(\tilde{\tau})}_{\geq a}f\big)
		\end{align*}
		Then the desired inequality follows from Plancherel's theorem as well as the relation 
		\begin{align*}
			(\mathcal{F}^{(\tau)})^{-1}\partial_{\hat{\tau}}^{\alpha}\big(\zeta(\hat{\tau})\cdot \mathcal{F}^{(\tau)}Q^{(\tau)}_{\geq a\cdot 2^{-(\frac12+\frac{1}{4\nu}+)j}}(\cdot)\big): \tau^{-N}L^2_{d\tau}\longrightarrow 2^{\alpha\cdot (\frac12+\frac{1}{4\nu}+)j}\tau^{-N}L^2_{d\tau} + \tau^{-N+\alpha}L^2_{d\tau}
		\end{align*}
	\end{proof}

	\begin{lem}\label{lem:yzWbound1} Letting 
		\[
		f(\tau, R) := Q^{(\tilde{\tau})}_{<a}\big(\lambda^{-2}y_z\cdot W\big), 
		\]
		we have the bounds (for $l\geq 0$)
		\begin{align*}
			\Big\|\langle\xi\partial_{\xi}\rangle^{1+\delta_1}\partial_{\tau}^l\langle  f, \frac{\phi(R;\xi) - \phi(R;0)}{\xi^2}\rangle_{L^2_{R^3\,dR}} \Big\|_{a^{l}\tau^{-N-l(\frac12+\frac{1}{2\nu})+}L^2_{d\tau}L^2_{\rho(\xi)\,d\xi}}\lesssim \big\|z\big\|_{S}. 
		\end{align*}
		Also, we have the bounds (for $l\geq 0$)
		\begin{align*}
			\Big\|\langle\xi\partial_{\xi}\rangle^{1+\delta_1}\partial_{\tau}^l\langle  f, \phi(R;\xi)\rangle_{L^2_{R^3\,dR}} \Big\|_{a^{l}\tau^{-N-l(\frac12+\frac{1}{2\nu})+}L^2_{d\tau}L^2_{\rho(\xi)\,d\xi}}\lesssim \big\|z\big\|_{S}. 
		\end{align*}
	\end{lem}
	\begin{proof} The second estimate is a straightforward consequence of Lemma~\ref{cor:yzW} and the fact that 
		\[
		\partial_{\tau}^l Q_{<a}^{(\tilde{\tau})}: \tau^{-N}L^2_{d\tau}\longrightarrow a^l (\frac{\tilde{\tau}}{\tau})^l\tau^{-N}L^2_{d\tau}. 
		\]
		For the first bound, we split the parenthesis into a non-oscillatory and an oscillatory part (with respect to $\xi$) by writing 
		\begin{equation}\label{eq:fphiR0decomp}\begin{split}
				\langle  f, \frac{\phi(R;\xi) - \phi(R;0)}{\xi^2}\rangle_{L^2_{R^3\,dR}} &= \langle  f, \chi_{R\xi\lesssim 1}\frac{\phi(R;\xi) - \phi(R;0)}{\xi^2}\rangle_{L^2_{R^3\,dR}}\\
				& +  \langle f, \chi_{R\xi\gtrsim 1}\frac{\phi(R;\xi) - \phi(R;0)}{\xi^2}\rangle_{L^2_{R^3\,dR}}\\
		\end{split}\end{equation}
		The desired estimate for the first term on the right follows from a straightforward modification of Lemma~\ref{lem:yzWnonosc} (namely to include the temporal smoothing operator $Q^{(\tilde{\tau})}_{<a}$), which upon setting 
		\[
		\psi(R;\xi): = \chi_{R\xi\lesssim 1}\frac{\phi(R;\xi) - \phi(R;0)}{\xi^2}
		\]
		and keeping in mind subsection~\ref{subsec:basicfourier} implies
		\begin{align*}
			\Big\|\langle\xi\partial_{\xi}\rangle^{1+\delta_1}\langle  f, \chi_{R\xi\lesssim 1}\frac{\phi(R;\xi) - \phi(R;0)}{\xi^2}\rangle_{L^2_{R^3\,dR}} \Big\|_{\tau^{-N+}L^2_{d\tau}L^2_{\rho(\xi)\,d\xi}}\lesssim \big\|z\big\|_{S}
		\end{align*}
		The bound including $\partial_{\tau}^l$ follows by the same argument, taking into account the definition of $Q^{(\tilde{\tau})}_{<a}$. 
		\\
		As for the second term on the right in \eqref{eq:fphiR0decomp}, which is in the oscillatory regime for $\phi(R;\xi)$, 
		we have the schematic expansion
		\begin{align*}
			\chi_{R\xi\gtrsim 1}\frac{\phi(R;\xi) - \phi(R;0)}{\xi^2} = \chi_{R\xi\gtrsim 1}\sum_{\pm}\frac{e^{\pm iR\xi}}{R^{\frac32}\xi^{\frac32}} - \chi_{R\xi\gtrsim 1}\frac{1}{(1+R^2)\xi^2}.
		\end{align*}
		The contribution of the second term on the right is again handled by means of Lemma~\ref{lem:yzWnonosc}, upon choosing $\psi(R;\xi) = \chi_{R\xi\gtrsim 1}\frac{1}{(1+R^2)\xi^2}$. As for the first term on the right, applying 
		$\langle\xi\partial_{\xi}\rangle^{1+\delta_1}$ 'costs' $(R\xi)^{1+\delta_1}$, and so taking the inner product with $f$ results in a term falling under the purview of Lemma~\ref{lem:specialF1}. 
	\end{proof}
	
	Still in the general context of refined estimates around the problematic term $f$ from previous lemma and certain variants, we have the following improvement of Lemma~\ref{lem:wavebasicinhom}: 
	\begin{lem}\label{lem:refinedwavepropagatorwithphysicallocalization} Let $n$ be the wave evolution given by \eqref{eq:nflatfourierrepresent}, \eqref{eq:wavepropagator}. Then we have the bound 
		\begin{align*}
			\big\|\chi_{R\gtrsim \tau^{\frac12-}}n\big\|_{\tau^{-N}L^2_{d\tau}\dot{H}^1_{R^3\,dR}}\lesssim \big\|\lambda^{-2}\langle R\rangle F\big\|_{\tau^{-N-}L^2_{d\tau}L^2_{R^3\,dR}} + \big\|\lambda^{-2}\nabla^{-1}F\big\|_{\tau^{-N}L^2_{d\tau}L^2_{R^3\,dR}}
		\end{align*}
		In particular, if $F = \lambda^2\triangle \Re(W\bar{z})$, we obtain the bound 
		\begin{align*}
			&\big\|\chi_{R\gtrsim \tau^{\frac12-}}n\big\|_{\tau^{-N}L^2_{d\tau}L^{2+}_{R^3\,dR}}\lesssim \big\|z\big\|_{S},\\
			&\big\|R\partial_R\big(\chi_{R\gtrsim \tau^{\frac12-}}n\big)\big\|_{\tau^{-N}L^2_{d\tau}L^{2+}_{R^3\,dR}}\lesssim \big\|z\big\|_{S}
		\end{align*}
	\end{lem}
	\begin{proof} Recalling \eqref{eq:wavepropagator}, in conjunction with \eqref{eq:nflatfourierrepresent}, we observe that the two oscillatory phases $\phi_{\R^4}(R;\xi)$ and $\sin\big(\lambda(\tilde{\tau})\xi\int_{\tilde{\sigma}}^{\tilde{\tau}}\lambda^{-1}(s)\,ds\big)$ cannot be in resonance under the assumption $R\gtrsim \tau^{\frac12-}$, due to the fact that 
		\begin{align*}
			\big|\lambda(\tilde{\tau})\int_{\tilde{\sigma}}^{\tilde{\tau}}\lambda^{-1}(s)\,ds\big|\lesssim \tilde{\tau}\sim \tau^{\frac12-\frac{1}{4\nu}}\ll \tau^{\frac12-}. 
		\end{align*}
		The first inequality of the lemma then follows in straightforward manner via integration by parts with respect to $\xi$ in \eqref{eq:nflatfourierrepresent}, which produces a gain of $\ll \tilde{\tau}^{-1}$, ensuring we can compensate the (wave) temporal integration. To deduce the second bound of the lemma, it suffices to observe that 
		\begin{align*}
			\big\|\langle R\rangle\nabla \Re(W\bar{z})\big\|_{\tau^{-N}L^2_{d\tau}L^{2+}_{R^3\,dR}} + \big\|\Re(W\bar{z})\big\|_{\tau^{-N}L^2_{d\tau}L^{2+}_{R^3\,dR}}\lesssim \big\|z\big\|_{S}. 
		\end{align*}
		The third bound follows by using another integration by parts with respect to $\xi$ in \eqref{eq:nflatfourierrepresent}.
	\end{proof}
	
	In a similar vein and using an entirely analogous proof, we also have the following 'dual' version:
	\begin{lem}\label{lem:refinedwavepropagatorwithphysicallocalization2} The following bound obtains:
		\begin{align*}
			\Big\|\big(\nabla\Box^{-1}\big(\chi_{R\gtrsim \tau^{\frac12-}}\cdot F\big)\big)\Big\|_{\tau^{-N}L^2_{d\tau}\dot{H}^1_{R^3\,dR}}\lesssim \big\|\tilde{\tau}\cdot \langle R\rangle^{-1}F\big\|_{\tau^{-N}L^2_{d\tau}L^2_{R^3\,dR}}
		\end{align*}
	\end{lem}
	
	We shall also require a {\it{high-temporal frequency version}} of Lemma~\ref{lem:yzWbound1}, which avoids the small loss of temporal decay, and even gains smallness:
	\begin{lem}\label{lem:yzWbound3} Letting $0<\gamma = \gamma(\tau_*)$ with $\lim_{\tau_*\rightarrow+\infty}\gamma(\tau_*) = 0$, and setting $f(\tau, R): = Q_{>\gamma^{-1}}^{(\tilde{\tau})}\big(\lambda^{-2}y_z\cdot W\big)$, we have 
		\begin{align*}
			\Big\|\langle\xi\partial_{\xi}\rangle^{2+\delta_1}\langle  f, \frac{\phi(R;\xi) - \phi(R;0)}{\xi^2}\rangle_{L^2_{R^3\,dR}} \Big\|_{\tau^{-N}L^2_{d\tau}L^p_{\rho(\xi)\,d\xi}(\xi\lesssim1)}\ll_{\tau_*} \big\|z\big\|_{S},\,2\leq p\leq \infty. 
		\end{align*}
		For the high frequency regime $\xi\gtrsim 1$, setting $g(\sigma,\xi): = \chi_{\xi\gtrsim 1}\langle f, \phi(R;\xi)\rangle_{L^2_{R^3\,dR}}$, we have the bound 
		\begin{align*}
			\Big\|\int_{\tau}^\infty S(\tau,\sigma,\xi)\cdot g(\sigma, \frac{\lambda(\tau)}{\lambda(\sigma)}\xi)\,d\sigma\Big\|_{\tau^{-N}L^2_{d\tau}L^2_{\rho(\xi)\,d\xi}}\ll_{\tau_*} \big\|z\big\|_{S}.
		\end{align*}
	\end{lem}
	\begin{proof} {\it{First part}}. One splits $f$ into $f = P_{<\gamma^{-\frac14}}f + P_{\geq \gamma^{-\frac14}}f$, where the frequency localizers are standard Littlewood-Paley frequency cutoffs. For the first term, the action of $\Box^{-1}$ is then essentially given by $\partial_{\tilde{\tau}}^{-2}$, which gains $\gamma^2$ and causes no loss of temporal decay. For the second term $P_{\geq \gamma^{-\frac14}}f$, one follows the proof of Lemma~\ref{lem:yzWbound1} and exploits the fact that the high spatial frequency localization ensures convergence of the $R$-integral. 
		\\
		{\it{Second part}}. This follows by repeated integration by parts with respect to $\sigma$, taking advantage of Lemma~\ref{lem:wavebasicinhom} and the fact that $\frac{\lambda(\tau)}{\lambda(\sigma)}\xi\gtrsim 1$.
	\end{proof}

	The delicate term $\lambda^{-2}y_z\cdot W$ is counterbalanced in some sense by the term $\lambda^{-2}y_{\tilde{\lambda}}^{mod}\cdot W$, which differs subtly from the former as there is no more $\triangle$-operator present in the definition of $y_{\tilde{\lambda}}$, see \eqref{eq:ylamndatildemod}, \eqref{eq:E2mod}. We shall require the following precise result:
	\begin{lem}\label{lem:ytildelambdatimesWlocalizedbound} Given $\delta_*>0$, and letting $P_{<a}^{(\mathcal{L})}$ denote frequency localization with respect to the operator $\mathcal{L}$ we have the bound
		\begin{align*}
			\Big\|\langle\partial_{\xi}\rangle\mathcal{F}\Big(P_{<\tau^{-\delta_*}}^{(\mathcal{L})}\big(W\cdot \lambda^{-2}y_{\tilde{\lambda}}^{mod}\Big)\Big\|_{\tau^{-N}L^2_{d\tau}L^\infty_{d\xi}}\lesssim \big\|\langle\partial_{\tilde{\tau}}^2\rangle^{-1}\tilde{\lambda}_{\tilde{\tau}\tilde{\tau}}\big\|_{\tau^{-N}L^2_{d\tau}}.
		\end{align*}
		We also have the more crude bound 
		\begin{align*}
			\Big\|\langle\partial_{\xi}\rangle\mathcal{F}\Big(W\cdot \lambda^{-2}y_{\tilde{\lambda}}^{mod}\Big)\Big\|_{\tau^{-N+}L^2_{d\tau}L^\infty_{d\xi}}\lesssim \big\|\langle\partial_{\tilde{\tau}}^2\rangle^{-1}\tilde{\lambda}_{\tilde{\tau}\tilde{\tau}}\big\|_{\tau^{-N}L^2_{d\tau}}.
		\end{align*}
	\end{lem}
	\begin{proof}(sketch, first estimate) Write the Fourier coefficient schematically\footnote{Here $\phi(R;\xi)$ refers to the Fourier basis of subsection~\ref{subsec:basicfourier}} as 
		\begin{align*}
			\chi_{\xi<\tau^{-\delta_*}}\int_0^\infty \phi(R;\xi)\cdot W\cdot\lambda^{-2}y_{\tilde{\lambda}}^{mod}\cdot R^3\,dR, 
		\end{align*}
		and observe that applying the operator $\xi^{1-\delta_1}\partial_{\xi}$ 'costs' $\xi^{1-\delta_1}\cdot R$. Then we treat various cases:
		\\
		{\it{(1): large $R$ case, $R>\tau^{1000}$.}} Here we can use the crude bound
		\begin{align*}
			&\Big\|\langle\partial_{\xi}\rangle\big(\chi_{\xi<\tau^{-\delta_*}}\int_0^\infty \chi_{R>\tau^{1000}}\phi(R;\xi)\cdot W\cdot\lambda^{-2}y_{\tilde{\lambda}}^{mod}\cdot R^3\,dR\big)\Big\|_{\tau^{-N-1}L^2_{d\tau}L^\infty_{d\xi}}\\
			&\lesssim\big\| \tau^2\chi_{R>\tau^{1000}}\phi(R;\xi)\cdot W\cdot R\big\|_{L^\infty_{d\tau}L^2_{R^3\,dR}}\cdot \big\|\lambda^{-2}y_{\tilde{\lambda}}^{mod}\big\|_{\tau^{-(N-1)}L^2_{d\tau}L^2_{R^3\,dR}}\\
			&\ll_{\tau_*}\big\|\langle \partial_{\tilde{\tau}}^2\rangle^{-1}\tilde{\lambda}_{\tilde{\tau}\tilde{\tau}}\big\|_{\tau^{-N}L^2_{d\tau}},
		\end{align*}
		where we have taken advantage of Lemma~\ref{lem:ytildelambdaE2}.
		\\
		{\it{(2): $R$ bounded}}.  Consider next the contribution of 
		\begin{align*}
			\chi_{\xi<\tau^{-\delta_*}}\int_0^\infty \chi_{R<\tau^{1000}}\phi(R;\xi)\cdot W\cdot\lambda^{-2}y_{\tilde{\lambda}}^{mod}\cdot R^3\,dR,
		\end{align*}
		Expand $\lambda^{-2}y_{\tilde{\lambda}}^{mod}$ as in \eqref{eq:nflatfourierrepresent}, \eqref{eq:wavepropagator}, and label $\eta$ the frequency variable in this Fourier representation. To control the time integration in \eqref{eq:wavepropagator}, we have to 'spend' slightly more than one power of $R$ (see the proof of Lemma~\ref{lem:wavebasicinhom}), and the operator $\xi^{1-\delta_1}\partial_{\xi}$ 'costs' another factor $R$. Then perform integration by parts with respect to $R$, by combining the oscillatory phases $\phi(R;\xi), \phi_{\R^4}(R;\eta)$. Since either of these functions decay like $R^{-\frac32}$ towards $R = +\infty$, we thereby arrive at an $R$-integral which only diverges by $R^{0+}$, and so we lose $\tau^{1000\cdot 0+} = \tau^{0+}$ due to out restriction on $R$. The $\eta$-integral then converges, using the simple estimate (following from Lemma~\ref{lem:ytildelambdaE2} )
		\begin{align*}
			\big\|\mathcal{F}_{\R^4}\big(\lambda^{-4}E_2^{mod}\big)(\tau,\eta)\big\|_{\tau^{-N}L^2_{d\tau}\langle \eta\rangle^{10}L^M_{d\eta}}\lesssim \big\|\tilde{\lambda}_{\tilde{\tau}\tilde{\tau}}\big\|_{\tau^{-N}L^2_{d\tau}}.
		\end{align*}
		The extra factor $\xi^{\frac12}$ coming from $\phi(R;\xi)$ and our restriction on $\xi$ compensate for small loss $\tau^{0+}$ due to the $R$-integral (see subsection~\ref{subsec:basicfourier}). This gives the desired estimate in the low temporal frequency regime for $\tilde{\lambda}$, while in the high temporal frequency regime, the operator $\Box^{-1}$  allows us to replace the right hand side by $\big\|\langle \partial_{\tilde{\tau}}^2\rangle^{-1}\tilde{\lambda}_{\tilde{\tau}\tilde{\tau}}\big\|_{\tau^{-N}L^2_{d\tau}}$.

	\end{proof}
	
	The next lemma deals with inversion of the operator $I + \tilde{T}$ which arose at the end of subsection~\ref{subsec:nonresIIsmallfreq} 
	\begin{lem}\label{lem:oneplustildetinverse} Recalling \eqref{eq:tildeT} for the definition of $\tilde{T}$, the equation 
		\begin{align*}
			\big(I + \tilde{T}\big)u = f
		\end{align*}
		with $f\in L^2_{R^3\,dR}$ admits a unique solution $u\in L^2_{R^3\,dR}$. 
	\end{lem}
	\begin{proof} First consider the special case 
		\[
		f = \triangle^{-1}\big(\Lambda W\cdot W\big)\cdot W
		\]
		Then we make the ansatz 
		\[
		u = \gamma\cdot  \triangle^{-1}\big(\Lambda W\cdot W\big)\cdot W
		\]
		for suitable $\gamma\in \R$ which leads to the condition 
		\begin{align*}
			\gamma\cdot \big(1 + \alpha_*\beta*\big) = 1, 
		\end{align*}
		where we recall Lemma~\ref{lem:lowtempfreqznresprin2} and its proof for the definition of $\alpha_*,\phi$ and we set 
		\[
		\beta_*: = \int_0^\infty \phi\cdot \triangle^{-1}\big(\Lambda W\cdot W\big)\cdot W R^3\,dR. 
		\]
		Since in fact $\phi = -\frac{W}{2} - \frac{1}{16}\Lambda W$, we infer from the explicit value of $\alpha_*$ (see e. g. Lemma~\ref{lem:lowtempfreqznresprin2}) that 
		\begin{align*}
			1 + \alpha_*\beta_*&=  -\frac{1}{16}\int_0^\infty \Lambda W\cdot \triangle^{-1}\big(\Lambda W\cdot W\big)\cdot W R^3\,dR\\
			& = \frac{\alpha_*}{16}\int_0^\infty \big|\nabla\triangle^{-1}\big(\Lambda W\cdot W\big)\big|^2 R^3\,dR\neq 0, 
		\end{align*}
		which allows to determine $\gamma$ uniquely. 
		\\
		Next, for general $f$, determine $\gamma_1$ such that 
		\begin{align*}
			\gamma_1\int_0^\infty \phi\cdot \triangle^{-1}\big(\Lambda W\cdot W\big)\cdot W R^3\,dR = \int_0^\infty \phi\cdot fR^3\,dR, 
		\end{align*}
		which can be done since $\beta_*\neq 0$, thanks to numerical assumption {\bf{(B1)}} in subsection~\ref{subsec:numerics}. Then setting 
		\[
		\tilde{u}: = f - \gamma_1\cdot  \triangle^{-1}\big(\Lambda W\cdot W\big)\cdot W,
		\]
		we have $\tilde{T}(\tilde{u}) = 0$, whence 
		\[
		\big(I + \tilde{T}\big)(\tilde{u}) = \tilde{u}.
		\]
		Finally, it suffices to set 
		\begin{align*}
			u = \tilde{u} + \gamma_1\gamma\cdot  \triangle^{-1}\big(\Lambda W\cdot W\big)\cdot W.
		\end{align*}
	\end{proof}
	
	The following lemma is in a similar vein:
	\begin{lem}\label{lem:T1inversion} Let $T_1(z): = \int_0^\infty z\cdot W\cdot\triangle(W^2)R^3\,dR$. Then assuming non-degeneracy condition {\bf{(B2)}}, there is a linear continuous map $\Phi: S\longrightarrow S$ such that the equation (recall \eqref{eq:psidefn})
		\[
		z - \frac{2}{\alpha_{**}}\cdot T_1(z)\cdot \psi = g\in S
		\]
		is solved by $z = \Phi(g)$. 
	\end{lem}
	
	The next couple of lemmas deal with the control of the transference operator \eqref{eq: transference1}. The following lemma shall help deal with certain resonant situations:
	\begin{lem}\label{lem:oscillatoryintegral1} There is an absolute constant $C$ such that for any $\xi\in\R_+, \kappa\in \R_+$, $a\in (0,\frac12)$, we have 
		\begin{align*}
			\big|\int_{-a\cdot\xi}^{a\cdot\xi} \frac{e^{i(\eta^2 + 2\xi\cdot\eta)\kappa}}{\eta}\,d\eta\big|\leq C,
		\end{align*}
		where we use the principal value interpretation of the integral. 
		Furthermore, calling the integral $g_a(\xi,\kappa)$, we have 
		\begin{align*}
			\big|\partial_{\xi}g_a(\xi,\kappa)\big|\leq C_1\xi^{-1}
		\end{align*}
		for a universal constant $C_1$, provided $0<\xi\lesssim 1$. We can also obtain similar higher derivative bounds, provided we extend the integral over $(-\infty, \infty)$ and use a smooth even cutoff $\chi_{|\eta|<a\xi}$ instead. 
	\end{lem}
	\begin{proof} Split the integral into 
		\begin{align*}
			\int_0^{\min\{(\xi\kappa)^{-1}, a\xi\}} e^{i\eta^2\kappa}\cdot \frac{\sin(2\xi\kappa\eta)}{\eta}\,d\eta + \sum_{\pm}\int_{\pm\min\{(\xi\kappa)^{-1},a\xi\}}^{\pm a\cdot\xi} \frac{e^{i(\eta^2 + 2\xi\cdot\eta)\kappa}}{\eta}\,d\eta
		\end{align*}
		The first integral is bounded by 
		\begin{align*}
			\big|\int_0^{\min\{(\xi\kappa)^{-1},a\xi\}} e^{i\eta^2\kappa}\cdot \frac{\sin(2\xi\kappa\eta)}{\eta}\,d\eta \big|\leq (\xi\kappa)^{-1}\cdot C\cdot (\xi\kappa) = C. 
		\end{align*}
		For the second integral we perform integration by parts, using that 
		\begin{align*}
			e^{i(\eta^2 + 2\xi\cdot\eta)\kappa} = \frac{1}{2i(\eta + \xi)\kappa}\cdot\partial_{\eta}\big(e^{i(\eta^2 + 2\xi\cdot\eta)\kappa}\big),
		\end{align*}
		and we have 
		\begin{align*}
			&\big|\frac{1}{2i(\eta + \xi)\kappa\cdot\eta}\cdot e^{i(\eta^2 + 2\xi\cdot\eta)\kappa}|_{\pm\min\{(\xi\kappa)^{-1},a\xi\}}^{\pm a\cdot\xi} + \big|\int_{\pm\min\{(\xi\kappa)^{-1},a\xi\}}^{\pm a\cdot\xi}e^{i(\eta^2 + 2\xi\cdot\eta)\kappa}\cdot \partial_{\eta}\big( \frac{1}{2i\eta(\eta + \xi)\kappa}\big)\,d\eta\big|\\
			&\leq C
		\end{align*}
		due to our assumption on $a$. For the derivative bound, observe that 
		\begin{align*}
			\big|\int_{-a\cdot\xi}^{a\cdot\xi} \kappa\cdot e^{i(\eta^2 + 2\xi\cdot\eta)\kappa}\,d\eta\big| = \big|\int_{-a\cdot\xi}^{a\cdot\xi} \frac{\kappa}{2(\eta+\xi)\kappa}\cdot \partial_{\eta}\big(e^{i(\eta^2 + 2\xi\cdot\eta)\kappa}\big)\,d\eta\big|\leq C_2\xi^{-1}
		\end{align*}
		using integration by parts. 
	\end{proof}
	
	\begin{rem}\label{rem:oscillatoryintegral1} The same bound obtains if we also include a factor $g(\eta)$ into the integral where $g$ is a bounded $C^1$-function satisfying a bound $\big| g'(\eta)\big|\lesssim \langle\eta\rangle^{-\delta}$ for some $\delta>1$. In particular, we can infer the bound 
		\begin{align*}
			&\big|\int_{-\infty}^{\infty} \chi_{|\eta|<\frac{\xi}{2}}F(\xi,\eta+\xi)\frac{e^{i(\eta^2 + 2\xi\cdot\eta)\kappa}}{\eta}\,d\eta\big|\leq C,\\&\big|\partial_{\xi}\int_{-\infty}^{\infty} \chi_{|\eta|<\frac{\xi}{2}}F(\xi,\eta+\xi)\frac{e^{i(\eta^2 + 2\xi\cdot\eta)\kappa}}{\eta}\,d\eta\big|\leq \frac{C}{\xi},
		\end{align*}
		where $F(\cdot, \cdot)$ is as in the kernel of $\mathcal{K}$, see \eqref{eq:Kformula}, where $C$ does not depend on $\xi,\kappa$, and analogously for higher order derivatives. 
	\end{rem}
	
	In order to deal with the propagator $S_{\mathcal{K}}$, defined in \eqref{eq:ScalK2}, we shall require a kind of {\it{concatenation lemma}}, which helps deal with strings of many $\mathcal{K}$'s:
	\begin{lem}\label{lem:concatenation1} Let $j\geq 1$ We have the bound 
		\begin{align*}
			&\Big\|\int_0^\infty \xi^2\cdot S\big(-i\frac{\lambda_{\tau}}{\lambda}\mathcal{K}\circ S(G)\big)^j\rho(\xi)\,d\xi\Big\|_{\tau^{-N}L^2_{d\tau}}\\&\hspace{5cm}\lesssim (\sqrt{N})^{-j}\cdot \big\|\langle \xi\partial_{\xi}\rangle G\big\|_{\tau^{-N}L^2_{d\tau}L^2_{\rho(\xi)\,d\xi}}. 
		\end{align*}
	\end{lem}
	\begin{proof} Observe that the operator 
		\[
		S\big(-i\frac{\lambda_{\tau}}{\lambda}\mathcal{K}\circ S\big)
		\]
		involves integration against $\sigma,\sigma_1,\eta$ of the kernel
		\begin{align*}
			-i\frac{\lambda_{\sigma}}{\lambda}S(\tau,\sigma,\xi)\cdot \frac{F(\frac{\lambda(\tau)}{\lambda(\sigma)}\xi,\eta)\rho(\eta)}{\frac{\lambda(\tau)}{\lambda(\sigma)}\xi - \eta}\cdot S(\sigma,\sigma_1,\eta). 
		\end{align*}
		Here it is natural to introduce the variables
		\begin{align*}
			\tilde{\eta}: = \frac{\lambda(\sigma)}{\lambda(\tau)}\cdot \eta,\,\triangle \tilde{\eta}: = \tilde{\eta} - \xi,
		\end{align*}
		by means of which the preceding kernel can be re-expressed in the form
		\begin{equation}\label{eq:SKfirsttermexplicit}
			-i\frac{\lambda_{\sigma}}{\lambda}S(\tau,\sigma_1,\xi)\cdot \frac{ F\big(\frac{\lambda(\tau)}{\lambda(\sigma)}\xi,\frac{\lambda(\tau)}{\lambda(\sigma)}(\xi + \triangle \tilde{\eta})\big)\rho(\frac{\lambda(\tau)}{\lambda(\sigma)}(\xi + \triangle \tilde{\eta})}{\frac{\lambda(\tau)}{\lambda(\sigma)}\cdot \triangle \tilde{\eta}}\cdot e^{i(\triangle \tilde{\eta}^2 + 2\triangle \tilde{\eta}\cdot\xi)\kappa}
		\end{equation}
		where in the final exponential we have 
		\[
		\kappa = \kappa(\tau,\sigma,\sigma_1) = \lambda^2(\tau)\cdot\int_{\sigma_1}^{\sigma}\lambda^{-2}(s)\,ds. 
		\]
		Write
		\begin{equation}\label{eq:SKproduct1}
			S\big(-i\frac{\lambda_{\tau}}{\lambda}\mathcal{K}\circ S\big)^j = S\circ \prod_{l=1}^j (-i\frac{\lambda_{\tau}}{\lambda}\mathcal{K})\circ S,
		\end{equation}
		and decompose the $l$-th term in the right hand product into 
		\begin{align*}
			(-i\frac{\lambda_{\tau}}{\lambda}\mathcal{K})\circ S = (-i\frac{\lambda_{\tau}}{\lambda}\mathcal{K}_{res}^{(l)})\circ S +  (-i\frac{\lambda_{\tau}}{\lambda}\mathcal{K}_{nres}^{(l)})\circ S
		\end{align*}
		where we define the kernel of $\mathcal{K}_{res}^{(l)}$ as\footnote{Here we have replaced the integration variables $\sigma, \triangle\tilde{\eta}$ in the preceding by $\sigma_l, \triangle\tilde{\eta}_l$, while the variables $\xi, \tau$ get replaced by $\eta_{l-1}, \sigma_{l-1}$. The $\xi$ in the additional cutoff refers to the 'output' frequency of the entire product.}
		\[
		\frac{F(\eta_{l-1}, \eta_{l-1} + \frac{\lambda(\sigma_{l-1})}{\lambda(\sigma_l)}\triangle\tilde{\eta}_{l})}{ \frac{\lambda(\sigma_{l-1})}{\lambda(\sigma_l)}\triangle\tilde{\eta}_{l}}\cdot\rho(\eta_{l-1} + \frac{\lambda(\sigma_{l-1})}{\lambda(\sigma_l)}\triangle\tilde{\eta}_{l})\cdot\chi_{\frac{\lambda(\sigma_{l})}{\lambda(\tau)}\triangle\tilde{\eta}_{l}<\frac{\xi}{10 l^2}}
		\]
		Then we expand 
		\begin{equation}\label{eq:SKproductexpansion1}\begin{split}
				&S\circ \prod_{l=1}^j (-i\frac{\lambda_{\tau}}{\lambda}\mathcal{K})\circ S\\& = S\circ\sum_{a=0}^j\prod_{l=1}^a(-i\frac{\lambda_{\tau}}{\lambda}\mathcal{K}_{res}^{(l)})\circ S\circ (-i\frac{\lambda_{\tau}}{\lambda}\mathcal{K}_{nres}^{(a+1)})\circ S\circ \prod_{l=a+2}^k  (-i\frac{\lambda_{\tau}}{\lambda}\mathcal{K})\circ S
		\end{split}\end{equation}
		where by definition we set $\prod_{l=1}^0(-i\frac{\lambda_{\tau}}{\lambda}\mathcal{K}_{res}^{(l)})\circ S := \text{id}$. The main point then is to understand the composition of the first five operators on the left. Inductively using 
		\begin{align*}
			\eta_l = \eta_{l-1} + \frac{\lambda(\sigma_{l-1})}{\lambda(\sigma_l)}\triangle\tilde{\eta}_{l},
		\end{align*}
		and combining all exponential phases in the composition of the first three operators, we arrive at the following $\xi$-dependent phase function: 
		\begin{equation}\label{eq:resprodxiphase1}
			e^{i\lambda^2(\tau)[\xi^2\int_{\sigma_a}^{\tau}\lambda^{-2}(s)\,ds + 2\xi\sum_{j=0}^{a}\frac{\lambda(\sigma_j)}{\lambda(\tau)}\triangle\tilde{\eta}_{j}\cdot\int_{\sigma_a}^{\sigma_{j}}\lambda^{-2}(s)\,ds]}
		\end{equation}
		Note that the cutoffs in the kernels for $\mathcal{K}_{res}^{(l)}$ and the fact that the we have $\tau\leq \sigma\leq \sigma_l\leq \sigma_{l+1}$ 
		imply that 
		\begin{align*}
			\Big|\sum_{j=0}^{a}\frac{\lambda(\sigma_j)}{\lambda(\tau)}\triangle\tilde{\eta}_{j}\cdot\int_{\sigma_a}^{\sigma_{j}}\lambda^{-2}(s)\,ds\Big|\leq \frac{\xi}{6}\cdot \int_{\sigma_a}^{\tau}\lambda^{-2}(s)\,ds.
		\end{align*}
		This in turn implies that the phase in \eqref{eq:resprodxiphase1} is in the non-stationary case for its dependence on $\xi$. Next, from the definition of the non-resonant kernel we deduce the bound 
		\begin{align*}
			\Big\|(-i\frac{\lambda_{\tau}}{\lambda}\langle\xi\partial_{\xi}\rangle\mathcal{K}_{nres}^{(a+1)})\circ S\Big\|_{\tau^{-N}L^2_{d\tau}L^2_{\rho(\xi)\,d\xi}\rightarrow \tau^{-N}L^2_{d\tau}L^2_{\rho(\xi)\,d\xi}}\lesssim \frac{a^2}{N}. 
		\end{align*}
		Performing integration by parts with respect to $\xi$ in the integral displayed in the lemma, and 'disentangling' the variables $\xi, \triangle\tilde{\eta}_{j}$ by a power series expansion for \footnote{These remarks apply on the support of the integrand where the variables $\triangle\tilde{\eta}_{j}$ are restricted as indicated above.}
		\begin{align*}
			&\Big(\partial_{\xi}\big(i\lambda^2(\tau)[\xi^2\int_{\sigma_a}^{\tau}\lambda^{-2}(s)\,ds + 2\xi\sum_{j=0}^{a}\frac{\lambda(\sigma_j)}{\lambda(\tau)}\triangle\tilde{\eta}_{j}\cdot\int_{\sigma_a}^{\sigma_{j}}\lambda^{-2}(s)\,ds]\big)\Big)^{-1}\\
			&\sim (-i)\xi^{-1}\cdot\big(\lambda^2(\tau)\int_{\sigma_a}^{\tau}\lambda^{-2}(s)\,ds\big)^{-1}, 
		\end{align*}
		the lemma is then a consequence of repeated application of the bounds 
		\begin{align*}
			\Big\|(-i\frac{\lambda_{\tau}}{\lambda}\mathcal{K}_{res}^{(l)})\circ S\Big\|_{\tau^{-N}L^2_{d\tau}L^2_{\rho(\xi)\,d\xi}\rightarrow \tau^{-N}L^2_{d\tau}L^2_{\rho(\xi)\,d\xi}}\lesssim \frac{1}{N}, 
		\end{align*}
		and similarly when $\mathcal{K}_{res}^{(l)}$ is replaced by $\mathcal{K}$. 
	\end{proof}
	
	For the proof of Proposition~\ref{prop:solnoftildelambdaeqn}, and more specifically Lemma~\ref{lem:deltaPhibound}, we shall need the following lemma: 
	\begin{lem}\label{lem:SKdeltaPhi} Let $g(\tau)\in \tau^{-N}L^2_{d\tau}$. Then we have  
		\begin{align*}
			\Big\|\int_0^\infty \xi^2\cdot S_{\mathcal{K}}\big(g(\cdot)\big)(\tau,\xi)\rho_1(\xi)\,d\xi\Big\|_{\log^{-2}\tau\cdot\tau^{-N}L^2_{d\tau}}\ll_{\tau_*}\big\|g\big\|_{ \tau^{-N}L^2_{d\tau}}
		\end{align*}
	\end{lem}
	\begin{proof}(sketch) It suffices to show that 
		\begin{align*}
			\Big\||\int_0^\infty \xi^2\cdot S\circ \Big(\prod_{l=1}^j (-i\frac{\lambda_{\tau}}{\lambda}\mathcal{K})\circ S\Big)(g)\rho_1(\xi)\,d\xi\Big\|_{\log^{-2}\tau\cdot\tau^{-N}L^2_{d\tau}}\ll_{\tau_*}(\sqrt{N})^{-(j-1)}\big\|g\big\|_{ \tau^{-N}L^2_{d\tau}},
		\end{align*}
		since then summing over $j$ furnishes the desired result. We first consider the case $j = 1$. As in the preceding proof split $\mathcal{K}$ into $\mathcal{K}_{res}^{(1)} + \mathcal{K}_{nres}^{(1)}$. 
		\\
		{\it{Contribution of $\mathcal{K}_{res}^{(1)}$.}}
		Calling $\sigma,\sigma_1$ the time variables in the two propagators $S$, while $\tau$ is the 'output time', it is straightforward to see that we may assume $\sigma_1 - \tau \geq \tau^{\delta}$ for some small $\delta>0$ as otherwise the factor $\frac{\lambda_{\tau}}{\lambda}$ compensates for the two time integrations and even results in a power gain in $\tau^{-1}$: denoting by $\tau$ as usual the 'output Schr\"odinger time variable', we have 
		\begin{align*}
			&\Big|\int_0^\infty \xi^2\cdot S\circ (-i\frac{\lambda_{\tau}}{\lambda}\mathcal{K}_{res}^{(1)})\circ S(\chi_{|\sigma_1 - \tau|<\tau^{\delta}}g)\rho_1(\xi)\,d\xi\Big|\\
			&\lesssim \tau^{\delta}\cdot \sup_{|\sigma-\tau|<\tau^{\delta}}\big\| (-i\frac{\lambda_{\sigma}}{\lambda}\mathcal{K}_{res}^{(1)})\circ S(\chi_{|\sigma_1 - \tau|<\tau^{\delta}}g)\big\|_{L^2_{d\xi}}
		\end{align*}
		and from here we deduce 
		\begin{align*}
			&\Big\|\int_0^\infty \xi^2\cdot S\circ (-i\frac{\lambda_{\tau}}{\lambda}\mathcal{K}_{res}^{(1)})\circ S(\chi_{|\sigma_1 - \tau|<\tau^{\delta}}g)\rho_1(\xi)\,d\xi\Big\|_{\tau^{-N-1+2\delta}}\lesssim \big\|g\big\|_{\tau^{-N}L^2_{d\tau}}. 
		\end{align*}
		Letting $\eta$ be the integration variable in $\mathcal{K}_{res}^{(1)}$, we have by assumption $\eta\sim \frac{\lambda(\tau)}{\lambda(\sigma)}\xi$. We may further assume $\xi<\tau^{-\delta}$. For if not, recalling that the $\xi$-dependent phase coming from the two propagators $S$ can be written as in \eqref{eq:resprodxiphase1}, performing integration by parts twice with respect to $\xi$ (after including a smooth cutoff $\chi_{\xi>\tau^{-\delta}}$) compensates for the time integrations in both $S$ propagators, at the cost of $\tau^{2\delta+}$, which in turn gets more than compensated by the factor $\frac{\lambda_{\tau}}{\lambda}$; here it suffices again to work with a simple $L^2_{d\xi}$ estimate for the transference operator. We can henceforth include a smooth cutoff $\chi_{\xi<\tau^{-\delta}}$. Finally, to complete the contribution of $\mathcal{K}_{res}^{(1)}$, we observe that there is a factor $\rho(\eta)\sim \rho(\xi)$ in the kernel of $\mathcal{K}_{res}^{(1)}$, and so in total we have two factors $\rho(\xi)$, if we replace $\rho(\eta)$ in the kernel by $\frac{\rho(\eta)}{\rho(\xi)}\sim 1$. If we then again use the parameters $\xi, \triangle\tilde{\eta}$ from the preceding proof and collect the $\xi$-dependent phases according to \eqref{eq:resprodxiphase1} and treat the $\xi$-integral by splitting into two cases as in the proof of Lemma~\ref{lem:tildeKfcontrol}, and finally take advantage of the fact that $\sigma_1 - \tau \geq \tau^{\delta}$, we gain a factor $\log^{-3}\tau$ after executing the $\sigma_1$-integral. The remaining integral over $\triangle\tilde{\eta}$ can then be performed by taking advantage of Lemma~\ref{lem:oscillatoryintegral1} and Remark~\ref{rem:oscillatoryintegral1}.
		This is easily seen to lead to the estimate claimed in the lemma for this contribution with a $\log^{-1}\tau_*$-gain. More specifically, we recall \eqref{eq:SKfirsttermexplicit} where we include an additional smooth cutoff $\chi_{\big|\triangle\tilde{\eta}\big|\leq \frac{\xi}{10}}$, and apply Lemma~\ref{lem:tildeKfcontrol}, Remark~\ref{rem:tildeKfcontrol} where we replace $\sigma$ by $\sigma_1$ and we let 
		\begin{align*}
			&f(\tau,\sigma_1, \xi): = g(\sigma_1)\int_{\tau}^{\sigma_1} \frac{\lambda_{\sigma}}{\lambda}\cdot \int_{-\infty}^{\infty} \frac{\tilde{\chi} F\big(\xi,\xi + \triangle \tilde{\eta}\big)\rho(\xi + \triangle \tilde{\eta})}{\triangle \tilde{\eta}}\cdot e^{i(\triangle \tilde{\eta}^2 + 2\triangle \tilde{\eta}\cdot\xi)\kappa}\,d(\triangle \tilde{\eta}) d\sigma, 
		\end{align*}
		and where the unspecified cutoff $\tilde{\chi}$ localizes simultaneously to $\big|\triangle\tilde{\eta}\big|\leq \frac{\xi}{10}, \xi<\tau^{-\delta}$. We can then obtain the desired estimate as described before by using first Lemma~\ref{lem:oscillatoryintegral1} and Remark~\ref{rem:oscillatoryintegral1}, and then Lemma~\ref{lem:tildeKfcontrol}, Remark~\ref{rem:tildeKfcontrol}, by observing that 
		\begin{align*}
			\big|\langle\xi\partial_{\xi}\rangle^{1+\delta_1}\big(\tilde{\chi}F\big(\xi,\xi + \triangle \tilde{\eta}\big)\rho(\xi + \triangle \tilde{\eta})\big)\big|\lesssim \log^{-2}\tau.
		\end{align*}
		
		{\it{Contribution of $\mathcal{K}_{nres}^{(1)}$.}} Here the strategy is to simplify the phase of the left most propagator $S$ to the simpler form $e^{i\xi^2(\tau-\sigma)}$ and further remove all factors $\frac{\lambda(\tau)}{\lambda(\sigma)}$, so that it becomes simple to compute the (Schr\"odinger) temporal Fourier transform of the expression. To begin with, we claim that we can restrict the variable $\eta$ in $\mathcal{K}_{nres}^{(1)}$ to size $\eta<\tau^{-\delta}$ via a smooth cutoff. For smoothly restricting to the regime $\eta>\tau^{-\delta}$ we can perform integration by parts with respect to $\eta$ in $\mathcal{K}_{nres}^{(1)}\circ S(g)$, which up to a logarithmic loss compensates for the time integration over $\sigma_1$. Then invoking Lemmas~\ref{lem:K_frefined}, ~\ref{lem:tildeKfcontrol}, and taking advantage of the extra factor $\frac{\lambda_{\sigma}}{\lambda}$, we obtain a polynomial gain in $\tau^{-1}$ for this contribution. 
		Once $\eta$ is restricted to $\eta<\tau^{-\delta}$, the $\eta$-integral gains $\log^{-1}\tau$ due to the factor $\rho(\eta)\sim \frac{1}{\eta\log^2\eta}$. By following the proof of Lemma~\ref{lem:reductionsteps1}, we remove the scaling factor $\frac{\lambda(\tau)}{\lambda(\sigma)}$, and simplify kernel of the propagator $S$ on the left to  $e^{i(\sigma - \tau)\xi^2}$. Observe that the modified expression $S\circ (-i\frac{\lambda_{\tau}}{\lambda}\mathcal{K})\circ S(g)$ is now given by 
		\begin{equation}\label{eq:Knres1Sgsimplified}
			\int_{\tau}^\infty e^{i(\sigma - \tau)\xi^2}\cdot (-i)\frac{\lambda_{\sigma}}{\lambda}\cdot \big(\mathcal{K}_{nres}^{(1)}\circ S\big)(g)(\sigma,\xi)\,\rho(\xi)\,d\xi,
		\end{equation}
		Since $\tau$ is restricted to $[\tau_*,\infty)$, we may replace the expression\footnote{The extra $\log^{-1}\sigma$ comes from the localization of the variable $\eta$} $\frac{\lambda_{\sigma}}{\lambda}\cdot \big(\mathcal{K}_{nres}^{(1)}\circ S\big)(g)\in \log^{-1}\sigma\cdot\sigma^{-N}L^2_{d\sigma}L^2_{\rho(\xi)\,d\xi}$ by 
		\begin{align*}
			\Pi\Big(\frac{\lambda_{\sigma}}{\lambda}\cdot \big(\mathcal{K}_{nres}^{(1)}\circ S\big)(g)\Big), 
		\end{align*}
		where the operator $\Pi$ is defined in \eqref{eq:Pidefinition}. We can then interpret \eqref{eq:Knres1Sgsimplified} as convolution of two functions with respect to (Schr\"odinger) time, and compute the temporal Fourier transform as in the proof of Lemma~\ref{lem:Fouriertransform1}, leading to a function which lives in $\log^{-2}\hat{\tau}\cdot\hat{\tau}^NL^2_{d\hat{\tau}}\cap W^{N,2}_{\hat{\tau}}$, and which actually gains a smallness factor $\log^{-1}\tau_*$ when evaluating the corresponding norm. 
		This implies the assertion of the lemma for this contribution, concluding the case $j = 1$. 
		\\
		The case $j>1$ is handled similarly, taking advantage of the decomposition \eqref{eq:SKproductexpansion1}. If the product $\prod_{l=1}^a(-i\frac{\lambda_{\tau}}{\lambda}\mathcal{K}_{res}^{(l)})$ is non-trivial, we proceed in analogy to the first situation in the case $j = 1$, using the $\xi$-dependent phase \eqref{eq:resprodxiphase1}. On the other hand, for the term 
		\begin{align*}
			S\circ (-i\frac{\lambda_{\tau}}{\lambda}\mathcal{K}_{nres}^{(1)})\circ S\circ \prod_{l=2}^k  (-i\frac{\lambda_{\tau}}{\lambda}\mathcal{K})\circ S
		\end{align*}
		if the first operator $\mathcal{K}$in the final product  is replaced by $\mathcal{K}_{nres}^{(2)}$ we can replicate the argument for the second situation in the case $j = 1$. If the first operator there is replaced by $\mathcal{K}_{res}^{(2)}$, we form the maximal string of operators of this form, say
		\begin{align*}
			\prod_{l=2}^{r}  (-i\frac{\lambda_{\tau}}{\lambda}\mathcal{K}_{res}^{(l)})\circ S,
		\end{align*}
		and write the resulting oscillatory phase in the output frequency in analogy to \eqref{eq:resprodxiphase1} before replicating the argument in the first situation of the case $j = 1$. We observe that the power gains in $N^{-1}$ simply result from integrating the functions $\sigma^{-N}$, which result from the weight of the norm used. 
	\end{proof}
	
	In order to complete the control of the real resonant part $\tilde{\kappa}_1$, we shall also require the fact that we can 'move' Schr\"odinger time derivatives across concatenations of the operators $-i\frac{\lambda_{\tau}}{\lambda}\mathcal{K}\circ S$. The following lemma results from a combination of the proof of Lemma~\ref{lem:concatenation1} and Lemma~\ref{lem:derivativemovesthrough1}: 
	\begin{lem}\label{lem:derivativemovesthrough2} We can write 
		\begin{align*}
			\int_0^\infty \xi^2\cdot S\big(-i\frac{\lambda_{\tau}}{\lambda}\mathcal{K}\circ S(\partial_{\sigma}G)\big)^j\rho(\xi)\,d\xi = \partial_{\tau}M_G^{(1,j)} + \tau^{-1}M_{G}^{(2,j)}, 
		\end{align*}
		where the terms on the right enjoy the bound
		\begin{align*}
			\big\|M_G^{(1,j)}\big\|_{\tau^{-N}L^2_{d\tau}} + \big\|M_G^{(2,j)}\big\|_{\tau^{-N}L^2_{d\tau}}\lesssim (\sqrt{N})^{-j}\cdot \big\|\langle \xi\partial_{\xi}\rangle G\big\|_{\tau^{-N}L^2_{d\tau}L^2_{\rho(\xi)\,d\xi}}.
		\end{align*}
		We also have the relation 
		\begin{align*}
			\partial_{\tau}\int_0^\infty \xi^2\cdot S\big(-i\frac{\lambda_{\tau}}{\lambda}\mathcal{K}\circ S(G)\big)^j\rho(\xi)\,d\xi &= \int_0^\infty \xi^2\cdot S\big(-i\frac{\lambda_{\tau}}{\lambda}\mathcal{K}\circ S(\partial_{\sigma}G)\big)^j\rho(\xi)\,d\xi\\& +N_G^{(j)},
		\end{align*}
		Finally, wave temporal frequency localization is essentially passed to the interior of the integral:
		\begin{align*}
			&Q^{(\tilde{\tau})}_{>\tau^{\frac12+}}\big(\int_0^\infty \xi^2\cdot S\big(-i\frac{\lambda_{\tau}}{\lambda}\mathcal{K}\circ S(G)\big)^j\rho(\xi)\,d\xi\big)\\& = \int_0^\infty \xi^2\cdot S\big(-i\frac{\lambda_{\tau}}{\lambda}\mathcal{K}\circ S(Q^{(\tilde{\sigma})}_{>\sigma^{\frac12+}}G)\big)^j\rho(\xi)\,d\xi + O_G^{(j)},
		\end{align*}
		where we have error bound 
		\begin{align*}
			\big\| O_G^{(j)}\big\|_{\tau^{-N-K}L^2_{d\tau}}\lesssim (\sqrt{N})^{-j}\cdot \big\|\langle \xi\partial_{\xi}\rangle G\big\|_{\tau^{-N}L^2_{d\tau}L^2_{\rho(\xi)\,d\xi}}.
		\end{align*}
		where $K\geq 1$ is arbitrary and $N\geq N(K)$ is sufficiently large. 
	\end{lem}

	When estimating the contribution of the source term $\big(\lambda^{-2}n_*^{(\tilde{\lambda}, \underline{\tilde{\alpha}})} - W^2\big) z$ in \eqref{eq:zeqn2} we shall have to resort to an analogue of the transference operator, as the factor $\big(\lambda^{-2}n_*^{(\tilde{\lambda}, \underline{\tilde{\alpha}})} - W^2\big)$ is too large in the wave  radiation regime $R\gtrsim \tau^{\frac12-\frac{1}{4\nu}}$. The following lemma shall turn out to be useful for this:
	\begin{lem}\label{lem:pseudotransferenceoperator1} 
		Let $f(R)$ be a smooth function on $(0,\infty)$ which is bounded and has symbol type behavior with respect to $R$ for $R\gtrsim 1$. Then for $\lambda\gg 1$ and setting 
		\begin{align*}
			\tilde{F}(\xi, \eta) = \langle \phi(R;\xi),\,\chi_{\lesssim\lambda}(R)f(R)\cdot \phi(R;\eta)\rangle_{L^2_{R^3\,dR}},
		\end{align*}
		we have the bounds 
		\begin{align*}
			\Big|\tilde{F}(\xi, \eta)\big|\lesssim_N \frac{\xi^{\frac12}}{\langle \xi\rangle^2}\cdot \frac{\eta^{\frac12}}{\langle \eta\rangle^2}\cdot \min\{\frac{1}{|\xi - \eta|}, \lambda\}\cdot\langle \xi - \eta\rangle^{-N}
		\end{align*}
		for arbitrary $N\geq 1$. In particular, we have the bound 
		\begin{align*}
			\Big\|\tilde{F}(\xi, \eta)\Big\|_{L^1_{\rho\,d\xi}} + \Big\|F(\xi, \eta)\Big\|_{L^1_{\rho\,d\eta}}\lesssim \log\lambda. 
		\end{align*}
		Setting 
		\[
		\tilde{F}(\xi, \eta) = G(\xi, \eta - \xi) =: G(\xi, \tilde{\eta})
		\]
		the partially differentiated functions $\partial_{\xi}^{\alpha}G(\xi, \tilde{\eta})$ have symbol behavior with respect to $\xi$ and otherwise obey the same pointwise bounds. The integral operators 
		\[
		f\longrightarrow \int_0^\infty (\xi\partial_{\xi})^{\alpha}G(\xi,\eta-\xi)f(\eta)\rho(\eta)\,d\eta,\alpha = 0, 1,
		\]
		act boundedly on $L^p_{\rho\,d\xi}$, $2\leq p<\infty$.
	\end{lem}
	\begin{proof} This is a consequence of the asymptotic expansion of the distorted Fourier basis $\phi(R;\xi)$, see subsection~\ref{subsec:basicfourier}, and integration by parts, see for example the proof of Theorem 6.1 in \cite{KST2}.  
	\end{proof}

	\subsection{Proofs of some technical lemmas}
	
	\subsubsection{Proof of Lemma~\ref{lem:wavebasicinhomstructure}} Our point of departure is the Fourier representation 
	\begin{align*}
		n(\tau, R) = \int_0^\infty \phi_{\R^4}(R;\eta)x(\tilde{\tau};\eta)\rho_{\R^4}(\eta)\,d\eta, 
	\end{align*}
	with $x(\tilde{\tau};\eta)$ given by \eqref{eq:wavepropagator}. Then we decompose the resulting double integral (over $\eta, \tilde{\sigma}$) into a number of contributions: 
	\begin{equation}\label{eq:nparts}
		n = n_0 + n_{I1} + n_{I2} + n_{II} + n_{III} + n_{IV}
	\end{equation}
	where we set 
	\begin{align*}
		&n_{I1}(\tau, R): =\\& \sum_{\log\tau\geq j\geq 0}\chi_{R\sim 2^j}\cdot \int_0^\infty \int_{\tilde{\tau}}^{\infty}\chi_{2^j\eta\gtrsim 1}\phi_{\R^4}(R;\eta)\cdot \chi_{\lambda(\tilde{\tau})\int_{\tilde{\tau}}^{\tilde{\sigma}}\lambda^{-1}(\tilde{s})\,d\tilde{s}\sim 2^j}\\&\hspace{3cm}\cdot U(\tilde{\tau}, \tilde{\sigma}, \eta)\cdot \lambda^{-2}(\tilde{\sigma})\mathcal{F}_{\R^4}(F)(\tilde{\sigma},\frac{\lambda(\tilde{\tau})}{\lambda(\tilde{\sigma})}\eta)\rho_{\R^4}(\eta)\,d\tilde{\sigma}d\eta
	\end{align*}
	while $n_{I2}(\tau, R)$ is defined analogously but by inclusion of a cutoff $\chi_{2^j\eta\lesssim 1}$, while for $n_{II}, n_{III}$, we include $\chi_{R\ll 2^j}, \chi_{R\gg 2^j}$, respectively, and finally we set 
	\begin{align*}
		n_{IV} = \chi_{R\gtrsim\tau}\cdot n. 
	\end{align*}
	Then we can formally write 
	\begin{align*}
		n_{I1}(\tau, R) = \sum_{\pm} \int_0^\infty e^{\pm iR\eta}\cdot N_{\pm}(R,\eta;\tau)\,d\eta,
	\end{align*}
	where we set 
	\begin{align*}
		&N_{\pm}(R,\eta;\tau): =  \sum_{\log\tau\geq j\geq 0}N_{\pm}^{(j)}(R,\eta;\tau)\\&= \sum_{\log\tau\geq j\geq 0}\chi_{R\sim 2^j}\cdot \int_{\tilde{\tau}}^{\infty}\frac{\chi_{2^j\eta\gtrsim 1}}{(R\eta)^{\frac32}}\cdot \sigma_{\pm}(R,\eta)\cdot\chi_{\lambda(\tilde{\tau})\int_{\tilde{\tau}}^{\tilde{\sigma}}\lambda^{-1}(\tilde{s})\,d\tilde{s}\sim 2^j}\\&\hspace{3cm}\cdot U(\tilde{\tau}, \tilde{\sigma}, \eta)\cdot \lambda^{-2}(\tilde{\sigma})\mathcal{F}_{\R^4}(F)(\tilde{\sigma},\frac{\lambda(\tilde{\tau})}{\lambda(\tilde{\sigma})}\eta)\rho_{\R^4}(\eta)\,d\tilde{\sigma}
	\end{align*}
	Further, we define 
	\[
	n_2: =  n_0 + n_{I2} + n_{II} + n_{III} + n_{IV}. 
	\]
	To prove the first inequality of the lemma, observe that inclusion of $\triangle^{-1}$ in front of $F$ allows us to gain a factor $\eta^2$. Then write formally\footnote{The function $ \sigma^{(1)}_{\pm}(R,\eta)$ has the same properties as $\sigma_{\pm}(R,\eta)$.}
	\begin{align*}
		&\frac{1}{R^{\frac12}\xi}\cdot W(R)N_{\pm}(R,\eta;\tau)\cdot R^3 =\\& \sum_{\log\tau\geq j\geq 0}\chi_{R\sim 2^j}\cdot \int_{\tilde{\tau}}^{\infty}\frac{R^{\frac12}}{\xi}\cdot \frac{\chi_{2^j\eta\gtrsim 1}}{(R\eta)^{\frac32}}\cdot \sigma^{(1)}_{\pm}(R,\eta)\cdot \chi_{\lambda(\tilde{\tau})\int_{\tilde{\tau}}^{\tilde{\sigma}}\lambda^{-1}(\tilde{s})\,d\tilde{s}\sim 2^j}\\&\hspace{3cm}\cdot \eta^2\tilde{U}(\tilde{\tau}, \tilde{\sigma}, \eta)\cdot \lambda^{-2}(\tilde{\sigma})\mathcal{F}_{\R^4}(\triangle^{-1}F)(\tilde{\sigma},\frac{\lambda(\tilde{\tau})}{\lambda(\tilde{\sigma})}\eta)\rho_{\R^4}(\eta)\,d\tilde{\sigma}
	\end{align*}
	Using one factor $R^{-1}\sim 2^{-j}$ to ensure time integrability, we claim that setting
	\begin{align*}
		\frac{\eta}{\xi}\cdot\int_0^\infty  \tilde{\sigma}(R, \xi)\cdot\chi_{R\xi\gtrsim 1}\cdot\sigma^{(1)}(R,\eta)\cdot e^{iR(\pm\xi\pm\eta)}\cdot \chi_{R\sim 2^j}\cdot\frac{2^j}{R}\,dR =: g_j(\xi, \eta),
	\end{align*}
	we have that (throughout the variables $\xi, \eta$ are restricted to $\R_+$)
	\begin{align*}
		\big\|\int_0^\infty g_j(\xi, \eta) f(\eta)\,d\eta\big\|_{L^2_{d\xi}}\lesssim \big\|f\big\|_{L^2_{d\eta}}. 
	\end{align*}
	To see this, note that in case $\xi\ll\eta$, using twofold integration by parts in the $R$ integral and the fact that $R\gtrsim \xi^{-1}$ on the support of the integrand, we can replace the factor $\frac{\eta}{\xi}$ by $\frac{\xi}{\eta}$ while replacing the integral by an analogous one. Fixing now as we may the regime $\eta\lesssim \xi$, first consider the diagonal case $\xi\sim \eta$. 
	By orthogonality, to get $L^2$-boundedness of the operator with kernel $\chi_{\xi\sim \eta}g_j(\xi, \eta)$, it suffices to show this for the localized version $\chi_{\xi\sim\eta\sim 2^k}g_j(\xi, \eta)$. This is a consequence of the easily verified $L^2$-boundedness of the operators 
	\begin{align*}
		&T_1f(R): =  \chi_{R\sim 2^j}\cdot\frac{2^j}{R}\cdot\int_0^\infty e^{\pm iR\eta}\cdot \sigma^{(1)}(R,\eta)\cdot \chi_{\eta\sim 2^k} f(\eta)\,d\eta\\
		&T_2f(\xi): = \chi_{\xi\sim 2^k}\int_0^\infty  \tilde{\sigma}(R, \xi)\cdot e^{\pm iR\xi}\cdot \chi_{R\xi\gtrsim 1}\cdot f(R)\,dR. 
	\end{align*}
	In the case $\eta\ll\xi$ we use the bound 
	\begin{align*}
		\big|\chi_{\eta\ll\xi}g_j(\xi,\eta)\big|\lesssim \chi_{\eta\ll\xi}\cdot\frac{\eta}{\xi^2}, 
	\end{align*}
	and so the desired $L^2$-boundedness follows from Schur's criterion. 
	Since 
	\begin{align*}
		&\Big\|2^{-j}\int_{\tilde{\tau}}^{\infty}\chi_{\lambda(\tilde{\tau})\int_{\tilde{\tau}}^{\tilde{\sigma}}\lambda^{-1}(\tilde{s})\,d\tilde{s}\sim 2^j}\cdot \eta \tilde{U}(\tilde{\tau}, \tilde{\sigma}, \eta)\cdot\frac{\mathcal{F}_{\R^4}(\triangle^{-1}F)(\tilde{\sigma},\frac{\lambda(\tilde{\tau})}{\lambda(\tilde{\sigma})}\eta)\rho_{\R^4}(\eta)}{\eta^{\frac32}}\,d\tilde{\sigma}\Big\|_{\tau^{-N} L^2_{d\tau} L^2_{d\eta}}\\
		&\lesssim \big\|\triangle^{-1}F\big\|_{\tau^{-N} L^2_{d\tau} L^2_{R^3\,dR}}, 
	\end{align*}
	we infer the bound 
	\begin{align*}
		&\Big\|\int_0^\infty\Phi_{\pm}(R,\xi,\eta)\cdot W(R)\cdot N_{\pm}^{(j)}(R,\eta;\tau) R^3\,dR d\eta\Big\|_{\tau^{-N+} L^2_{d\tau} L^2_{d\xi}}\\&\hspace{7cm}\lesssim 
		\big\|\triangle^{-1}\big(\lambda^{-2}F\big)\big\|_{\tau^{-N} L^2_{d\tau}L^2_{R^3\,dR}}, 
	\end{align*}
	where we recall the statement of Lemma~\ref{lem:wavebasicinhomstructure} fo the notation. The first estimate of the lemma is then obtained by summing over $j\in [0,\log\tau]$. 
	\\
	Continuing with the inequality for $n_2$, we first consider the contribution of $n_{I2}$, which is handled by using 
	\begin{align*}
		\big|n_{I2}\big|\lesssim& \sum_{0\leq j\leq\log\tau}\chi_{R\sim 2^j}\cdot R^{-2}\cdot 2^{-j}\cdot  \int_{\tilde{\tau}}^{\infty}\chi_{2^j\eta\lesssim 1}\chi_{\lambda(\tilde{\tau})\int_{\tilde{\tau}}^{\tilde{\sigma}}\lambda^{-1}(\tilde{s})\,d\tilde{s}\sim 2^j}\\&\hspace{5cm}\cdot\big\|\lambda^{-2}(\tilde{\sigma})\mathcal{F}_{\R^4}(\triangle^{-1}F)(\tilde{\sigma},\frac{\lambda(\tilde{\tau})}{\lambda(\tilde{\sigma})}\eta)\big\|_{L^2_{\rho_{\R^4}(\eta)\,d\eta}}\,d\tilde{\sigma}.
	\end{align*}
	The inequality 
	\begin{align*}
		\Big\|W \cdot n_{I2}(\tau, R)\Big\|_{\tau^{-N+} L^2_{d\tau}L^1_{R^3\,dR}}\lesssim \big\|\triangle^{-1}\big(\lambda^{-2}F\big)\big\|_{\tau^{-N} L^2_{d\tau}L^2_{R^3\,dR}}\\
	\end{align*}
	results. 
	In order to deal with the terms $n_{II}, n_{III}$, we perform integration by parts with respect to $\eta$, as the oscillatory factors $U(\tilde{\tau},\tilde{\sigma}, \eta), \phi_{\R^4}(R;\eta)$ are out of phase. Schematically we arrive at either the expression 
	\begin{align*}
		&A: = \sum_{\log\tau\geq j\geq 0} \int_0^\infty \int_{\tilde{\tau}}^{\infty}2^{-j}\phi(R;\eta)\cdot \chi_{\lambda(\tilde{\tau})\int_{\tilde{\tau}}^{\tilde{\sigma}}\lambda^{-1}(\tilde{s})\,d\tilde{s}\sim 2^j}\\&\hspace{3cm}\cdot U(\tilde{\tau}, \tilde{\sigma}, \eta)\cdot \lambda^{-2}(\tilde{\sigma})\partial_{\eta}\mathcal{F}_{\R^4}(F)(\tilde{\sigma},\frac{\lambda(\tilde{\tau})}{\lambda(\tilde{\sigma})}\eta)\rho_{\R^4}(\eta)\,d\tilde{\sigma}d\eta
	\end{align*}
	where the extra $2^{-j}$ and $\partial_{\eta}$ come from the integration by parts, or else the expression 
	\begin{align*}
		&B: = \sum_{\log\tau\geq j\geq 0} \int_0^\infty \int_{\tilde{\tau}}^{\infty}2^{-j}\phi(R;\eta)\cdot \chi_{\lambda(\tilde{\tau})\int_{\tilde{\tau}}^{\tilde{\sigma}}\lambda^{-1}(\tilde{s})\,d\tilde{s}\sim 2^j}\\&\hspace{3cm}\cdot \eta^{-1}U(\tilde{\tau}, \tilde{\sigma}, \eta)\cdot \lambda^{-2}(\tilde{\sigma})\mathcal{F}_{\R^4}(F)(\tilde{\sigma},\frac{\lambda(\tilde{\tau})}{\lambda(\tilde{\sigma})}\eta)\rho_{\R^4}(\eta)\,d\tilde{\sigma}d\eta.
	\end{align*}
	The contribution of $A$ is then handled via 
	\begin{align*}
		\big\|A\big\|_{\tau^{-N+} L^2_{d\tau}L^2_{R^3\,dR}}\lesssim \big\|\triangle^{-1}\nabla\big(\langle R\rangle\cdot\lambda^{-2}F\big)\big\|_{\tau^{-N} L^2_{d\tau}L^2_{R^3\,dR}},
	\end{align*}
	leading to 
	\begin{align*}
		\Big\|W \cdot A(\tau, R)\Big\|_{\tau^{-N+} L^2_{d\tau}L^1_{R^3\,dR}}\lesssim \big\|\triangle^{-1}\nabla\big(\langle R\rangle\cdot\lambda^{-2}F\big)\big\|_{\tau^{-N} L^2_{d\tau}L^2_{R^3\,dR}}.
	\end{align*}
	
	The estimate for the contribution of $B$ is similar. It remains to estimate the contribution of the remaining term $n_{IV}$. Here we again perform integration by parts with respect to $\eta$. Note that on account of 
	\begin{align*}
		\lambda(\tilde{\tau})\cdot\int_{\tilde{\tau}}^{\tilde{\sigma}}\lambda^{-1}(s)\,ds\lesssim \tilde{\tau}\ll\tau, 
	\end{align*}
	the oscillating terms $U(\tilde{\tau},\tilde{\sigma}, \eta), \phi_{\R^4}(R;\eta)$ will be out of phase, and integration by parts gains a factor $\tau^{-1}\ll\tilde{\tau}^{-1}$. Proceeding as for the terms $n_{II}, n_{III}$, we then obtain a much stronger bound 
	\begin{align*}
		\Big\|W \cdot n_{IV}(\tau, R)\Big\|_{\tau^{-N-} L^2_{d\tau}L^1_{R^3\,dR}}&\lesssim  \big\|\triangle^{-1}\nabla\big(\langle R\rangle\cdot\lambda^{-2}F\big)\big\|_{\tau^{-N} L^2_{d\tau}L^2_{R^3\,dR}}\\
		& +  \big\|\triangle^{-1}\big(\lambda^{-2}F\big)\big\|_{\tau^{-N} L^2_{d\tau}L^2_{R^3\,dR}}\\
	\end{align*}
	The contribution of $n_0$ is straightforward to bound since $R\lesssim 1$ on its support
	The last statement of the lemma follows since the factors $\langle R\rangle$ are due to partial integration with respect to $\eta$, in turn required to force integrability with respect to wave time $\tilde{\sigma}$.

	\subsubsection{Completion of the proof of Corollary~\ref{cor:yzW}} We need to prove the second estimate of the corollary, which is done by using the asymptotic structure of $\phi(R;\xi)$ given in subsection~\ref{subsec:basicfourier}. Note that if we include an extra smooth cutoff $\chi_{R\xi\lesssim 1}$ in front of $\phi(R;\xi)$, the operator $\langle \xi\partial_{\xi}\rangle^{1+\delta_0}$ can be 'absorbed' by $\phi(R;\xi)$ since $(\xi\partial_{\xi})^k\phi(R;\xi)$ has a similar expansion as $\phi(R;\xi)$, and so the desired bound follows simply from 
	\begin{align*}
		&\Big\|\langle \lambda^{-2}\Box^{-1}(F)\cdot W, \phi(R;\xi)\rangle_{L^2_{R^3\,dR}}\Big\|_{\tau^{-N}L^2_{d\tau}L^\infty_{d\xi}}\\&\lesssim \big\|\lambda^{-2}\Box^{-1}(F)\big\|_{\tau^{-N}L^2_{d\tau}(\langle R\rangle^{1+\delta_0}L^2_{R^3\,dR}+L^4_{R^3\,dR})}\cdot \Big\|\big\|W\cdot \phi(R;\xi)\big\|_{\langle R\rangle^{-(1+\delta_0)}L^2_{R^3\,dR}\cap L^{\frac43}_{R^3\,dR}}\Big\|_{L^\infty_{d\xi}}\\
		&\lesssim  \big\|\lambda^{-2}\Box^{-1}(F)\big\|_{\tau^{-N}L^2_{d\tau}(\langle R\rangle^{1+\delta_0}L^2_{R^3\,dR}+L^4_{R^3\,dR})},
	\end{align*}
	where the final expression can be bounded by $\lesssim \big\|z\big\|_{S}$ by the first part of the proof. 
	If instead we include a smooth localizer $\chi_{R\xi\gtrsim 1}$, the same argument also gives the desired bound for the contribution when $R\xi\gtrsim 1$, provided we omit the operator $\langle \xi\partial_{\xi}\rangle^{1+\delta_0}$. Henceforth we shall apply $(\xi\partial_{\xi})^{1+\delta_0}$. Arguing as for the proof of Lemma~\ref{lem:specialF1} we see that
	\begin{align*}
		&\Big\|(\xi\partial_{\xi})^{1+\delta_0}\langle \lambda^{-2}\Box^{-1}(\chi_{R\gtrsim \tau}F)\cdot W, \phi(R;\xi)\rangle_{L^2_{R^3\,dR}}\Big\|_{\tau^{-N}L^2_{d\tau}L^\infty_{d\xi}}\\&\lesssim \big\|\tilde{\tau}\lambda^{-2}\nabla^{-1}\big(\chi_{R\gtrsim \tau}F\big)\big\|_{\tau^{-N}L^2_{d\tau}L^2_{R^4\,dR}}\cdot \big\|\langle R\rangle^{1+\delta_0}W(R)\cdot \phi(R;\xi)\big\|_{L^2_{R^3\,dR}L^\infty_{d\xi}}\\
		&\lesssim \big\|z\big\|_{S}. 
	\end{align*}
	The preceding argument also applies if we replace $F$ by $\lambda^2\chi_{R\lesssim \tau}W\cdot \triangle z$, so from now on we shall replace $F$ by $F' = \triangle W\bar{z} + 2\nabla W\cdot \nabla \bar{z}$. 
	Further, arguing as in the first part of the proof of the corollary, the case when we include a cutoff $\chi_{R\gtrsim \tau^{100}}$ in front of $W$ can be handled using straightforward integration by parts with respect to the frequency in the wave propagator, so we shall include a cutoff $\chi_{R\lesssim \tau^{100}}$ in front of $W$ and restrict $F'$ to $\chi_{R\lesssim \tau}F'$. 
	Then write $\chi_{R\xi\gtrsim 1}\phi(R;\xi) =\sum_{\pm}a_{\pm}(\xi)\frac{e^{\pm iR\xi}}{R^{\frac32}}\cdot\sigma_{\pm}(R; R\xi^{\frac12})$ where the factors $\sigma_{\pm}(R;q)$ are smooth, bounded, and have symbol behavior with respect to their arguments, and $|a_{\pm}(\xi)|\sim \xi^{\frac12}, \xi\ll 1$, $|a_{\pm}(\xi)|\sim \xi^{-\frac32},\,\xi\gg 1$, and is smooth with symbol behavior on $\R_+$. In the following we treat the more difficult case of small frequencies $\xi\lesssim 1$. Further use \eqref{eq:wavepropagator} together with \eqref{eq:nflatfourierrepresent} to expand out the wave propagator $ \lambda^{-2}\Box^{-1}F$. Calling $\eta$ the frequency in the Fourier representation of $\lambda^{-2}\Box^{-1}F$, we can gain $\min\{\frac{\eta}{\xi}, \frac{\xi}{\eta}\}^M + R^{-M}$ by integrating by parts with respect to $R$. This allows us to reduce to the situation $\xi\sim \eta$. Combining the oscillatory $R$-dependent phases in the Fourier representation of $\lambda^{-2}\Box^{-1}F$ and in $\phi(R;\xi)$ and performing integration by parts with respect to $R$ allows us to gain $R^{-\frac12+}$ at the expense of $(\xi-\eta)^{-\frac12+}$. Then we trade the factor $\xi$ from $(\xi\partial_{\xi})^{1+\delta_0}$ to abolish the $\eta^{-1}$ in  \eqref{eq:wavepropagator} (where the frequency is now $\eta$) and we use an extra factor $\eta^{\frac12}$ from $|a_{\pm}(\xi)|\sim \xi^{\frac12}\sim \eta^{\frac12}$ for the estimate 
	\begin{align*}
		&\Big\|\eta^{\frac12}\cdot \mathcal{F}_{\R^4}\big(\chi_{R\lesssim\tau}(\triangle W\bar{z} + 2\nabla W\cdot \nabla \bar{z})\big)\Big\|_{\tau^{-N+}L^2_{d\tau}L^\infty_{d\eta}}\\&\lesssim \Big\|R^{-\frac12}\cdot \chi_{R\lesssim\tau}(\triangle W\bar{z}+ 2\nabla W\cdot \nabla \bar{z})\Big\|_{\tau^{-N+}L^2_{d\tau}L^1_{R^3\,dR}}\lesssim \big\|z\big\|_{S},
	\end{align*}
	which in turn implies the following 
	\begin{align*}
		\Big\|\eta^{-1}\cdot(\xi - \eta)^{-\frac12+}\eta^{\frac12}\cdot \mathcal{F}_{\R^4}\big(\chi_{R\lesssim\tau}(\triangle W\bar{z} + 2\nabla W\cdot \nabla \bar{z})\big)\Big\|_{\tau^{-N+}L^2_{d\tau}L^2_{\eta^3 d\eta}(\eta\lesssim1)}\lesssim \big\|z\big\|_{S}.
	\end{align*}
	Using integration by parts with respect to the frequency $\eta$ in the Fourier representation of $\lambda^{-2}\Box^{-1}F$ or a factor $R^{-1-\delta_0}$ to compensate for the time integration, as in the proof of Lemma~\ref{lem:wavebasicinhom}, the desired estimate follows finally from the bound (recall the assumption $\xi\lesssim 1$)
	\begin{align*}
		\big\|\chi_{R\lesssim \tau^{100}}R^{-\frac12+}\cdot\frac{\phi(R;\xi)}{\xi^{\frac12}}\big\|_{L^2_{R^3\,dR}}\lesssim \tau^{0+}. 
	\end{align*}

	\subsubsection{Proof of Lemma~\ref{lem:yzWnonosc}} Using interpolation it suffices to treat the case of integral $\delta_1$ and the symbol behavior of $\psi(R;\xi)$ allows us to suppress $\langle \xi\partial_{\xi}\rangle^{1+\delta_1}$. Write 
	\begin{align*}
		\Box^{-1}F = \int_0^\infty \int_{\tilde{\tau}}^\infty \phi_{\R^4}(R;\eta)U(\tilde{\tau}, \tilde{\sigma}, \eta)\cdot \lambda^{-2}(\tilde{\sigma})\mathcal{F}_{\R^4}(F)(\tilde{\sigma},\frac{\lambda(\tilde{\tau})}{\lambda(\tilde{\sigma})}\eta)\rho_{\R^4}(\eta)\,d\tilde{\sigma}d\eta.
	\end{align*}
	where $U(\tilde{\tau}, \tilde{\sigma}, \eta)$ is as in the proof of Lemma~\ref{lem:wavebasicinhom}. From the definition of $\|\cdot\|_{S}$ (see \eqref{eq:Snormdefi}), we get 
	\begin{align*}
		A: = \big\|\eta^{-(2-)}\mathcal{F}_{\R^4}(F)(\sigma, \eta)\big\|_{\sigma^{-N}L^2_{d\sigma}L^2_{\rho_{\R^4}d\eta}} + \big\|\mathcal{F}_{\R^4}(F)(\sigma, \eta)\big\|_{\sigma^{-N}L^2_{d\sigma}L^2_{\rho_{\R^4}d\eta}} \lesssim \big\|z\big\|_{S}. 
	\end{align*}
	We first consider the case $\big|\tilde{\sigma} - \tilde{\tau}\big| + R\lesssim \eta^{-1}$. Calling the corresponding term $\big(\Box^{-1}F\big)_1$, changing the integration order we can write 
	\begin{align*}
		\langle \big(\Box^{-1}F\big)_1\cdot W, \psi(R;\xi)\rangle_{L^2_{R^3\,dR}} = \int_0^\infty\int_{\tilde{\tau}}^\infty U_1(\tilde{\tau}, \tilde{\sigma}, \eta)\cdot \mathcal{F}_{\R^4}(\lambda^{-2}F)\cdot g(\eta,\xi)\rho_{\R^4}(\eta)\,d\tilde{\sigma}d\eta,
	\end{align*}
	where we set $g(\eta,\xi) = \int_0^\infty \chi_{R\eta\lesssim 1}\psi(R;\xi)\cdot W\cdot  \phi_{\R^4}(R;\eta) R^3\,dR$ and $U_1 = \chi_{|\tilde{\sigma} - \tilde{\tau}|<\eta^{-1}}U$. Hence $\big|g(\eta,\xi)\big|\lesssim \langle\xi\rangle^{-2-}\min\{\eta^{-2}, \eta^{-4}\}$. Using $\eta^{1-}\lesssim |\tilde{\sigma} - \tilde{\tau}|^{-1+}$, we infer from the Cauchy-Schwarz inequality and Schur's criterion that 
	\begin{align*}
		\big\|\langle\xi\rangle^{2+}\langle \big(\Box^{-1}F\big)_1\cdot W, \psi(R;\xi)\rangle_{L^2_{R^3\,dR}}\big\|_{\tau^{-N+}L^2_{d\tau}L^\infty_{\rho(\xi)\,d\xi}}\lesssim A \lesssim \big\|z\big\|_{S}. 
	\end{align*}
	This bound suffices to establish the bound of the lemma for this contribution due to the asymptotics of $\rho(\xi)$.
	In the case $\big|\tilde{\sigma} - \tilde{\tau}\big| \sim R\gg\eta^{-1}$, we use $\big|g(\eta,\xi)\big|\lesssim \big|\tilde{\sigma} - \tilde{\tau}\big|^{-(1-)}\cdot \eta^{-3+}$, and one concludes as before by taking advantage of Cauchy-Schwarz and Schur's criterion. The remaining situations $\big|\tilde{\sigma} - \tilde{\tau}\big| \gg R, R\gg \big|\tilde{\sigma} - \tilde{\tau}\big|$ are handled by integration by parts, we omit the similar details.

	\subsubsection{Proof completion of Lemma~\ref{lem: tildelambdaeqnsourcebounds1}} In light of the definition \eqref{eq:Xdef}, the following steps conclude the proof: combining Lemma~\ref{lem:Xtildelambdaperturbterms2} with Lemmas~\ref{lem:K_frefined}, ~\ref{lem:tildeKfcontrol} gives the desired bound for the contributions of the terms  
	\begin{align*}
		Q^{(\tilde{\tau})}_{<\tau^{\frac12+}}\big(\lambda^{-2}y_z\cdot (\tilde{u}_*^{(\tilde{\lambda}, \tilde{\alpha})}-W)\big),\,Q^{(\tilde{\tau})}_{<\tau^{\frac12+}}\big(\lambda^{-2}(y - y_z)\cdot \tilde{u}_*^{(\tilde{\lambda}, \tilde{\alpha})} -  \lambda^{-2}y^{\text{mod}}_{\tilde{\lambda}}\cdot W\big)
	\end{align*}
	to the integral 
	\begin{align*}
		\Im \int_{\tau}^\infty\int_0^\infty \xi^2S(\tau, \sigma;\xi)\cdot \tilde{X}^{(\tilde{\lambda})}(\sigma, 0)\rho_1(\xi)\,d\xi d\sigma.
	\end{align*}
	Dealing with the contribution of the remaining term 
	\begin{align*}
		\mathcal{F}\Big(Q^{(\tilde{\tau})}_{<\tau^{\frac12+}}\big((\lambda^{-2}n_*^{(\tilde{\lambda}, \tilde{\alpha})} - W^2)z\big)\Big),
	\end{align*}
	is accomplished by means of Lemma~\ref{lem:Xtildelambdafinaltermcrudebound} in conjunction with Lemma~\ref{lem:K_frefined} as well as Lemma~\ref{lem:tildeKfcontrol}. More precisely, this argument furnishes the bound for this contribution to $Z$ without the operator $\langle\partial_{\tilde{\tau}}^2\rangle$. To also allow for such derivatives, note that if both of these fall on the factor $z$, we can use 
	\begin{align*}
		\partial_{\tilde{\tau}}^2\big(Q_{<\tau^{\frac12+}}^{(\tilde{\tau})}z\big) = \partial_{\tilde{\tau}}\big(Q_{<\tau^{\frac12+}}^{(\tilde{\tau})}(\frac{\partial\tau}{\partial\tilde{\tau}}\cdot\partial_{\tau}z)\big).
	\end{align*}
	as well as the fact that by definition $\big\|\partial_{\tau}z\big\|_{L^2_{R^3\,dR}}\lesssim \big\|z\big\|_{S}$. 
	The operator $\partial_{\tilde{\tau}}Q_{<\tau^{\frac12+}}^{(\tilde{\tau})}$ 'costs' $\tau^{\frac12+}$, while $\frac{\partial\tau}{\partial\tilde{\tau}}\sim \tau^{\frac12+\frac{1}{4\nu}}$, and one can then use the direct bound 
	\begin{align*}
		\Big\|\mathcal{F}\Big(Q^{(\tilde{\tau})}_{<\tau^{\frac12+}}\big((\lambda^{-2}n_*^{(\tilde{\lambda}, \tilde{\alpha})} - W^2)\partial_{\tilde{\tau}}^2z\big)\Big)(\tau, 0)\Big\|_{\tau^{-N-\frac{1}{4\nu}+}L^2_{d\tau}}\lesssim \big\|z\big\|_{S}. 
	\end{align*}
	The cases when fewer derivatives $\partial_{\tilde{\tau}}$ fall on $z$ are handled similarly, also taking advantage of Lemma~\ref{lem:approxsolasymptotics3}.

	\subsubsection{Proof completion of Lemma~\ref{lem:Lsmalltildelambda}} Recalling the definition \eqref{eq:Xdef} of $X^{(\tilde{\lambda})}(\tau,\xi)$, we can control the contributions of 
	\begin{align*}
		Q^{(\tilde{\tau})}_{<\tau^{\frac12+}}\big(\lambda^{-2}y_z\cdot (\tilde{u}_*^{(\tilde{\lambda}, \tilde{\alpha})}-W)\big),\,Q^{(\tilde{\tau})}_{<\tau^{\frac12+}}\big(\lambda^{-2}(y - y_z)\cdot \tilde{u}_*^{(\tilde{\lambda}, \tilde{\alpha})} -  \lambda^{-2}y^{\text{mod}}_{\tilde{\lambda}}\cdot W\big)
	\end{align*}
	to $L_{small}^{(\tilde{\lambda})}$ by combining Lemmas~\ref{lem:Xtildelambdaperturbterms}, ~\ref{lem:Xtildelambdaperturbterms2} with Lemmas~\ref{lem:K_frefined}, ~\ref{lem:tildeKfcontrol}. In order to control the contribution of the more delicate term 
	\begin{align*}
		\mathcal{F}\Big(Q^{(\tilde{\tau})}_{<\tau^{\frac12+}}\big((\lambda^{-2}n_*^{(\tilde{\lambda}, \tilde{\alpha})} - W^2)z\big)\Big),
	\end{align*}
	we proceed as in the proof in the next sub-subsection. 
	
	\subsubsection{Proof outline for Lemma~\ref{lem:LsmallcalKtildelambda}} We shall treat the contributions of the various terms constituting $X^{(\tilde{\lambda})}(\tau,\xi)$ (recall \eqref{eq:Xdef}) to $L^{(\tilde{\lambda})}_{\mathcal{K},\text{small}}$, the latter as in \eqref{eq:LtildelambdacalK}.
	\\
	{\it{Contribution of leading term $ -Q^{(\tilde{\tau})}_{<\tau^{\frac12+}}\big(\lambda^{-2}y_z\cdot W\big)$.}}
	Recalling \eqref{eq:LtildelambdacalK} as well as the definition \eqref{eq:ScalK2}, it suffices to combine Lemma~\ref{lem:concatenation1} with Lemma~\ref{lem:yzWbound1}. Observe that division by the frequency $\xi^{-2}$ frees an extra factor $\xi^2$ which can be used to implement integration by parts with respect to time for the propagator $S$. This in turn either results in a boundary term or a term where an additonal time derivative falls on the source term, and Lemma~\ref{lem:yzWbound1} then furnishes improved temporal decay. In particular, this contribution can be placed into $\tau^{-N-\frac12-\frac{1}{2\nu}+}L^2_{d\tau}$, which is much better than what is needed. 
	\\
	{\it{Contributions of remaining terms}}. The contribution of the terms 
	\begin{align*}
		Q^{(\tilde{\tau})}_{<\tau^{\frac12+}}\big(\lambda^{-2}y_z\cdot (\tilde{u}_*^{(\tilde{\lambda}, \tilde{\alpha})}-W)\big),\,Q^{(\tilde{\tau})}_{<\tau^{\frac12+}}\big(\lambda^{-2}(y - y_z)\cdot \tilde{u}_*^{(\tilde{\lambda}, \tilde{\alpha})} -  \lambda^{-2}y^{\text{mod}}_{\tilde{\lambda}}\cdot W\big)
	\end{align*}
	is handled by combining Lemma~\ref{lem:Xtildelambdaperturbterms} with Lemma~\ref{lem:concatenation1}. The contribution of the remaining term
	\[
	Q^{(\tilde{\tau})}_{<\tau^{\frac12+}}\big((\lambda^{-2}n_*^{(\tilde{\lambda}, \tilde{\alpha})} - W^2)z\big)
	\]
	is handled by re-iterating the equation for $z$ one more time. Precisely, we express 
	\begin{align*}
		\mathcal{F}\Big(Q^{(\tilde{\tau})}_{<\tau^{\frac12+}}\big((\lambda^{-2}n_*^{(\tilde{\lambda}, \tilde{\alpha})} - W^2)z\big)\Big)
	\end{align*}
	by means of Lemma~\ref{lem:pseudotransferenceoperator2}, where we express $\mathcal{F}(z)$ by means of \eqref{eq:formalexpansion}, with $E$ given by \eqref{eq:recallE}. Then we take advantage of Lemma~\ref{lem:concatenation2} in the case of the presence of at least one $\mathcal{K}_*$, and 
	$G = \mathcal{F}(E)$, and we repeat application of Lemma~\ref{lem:pseudotransferenceoperator2} together with \eqref{eq:formalexpansion} for the third source term of $E$. For the remaining source terms of $E$, we control their contribution by means of Lemma~\ref{lem:Xtildelambdaperturbterms}, Lemma~\ref{lem:basicboundsfore_1modandnonlinearterms}, as well as Lemma~\ref{lem:basicboundsfore_1modalphaterm}.
	
	\subsubsection{Completion of the proof of Lemma~\ref{lem:e1modgooderrors}} Keeping in mind \eqref{eq:E1mod}, we need to bound the contributions of the third to sixth terms there, in addition to the error term $O(|\tilde{\alpha}|^2)$, noticing that each of these terms contributes to both the double integral as well as the final term forming $E(\tau)$. Now the third to fifth terms contribute the {\it{real terms}} (where we approximate $\psi^{(\tilde{\lambda})}$ by $\lambda\cdot W$)
	\begin{align*}
		\lambda^{-2}\cdot\partial_{t}\big(\chi_1\big)\tilde{\alpha}\cdot W,\,-\chi_1\cdot\partial_{\tau}\tilde{\alpha}\cdot W
	\end{align*}
	to $e_1^{mod}$; these terms of course only contribute to the double integral in $E(\tau)$, and not the boundary term at $R = 0$. By definition of $E(\tau)$ these terms get further localized by applying $Q^{(\tilde{\tau})}_{<\tau^{\delta}}$. The contributions of these terms to the double integral in $E(\tau)$ cancel against two terms arising upon applying integration by parts with respect to $\sigma$ to the Schr\"odinger propagator of two purely imaginary terms treated further below. 
	\\
	As for the terms 
	\begin{align*}
		\lambda^{-2}\cdot\partial_{t}\big(\chi_1\big)\tilde{\alpha}\cdot \big(W - \lambda^{-1}\psi^{(\tilde{\lambda})}\big),\,-\chi_1\cdot\partial_{\tau}\tilde{\alpha}\cdot \big(W - \lambda^{-1}\psi^{(\tilde{\lambda})}\big)
	\end{align*}
	their contribution to $E(\tau)$ can be controlled by cruder estimates taking advantage of Lemma~\ref{lem:approxsolasymptotics1}, which allows us to infer the bounds 
	\begin{align*}
		&\Big\|\langle\xi\partial_{\xi}\rangle^{1+\delta_0}\mathcal{F}\Big(\lambda^{-2}\cdot\partial_{t}\big(\chi_1\big)\tilde{\alpha}\cdot \big(W - \lambda^{-1}\psi^{(\tilde{\lambda})}\big)\Big)\Big\|_{\tau^{-N-}L^2_{d\tau}L^\infty_{d\xi}}\lesssim \big\|\frac{\tilde{\alpha}}{\tau}\big\|_{\tau^{-N}L^2_{d\tau}},\\
		&\Big\|\langle\xi\partial_{\xi}\rangle^{1+\delta_0}\mathcal{F}\Big(\chi_1\cdot\partial_{\tau}\tilde{\alpha}\cdot \big(W - \lambda^{-1}\psi^{(\tilde{\lambda})}\big)\Big)\Big\|_{\tau^{-N-}L^2_{d\tau}L^\infty_{d\xi}}\lesssim \big\|\tilde{\alpha}_{\tau}\big\|_{\tau^{-N}L^2_{d\tau}}.
	\end{align*} 
	Then the desired bound for these contributions to $E(\tau)$ is a consequence of Lemmas~\ref{lem:K_frefined}, ~\ref{lem:tildeKfcontrol}.
	\\
	We next consider the following two {\it{purely imaginary terms}} which are contributed by the third and fourth term in \eqref{eq:E1mod}, namely
	\begin{align*}
		i\tilde{\alpha}\big(\triangle_R(\chi_1)\cdot W + 2\partial_R(\chi_1)\cdot\partial_RW\big) = -i\tilde{\alpha}\cdot\mathcal{L}\big(\chi_1\cdot W),
	\end{align*}
	where we have again replaced $ \lambda^{-1}\psi^{(\tilde{\lambda})}$ by $W$. Thus the imaginary part of the Schr\"odinger propagator \eqref{prop:linpropagator} applied to these terms is given by the kernel 
	\[
	\sin\big(\lambda^2(\tau)\xi^2\int_{\sigma}^{\tau}\lambda^{-2}(s)\,ds\big),
	\]
	and the corresponding contribution to $\Im\mathcal{L}z|_{R = 0}$ is given by 
	\begin{align*}
		&-\int_0^\infty\int_{\tau}^\infty \xi^2\frac{\lambda^2(\tau)}{\lambda^2(\sigma)}\cdot\sin\big(\lambda^2(\tau)\xi^2\int_{\sigma}^{\tau}\lambda^{-2}(s)\,ds\big)\\&\hspace{4cm}\cdot\tilde{\alpha}(\sigma)\cdot\mathcal{F}\big(\mathcal{L}(\chi_1\cdot W)\big)(\sigma, \frac{\lambda(\tau)}{\lambda(\sigma)}\xi)\rho(\xi)\,d\sigma d\xi.
	\end{align*}
	Performing integration by parts with respect to $\sigma$ does not produce a boundary term at $\sigma = \tau$ since $\mathcal{L}(\chi_1\cdot W)|_{R = 0} = 0$. We then arrive at the double integral 
	\begin{align*}
		&-\int_0^\infty\int_{\tau}^\infty\cos\big(\lambda^2(\tau)\xi^2\int_{\sigma}^{\tau}\lambda^{-2}(s)\,ds\big)\\&\hspace{4cm}\cdot\partial_{\sigma}\Big(\tilde{\alpha}(\sigma)\cdot\mathcal{F}\big(\mathcal{L}(\chi_1\cdot W)\big)(\sigma, \frac{\lambda(\tau)}{\lambda(\sigma)}\xi)\Big)\rho(\xi)\,d\sigma d\xi.
	\end{align*}
	When $\partial_{\sigma}$ falls on either $\tilde{\alpha}(\sigma)$ or $\chi_1$ we obtain terms which cancel exactly against corresponding double integrals contributed from the real part of the third and fifth terms, as asserted earlier. This means we can replace the preceding expression up to a constant by 
	\begin{align*}
		&-\int_0^\infty\int_{\tau}^\infty\cos\big(\lambda^2(\tau)\xi^2\int_{\sigma}^{\tau}\lambda^{-2}(s)\,ds\big)\\&\hspace{4cm}\cdot\frac{\tilde{\alpha}(\sigma)}{\sigma}\cdot(\xi\partial_{\xi})\mathcal{F}\big(\mathcal{L}(\chi_1\cdot W)\big)(\sigma, \frac{\lambda(\tau)}{\lambda(\sigma)}\xi)\rho(\xi)\,d\sigma d\xi.
	\end{align*}
	The operator $\mathcal{L}$ gains an additional factor $\frac{\lambda^2(\tau)}{\lambda^2(\sigma)}\cdot\xi^2$ which allows us to perform another integration by parts with respect to $\sigma$, again without generating a boundary term at $\sigma = \tau$, and which allows us to replace the previous double integral by 
	\begin{align*}
		&-\int_0^\infty\int_{\tau}^\infty\sin\big(\lambda^2(\tau)\xi^2\int_{\sigma}^{\tau}\lambda^{-2}(s)\,ds\big)\\&\hspace{3.5cm}\cdot\frac{\partial_{\sigma}^{\kappa}\tilde{\alpha}(\sigma)}{\sigma^{2-\kappa}}\cdot(\xi\partial_{\xi})^{\kappa_1}\mathcal{F}\big(\chi_1\cdot W\big)(\sigma, \frac{\lambda(\tau)}{\lambda(\sigma)}\xi)\rho(\xi)\,d\sigma d\xi,
	\end{align*}
	where $\kappa\in \{0, 1\}$ and $\kappa_1\in \{0, 1, 2\}$. Our assumption \eqref{eq:chi1vanishingcond} implies that we can write 
	\begin{align*}
		\mathcal{F}\big(\chi_1\cdot W\big)(\sigma,\xi) = \langle \chi_1\cdot W, \phi(R;\xi) - W(R)\rangle_{L^2_{R^3\,dR}}, 
	\end{align*}
	and the asymptotics in subsection~\ref{subsec:basicfourier} together with our choice for $\chi_1$ imply that inclusion of an extra cutoff $\chi_{\xi\lesssim \sigma^{-\frac12+}}$ allows us to bound the double integral by 
	\begin{align*}
		\Big\|\cdot\Big\|_{\tau^{-N-}L^2_{d\tau}}\lesssim \big\|\partial_{\tau}\tilde{\alpha}\big\|_{\tau^{-N}L^2_{d\tau}} + \big\|\frac{\tilde{\alpha}(\tau)}{\tau}\big\|_{\tau^{-N}L^2_{d\tau}}
	\end{align*}
	It remains to deal with the double integral with an extra cutoff $\chi_{\xi\gtrsim \sigma^{-\frac12+}}$. Dividing into the regions $\sigma - \tau\lesssim \xi^{-2}$, $\sigma - \tau\gtrsim \xi^{-2}$ and performing integration by parts with respect to $\xi$ in the latter, we see that we can restrict to $\sigma - \tau\lesssim \sigma^{1-}$ for all intents and purposes. Thus the $\sigma$ integral 'costs $\tau^{1-}$ but we have gained $\tau^{-1}$ before, resulting in a similar bound 
	\begin{align*}
		\Big\|\cdot\Big\|_{\tau^{-N-}L^2_{d\tau}}\lesssim \big\|\partial_{\tau}\tilde{\alpha}\big\|_{\tau^{-N}L^2_{d\tau}} + \big\|\frac{\tilde{\alpha}(\tau)}{\tau}\big\|_{\tau^{-N}L^2_{d\tau}}
	\end{align*}
	for the remaining double integral restricted to $\xi\gtrsim \sigma^{-\frac12+}$. 
	\\
	In the preceding we replaced $ \lambda^{-1}\psi^{(\tilde{\lambda})}$ by $W$, so we still need to account for the contribution of the error term 
	\begin{align*}
		&i\tilde{\alpha}\big(\triangle_R(\chi_1)\cdot (W - \lambda^{-1}\psi^{(\tilde{\lambda})}) + 2\partial_R(\chi_1)\cdot\partial_R(W - \lambda^{-1}\psi^{(\tilde{\lambda})})\big)\\
		& =- i\tilde{\alpha}\mathcal{L}\big(\chi_1\cdot (W - \lambda^{-1}\psi^{(\tilde{\lambda})})\big) - i\tilde{\alpha}\chi_1\cdot \partial_{RR}(W - \lambda^{-1}\psi^{(\tilde{\lambda})})\\
		& - i\tilde{\alpha}\chi_1\cdot W^2\cdot (W - \lambda^{-1}\psi^{(\tilde{\lambda})})
	\end{align*}
	Here we take advantage of the final part of Lemma~\ref{lem:approxsolasymptotics1}. For the last term on the right, using that the principal part of $W - \lambda^{-1}\psi^{(\tilde{\lambda})}$ is purely imaginary, whence the corresponding contribution is {\it{real valued}}, we can use Lemma~\ref{lem:K_frefined} to show that its contribution to $E(\tau)$ satisfies 
	\begin{align*}
		\Big\|\cdot\Big\|_{\log^{-3}\tau\cdot\tau^{-N}L^2_{d\tau}}\lesssim \big\|\frac{\tilde{\alpha}}{\tau}\big\|_{\log^{-1}\tau\cdot\tau^{-N}L^2_{d\tau}}. 
	\end{align*}
	For the first term on the right, again using that the principal part of $W - \lambda^{-1}\psi^{(\tilde{\lambda})}$ is purely imaginary, we use $\mathcal{L}$ to gain a factor $\frac{\lambda^2(\tau)}{\lambda^2(\sigma)}\cdot \xi^2$ and perform integration by parts with respect to $\sigma$ in the double integral contributing to $E(\tau)$ (which does not generate a boundary term at $\sigma = \tau$), arriving at the schematically written double integrals 
	\begin{align*}
		&\int_0^\infty\int_{\tau}^\infty\frac{\lambda^2(\tau)}{\lambda^2(\sigma)}\cdot \xi^2\sin\big(\lambda^2(\tau)\xi^2\int_{\sigma}^{\tau}\lambda^{-2}(s)\,ds\big)
		\\&\hspace{3cm}\cdot\frac{\partial_{\sigma}^{\kappa}\tilde{\alpha}}{\sigma^{1-\kappa}}\cdot \mathcal{F}\big(\chi_1\cdot \frac{\log R}{\sigma}\big)(\sigma, \frac{\lambda(\tau)}{\lambda(\sigma)}\xi)\rho(\xi)\,d\sigma d\xi,\,\kappa\in \{0, 1\}.
	\end{align*}
	Using Lemma~\ref{lem:K_frefined} as well as the definition of $\chi_1$, we easily find that the corresponding contribution is in $\tau^{-N-}L^2_{d\tau}$, bounded in terms of $\lesssim \big\|\frac{\partial_{\sigma}^{\kappa}\tilde{\alpha}}{\sigma^{1-\kappa}}\big\|_{\tau^{-N}L^2_{d\tau}}$, $\kappa\in \{0, 1\}$. 
	\\
	For the remaining term $i\tilde{\alpha}\chi_1\cdot \partial_{RR}(W - \lambda^{-1}\psi^{(\tilde{\lambda})})$, which to leading order is again {\it{real valued}}, we can approximate it up to a term(see Lemma~\ref{lem:approxsolasymptotics1}) in $O\big(\chi_1\frac{\log R}{\tau^{1+}}\big)$ by 
	\begin{align*}
		\tilde{\alpha}\chi_1\cdot R^{-2}\tau^{-1} = \tilde{\alpha}\chi_1\cdot W\tau^{-1}  + \tilde{\alpha}\cdot O\big(\chi_1 R^{-4}\tau^{-1}\big). 
	\end{align*}
	Each type of term was treated earlier, and so we are done with the control of the contribution of the first to fifth terms in \eqref{eq:E1mod}.
	\\
	
	It remains to treat the contribution of the term $E_{nl}^{\text{mod}}$ to $E(\tau)$. Write 
	\begin{align*}
		E_{nl}^{\text{mod}} &= \big(\chi_3^2 - \chi_3\big)\cdot \big(n_*^{(\tilde{\lambda})} - n_*\big)\cdot \psi_*^{(\tilde{\lambda})}\\
		& +  \big(\chi_3^2 - \chi_3\big)\cdot n_*\cdot \big(\psi_*^{(\tilde{\lambda})} - \psi_*\big)
	\end{align*}
	To control the contribution of the first term on the right to $E(\tau)$, we use the estimate 
	\begin{align*}
		&\Big\|\langle \xi\partial_\xi\rangle^{1+\delta_0}\mathcal{F}\Big(\big(\chi_3^2 - \chi_3\big)\cdot \big(n_*^{(\tilde{\lambda})} - n_*\big)\cdot \psi_*^{(\tilde{\lambda})}\Big)\Big\|_{\tau^{-N-}L^2_{d\tau}L^\infty_{d\xi}}\\
		&\lesssim \big\|\frac{\tilde{\lambda}}{\lambda^2}\big\|_{\tau^{-N-}L^2_{d\tau}}\lesssim \big\|\langle\partial_{\tilde{\tau}}^2\rangle^{-1}\tilde{\lambda}_{\tilde{\tau}\tilde{\tau}}\big\|_{\tau^{-N}L^2_{d\tau}},
	\end{align*}
	the last inequality on account of $\lambda^2\sim \tau^{1+\frac{1}{2\nu}},\,\tilde{\tau}^2\sim \tau^{1-\frac{1}{2\nu}}$. Here we have also taken advantage of Lemma~\ref{lem:approxsolasymptotics3} and the fact that the support of $\chi_3^2 - \chi_3$ is contained in the set $R\gtrsim \tau^{\frac12-}$. 
	The estimate for the second term on the right constituting $E_{nl}^{\text{mod}}$ is similar. The desired bound for the contribution of $E_{nl}^{\text{mod}}$ to $E(\tau)$ is then a consequence of the preceding bounds and Lemmas~\ref{lem:K_frefined}, ~\ref{lem:tildeKfcontrol}. 
	
	\subsubsection{Completion of the proof of Lemma~\ref{lem:tildealphafixedpoint1}} In order to bound the integral 
	\begin{align*}
		\Re\int_0^\infty \xi^2 S_{\mathcal{K}}\big(e_1^{\text{mod},\tilde{\alpha}}\big)\rho(\xi)\,d\xi,
	\end{align*}
	and recalling \eqref{eq:ScalK2}, we use integration by parts with respect to time in the rightmost propagator $S(E_1^{\text{mod},\tilde{\alpha}})$ in order to only have factors $\tilde{\alpha}_{\sigma}$ or $\frac{\tilde{\alpha}}{\sigma}$, as in the first part of the proof. Then we take advantage of Lemma~\ref{lem:concatenation1} but with the space $\tau^{-N}L^2_{d\tau}$ replaced by $\log^{-1}(\tau)\cdot \tau^{-N}L^2_{d\tau}$. 
	As for the contributions of the error terms due to replacing $\psi_*^{(\tilde{\lambda})}$ by $\lambda W$, recalling \eqref{eq:E1mod}, these are of the form 
	\begin{align*}
		\Re\int_0^\infty \xi^2 (S + S_{\mathcal{K}})\mathcal{F}\big((-\partial_{\tau}+ i\triangle_R)(\chi_1)\cdot\tilde{\alpha}\cdot(\lambda^{-1}\psi_*^{(\tilde{\lambda})} - W)\big)\,d\xi
	\end{align*}
	as well as similar expressions for the fourth and fifth terms in \eqref{eq:E1mod}. Taking advantage of Lemma~\ref{lem:approxsolasymptotics1}, as well as Lemma~\ref{lem:K_frefined}, ~\ref{lem:tildeKfcontrol}, ~\ref{lem:concatenation1}, we can bound the $\big\|\cdot\big\|_{\log^{-1}(\tau)\cdot\tau^{-N}L^2_{d\tau}}$-norm, of all these expressions by
	\[
	\lesssim\big\|\frac{\tilde{\alpha}}{\tau}\big\|_{\log^{-1}(\tau)\cdot \tau^{-N}L^2_{d\tau}} + \big\|\tilde{\alpha}_{\tau}\big\|_{\log^{-1}(\tau)\cdot\tau^{-N-1}L^2_{d\tau}}\ll_{N}\big\|\tilde{\alpha}_{\tau}\big\|_{\log^{-1}(\tau)\cdot\tau^{-N}L^2_{d\tau}}
	\]
	
	\subsubsection{Completion of the proof of Lemma~\ref{lem:tildealphasourceterms}} 
	
	Recalling the right hand side of the first equation of \eqref{eq:zeqn2}, we need to bound the remaining contributions of it to the right hand side of \eqref{eq:tildealpha1} both via the term $\mathcal{L}z|_{R = 0}$, expressed via the Schr\"odinger propagator, as well as to $- \Re\big(\lambda^{-2}(y \tilde{u}_*^{(\tilde{\lambda}, \underline{\tilde{\alpha}})})\big)|_{R = 0}$, and finally we also need to control the contribution of  $\Re( e_1)|_{R = 0}$. 
	\\
	To begin with, the contribution to $\mathcal{L}z|_{R = 0}$ of the terms 
	\begin{align*}
		\lambda^{-2}y_z\cdot (\tilde{u}_*^{(\tilde{\lambda}, \tilde{\alpha})}-W),\,\lambda^{-2}(y - y_z)\cdot \tilde{u}_*^{(\tilde{\lambda}, \tilde{\alpha})} -  y^{\text{mod}}_{\tilde{\lambda}}\cdot W,
	\end{align*}
	whose sum forms the difference of the first term on the right in \eqref{eq:zeqn2} and $y^{\text{mod}}_{\tilde{\lambda}}\cdot W$, is handled by combining Lemma~\ref{lem:Xtildelambdaperturbterms} with Lemmas~\ref{lem:K_frefined}, ~\ref{lem:tildeKfcontrol}, resulting in a bound which even replaces the $\log^{-1}(\tau)\cdot\tau^{-N}L^2_{d\tau}$ by $\tau^{-N-}L^2_{d\tau}$. We then also need to control the contribution of the term $\lambda^{-2}y^{\text{mod}}_{\tilde{\lambda}}\cdot W$ (recalling \eqref{eq:ylamndatildemod}) to both $\mathcal{L}z|_{R = 0}$ (via the Schr\"odinger propagator) as well as to $- \Re\big(\lambda^{-2}(y \tilde{u}_*^{(\tilde{\lambda}, \underline{\tilde{\alpha}})})\big)|_{R = 0}$; here we again take advantage of Lemma~\ref{lem:Lfdifferencebound} (keeping in mind \eqref{eq:ReR0}), and more specifically the following Remark~\ref{rem:Lfdifferencebound} in the case of not too small frequencies, in conjunction with Lemma~\ref{lem:ytildelambdatimesWlocalizedbound}. 
	\\
	As for the second term on the right of  \eqref{eq:zeqn2}, its contribution to $\mathcal{L}z|_{R = 0}$ is controlled by means of Lemma~\ref{lem:concatenation2} (re-writing $z$ as Schr\"odinger propagator applied to the source term on the right in  \eqref{eq:zeqn2}), and using Lemma~\ref{lem:yzWbound1}, repeating the argument in the preceding paragraph to control the contribution of the first term on the right of \eqref{eq:zeqn2}, using the first bound of Lemma~\ref{lem:Xtildelambdafinaltermcrudebound} to control the contribution of the second term on the right in \eqref{eq:zeqn2}, and finally Lemmas~\ref{lem:basicboundsfore_1modandnonlinearterms}, ~\ref{lem:basicboundsfore_1modalphaterm} to control the remaining terms in \eqref{eq:zeqn2}.
	\\
	It remains to deal with the last three terms on the right of the first hand side of \eqref{eq:zeqn2} with the term $e_1^{\text{mod}}$ replaced by $e_1^{\text{mod}} - e_1^{\text{mod}, \tilde{\alpha}}$ , which is accomplished by combining Lemma~\ref{lem:basicboundsfore_1modandnonlinearterms} together with Lemmas~\ref{lem:K_frefined}, ~\ref{lem:tildeKfcontrol}. 
	
	\subsubsection{Completion of the proof of Lemma~\ref{lem:tildekappaoneiicontrib1}} We start with the following 
	\begin{lem}\label{lem:tildeylambdaWhighmodulationreduction} Recalling \eqref{eq:ytildemodmodified}, and letting 
		\begin{align*}
			\tilde{R}_{\text{small},\geq{\sqrt{\gamma}^{-1}}}^{(\tilde{\lambda})}
		\end{align*}
		be defined like $R_{\text{small},\geq{\sqrt{\gamma}^{-1}}}^{(\tilde{\lambda})}$ but with $\lambda^{-2}Q^{(\tilde{\sigma})}_{\geq\sqrt{\gamma}^{-1}}\tilde{y}_{\tilde{\lambda}}^{\text{mod}}\cdot W$ replaced by
		\begin{align*}
			Z(\tau, R): = \big(\lambda^{-2}Q^{(\tilde{\sigma})}_{\geq\sqrt{\gamma}^{-1}}\tilde{y}_{\tilde{\lambda}}^{\text{mod}}\cdot W - Q^{(\tilde{\sigma})}_{\geq\sqrt{\gamma}^{-1}}\tilde{\lambda}\cdot \Lambda W\cdot W^2\big),
		\end{align*}
		we can write 
		\begin{align*}
			\tilde{R}_{\text{small},\geq{\sqrt{\gamma}^{-1}}}^{(\tilde{\lambda})} = \partial_{\tau}\big(\delta \tilde{R}_{\text{small},\geq{\sqrt{\gamma}^{-1}}}^{(\tilde{\lambda})}\big) + \frac{1}{\tau}\cdot \delta \tilde{\tilde{R}}_{\text{small},\geq{\sqrt{\gamma}^{-1}}}^{(\tilde{\lambda})}
		\end{align*}
		where the terms on the right enjoy the bound 
		\begin{align*}
			\big\|\delta \tilde{R}_{\text{small},\geq{\sqrt{\gamma}^{-1}}}^{(\tilde{\lambda})}\big\|_{\tau^{-N}L^2_{d\tau}} + \big\|\delta \tilde{\tilde{R}}_{\text{small},\geq{\sqrt{\gamma}^{-1}}}^{(\tilde{\lambda})}\big\|_{\tau^{-N}L^2_{d\tau}}\ll_{\tau_*}\big\|\langle\partial_{\tilde{\tau}}^2\rangle^{-2}\tilde{\lambda}_{\tilde{\tau}\tilde{\tau}}\big\|_{\tau^{-N}L^2_{d\tau}}.
		\end{align*}
	\end{lem}
	\begin{proof} The first step is to establish the following analogue of the bound in Lemma~\ref{lem:yzWbound1}: 
		\begin{equation}\label{eq:intermediateZbound}
			\Big\|\langle \xi\partial_{\xi}\rangle^{1+\delta_0}\langle Z(\tau, R), \frac{\phi(R;\xi) - \phi(R; 0)}{\xi^2}\rangle_{L^2_{R^3\,dR}}\Big\|_{\tau^{-N}L^2_{d\tau}}\ll_{\tau_*}\big\|\langle\partial_{\tilde{\tau}}^2\rangle^{-2}\tilde{\lambda}_{\tilde{\tau}\tilde{\tau}}\big\|_{\tau^{-N}L^2_{d\tau}}.
		\end{equation}
		To see this, we decompose 
		\begin{align*}
			&Q^{(\tilde{\sigma})}_{\geq\sqrt{\gamma}^{-1}}\big(\lambda^{-2}\tilde{y}_{\tilde{\lambda}}^{\text{mod}}\cdot W - \tilde{\lambda}\cdot \Lambda W\cdot W^2\big)\\&= \lambda^{-2}Q^{(\tilde{\sigma})}_{\geq\sqrt{\gamma}^{-1}}P_{<\gamma^{-\frac14}}\big(\lambda^{-2}\tilde{y}_{\tilde{\lambda}}^{\text{mod}}\cdot W -  \tilde{\lambda}\cdot \Lambda W\cdot W^2\big)\\
			& + \lambda^{-2}Q^{(\tilde{\sigma})}_{\geq\sqrt{\gamma}^{-1}}P_{\geq \gamma^{-\frac14}}\big(\lambda^{-2}\tilde{y}_{\tilde{\lambda}}^{\text{mod}}\cdot W - \tilde{\lambda}\cdot \Lambda W\cdot W^2\big),\\
		\end{align*}
		where the additional frequency localizers $P_{<>\gamma^{-\frac14}}$ refer to the standard Littlewood-Paley frequency with respect to the physical variable $R$. 
		To control the contribution of the first term on the right, recalling that $\tilde{y}_{\tilde{\lambda}}^{\text{mod}}$ involves inversion of $\Box$, we write 
		\begin{align*}
			Q^{(\tilde{\tau})}_{\geq\sqrt{\gamma}^{-1}}P_{<\gamma^{-\frac14}}\Box^{-1} &= Q^{(\tilde{\tau})}_{\geq\sqrt{\gamma}^{-1}}P_{<\gamma^{-\frac14}}(-\partial_{\tilde{\tau}}^2 + L)^{-1}\\
			& = Q^{(\tilde{\tau})}_{\geq\sqrt{\gamma}^{-1}}P_{<\gamma^{-\frac14}}(-I + \partial_{\tilde{\tau}}^{-2}L)^{-1}\partial_{\tilde{\tau}}^{-2},\\
		\end{align*}
		where the operator $\partial_{\tilde{\tau}}^{-2}$ is given by division by the symbol $\hat{\tilde{\tau}}^2$ on the wave temporal Fourier side. The effect of $(-I + \partial_{\tilde{\tau}}^{-2}L)^{-1}$ on $\partial_{\tilde{\tau}}^{-2}\tilde{E}_2^{\text{mod}}$ (see \eqref{eq:tildeE2mod}) is determined by means of a Neumann series expansion, taking advantage of the bounds 
		\begin{align*}
			&\Big\|\langle \xi\partial_{\xi}\rangle^{1+\delta_0}\langle Q^{(\tilde{\tau})}_{\geq\sqrt{\gamma}^{-1}}P_{<\gamma^{-\frac14}}(\partial_{\tilde{\tau}}^{-2}L)^k \partial_{\tilde{\tau}}^{-2}\lambda^{-4}\tilde{E}_2^{\text{mod}},\,\frac{\phi(R;\xi) - \phi(R; 0)}{\xi^2}\rangle_{L^2_{R^3\,dR}}\Big\|_{\tau^{-N}L^2_{d\tau}}\\
			&\lesssim \delta(\gamma)^k\cdot \big\|\langle\partial_{\tilde{\tau}}^2\rangle^{-2}\tilde{\lambda}_{\tilde{\tau}\tilde{\tau}}\big\|_{\tau^{-N}L^2_{d\tau}},\,k\geq 1,\,\lim_{\gamma\rightarrow 0}\delta(\gamma) = 0,
		\end{align*}
		which follows from Lemma~\ref{lem:ytildelambdaE2}. This bound, together with the easily verified bound 
		\begin{align*}
			\Big\|\langle Q^{(\tilde{\sigma})}_{\geq\sqrt{\gamma}^{-1}}P_{\geq\gamma^{-\frac14}}\big(\tilde{\lambda}\cdot \Lambda W\cdot W^2\big),  \frac{\phi(R;\xi) - \phi(R; 0)}{\xi^2}\rangle_{L^2_{R^3\,dR}}\Big\|_{\tau^{-N}L^2_{d\tau}}\ll_{\tau_*}\big\|\langle\partial_{\tilde{\tau}}^2\rangle^{-2}\tilde{\lambda}_{\tilde{\tau}\tilde{\tau}}\big\|_{\tau^{-N}L^2_{d\tau}}.
		\end{align*}
		as well as the following consequence of Lemma~\ref{lem:ytildelambdaE2}
		\begin{align*}
			&\Big\|\langle \xi\partial_{\xi}\rangle^{1+\delta_0}\langle Q^{(\tilde{\tau})}_{\geq\sqrt{\gamma}^{-1}}P_{<\gamma^{-\frac14}}\big(\partial_{\tilde{\tau}}^{-2}\lambda^{-4}\tilde{E}_2^{\text{mod}} - \tilde{\lambda}\cdot \Lambda W\cdot W\big),\,\frac{\phi(R;\xi) - \phi(R; 0)}{\xi^2}\rangle_{L^2_{R^3\,dR}}\Big\|_{\tau^{-N}L^2_{d\tau}}\\
			&\ll_{\tau_*}\big\|\langle\partial_{\tilde{\tau}}^2\rangle^{-2}\tilde{\lambda}_{\tilde{\tau}\tilde{\tau}}\big\|_{\tau^{-N}L^2_{d\tau}}
		\end{align*}
		gives the desired bound if we replace $Z(\tau, R)$ in \eqref{eq:intermediateZbound} by 
		\begin{align*}
			\lambda^{-2}Q^{(\tilde{\sigma})}_{\geq\sqrt{\gamma}^{-1}}\big(P_{<\gamma^{-\frac14}}\tilde{y}_{\tilde{\lambda}}^{\text{mod}}\cdot W\big) - Q^{(\tilde{\sigma})}_{\geq\sqrt{\gamma}^{-1}}\tilde{\lambda}\cdot \Lambda W\cdot W^2.
		\end{align*}
		It remains to establish the bound when we replace $Z(\tau, R)$ by 
		\[
		\lambda^{-2}Q^{(\tilde{\sigma})}_{\geq\sqrt{\gamma}^{-1}}\big(P_{\geq \gamma^{-\frac14}}\tilde{y}_{\tilde{\lambda}}^{\text{mod}}\cdot W\big),
		\]
		which is accomplished by replicating the argument for Lemma~\ref{lem:yzWbound1} after performing integration by parts twice with respect to $\tilde{\sigma}$ in the Duhamel propagator to shift the temporal derivatives away from $\tilde{\lambda}_{\tilde{\sigma}\tilde{\sigma}}$, and using the fact that the localizer $P_{\geq \gamma^{-\frac14}}$ when applied to the source term $\Lambda W\cdot W$ gains $\gamma^M$ for arbitrary $M>0$. In fact, the preceding argument even implies the bound \eqref{eq:intermediateZbound} with $\delta_0 = 1+$. 
		\\
		With the bound \eqref{eq:intermediateZbound} in hand, we now come back to \eqref{eq:Rtildelambda} where we have to replace $\lambda^{-2}Q^{(\tilde{\sigma})}_{<\gamma^{-1}}\tilde{y}_{\tilde{\lambda}}^{\text{mod}}\cdot W$ by $Z(\sigma, R)$. Then we observe the general identity 
		\begin{align*}
			&\int_{\tau}^\infty\int_0^\infty \xi^4 S_1(\tau,\sigma,\xi)\cdot f(\sigma, \frac{\lambda(\tau}{\lambda(\sigma)}\xi)\rho_1(\xi)\,d\xi = A_1(\tau) + A_2(\tau)\\
			&A_1(\tau) = \partial_{\tau}\int_{\tau}^\infty\int_0^\infty \xi^2 \frac{\lambda^2(\sigma)}{\lambda^2(\tau)}S_2(\tau,\sigma,\xi)\cdot f(\sigma, \frac{\lambda(\tau}{\lambda(\sigma)}\xi)\rho_1(\xi)\,d\xi,\\
			&A_2(\tau) = -\frac{\lambda_{\tau}}{\lambda}\int_{\tau}^\infty\int_0^\infty \xi^2 \frac{\lambda^2(\sigma)}{\lambda^2(\tau)}S_2(\tau,\sigma,\xi)\cdot f(\sigma, \frac{\lambda(\tau}{\lambda(\sigma)}\xi)\tilde{\rho}_1(\xi)\,d\xi,
		\end{align*}
		where we set $\tilde{\rho}_1(\xi) = 1 + \frac{\rho_1'(\xi)\xi}{\rho_1(\xi)}$. Then the required estimate for the (modified) second term in \eqref{eq:Rtildelambda} is obtained by replacing $f$ by the term in parentheses on the left of \eqref{eq:intermediateZbound}, and combining \eqref{eq:intermediateZbound}
		with Lemma~\ref{lem:tildeKfcontrol}. The (modified) first term of \eqref{eq:Rtildelambda} is handled similarly as we have $\xi\geq 1$ on its support and so we can always extract an extra factor $\xi^2$. 
	\end{proof}
	
	The preceding lemma allows us to replace $\lambda^{-2}Q^{(\tilde{\sigma})}_{\geq\sqrt{\gamma}^{-1}}\tilde{y}_{\tilde{\lambda}}^{\text{mod}}\cdot W$ by $Q^{(\tilde{\sigma})}_{\geq\sqrt{\gamma}^{-1}}\tilde{\lambda}\cdot \Lambda W\cdot W^2$ in the definition of $R_{\text{small},\geq{\sqrt{\gamma}^{-1}}}^{(\tilde{\lambda})}$ as far as the proof of Lemma~\ref{lem:tildekappaoneiicontrib1} is concerned. 
	\\
	We still need to reduce this to the expression $X(\tau)$ displayed at the beginning of the proof of Lemma~\ref{lem:tildekappaoneiicontrib1}. For this the propagator $\xi^2S_1(\tau,\sigma,\xi)$ needs to be replaced by the simpler propagator $\xi^2\cos\big([\tau - \sigma]\xi^2\big)$ up to errors satisfying the conclusion of the lemma. The main point here is to gain smallness ($\ll_{\tau_*}$) in the estimates, which is no longer a consequence of the temporal frequency localization in $Q^{(\tilde{\sigma})}_{\geq\sqrt{\gamma}^{-1}}\tilde{\lambda}$ due to the norm used for $\tilde{\lambda}$. Still the reduction to the simpler propagator follows by replicating the argument for Lemma~\ref{lem:reductionsteps1}. At this stage observe that first expression in of  \eqref{eq:Rtildelambda} but with propagator $\xi^2\cos\big([\tau - \sigma]\xi^2\big)$ and $\lambda^{-2}Q^{(\tilde{\sigma})}_{\geq\sqrt{\gamma}^{-1}}\tilde{y}_{\tilde{\lambda}}^{\text{mod}}\cdot W$  replaced by $Q^{(\tilde{\sigma})}_{\geq\sqrt{\gamma}^{-1}}\tilde{\lambda}\cdot \Lambda W\cdot W^2$ has (Schr\"odinger) temporal Fourier transform which gains $\gamma^M$ due to smallness of the large frequency part of $ \Lambda W\cdot W^2$ (see the first part of the proof of Lemma~\ref{lem:tildekappaoneiicontrib1}), and since 
	\[
	(\triangle)\mathcal{F}\big(Q^{(\tilde{\sigma})}_{\geq\sqrt{\gamma}^{-1}}\tilde{\lambda}\cdot \Lambda W\cdot W^2\big) = \mathcal{F}\big(Q^{(\tilde{\sigma})}_{\geq\sqrt{\gamma}^{-1}}\tilde{\lambda}\cdot \Lambda W\cdot W^2\big),
	\]
	the reductions outlined at the beginning of the proof of  Lemma~\ref{lem:tildekappaoneiicontrib1}) are rigorously justified. 
	\\
	
	\subsubsection{Completion of the proof of Lemma~\ref{lem:tildekappaoneiicontrib2}}. Recalling \eqref{eq:Ltildelambda}, \eqref{eq:Xdef}, to complete the proof for the structure of $L^{(\tilde{\lambda})}_{\text{small}, \geq \sqrt{\gamma}^{-1}}$ stated in the lemma, we need to analyze the contributions of the second, third and fourth terms in \eqref{eq:Xdef}. For the second and third terms, this is achieved by taking advantage of Lemma~\ref{lem:Xtildelambdaperturbtermshighmod} and arguing as in the beginning of the proof of the lemma to write the corresponding double integrals as derivative term or a double integral weighted with a factor $\tau^{-1}$, and then taking advantage of Lemma~\ref{lem:tildeKfcontrol} to control the double integrals, and the last part of Lemma~\ref{lem:Xtildelambdaperturbtermshighmod} to control the boundary terms thus arising (which did not occur for the contribution of $Q^{(\tilde{\tau})}_{<\tau^{\frac12+}}\big(\lambda^{-2}y_z\cdot W\big)$ at the beginning of the proof).
	\\
	It remains to deal with the fourth term in \eqref{eq:Xdef}, which in light of Lemma~\ref{lem:approxsolasymptotics3} we split into two terms, corresponding to $R\lesssim \tau^{\frac12-\frac{1}{4\nu}},\,R\gtrsim \tau^{\frac12-\frac{1}{4\nu}}$. The former regime is easier to deal with due to the better asymptotics for $\lambda^{-2}n_*^{(\tilde{\lambda})} - W^2$ there. In fact, we can estimate 
	\begin{align*}
		\big\|\langle R\partial_R\rangle^{1+}Q^{(\tilde{\tau})}_{>\gamma^{-1}}\big(\chi_{R\lesssim \tau^{\frac12-\frac{1}{4\nu}}}(\lambda^{-2}n_*^{(\tilde{\lambda}, \tilde{\alpha})} - W^2)z\big)\big\|_{\tau^{-N-1+O(\frac{1}{\nu})}L^2_{d\tau}L^2_{R^3\,dR}}\lesssim \big\|z\big\|_{S}. 
	\end{align*}
	Using $Q^{(\tilde{\tau})}_{>\gamma^{-1}} = \frac{\partial\tau}{\partial_{\tilde{\tau}}}\cdot \partial_{\tau}\circ \partial_{\tilde{\tau}}^{-1}Q^{(\tilde{\tau})}_{>\gamma^{-1}}$, we can write  
	\begin{align*}
		Q^{(\tilde{\tau})}_{>\gamma^{-1}}\big(\chi_{R\lesssim \tau^{\frac12-\frac{1}{4\nu}}}(\lambda^{-2}n_*^{(\tilde{\lambda}, \tilde{\alpha})} - W^2)z\big) = \partial_{\tau}g, 
	\end{align*}
	where $g$ satisfies the bound $\big\|\langle R\partial_R\rangle^{1+}g\big\|_{\tau^{-N-\frac12+O(\frac{1}{\nu})}L^2_{d\tau}L^2_{R^3\,dR}}\lesssim \big\|z\big\|_{S}$. The desired conclusion for this contribution to $L^{(\tilde{\lambda})}_{\text{small}, \geq \sqrt{\gamma}^{-1}}$ then follows from Lemma~\ref{lem:derivativemovesthrough2}.  \\
	In order to handle the remaining term $Q^{(\tilde{\tau})}_{>\gamma^{-1}}\big(\chi_{R\gtrsim \tau^{\frac12-\frac{1}{4\nu}}}(\lambda^{-2}n_*^{(\tilde{\lambda}, \tilde{\alpha})} - W^2)z\big)$, we use Lemma~\ref{lem:pseudotransferenceoperator2} and express the distorted Fourier transform of $z$ by means of \eqref{eq:ScalK1}, \eqref{eq:ScalK2}, with $E$ being given by the difference of the right hand side of the first equation in \eqref{eq:zeqn2} and $\lambda^{-2}y_z\cdot W$. For the source term $(\lambda^{-2}n_*^{(\tilde{\lambda}, \tilde{\alpha})} - W^2)z$, we repeat the preceding procedure. Taking advantage of the fact that the frequency localizer $Q^{(\tilde{\tau})}_{>\gamma^{-1}}$ can be moved to the source term up to negligible errors (Lemma~\ref{lem:derivativemovesthrough2}), the desired conclusion then follows by combining Lemma~\ref{lem:transferenceKstarhightempfreqgeneral} (with $j = 0$) with Lemma~\ref{lem:yzWbound3}, Lemma~\ref{lem:Xtildelambdaperturbtermshighmod}, Lemma~\ref{lem:basicboundsfore_1modandnonlinearterms}, Lemma~\ref{lem:ytildelambdamodhightempfreq} and Lemma~\ref{lem:E1hightempfreq}. The same lemmas also imply the conclusion for $L^{(\tilde{\lambda})}_{\mathcal{K},\text{small},\geq \sqrt{\gamma}^{-1}},\,R^{(\tilde{\lambda})}_{\mathcal{K},\text{small},\geq \sqrt{\gamma}^{-1}}$, where now one uses the general version of Lemma~\ref{lem:transferenceKstarhightempfreqgeneral}.

	\subsubsection{Completion of contribution of the terms (i) in subsubsection~\ref{subsubsec:tildekappaoneimprov} } Here our task consists in bounding the contributions of the second to fourth terms in \eqref{eq:Xdef}, but with $Q^{(\tilde{\tau})}_{<\tau^{\frac12+}}$ replaced by $Q^{(\tilde{\tau})}_{>\tau^{\frac12+}}$, to the evolution of $z$ via the Schr\"odinger propagator and thence to the right hand side of \eqref{eq:kappa1eqn}.This shall as usual rely on Lemmas~\ref{lem:K_frefined}, ~\ref{lem:tildeKfcontrol}, Lemma~\ref{lem:concatenation1}, as well as the asymptotic structure of the approximate solution. Using that for $1\leq p\leq \infty$
	\begin{equation}\label{eq:hightempfreqgain}
		\big\|Q^{(\tilde{\tau})}_{>\tau^{\frac12+}}f(\tilde{\tau},\xi)\big\|_{\tau^{-N-1+}L^2_{d\tau}L^p_{\rho(\xi)\,d\xi}}\lesssim \big\|\partial_{\tilde{\tau}}^2f\big\|_{\tau^{-N}L^2_{d\tau}L^p_{\rho(\xi)\,d\xi}}, 
	\end{equation}
	and also taking advantage of Lemma~\ref{lem:wavebasicinhom} as well as Lemma~\ref{lem:approxsolasymptotics1} to handle the region $R\lesssim \tau^{\frac12-}$ and further Lemma~\ref{lem:refinedwavepropagatorwithphysicallocalization} to handle the case $R\gtrsim \tau^{\frac12-}$, we find
	\begin{align*}
		&\Big\|\langle\xi\partial_{\xi}\rangle^{1+\delta_0}Q^{(\tilde{\tau})}_{>\tau^{\frac12+}}\mathcal{F}\big(\lambda^{-2}y_z\cdot (\tilde{u}_*^{(\tilde{\lambda}, \tilde{\alpha})}-W)\big)\Big\|_{\tau^{-N-1-}L^2_{d\tau}L^2_{\rho(\xi)\,d\xi}}\\
		&\lesssim  \big\|z_{nres}\big\|_{S} + \big\|\langle\partial_{\tilde{\tau}}^2\rangle^{-2}\partial_{\tilde{\tau}}^2\tilde{\lambda}\big\|_{\tau^{-N}L^2_{d\tau}} + \big\|(\tilde{\kappa}_1,\kappa_2)\big\|_{\tau^{-N}L^2_{d\tau}}.
	\end{align*}
	Taking advantage of Lemmas~\ref{lem:K_frefined}, ~\ref{lem:tildeKfcontrol},, the the corresponding contribution to the right hand side of \eqref{eq:kappa1eqn}, can then be bounded by 
	\begin{align*}
		\big\|\cdot\big\|_{\tau^{-N-1}L^2_{d\tau}}\lesssim c(\tau_*)\cdot\big(\big\|z_{nres}\big\|_{S} + \big\|\langle\partial_{\tilde{\tau}}^2\rangle^{-2}\partial_{\tilde{\tau}}^2\tilde{\lambda}\big\|_{\tau^{-N}L^2_{d\tau}} + \big\|(\tilde{\kappa}_1,\kappa_2)\big\|_{\tau^{-N}L^2_{d\tau}}\big), 
	\end{align*}
	which in turn leads to the desired improved bound for $\tilde{\kappa}_1$. 
	\\
	To handle the contribution of the third term in \eqref{eq:Xdef} localized to high wave temporal frequencies, it suffices to control the contributions of 
	\begin{align*}
		Q^{(\tilde{\tau})}_{>\tau^{\frac12+}}\big(y_{\tilde{\lambda}}^{\text{mod}}\cdot(\tilde{u}_*^{(\tilde{\lambda}, \tilde{\alpha})}-W)\big),\,Q^{(\tilde{\tau})}_{>\tau^{\frac12+}}\big(\lambda^{-2}(y - y_z - y_{\tilde{\lambda}}^{\text{mod}})\cdot \tilde{u}_*^{(\tilde{\lambda}, \tilde{\alpha})}\big). 
	\end{align*}
	Of these the contribution of the first term is handled in analogy to the immediately preceding term, we omit the details. As for the second term, we take advantage of \eqref{eq:hightempfreqgain} as well as the second inequality in Lemma~\ref{lem:Xtildelambdaperturbterms}.
	\\
	This leaves us to bound the contribution of the fourth term in \eqref{eq:Xdef}, but localized by means of $Q^{(\tilde{\tau})}_{>\tau^{\frac12+}}$.  In light of the estimates in Lemma~\ref{lem:approxsolasymptotics3}, we split this term into two:
	\begin{equation}\label{eq:thirdterminXdefcontribtotildekappaone}\begin{split}
			Q^{(\tilde{\tau})}_{>\tau^{\frac12+}}\Big(\big(\lambda^{-2}n_*^{(\tilde{\lambda}, \underline{\tilde{\alpha}})} - W^2\big) z\Big)  &=  Q^{(\tilde{\tau})}_{>\tau^{\frac12+}}\Big(\chi_{R<\tau^{\frac12-}}\big(\lambda^{-2}n_*^{(\tilde{\lambda}, \underline{\tilde{\alpha}})} - W^2\big) z\Big)\\
			& +   Q^{(\tilde{\tau})}_{>\tau^{\frac12+}}\Big(\chi_{R\geq \tau^{\frac12-}}\big(\lambda^{-2}n_*^{(\tilde{\lambda}, \underline{\tilde{\alpha}})} - W^2\big) z\Big)\\
	\end{split}\end{equation}
	{\it{First term on the right}}: We can write this term in the form $\partial_{\tau}A_1+ \tau^{-1}A_2$ where 
	\begin{align*}
		&\big\|\langle\xi\partial_\xi\rangle^{1+\delta_0}\mathcal{F}(A_j)\big\|_{\tau^{-N-}L^2_{d\tau}L^2_{\rho(\xi)\,d\xi}}\\&\hspace{1.5cm}\lesssim \big(\big\|z_{nres}\big\|_{S} + \big\|\langle\partial_{\tilde{\tau}}^2\rangle^{-2}\partial_{\tilde{\tau}}^2\tilde{\lambda}\big\|_{\tau^{-N}L^2_{d\tau}} + \big\|(\tilde{\kappa}_1,\kappa_2)\big\|_{\tau^{-N}L^2_{d\tau}}\big),\,j = 1, 2. 
	\end{align*}
	To see this, we observe the general identity (where $f = f(\tau,\xi)$)
	\begin{equation}\label{eq:highmodgivesgradientstructure}\begin{split}
			Q^{(\tilde{\tau})}_{>\tau^{\frac12+}}f &= \partial_{\tilde{\tau}}\big(\partial_{\tilde{\tau}}^{-1}Q^{(\tilde{\tau})}_{>\tau^{\frac12+}}\big)f\\
			& = \partial_{\tau}\big(\frac{\partial\tau}{\partial\tilde{\tau}}\cdot \partial_{\tilde{\tau}}^{-1}Q^{(\tilde{\tau})}_{>\tau^{\frac12+}}\big)f - \partial_{\tau}\big(\frac{\partial\tau}{\partial\tilde{\tau}}\big)\cdot \partial_{\tilde{\tau}}^{-1}Q^{(\tilde{\tau})}_{>\tau^{\frac12+}}f\\
			& = : \partial_{\tau}f_1 -  f_2, 
	\end{split}\end{equation}
	where we have the estimates 
	\begin{align*}
		\big\|f_1\big\|_{\tau^{-N+\frac{1}{4\nu}-}L^2_{d\tau}Y}\lesssim \big\|f\big\|_{\tau^{-N}L^2_{d\tau}Y},\,\big\|f_2\big\|_{\tau^{-N-1+\frac{1}{4\nu}-}L^2_{d\tau}Y}\lesssim  \big\|f\big\|_{\tau^{-N}L^2_{d\tau}Y},
	\end{align*}
	where $\|\cdot\|_{Y}$ refers to a norm with respect to the $\xi$ variable.  
	We apply this to the first term $F_1$ on the right in \eqref{eq:thirdterminXdefcontribtotildekappaone} but with $N$ replaced by $N+\frac14$, and obtain (with $A_1 = f_1$, $A_2 = f_2$)
	\begin{align*}
		\big\|\langle \xi\partial_{\xi}\rangle^{1+\delta_0}\mathcal{F}(A_1)\big\|_{\tau^{-N-}L^2_{d\tau}L^2_{\rho(\xi)\,d\xi}}&\lesssim  \big\|\langle \xi\partial_{\xi}\rangle^{1+\delta_0}\mathcal{F}(A_1)\big\|_{\tau^{-N-\frac14+\frac{1}{4\nu}-}L^2_{d\tau}L^2_{\rho(\xi)\,d\xi}}\\
		&\lesssim  \big\|\langle \xi\partial_{\xi}\rangle^{1+\delta_0}\mathcal{F}(F_1)\big\|_{\tau^{-N-\frac14}L^2_{d\tau}L^2_{\rho(\xi)\,d\xi}}\\
	\end{align*}
	and similarly 
	\begin{align*} 
		\big\|\langle \xi\partial_{\xi}\rangle^{1+\delta_0}\mathcal{F}(A_2)\big\|_{\tau^{-N-1-}L^2_{d\tau}L^2_{\rho(\xi)\,d\xi}}\lesssim  \big\|\langle \xi\partial_{\xi}\rangle^{1+\delta_0}\mathcal{F}(F_1)\big\|_{\tau^{-N-\frac14}L^2_{d\tau}L^2_{\rho(\xi)\,d\xi}}.
	\end{align*}
	It then suffices to verify the basic estimate 
	\begin{align*}
		&\big\|\langle \xi\partial_{\xi}\rangle^{1+\delta_0}\mathcal{F}(F_1)\big\|_{\tau^{-N-\frac14}L^2_{d\tau}L^2_{\rho(\xi)\,d\xi}}\\&\hspace{1.5cm}\lesssim \big(\big\|z_{nres}\big\|_{S} + \big\|\langle\partial_{\tilde{\tau}}^2\rangle^{-2}\partial_{\tilde{\tau}}^2\tilde{\lambda}\big\|_{\tau^{-N}L^2_{d\tau}} + \big\|(\tilde{\kappa}_1,\kappa_2)\big\|_{\tau^{-N}L^2_{d\tau}}\big),
	\end{align*}
	which is indeed a consequence of Lemma~\ref{lem:approxsolasymptotics3} and the definition of $\|\cdot\|_{S}$. This implies the bounds for $A_j, j =  1, 2$. 
	\\
	The desired bound for the contribution of the first term on the right hand side to the Schr\"odinger evolution of $z$ and thence to the term $\mathcal{L}z|_{R = 0}$ in \eqref{eq:kappa1eqn} is now a consequence of Lemma~\ref{lem:derivativemovesthrough1}.
	\\
	
	{\it{Second term on the right of \eqref{eq:thirdterminXdefcontribtotildekappaone}}} Here we have to re-iterate the equation for $z$ until we get to a source term such as the first, second, fourth, fifth or sixth term in \eqref{eq:recallE}, or the third term with a cutoff $\chi_{R<\tau^{\frac12-}}$ (otherwise continue iterating the equation), and take advantage of Lemma~\ref{lem:concatenation2}. The last part of the latter (see also Lemma~\ref{lem:derivativemovesthrough2}) allows us to move the frequency localizer $Q^{(\tilde{\tau})}_{>\tau^{\frac12+}}$ from the outside of the term inside the iteration, and again taking advantage of Lemma~\ref{lem:derivativemovesthrough2} we obtain the desired bound for this contribution of $\mathcal{L}z|_{R = 0}$ by exploiting the already established estimates for the other source terms contributing to (i), as well as the bound for the contributions of the frequency localized error term $e_1^{\text{mod}}$ treated in (iv) in the next sub-subsection. 
	
	\subsubsection{Treatment of the terms in (iv) from subsubsection~\ref{subsubsec:tildekappaoneimprov} } Here we need to derive bounds for the expressions 
	\begin{equation}\label{eq:highmodeoneSchrodprop}\begin{split}
			&\Im \int_{\tau}^\infty\int_0^\infty \xi^2\cdot (S+S_{\mathcal{K}})(\tau, \sigma;\xi)\cdot \mathcal{F}\big(Q^{(\tilde{\tau})}_{\geq \tau^{\delta}}\big(e_1^{\text{mod}}\big)(\sigma, \frac{\lambda(\tau)}{\lambda(\sigma)}\xi)\rho(\xi)\,d\xi d\sigma,\\
			& \Im Q^{(\tilde{\tau})}_{\geq\tau^{\delta}}(e_1^{\text{mod}})|_{R = 0}.
	\end{split}\end{equation}
	Starting with the Schr\"odinger propagator expression, we need to control the contributions coming from the various terms constituting $E_1^{\text{mod}}$ in \eqref{eq:E1mod}. 
	\\
	{\it{Contribution of first term in $E_1^{\text{mod}}$}}. Recalling the passage from $E_1^{\text{mod}}$ to $e_1^{\text{mod}}$, see \eqref{eq:e1moddef}, we reduce to bounding the contribution of 
	\begin{align*}
		Q^{(\tilde{\tau})}_{\geq \tau^{\delta}}\Big(\big(i\partial_{\tau} + \triangle_R\big)\chi_3\cdot \lambda^{-1}(\psi_*^{(\tilde{\lambda})} - \psi_*)\Big)
	\end{align*}
	Using a simple analogue of \eqref{eq:highmodgivesgradientstructure}, we can equate the preceding term with a sum 
	\[
	\partial_{\tau}G_1 + G_2
	\]
	where we have the bound
	\begin{align*}
		\big\|G_1\big\|_{\tau^{-N-\frac12+\frac{1}{4\nu}-\delta}L^2_{d\tau}} + \big\|G_2\big\|_{\tau^{-N-\frac32+\frac{1}{4\nu}-\delta}L^2_{d\tau}}\lesssim \big\|\langle\partial_{\tilde{\tau}}^2\rangle^{-1}\tilde{\lambda}_{\tilde{\tau}\tilde{\tau}}\big\|_{\tau^{-N}L^2_{d\tau}}
	\end{align*}
	Using Lemma~\ref{lem:derivativemovesthrough2}, we infer that the contribution of this term to the Schr\"odinger propagator in \eqref{eq:highmodeoneSchrodprop} is again of the form $\partial_{\tau}\tilde{G}_1 + \tilde{G}_2$ with analogous bounds, from which the desired bound for the corresponding contribution to $\tilde{\kappa}_1$ follows as in the proof of Corollary~\ref{cor:Rtildelambdageqgamma-1maineffect}.
	\\
	{\it{Contribution of second term in $E_1^{\text{mod}}$}}. This term is analogous to the preceding one and omitted,
	\\
	{\it{Contribution of third to fifth term in $E_1^{\text{mod}}$}}. These terms all involve a factor $\tilde{\alpha}$ or $\tilde{\alpha}_{\tau}$. Since $\tilde{\alpha}$ is at wave temporal frequency $<1$, we easily see that the contribution of all these terms to $\tilde{\kappa}_1$ is much better, and in fact in $\tau^{-2N+1}L^2_{d\tau}$. 
	\\
	{\it{Contribution of sixth term in $E_1^{\text{mod}}$}}. Here we again use a simple analogue of \eqref{eq:highmodgivesgradientstructure} to deduce that (recall the definition of $E_{\text{nl}}^{\text{mod}}$ after \eqref{eq:E1mod}) 
	\begin{align*}
		Q^{(\tilde{\tau})}_{\geq \tau^{\delta}}\big(E_{\text{nl}}^{\text{mod}}\big) = \partial_{\tau}H, 
	\end{align*}
	where we have the bound 
	\begin{align*}
		\big\|H\big\|_{\tau^{-N-\frac12+\frac{1}{4\nu}+}L^2_{d\tau}}\lesssim \big\|\langle\partial_{\tilde{\tau}}^2\rangle^{-1}\tilde{\lambda}_{\tilde{\tau}\tilde{\tau}}\big\|_{\tau^{-N}L^2_{d\tau}}.
	\end{align*}
	The desired bound for the contribution to $\tilde{\kappa}_1$ then follows as before via Lemma~\ref{lem:derivativemovesthrough2} and the proof of Corollary~\ref{cor:Rtildelambdageqgamma-1maineffect}.
	\\
	{\it{Contribution of seventh term in $E_1^{\text{mod}}$}}. This contribution is more delicate, and in fact we require a cancellation between the term involving the Schr\"odinger propagator $S$ in the first expression in \eqref{eq:highmodeoneSchrodprop} and the second term 
	$ \Im Q^{(\tilde{\tau})}_{\geq\tau^{\delta}}(e_1^{mod})|_{R = 0}$: to begin with, we note that 
	\begin{align*}
		\lambda^{-1}\cdot \partial_{\tilde{\lambda}}\psi_*^{(\tilde{\lambda})} =  \chi_{R\lesssim\tau^{\frac12-}}W + \chi_{R\lesssim\tau^{\frac12-}}\tau^{-1}\cdot O(\log R),
	\end{align*}
	where we have taken advantage of Lemma~\ref{lem:approxsolasymptotics1}. The seventh term  in $E_1^{\text{mod}}$ involving the additional factor $i\tilde{\lambda}_{\tau}$, we take advantage of 
	\begin{align*}
		Q^{(\tilde{\tau})}_{\geq \tau^{\delta}}\big(\tilde{\lambda}_{\tau}\cdot\chi_{R\lesssim\tau^{\frac12-}}\tau^{-1}\cdot O(\log R)\big) = \partial_{\tau}M_1 + M_2,
	\end{align*}
	where the terms on the right enjoy the bound
	\begin{align*}
		&\big\|\langle\xi\partial_{\xi}\rangle^{1+\delta_0}\mathcal{F}(M_1)\big\|_{\tau^{-N-}L^2_{d\tau}L^2_{\rho(\xi)\,d\xi}} + \big\|\langle\xi\partial_{\xi}\rangle^{1+\delta_0}\mathcal{F}(M_2)\big\|_{\tau^{-N-1-}L^2_{d\tau}L^2_{\rho(\xi)\,d\xi}}\\&\lesssim \big\|\langle\partial_{\tilde{\tau}}^2\rangle^{-1}\tilde{\lambda}_{\tilde{\tau}\tilde{\tau}}\big\|_{\tau^{-N}L^2_{d\tau}}.
	\end{align*}
	The contribution to $\tilde{\kappa}_1$ of the 'error term' the preceding formula for $\lambda^{-1}\cdot \partial_{\tilde{\lambda}}\psi_*^{(\tilde{\lambda})}$ is then handled by means of Lemma~\ref{lem:derivativemovesthrough1}, ~\ref{lem:derivativemovesthrough2}, in conjunction with the the proof of Corollary~\ref{cor:Rtildelambdageqgamma-1maineffect}. It remains to deal with the contribution of the main term $i\tilde{\lambda}_{\tau}\cdot  \chi_{R\lesssim\tau^{\frac12-}}W(R) =: e_{1,\text{main}}^{\text{mod}}$. Here we exploit a cancelation, and specifically, we claim that 
	\begin{align*}
		&\Big\|\Im \int_{\tau}^\infty\int_0^\infty \xi^2\cdot (S+S_{\mathcal{K}})(\tau, \sigma;\xi)\cdot \mathcal{F}\big(Q^{(\tilde{\tau})}_{\geq \tau^{\delta}}\big(e_{1,\text{main}}^{\text{mod}}\big)(\sigma, \frac{\lambda(\tau)}{\lambda(\sigma)}\xi)\rho(\xi)\,d\xi d\sigma,\\
		& +\Im Q^{(\tilde{\tau})}_{\geq\tau^{\delta}}(e_{1,\text{main}}^{\text{mod}})|_{R = 0}\Big\|_{\tau^{-N}L^2_{d\tau}}\ll_{\tau*}\big\|\langle\partial_{\tilde{\tau}}^2\rangle^{-1}\tilde{\lambda}_{\tilde{\tau}\tilde{\tau}}\big\|_{\tau^{-N}L^2_{d\tau}}.
	\end{align*}
	The main contribution from the first integral expression comes from the propagator the $S$, and it is this contribution which partly cancels against the last term, in analogy to Lemma~\ref{lem:Lfdifferencebound}, although the structure of the source term is slightly different here. Since $e_{1,\text{main}}^{\text{mod}}$ is purely imaginary, we may replace $S$ by $S_2$, and performing integration by parts with respect to $\sigma$, using $S_2 = \partial_{\sigma}\big(\cos(\lambda^2(\tau)\xi^2\int_{\sigma}^{\tau}\lambda^{-2}(s)\,ds)\big)$, the boundary term cancels against 
	\[
	\Im Q^{(\tilde{\tau})}_{\geq\tau^{\delta}}(e_{1,\text{main}}^{\text{mod}})|_{R = 0},
	\]
	leading to  
	\begin{align*}
		&\Im \int_{\tau}^\infty\int_0^\infty \xi^2\cos(\lambda^2(\tau)\xi^2\int_{\sigma}^{\tau}\lambda^{-2}(s)\,ds)\\&\hspace{5cm} \cdot\partial_{\sigma}\big(\mathcal{F}\big(Q^{(\tilde{\tau})}_{\geq \tau^{\delta}}\big(e_{1,\text{main}}^{\text{mod}}\big)(\sigma, \frac{\lambda(\tau)}{\lambda(\sigma)}\xi)\big)\rho(\xi)\,d\xi d\sigma
	\end{align*}
	We then observe that 
	\begin{align*}
		\big\|\langle\xi\partial_{\xi}\rangle^{1+\delta_0}Q^{(\tilde{\tau})}_{\geq \tau^{\delta}}\mathcal{F}\big(e_{1,\text{main}}^{\text{mod}}\big)\big\|_{\tau^{-N-\frac{1}{4\nu}+}L^2_{d\tau}L^2_{R^3\,dR}}\lesssim  \big\|\langle\partial_{\tilde{\tau}}^2\rangle^{-1}\tilde{\lambda}_{\tilde{\tau}\tilde{\tau}}\big\|_{\tau^{-N}L^2_{d\tau}}.
	\end{align*}
	Using Lemma~\ref{lem:derivativemovesthrough1} we can 'move' the operator $\partial_{\sigma}$ to the outside of the preceding double integral, and infer the desired bound for the corresponding contribution to $\tilde{\kappa}_1$ as usual. 
	\\
	It remains to deal with the contribution of the propagator $S_{\mathcal{K}}$, i. e. the expression 
	\begin{align*}
		\Im \int_{\tau}^\infty\int_0^\infty \xi^2\cdot S_{\mathcal{K}}(\tau, \sigma;\xi)\cdot \mathcal{F}\big(Q^{(\tilde{\tau})}_{\geq \tau^{\delta}}\big(e_{1,\text{main}}^{\text{mod}}\big)(\sigma, \frac{\lambda(\tau)}{\lambda(\sigma)}\xi)\rho(\xi)\,d\xi d\sigma
	\end{align*}
	This term can be written as $\partial_{\tau}M_1 + M_2$ where we have the bound
	\begin{align*}
		\big\|M_1\big\|_{\tau^{-N}L^2_{d\tau}} + \big\|M_2\big\|_{\tau^{-N-1}L^2_{d\tau}}\ll_{\tau_*} \big\|\langle\partial_{\tilde{\tau}}^2\rangle^{-1}\tilde{\lambda}_{\tilde{\tau}\tilde{\tau}}\big\|_{\tau^{-N}L^2_{d\tau}},  
	\end{align*}
	by essentially replicating the argument for Lemma~\ref{lem:SKdeltaPhi}.
	
	\subsubsection{Completion of the proof of Lemma~\ref{lem:znresperturbative1}, small frequency part} Here we need to control the contribution to the first term on the right in \eqref{eq:snresdecomp}, and we use $e_1^{\text{mod}}$ for $E$, see \eqref{eq:Modbulkeqns}. While for some of these terms this can be achieved relatively easily by means of Lemma~\ref{lem:nonresbasicsmallfreq1}, the terms involving $\tilde{\alpha}$ require more care. Let us start with the latter, distinguishing between the contribution of the imaginary and real valued sources, and their contribution to the real and imaginary parts of $z_{nres*,<\epsilon_1}$, where the identities \eqref{eq:ImzR0}, \eqref{eq:ReR0} shall be useful. Recalling \eqref{eq:E1mod}, write 
	\begin{equation}\label{eq:lemma61remainingtildealphaterms1}\begin{split}
			i\triangle\chi_1\tilde{\alpha}\cdot \psi_*^{(\tilde{\lambda})} + 2i\partial_r(\chi_1)\cdot\tilde{\alpha}\cdot \partial_r \psi_*^{(\tilde{\lambda})} &= i\triangle\chi_1\tilde{\alpha}\cdot \lambda W + 2i\partial_r(\chi_1)\cdot\tilde{\alpha}\cdot \partial_r \lambda W\\
			& + i\triangle\chi_1\tilde{\alpha}\cdot \big(\psi_*^{(\tilde{\lambda})} - \lambda W) \\&\hspace{2cm}+ 2i\partial_r(\chi_1)\cdot\tilde{\alpha}\cdot \partial_r \big(\psi_*^{(\tilde{\lambda})} - \lambda W\big) 
	\end{split}\end{equation}
	We take advantage of Lemma~\ref{lem:approxsolasymptotics1} to handle the contribution of the third and fourth terms on the right: note that the sum of the last two terms on the right and divided by $\lambda^3$ equals 
	\begin{equation}\label{eq:lemma61remainingtildealphaterms2}
		i\mathcal{L}\Big(\chi_1\tilde{\alpha}\cdot \big(\lambda^{-1}\psi_*^{(\tilde{\lambda})} - W)\Big) - i\chi_1\frac{\tilde{\alpha}}{\tau}\cdot W(R) + O\big(\chi_1\big|\frac{\tilde{\alpha}}{\tau}\big|\cdot W^2(R)\big)
	\end{equation}
	To estimate the contribution of the first term on the right, we essentially use  Lemma~\ref{lem:nonresbasicsmallfreq1},  Remark~\ref{rem:nonresbasicsmallfreq1}, except we have to argue more carefully for the boundary term at $\sigma = \tau$ arising after integration by parts with respect to $\sigma$ in the Schr\"odinger propagator. This boundary term equals 
	\begin{align*}
		\int_0^\infty \chi_{\xi<\epsilon_1}[\phi(R;\xi) - \phi(R;0)]\cdot \mathcal{F}\big(\chi_1\tilde{\alpha}\cdot \big(\lambda^{-1}\psi_*^{(\tilde{\lambda})} - W)\big)(\tau,\xi)\rho(\xi)\,d\xi. 
	\end{align*}
	The Fourier coefficient vanishes rapidly beyond $\xi = \tau^{-\frac12+\frac{\epsilon}{2\nu}}$ by choice of $\chi_1$, and we have (using $\big|\chi_1\cdot \big(\lambda^{-1}\psi_*^{(\tilde{\lambda})} - W)\big)\big|\lesssim \frac{\log R}{\tau}$)
	\[
	\big\|\mathcal{F}\big(\chi_1\tilde{\alpha}\cdot \big(\lambda^{-1}\psi_*^{(\tilde{\lambda})} - W)\big)\big\|_{\log\tau\cdot \tau^{-N+1-\frac{\epsilon}{\nu}}L^2_{d\tau}L^\infty_{d\xi}}\lesssim \big\|\frac{\tilde{\alpha}}{\tau}\big\|_{\log^{-1}(\tau)\cdot\tau^{-N}L^2_{d\tau}}.
	\] 
	Using the low frequency asymptotics of $\rho$ from subsection~\ref{subsec:basicfourier}, we then infer 
	\begin{align*}
		&\big\|\int_0^\infty \chi_{\xi<\epsilon_1}[\phi(R;\xi) - \phi(R;0)]\cdot \mathcal{F}\big(\chi_1\tilde{\alpha}\cdot \big(\lambda^{-1}\psi_*^{(\tilde{\lambda})} - W)\big)(\tau,\xi)\rho(\xi)\,d\xi\big\|_{\tau^{-N}L^2_{d\tau}\langle R\rangle^{\delta_0}L^\infty_{dR}}\\
		&\lesssim \big\|\chi_{\xi\lesssim  \tau^{-\frac12+\frac{\epsilon}{2\nu}}}\xi^2\rho(\xi)\cdot |\mathcal{F}\big(\chi_1\tilde{\alpha}\cdot \big(\lambda^{-1}\psi_*^{(\tilde{\lambda})} - W)\big)\big\|_{\tau^{-N}L^2_{d\tau}L^1_{d\xi}}\ll_{\tau_*}\big\|\frac{\tilde{\alpha}}{\tau}\big\|_{\log^{-1}(\tau)\cdot\tau^{-N}L^2_{d\tau}}.
	\end{align*}
	The bound for the remaining norms in the definition of $\|\cdot\|_{S}$ proceeds similarly: for example, applying $\mathcal{L}$ gains an additional $\xi^2\lesssim \tau^{-1+\frac{\epsilon}{\nu}}$, which is better than required. For the second term in \eqref{eq:lemma61remainingtildealphaterms2}, we use the cancellation condition $\langle \chi_1W, W\rangle = 0$ to write 
	\begin{align*}
		\mathcal{F}\big(\chi_1\frac{\tilde{\alpha}}{\tau}\cdot W(R)\big) &= \langle \chi_1\frac{\tilde{\alpha}}{\tau}\cdot W(R), \phi(R;\xi) - \phi(R;0)\rangle_{L^2_{R^3\,dR}}\\
		& = \xi^2\cdot  \langle \chi_1\frac{\tilde{\alpha}}{\tau}\cdot W(R), \frac{\phi(R;\xi) - \phi(R;0)}{\xi^2}\rangle_{L^2_{R^3\,dR}},\\
	\end{align*}
	at which point one can repeat the argument for the first term in \eqref{eq:lemma61remainingtildealphaterms2}, the factor $\xi^2$ being equivalent to the action of $\mathcal{L}$. For the last term in \eqref{eq:lemma61remainingtildealphaterms2}, calling its distorted Fourier transform $f$ it suffices to note that the right hand side in \eqref{lem:nonresbasicsmallfreq1} is bounded by $\ll_{\tau_*}\big\|\frac{\tilde{\alpha}}{\tau}\big\|_{\tau^{-N}L^2_{d\tau}}$, which suffices due to Lemma~\ref{lem:nonresbasicsmallfreq1}.
	\\
	
	{\it{Completion of the proof of Lemma~\ref{lem:smalltempfreqznresprinimprovedbound}.}} {\it{(1): the bound for $\mathcal{L}_*^{-1}(z_{nres, small}^{prin})$.}} Write 
	\begin{equation}\label{eq:lstarinverseznressmallprin}\begin{split}
			\mathcal{L}_*^{-1}(z_{nres, small}^{prin})(R) &= \phi_{0,*}(R)\cdot \int_0^R \theta_{0,*}(s)z_{nres, small}^{prin}(\cdot, s)s^3\,ds\\&\hspace{2cm} - \theta_{0,*}(R)\cdot \int_0^R \phi_{0,*}(s)z_{nres, small}^{prin}(\cdot, s)s^3\,ds,
	\end{split}\end{equation}
	where $\{\phi_{0,*}, \theta_{0,*}\}$ is a fundamental system for $\mathcal{L}*\in \{\mathcal{L}, \tilde{\mathcal{L}}\}$. Recalling \eqref{eq:znressmall} express $z_{nres, small}^{prin} = \mathcal{L}(\tilde{z}_{nres, small}^{prin})$, where $\tilde{z}_{nres, small}^{prin}$ is given by the same double integral as $z_{nres, small}^{prin}$ except with an extra factor $\xi^{-2}$ (but we have the range restriction $\epsilon_1\lesssim \xi\lesssim \epsilon_1^{-1}$). Performing integration by parts with respect to $s$ in the preceding variation of constants formula, we see that to control the first two norms in the definition of $\|\cdot\|_{S}$ in \eqref{eq:Snormdefi}, it suffices to bound $\big\|R^{\frac32-\delta_0}\cdot \partial_R^{\iota}(\tilde{z}_{nres, small}^{prin})\big\|_{L^\infty_{dR}}$, $\iota\in \{0, 1\}$. For the contribution of $y_z$ to \eqref{eq:znressmall} (with an extra $\xi^{-2}$), this follows by means of one integration by parts with respect to $\sigma$ and application of Corollary~\ref{cor:yzWpartialtau}, Lemma~\ref{lem:wavebasicinhomstructure1} but with $\partial_{\tilde{\tau}}^2$ replaced by $\partial_{\tau}^2$ which results in a gain of $\tau^{-1-\frac{1}{2\nu}}$,  Plancherel's theorem for the distorted Fourier transform and the asymptotics of the Fourier basis $\phi(R;\xi)$ in subsection~\ref{subsec:basicfourier}. This bounds the sum of the first two norms in $\|\cdot\|_{S}$ of this contribution by $\ll_{\tau_*}\big\|z_{nres}\big\|_{S} + \big\|(\tilde{\kappa_1}, \kappa_2)\big\|_{\tau^{-N}L^2_{d\tau}}$. For the contribution of $y_{\tilde{\lambda}}$, one uses Lemma~\ref{lem:ytildelambdamodhightempfreq} and multiple integrations by parts with respect to $\sigma$, each of which results in a small power gain in $\sigma$ due to the wave temporal frequency localization of $\tilde{\lambda}$ to frequencies $<\tau^{\frac12+}$ up to rapidly decaying tails. This allows us to bound the corresponding contribution by $\ll_{\tau_*}\big\|\langle\partial_{\tilde{\tau}}^2\rangle^{-2}\tilde{\lambda}_{\tilde{\tau}\tilde{\tau}}\big\|_{\tau^{-N}L^2_{d\tau}}$, which in turn is bounded by $\lesssim \big\|z_{nres}\big\|_{S} + \big\|(\tilde{\kappa}_1,\kappa_2)\big\|_{\tau^{-N}L^2_{d\tau}}$ by Proposition~\ref{prop:tilealphatildelambdasimultaneous}. It now remains to bound the third component (see  \eqref{eq:Snormdefi}) of $\big\|\cdot\big\|_{S}$ for \eqref{eq:lstarinverseznressmallprin}. 
	Writing $\mathcal{L} = \mathcal{L}_* + c\cdot W^2$, we can estimate 
	\begin{align*}
		\big\|c\cdot W^2\cdot \mathcal{L}_*^{-1}(z_{nres, small}^{prin})\big\|_{\tau^{-N}L^2_{d\tau}L^{2+}_{R^3\,dR}}\lesssim \big\|\mathcal{L}_*^{-1}(z_{nres, small}^{prin})\big\|_{\tau^{-N}L^2_{d\tau}L^{\infty}_{R^3\,dR}},
	\end{align*}
	and so the preceding estimates yield the desired bound for this contribution. On the other hand, we can bound 
	\begin{align*}
		\mathcal{L}_*\big(\mathcal{L}_*^{-1}(z_{nres, small}^{prin})\big) = z_{nres, small}^{prin}
	\end{align*}
	in $\tau^{-\frac12-\frac{1}{4\nu}+}\cdot \tau^{-N}L^2_{d\tau}(\langle R\rangle^{\delta_0}L^\infty_{R^3\,dR}\cap \mathcal{L}^{-\frac14}\langle R\rangle^{-\frac12+\delta_0}L^\infty_{R^3\,dR})$ by using Lemma~\ref{lem:nonresbasicsmallfreq1}, Corollary~\ref{cor:yzWpartialtau}, Lemma~\ref{lem:ytildelambdamodhightempfreq}.  
	\\
	{\it{(2): the bound for $\tilde{E}_{main}$.}} Here we need to estimate the various terms in \eqref{eq:ndecompose}, \eqref{eq:ntermlist}, with respect to $\tau^{-N}L^2_{d\tau}L^{2+}_{R^3\,dR}\cap \langle R\rangle^{\frac{\delta_0}{2}}L^2_{R^3\,dR}$. Observe that (with $P_{<\epsilon}$ the spectral localization operator associated to $\mathcal{L}$)
	\begin{align*}
		\big\|P_{<\epsilon_1}f\big\|_{L^{2+}_{R^3\,dR}\cap \langle R\rangle^{\frac{\delta_0}{2}}L^2_{R^3\,dR}}\lesssim \big\|P_{<\epsilon_1}f\big\|_{L^2_{R^3\,dR}}^{(1-\beta)}\cdot \big\|P_{<\epsilon_1}f\big\|_{L^\infty_{R^3\,dR}}^{\beta}
	\end{align*}
	for suitable $\beta\in (0,1)$, and further $\big\|P_{<\epsilon_1}f\big\|_{L^\infty}\lesssim (|\log\epsilon_1|)^{-\frac12}\cdot \big\|f\big\|_{L^2_{R^3\,dR}}$.  We hence conclude from Corollary~\ref{cor:yzW} that (recall \eqref{eq:yzdfn}) 
	\[
	\big\||P_{<\epsilon_1}\big(\lambda^{-2}y_z\cdot W\big)\big\|_{\tau^{-N}L^2_{d\tau}L^{2+}_{R^3\,dR}\cap \langle R\rangle^{\frac{\delta_0}{2}}L^2_{R^3\,dR}}\ll_{\epsilon_1}\big\|z_{nres}\big\|_{S} + \big\|(\tilde{\kappa}_1, \kappa_2)\big\|_{\tau^{-N}L^2_{d\tau}}. 
	\]
	For the high frequency term $P_{>\epsilon_1^{-1}}\big(\lambda^{-2}y_z\cdot W\big)$, it suffices to observe that we may include a standard Littlewood -Paley cutoff localizing to (Littlewood-Paley) frequency $\gtrsim \epsilon_1^{-1+}$ in front of $y_z$, up to errors of order $O(\epsilon_1^M)$, and to invoke Lemma~\ref{lem:wavebasicinhom}  in conjunction with the definition of $\big\|\cdot\big\|_{S}$ and arguments as in the proof of Corollary~\ref{cor:yzW} to conclude 
	\begin{align*}
		\big\|P_{>\epsilon_1^{-1}}\big(\lambda^{-2}y_z\cdot W\big)\big\|_{\tau^{-N}L^2_{d\tau}L^2_{R^3\,dR}}\ll_{\epsilon_1}\big\|z\big\|_{S}\lesssim \big\|z_{nres}\big\|_{S} + \big\|(\tilde{\kappa}_1, \kappa_2)\big\|_{\tau^{-N}L^2_{d\tau}}. 
	\end{align*}
	We conclude by combining Lemma~\ref{lem:znresperturbative1}, Lemma~\ref{lem:nonresconnect1}, Lemma~\ref{lem:znresrestbound} with Corollary~\ref{cor:yzW} and Remark~\ref{rem:cor:yzW}, namely
	\begin{align*}
		&\big\|n_{*,<\epsilon_1}\cdot W\big\|_{\tau^{-N}L^2_{d\tau}L^2_{R^3\,dR}} + \big\|n_{*,<\epsilon_1}\cdot W\big\|_{\tau^{-N}L^2_{d\tau}L^2_{R^3\,dR}} +  \big\|n_{rest}\cdot W\big\|_{\tau^{-N}L^2_{d\tau}L^2_{R^3\,dR}}
		\\&+ \big\|n_{nres, \mathcal{K}}\cdot W\big\|_{\tau^{-N}L^2_{d\tau}L^2_{R^3\,dR}} +  \big\|n_{res}\cdot W\big\|_{\tau^{-N}L^2_{d\tau}L^2_{R^3\,dR}}\lesssim c(\tau_*,\epsilon_1,N)\cdot \big\|z_{nres}\big\|_{S}\\&\hspace{7.5cm} + \big\|(\tilde{\kappa}_1,\kappa_2)\big\|_{\tau^{-N}L^2_{d\tau}} + \big\|e_1\big\|_{\tau^{-N}L^2_{d\tau}L^2_{R^3\,dR}}. 
	\end{align*}
	where we have the limiting relation $\lim_{\epsilon_1^{-1}\tau_*, N\rightarrow +\infty}c(\tau_*,\epsilon_1,N) = 0$. 
	\\
	
	{\it{Proof outline for Lemma~\ref{lem:tildeboxwpropagatorapproximate}.}} To begin with, we translate the equation $\tilde{\Box}_W u = f$ to the distorted Fourier side in relation to the operator $\mathcal{L}_*$, see subsection~\ref{subsec:Lstarbasics}. Using vectorial notation for to encode the discrete and continuous spectral part of $u$, $\underline{\hat{u}} = \left(\begin{array}{c} \mathcal{F}_{*,d}(u)\\ \mathcal{F}_*(u)\end{array}\right)$, we infer the following equation on the Fourier side 
	\begin{equation}\label{eq:tildeboxwtranslatedtofouriervectorial}
		\big(\mathcal{D}_{*,\tilde{\tau}}^2 + \frac{\lambda_{\tilde{\tau}}}{\lambda}\mathcal{D}_{*,\tilde{\tau}} + \underline{\xi^2}\big)\underline{\hat{u}} = \mathcal{R}_*(\tilde{\tau}, \underline{\hat{u}}) + \underline{\hat{f}},\,\underline{\xi^2} = \left(\begin{array}{c}-\xi_d^2\\ \xi^2\end{array}\right)
	\end{equation}
	where we set $\mathcal{D}_{*,\tilde{\tau}}: = \partial_{\tilde{\tau}} + \frac{\lambda_{\tilde{\tau}}}{\lambda}\mathcal{A}_*$, $\mathcal{A}_* = \left(\begin{array}{cc}0& 0\\ 0&\mathcal{A}_{*,c}\end{array}\right)$ and $\mathcal{A}_{*,c} = -\xi\partial_{\xi} - 4 - \frac{\rho_*'(\xi)\cdot\xi}{\rho_*(\xi)}$. The matrix valued operator $\mathcal{R}_*$ is expressed in terms of the {\it{transference operator}} $\mathcal{K}_*$ associated to $\mathcal{L}_*$, and given by 
	\begin{align*}
		\widehat{\big(\underline{(R\partial_R) f}\big)}(\xi) = -(\xi\partial_{\xi})\underline{\hat{f}}(\xi)- 4\underline{\hat{f}}(\xi) + \mathcal{K}_*\big(\underline{\hat{f}}\big)(\xi)
	\end{align*}
	Then as in \cite{KST}, \cite{DS}, one shows that $\mathcal{K}_*$ is a bounded operator on $\C\times L^2_{\rho_*(\xi)\,d\xi}$, and the operator $ \mathcal{R}_*(\tilde{\tau}, \underline{\hat{u}})$ can be expressed as linear combination of 
	\begin{align*}
		\beta_{\tilde{\tau}}^2\mathcal{K}_*,\,\beta_{\tilde{\tau}}^2\mathcal{K}_*^2,\,\beta_{\tilde{\tau}}\mathcal{K}_*\partial_{\tilde{\tau}},\,\beta_{\tilde{\tau}}^2[\mathcal{A},\mathcal{K}_*]. 
	\end{align*}
	The operator on the left in \eqref{eq:tildeboxwtranslatedtofouriervectorial} has the vector valued propagator 
	\begin{align*}
		\underline{\tilde{U}}_*(\underline{\hat{f}}): = \left(\begin{array}{c}\int_{\tilde{\tau}}^\infty\tilde{U}_*(\tilde{\tau},\tilde{\sigma},\xi)\cdot\hat{f}(\frac{\lambda(\tilde{\tau})}{\lambda(\tilde{\sigma})}\xi)\,d\tilde{\sigma}\\ -\frac{\xi_d}{2}\int_{\tilde{\tau}}^\infty H_{*,d}(\tilde{\tau}, \tilde{\sigma})\hat{f}_d(\tilde{\sigma})d\tilde{\sigma} \end{array}\right).
	\end{align*}
	Further, one checks the key smallness gains in the following estimate:
	\begin{align*}
		\Big\|\underline{\tilde{U}}_*\big(\mathcal{R}_*(\tilde{\tau}, \underline{\hat{u}})\big)\Big\|_{\tau^{-N}L^2_{d\tau}(\C\times L^2_{\rho_*(\xi)\,d\xi})}\ll_{\tau_*, N}\big\|\underline{\hat{u}}\big\|_{\tau^{-N}L^2_{d\tau}(\C\times L^2_{\rho_*(\xi)\,d\xi}}
	\end{align*}
	This estimate shows that the first term on the right in \eqref{eq:tildeboxwtranslatedtofouriervectorial} can be iterated away by successive applications of $\underline{\tilde{U}}_*$ and the statement of Lemma~\ref{lem:tildeboxwpropagatorapproximate} follows.

	\subsection{The proof of Proposition~\ref{prop: solnforylambdaW}}
	
	The strategy shall be to re-write \eqref{eq:zfrelation} in terms of a standard Schr\"odinger propagator on the left hand side and moving the errors thereby generated to the right hand side. A key aspect is that we only require the solution to solve the equation on $[\tau_*,\infty)$, allowing us to modify functions (when globally defined) arbitrarily below time $\tau_*$. The following lemma will allow us to force certain vanishing conditions, which shall be useful to treat the error terms when working with the temporal Fourier transform:  
	
	\begin{lem}\label{lem:fmodif} Let $f\in \tau^{-N}L^2_{d\tau}([\tau_*,\infty))$. Then there is $\tilde{f}\in \tau^{-N} L^2_{d\tau}([\frac{\tau_*}{2},\infty])$ such that $\tilde{f}|_{[\tau_*,\infty)} = f$ and further 
		\begin{align*}
			\partial_{\hat{\tau}}^l\hat{\tilde{f}}(0) = 0,\,l = 0, 1, \ldots, N-1,\,\big\|\tilde{f}\big\|_{\tau^{-N} L^2_{d\tau}([\frac{\tau_*}{2},\infty])}\lesssim_N \big\|f\big\|_{\tau^{-N} L^2_{d\tau}([\tau_*, \infty))}.
		\end{align*}
		Moreover, the function $\tilde{f}$ may be chosen to depend linearly on $f$. 
	\end{lem} 
	\begin{proof} We shall set
		\[
		\tilde{f} = f(\tau) - \kappa(\tau)
		\]
		where we choose 
		\[
		\kappa(\tau) = \tau_*^{-N-\frac12}\chi\big(\frac{\tau}{\tau_*}\big).
		\]
		for a suitable $\chi\in C_0^\infty$ supported on $[\frac12, 1]$. Then extending $f$ by $0$ below $\tau = \tau_*$, we obtain the conditions 
		\begin{align*}
			\int_{-\infty}^{\infty}\tau^l \kappa(\tau)\,d\tau = \int_{-\infty}^{\infty}\tau^l f(\tau)\,d\tau,\,l = 0, 1,\ldots, N-1, 
		\end{align*}
		and we have 
		\begin{align*}
			\int_{-\infty}^{\infty}\tau^l \kappa(\tau)\,d\tau = \tau_*^{\frac12 - (N-l)}\cdot \int_{-\infty}^{\infty} x^l \chi(x)\,dx, 
		\end{align*}
		while we have the bounds
		\begin{align*}
			\big| \int_{-\infty}^{\infty}\tau^l f(\tau)\,d\tau\big|\lesssim \big\|f\big\|_{\tau^N L^2_{d\tau}}\cdot\tau_*^{l-N+\frac12},\,l = 0, 1,\ldots, N-1. 
		\end{align*}
		We conclude that it suffices to impose $N$ conditions 
		\begin{align*}
			\int_{-\infty}^{\infty} x^l \chi(x)\,dx = a_l,\,l = 0, 1, \ldots, N-1,\,|a_l|\lesssim \big\|f\big\|_{\tau^{-N} L^2_{d\tau}}\,\forall l,
		\end{align*}
		and then we have 
		\begin{align*}
			\big\|\kappa(\tau)\big\|_{\tau^{-N} L^2_{d\tau}}\lesssim_N  \big\|f\big\|_{\tau^{-N} L^2_{d\tau}}. 
		\end{align*}
	\end{proof}
	
	In the sequel, we shall write
	\begin{equation}\label{eq:Pidefinition}
		f(\tau) - \kappa(\tau) =:\Pi(f). 
	\end{equation}
	Returning to the solution of \eqref{eq:zfrelation}, we now consider a simpler model problem which we can solve by means of the standard Fourier transform. We recall the paragraph following \eqref{eq:zedcomprefined} for the definition of $\rho_1(\xi)$:
	\begin{lem}\label{lem:Fouriertransform1} Consider the equation 
		\begin{align*}
			\int_{\tau}^\infty \int_0^\infty \eta^2\cos\big([\tau-\sigma]\eta^2\big)\cdot \phi(\sigma)\rho_1(\eta)\,d\eta d\sigma = \psi(\tau). 
		\end{align*}
		Then denoting by $\hat{\phi}(\hat{\tau})$ the standard Fourier transform of $\phi$ evaluated in $\hat{\tau}$, we have 
		\begin{align*}
			\hat{\psi}(\hat{\tau}) = \Big(c_1\sqrt{|\hat{\tau}|}\rho_1(\sqrt{|\hat{\tau}|}) + ic_2\int_0^\infty \frac{\hat{\tau}\sqrt{\xi}_1}{\hat{\tau}^2 - \xi_1^2}\rho_1(\sqrt{\xi}_1)\,d\xi_1\Big)\cdot \hat{\phi}(\hat{\tau}). 
		\end{align*}
		for suitable constants $c_{1,2}\in \R\backslash\{0\}$. More generally, if we replace $\phi(\sigma)$ by $\phi(\sigma,\eta)$ depending Holder continuously on $\eta$, we have 
		\begin{align*}
			\hat{\psi}(\hat{\tau}) = \Big(c_1\sqrt{|\hat{\tau}|}\rho_1(\sqrt{|\hat{\tau}|})\cdot \hat{\phi}(\hat{\tau},|\hat{\tau}|) + ic_2\int_0^\infty \frac{\hat{\tau}\sqrt{\xi}_1\hat{\phi}(\hat{\tau},\sqrt{\xi}_1)}{\hat{\tau}^2 - \xi_1^2}\rho_1(\sqrt{\xi}_1)\,d\xi_1\Big).
		\end{align*}
		The singular integral is always in the principal value sense. 
	\end{lem}
	\begin{proof}
		Write 
		\begin{align*}
			&\int_{\tau}^\infty \int_0^\infty \eta^2\cos\big([\tau-\sigma]\eta^2\big)\cdot \phi(\sigma)\rho_1(\eta)\,d\eta d\sigma\\
			& = \sum_{\pm}\frac14\int_{-\infty}^{\infty}\int_0^\infty\big(1+H(\sigma - \tau)\big)\cdot e^{\pm i(\sigma - \tau)\tilde{\eta}}\sqrt{\tilde{\eta}}\rho_1(\sqrt{\tilde{\eta}})\phi(\sigma)\,d\tilde{\eta} d\sigma, 
		\end{align*}
		where $H$ denotes the sign function. If we apply the Fourier transform of the expression on the right, we obtain 
		\begin{align*}
			&\sum_{\pm}\Big[\frac{\pi}{2}\cdot \big(\int_0^\infty\delta_0(\hat{\tau} \mp \tilde{\eta})\sqrt{\tilde{\eta}}\rho_1(\sqrt{\tilde{\eta}})\,d\tilde{\eta}\big) -  i\frac12\cdot \int_0^\infty\frac{\sqrt{\tilde{\eta}}\rho_1(\sqrt{\tilde{\eta}})}{\hat{\tau} \mp \tilde{\eta}}\,d\tilde{\eta}\Big]\cdot \hat{\phi}(\hat{\tau}).\\
		\end{align*}
		The first part lemma follows with 
		\[
		c_1 = \frac{\pi}{2},\,c_2 = -\frac12. 
		\]
		The second part is similar. 
	\end{proof}
	
	We observe that the multiplier expressing $\hat{\psi}(\hat{\tau})$ in the preceding lemma in terms of $\hat{\phi}(\hat{\tau})$ vanishes at $\hat{\tau} = 0$ like $\log^{-2}\hat{\tau}$. It is this logarithmic 'degeneracy' which is responsible for the $\log^2\tau$-loss in Proposition~\ref{prop: solnforylambdaW}. 
	\\
	We now reduce the propagator in Proposition~\ref{prop: solnforylambdaW} to the one in the preceding lemma. For this it shall be useful to work in terms of the Fourier transform with respect to Schr\"odinger time $\tau$. The singularity at $\hat{\tau} =0$ just mentioned will force us to work with functions vanishing rapidly at the origin, which we can achieve by means of the projector $\Pi$. The following lemma describes the appropriate functions in terms of their Fourier transform: we shall use the space 
	\begin{equation}\label{eq:hatzspace}
		\hat{z}(\hat{\tau})\in \hat{\tau}^N L^2_{d\hat{\tau}}\cap W^{N, 2}. 
	\end{equation}
	We note 
	\begin{lem}\label{lem:hatzcompatibility} Assume that \eqref{eq:hatzspace} holds, with $z$ supported on $[\frac{\tau_*}{2},\infty)$. Then letting 
		\begin{align*}
			z(\tau): = \int_{-\infty}^{\infty} e^{i\hat{\tau}\tau}\cdot \hat{z}(\hat{\tau})\,d\hat{\tau}, 
		\end{align*}
		we have 
		\begin{align*}
			\big\|z\big\|_{\tau^{-N} L^2_{d\tau}}\lesssim \big\|\hat{z}\big\|_{W^{N,2}}\leq \big\|\hat{z}\big\|_{\hat{\tau}^N L^2_{d\hat{\tau}}\cap W^{N, 2}}.
		\end{align*}
		Furthermore, if $z(\tau)\in  \log^{-2}\tau\cdot \tau^{-N} L^2_{d\tau}$ and moreover $\int_{-\infty}^{\infty} \tau^l z(\tau)\,d\tau = 0,\,l = 0, 1,\ldots, N-1$, then 
		\begin{align*}
			\big\|\hat{z}\big\|_{\log^{-2}\hat{\tau}\cdot (\hat{\tau}^N L^2_{d\hat{\tau}}\cap W^{N, 2})}\lesssim \big\|z\big\|_{\log^{-2}\tau\cdot \tau^{-N} L^2_{d\tau}}.
		\end{align*}
	\end{lem}
	\begin{proof} The first part of the lemma follows from Plancherel's theorem. 
		\\
		
		Now consider the second inequality of the lemma. Write 
		\begin{align*}
			\hat{z}(\hat{\tau}) &= \int_{-\infty}^{\infty}\chi_{\tau\hat{\tau}\lesssim 1}e^{-i\tau\hat{\tau}} z(\tau)\,d\tau\\
			& + \int_{-\infty}^{\infty}\chi_{\tau\hat{\tau}\gtrsim 1}e^{-i\tau\hat{\tau}} z(\tau)\,d\tau\\
		\end{align*}
		Write the first term on the right as 
		\begin{align*}
			& \int_{-\infty}^{\infty}\chi_{\tau\hat{\tau}\lesssim 1}e^{-i\tau\hat{\tau}} z(\tau)\,d\tau\\
			& =  \int_{-\infty}^{\infty}\chi_{\tau\hat{\tau}\lesssim 1}[e^{-i\tau\hat{\tau}} - \sum_{l=0}^{N-1}\frac{(-i\tau\hat{\tau})^l}{l!}] z(\tau)\,d\tau\\
			& -  \int_{-\infty}^{\infty}\chi_{\tau\hat{\tau}\gtrsim1}\sum_{l=0}^{N-1}\frac{(-i\tau\hat{\tau})^l}{l!}z(\tau)\,d\tau.\\
		\end{align*}
		Using that 
		\begin{align*}
			\chi_{\tau\hat{\tau}\lesssim 1}[e^{-i\tau\hat{\tau}} - \sum_{l=0}^{N-1}\frac{(-i\tau\hat{\tau})^l}{l!}] = O\big((\tau\hat{\tau})^N\big), 
		\end{align*}
		we can write 
		\begin{align*}
			\int_{-\infty}^{\infty}\chi_{\tau\hat{\tau}\lesssim 1}[e^{-i\tau\hat{\tau}} - \sum_{l=0}^{N-1}\frac{(-i\tau\hat{\tau})^l}{l!}] z(\tau)\,d\tau = \log^{-2}\hat{\tau}\cdot \hat{\tau}^N\cdot g(\hat{\tau})
		\end{align*}
		where 
		\[
		\big\|g\big\|_{L^2_{d\hat{\tau}}}\lesssim \big\|z\big\|_{\log^{-2}(\tau)\cdot \tau^{-N}L^2_{d\tau}}. 
		\]
		Furthermore, by simple differentiation we infer for $j\leq N$
		\begin{align*}
			&\big\| \log^2\hat{\tau}\cdot \partial_{\hat{\tau}}^j\big[\int_{-\infty}^{\infty}\chi_{\tau\hat{\tau}\lesssim 1}[e^{-i\tau\hat{\tau}} - \sum_{l=0}^{N-1}\frac{(-i\tau\hat{\tau})^l}{l!}] z(\tau)\,d\tau\big]\big\|_{L^2_{d\hat{\tau}}}\\
			&\leq \big\| \partial_{\hat{\tau}}^j\big[\int_{-\infty}^{\infty}\chi_{\tau\hat{\tau}\lesssim 1} e^{-i\tau\hat{\tau}}  z(\tau)\,d\tau\big]\big\|_{ \log^{-2}\hat{\tau}\cdot L^2_{d\hat{\tau}}} + \big\|\partial_{\hat{\tau}}^j\big[\int_{-\infty}^{\infty}\chi_{\tau\hat{\tau}\lesssim 1}\sum_{l=0}^{N-1}\frac{(-i\tau\hat{\tau})^l}{l!} z(\tau)\,d\tau\big]\big\|_{ \log^{-2}\hat{\tau}\cdot L^2_{d\hat{\tau}}}\\
			&\lesssim \big\|z\big\|_{\log^{-2}(\tau)\cdot\tau^{-N}L^2_{d\tau}}. 
		\end{align*}
		Next, we have 
		\begin{align*}
			&\Big\|\log^2(\hat{\tau})\cdot\hat{\tau}^{-N}\int_{-\infty}^{\infty}\chi_{\tau\hat{\tau}\gtrsim1}\sum_{l=0}^{N-1}\frac{(-i\tau\hat{\tau})^l}{l!}z(\tau)\,d\tau\Big\|_{L^2_{d\hat{\tau}}}\\
			&\lesssim \sum_{l=0}^{N-1}\Big\|\log^2(\hat{\tau})\cdot\int_{-\infty}^{\infty}\chi_{\tau\hat{\tau}\gtrsim1}\frac{\tau^l \hat{\tau}^{l-N}}{l!}z(\tau)\,d\tau\Big\|_{L^2_{d\hat{\tau}}},\\
		\end{align*}
		and a simple orthogonality argument allows us to bound the last term by 
		\[
		\lesssim \big\|z\big\|_{\log^{-2}(\tau)\cdot\tau^{-N}L^2_{d\tau}}. 
		\]
		The norm 
		\begin{align*}
			\Big\|\log^2(\hat{\tau})\cdot\int_{-\infty}^{\infty}\chi_{\tau\hat{\tau}\gtrsim1}\sum_{l=0}^{N-1}\frac{(-i\tau\hat{\tau})^l}{l!}z(\tau)\,d\tau\Big\|_{W^{N,2}}
		\end{align*}
		is bounded analogously. 
		\\
		Finally, we have 
		\begin{align*}
			&\big\|\log^2\hat{\tau}\cdot\hat{\tau}^{-N}\int_{-\infty}^{\infty}\chi_{\tau\hat{\tau}\gtrsim 1}e^{-i\tau\hat{\tau}} z(\tau)\,d\tau\big\|_{L^2_{d\hat{\tau}}}\\
			&=\big\|\int_{-\infty}^{\infty}\frac{\log^2\hat{\tau}}{\log^2\tau}\cdot\frac{\chi_{\tau\hat{\tau}\gtrsim 1}}{(\tau\hat{\tau})^N}e^{-i\tau\hat{\tau}} \log^2\tau\cdot\tau^N z(\tau)\,d\tau\big\|_{L^2_{d\hat{\tau}}}\\
			&\lesssim \sum_{j\geq 0}\big\|\int_{-\infty}^{\infty}\frac{\log^2\hat{\tau}}{\log^2\tau}\cdot\frac{\chi_{\tau\hat{\tau}\sim 2^j}}{(\tau\hat{\tau})^N}e^{-i\tau\hat{\tau}}  \log^2\tau\cdot\tau^N z(\tau)\,d\tau\big\|_{L^2_{d\hat{\tau}}},\\
		\end{align*}
		and by orthogonality we have 
		\begin{align*}
			&\big\|\int_{-\infty}^{\infty}\frac{\log^2\hat{\tau}}{\log^2\tau}\cdot\frac{\chi_{\tau\hat{\tau}\sim 2^j}}{(\tau\hat{\tau})^N}e^{-i\tau\hat{\tau}}  \log^2\tau\cdot\tau^N z(\tau)\,d\tau\big\|_{L^2_{d\hat{\tau}}(0,1)}\\
			&\lesssim \Big(\sum_{k\geq 0}\big\|\int_{-\infty}^{\infty}\frac{\log^2\hat{\tau}}{\log^2\tau}\cdot\frac{\chi_{\tau\hat{\tau}\sim 2^j}}{(\tau\hat{\tau})^N}e^{-i\tau\hat{\tau}}  \log^2\tau\cdot\tau^N z(\tau)\,d\tau\big\|_{L^2_{d\hat{\tau}}(\hat{\tau}\sim 2^{-k})}^2\Big)^{\frac12}\\
			&\lesssim \Big(\sum_{k\geq 0} 2^{-(2N-1)j}\cdot\big\| \log^2\tau\cdot\tau^N z(\tau)\big\|_{L^2_{d\tau}(\tau\sim 2^{j+k})}^2\Big)^{\frac12}\\
			&\lesssim 2^{-(N-\frac12)j}\cdot \big\| \log^2\tau\cdot\tau^N z(\tau)\big\|_{L^2_{d\tau}}, 
		\end{align*}
		which can be summed over $j\geq 0$. 
		\\
		Control over the logarithmically weighted $W^{N, 2}$-norm follows in similar fashion. 
	\end{proof}
	
	We shall now reduce the solution of \eqref{eq:zfrelation} to the solution of a fixed point problem by means of the following lemma. To formulate it, introduce 
	\begin{align*}
		&  U_1z(\tau): = \int_{\tau}^\infty\int_0^\infty \chi_{\sigma - \tau\lesssim \tau}\cdot \xi^2 S_1(\tau, \sigma;\xi)\cdot z(\sigma)\rho_1(\xi)\,d\xi d\sigma\\
		& U_2z(\tau): =  \int_{\tau}^\infty\int_0^\infty \chi_{\sigma - \tau\gtrsim\tau}\cdot \xi^2 S_1(\tau, \sigma;\xi)\cdot z(\sigma)\rho_1(\xi)\,d\xi d\sigma\\
	\end{align*}
	
	\begin{lem}\label{lem:reductionsteps1} The second term $U_2z$ is perturbative in the sense that 
		\begin{align*}
			\big\|U_2z(\tau)\big\|_{\log^{-2}(\tau)\cdot\tau^{-N}L^2_{d\tau}}\lesssim (\log\tau_*)^{-1}\cdot \big\|\hat{z}\big\|_{\hat{\tau}^NL^2_{d\hat{\tau}}\cap W^{N,2}}.  
		\end{align*}
		As for the first term $ U_1z$, denoting 
		\[
		U_*z(\tau): = \int_{\tau}^\infty \int_0^\infty \xi^2\cos\big([\tau - \sigma]\xi^2\big)\cdot z(\sigma)\rho_1(\xi)\,d\xi d\sigma,
		\]
		we have 
		\begin{align*}
			\big\|U_1z(\tau)  - U_*z(\tau)\big\|_{\log^{-2}(\tau)\cdot\tau^{-N}L^2_{d\tau}}\lesssim (\log\tau_*)^{-1}\cdot \big\|\hat{z}\big\|_{\hat{\tau}^NL^2_{d\hat{\tau}}\cap W^{N,2}}.  
		\end{align*}
	\end{lem}
	\begin{proof} 
		{\it{First inequality}}. We observe that 
		\begin{align*}
			&\int_0^\infty \chi_{\sigma - \tau\gtrsim\tau}\xi^2\cos\big(\lambda^2(\tau)\xi^2\int_{\sigma}^{\tau}\lambda^{-2}(s)\,ds\big)\cdot\rho_1(\xi)\,d\xi\\
			& = -\frac12 \int_0^\infty \chi_{\sigma - \tau\gtrsim\tau}\frac{\sin\big(\lambda^2(\tau)\xi^2\int_{\sigma}^{\tau}\lambda^{-2}(s)\,ds\big)}{\lambda^2(\tau)\int_{\sigma}^{\tau}\lambda^{-2}(s)\,ds}\cdot\partial_{\xi}\big(\xi\rho_1(\xi)\big)\,d\xi\\
		\end{align*}
		The last expression can be bounded by 
		\begin{align*}
			&\Big|\frac12 \int_0^\infty \chi_{\sigma - \tau\gtrsim\tau}\frac{\sin\big(\lambda^2(\tau)\xi^2\int_{\sigma}^{\tau}\lambda^{-2}(s)\,ds\big)}{\lambda^2(\tau)\int_{\sigma}^{\tau}\lambda^{-2}(s)\,ds}\cdot\partial_{\xi}\big(\xi\rho_1(\xi)\big)\,d\xi\Big|\\
			&\lesssim \tau^{-1}\log^{-3}\tau. 
		\end{align*}
		If we then apply Lemma~\ref{lem:hatzcompatibility} as well as Schur's test, the first inequality of the lemma easily follows. 
		\\
		
		For the second inequality, write for $0\leq \sigma-\tau<\tau$ 
		\[
		\lambda^2(\tau)\cdot\int_{\tau}^{\sigma}\lambda^{-2}(s)\,ds = (\sigma - \tau)\cdot \big(1+s(\tau, \sigma)\cdot \frac{\sigma - \tau}{\tau}\big)^{-1},
		\]
		where the function $s(\tau, \sigma)$ is uniformly bounded and has symbol type behavior with respect to its arguments. Taking advantage of a simple change of variables, we now infer 
		\begin{align*}
			U_1z(\tau) = \int_{\tau}^\infty\int_0^\infty z(\sigma)\chi_{\sigma - \tau\lesssim \tau}\eta\cos\big([\sigma - \tau]\eta^2\big)\cdot \zeta^2(\tau, \sigma)\cdot \tilde{\rho}\big(\eta\cdot \zeta(\tau, \sigma)\big)\,d\eta,
		\end{align*}
		where we set $\tilde{\rho}(\xi): = \xi\cdot\rho(\xi)$, $\zeta(\tau, \sigma): =  \sqrt{\big(1+s(\tau, \sigma)\cdot \frac{\sigma - \tau}{\tau}\big)}$. Then note that 
		\begin{align*}
			\tilde{\rho}\big(\eta\cdot \zeta(\tau, \sigma)\big) - \tilde{\rho}(\eta) = \frac{\frac{\sigma - \tau}{\tau}}{\nu(\eta, \tau, \sigma)}, 
		\end{align*}
		where we have the symbol type bounds 
		\begin{align*}
			\big|\partial_{\eta}^{j+1}\big(\frac{1}{\nu(\eta, \tau, \sigma)}\big)\big|\lesssim \frac{1}{\eta^j\cdot\langle\log\rangle^{4+j}\eta},\,0<\eta\lesssim 1,\,j\geq 0. 
		\end{align*}
		Consider then then difference 
		\begin{align*}
			&U_1z(\tau) -  U_*z(\tau)\\& = \int_{\tau}^\infty\int_0^\infty z(\sigma)\chi_{\sigma - \tau\lesssim \tau}\eta\cos\big([\sigma - \tau]\eta^2\big)\cdot [\zeta^2(\tau, \sigma)-1]\cdot \tilde{\rho}_1\big(\eta,\tau, \sigma\big)\,d\eta\\
			& - \int_{\tau}^\infty\int_0^\infty z(\sigma) \chi_{\sigma - \tau\gtrsim \tau}\eta\cos\big([\sigma - \tau]\eta^2\big)\cdot \tilde{\rho}(\eta)\,d\eta\\
		\end{align*}
		where the function $\tilde{\rho}_1$ satisfies the same symbol bounds as does $\tilde{\rho}$ with respect to $\eta$. To estimate the first term on the right, we note that using integration by parts with respect to $\eta$ as in the proof of the first inequality of the lemma and exploiting the bound 
		\[
		\big|\chi_{\sigma - \tau\lesssim \tau^{1-\epsilon}}\cdot [\zeta^2(\tau, \sigma)-1]\big|\lesssim \tau^{-\epsilon}, \epsilon>0, 
		\]
		we infer the bound 
		\begin{align*}
			&\Big\|\int_{\tau}^\infty\int_0^\infty z(\sigma)\chi_{\sigma - \tau\lesssim \tau^{1-\epsilon}}\eta\cos\big([\sigma - \tau]\eta^2\big)\cdot [\zeta^2(\tau, \sigma)-1]\cdot \tilde{\rho}_1\big(\eta,\tau, \sigma\big)\,d\eta\Big\|_{\tau^{-N-\epsilon+}L^2_{d\tau}}\\
			&\lesssim \big\|z\big\|_{\tau^{-N}L^2_{d\tau}}, 
		\end{align*}
		which in light of Lemma~\ref{lem:hatzcompatibility}  is even better than what we need. As for the term with the cutoff $\chi_{\sigma - \tau\gtrsim\tau^{1-\epsilon}}$
		as well as  the term 
		\[
		\int_{\tau}^\infty\int_0^\infty z(\sigma) \chi_{\sigma - \tau\gtrsim \tau}\eta\cos\big([\sigma - \tau]\eta^2\big)\cdot \tilde{\rho}(\eta)\,d\eta,
		\]
		they are handled analogously to $U_2z(\tau)$. 
	\end{proof}
	
	We now write \eqref{eq:zfrelation} in the form 
	\begin{equation}\label{eq:Ustarfixedpoint}
		U_*z(\tau) = \Pi\big(U_* - U_1\big)z(\tau) -\Pi U_2 z(\tau) + \Pi(f)(\tau),
	\end{equation}
	valid as long as $\tau\in [\tau_*, \infty)$.
	Assuming $f(\tau)\in \log^{-2}(\tau)\cdot \tau^{-N}L^2_{d\tau}$, and taking advantage of Lemma~\ref{lem:reductionsteps1} as well as Lemma~\ref{lem:hatzcompatibility}, we find 
	\begin{align*}
		&\big\|\mathcal{F}_{\tau}\big(\Pi\big(U_* - U_1\big)z\big)\big\|_{\log^{-2}\hat{\tau}\cdot(\hat{\tau}^NL^2_{d\hat{\tau}}\cap W^{N,2})}\\&\hspace{2cm}+ \big\|\mathcal{F}_{\tau}\big(\Pi (U_2 z)\big)\big\|_{\log^{-2}\hat{\tau}\cdot(\hat{\tau}^NL^2_{d\hat{\tau}}\cap W^{N,2})}\ll_{\tau_*}\big\|z\big\|_{\tau^{-N}L^2_{d\tau}}. 
	\end{align*}
	If we then define the inverse operator $U_*^{-1}$ by division by the multiplier in Lemma~\ref{lem:Fouriertransform1} on the Fourier side, we get 
	\begin{align*}
		z = U_*^{-1}\circ  \Pi\big(U_* - U_1\big)z(\tau) -U_*^{-1}\circ \Pi U_2 z(\tau) + U_*^{-1}\circ\Pi(f)(\tau),
	\end{align*}
	where we have the bound 
	\[
	\big\|U_*^{-1}\circ  \Pi\big(U_* - U_1\big)z(\tau)\big\|_{\tau^{-N}L^2_{d\tau}}\ll_{\tau_*}\big\|z\big\|_{\tau^{-N}L^2_{d\tau}},
	\]
	and similarly for the second term on the right, while we also have the estimate 
	\begin{align*}
		\big\|U_*^{-1}\circ\Pi(f)(\tau)\big\|_{\tau^{-N}L^2_{d\tau}}\lesssim \big\|f\big\|_{\log^{-2}(\tau)\cdot \tau^{-N}L^2_{d\tau}}.
	\end{align*}
	
	The proof of Proposition~\ref{prop: solnforylambdaW} is now completed in the case $f(\tau)\in \log^{-2}(\tau)\cdot \tau^{-N}$ by using a fixed point argument in $\tau^{-N}L^2_{d\tau}$ for $z$. The case $f\in \tau^{-N}L^2_{d\tau}$ is analogous.

	\subsection{The proof of Proposition~\ref{prop:Phitildelambdasolution}} The key point will be to reduce the implicitly defined quantity $\tilde{y}_{\tilde{\lambda}}^{\text{mod}}$ (recall \eqref{eq:ytildemodmodified}) to a principal term, contributed by the first two terms in \eqref{eq:tildeE2mod}, as well as a number of perturbative terms. Moreover, a key step shall consist in simplifying the propagator $U(\tilde{\tau}, \tilde{\sigma},\xi)$ as defined in \eqref{eq:wavepropagator} to the standard wave propagator which is amenable to simple Fourier techniques. The following lemmas provide the key reduction steps:
	
	\begin{lem}\label{lem:reductionsteps2} Let 
		\[
		U_*(\tilde{\tau},\tilde{\sigma}, \xi): = \frac{\sin\big([\tilde{\sigma} - \tilde{\tau}]\xi\big)}{\xi}
		\]
		the inhomogeneous free wave propagator on the (free) Fourier side,  while we recall \eqref{eq:wavepropagator} for the dynamically dilated wave propagator. Then we have the estimate 
		\begin{align*}
			&\Big\|\langle\int_{\tilde{\tau}}^\infty \int_0^\infty \lambda^{-2}(\tilde{\tau})\big[U(\tilde{\tau}, \tilde{\sigma}, \xi) - U_*(\tilde{\tau},\tilde{\sigma}, \xi)\big]\cdot \phi_{\R^4}(R;\xi)\\&\hspace{2cm}\cdot \phi(\tilde{\sigma})\mathcal{F}_{\R^4}\big(\lambda^2(\tilde{\sigma})\Lambda W\cdot W\big)(\frac{\lambda(\tilde{\tau})}{\lambda(\tilde{\sigma})}\xi)\rho_{\R^4}(\xi)\,d\xi d\tilde{\sigma},\,W^2\rangle_{L^2_{R^3\,dR}}\Big\|_{\tau^{-N-\frac12+\frac{1}{2\nu}}L^2_{d\tau}}\\
			&\lesssim \big\|\phi\big\|_{\tau^{-N}L^2_{d\tau}}. 
		\end{align*}
		Furthermore, we can write 
		\begin{align*}
			&\langle\int_{\tilde{\tau}}^\infty \int_0^\infty \lambda^{-2}(\tilde{\tau})\big[U(\tilde{\tau}, \tilde{\sigma}, \xi) - U_*(\tilde{\tau},\tilde{\sigma}, \xi)\big]\cdot \phi_{\R^4}(R;\xi)\\&\hspace{2cm}\cdot \phi(\tilde{\sigma})\mathcal{F}_{\R^4}\big(\lambda^2(\tilde{\sigma})\Lambda W\cdot W\big)(\frac{\lambda(\tilde{\tau})}{\lambda(\tilde{\sigma})}\xi)\rho_{\R^4}(\xi)\,d\xi d\tilde{\sigma},\,W^2\rangle_{L^2_{R^3\,dR}} = F_1 + \partial_{\tau}F_2, 
		\end{align*}
		where we have the bounds 
		\begin{align*}
			&\Big\|\langle\partial_{\tilde{\tau}}^2\rangle F_1\Big\|_{\tau^{-N-\frac12+O(\frac{1}{\nu})}L^2_{d\tau}}\lesssim \big\|\langle\partial_{\tilde{\sigma}}^2\rangle^{-1}\phi\big\|_{\sigma^{-N}L^2_{d\sigma}}\\
			&\big\|F_2\big\|_{\tau^{-N-\frac{1}{\nu}}L^2_{d\tau}}\lesssim \big\|\langle\partial_{\tilde{\sigma}}^2\rangle^{-1}\phi\big\|_{\sigma^{-N}L^2_{d\sigma}}\\
		\end{align*}
		
	\end{lem}
	\begin{proof}(sketch) For the first part of the lemma, we observe to begin with that 
		\begin{align*}
			\Big|U(\tilde{\tau}, \tilde{\sigma}, \xi) - U_*(\tilde{\tau},\tilde{\sigma}, \xi)\Big|\lesssim\frac{\langle\xi\rangle}{\xi}\cdot \frac{\tilde{\sigma} - \tilde{\tau}}{\tilde{\tau}}.
		\end{align*}
		Note that the Fourier coefficient $\mathcal{F}_{\R^4}\big(\lambda^2(\tilde{\sigma})\Lambda W\cdot W\big)(\xi)$ decays rapidly with respect to $\xi\gg 1$, so in effect we may restrict to frequencies $\xi\lesssim 1$. 
		Then we split the $R$ integral in the expression of the lemma into three regions:
		\\
		
		{\it{$R\sim \tilde{\sigma} - \tilde{\tau}$.}} Recalling the asymptotic structure(see subsection~\ref{subse:waveparametrix}
		) of $\phi_{\R^4}(R;\xi)$, we arrive at the schematically written $R$- integral 
		\begin{align*}
			\int_0^\infty \chi_{R\sim \tilde{\sigma} - \tilde{\tau}}\cdot \frac{e^{\pm iR\xi}}{R^{\frac32}\xi^{\frac32}}\cdot R^{-4}\cdot R^3\,dR. 
		\end{align*}
		The spectral measure $\rho_{\R^4}(\xi)\sim \xi^3$. We have to 'spend' one factor $R^{-1}$ to force integrability in $\tilde{\sigma}$, and we shall gain additional factors $R^{-1}\sim (\tilde{\sigma} - \tilde{\tau})^{-1}$ by means of integration by parts (which 'costs' factors $\xi^{-1}$). Precisely, we integrate by parts twice with respect to $R$, which 'costs' $\xi^{-2}$, and leads to
		\begin{align*}
			\Big|\int_0^\infty \chi_{R\sim \tilde{\sigma} - \tilde{\tau}}\cdot \frac{e^{\pm iR\xi}}{R^{\frac32}\xi^{\frac32}}\cdot R^{-4}\cdot R^3\,dR\Big|\lesssim \xi^{-2}\cdot(\tilde{\sigma} - \tilde{\tau})^{-2}. 
		\end{align*}
		Combining all these observations, it is easily seen that the contribution of this case to the integral in the lemma is in $\tau^{-N-\frac12+\frac{1}{2\nu}}L^2_{d\tau}$. This means in effect we gain $\tilde{\tau}^{-1}$ decay. 
		\\
		
		{\it{$R\ll \tilde{\sigma} - \tilde{\tau}$.}} Here we first reduce the variable $R$ in the inside expression $\Lambda W\cdot W$ to size $\ll  \tilde{\sigma} - \tilde{\tau}$. In fact, observe that 
		\begin{align*}
			\Big|\mathcal{F}_{\R^4}\big(\chi_{R\gtrsim \tilde{\sigma} - \tilde{\tau}}\Lambda W\cdot W\big)(\frac{\lambda(\tilde{\tau})}{\lambda(\tilde{\sigma})}\xi)\Big|\lesssim \xi^{-2}(\tilde{\sigma }- \tilde{\tau})^{-2}, 
		\end{align*}
		and from here one completes the estimate as in the preceding case. On the other hand, if we include a cutoff $\chi_{R\ll\tilde{\sigma} - \tilde{\tau}}$ in front of $\Lambda W\cdot W$, we can perform integration by parts in the $\xi$ integral twice, thereby gaining 
		\[
		\xi^{-2}(\tilde{\sigma }- \tilde{\tau}-R)^{-2}\sim \xi^{-2}(\tilde{\sigma }- \tilde{\tau})^{-2}, 
		\]
		and the estimate is again completed as before. 
		\\
		
		{\it{$R\gg \tilde{\sigma} - \tilde{\tau}$.}} This case is handled analogously to the first one. 
		\\
		
		To see the last assertion of the lemma, we decompose $\phi$ into a low-frequency part and a high frequency part with respect to wave time, and specifically
		\begin{align*}
			\phi(\sigma) = Q^{(\tilde{\sigma})}_{<\sigma^{10\nu^{-1}}}\phi + Q^{(\tilde{\sigma})}_{\geq\sigma^{10\nu^{-1}}}\phi. 
		\end{align*}
		The contribution of the first term can be handled by means of the already proven part of the lemma. As for the contribution of the second term, we perform integration by parts with respect to $\tilde{\sigma}$ twice, which replaces this term by 
		\begin{align*}
			\partial_{\tilde{\sigma}}^{-2}Q^{(\tilde{\sigma})}_{\geq\sigma^{10\nu^{-1}}}\phi
		\end{align*}
		and produces boundary terms which are easily seen to be of the form $\partial_{\tau}F_2$ as claimed in the lemma, or double integrals which can be included into $F_1$. 
		
	\end{proof}
	
	The following lemma follows the same pattern of proof: 
	\begin{lem}\label{lem:reductionsteps3} We have the estimate 
		\begin{align*}
			&\Big\|\langle\int_{\tilde{\tau}}^\infty \int_0^\infty \lambda^{-2}(\tilde{\tau})U_*(\tilde{\tau},\tilde{\sigma}, \xi)\cdot \phi_{\R^4}(R;\xi) \phi(\tilde{\sigma})\\&\cdot\big[\mathcal{F}_{\R^4}\big(\lambda^2(\tilde{\sigma})\Lambda W\cdot W\big)(\frac{\lambda(\tilde{\tau})}{\lambda(\tilde{\sigma})}\xi) - \lambda^2(\tilde{\tau})\mathcal{F}_{\R^4}\big(\Lambda W\cdot W\big)(\xi)\big]\rho_{\R^4}(\xi)\,d\xi d\tilde{\sigma},\,W^2\rangle_{L^2_{R^3\,dR}}\Big\|_{\tau^{-N-\frac12+\frac{1}{2\nu}}L^2_{d\tau}}\\
			&\lesssim \big\|\phi\big\|_{\tau^{-N}L^2_{d\tau}}. 
		\end{align*}
		Furthermore, we can write 
		\begin{align*}
			&\langle\int_{\tilde{\tau}}^\infty \int_0^\infty \lambda^{-2}(\tilde{\tau})U_*(\tilde{\tau},\tilde{\sigma}, \xi)\cdot \phi_{\R^4}(R;\xi) \phi(\tilde{\sigma})\\&\cdot\big[\mathcal{F}_{\R^4}\big(\lambda^2(\tilde{\sigma})\Lambda W\cdot W\big)(\frac{\lambda(\tilde{\tau})}{\lambda(\tilde{\sigma})}\xi) - \lambda^2(\tilde{\tau})\mathcal{F}_{\R^4}\big(\Lambda W\cdot W\big)(\xi)\big]\rho_{\R^4}(\xi)\,d\xi d\tilde{\sigma},\,W^2\rangle_{L^2_{R^3\,dR}} \\
			&= F_1 + \partial_{\tau}F_2, 
		\end{align*}
		where we have the bounds 
		\begin{align*}
			&\Big\|\langle\partial_{\tilde{\tau}}^2\rangle F_1\Big\|_{\tau^{-N-\frac12+O(\frac{1}{\nu})}L^2_{d\tau}}\lesssim \big\|\langle\partial_{\tilde{\sigma}}^2\rangle^{-1}\phi\big\|_{\sigma^{-N}L^2_{d\sigma}}\\
			&\big\|F_2\big\|_{\tau^{-N-\frac{1}{\nu}}L^2_{d\tau}}\lesssim \big\|\langle\partial_{\tilde{\sigma}}^2\rangle^{-1}\phi\big\|_{\sigma^{-N}L^2_{d\sigma}}\\
		\end{align*}
	\end{lem}
	
	To continue the eventual determination of the variable $\tilde{\lambda}$, we now consider the simplified model equation 
	\begin{equation}\label{eq:modelfortildelambda}
		\mathcal{F}\big( y^{**}_{\tilde{\lambda}}\cdot W\big)(\tilde{\tau}, 0) = f(\tilde{\tau}),
	\end{equation}
	where $y^{**}_{\tilde{\lambda}}$ is defined via the free wave propagator as\footnote{Recall \eqref{eq:tildelambdaleadcontrib}
		.} 
	\begin{equation}\label{eq:ydoublestardef}\begin{split}
			y^{**}_{\tilde{\lambda}} &= -\Box_{\tilde{\tau}, R}^{-1}\big((\tilde{\lambda}_{\tilde{\tau}\tilde{\tau}}\cdot \Lambda W\cdot W\big)\\
			& - \Box_{\tilde{\tau}, R}^{-1}\big(Q^{(\tilde{\tau})}_{<\tilde{\tau}^{\frac{10}{\nu}}}\big(\frac{\tilde{\lambda}_{\tilde{\tau}}}{\tilde{\tau}}\cdot c\Lambda(\Lambda W\cdot W)\big)\big),\,c = c(\nu)
	\end{split}\end{equation}
	and we use the notation $ c\Lambda(\Lambda W\cdot W) :=\lambda^{-2}(t\partial_t)\big(\lambda^2 \Lambda W(\lambda r)\cdot W(\lambda r)\big)$
	\[
	\Box_{\tilde{\tau}, R}: = -\partial_{\tilde{\tau}\tilde{\tau}} + \partial_{RR} + \frac{3}{R}\partial_R. 
	\]
	We shall  resort to Fourier methods to solve \eqref{eq:modelfortildelambda}. Write 
	\begin{align*}
		\mathcal{F}\big(y^{**}_{\tilde{\lambda}}\cdot W\big)(\tilde{\tau}, 0) = \int_0^\infty y^{**}_{\tilde{\lambda}}\cdot W^2 R^3\,dR. 
	\end{align*}
	
	\begin{lem}\label{lem:Fouriertransform2} Denoting by $\hat{f}= \mathcal{F}_{\tilde{\tau}}(f)$ the (standard) one dimensional Fourier transform with respect to the wave time $\tilde{\tau}$, and using the notation $\hat{\tilde{\tau}}$ for the corresponding frequency,  the equation  
		\[
		\mathcal{F}\big( y^{**}_{\tilde{\lambda}}\cdot W\big)(\tilde{\tau}, 0) = f(\tilde{\tau})
		\]
		can be written as 
		\begin{align*}
			&\mathcal{F}_{\tilde{\tau}}\big(\tilde{\lambda}_{\tilde{\tau}\tilde{\tau}}\big) + c_3(\hat{\tilde{\tau}},\nu)\cdot \mathcal{F}_{\tilde{\tau}}\big(\frac{\tilde{\lambda}_{\tilde{\tau}}}{\tilde{\tau}}\big)(\hat{\tilde{\tau}})\cdot\big[ic_1\cdot \mathcal{F}_{\R^4}\big(W^2\big)(|\hat{\tilde{\tau}}|)\cdot \mathcal{F}_{\R^4}\big(\Lambda W\cdot W\big)(|\hat{\tilde{\tau}}|)\cdot \frac{\rho_{\R^4}(|\hat{\tilde{\tau}}|)}{\hat{\tilde{\tau}}}\\&\hspace{3cm} - c_2 \int_{0}^\infty \frac{1}{\hat{\tilde{\tau}}^2 - \xi^2}\cdot \mathcal{F}_{\R^4}\big(W^2\big)(\xi)\cdot \mathcal{F}_{\R^4}\big(\Lambda W\cdot W\big)(\xi)\rho_{\R^4}(\xi)\,d\xi\big]\\
			& = \hat{f}(\hat{\tilde{\tau}}) + \delta\zeta(\hat{\tilde{\tau}})
		\end{align*}
		Here $c_{1,2}$ are non-vanishing real constants(in fact, $c_1 = \frac\pi2, c_2 = \frac{1}{2}$), while $c_3$ is as in the statement of Proposition~\ref{prop:Phitildelambdasolution}, and we can write 
		\begin{align*}
			\delta\zeta(\hat{\tilde{\tau}}) = \hat{\tilde{\tau}}^{-2}\cdot \kappa(\hat{\tilde{\tau}})\cdot \mathcal{F}_{\tilde{\tau}}\big(Q^{(\tilde{\tau})}_{<\tilde{\tau}^{\frac{10}{\nu}}}\big(\frac{\tilde{\lambda}_{\tilde{\tau}}}{\tilde{\tau}}\big)\big)(\hat{\tilde{\tau}}),
		\end{align*}
		where the function $ \kappa(\hat{\tilde{\tau}})\in C^\infty(\R_+)$ is bounded, supported away from $\hat{\tilde{\tau}} = 0$, and satisfies the conjugation symmetry relation $\kappa(-\hat{\tilde{\tau}}) = \overline{\kappa(\hat{\tilde{\tau}})}$, as well as symbol type bounds. 
	\end{lem}
	\begin{proof} Write 
		\begin{align*}
			-\mathcal{F}\big( y^{**}_{\tilde{\lambda}}\cdot W\big)(\tilde{\tau}, 0)  &= \int_{\tilde{\tau}}^{\infty}\tilde{\lambda}_{\tilde{\sigma}\tilde{\sigma}}\cdot\int_0^\infty \frac{\sin\big([\tilde{\sigma} - \tilde{\tau}]\xi\big)}{\xi}\cdot g(\xi)\,d\xi d\tilde{\sigma}\\
			& +  \int_{\tilde{\tau}}^{\infty}Q^{(\tilde{\sigma})}_{<\tilde{\sigma}^{\frac{10}{\nu}}}\big(\frac{\tilde{\lambda}_{\tilde{\sigma}}}{\tilde{\sigma}}\big)\cdot\int_0^\infty \frac{\sin\big([\tilde{\sigma} - \tilde{\tau}]\xi\big)}{\xi}\cdot \tilde{g}(\xi)\,d\xi d\tilde{\sigma}\\
		\end{align*}
		where we set (recall \eqref{eq:ydoublestardef})
		\begin{align*}
			&g(\xi): = \mathcal{F}_{\R^4}\big(\Lambda W\cdot W\big)(\xi)\rho_{\R^4}(\xi)\cdot \int_0^\infty \phi_{\R^4}(R;\xi)\cdot W^2\cdot R^3\,dR,\\
			&\tilde{g}(\xi): = c\mathcal{F}_{\R^4}\big(\Lambda(\Lambda W\cdot W)\big)(\xi)\rho_{\R^4}(\xi)\cdot \int_0^\infty \phi_{\R^4}(R;\xi)\cdot W^2\cdot R^3\,dR
		\end{align*}
		Interpret the time integrals as convolution of the functions $\tilde{\lambda}_{\tilde{\sigma}\tilde{\sigma}}, Q^{(\tilde{\sigma})}_{<\tilde{\sigma}^{\frac{10}{\nu}}}\big(\frac{\tilde{\lambda}_{\tilde{\sigma}}}{\tilde{\sigma}}\big)$ with the function\footnote{Recall that $H$ denotes the sign function.} 
		\[
		\frac12 (1+H)\cdot\sin (\tilde{\sigma}\xi),
		\]
		and proceed in analogy to the proof of Lemma~\ref{lem:Fouriertransform1}. It then suffices to set 
		\begin{align*}
			c_3: = \chi_{|\hat{\tilde{\tau}}|\lesssim 1}\cdot \frac{\beta_1(\hat{\tilde{\tau}},\nu)}{\beta_2(\hat{\tilde{\tau}})}, 
		\end{align*}
		where the function $\chi_{|\hat{\tilde{\tau}}|\lesssim 1}$ is a smooth cutoff, and further (recall \eqref{eq:ydoublestardef})
		\begin{equation}\label{eq:beta1}\begin{split}
				\beta_1(\hat{\tilde{\tau}},\nu) &= ic\cdot c_1\cdot \mathcal{F}_{\R^4}\big(W^2\big)(|\hat{\tilde{\tau}}|)\cdot \mathcal{F}_{\R^4}\big(\Lambda(\Lambda W\cdot W)\big)(|\hat{\tilde{\tau}}|)\cdot \frac{\rho_{\R^4}(|\hat{\tilde{\tau}}|)}{\hat{\tilde{\tau}}}\\&\hspace{1cm} - c\cdot c_2 \int_{0}^\infty \frac{1}{\hat{\tilde{\tau}}^2 - \xi^2}\cdot \mathcal{F}_{\R^4}\big(W^2\big)(\xi)\cdot \mathcal{F}_{\R^4}\big(\Lambda(\Lambda W\cdot W)\big)(\xi)\rho_{\R^4}(\xi)\,d\xi
		\end{split}\end{equation}
		\begin{equation}\label{eq:beta2}\begin{split}
				\beta_2(\hat{\tilde{\tau}}) &= ic_1\cdot \mathcal{F}_{\R^4}\big(W^2\big)(|\hat{\tilde{\tau}}|)\cdot \mathcal{F}_{\R^4}\big(\Lambda W\cdot W\big)(|\hat{\tilde{\tau}}|)\cdot \frac{\rho_{\R^4}(|\hat{\tilde{\tau}}|)}{\hat{\tilde{\tau}}}\\&\hspace{1cm} - c_2 \int_{0}^\infty \frac{1}{\hat{\tilde{\tau}}^2 - \xi^2}\cdot \mathcal{F}_{\R^4}\big(W^2\big)(\xi)\cdot \mathcal{F}_{\R^4}\big(\Lambda W\cdot W\big)(\xi)\rho_{\R^4}(\xi)\,d\xi,
		\end{split}\end{equation}
		see Lemma~\ref{lem:FourierNV1}, ~\ref{lem:FourierNV2}, and numerical assumption {\bf{(C2)}}. Finally it suffices to define $ \kappa(\hat{\tilde{\tau}})$ by means of 
		\[
		\kappa(\hat{\tilde{\tau}}) =  \chi_{|\hat{\tilde{\tau}}|\gtrsim1}\hat{\tilde{\tau}}^2\cdot \beta_1(\hat{\tilde{\tau}},\nu)
		\]
	\end{proof}
	
	Next, we take into account the precise definition of $\Phi^{(\tilde{\lambda})}$, which also involves the high-frequency term (see \eqref{eq:Phitildelambdadef})
	\[
	-c_*\lambda^{-2}\mathcal{F}\big(\partial_{\tilde{\sigma}}^{-2}\triangle\big(\lambda^2Q^{(\tilde{\sigma})}_{\gamma^{-1}<\cdot<\sigma^{\frac12+}}\tilde{\lambda}W^2\big)\cdot W\big)(\sigma, 0)  .
	\]
	To simplify this term, we note the simple 
	\begin{lem}\label{lem:Phitildelambdahighfreqsimplfied} Letting $X_1(\sigma)$ denote the preceding expression, and $\tilde{X}_1$ the same expression except without the factors $\lambda^{-2}, \lambda^2$, then we have the bound 
		\begin{align*}
			\Big\|\langle\partial_{\tilde{\sigma}}^2\rangle (X_1 - \tilde{X}_1)\Big\|_{\sigma^{-N}L^2_{d\sigma}}\ll_{\tau_*}\big\|\langle\partial_{\tilde{\sigma}}^2\rangle^{-1}\partial_{\tilde{\sigma}}^2\tilde{\lambda}\big\|_{\sigma^{-N}L^2_{d\sigma}}. 
		\end{align*}
		We can also suppress the upper frequency localizer $Q_{\cdot <\sigma^{\frac12+}}$ in $X_1$ at the expense of an error term of the form 
		\[
		\partial_{\sigma}E,\,\big\|E\big\|_{\sigma^{-N-}L^2_{d\sigma}}\lesssim \big\|\langle\partial_{\tilde{\sigma}}^2\rangle^{-1}\partial_{\tilde{\sigma}}^2\tilde{\lambda}\big\|_{\sigma^{-N}L^2_{d\sigma}}.
		\]
	\end{lem}
	\begin{proof} Referring to $X_1$, we can include a frequency cutoff $Q^{(\tilde{\sigma})}_{\gamma^{-1-}<\cdot<\sigma^{\frac12+}}$ to the left of $\triangle$ up to an error of size $O(\sigma^{-M})$ for any $M\gg 1$. The operator 
		\[
		\langle \partial_{\tilde{\sigma}}^2\rangle \partial_{\tilde{\sigma}}^{-2}Q^{(\tilde{\sigma})}_{\gamma^{-1-}<\cdot<\sigma^{\frac12+}}
		\]
		is given by convolution with a function decaying rapidly beyond scale $\gamma\ll 1$ and of $L^1_{d\tilde{\sigma}}$-mass $\lesssim 1$. The first part of the lemma follows easily from this. For the second part, observe that 
		\[
		\partial_{\sigma} = \frac{\partial\tilde{\sigma}}{\partial\sigma}\cdot \partial_{\tilde{\sigma}}, 
		\]
		and application of $ \partial_{\tilde{\sigma}}^{-1}$ to a function at $\tilde{\sigma}$-frequency $>\sigma^{\frac12+}$ gains $\sigma^{-\frac12-}$. 
	\end{proof}
	The preceding lemmas suggest that we replace $\Phi^{(\tilde{\lambda})}$ by the simpler expression 
	\begin{equation}\label{eq:Phitildelambdamodel}\begin{split}
			\Phi^{(\tilde{\lambda})}_{\text{model}} =\mathcal{F}\big(y_{\tilde{\lambda}}^{**}\cdot W\big)(\sigma, 0)-c_*\mathcal{F}\big(\partial_{\tilde{\sigma}}^{-2}\triangle\big(Q^{(\tilde{\sigma})}_{\gamma^{-1}<\cdot}\tilde{\lambda}W^2\big)\cdot W\big)(\sigma, 0) 
	\end{split}\end{equation}
	To formulate the next result, we need to introduce the projection operator $\Pi^{(\tilde{\tau})}$:
	\begin{lem}\label{lem:Pitildetaudef} Let $f\in \tau^{-N}L^2_{d\tau}$, and let $M = \frac{N}{\frac12-\frac{1}{4\nu}}$. Then there exists 
		\[
		\tilde{f}\in \tau^{-N}L^2_{d\tau}([\frac{\tau_*}{2},\infty))
		\]
		with $\tilde{f}|_{[\tau_*,\infty)} = f|_{[\tau_*,\infty)}$, and such that 
		\[
		\partial_{\hat{\tilde{\tau}}}^l\mathcal{F}_{\tilde{\tau}}(\tilde{f})(0) = 0,\,l = 0, 1,\ldots, \lfloor M\rfloor. 
		\]
		Furthermore, we have the bounds
		\begin{align*}
			&\big\|\tilde{f}\big\|_{\tau^{-N}L^2_{d\tau}}\lesssim _N \big\|f\big\|_{\tau^{-N}L^2_{d\tau}},\\
			&\big\|\mathcal{F}_{\tilde{\tau}}\tilde{f}\big\|_{W^{M,2}\cap \hat{\tilde{\tau}}^M L^2_{d\hat{\tilde{\tau}}}}\lesssim _N \big\|f\big\|_{\tau^{-N}L^2_{d\tau}},
		\end{align*}
		where as usual we denote the Fourier variable associated to the wave time $\tilde{\tau}$ by $\hat{\tilde{\tau}}$. 
		The function $\tilde{f}$ can be chosen to depend linearly on $f$, and we set 
		\[
		\tilde{f} = : \Pi^{(\tilde{\tau})}(f).
		\]
	\end{lem}
	\begin{proof} Analogous to the one of Lemma~\ref{lem:fmodif}. 
		
	\end{proof}
	
	\begin{lem}\label{lem:Phitildelambdamodeleqn} The equation 
		\begin{align*}
			\Phi^{(\tilde{\lambda})}_{\text{model}}(\tilde{\tau}) = f(\tilde{\tau})
		\end{align*}
		admits a solution $\tilde{\lambda}$ on $[\tau_*,\infty)$, which can be written on the Fourier side in the form
		\begin{align*}
			\mathcal{F}_{\tilde{\tau}}\big(\tilde{\lambda}_{\tilde{\tau}\tilde{\tau}}\big)(\hat{\tilde{\tau}}) + c_3(\hat{\tilde{\tau}},\nu)\cdot\mathcal{F}_{\tilde{\tau}}\big(\frac{\tilde{\lambda}_{\tilde{\tau}}}{\tilde{\tau}}\big)(\hat{\tilde{\tau}}) = \langle\hat{\tilde{\tau}}^4\rangle\cdot \beta(\hat{\tilde{\tau}})\cdot\mathcal{F}_{\tilde{\tau}}\big(\Pi^{(\tilde{\tau})}f\big)(\hat{\tilde{\tau}}) + \delta\zeta(\hat{\tilde{\tau}}),
		\end{align*}
		where we have the bounds 
		\begin{align*}
			&\big\|\langle\partial_{\tilde{\tau}}^2\rangle^{-1}\partial_{\tilde{\tau}}^2\tilde{\lambda}\big\|_{\tau^{-N}L^2_{d\tau}}\lesssim_N \big\|\langle\partial_{\tilde{\tau}}^2\rangle f\big\|_{\tau^{-N}L^2_{d\tau}},\\
			&\big\|\langle\partial_{\tilde{\tau}}^2\rangle\mathcal{F}_{\tilde{\tau}}^{-1}(\delta\zeta)\big\|_{\tau^{-N}L^2_{d\tau}}\ll_{\tau_*}\big\|\langle\partial_{\tilde{\tau}}^2\rangle f\big\|_{\tau^{-N}L^2_{d\tau}}.\\
		\end{align*}
		The complex valued function $\beta(\hat{\tilde{\tau}})$ is smooth on $\R\backslash \{0\}$, satisfies the conjugation symmetry 
		\[
		\beta(-\hat{\tilde{\tau}}) = \overline{\beta(\hat{\tilde{\tau}})},
		\]
		and is bounded from above and from belpw, all subject to by a positive constant subject to numerical non-degeneracy assumption {\bf{(C3)}}. Furthermore, its imaginary part is nonzero on $\R\backslash\{0, \pm\tau_*\}$ for some $\tau_*\in \R_+$. The function $c_3$ is as in Lemma~\ref{lem:Fouriertransform2}.
	\end{lem}
	\begin{proof} To begin with, we compute the Fourier transform of the second expression on the right in \eqref{eq:Phitildelambdamodel}:
		\begin{align*}
			&\mathcal{F}_{\tilde{\tau}}\Big(c_*\mathcal{F}\big(\partial_{\tilde{\tau}}^{-2}\triangle\big(Q^{(\tilde{\tau})}_{\gamma^{-1}<\cdot<\tau^{\frac12+}}\tilde{\lambda}W^2\big)\cdot W\big)(\cdot, 0)\Big)(\hat{\tilde{\tau}})\\
			& = \alpha_*\cdot \hat{\tilde{\lambda}}(\hat{\tilde{\tau}})\cdot\frac{\chi_{>\gamma^{-1}}(\hat{\tilde{\tau}})}{\hat{\tilde{\tau}}^2}
		\end{align*}
		Here, the constant $\alpha_*$ is given by the explicit product(recall \eqref{eq:cstardef})
		\begin{equation}\label{eq:alphastar}
			\alpha_* = -c_*\cdot \int_0^\infty \triangle(W^2)\cdot W^2\,R^3\,dR.
		\end{equation}
		We note that the function $\chi_{>\gamma^{-1}}(\hat{\tilde{\tau}})$ will be chosen complex valued in the 'transition region' where its real value changes between $0$ and $1$, and equals $1$ in the region $\big|\hat{\tilde{\tau}}\big|\gg \gamma^{-1}$. Also, it is assumed to satisfy the customary conjugation symmetry. Then recalling Lemma~\ref{lem:Fouriertransform2}, we observe that for $\big|\hat{\tilde{\tau}}\big|\gg 1$, we have 
		\begin{align*}
			&- c_2 \int_{0}^\infty \frac{1}{\hat{\tilde{\tau}}^2 - \xi^2}\cdot \mathcal{F}_{\R^4}\big(W^2\big)(\xi)\cdot \mathcal{F}_{\R^4}\big(\Lambda W\cdot W\big)(\xi)\rho_{\R^4}(\xi)\,d\xi\\
			& = -c_2\hat{\tilde{\tau}}^{-4}\cdot \int_0^\infty \triangle(W^2)\cdot \Lambda W\cdot W\cdot R^3\,dR + O(\hat{\tilde{\tau}}^{-6}). 
		\end{align*}
		We make the numerical non-degeneracy assumption {\bf{(C3)}} that 
		\begin{equation}\label{eq:alphastarstar}
			\alpha_{**}: = -\alpha_* + c_2\cdot \int_0^\infty \triangle(W^2)\cdot \Lambda W\cdot W\cdot  R^3\,dR\neq 0
		\end{equation}
		Recalling Lemma~\ref{lem:Fouriertransform2} and its proof, it follows that 
		\begin{align*}
			&\mathcal{F}_{\tilde{\tau}}\big(\Phi^{(\tilde{\lambda})}_{\text{model}}\big)(\hat{\tilde{\tau}})\\& = 
			-\hat{\tilde{\tau}}^2\big(\hat{\tilde{\lambda}}(\hat{\tilde{\tau}}) + c_3(\hat{\tilde{\tau}},\nu)\cdot\hat{\tilde{\tau}}^{-2}\mathcal{F}_{\tilde{\tau}}\big(\frac{\tilde{\lambda}_{\tilde{\tau}}}{\tilde{\tau}}\big)(\hat{\tilde{\tau}})\big)\cdot\beta_2(\hat{\tilde{\tau}})\\
			& -  \alpha_*\cdot \big(\hat{\tilde{\lambda}}(\hat{\tilde{\tau}}) + c_3\hat{\tilde{\tau}}^{-2}\mathcal{F}_{\tilde{\tau}}\big(\frac{\tilde{\lambda}_{\tilde{\tau}}}{\tilde{\tau}}\big)(\hat{\tilde{\tau}})\big)\cdot\frac{\chi_{>\gamma^{-1}}(\hat{\tilde{\tau}})}{\hat{\tilde{\tau}}^2} -\delta_1\zeta(\hat{\tilde{\tau}})\\
			&\delta_1\zeta(\hat{\tilde{\tau}})= \delta\zeta(\hat{\tilde{\tau}}) -  \alpha_*\cdot c_3\hat{\tilde{\tau}}^{-2}\mathcal{F}_{\tilde{\tau}}\big(\frac{\tilde{\lambda}_{\tilde{\tau}}}{\tilde{\tau}}\big)(\hat{\tilde{\tau}})\cdot\frac{\chi_{>\gamma^{-1}}(\hat{\tilde{\tau}})}{\hat{\tilde{\tau}}^2} 
		\end{align*}
		which can be equated with 
		\begin{align*}
			\tilde{\beta}(\hat{\tilde{\tau}})\cdot  \big(\hat{\tilde{\lambda}}(\hat{\tilde{\tau}}) + c_3(\hat{\tilde{\tau}},\nu)\hat{\tilde{\tau}}^{-2}\mathcal{F}_{\tilde{\tau}}\big(\frac{\tilde{\lambda}_{\tilde{\tau}}}{\tilde{\tau}}\big)(\hat{\tilde{\tau}})\big) - \delta_1\zeta(\hat{\tilde{\tau}}), 
		\end{align*}
		where 
		\begin{align*}
			\tilde{\beta}(\hat{\tilde{\tau}}) = \alpha_{**}\cdot \hat{\tilde{\tau}}^{-2} + O(\hat{\tilde{\tau}}^{-4}),\,\delta_1\zeta(\hat{\tilde{\tau}})
			&= \hat{\tilde{\tau}}^{-2}\cdot \kappa(\hat{\tilde{\tau}})\cdot \mathcal{F}_{\tilde{\tau}}\big(Q^{(\tilde{\tau})}_{<\tilde{\tau}^{\frac{10}{\nu}}}\big(\frac{\tilde{\lambda}_{\tilde{\tau}}}{\tilde{\tau}}\big)\big)(\hat{\tilde{\tau}})\\
			& -  \alpha_*\cdot c_3\hat{\tilde{\tau}}^{-2}\mathcal{F}_{\tilde{\tau}}\big(\frac{\tilde{\lambda}_{\tilde{\tau}}}{\tilde{\tau}}\big)(\hat{\tilde{\tau}})\cdot\frac{\chi_{>\gamma^{-1}}(\hat{\tilde{\tau}})}{\hat{\tilde{\tau}}^2} 
		\end{align*}
		for $\big| \hat{\tilde{\tau}}\big|\gg 1$. It is straightforward to verify that $\tilde{\beta}(\hat{\tilde{\tau}})$ is $C^\infty$ away from zero, and furthermore the imaginary part is non-zero away from the origin. Also, Lemma~\ref{lem:FourierNV1}, Lemma~\ref{lem:FourierNV2} and non-degeneracy assumptions {\bf{(C2)}}, {\bf{(C3)}} as well as suitable choice of the (complex valued) cutoff $\chi_{<\gamma^{-1}}$ give a positive lower bound for the absolute value of $\tilde{\beta}(\hat{\tilde{\tau}})$, independent of all parameter choices. Taking advantage of Lemma~\ref{lem:Fouriertransform2} and the Fourier localization of the first term constituting $\delta_1\zeta(\hat{\tilde{\tau}})$ to frequencies $<\tilde{\tau}^{\frac{10}{\nu}}$, we check that for $\nu$ sufficiently large, we have the bound 
		\begin{equation}\label{eq:deltazetabound}
			\Big\|\langle \hat{\tilde{\tau}}^2\rangle\cdot\delta\zeta(\hat{\tilde{\tau}})\Big\|_{W^{M,2}}\ll_{\tau_*}\big\|\langle\partial_{\tilde{\tau}}^2\rangle^{-1}\tilde{\lambda}_{\tilde{\tau}\tilde{\tau}}\big\|_{\tau^{-N}L^2_{d\tau}}. 
		\end{equation}
		Since 
		\[
		\Pi^{(\tilde{\tau})}f
		\]
		co-incides with $f$ on $[\tau_*, \infty)$, it suffices to solve 
		\begin{equation}\label{eq:solnPhitildelambdamodelonFourierside}
			\tilde{\beta}(\hat{\tilde{\tau}})\cdot  \big(\hat{\tilde{\lambda}}(\hat{\tilde{\tau}}) + c_3(\hat{\tilde{\tau}},\nu)\hat{\tilde{\tau}}^{-2}\mathcal{F}_{\tilde{\tau}}\big(\frac{\tilde{\lambda}_{\tilde{\tau}}}{\tilde{\tau}}\big)(\hat{\tilde{\tau}})\big) - \delta\zeta(\hat{\tilde{\tau}}) = \mathcal{F}_{\tilde{\tau}}\big(\Pi^{(\tilde{\tau})}f\big)(\hat{\tilde{\tau}}). 
		\end{equation}
		In order to solve the simpler model equation 
		\begin{align*}
			\tilde{\beta}(\hat{\tilde{\tau}})\cdot  \big(\hat{\tilde{\lambda}}(\hat{\tilde{\tau}}) + c_3(\hat{\tilde{\tau}},\nu)\hat{\tilde{\tau}}^{-2}\mathcal{F}_{\tilde{\tau}}\big(\frac{\tilde{\lambda}_{\tilde{\tau}}}{\tilde{\tau}}\big)(\hat{\tilde{\tau}})\big) = \mathcal{F}_{\tilde{\tau}}\big(\Pi^{(\tilde{\tau})}f\big)(\hat{\tilde{\tau}}),
		\end{align*}
		use that
		\[
		\hat{\tilde{\lambda}}(\hat{\tilde{\tau}}) = \hat{\tilde{\tau}}^{-1}\cdot \partial_{\hat{\tilde{\tau}}}\mathcal{F}_{\tilde{\tau}}\big(\frac{\tilde{\lambda}_{\tilde{\tau}}}{\tilde{\tau}}\big)(\hat{\tilde{\tau}}),
		\]
		and so we infer the ordinary differential equation 
		\begin{align*}
			\partial_{\hat{\tilde{\tau}}}\mathcal{F}_{\tilde{\tau}}\big(\frac{\tilde{\lambda}_{\tilde{\tau}}}{\tilde{\tau}}\big)(\hat{\tilde{\tau}})+ c_3(\hat{\tilde{\tau}},\nu)\hat{\tilde{\tau}}^{-1}\mathcal{F}_{\tilde{\tau}}\big(\frac{\tilde{\lambda}_{\tilde{\tau}}}{\tilde{\tau}}\big)(\hat{\tilde{\tau}}) = \tilde{\beta}^{-1}(\hat{\tilde{\tau}})\cdot \hat{\tilde{\tau}}\cdot \mathcal{F}_{\tilde{\tau}}\big(\Pi^{(\tilde{\tau})}f\big).
		\end{align*}
		We note that $\tilde{\beta}^{-1}(\hat{\tilde{\tau}})$ is bounded away from zero for $\hat{\tilde{\tau}}\rightarrow 0$ due to assumption {\bf{(C1)}} in subsection~\ref{subsec:numerics}. 
		The preceding equation is solved explicitly by means of 
		\begin{align*}
			&\mathcal{F}_{\tilde{\tau}}\big(\frac{\tilde{\lambda}_{\tilde{\tau}}}{\tilde{\tau}}\big)(\hat{\tilde{\tau}}) = H^{-1}(\hat{\tilde{\tau}},\nu)\cdot \int_0^{\hat{\tilde{\tau}}}H(s,\nu)\cdot \tilde{\beta}^{-1}(s)\cdot s\cdot \mathcal{F}_{\tilde{\tau}}\big(\Pi^{(\tilde{\tau})}f\big)(s)\,ds,\\
			&H(\hat{\tilde{\tau}},\nu) = e^{\int_{\hat{\tilde{\tau}}}^{\text{sgn}(\hat{\tilde{\tau}})} c_3(s,\nu)s^{-1}\,ds},
		\end{align*}
		and we observe that there is no problem with convergence of the integral near $\hat{\tilde{\tau}} = 0$ due to the high order of vanishing of $\mathcal{F}_{\tilde{\tau}}\big(\Pi^{(\tilde{\tau})}f\big)$ there. In turn we can write 
		\begin{align*}
			\hat{\tilde{\lambda}}(\hat{\tilde{\tau}}) = \hat{\tilde{\tau}}^{-1}\partial_{\hat{\tilde{\tau}}}\big(\mathcal{F}_{\tilde{\tau}}\big(\frac{\tilde{\lambda}_{\tilde{\tau}}}{\tilde{\tau}}\big)\big), 
		\end{align*}
		and the compactness of the support of $c_3(\hat{\tilde{\tau}},\nu)$ as well as the above large frequency asymptotics for $\tilde{\beta}(s)$ then easily imply the estimate 
		\begin{align*}
			\big\|\hat{\tilde{\lambda}}(\hat{\tilde{\tau}}) \big\|_{W_{\hat{\tilde{\tau}}}^{M,2}}\lesssim \big\|\langle\partial_{\tilde{\tau}}^2\rangle f\big\|_{\tau^{-N}L^2_{d\tau}}. 
		\end{align*}
		Using Plancherel's theorem this translates into the bound asserted in the Lemma, but for the solution of the model equation stated below \eqref{eq:solnPhitildelambdamodelonFourierside}. To solve the latter equation, it suffices to use the preceding argument, the bound \eqref{eq:deltazetabound} and a standard fixed point argument. 
	\end{proof}
	\begin{rem}\label{rem:lem:Phitildelambdamodeleqnsmallness}  The operator $\Pi^{(\tilde{\tau})}f$ has very large bounds dependent on $N$, and we shall have to apply the preceding construction to functions $f$ which only gain $N^{-1}$ at times. In order not to lose this smallness gain due to application of the projection operator, the following observation is important: from Lemma~\ref{lem:Fouriertransform2} we infer that for $|\hat{\tilde{\tau}}|\ll 1$, the factor $\beta(\hat{\tilde{\tau}}) = \tilde{\beta}^{-1}(\hat{\tilde{\tau}})$, with 
		\begin{align*}
			\tilde{\beta}(\hat{\tilde{\tau}}) &= [ic_1\cdot \mathcal{F}_{\R^4}\big(W^2\big)(|\hat{\tilde{\tau}}|)\cdot \mathcal{F}_{\R^4}\big(\Lambda W\cdot W\big)(|\hat{\tilde{\tau}}|)\cdot \frac{\rho_{\R^4}(|\hat{\tilde{\tau}}|)}{\hat{\tilde{\tau}}}\\&\hspace{2cm} - c_2 \int_{0}^\infty \frac{1}{\hat{\tilde{\tau}}^2 - \xi^2}\cdot \mathcal{F}_{\R^4}\big(W^2\big)(\xi)\cdot \mathcal{F}_{\R^4}\big(\Lambda W\cdot W\big)(\xi)\rho_{\R^4}(\xi)\,d\xi
		\end{align*}
		This can be written in the form $\tilde{\beta}(\hat{\tilde{\tau}}) = \alpha_* + O(|\log\hat{\tilde{\tau}}|^2\cdot |\hat{\tilde{\tau}}|^2)$, with the constant $\alpha_*$ as in the statement of Lemma~\ref{lem:lowtempfreqznresprin2}, and the error term has symbol behavior. Then decompose
		\begin{align*}
			\big(\beta(\hat{\tilde{\tau}}) - \alpha_*^{-1}\big)\cdot \mathcal{F}_{\tilde{\tau}}\big(\Pi^{(\tilde{\tau})}f\big) &= \chi_{\hat{\tilde{\tau}}<\log^{-1}\tau_*}\cdot \big(\beta(\hat{\tilde{\tau}}) - \alpha_*^{-1}\big)\cdot \mathcal{F}_{\tilde{\tau}}\big(\Pi^{(\tilde{\tau})}f\big)\\
			& + \chi_{\hat{\tilde{\tau}}\geq \log^{-1}\tau_*}\cdot \big(\beta(\hat{\tilde{\tau}}) - \alpha_*^{-1}\big)\cdot \mathcal{F}_{\tilde{\tau}}\big(\Pi^{(\tilde{\tau})}f\big)\\
		\end{align*}
		where the cutoffs are smoothly localizing to the indicated regions, and we then have the bounds 
		\begin{align*}
			&\Big\| \mathcal{F}_{\tilde{\tau}}^{-1}\Big(\chi_{\hat{\tilde{\tau}}\geq \log^{-1}\tau_*}\cdot \big(\beta(\hat{\tilde{\tau}}) - \alpha_*^{-1}\big)\cdot \mathcal{F}_{\tilde{\tau}}\big(\Pi^{(\tilde{\tau})}f\big)\Big)\Big\|_{\tau^{-N}L^2_{d\tau}}\lesssim \big\|f\big\|_{\tau^{-N}L^2_{d\tau}},\\
			&\Big\| \mathcal{F}_{\tilde{\tau}}^{-1}\Big(\chi_{\hat{\tilde{\tau}}<\log^{-1}\tau_*}\cdot \big(\beta(\hat{\tilde{\tau}}) - \alpha_*^{-1}\big)\cdot \mathcal{F}_{\tilde{\tau}}\big(\Pi^{(\tilde{\tau})}f\big)\Big)\Big\|_{\tau^{-N}L^2_{d\tau}}\ll_{\tau_*} \big\|\Pi^{(\tilde{\tau})}f\big\|_{\tau^{-N}L^2_{d\tau}},
		\end{align*}
		and for $\tau_*\gg N$ we can replace $\big\|\Pi^{(\tilde{\tau})}f\big\|_{\tau^{-N}L^2_{d\tau}}$ by $\big\|f\big\|_{\tau^{-N}L^2_{d\tau}}$ at the end. Finally choosing a cutoff $\tilde{\chi}(\tilde{\tau})$ which equals $1$ on $[\tau_*,\infty)$ and satisfies $\tilde{\chi}f = \tilde{\chi}\Pi^{(\tilde{\tau})}f$, we infer that 
		\begin{align*}
			\Big\|\tilde{\chi}\cdot\mathcal{F}_{\tilde{\tau}}^{-1}\big(\beta(\hat{\tilde{\tau}})\mathcal{F}_{\tilde{\tau}}\big(\Pi^{(\tilde{\tau})}f\big)\big)\Big\|_{\tau^{-N}L^2_{d\tau}}\lesssim \big\|f\big\|_{\tau^{-N}L^2_{d\tau}}
		\end{align*}
		where the implied constant is independent of $N$. 
		
	\end{rem}
	
	Finally we have all the tools to complete the proof of Proposition~\ref{prop:Phitildelambdasolution}:
	\begin{proof}(Prop.~\ref{prop:Phitildelambdasolution})  Keeping in mind \eqref{eq:Phitildelambdadef}, as well as \eqref{eq:tildeE2mod}, \eqref{eq:ytildemodmodified}, we shall first show that all the terms in  \eqref{eq:tildeE2mod} except the first one are perturbative:
		\begin{lem}\label{lem:E2modperturbativeterms} Let us denote the sum of all terms in \eqref{eq:tildeE2mod} with the exception of the first two ones as well as the term 
			\[
			-Q^{(\tilde{\tau})}_{<\tilde{\tau}^{\frac{10}{\nu}}}\Big(\tilde{\lambda}_t\cdot \partial_t\big(\partial_{\tilde{\lambda}}n_*^{(\tilde{\lambda})}-\lambda^2\Lambda W\cdot W\big)\Big)
			\]
			as $\tilde{E}_{2*}^{\text{mod}}$, and further set 
			\begin{align*}
				\tilde{y}_{\tilde{\lambda}*}^{\text{mod}}: = \Box^{-1}\tilde{E}_{2*}^{\text{mod}}. 
			\end{align*}
			Then we have the bound 
			\begin{align*}
				\Big\|\langle\partial_{\tilde{\tau}}^2\rangle\mathcal{F}\big(\lambda^{-2}\tilde{y}_{\tilde{\lambda}*}^{\text{mod}}\cdot W\big)(\tau, 0)\Big\|_{\tau^{-N}L^2_{d\tau}}\ll_{\tau_*}\big\|\langle\partial_{\tilde{\tau}}^2\rangle^{-1}\partial_{\tilde{\tau}}^2\tilde{\lambda}\big\|_{\tau^{-N}L^2_{d\tau}}
			\end{align*}
		\end{lem}
		\begin{proof} We deal with the contributions of the third and fourth term in \eqref{eq:tildeE2mod} as well as the term explicitly displayed in the lemma, the others following a similar pattern. 
			\\
			
			{\it{(1): contribution of the term $Q^{(\tilde{\tau})}_{<1}\Big[\tilde{\lambda}_{tt}\cdot \big(\partial_{\tilde{\lambda}}n_*^{(\tilde{\lambda})}-\lambda^2\Lambda W\cdot W\big)\Big]$.}} Explicitly we need to bound 
			\begin{align*}
				\mathcal{F}\Big(\lambda^{-2}\Box^{-1}Q^{(\tilde{\tau})}_{<1}\Big[\lambda^2\cdot \tilde{\lambda}_{\tilde{\tau}\tilde{\tau}}\cdot \big(\partial_{\tilde{\lambda}}n_*^{(\tilde{\lambda})}-\lambda^2\Lambda W\cdot W\big)\Big]\cdot W\Big)(\tau, 0). 
			\end{align*}
			Here the operator $Q^{(\tilde{\tau})}_{<1}$ 'neutralizes' the derivatives in $\tilde{\lambda}_{\tilde{\tau}\tilde{\tau}}$, and also ensures we can hit the expression with an operator $\partial_{\tilde{\tau}}^2$. Observe the bound (see Lemma~\ref{lem:approxsolasymptotics3})
			\begin{align*}
				\big\|\partial_{\tilde{\lambda}}n_*^{(\tilde{\lambda})}-\lambda^2\Lambda W\cdot W\big\|_{L^{1+}_{R^3\,dR}}\lesssim \tau^{0-}\cdot \lambda^2(\tau), 
			\end{align*}
			and from there
			\begin{align*}
				&\Big\|\mathcal{F}_{\R^4}Q^{(\tilde{\tau})}_{<1}\Big[\tilde{\lambda}_{\tilde{\tau}\tilde{\tau}}\cdot \big(\partial_{\tilde{\lambda}}n_*^{(\tilde{\lambda})}-\lambda^2\Lambda W\cdot W\big)\Big]\Big\|_{\lambda^2(\tau)\cdot\tau^{-N-}L^2_{d\tau}L^M_{\rho(\xi) d\xi}\cap L^2_{\rho(\xi)\,d\xi}}\\
				&\lesssim \big\|\langle\partial_{\tilde{\tau}}^2\rangle^{-2}\partial_{\tilde{\tau}}^2\tilde{\lambda}\big\|_{\tau^{-N}L^2_{d\tau}}. 
			\end{align*}
			In a similar vein, we also have the bound 
			\begin{align*}
				&\Big\|\langle\partial_{\xi}\rangle^{1+\delta_0}\mathcal{F}_{\R^4}Q^{(\tilde{\tau})}_{<1}\Big[\lambda^2\cdot \tilde{\lambda}_{\tilde{\tau}\tilde{\tau}}\cdot \big(\partial_{\tilde{\lambda}}n_*^{(\tilde{\lambda})}-\lambda^2\Lambda W\cdot W\big)\Big]\Big\|_{\lambda^2(\tau)\cdot\tau^{-N-}L^2_{d\tau} L^2_{\rho(\xi)\,d\xi}}\\
				&\lesssim \big\|\langle\partial_{\tilde{\tau}}^2\rangle^{-2}\partial_{\tilde{\tau}}^2\tilde{\lambda}\big\|_{\tau^{-N}L^2_{d\tau}}. 
			\end{align*}
			
			Taking advantage of Lemma~\ref{lem:wavebasicinhom} and the remarks above, we then infer 
			\begin{align*}
				&\Big\|\langle\partial_{\tilde{\tau}}^2\rangle \mathcal{F}\Big(\lambda^{-2}\Box^{-1}Q^{(\tilde{\tau})}_{<1}\Big[\lambda^2\cdot \tilde{\lambda}_{\tilde{\tau}\tilde{\tau}}\cdot \big(\partial_{\tilde{\lambda}}n_*^{(\tilde{\lambda})}-\lambda^2\Lambda W\cdot W\big)\Big]\cdot W\Big)(\tau, 0)\Big\|_{\tau^{-N-}L^2_{d\tau}}\\
				&\lesssim  \big\|\langle\partial_{\tilde{\tau}}^2\rangle^{-2}\partial_{\tilde{\tau}}^2\tilde{\lambda}\big\|_{\tau^{-N}L^2_{d\tau}}. 
			\end{align*}
			
			{\it{(2): contribution of the term $2\partial_t(\chi_3)\cdot  \tilde{\lambda}_t\cdot \partial_{\tilde{\lambda}}n_*^{(\tilde{\lambda})}$.}} Explicitly we need to bound 
			\begin{align*}
				\mathcal{F}\Big(\lambda^{-2}\Box^{-1}Q^{(\tilde{\tau})}_{<1}\Big[2\lambda^2\partial_{\tilde{\tau}}(\chi_3)\cdot  \tilde{\lambda}_{\tilde{\tau}}\cdot \partial_{\tilde{\lambda}}n_*^{(\tilde{\lambda})}\Big]\cdot W\Big)(\tau, 0).
			\end{align*}
			Exploiting the definition of $\chi_3$, we have 
			\begin{align*}
				&\Big\|\mathcal{F}_{\R^4}Q^{(\tilde{\tau})}_{<1}\Big[\partial_{\tilde{\tau}}(\chi_3)\cdot  \tilde{\lambda}_{\tilde{\tau}}\cdot \partial_{\tilde{\lambda}}n_*^{(\tilde{\lambda})}\Big]\Big\|_{\lambda^2(\tau)\cdot\tau^{-N-}L^2_{d\tau}L^M_{\rho(\xi) d\xi}\cap L^2_{\rho(\xi)\,d\xi}}\\
				&\lesssim  \big\|\langle\partial_{\tilde{\tau}}^2\rangle^{-2}\partial_{\tilde{\tau}}^2\tilde{\lambda}\big\|_{\tau^{-N}L^2_{d\tau}},
			\end{align*}
			as well as 
			\begin{align*}
				&\Big\|\langle\partial_{\xi}\rangle^{1+\delta_0}\mathcal{F}_{\R^4}Q^{(\tilde{\tau})}_{<1}\Big[\partial_{\tilde{\tau}}(\chi_3)\cdot  \tilde{\lambda}_{\tilde{\tau}}\cdot \partial_{\tilde{\lambda}}n_*^{(\tilde{\lambda})}\Big]\Big\|_{\lambda^2(\tau)\cdot\tau^{-N-}L^2_{d\tau} L^2_{\rho(\xi)\,d\xi}}\\
				&\lesssim \big\|\langle\partial_{\tilde{\tau}}^2\rangle^{-2}\partial_{\tilde{\tau}}^2\tilde{\lambda}\big\|_{\tau^{-N}L^2_{d\tau}}. 
			\end{align*}
			Taking advantage of Lemma~\ref{lem:wavebasicinhom} we obtain that 
			\begin{align*}
				&\Big\|\langle\partial_{\tilde{\tau}}^2\rangle \mathcal{F}\Big(\lambda^{-2}\Box^{-1}Q^{(\tilde{\tau})}_{<1}\Big[2\lambda^2\partial_{\tilde{\tau}}(\chi_3)\cdot  \tilde{\lambda}_{\tilde{\tau}}\cdot \partial_{\tilde{\lambda}}n_*^{(\tilde{\lambda})}\Big]\cdot W\Big)(\tau, 0)\Big\|_{\tau^{-N-}L^2_{d\tau}}\\
				&\lesssim  \big\|\langle\partial_{\tilde{\tau}}^2\rangle^{-2}\partial_{\tilde{\tau}}^2\tilde{\lambda}\big\|_{\tau^{-N}L^2_{d\tau}}. 
			\end{align*}
			
			{\it{(3): contribution of the term $-Q^{(\tilde{\tau})}_{<\tilde{\tau}^{\frac{10}{\nu}}}\Big(\tilde{\lambda}_t\cdot \partial_t\big(\partial_{\tilde{\lambda}}n_*^{(\tilde{\lambda})}-\lambda^2\Lambda W\cdot W\big)\Big)$.}} Explicitly we need to bound 
			\begin{align*}
				\mathcal{F}\Big(\lambda^{-2}\Box^{-1}Q^{(\tilde{\tau})}_{<\tilde{\tau}^{\frac{10}{\nu}}}\Big(\lambda^2\tilde{\lambda}_{\tilde{\tau}}\cdot \partial_{\tilde{\tau}}\big(\partial_{\tilde{\lambda}}n_*^{(\tilde{\lambda})}-\lambda^2\Lambda W\cdot W\big)\Big)\cdot W\Big)(\tau, 0). 
			\end{align*}
			Observe that the operator $\partial_{\tilde{\tau}}$ gains an extra factor $\tilde{\tau}^{-1}$, and we have 
			\begin{align*}
				\Big|\partial_{\tilde{\lambda}}n_*^{(\tilde{\lambda})}-\lambda^2\Lambda W\cdot W\Big|\lesssim \chi_3\cdot \frac{\lambda^2}{(\lambda t)^2}\cdot\frac{\log^2 R}{\langle R\rangle^2}.
			\end{align*}
			Furthermore, we have the inequality 
			\begin{align*}
				\big\|\frac{Q^{(\tilde{\tau})}_{<\tilde{\tau}^{\frac{10}{\nu}}}\tilde{\lambda}_{\tilde{\tau}}}{\tilde{\tau}}\big\|_{\tau^{-N}L^2_{d\tau}}\lesssim \big\|\langle\partial_{\tilde{\tau}}^2\rangle^{-1}\partial_{\tilde{\tau}}^2\tilde{\lambda}\big\|_{\tau^{-N}L^2_{d\tau}}. 
			\end{align*}
			We then infer in analogy to the cases {\it{(1), (2)}} the bound
			\begin{align*}
				&\Big\|\langle\partial_{\tilde{\tau}}^2\rangle \mathcal{F}\Big(\lambda^{-2}\Box^{-1}Q^{(\tilde{\tau})}_{<\tilde{\tau}^{\frac{10}{\nu}}}\Big(\lambda^2\tilde{\lambda}_{\tilde{\tau}}\cdot \partial_{\tilde{\tau}}\big(\partial_{\tilde{\lambda}}n_*^{(\tilde{\lambda})}-\lambda^2\Lambda W\cdot W\big)\Big)\cdot W\Big)(\tau, 0)\Big\|_{\tau^{-N-}L^2_{d\tau}}\\
				&\lesssim  \big\|\langle\partial_{\tilde{\tau}}^2\rangle^{-2}\partial_{\tilde{\tau}}^2\tilde{\lambda}\big\|_{\tau^{-N}L^2_{d\tau}}. 
			\end{align*}
		\end{proof}
		
		Recalling \eqref{eq:Phitildelambdadef} which in turn relies on \eqref{eq:ytildemodmodified}, \eqref{eq:tildeE2mod}, and taking advantage of the preceding lemma as well as Lemmas~\ref{lem:Phitildelambdahighfreqsimplfied}, ~\ref{lem:reductionsteps3}, ~\ref{lem:reductionsteps2}, we can write 
		\eqref{eq:Phitildelambdaf} in the modified form 
		\begin{equation}\label{eq:Phitildelambdafmodified}
			\Phi^{(\tilde{\lambda})}_{\text{model}}(\sigma) = f(\sigma) + \Phi^{(\tilde{\lambda})}_{\text{small}} + \partial_{\sigma}E, 
		\end{equation}
		where we have the bounds 
		\begin{align*}
			&\Big\|\langle\partial_{\tilde{\sigma}}^2\rangle\Phi^{(\tilde{\lambda})}_{\text{small}}\Big\|_{\sigma^{-N}L^2_{d\sigma}}\ll_{\tau_*}\big\|\langle\partial_{\tilde{\sigma}}^2\rangle^{-1}\partial_{\tilde{\sigma}}^2\tilde{\lambda}\big\|_{\sigma^{-N}L^2_{d\sigma}},\\
			&\Big\|E\Big\|_{\sigma^{-N-}L^2_{d\sigma}}\ll_{\tau_*}\big\|\langle\partial_{\tilde{\sigma}}^2\rangle^{-1}\partial_{\tilde{\sigma}}^2\tilde{\lambda}\big\|_{\sigma^{-N}L^2_{d\sigma}}.
		\end{align*}
		Neglecting the term $\partial_{\sigma}E$ as an error term, we observe that we can replace the remaining equation, which we only need to satisfy on $[\tau_*, \infty)$, by the following:
		\begin{align*}
			\Phi^{(\tilde{\lambda})}_{\text{model}}(\sigma) = f(\sigma) + \Pi^{(\tilde{\tau})}\big(\Phi^{(\tilde{\lambda})}_{\text{small}}\big). 
		\end{align*}
		But if we recall Lemma~\ref{lem:Pitildetaudef}, then the conclusion of Proposition~\ref{prop:Phitildelambdasolution} follows from Lemma~\ref{lem:Phitildelambdamodeleqn} and a simple fixed point argument. The last statement of Proposition~\ref{prop:Phitildelambdasolution} is a consequence of Lemma~\ref{lem:Fouriertransform2}.
	\end{proof}
	
	\subsection{Proof of Lemma~\ref{lem:hvanishing1}}
	
	The fact that $h\in L^2_{R^3\,dR}$ follows from Lemma~\ref{lem:goodinversewithK} together with the fact that $\mathcal{F}_*(\tilde{\Pi}g)(|\hat{\tau}|) = 0$, and higher regularity follows from standard elliptic theory.  Since $(\hat{\tau}^2 + \triangle +2W^2)h = K^*_{main}h$, using the variation of constants formula in terms of the fundamental system $\{ \phi_{\hat{\tau}}, \theta_{\hat{\tau}}\}$ introduced right after Lemma~\ref{lem:Fredholmproperty}, we find 
	\begin{equation}\label{eq:goodhrepres}
		h(R) = \phi_{\hat{\tau}}(R)\int_{R}^{\infty}\theta_{\hat{\tau}}(s)H(s)s^3\,ds - \theta_{\hat{\tau}}(R)\int_{R}^{\infty}\phi_{\hat{\tau}}(s)H(s)s^3\,ds,\,R\gg 1. 
	\end{equation}
	where we use the notation 
	\begin{align*}
		H(R): = K^*_{main}h. 
	\end{align*}
	The term on the right can be made explicit upon recalling \eqref{eq:Kdefinition}: thus we can set 
	\begin{equation}\label{eq:Kmainstructure}
		K_{main}^*h = \sum_{j=1}^3 K_j^*h,
	\end{equation}
	with the following terms on the right:
	\begin{align*}
		&\frac12 K_1^*h(R):\\& = W(R)\cdot \Big[\theta_0(R)\int_R^\infty 2\triangle W(s)h(s)\phi_0(s)s^3\,ds\\&\hspace{6cm} - \phi_0(R)\int_R^\infty 2\triangle W(s)h(s)\theta_0(s)s^3\,ds\Big]\\
		&\frac14 K_2^*h(R):\\& = W(R)\cdot \Big[\nabla\theta_0(R)\int_R^\infty 4\nabla W(s)h(s)\phi_0(s)s^3\,ds\\&\hspace{6cm} - \nabla\phi_0(R)\int_R^\infty 2\nabla W(s)h(s)\theta_0(s)s^3\,ds\Big]\\
		&\frac12K_3^*h(R):\\& = -W(R)\cdot \Big[\theta_0(R)\int_R^\infty W^3(s)h(s)\phi_0(s)s^3\,ds - \phi_0(R)\int_R^\infty W^3(s)h(s)\theta_0(s)s^3\,ds\Big]\\
	\end{align*}
	Here by abuse of notation of label by $\{\phi_0, \theta_0\}$ a fundamental system for $\mathcal{L}$ analogous to $\{\phi_{\hat{\tau}}, \theta_{\hat{\tau}}\}$. The rapid decay of $h$ then follows from a straightforward induction argument. 
	
	\subsection{Some spectral theory}\label{subsec:Lstarbasics}
	
	Here we discuss the basic spectral representation associated with the operators $\tilde{\mathcal{L}}$, $-\triangle - 2W^2$. To begin with, observe that we can conjugate the $4$-d radial Laplacian to the one-dimensional operator 
	\[
	\partial_{RR} - \frac34 R^{-2}. 
	\]
	on the half line $(0, \infty)$. 
	Furthermore, the operator $-\triangle - 2W^2$ can be conjugated into
	\[
	\mathcal{L}_*: = -\partial_{RR} + \frac34 R^{-2} - \frac{2}{(1+\frac{R^2}{8})^2}.
	\]
	The following proposition can be proven in close analogy to \cite{KST}, \cite{KST1}, \cite{KST2}. 
	\begin{prop}\label{prop:triangle+2W2} The operator $\mathcal{L}_*$, defined on $(0,\infty)$ with domain 
		\begin{align*}
			\text{Dom}(\mathcal{L}_*) = \{f\in L^2\big((0,\infty)\big)\,;\,f, f'\in \text{AC}_{loc}(0,\infty),\,\mathcal{L}_*f\in L^2(0,\infty)\},
		\end{align*}
		is self-adjoint. There exists a unique $\xi_d>0$ such that 
		\begin{align*}
			\text{spec}(\mathcal{L}_*) = \{-\xi_d^2\}\cup (0,\infty), 
		\end{align*}
		and $0$ is neither an eigenvalue nor a resonance. For a function $f\in L^2\big((0,\infty)\big)$ we have the spectral representation 
		\begin{align*}
			f(R) = \phi_{*d}(R)\cdot \hat{f}_*(\xi_d) + \int_0^\infty \phi_*(R;\xi)\cdot \hat{f}_*(\xi)\rho_*(\xi)\,d\xi. 
		\end{align*}
		The generalized Fourier basis $\phi_*(R;\xi)$ is uniformly bounded on $(0,\infty)\times (0, \infty)$ and 
		\begin{itemize}
			\item for $R\xi\lesssim 1$ admits a Taylor expansion 
			\[
			\phi_*(R;\xi) = \sum_{j\geq 0}(R\xi)^{2j}\phi_{*,j}(R^2)
			\]
			where the $\phi_{*,j}$ are smooth and bounded by $\big|\phi_{*,j}\big|\lesssim \frac{C_*^j}{j!}$. 
			\item In the region $R\xi\gtrsim 1$, we can write 
			\begin{align*}
				\phi_*(R;\xi) = a_*(\xi)\cdot \psi_{*,+}(R;\xi) + \overline{a_*(\xi)}\cdot \psi_{*,-}(R;\xi), 
			\end{align*}
			where $a_*(\xi)$ is bounded and analytic on $(0,\infty)$ with symbol type bounds, while $ \psi_{*,\pm}(R;\xi)$ admit the asymptotic representation 
			\[
			\psi_{*,\pm}(R;\xi) = \frac{e^{\pm iR\xi}}{\xi^{\frac32}}\cdot \sigma_{*,\pm}(R\xi, R)
			\]
			where the coefficient functions $\sigma_{*,\pm}$ are uniformly bounded and admit symbol type bounds with respect to either variable. 
		\end{itemize}
	\end{prop}
	
	We note that the absence of a root mode or resonance at $\xi = 0$ follows from the first of the numerical assumptions ${\bf{(S1)}}$ stated at the end of this paper. The presence of a negative eigenvalue is implied by considering $-\xi_d^2$ by considering 
	\begin{align*}
		Z_M: = \langle \mathcal{L}_*\big(\chi_M R^{\frac32}W\big),\,\chi_M R^{\frac32}W\rangle_{L^2_{dR}};
	\end{align*}
	here $M\ll 1$ and $\chi_M$ is a smooth cutoff which equals $1$ on $0\leq R\leq M$ and vanishes for $R\geq 2M$, with $|\nabla^k_R\chi_M|\lesssim M^{-k}$. One easily checks that 
	\begin{align*}
		Z_M <\langle \mathcal{L}\big(\chi_M R^{\frac32}W\big),\,\chi_M R^{\frac32}W\rangle_{L^2_{dR}} = o(\frac{1}{M}), 
	\end{align*}
	which implies that $Z_M$ is negative fo $M$ sufficiently large. \\
	On the other hand, letting $\tilde{\phi}_d(R)$ be an eigenfunction corresponding to the unique negative eigenvalue of $\tilde{\mathcal{L}} =  -\partial_{RR} + \frac34 R^{-2} - \frac{3}{(1+\frac{R^2}{8})^2}$, then for $\phi\in \text{Dom}(\mathcal{L}_*) $ with $\langle \phi, \tilde{\phi}_d\rangle_{L^2_{dR}} = 0$ we have 
	\begin{align*}
		\langle  \mathcal{L}_*\phi, \phi\rangle_{L^2_{dR}} \geq \langle  \tilde{\mathcal{L}}\phi, \phi\rangle_{L^2_{dR}}\geq 0 
	\end{align*}
	according to the characterization of the spectrum of $\tilde{\mathcal{L}}$ given below. This implies that the negative spectrum consists of a unique negative eigenvalue, as asserted. 
	\\
	
	The following proposition is analogous to the results in section 4 of \cite{KST}, and we also refer to \cite{DS}. 
	\begin{prop}\label{prop:triangle+3W2} The operator $\tilde{\mathcal{L}} = -\partial_{RR} + \frac34 R^{-2} - \frac{3}{(1+\frac{R^2}{8})^2}$ defined on $(0,\infty)$ with domain 
		\begin{align*}
			\text{Dom}(\tilde{\mathcal{L}}) = \{f\in L^2\big((0,\infty)\big)\,;\,f, f'\in \text{AC}_{loc}(0,\infty),\,\tilde{\mathcal{L}}f\in L^2(0,\infty)\},
		\end{align*}
		is self-adjoint. There exists a unique $\tilde{\xi}_d>0$ such that 
		\begin{align*}
			\text{spec}(\mathcal{L}_*) = \{-\tilde{\xi}_d^2\}\cup (0,\infty), 
		\end{align*}
		Passing to the conjugate operator $-\triangle  - 3W^2$ in the context of $\mathbb{R}^4$, we have the spectral representation 
		\begin{align*}
			f(R) =  \mathcal{F}(f)(-\tilde{\xi}_d^2)\cdot \tilde{\phi}_d(R) + \int_0^\infty \tilde{\phi}(R;\xi)\mathcal{F}(f)(\xi)\cdot\tilde{\rho}(\xi)\,d\xi
		\end{align*}
		where $\tilde{\phi}_d\in \mathcal{S}(\mathbb{R}^4)$ is the unique normalized eigenfunction associated to eigenvalue $-\tilde{\xi}_d^2$, while the Fourier basis $\tilde{\phi}(R;\xi)$ and spectral measure $\tilde{\rho}(\xi)$ admit the same asymptotics as the functions $\phi(R;\xi), \rho(\xi)$ in subsection~\ref{subsec:basicfourier}, except that $\tilde{\phi}(R; 0) = \Lambda W$. 
	\end{prop}
	
	\subsection{The operator $\big(\hat{\tau}^2 + \triangle +2W^2\big)_{good}^{-1}\circ K_{main}$.}\label{subsec:goodinverse} Using the spectral representation associated to $\mathcal{L}_*$ developed in the preceding subsection, we shall specify the construction of 
	\[
	\big(\hat{\tau}^2 + \triangle +2W^2\big)_{good}^{-1}\circ K_{main}.
	\]
	To begin with, we note that a simple choice of $\big(\hat{\tau}^2 + \triangle +2W^2\big)^{-1}f$ is given by the following formula:
	\begin{equation}\label{eq:simpleinverseoftriangle+hattausquare+2Wsquare}
		\big(\hat{\tau}^2 + \triangle +2W^2\big)^{-1}f(R) = \frac{\mathcal{F}_*(\xi_d)}{\hat{\tau}^2 +\xi_d^2}\cdot\phi_{*d}(R) + \int_0^\infty \phi_*(R;\xi)\cdot\frac{\mathcal{F}_*(f)(\xi)}{\hat{\tau}^2 - \xi^2}\rho_*(\xi)\,d\xi. 
	\end{equation}
	This definition, however, is not adequate since the right hand expression is not smooth with respect to $\hat{\tau}$, a necessary requirement for us if we intend to recover rapid decay with respect to time. To correct for this, we shall subtract a suitable linear combination of the Jost functions $\psi_{\pm}(R;\hat{\tau})$, which are of course in the kernel of $\hat{\tau}^2 + \triangle +2W^2$. Specifically, we recall from the preceding subsection that 
	\begin{equation}\label{eq:Jostrep}
		\phi_*(R;\xi) = a_*(\xi)\psi_{*,+}(R;\xi) + \overline{a_*(\xi)}\psi_{*,-}(R;\xi). 
	\end{equation}
	We shall first define a 'preliminary inverse' replacing \eqref{eq:simpleinverseoftriangle+hattausquare+2Wsquare} by an operator which has better properties for large $R$, but loses regularity at $R = 0$, as follows:
	\begin{equation}\label{eq:oneinverseoftriangle+hattausquare+2Wsquare}\begin{split}
			&\big(\hat{\tau}^2 + \triangle +2W^2\big)_{1}^{-1}f(R):\\&= \frac{\mathcal{F}_*(f)(\xi_d)}{\hat{\tau}^2 +\xi_d^2}\cdot\phi_{*d}(R) + \int_0^\infty \phi_*(R;\xi)\cdot\frac{\mathcal{F}_*(f)(\xi)}{\hat{\tau}^2 - \xi^2}\rho_*(\xi)\,d\xi\\
			& + \pi i\cdot \big[a_*(|\hat{\tau}|)\psi_{*,+}(R;|\hat{\tau}|) - \overline{a_*(|\hat{\tau})|}\psi_{*,-}(R;|\hat{\tau}|)\big]\cdot \frac{\mathcal{F}_*(f)(|\hat{\tau}|)}{2|\hat{\tau}|}\rho_*(|\hat{\tau}|). 
	\end{split}\end{equation}
	\begin{lem}\label{lem:goodinverse1} Assuming that $|\hat{\tau}|\in (\gamma_1,\gamma_2)$ for some $0<\gamma_1<\gamma_2$, we have the bound 
		\begin{align*}
			\Big\|\partial_{\hat{\tau}}^k\big(\hat{\tau}^2 + \triangle +2W^2\big)_{1}^{-1}\big(\chi_{R\lesssim M}f\big)(R)\Big\|_{\langle R\rangle^{1+\delta_0}R^{-\delta_0}L^2_{R^3\,dR}}\lesssim_{k,M,\gamma_{1,2}}\big\|f\big\|_{L^2_{R^3\,dR}},
		\end{align*}
		where $k\geq 0$ is arbitrary. We also have the bound 
		\begin{align*}
			\Big\|\partial_{\hat{\tau}}^k(R\partial_R)^l\big(\hat{\tau}^2 + \triangle +2W^2\big)_{1}^{-1}\big(\chi_{R\lesssim M}f\big)(R)\Big\|_{\langle R\rangle^{1+\delta_0}R^{-\delta_0}L^2_{R^3\,dR}}\lesssim_{k,l,M,\gamma_{1,2}}\big\|f\big\|_{L^2_{R^3\,dR}},
		\end{align*}
	\end{lem}
	\begin{proof} We may restrict to $\hat{\tau}>0$. We shall again use the representation of $\phi_*$ in terms of the Jost functions $\psi_{\pm}$, conjugated back to the radial $4d$ setting, as well as the asymptotic structure of the latter in the oscillatory regime $R\xi\gtrsim 1$, given by 
		\begin{align*}
			\psi_{*,\pm}(R;\xi) = \frac{e^{\pm iR\xi}}{R^{\frac32}\xi^{\frac32}}\cdot \sigma_{*,\pm}(R\xi, R).
		\end{align*}
		With this normalisation we have $a(\xi)\sim 1$ and the functions $\sigma_{\pm}(R\xi, R), a(\xi)$ have symbol type behavior with respect to their arguments $R,\xi$. We then conclude that
		\begin{align*}
			\phi_*(R;\xi) =& a_*(\xi)\cdot \sigma_+(R\xi, R) \frac{e^{iR\xi}}{R^{\frac32}\xi^{\frac32}} +  \overline{a_*(\xi)}\cdot \sigma_-(R\xi, R) \frac{e^{-iR\xi}}{R^{\frac32}\xi^{\frac32}}\\
			& =  a_*(\hat{\tau})\cdot \sigma_+(R\hat{\tau}, R) \frac{e^{iR\xi}}{R^{\frac32}\xi^{\frac32}} +  \overline{a_*(\hat{\tau})}\cdot \sigma_-(R\hat{\tau}, R) \frac{e^{-iR\xi}}{R^{\frac32}\xi^{\frac32}}\\
			& + (\xi - \hat{\tau})\cdot\big[\Gamma_+(\hat{\tau},\xi, R)\cdot \frac{e^{iR\xi}}{R^{\frac32}\xi^{\frac32}} +  \Gamma_-(\hat{\tau},\xi, R)\cdot \frac{e^{-iR\xi}}{R^{\frac32}\xi^{\frac32}}\big],
		\end{align*}
		where, provided $|\xi - \hat{\tau}|\leq \frac{\hat{\tau}}{10}$ and $\hat{\tau}\in [\gamma_1, \gamma_2]$, we have 
		\begin{align*}
			\Big|\partial_{\hat{\tau}}^k \partial_{\xi}^l \Gamma_{\pm}(\hat{\tau},\xi, R)\Big|\lesssim_{k,l,\gamma_{1,2}}1. 
		\end{align*}
		We also observe that (still for $\hat{\tau}\in[\gamma_{1},\gamma_2]$)
		\begin{align*}
			\frac{\mathcal{F}_*(\chi_{R\lesssim M}f)(\xi)}{\hat{\tau} + \xi}\rho_*(\xi) - \frac{\mathcal{F}_*(\chi_{R\lesssim M}f)(\hat{\tau})}{2\hat{\tau}}\rho_*(\hat{\tau}) = (\xi - \hat{\tau})\cdot \tilde{\Gamma}(\hat{\tau},\xi,M), 
		\end{align*}
		where we have analogous bounds for $ \tilde{\Gamma}$ as for $\Gamma_{\pm}$, with the implicit constant also depending on $M$. Then we observe that 
		\begin{align*}
			&\Big\|\partial_{\hat{\tau}}^k\int_0^\infty \chi_{R\xi\gtrsim 1}\big(a_*(\xi) \sigma_{*,+}(R\xi, R)\frac{e^{iR\xi}}{R^{\frac32}\xi^{\frac32}} + \overline{a_*(\xi)}\sigma_{*,-}(R\xi, R)\frac{e^{-iR\xi}}{R^{\frac32}\xi^{\frac32}}\big)\cdot \tilde{\Gamma}(\hat{\tau},\xi,M)\,d\xi\Big\|_{\langle R\rangle^{1+\delta_0}L^2_{R^3\,dR}}\\
			&\lesssim_{k,M,\gamma_{1,2}}\big\|f\big\|_{L^2_{R^3\,dR}},
		\end{align*}
		which allows us to replace $\frac{\mathcal{F}_*(\chi_{R\lesssim M}f)(\xi)}{\hat{\tau} + \xi}\rho_*(\xi)$ by $ \frac{\mathcal{F}_*(\chi_{R\lesssim M}f)(\hat{\tau})}{2\hat{\tau}}\rho_*(\hat{\tau}) $ in \eqref{eq:simpleinverseoftriangle+hattausquare+2Wsquare}, and similarly we can replace 
		$ \phi_*(R;\xi)$ by $ \frac{a_*(\hat{\tau})}{\hat{\tau}^{\frac32}}\cdot \sigma_{*,+}(R\hat{\tau}, R) \frac{e^{iR\xi}}{R^{\frac32}} +  \overline{\frac{a_*(\hat{\tau})}{\hat{\tau}^{\frac32}}}\cdot \sigma_{*,-}(R\hat{\tau}, R) \frac{e^{-iR\xi}}{R^{\frac32}}$, up to error terms which satisfy the bound of the lemma. We have then reduced the integral term in \eqref{eq:oneinverseoftriangle+hattausquare+2Wsquare} to the expression 
		\begin{align*}
			\frac{\mathcal{F}_*(\chi_{R\lesssim M}f)(\hat{\tau})}{2\hat{\tau}}\rho_*(\hat{\tau})\cdot&\Big[ \frac{a_*(\hat{\tau})}{\hat{\tau}^{\frac32}}\cdot \frac{\sigma_{*,+}(R\hat{\tau}, R)}{R^{\frac32}}\cdot \int_0^\infty \frac{e^{iR\xi}}{\hat{\tau} - \xi}\,d\xi\\&\hspace{2cm} +  \overline{\frac{a_*(\hat{\tau})}{\hat{\tau}^{\frac32}}}\cdot \frac{\sigma_{*,-}(R\hat{\tau}, R)}{R^{\frac32}}\cdot \int_0^\infty \frac{e^{-iR\xi}}{\hat{\tau} - \xi}\,d\xi\Big]
		\end{align*}
		To conclude we observe that, restricting to $R\geq 1$ as we may, we have  
		\begin{align*}
			\int_0^\infty \frac{e^{iR\xi}}{\hat{\tau} - \xi}\,d\xi = -e^{iR\hat{\tau}}\cdot \int_{-\hat{\tau}}^\infty \frac{e^{i Rx}}{x}\,dx =& -e^{iR\hat{\tau}}\cdot\lim_{L\rightarrow +\infty}\int_{-L}^L \frac{e^{i Rx}}{x}\,dx\\
			& + e^{iR\hat{\tau}}\cdot\lim_{L\rightarrow +\infty}\int_{-L}^{-\hat{\tau}} \frac{e^{i Rx}}{x}\,dx.\\
		\end{align*}
		Then the function 
		\begin{align*}
			H(R,\hat{\tau}): =  e^{iR\hat{\tau}}\cdot\lim_{L\rightarrow +\infty}\int_{-L}^{-\hat{\tau}} \frac{e^{i Rx}}{x}\,dx = \lim_{L\rightarrow +\infty}\int_{-L+\hat{\tau}}^{0} \frac{e^{i Ry}}{y-\hat{\tau}}\,dy
		\end{align*}
		satisfies symbol type bounds with respect to $\hat{\tau}$ uniformly in $R$, and we also have $\lim_{L\rightarrow +\infty}\int_{-L}^L \frac{e^{i Rx}}{x}\,dx = i\pi$, from which the first conclusion of the lemma easily follows. To obtain the second inequality, we note that in the preceding we have written 
		\[
		\big(\hat{\tau}^2 + \triangle +2W^2\big)_{good}^{-1}\big(\chi_{R\lesssim M}f\big)(R)
		\]
		as a sum of terms which either have symbol behavior with respect to $R$, or where we can invoke integration by parts with respect to $\xi$ to absorb additional factors of $R$. 
	\end{proof}
	
	\begin{rem}\label{rem:lem:goodinverse1} The following slight modification of the preceding lemma shall also be useful: replacing $f$ by $f_{\hat{\tau}\in \langle R\rangle^{-1-\delta_0}L^2_{R^3\,dR}}W^{L,2}_{\hat{\tau}}$, $L\gg 1$, with $\mathcal{F}_{\tilde{\tau}}^{-1}(f_{\hat{\tau}})$ supported on $[\tau_*,\infty)$, we have that for $|\hat{\tau}|\in [\gamma_1,\,\gamma_2],\,\gamma_{1,2}>0$ 
		\begin{align*}
			\Big\|\big(\hat{\tau}^2 + \triangle +2W^2\big)_{1}^{-1}\big(\chi_{R\lesssim M}f_{\hat{\tau}}\big)(R)\Big\|_{\langle R\rangle^{1+\delta_0}R^{-\delta_0}L^2_{R^3\,dR}W^{L,2}_{\hat{\tau}}}\lesssim_{\gamma_{1,2}}\big\|f_{\hat{\tau}}\big\|_{L^2_{\langle R\rangle^{-1-\delta_0}R^3\,dR}W^{L,2}_{\hat{\tau}}},
		\end{align*}
		provided $\tau_*\geq \tau_*(M,L)$. 
	\end{rem}
	
	We also include here the following useful 
	\begin{lem}\label{lem:goodinversewithK} If $\mathcal{F_*}\big(\chi_{R\lesssim M}f\big)(|\hat{\tau}|) = 0$, $\hat{\tau}\in \R\backslash\{0\}$, $f\in L^2_{R^3\,dR}$, then 
		\begin{align*}
			\frac{\mathcal{F}_*(\chi_{R\lesssim M}f)(\xi_d)}{\hat{\tau}^2 +\xi_d^2}\cdot\phi_{*d}(R) + \int_0^\infty \phi_*(R;\xi)\cdot\frac{\mathcal{F}_*(\chi_{R\lesssim M}f)(\xi)}{\hat{\tau}^2 - \xi^2}\rho_*(\xi)\,d\xi\in L^2_{R^3\,dR}. 
		\end{align*}
	\end{lem}
	In fact, it suffices to replace $\mathcal{F}_*(\chi_{R\lesssim M}f)(\xi)$ by $\mathcal{F}_*(\chi_{R\lesssim M}f)(\xi) - \mathcal{F}_*(\chi_{R\lesssim M}f)(|\hat{\tau}|)$ and to observe that $\frac{\mathcal{F}_*(\chi_{R\lesssim M}f)(\xi) -  \mathcal{F}_*(\chi_{R\lesssim M}f)(|\hat{\tau}|)}{\hat{\tau}^2 - \xi^2}\in L^2_{\rho_*\,d\xi}$ and use Plancherel's theorem for the distorted Fourier transform. 
	\\
	
	For the 'good inverse', we shall sacrifice decay in one of the temporal directions (namely $\tilde{\tau}\rightarrow -\infty$, which is irrelevant for us) for smoothness at the origin, as follows: with $\eta(\hat{\tau}): = -i\pi ,\,\hat{\tau}>0,\,\eta(\hat{\tau}): = \pi i,\,\hat{\tau}<0$, set
	\begin{equation}\label{eq:goodinverseoftriangle+hattausquare+2Wsquare}\begin{split}
			&\boxed{\big(\hat{\tau}^2 + \triangle +2W^2\big)_{good}^{-1}f(R):}\\&= \boxed{\frac{\mathcal{F}_*(f)(\xi_d)}{\hat{\tau}^2 +\xi_d^2}\cdot\phi_{*d}(R) + \int_0^\infty \phi_*(R;\xi)\cdot\frac{\mathcal{F}_*(f)(\xi)}{\hat{\tau}^2 - \xi^2}\rho_*(\xi)\,d\xi}\\
			&\hspace{0.6cm} \boxed{+\eta(\hat{\tau})\cdot \frac{\mathcal{F}_*(f)(|\hat{\tau}|)}{2|\hat{\tau}|}\rho_*(|\hat{\tau}|)\cdot \phi_*(R;\,|\hat{\tau}|)}, 
	\end{split}\end{equation}
	This expression now is regular at the origin $R = 0$ and in light of \eqref{eq:Jostrep} it differs from $\big(\hat{\tau}^2 + \triangle +2W^2\big)_{1}^{-1}f(R)$ by a multiple of $\psi_{*,+}(R;|\hat{\tau}|)$ in the region $\hat{\tau}>0$ and by a multiple of $\psi_{*,-}(R;|\hat{\tau}|)$ for $\hat{\tau}<0$.  
	Then we have 
	\begin{lem}\label{lem:keygoodinverse} Let $f\in \tau^{-N}L^2_{d\tau}\langle R\rangle^{-1-\delta_0}L^2_{R^3\,dR}$, supported on $[\tau_*,\infty)$, and let $\mathcal{F}_{\tilde{\tau}}$ the Fourier transform with respect to wave time. Then for $\gamma>0$ we have 
		\begin{align*}
			&\Big\|\mathcal{F}_{\tilde{\tau}}^{-1}\big[\big(\hat{\tau}^2 + \triangle +2W^2\big)_{good}^{-1}\big(\chi_{R\lesssim M}\mathcal{F}_{\tilde{\tau}}(Q^{(\tilde{\tau})}_{[\gamma, \gamma^{-1}]}f)\big)\big]|_{[\tau_*,\infty)}\Big\|_{\tau^{-N}L^2_{d\tau}\langle R\rangle R^{\delta_0}L^2_{R^3\,dR}}\\
			&\lesssim_{\gamma} \big\|f\big\|_{\tau^{-N}L^2_{d\tau}\langle R\rangle^{-1-\delta_0}L^2_{R^3\,dR}},
		\end{align*}
		provided $\tau_*\geq \tau_*(\gamma, M)$. The same bound obtains if we replace $f$ by $T_{\hat{\tau}}(f)$, where $\hat{\tau}\rightarrow T_{\hat{\tau}}$ is a smooth family of bounded operators from $\langle R\rangle^{-1-\delta_0}L^2_{R^3\,dR}$ to itself. 
	\end{lem}
	\begin{proof} In light of Lemma~\ref{lem:goodinverse1} and the standard Plancherel's theorem, as well as the fact that our choice of $\big(\hat{\tau}^2 + \triangle +2W^2\big)_{good}^{-1}$ differs from $\big(\hat{\tau}^2 + \triangle +2W^2\big)_{1}^{-1}$ by a multiple of $\frac{e^{iR\hat{\tau}}}{R^{\frac32}}\cdot \sigma_{*,\pm}(R\hat{\tau}, R)$, where the sign $\pm = \sign(\hat{\tau})$, it suffices to check that
		\begin{align*}
			&\Big\|\mathcal{F}_{\tilde{\tau}}^{-1}\big[\frac{e^{iR\hat{\tau}}}{R^{\frac32}}\sigma_{*,\pm}(R\hat{\tau}, R)\cdot\chi_{[\gamma, \gamma^{-1}]}(\hat{\tau})\cdot\zeta(\hat{\tau})\cdot \mathcal{F}_*(\chi_{R\lesssim M}\mathcal{F}_{\tilde{\tau}}(f))(|\hat{\tau}|)\Big\|_{\tau^{-N}L^2_{d\tau}\langle R\rangle R^{\delta_0}L^2_{R^3\,dR}}\\&\lesssim \big\|f\big\|_{\tau^{-N}L^2_{d\tau}\langle R\rangle^{-1-\delta_0}L^2_{R^3\,dR}}.
		\end{align*}
		Here the function $\zeta(\cdot)\in C^\infty(\R)$ and $\sigma_{*,\pm}(R\hat{\tau}, R)$ is as in the asymptotic expansion of the Jost solutions. Expanding out 
		\begin{align*}
			\mathcal{F}_*(\chi_{R\lesssim M}\mathcal{F}_{\tilde{\tau}}(f))(|\hat{\tau}|) = \langle \phi_*(R; |\hat{\tau}|),\,\chi_{R\lesssim M}\mathcal{F}_{\tilde{\tau}}(f)(R)\rangle_{L^2_{R^3\,dR}}, 
		\end{align*}
		we can develop the smooth function $\kappa_R(\hat{\tau}): = \tilde{\chi}_{[\gamma, \gamma^{-1}]}(\hat{\tau})\phi_*(R; |\hat{\tau}|)\cdot \chi_{R\lesssim M}$, where the smooth cutoff $\tilde{\chi}$ satisfies $\tilde{\chi}\chi = \chi$, into a discrete Fourier series $\sum_{n\in \Z} a_n(R)\cdot e^{in\gamma\hat{\tau}}$,  
		where we have the bounds $\big|a_n(R)\big|\lesssim_{M,\gamma,L}\langle n\rangle^{-L}$ for any $L\in \N$, uniformly in $R$. Using a similar discrete Fourier expansion for the function $\hat{\tau}\rightarrow \sigma_{*,\pm}(R\hat{\tau}, R)\cdot\chi_{[\gamma, \gamma^{-1}]}(\hat{\tau})\cdot\zeta(\hat{\tau})$, 
		we infer 
		\begin{equation}\label{eq:nowundowFouriertransform}\begin{split}
				&\frac{e^{iR\hat{\tau}}}{R^{\frac32}}\sigma_{*,\pm}(R\hat{\tau}, R)\cdot\chi_{[\gamma, \gamma^{-1}]}(\hat{\tau})\cdot\zeta(\hat{\tau})\cdot \mathcal{F}_*(\chi_{R\lesssim M}\mathcal{F}_{\tilde{\tau}}(f))(|\hat{\tau}|)\\
				& = \frac{e^{iR\hat{\tau}}}{R^{\frac32}}\cdot \sum_{n\in \Z}e^{in\gamma\hat{\tau}}\cdot \sum_{p+q = n}b_p(R)\cdot \langle c_q(R_1),\,\chi_{R_1\lesssim M}\mathcal{F}_{\tilde{\tau}}(f)(R_1)\rangle_{L^2_{R_1^3\,dR_1}},
		\end{split}\end{equation}
		where we have the estimates (uniformly in $R>0$)
		\begin{align*}
			&\big|b_p(R)\big|\lesssim_{\gamma} \langle p\rangle^{-1-},\,\frac{\big|c_q(R)\big|}{\langle R\rangle^{1+}}\lesssim_{\gamma}\langle q\rangle^{-1-}\cdot\langle R\rangle^{-\frac32}\cdot \langle R - \gamma |q|\rangle^{-1},\\
			&\big|c_q\big|\lesssim_{\gamma, M, L} \langle q\rangle^{-L}
		\end{align*}
		Applying the inverse temporal Fourier transform to \eqref{eq:nowundowFouriertransform}, we arrive at the double sum 
		\begin{align*}
			Y(\tau, R): = R^{-\frac32}\sum_{p, q\in \Z}b_p(R)\cdot \langle c_q(R_1),\,\chi_{R_1\lesssim M}f(\tilde{\tau} + R + \gamma(p+q),\,R_1)\rangle_{L^2_{R_1^3\,dR_1}},
		\end{align*}
		and we have the estimate 
		\begin{align*}
			&\big\|Y|_{[\tau_*,\infty)}\big\|_{\tau^{-N}L^2_{d\tau}\langle R\rangle R^{\delta_0}L^2_{R^3\,dR}}\\&\lesssim \sum_p\big\|\frac{b_p(R)}{R^{\frac32}}\big\|_{\langle R\rangle^{1+\delta_0}L^2_{R^3\,dR}}\cdot \Big\|\sum_q\langle c_q(R_1),\,\chi_{R_1\lesssim M}f(\tilde{\tau} + R + \gamma(p+q),\,R_1)\rangle_{L^2_{R_1^3\,dR_1}}\big|\Big\|_{L^\infty_{dR}\tau^{-N}L^2_{d\tau}}\\
			&\lesssim_{\gamma}\big\|f\big\|_{\tau^{-N}L^2_{d\tau}\langle R\rangle^{-1-\delta_0}L^2_{R^3\,dR}},
		\end{align*}
		provided $\tau_*\geq \tau_*(\gamma, M)$, as desired. 
	\end{proof}
	\begin{rem}\label{rem:lem:keygoodinverse} The preceding lemma is quite natural as one alternatively obtains \eqref{eq:goodinverseoftriangle+hattausquare+2Wsquare} as the temporal Fourier transform of the inhomogeneous Duhamel propagator vanishing at temporal $+\infty$ and associated to the wave operator $\tilde{\Box}: = -\partial_{\tilde{\tau}}^2 + \triangle + 2W^2$, up to an arbitrary multiple of the exponentially decaying mode $\phi_d(R)$. 
		
	\end{rem}
	
	In a similar vein, we also mention the following useful lemma, we omit the simple proof: 
	\begin{lem}\label{lem:goodinverse2} For $\hat{\tau}_1>0, \hat{\tau}_2>0$ we have the difference bound 
		\begin{align*}
			&\Big\|\big(\hat{\tau}_1^2 + \triangle +2W^2\big)_{good}^{-1}\big(\chi_{R\lesssim M}f\big)(R) - \big(\hat{\tau}_2^2 + \triangle +2W^2\big)_{good}^{-1}\big(\chi_{R\lesssim M}f\big)(R)\Big\|_{\langle R\rangle R^{\delta_0}L^2_{R^3\,dR}}\\
			&\ll_{|\hat{\tau}_1 - \hat{\tau}_2|}\big\|f\big\|_{\langle R\rangle^{-1-\delta_0}L^2_{R^3\,dR}},
		\end{align*}
	\end{lem}

	\subsection{Fine structure of some lower order terms}
	
	Considering the definition of $X^{(\tilde{\lambda})}$ in \eqref{eq:Xdef}, we encounter the term $Q^{(\tilde{\tau})}_{<\tau^{\frac12+}}\big((\lambda^{-2}n_*^{(\tilde{\lambda}, \tilde{\alpha})} - W^2)z\big)$. While we expect the factor $\lambda^{-2}n_*^{(\tilde{\lambda}, \tilde{\alpha})} - W^2$ to lead to a perturbative contribution, we have to deal with the poor spatial decay of this function, see Lemma~\ref{lem:approxsolasymptotics3}. This poses a difficulty for delicate estimates involving the Fourier transform of this term, such as the ones in Lemma~\ref{lem: tildelambdaeqnsourcebounds1},~\ref{lem:LsmallcalKtildelambda}. In fact, it appears that we cannot simply invoke Lemma~\ref{lem:wavebasicinhom}  to handle this term. Instead, we directly study the effect of multiplication by the problematic factor $\lambda^{-2}n_*^{(\tilde{\lambda}, \tilde{\alpha})} - W^2$ on the distorted Fourier side via the following 
	\begin{lem}\label{lem:pseudotransferenceoperator2} Setting 
		\begin{align*}
			\tilde{H}(\xi, \eta; \tau): =  \langle \phi(R;\xi),\,\chi_{R\gtrsim \tau^{\frac12-}}(\lambda^{-2}n_*^{(\tilde{\lambda}, \tilde{\alpha})} - W^2)\cdot \phi(R;\eta)\rangle_{L^2_{R^3\,dR}},
		\end{align*}
		the same kernel bounds as in Lemma~\ref{lem:pseudotransferenceoperator1} obtain, but with an extra decay factor $\lambda^{-2}\ll_{\tau_*}\tau^{-1}$.
	\end{lem}
	\begin{proof} This is a consequence of Lemma~\ref{lem:approxsolasymptotics3} in conjunction with Lemma~\ref{lem:pseudotransferenceoperator1}.
	\end{proof}
	
	The preceding lemma shall be used in the context of the following technical lemma comprising generalisations of both Lemma~\ref{lem:concatenation1} as well as Lemma~\ref{lem:SKdeltaPhi}: let $\mathcal{K}_*$ be the operator defined in analogy to \eqref{eq:Kformula} but with $F(\xi,\eta)$ replaced by
	$\tilde{H}(\xi, \eta; \tau)$. 
	\begin{lem}\label{lem:concatenation2} 
		Let $j\geq 1$ We have the bound 
		\begin{align*}
			&\Big\|\int_0^\infty \xi^2\cdot S\big([\prod_{l=1}^j\mathcal{K}_l\circ S](G)\big)\rho(\xi)\,d\xi\Big\|_{\tau^{-N+}L^2_{d\tau}}\\&\hspace{5cm}\lesssim (\sqrt{N})^{-j}\cdot \big\|\langle \xi\partial_{\xi}\rangle G\big\|_{\tau^{-N}L^2_{d\tau}L^2_{\rho(\xi)\,d\xi}}. 
		\end{align*}
		where for each $l$ we have either $\mathcal{K}_l = \frac{\lambda_{\tau}}{\lambda}\cdot\mathcal{K}$ or $\mathcal{K}_l =\mathcal{K}_*$. Furthermore we can improve the inequality by replacing $\tau^{-N+}L^2_{d\tau}$ by $\tau^{-N-}L^2_{d\tau}$ provided there is at least one operator $\mathcal{K}_*$ present. 
		\\
		There is an analogous variation on Lemma~\ref{lem:SKdeltaPhi}, as well as on Lemma~\ref{lem:derivativemovesthrough2}.
	\end{lem}
	The proofs of these assertions are identical to the ones of Lemma~\ref{lem:concatenation1} and Lemma~\ref{lem:SKdeltaPhi}. Note that the extra gain in the presence of one operator $\mathcal{K}_*$ comes from the fact that $\lambda^{-2}(\tau)\sim \tau^{-1-\frac{1}{2\nu}}\ll_{\tau_*}\tau^{-1}$. 
	We also make the following simple
	\begin{rem}\label{rem:concatenation2} One has a similar estimate as in the preceding lemma if one omits the factor $\xi^2$ and the operator $\langle\xi\partial_{\xi}\rangle$ and replaces $\tau^{-N+}$ by $\tau^{-N+1+}$. In fact, the estimate is then much more elementary. 
	\end{rem}
	
	For dealing with the high temporal frequency regime, we shall require the following 
	\begin{lem}\label{lem:transferenceKstarhightempfreq} Letting $0<\gamma = \gamma(\tau_*),\,\lim_{\tau_*\rightarrow\infty}\gamma(\tau_*) = 0$, we and letting $\mathcal{K}_*$ be defined as in the preceding lemma, we have the bounds (recall Proposition~\ref{prop:linpropagator}, we let $S$ denote the propagator in \eqref{eq:inhomprop})
		\begin{align*}
			\int_0^\infty \xi^2\cdot S\big(Q^{(\tilde{\tau})}_{>\gamma^{-1}}\mathcal{K}_*\circ S\big)G\cdot \rho(\xi)\,d\xi|_{[\tau_*,\infty)} = \partial_{\tau}G_1,
		\end{align*}
		where we have the estimate 
		\begin{align*}
			\big\|G_1\big\|_{\tau^{-N}L^2_{d\tau}}\ll_{\tau_*} \big\|\langle\xi\partial_{\xi}\rangle^{2+}G\big\|_{\tau^{-N}L^2_{d\tau}L^p_{\rho(\xi)\,d\xi}},\,2\leq p<\infty.
		\end{align*}
		We also have the estimate 
		\begin{align*}
			\big\|G_1\big\|_{\tau^{-N-}L^2_{d\tau}}\ll_{\tau_*} \big\|S G\big\|_{\tau^{-N}L^2_{d\tau}L^p_{\rho(\xi)\,d\xi}},\,2\leq p<\infty.
		\end{align*}
	\end{lem}
	\begin{proof} The second estimate is a straighforward consequence of the fact that $\mathcal{K}_*: \tau^{-N}L^2_{d\tau}L^p_{\rho(\xi)\,d\xi}\longrightarrow \tau^{-N-1-}L^2_{d\tau}L^p_{\rho(\xi)\,d\xi}$, see Lemma~\ref{lem:pseudotransferenceoperator1},  and that the extra $\xi^2$ allows us to write the expression as time derivative. We now consider the more delicate first inequality, which is a consequence of the relation $Q^{(\tilde{\tau})}_{>\gamma^{-1}} = \frac{\partial\tau}{\partial\tilde{\tau}}\cdot \partial_{\tau}\circ \partial_{\tilde{\tau}}^{-1}Q^{(\tilde{\tau})}_{>\gamma^{-1}}$, the mapping property
		\[
		\partial_{\tilde{\tau}}^{-1}Q^{(\tilde{\tau})}_{>\gamma^{-1}}: \tau^{-N}L^2_{d\tau}\rightarrow \gamma^2\cdot  \tau^{-N}L^2_{d\tau}, 
		\]
		and repeated integrations by parts. Specifically, write 
		\begin{equation}\label{eq:listofabstractterms1}\begin{split}
				\int_0^\infty \xi^2\cdot S\big(Q^{(\tilde{\tau})}_{>\gamma^{-1}}\mathcal{K}_*\circ S\big)G\cdot \rho(\xi)\,d\xi &= \partial_{\tau}\int_0^\infty \tilde{S}\big(Q^{(\tilde{\tau})}_{>\gamma^{-1}}\mathcal{K}_*\circ S\big)G\cdot \rho(\xi)\,d\xi\\
				& +  \int_0^\infty \tau^{-1}\tilde{\tilde{S}}\big(Q^{(\tilde{\tau})}_{>\gamma^{-1}}\mathcal{K}_*\circ S\big)G\cdot \rho(\xi)\,d\xi\\
				& + \int_0^\infty \xi^2\cdot S\frac{(\xi\partial_{\xi})}{\tau}\big(Q^{(\tilde{\tau})}_{>\gamma^{-1}}\mathcal{K}_*\circ S\big)G\cdot \rho(\xi)\,d\xi\\
				& - \int_0^\infty \big(Q^{(\tilde{\tau})}_{>\gamma^{-1}}\mathcal{K}_*\circ S\big)G\cdot \rho(\xi)\,d\xi,
		\end{split}\end{equation}
		where the propagators $\tilde{S}, \tilde{\tilde{S}}$ have a similar form and the same mapping properties as $S$. To handle the boundary term at the end, write
		\begin{equation}\label{eq:listofabstractterms2}\begin{split}
				\int_0^\infty \big(Q^{(\tilde{\tau})}_{>\gamma^{-1}}\mathcal{K}_*\circ S\big)G\cdot \rho(\xi)\,d\xi &= \partial_{\tau} \int_0^\infty \frac{\partial\tau}{\partial\tilde{\tau}}\cdot\big(\partial_{\tilde{\tau}}^{-1}Q^{(\tilde{\tau})}_{>\gamma^{-1}}\mathcal{K}_*\circ S\big)G\cdot \rho(\xi)\,d\xi\\
				& -   \int_0^\infty  \partial_{\tau}(\frac{\partial\tau}{\partial\tilde{\tau}})\cdot\big(\partial_{\tilde{\tau}}^{-1}Q^{(\tilde{\tau})}_{>\gamma^{-1}}\mathcal{K}_*\circ S\big)G\cdot \rho(\xi)\,d\xi.\\
		\end{split}\end{equation}
		Then we continue for the first term on the right by writing 
		\begin{align*}
			&\int_0^\infty \frac{\partial\tau}{\partial\tilde{\tau}}\cdot\big(\partial_{\tilde{\tau}}^{-1}Q^{(\tilde{\tau})}_{>\gamma^{-1}}\mathcal{K}_*\circ S\big)G\cdot \rho(\xi)\,d\xi\\
			& = \int_0^\infty \frac{\partial\tau}{\partial\tilde{\tau}}\cdot\big(\partial_{\tilde{\tau}}^{-2}Q^{(\tilde{\tau})}_{>\gamma^{-1}} \frac{\partial\tau}{\partial\tilde{\tau}}\mathcal{K}_*\circ \partial_{\tau}S\big)G\cdot \rho(\xi)\,d\xi,\\
		\end{align*}
		and using a simple variant of Lemma~\ref{lem:concatenation2} and the fact that $(\frac{\partial\tau}{\partial\tilde{\tau}})^2\sim \lambda^2$ to conclude that 
		\begin{align*}
			\Big\| \int_0^\infty \frac{\partial\tau}{\partial\tilde{\tau}}\cdot\big(\partial_{\tilde{\tau}}^{-2}Q^{(\tilde{\tau})}_{>\gamma^{-1}} \frac{\partial\tau}{\partial\tilde{\tau}}\mathcal{K}_*\circ \partial_{\tau}S\big)G\cdot \rho(\xi)\,d\xi\Big\|_{\tau^{-N}L^2_{d\tau}}\ll_{\tau_*}\big\|\langle\xi\partial_{\xi}\rangle^{1+}G\big\|_{\tau^{-N}L^2_{d\tau}L^p_{\rho(\xi)\,d\xi}}. 
		\end{align*}
		The second term in \eqref{eq:listofabstractterms2} is handled similarly. The procedure for the remaining terms in \eqref{eq:listofabstractterms1} is similar. One uses twofold integration by parts with respect to the time variables in the propagators $\tilde{S}, S$, which due to the operator $Q^{(\tilde{\tau})}_{>\gamma^{-1}} $ 'only costs' $\lambda^2$, and thereby either replaces either propagator by $\partial_{\tau}S, \partial_{\tau}\tilde{S}$. The resulting expression can then be handled in analogy to the proof of Lemma~\ref{lem:concatenation1}. Observe that we have to integrate by parts with respect to the spatial frequency twice, and this is responsible for the presence of the operator $\langle \xi\partial_{\xi}\rangle^{2+}$ on the right hand side.  
	\end{proof}
	The following lemma is obtained by similar arguments: 
	\begin{lem}\label{lem:transferenceKstarhightempfreqgeneral} Let  for each $l$ either $\mathcal{K}_l = \frac{\lambda_{\tau}}{\lambda}\cdot\mathcal{K}$ or $\mathcal{K}_l =\mathcal{K}_*$. Then 
		\begin{align*}
			&\int_0^\infty \xi^2\cdot S\big([\prod_{l=1}^j\mathcal{K}_l\circ S\circ Q^{(\tilde{\tau})}_{>\gamma^{-1}}\big(\mathcal{K}_*\circ S\circ[\prod_{l=j+1}^{k}\mathcal{K}_l\circ S\big)(G)\big)\rho(\xi)\,d\xi = \partial_{\tau}G_1,\\&
			\big\|G_1\big\|_{\tau^{-N}L^2_{d\tau}}\lesssim (\sqrt{N})^{-j-k}\cdot \big\|\langle \xi\partial_{\xi}\rangle^{2+} G\big\|_{\tau^{-N}L^2_{d\tau}L^2_{\rho(\xi)\,d\xi}}. 
		\end{align*}
		We can also improve the norm on the left to $\|\cdot\|_{\tau^{-N-}L^2_{d\tau}}$ provided we use $\big\|SG\big\|_{\tau^{-N}L^2_{d\tau}L^2_{\rho(\xi)\,d\xi}}$ on the right instead. 
	\end{lem}

	\subsection{Bounds of lower order error terms} Here we collect bounds on error terms which enjoy additional temporal decay, in particular those arising from the perturbative corrections used to build the approximate solution which serves as our starting point. To begin with, we state a lemma dealing with the perturbative terms in $X^{(\tilde{\lambda})}(\tau, \xi)$, see \eqref{eq:Xdef}:
	\begin{lem}\label{lem:Xtildelambdaperturbterms} Assuming the representation \eqref{eq:zdecompbasic} and further \eqref{eq:kapparefined}, we have the bounds
		\begin{align*}
			&\Big\|\langle\xi\partial_{\xi}\rangle\mathcal{F}\Big(Q^{(\tilde{\tau})}_{<\gamma^{-1}}\big(\lambda^{-2}y_z\cdot (\tilde{u}_*^{(\tilde{\lambda}, \tilde{\alpha})}-W)\big)\Big)\Big\|_{\tau^{-N-}L^2_{d\tau}L^2_{\rho(\xi)\,d\xi}}\\&
			\lesssim\big\|z_{nres}\big\|_{S} + \big\|\langle\partial_{\tilde{\tau}}^2\rangle^{-2}\partial_{\tilde{\tau}}^2\tilde{\lambda}\big\|_{\tau^{-N}L^2_{d\tau}} + \big\|(\tilde{\kappa}_1,\kappa_2)\big\|_{\tau^{-N}L^2_{d\tau}},\\
			&\Big\|\langle\xi\partial_{\xi}\rangle\mathcal{F}\big(\lambda^{-2}(y - y_z)\cdot \tilde{u}_*^{(\tilde{\lambda}, \tilde{\alpha})} -  y^{\text{mod}}_{\tilde{\lambda}}\cdot W\big)\Big\|_{\tau^{-N-}L^2_{d\tau}L^2_{\rho(\xi)\,d\xi}}\\&
			\lesssim\big\|z_{nres}\big\|_{S} + \big\|\langle\partial_{\tilde{\tau}}^2\rangle^{-2}\partial_{\tilde{\tau}}^2\tilde{\lambda}\big\|_{\tau^{-N}L^2_{d\tau}} + \big\|(\tilde{\kappa}_1,\kappa_2)\big\|_{\tau^{-N}L^2_{d\tau}}.
		\end{align*}
		The same bounds obtain without the multiplier $Q^{(\tilde{\tau})}_{<\gamma^{-1}}$. We can also replace $\langle\xi\partial_{\xi}\rangle$ by $\langle\xi\partial_{\xi}\rangle^{1+}$ and $L^2_{\rho(\xi)\,d\xi}$ by $L^{2+}_{\rho(\xi)\,d\xi}$, and apply $\partial_{\tau}$ in front. Finally, the same bounds hold with an extra factor $\xi^4$ included on the left. 
	\end{lem}
	\begin{proof} {\it{First inequality}}. Using the identity \eqref{eq: transference1}, as well as the $L^2_{\rho(\xi)\,d\xi}$-boundedness of the transference operator and Plancherel's theorem for the distorted Fourier transform,  we see that it suffices to estimate the norms
		\begin{align*}
			\Big\|\langle R\partial_R\rangle Q^{(\tilde{\tau})}_{<\gamma^{-1}}\big(\lambda^{-2}y_z\cdot (\tilde{u}_*^{(\tilde{\lambda}, \tilde{\alpha})}-W)\big)\Big\|_{\tau^{-N-}L^2_{d\tau}L^2_{R^3\,dR}}. 
		\end{align*}
		From the asymptotic expansion of the difference $\tilde{u}_*^{(\tilde{\lambda}, \tilde{\alpha})}-W$ we infer the bound 
		\begin{align*}
			\Big|\langle R\partial_R\rangle\big(\tilde{u}_*^{(\tilde{\lambda}, \tilde{\alpha})}-W\big)\Big|\lesssim \tau^{-\frac12-\frac{1}{4\nu}}, 
		\end{align*}
		see Lemma~\ref{lem:approxsolasymptotics1}. Next we split 
		\begin{align*}
			\lambda^{-2}y_z\cdot (\tilde{u}_*^{(\tilde{\lambda}, \tilde{\alpha})}-W)\big) &= \chi_{R\lesssim \tau^{\frac12-}}\cdot \lambda^{-2}y_z\cdot (\tilde{u}_*^{(\tilde{\lambda}, \tilde{\alpha})}-W)\big)\\
			& +  \chi_{R\gtrsim \tau^{\frac12-}}\cdot \lambda^{-2}y_z\cdot (\tilde{u}_*^{(\tilde{\lambda}, \tilde{\alpha})}-W)\big).\\
		\end{align*}
		For the latter term we use the preceding bound on $\tilde{u}_*^{(\tilde{\lambda}, \tilde{\alpha})}-W$ in conjunction with Lemma~\ref{lem:refinedwavepropagatorwithphysicallocalization} to infer an even better bound with $\tau^{-N-}$ the replaced by $\tau^{-N-\frac12-\frac{1}{4\nu}+}$. 
		As for the first term in the region $R\lesssim \tau^{\frac12-}$, by construction we have $\big|\tilde{u}_*^{(\tilde{\lambda}, \tilde{\alpha})}-W\big|\lesssim \log R\cdot \tau^{-1}$, which gains a factor $\log R\cdot\frac{R^2}{\tau}\lesssim \tau^{0-}$ over $W(R)$. The desired estimate then follows by repeating the proof of Lemma~\ref{lem:yzWbound1}.
		\\
		{\it{Second inequality}}. To begin with, we write
		\begin{equation}\label{eq:Xtildelambdaperturbtermseqn1}\begin{split}
				\lambda^{-2}(y - y_z)\cdot \tilde{u}_*^{(\tilde{\lambda}, \tilde{\alpha})} -  y^{\text{mod}}_{\tilde{\lambda}}\cdot W &=  \lambda^{-2}(y - y_z)\cdot \big(\tilde{u}_*^{(\tilde{\lambda}, \tilde{\alpha})} - W\big)\\
				& +  \big(\lambda^{-2}(y - y_z) -  \lambda^{-2}y^{\text{mod}}_{\tilde{\lambda}}\big)\cdot W
		\end{split}\end{equation}
		Keeping in mind \eqref{eq:y2def}, \eqref{eq:zeqn2} as well as \eqref{eq:ylamndatildemod}, we can estimate the contribution of the second term on the right as follows:
		\begin{align*}
			&\Big\|\langle \xi\partial_{\xi}\rangle \mathcal{F}\big(\big(\lambda^{-2}(y - y_z) -  \lambda^{-2}y^{\text{mod}}_{\tilde{\lambda}}\big)\cdot W\big)\Big\|_{\tau^{-N-}L^2_{d\tau}L^2_{\rho(\xi)\,d\xi}}\\&
			\lesssim \Big\|\langle \xi\partial_{\xi}\rangle \mathcal{F}\big(\lambda^{-2}\Box^{-1}\big(\lambda^2\triangle\Re\big[(\tilde{u}_*^{(\tilde{\lambda}, \tilde{\alpha})}-W)\bar{z}\big]\big)\cdot W\big)\Big\|_{\tau^{-N-}L^2_{d\tau}L^2_{\rho(\xi)\,d\xi}}\\&
			+  \Big\|\langle \xi\partial_{\xi}\rangle \mathcal{F}\big(\lambda^{-2}\Box^{-1}\triangle\big(\lambda^2|z|^2\big)\cdot W\big)\Big\|_{\tau^{-N-}L^2_{d\tau}L^2_{\rho(\xi)\,d\xi}}.
		\end{align*}
		For the first term on the right, we split 
		\begin{align*}
			(\tilde{u}_*^{(\tilde{\lambda}, \tilde{\alpha})}-W) = \chi_{R\gtrsim \tau^{\frac12-}}\cdot (\tilde{u}_*^{(\tilde{\lambda}, \tilde{\alpha})}-W) +  \chi_{R\lesssim \tau^{\frac12-}}\cdot (\tilde{u}_*^{(\tilde{\lambda}, \tilde{\alpha})}-W),
		\end{align*}
		and take advantage of Lemma~\ref{lem:refinedwavepropagatorwithphysicallocalization2}, resulting in the bound 
		\begin{align*}
			\Big\|\langle \xi\partial_{\xi}\rangle \mathcal{F}\big(\lambda^{-2}\Box^{-1}\big(\lambda^2\triangle\Re\big[ \chi_{R\gtrsim \tau^{\frac12-}}\cdot (\tilde{u}_*^{(\tilde{\lambda}, \tilde{\alpha})}-W)\bar{z}\big]\big)\cdot W\big)\Big\|_{\tau^{-N-}L^2_{d\tau}L^2_{\rho(\xi)\,d\xi}}\lesssim \big\|z\big\|_{S}. 
		\end{align*}
		Indeed, it suffices to check that 
		\begin{align*}
			\big\|\langle R\rangle^{-1}\triangle\Big(\chi_{R\gtrsim \tau^{\frac12-}}\cdot (\tilde{u}_*^{(\tilde{\lambda}, \tilde{\alpha})}-W)\bar{z}\Big)\big\|_{L^2_{R^3\,dR}}\lesssim \tilde{\tau}^{-1-}\cdot \big\|z\big\|_{S}, 
		\end{align*}
		which results from the definition \eqref{eq:Snormdefi} together with Lemma~\ref{lem:approxsolasymptotics1} . 
		On the other hand, for the contribution of $\chi_{R\lesssim \tau^{\frac12-}}\cdot (\tilde{u}_*^{(\tilde{\lambda}, \tilde{\alpha})}-W)\lesssim \chi_{R\lesssim \tau^{\frac12-}}\cdot\log R\cdot\frac{R^2}{\tau}\cdot W(R)$, one replicates the proof of  Lemma~\ref{lem:yzWbound1} but gains 
		\[
		\chi_{R\lesssim \tau^{\frac12-}}\cdot\log R\cdot\frac{R^2}{\tau}\lesssim \tau^{0-}
		\]
		in the process. 
		\\
		It remains to bound the term which is nonlinear in $z$, for which we can use the crude estimate 
		\begin{align*}
			\big\|\lambda^{-2}\Box^{-1}\triangle\big(\lambda^2|z|^2\big)\big\|_{\tau^{-2N+2}H^1_{d\tau}L^2_{R^3\,dR}}\lesssim \big\|z\big\|_{S}^2,
		\end{align*}
		where we 'spend' almost one power of $\tau$ to control the inhomogeneous wave propagator. This easily implies the bound 
		\begin{equation}\label{eq:easyquadraticwavepropagatorbound101}
			\Big\|\langle \xi\partial_{\xi}\rangle \mathcal{F}\big(\lambda^{-2}\Box^{-1}\triangle\big(\lambda^2|z|^2\big)\cdot W\big)\Big\|_{\tau^{-2N+2}L^2_{d\tau}L^2_{\rho(\xi)\,d\xi}}\lesssim  \big\|z\big\|_{S}^2
		\end{equation}
		by invoking Plancherel's theorem for the distorted Fourier transform, the $L^2_{\rho(\xi)\,d\xi}$-boundedness of the transference operator \eqref{eq:zFourier2} and the presence of the additional factor $W$ to absorb an extra weight $R$. The estimate in the lemma follows by invoking \eqref{eq:zedcomprefined}.
		\\
		We next estimate the first term in \eqref{eq:Xtildelambdaperturbtermseqn1}, for which we split 
		\begin{equation}\label{eq:technicalsplit101}\begin{split}
				\lambda^{-2}(y - y_z)\cdot \big(\tilde{u}_*^{(\tilde{\lambda}, \tilde{\alpha})} - W\big) &=  \lambda^{-2}(y - y_z)\cdot \chi_{R\gtrsim\tau^{\frac12-}}\big(\tilde{u}_*^{(\tilde{\lambda}, \tilde{\alpha})} - W\big)\\
				& +  \lambda^{-2}(y - y_z)\cdot \chi_{R\lesssim\tau^{\frac12-}}\big(\tilde{u}_*^{(\tilde{\lambda}, \tilde{\alpha})} - W\big)\\
		\end{split}\end{equation}
		For the first term on the right, we further decompose as before $\lambda^{-2}(y - y_z) = \lambda^{-2}(y - y_z - y_{\tilde{\lambda}}^{\text{mod}}) + \lambda^{-2}y_{\tilde{\lambda}}^{\text{mod}}$. Again recalling \eqref{eq:y2def}, \eqref{eq:zeqn2}, we now take advantage of several technical bounds, starting with
		\begin{equation}\label{eq:tecnhicalequation101}\begin{split}
				&\Big\|\langle R\partial_R\rangle\Big(\chi_{R\gtrsim\tau^{\frac12-}}\big(\tilde{u}_*^{(\tilde{\lambda}, \tilde{\alpha})} - W\big)\cdot\lambda^{-2}\Box^{-1}\big(\lambda^2\triangle\Re\big[ \chi_{R\gtrsim \tau^{\frac12-}}\cdot (\tilde{u}_*^{(\tilde{\lambda}, \tilde{\alpha})}-W)\bar{z}\big]\big)\Big)\Big\|_{\tau^{-N-}L^2_{d\tau}L^2_{R^3\,dR}}\\
				&\lesssim \big\|z\big\|_{S}.
		\end{split}\end{equation}
		To see this, we take advantage of Lemma~\ref{lem:approxsolasymptotics2}. To begin with, the latter implies the estimate
		\begin{equation}\label{eq:tildeustart-Wsimplebound}
			\big|\langle R\partial_R\rangle\big(\chi_{R\gtrsim\tau^{\frac12-}}(\tilde{u}_*^{(\tilde{\lambda}, \tilde{\alpha})} - W)\big)\big|\lesssim \lambda^{-1}. 
		\end{equation}
		Next, we decompose the factor $(\tilde{u}_*^{(\tilde{\lambda}, \tilde{\alpha})}-W)$ inside the wave propagator into $g = g_3 + g_4$, as in Lemma~\ref{lem:approxsolasymptotics2}. First consider the contribution of $g_4$. Calling $R_{1,2}$ the $R$ variable of the whole expression, respectively the $R$ variable of the function inside the wave propagator, we distinguish between (i) $\big|R_1- R_2\big|\gtrsim R_1$, as well as (ii) $\big|R_1-R_2\big|\ll R_1$ whence $R_1\sim R_2$. Note that if the operator $R\partial_R$ falls on $\Box^{-1}\big(\ldots\big)$, this 'costs' a factor $R_1\xi$ in the explicit wave propagator \eqref{eq:nflatfourierrepresent}, \eqref{eq:wavepropagator}. To compensate for it, we perform integration by parts with respect to $\xi$ in situation (i) by combining all three oscillatory phases, which compensates for this extra factor. In case (ii) it suffices to perform integration by parts with respect to $R_2$ inside the $R_2$-integral inside the wave propagator, which gains $(R_2\xi)^{-1}\sim (R_1\xi)^{-1}$ due to the symbol behavior of $g_4$. The writing $\triangle = \nabla\cdot\nabla$ one compensates for the wave propagator by means of one operator $\nabla$ and used the bound 
		\begin{align*}
			\big\|\lambda^{-2}\nabla\Box^{-1}\nabla \lambda^2\big(\chi_{R\gtrsim\tau^{\frac12-}}g_4\cdot z\big)\big\|_{\tau^{-N+\frac12-\frac{1}{4\nu}+}L^2_{d\tau}L^2_{R^3\,dR}}\lesssim \big\|z\big\|_{S}. 
		\end{align*}
		Then using \eqref{eq:tildeustart-Wsimplebound} for the first factor $\chi_{R\gtrsim\tau^{\frac12-}}(\tilde{u}_*^{(\tilde{\lambda}, \tilde{\alpha})} - W)$ easily compensates for the temporal decay loss in the previous estimate since $\lambda^{-1}\ll_{\tau_*}\tau^{-\frac12+\frac{1}{4\nu}-1}$. 
		Next assume we substitute $g_3$ for the  factor $(\tilde{u}_*^{(\tilde{\lambda}, \tilde{\alpha})}-W)$ inside the wave propagator. If the operator $R\partial_R$ falls on $\Box^{-1}\big(\ldots\big)$, which results in a loss $R_1\xi$, we absorb the extra $R_1$ by using the bound 
		\begin{align*}
			\big|R\cdot \chi_{R\gtrsim\tau^{\frac12-}}(\tilde{u}_*^{(\tilde{\lambda}, \tilde{\alpha})} - W)\big|\lesssim 1,
		\end{align*}
		and we use the extra factor $\xi$ to compensate for the $\xi^{-1}$ in the wave propagator. Then we use the bound 
		\begin{align*}
			\big\|\triangle\big(\chi_{R\gtrsim \tau^{\frac12-}}\cdot g_3z\big)\big\|_{\tau^{-N-\frac12+\frac{1}{4\nu}-}L^2_{d\tau}L^2_{R^3\,dR}}\lesssim \big\|z\big\|_{S}, 
		\end{align*}
		which is a consequence of Lemma~\ref{lem:approxsolasymptotics2}. The additional temporal decay in this bound more than compensates for the time integration in the wave propagator, giving the desired bound for this contribution. The case when $R\partial_R$ falls on the first factor $\chi_{R\gtrsim\tau^{\frac12-}}(\tilde{u}_*^{(\tilde{\lambda}, \tilde{\alpha})} - W)$ is handled in analogy to the corresponding contribution of $g_4$, and we are done proving \eqref{eq:tecnhicalequation101}.
		\\
		Next we have the estimate 
		\begin{equation}\label{eq:tecnhicalequation102}\begin{split}
				&\Big\|\langle R\partial_R\rangle\Big(\chi_{R\gtrsim\tau^{\frac12-}}\big(\tilde{u}_*^{(\tilde{\lambda}, \tilde{\alpha})} - W\big)\cdot\lambda^{-2}\Box^{-1}\big(\lambda^2\triangle\Re\big[ \chi_{R\lesssim\tau^{\frac12-}}\cdot (\tilde{u}_*^{(\tilde{\lambda}, \tilde{\alpha})}-W)\bar{z}\big]\big)\Big)\Big\|_{\tau^{-N-}L^2_{d\tau}L^2_{R^3\,dR}}\\
				&\lesssim \big\|z\big\|_{S}.
		\end{split}\end{equation}
		Here it suffices to replicate the argument for a similar term occuring in the estimate for the second term in \eqref{eq:Xtildelambdaperturbtermseqn1}, we omit thevery similar  details. 
		Again recalling \eqref{eq:zeqn2}, in order to complete the bound for the first term on the right in \eqref{eq:technicalsplit101} with $y-y_z$ replaced by $y - y_z - y_{\tilde{\lambda}}^{\text{mod}}$, we make use of the estimate 
		\begin{align*}
			\Big\|\langle R\partial_R\rangle\Big(\lambda^{-2}\Box^{-1}\triangle\big(\lambda^2|z|^2\big)\cdot \chi_{R\gtrsim\tau^{\frac12-}}(\tilde{u}_*^{(\tilde{\lambda}, \tilde{\alpha})} - W)\Big)\Big\|_{\tau^{-N-}L^2_{d\tau}L^2_{R^3\,dR}}\lesssim \big\|z\big\|_{S}^2,
		\end{align*}
		which is a straightforward consequence of \eqref{eq:easyquadraticwavepropagatorbound101} in conjunction with Lemma~\ref{lem:approxsolasymptotics1}.
		In order to finish the bound for the first term on the right in \eqref{eq:technicalsplit101}, we now need to control the norm 
		\begin{align*}
			\Big\|\langle R\partial_R\rangle\Big(\lambda^{-2}y_{\tilde{\lambda}}^{\text{mod}}\cdot \chi_{R\gtrsim\tau^{\frac12-}}\big(\tilde{u}_*^{(\tilde{\lambda}, \tilde{\alpha})} - W\big)\Big)\Big\|_{\tau^{-N-}L^2_{d\tau}L^2_{R^3\,dR}}
		\end{align*}
		This is accomplished by means of Lemma~\ref{lem:E2modbound1} as well as \eqref{eq:tildeustart-Wsimplebound}, bounding the preceding norm by $\lesssim \big\|\langle\partial_{\tilde{\tau}}^2\rangle^{-2}\partial_{\tilde{\tau}}^2\tilde{\lambda}\big\|_{\tau^{-N}L^2_{d\tau}}$. 
		We are left with bounding the second term in \eqref{eq:technicalsplit101}, which is accomplished by using Lemma~\ref{lem:approxsolasymptotics1} in conjunction with Lemma~\ref{lem:refinedwavepropagatorwithphysicallocalization2} as well as the proof of Lemma~\ref{lem:yzWbound1}, the latter being useful for the situation where the term to which the wave propagator gets applied is restricted to the region $R\lesssim \tau^{\frac12-}$.  
		\\
		The last part of the lemma is obtained by interpolation with cruder estimates obtained upon applying $\langle\xi\partial_{\xi}\rangle^2$. 
		
	\end{proof}
	
	In a similar vein, we also record the following lemma, which is proved analogously:
	\begin{lem}\label{lem:Xtildelambdaperturbterms2} Assuming the representation \eqref{eq:zdecompbasic} and further \eqref{eq:kapparefined}, we have the bounds
		\begin{align*}
			&\Big\|\mathcal{F}\Big(Q^{(\tilde{\tau})}_{<\gamma^{-1}}\big(\lambda^{-2}y_z\cdot (\tilde{u}_*^{(\tilde{\lambda}, \tilde{\alpha})}-W)\big)\Big)(\tau,0)\Big\|_{\tau^{-N-}L^2_{d\tau}}\\&
			\lesssim\big\|z_{nres}\big\|_{S} + \big\|\langle\partial_{\tilde{\tau}}^2\rangle^{-2}\partial_{\tilde{\tau}}^2\tilde{\lambda}\big\|_{\tau^{-N}L^2_{d\tau}} + \big\|(\tilde{\kappa}_1,\kappa_2)\big\|_{\tau^{-N}L^2_{d\tau}},\\
			&\Big\|\mathcal{F}\big(\lambda^{-2}(y - y_z)\cdot \tilde{u}_*^{(\tilde{\lambda}, \tilde{\alpha})} -  y^{\text{mod}}_{\tilde{\lambda}}\cdot W\big)(\tau,0)\Big\|_{\tau^{-N-}L^2_{d\tau}}\\&
			\lesssim\big\|z_{nres}\big\|_{S} + \big\|\langle\partial_{\tilde{\tau}}^2\rangle^{-2}\partial_{\tilde{\tau}}^2\tilde{\lambda}\big\|_{\tau^{-N}L^2_{d\tau}} + \big\|(\tilde{\kappa}_1,\kappa_2)\big\|_{\tau^{-N}L^2_{d\tau}}.
		\end{align*}
		The same estimate bound obtains if one replaces $Q^{(\tilde{\tau})}_{<\gamma^{-1}}$ by $Q^{(\tilde{\tau})}_{<\tau^{\frac12+}}$ or if one suppresses this operator. 
	\end{lem}
	\begin{proof} One writes 
		\begin{align*}
			\mathcal{F}\Big(Q^{(\tilde{\tau})}_{<\gamma^{-1}}\big(\lambda^{-2}y_z\cdot (\tilde{u}_*^{(\tilde{\lambda}, \tilde{\alpha})}-W)\big)\Big)(\tau,0)
			=\langle Q^{(\tilde{\tau})}_{<\gamma^{-1}}\big(\lambda^{-2}y_z\cdot (\tilde{u}_*^{(\tilde{\lambda}, \tilde{\alpha})}-W)\big), \phi(R;0)\rangle_{L^2_{R^3\,dR}}
		\end{align*}
		and follows essentially the estimates of the preceding proof. 
		
	\end{proof}
	
	For the same terms occuring in the preceding two lemmas, we shall also require a high temporal frequency analogue of Lemma~\ref{lem:yzWbound3}: 
	\begin{lem}\label{lem:Xtildelambdaperturbtermshighmod} Letting $f$ denote either one of 
		\[
		Q^{(\tilde{\tau})}_{>\gamma^{-1}}\big(\lambda^{-2}y_z\cdot (\tilde{u}_*^{(\tilde{\lambda}, \tilde{\alpha})}-W)\big),\,Q^{(\tilde{\tau})}_{>\gamma^{-1}}\big(\lambda^{-2}(y - y_z)\cdot \tilde{u}_*^{(\tilde{\lambda}, \tilde{\alpha})} -  y^{\text{mod}}_{\tilde{\lambda}}\cdot W\big),
		\]
		where $0<\gamma = \gamma(\tau_*)$ with $\lim_{\tau_*\rightarrow+\infty}\gamma(\tau_*) = 0$, then the following bound obtains:
		\begin{align*}
			\Big\|\langle \xi\partial_{\xi}\rangle^{1+}\langle f, \frac{\phi(R;\xi) - \phi(R;0)}{\xi^2}\rangle_{L^2_{R^3\,dR}}\Big\|_{\tau^{-N}L^2_{d\tau}L^p_{\rho(\xi)\,d\xi}}\ll_{\tau_*}\big\|z\big\|_{S},\,2\leq p\leq\infty. 
		\end{align*}
		Further, writing
		\begin{align}
			&\int_0^\infty\langle f,\,\phi(R;\xi) - \phi(R;0)\rangle_{L^2_{R^3\,dR}}\rho_1(\xi)\,d\xi = \partial_{\tau}\kappa,\label{eq::Xtildelambdaperturbtermshighmod1}\\&\int_0^\infty\langle f,\,\phi(R;\xi)\rangle_{L^2_{R^3\,dR}}[\rho(\xi) - \rho_1(\xi)]\,d\xi = \partial_{\tau}\tilde{\kappa}, \label{eq::Xtildelambdaperturbtermshighmod2}
		\end{align}
		we have the bounds 
		\[
		\big\|\kappa\big\|_{\tau^{-N}L^2_{d\tau}} + \big\|\tilde{\kappa}\big\|_{\tau^{-N}L^2_{d\tau}} \ll_{\gamma}\big\|z\big\|_{S}.
		\]
		Finally, for the distorted Fourier coefficients of $f$, we can write in the low frequency regime
		\begin{align*}
			\chi_{\xi\lesssim 1}\langle f, \phi(R;\xi)\rangle_{L^2_{R^3\,dR}} = f_1 + f_2, 
		\end{align*}
		where we have $\big\|\langle \xi\partial_{\xi}\rangle^{1+}(\xi^{-2}f_1)\big\|_{\tau^{-N+}L^2_{d\tau}L^p_{\rho(\xi)\,d\xi}}\ll_{\tau_*}\big\|z\big\|_{S}$, $\big\|\langle \xi\partial_{\xi}\rangle^{2+}f_2\big\|_{\tau^{-N}L^2_{d\tau}L^p_{\rho(\xi)\,d\xi}}\ll_{\tau_*}\big\|z\big\|_{S}$, $2\leq p\leq \infty$. In the high frequency regime $\xi\gtrsim 1$, writing $g: = \chi_{\xi\gtrsim 1}\langle f, \phi(R;\xi)\rangle_{L^2_{R^3\,dR}}$, the same bound as at the end of Lemma~\ref{lem:yzWbound3} applies. 
	\end{lem}
	\begin{proof} {\it{First part of lemma}}. We outline the proof for the first expression, the argument for the second expression following similar reasoning. We first reduce the factor $ (\tilde{u}_*^{(\tilde{\lambda}, \tilde{\alpha})}-W)$ to temporal frequency $\ll \gamma^{-1}$ since else writing
		\[
		Q^{(\tilde{\tau})}_{\gtrsim \gamma^{-1}}(\tilde{u}_*^{(\tilde{\lambda}, \tilde{\alpha})}-W) = (Q^{(\tilde{\tau})}_{\gtrsim\gamma^{-1}}\partial_{\tilde{\tau}}^{-l})\circ \partial_{\tilde{\tau}}^l(\tilde{u}_*^{(\tilde{\lambda}, \tilde{\alpha})}-W) 
		\]
		for $l>10$, say, and using Lemma~\ref{lem:approxsolasymptotics1}, the desired bound easily follows. Assuming this temporal frequency reduction implicitly in the following, we may let the operator $Q^{(\tilde{\tau})}_{>\gamma^{-1}}$ act directly on $y_z$. Next we split 
		\begin{equation}\label{eq:yzhighmodsplit}
			\lambda^{-2}Q^{(\tilde{\tau})}_{>\gamma^{-1}}y_z = \lambda^{-2}Q^{(\tilde{\tau})}_{>\gamma^{-1}}P_{<\gamma^{-\frac12}}y_z + \lambda^{-2}Q^{(\tilde{\tau})}_{>\gamma^{-1}}P_{\geq \gamma^{-\frac12}}y_z
		\end{equation}
		Arguing as in the proof of Lemma~\ref{lem:tildeylambdaWhighmodulationreduction}, the operator $Q^{(\tilde{\tau})}_{>\gamma^{-1}}P_{<\gamma^{-\frac12}}\Box^{-1}$ satisfies the same bounds as the operator $Q^{(\tilde{\tau})}_{>\gamma^{-1}}P_{<\gamma^{-\frac12}}\partial_{\tilde{\tau}}^{-2}$, and so we can replace the first expression on the right by $Q^{(\tilde{\tau})}_{>\gamma^{-1}}P_{<\gamma^{-\frac12}}\partial_{\tilde{\tau}}^{-2}\triangle\Re(Wz)$, recalling \eqref{eq:yzdfn}. Then perform integration by parts 
		\begin{align*}
			&\langle Q^{(\tilde{\tau})}_{>\gamma^{-1}}P_{<\gamma^{-\frac12}}\partial_{\tilde{\tau}}^{-2}\triangle\Re(Wz)\cdot  (\tilde{u}_*^{(\tilde{\lambda}, \tilde{\alpha})}-W), \frac{\phi(R;\xi) - \phi(R;0)}{\xi^2}\rangle_{L^2_{R^3\,dR}}\\
			& = -\langle Q^{(\tilde{\tau})}_{>\gamma^{-1}}P_{<\gamma^{-\frac12}}\partial_{\tilde{\tau}}^{-2}\nabla\Re(Wz), \nabla \big((\tilde{u}_*^{(\tilde{\lambda},\tilde{\alpha})}-W)\cdot \frac{\phi(R;\xi) - \phi(R;0)}{\xi^2}\big)\rangle_{L^2_{R^3\,dR}}\\
		\end{align*}
		Taking advantage of Lemma~\ref{lem:approxsolasymptotics1} as well as subsection~\ref{subsec:basicfourier}, we infer the bound 
		\begin{align*}
			\big\|R^{\frac32-}\cdot\langle \xi\partial_{\xi}\rangle^{1+}\nabla \big((\tilde{u}_*^{(\tilde{\lambda},\tilde{\alpha})}-W)\cdot \frac{\phi(R;\xi) - \phi(R;0)}{\xi^2}\big)\big\|_{L^p_{\rho(\xi)\,d\xi}L^\infty_{dR}}\lesssim 1, \,2\leq p\leq \infty
		\end{align*}
		while from \eqref{eq:Snormdefi} we have $\big\|Q^{(\tilde{\tau})}_{>\gamma^{-1}}P_{<\gamma^{-\frac12}}\partial_{\tilde{\tau}}^{-2}\nabla\Re(Wz)\big\|_{\tau^{-N}L^2_{d\tau}L^{\frac{4}{3}+}_{R^3\,dR}}\ll_{\gamma} \big\|z\big\|_{S}$. These observations imply that 
		\begin{align*}
			&\Big\|\langle \xi\partial_{\xi}\rangle^{1+}\langle \lambda^{-2}Q^{(\tilde{\tau})}_{>\gamma^{-1}}P_{<\gamma^{-\frac12}}y_z\cdot (\tilde{u}_*^{(\tilde{\lambda}, \tilde{\alpha})}-W),  \frac{\phi(R;\xi) - \phi(R;0)}{\xi^2}\rangle_{L^2_{R^3\,dR}}\Big\|_{\tau^{-N}L^2_{d\tau}L^p_{\rho(\xi)\,d\xi}}\\&\ll_{\gamma} \big\|z\big\|_{S},\,2\leq p\leq\infty,
		\end{align*}
		yielding the desired conclusion for the contribution of the first term on the right in \eqref{eq:yzhighmodsplit}. For the contribution of the second term, we take advantage of Lemma~\ref{lem:approxsolasymptotics1}: to begin with, if $\xi\lesssim 1$, then repeated integration by parts with respect to $R$ and repeating the proof of Corollary~\ref{cor:yzW} easily implies the result, due to non-resonance of the $R$ oscillatory phases. For $\xi\gg 1$, we distinguish between the cases (i) $R\lesssim \tau^{\frac12-\frac{1}{4\nu}+}$ and (ii) $R\gtrsim\tau^{\frac12-\frac{1}{4\nu}+}$. In case (i) the term $(\tilde{u}_*^{(\tilde{\lambda}, \tilde{\alpha})}-W)$ is given to leading order by $\frac{\log R}{\tau}\lesssim R^{-2}\cdot\tau^{-\frac{1}{2\nu}+}$, whence a simple modification of Corollary~\ref{cor:yzW} in fact results in a small power gain in $\tau^{-1}$. In case (ii), using the representation \eqref{eq:wavepropagator}, \eqref{eq:nflatfourierrepresent}  for the Duhamel propagator $y_z$ (recall \eqref{eq:yzdfn}), the oscillatory factors $\phi_{\R^4}(R;\xi)$, $\sin\big[\lambda(\tilde{\tau})\xi\int_{\tilde{\tau}}^{\tilde{\sigma}}\lambda^{-1}(s)\,ds\big]$ are out of phase (with respect to their $\xi$-dependence), and the required convergence of the $R$-integral follows by integrating by parts with respect to the frequency in the Fourier representation of $y_z$ and otherwise repeating the proof of Lemma~\ref{lem:yzWbound1} . 
		\\
		{\it{Second part of lemma}}. Again treating the first of the expressions, one uses the splitting \eqref{eq:yzhighmodsplit}. For the contribution of the first term on the right there, we use (for the second estimate see Lemma~\ref{lem:approxsolasymptotics1})
		\begin{align*}
			&\big\|\lambda^{-2}Q^{(\tilde{\tau})}_{>\gamma^{-1}}P_{<\gamma^{-\frac12}}y_z\big\|_{\tau^{-N}L^2_{d\tau}L^2_{R^3\,dR}}\ll_{\gamma}\big\|z\big\|_{S},\,\big|\tilde{u}_*^{(\tilde{\lambda}, \tilde{\alpha})}-W\big|\lesssim \tau^{-(\frac12+\frac{1}{4\nu})}\cdot \langle \frac{R}{\lambda}\rangle^{-2}\\
			&\big\|\lambda^{-2}Q^{(\tilde{\tau})}_{>\gamma^{-1}}P_{<\gamma^{-\frac12}}y_z\cdot (\tilde{u}_*^{(\tilde{\lambda}, \tilde{\alpha})}-W)\big\|_{L^{2-}_{R^3\,dR}}\lesssim \tau^{-(\frac12+\frac{1}{4\nu})+}\cdot\big\|\lambda^{-2}Q^{(\tilde{\tau})}_{>\gamma^{-1}}P_{<\gamma^{-\frac12}}y_z\big\|_{L^2_{R^3\,dR}}. 
		\end{align*}
		
		If we combine this with the identity\footnote{In this identity we define $\partial_{\tilde{\tau}}^{-1} = \partial_{\tilde{\tau}}^{-1} Q^{(\tilde{\tau})}_{>\frac{\gamma^{-1} }{2}} + \partial_{\tilde{\tau}}^{-1} Q^{(\tilde{\tau})}_{<\frac{\gamma^{-1} }{2}}$ where in the first term on the right $\partial_{\tilde{\tau}}^{-1} $ acts via division by the symbol on the Fourier side while for the second term on the right we define this via integration $\int_{\tilde{\tau}}^\infty\,d\tilde{\tau}$.}
		\begin{equation}\label{eq:partialtaupartialtildetau}
			Q^{(\tilde{\tau})}_{>\gamma^{-1}}g = \partial_{\tau}\big(\partial_{\tilde{\tau}}^{-1}(\frac{\partial\tau}{\partial\tilde{\tau}}\cdot Q^{(\tilde{\tau})}_{>\gamma^{-1}} g)\big)
		\end{equation}
		and the bound $\big\|\text{LHS}\eqref{eq::Xtildelambdaperturbtermshighmod1}\big\|_{\tau^{-N}L^2_{d\tau}} + \big\|\text{LHS}\eqref{eq::Xtildelambdaperturbtermshighmod2}\big\|_{\tau^{-N}L^2_{d\tau}}\lesssim \big\|f\big\|_{\tau^{-N}L^2_{d\tau}L^{2-}_{R^3\,dR}}$, the second conclusion of the lemma follows for the term $Q^{(\tilde{\tau})}_{>\gamma^{-1}}\big(\lambda^{-2}P_{<\gamma^{-\frac12}}y_z\cdot (\tilde{u}_*^{(\tilde{\lambda}, \tilde{\alpha})}-W)\big)$. The contribution of the second term in \eqref{eq:yzhighmodsplit} is handled by arguing as in the proof of the first part of the lemma, splitting into cases (i), (ii) and performing integrations by parts as needed to achieve convergence of the $R$-integral, and also using \eqref{eq:partialtaupartialtildetau}.
		\\
		For the last part of the lemma, it suffices to set $f_1: = \chi_{\xi\lesssim 1}\langle f, \chi_{R\xi\lesssim 1}\phi(R;\xi)\rangle_{L^2_{R^3\,dR}},\,f_2: = \chi_{\xi\lesssim 1}\langle f, \chi_{R\xi\gtrsim1}\phi(R;\xi)\rangle_{L^2_{R^3\,dR}}$ in the low frequency regime. For the assertion in the high-frequency regime $\xi\gtrsim 1$, this is obtained by integrating by parts with respect to $\sigma$ (cf. Lemma~\ref{lem:yzWbound3})
	\end{proof}

	As for the remaining term constituting $X^{(\tilde{\lambda})}(\tau,\xi)$, we have the following more crude estimate 
	\begin{lem}\label{lem:Xtildelambdafinaltermcrudebound} We have the estimate 
		\begin{align*}
			&\Big\|\langle\xi\partial_{\xi}\rangle\mathcal{F}\Big(Q^{(\tilde{\tau})}_{<\tau^{\frac12+}}\big((\lambda^{-2}n_*^{(\tilde{\lambda}, \tilde{\alpha})} - W^2)z\big)\Big)\Big\|_{\log\tau\cdot \tau^{-N}L^2_{d\tau}L^2_{\rho(\xi)\,d\xi}}\\
			&\lesssim \big\|z_{nres}\big\|_{S} + \big\|\langle\partial_{\tilde{\tau}}^2\rangle^{-2}\partial_{\tilde{\tau}}^2\tilde{\lambda}\big\|_{\tau^{-N}L^2_{d\tau}} + \big\|(\tilde{\kappa}_1,\kappa_2)\big\|_{\tau^{-N}L^2_{d\tau}}
		\end{align*}
		We also have the bound 
		\begin{align*}
			&\Big\|\mathcal{F}\Big(Q^{(\tilde{\tau})}_{<\tau^{\frac12+}}\big((\lambda^{-2}n_*^{(\tilde{\lambda}, \tilde{\alpha})} - W^2)z\big)\Big)(\tau,0)\Big\|_{\tau^{-N-}L^2_{d\tau}}\\
			&\lesssim \big\|z_{nres}\big\|_{S} + \big\|\langle\partial_{\tilde{\tau}}^2\rangle^{-2}\partial_{\tilde{\tau}}^2\tilde{\lambda}\big\|_{\tau^{-N}L^2_{d\tau}} + \big\|(\tilde{\kappa}_1,\kappa_2)\big\|_{\tau^{-N}L^2_{d\tau}}
		\end{align*}
		The same bounds apply without the temporal frequency cutoffs. 
	\end{lem}
	\begin{proof} Due to the $L^2$-boundedness of the transference operator it suffices to bound $\big\|\langle R\partial_R\rangle \Big(Q^{(\tilde{\tau})}_{<\tau^{\frac12+}}\big((\lambda^{-2}n_*^{(\tilde{\lambda}, \tilde{\alpha})} - W^2)z\big)\Big)\Big\|_{\log\tau\cdot \tau^{-N}L^2_{d\tau}L^2_{R^3\,dR}}$. 
		We decompose $z$ according to \eqref{eq:zedcomprefined}. Taking advantage of Lemma~\ref{lem:approxsolasymptotics3}, we easily bound the contribution of the {\it{resonant part}} of $z$ by 
		\[
		\lesssim \big\|\langle\partial_{\tilde{\tau}}^2\rangle^{-2}\partial_{\tilde{\tau}}^2\tilde{\lambda}\big\|_{\tau^{-N}L^2_{d\tau}} + \big\|(\tilde{\kappa}_1,\kappa_2)\big\|_{\tau^{-N}L^2_{d\tau}}.
		\]
		In order to bound the contribution of the {\it{non-resonant part}} of $z$, we also observe the inequality $\big\|\partial_R z_{nres}\big\|_{L^4_{R^3\,dR}}\lesssim \big\|z_{nres}\big\|_{S}$, which furnishes the required bound by again using Lemma~\ref{lem:approxsolasymptotics3}.
		The final bound of the lemma is proved by taking advantage of Lemma~\ref{lem:pseudotransferenceoperator2}, and reiteration of the equation for $z$ by using \eqref{eq:formalexpansion}, the observation that $\lambda^{-2}\sim \tau^{-1-\frac{1}{2\nu}}$, as well as Remark~\ref{rem:concatenation2} and Lemma~\ref{lem:Xtildelambdaperturbterms}, the first bound in this lemma, as well as Lemma~\ref{lem:basicboundsfore_1modandnonlinearterms}  and Lemma~\ref{lem:basicboundsfore_1modalphaterm} to bound the source terms in $E$. 
	\end{proof}
	The preceding lemma can be refined considerably if we localize the source term a bit more. We shall also need a version without temporal Fourier localization: 
	\begin{lem}\label{lem:Xtildelambdafinaltermcrudeboundrefined} We have the estimate 
		\begin{align*}
			&\Big\|\langle\xi\partial_{\xi}\rangle^{2+}\mathcal{F}\Big(\chi_{R\lesssim \tau^{\frac12-\frac{1}{4\nu}}}\big((\lambda^{-2}n_*^{(\tilde{\lambda}, \tilde{\alpha})} - W^2)z\big)\Big)(\tau,\cdot)\Big\|_{\tau^{-N-1+O(\frac{1}{\nu})}L^2_{d\tau}L^2_{\rho\,d\xi}\cap L^{\infty}_{d\xi}}\\
			&\lesssim \big\|z_{nres}\big\|_{S} + \big\|\langle\partial_{\tilde{\tau}}^2\rangle^{-2}\partial_{\tilde{\tau}}^2\tilde{\lambda}\big\|_{\tau^{-N}L^2_{d\tau}} + \big\|(\tilde{\kappa}_1,\kappa_2)\big\|_{\tau^{-N}L^2_{d\tau}}\\
			&\Big\|\langle\xi\partial_{\xi}\rangle^{1+}\partial_{\tau}\mathcal{F}\Big(\chi_{R\lesssim \tau^{\frac12-\frac{1}{4\nu}}}\big((\lambda^{-2}n_*^{(\tilde{\lambda}, \tilde{\alpha})} - W^2)z\big)\Big)(\tau,\cdot)\Big\|_{\tau^{-N-1+O(\frac{1}{\nu})}L^2_{d\tau}L^2_{\rho\,d\xi}\cap L^{\infty}_{d\xi}}\\
			&\lesssim \big\|z_{nres}\big\|_{S} + \big\|\langle\partial_{\tilde{\tau}}^2\rangle^{-2}\partial_{\tilde{\tau}}^2\tilde{\lambda}\big\|_{\tau^{-N}L^2_{d\tau}} + \big\|(\tilde{\kappa}_1,\kappa_2)\big\|_{\tau^{-N}L^2_{d\tau}}
		\end{align*}
	\end{lem}
	\begin{proof} This follows from Lemma~\ref{lem:approxsolasymptotics3} as well as the definition of $\|\cdot\|_{S}$ and decomposition~\ref{eq:zedcomprefined}.
		
	\end{proof}
	
	\begin{lem}\label{lem:basicboundsfore_1modandnonlinearterms} The third and fourth terms on the right hand side of the first equality in \eqref{eq:zeqn2} satisfy the following estimates:
		\begin{align*}
			&\Big\|\lambda^{-2}yz\Big\|_{\tau^{-2N+2}L^2_{d\tau}L^2_{R^3\,dR}}\\&\hspace{2cm}\lesssim  \big(\big\|z_{nres}\big\|_{S} + \big\|\langle\partial_{\tilde{\tau}}^2\rangle^{-2}\tilde{\lambda}_{\tilde{\tau}\tilde{\tau}}\big\|_{\tau^{-N}L^2_{d\tau}} + \big\|(\tilde{\kappa}_1,\kappa_2)\big\|_{\tau^{-N}L^2_{d\tau}}
			\big)\cdot \big\|y\big\|_{Y},\\
			&\Big\|\langle\xi\partial_{\xi}\rangle\mathcal{F}(e_1^{\text{mod}} - e_1^{\text{mod}, \tilde{\alpha}})\Big\|_{\log^{-1}(\tau)\cdot\tau^{-N}L^2_{d\tau}L^2_{\rho_{d\xi}}}\lesssim  \big\|\langle\partial_{\tilde{\tau}}^2\rangle^{-2}\tilde{\lambda}_{\tilde{\tau}\tilde{\tau}}\big\|_{\tau^{-N}L^2_{d\tau}} + \big\|\tilde{\alpha}_{\tau}\big\|_{\log^{-1}(\tau)\cdot\tau^{-N}L^2_{d\tau}}.
		\end{align*}
		where we recall that $e_1^{\text{mod},\tilde{\alpha}}$ denote the terms in $e_1^{\text{mod}}$ depending on $\tilde{\alpha}$. One may also include an operator $\langle\nabla^4\rangle$ in front of $\lambda^{-2}yz$. 
	\end{lem}
	\begin{proof} For the first inequality, we use the fixed time estimate 
		\begin{align*}
			\big\|\lambda^{-2}yz\big\|_{L^2_{R^3\,dR}}\leq \big\|\langle R\rangle^{\delta_0}\lambda^{-2}y\big\|_{L^2_{R^3\,dR}}\cdot \big\|\langle R\rangle^{-\delta_0}z\big\|_{L^\infty_{dR}}, 
		\end{align*}
		and as a consequence 
		\begin{align*}
			\Big\|\lambda^{-2}yz\Big\|_{\tau^{-2N+2}L^2_{d\tau}L^2_{R^3\,dR}}&\leq \Big\|\langle R\rangle^{\delta_0}\lambda^{-2}y\big\|_{\tau^{-N+2}L^2_{d\tau}L^{2+}_{R^3\,dR}}\cdot \Big\|\langle R\rangle^{-\delta_0}z\big\|_{\tau^{-N}L^2_{d\tau}L^M_{dR}}\\
			&\leq \big(\big\|z_{nres}\big\|_{S} + \big\|\langle\partial_{\tilde{\tau}}^2\rangle^{-2}\tilde{\lambda}_{\tilde{\tau}\tilde{\tau}}\big\|_{\tau^{-N}L^2_{d\tau}} + \big\|(\tilde{\kappa}_1,\kappa_2)\big\|_{\tau^{-N}L^2_{d\tau}}\big) \cdot \big\|y\big\|_{Y}. 
		\end{align*}
		If an operator $\nabla^4$ hits $z$, then we place this factor into $\tau^{-N}L^2_{d\tau}(L^{2+}_{R^3\,dR} + \langle R\rangle^{\delta_0}L^\infty_{R^3\,dR})$ instead, while we place $\lambda^{-2}y$ into $\tau^{-N+1}L^2_{d\tau}(\langle R\rangle^{-\delta_0}L^{2+}_{R^3\,dR}\cap L^M_{R^3\,dR})$, and similarly for $\lambda^{-2}\nabla^k y \nabla^lz$. 
		For the second inequality of the lemma, we consider the various terms forming $e_1^{\text{mod}}$: note that 
		\begin{align*}
			\Big\|\langle\xi\partial_{\xi}\rangle\mathcal{F}\Big(\lambda^{-3} \big(i\partial_t + \triangle\big)( \chi_3)\cdot \big(\psi_*^{(\tilde{\lambda})} - \psi_*\big)\Big)(\tau,\xi)\Big\|_{\tau^{-N-}L^2_{d\tau}L^2_{\rho(\xi)\,d\xi}}\lesssim \big\|\frac{\tilde{\lambda}}{\tau^{1-}}\big\|_{\tau^{-N-}L^2_{d\tau}}
		\end{align*}
		and we further have 
		\begin{align*}
			\big\|\frac{\tilde{\lambda}}{\tau^{1-}}\big\|_{\tau^{-N-}L^2_{d\tau}}\lesssim \big\|\langle\partial_{\tilde{\tau}}^2\rangle^{-2}\tilde{\lambda}_{\tilde{\tau}\tilde{\tau}}\big\|_{\tau^{-N}L^2_{d\tau}}
		\end{align*}
		since $\big\|\frac{\tilde{\lambda}}{\tau^{1-}}\big\|_{\tau^{-N-}L^2_{d\tau}}\lesssim \big\|\frac{\tilde{\lambda}}{\tilde{\tau}^2}\big\|_{\tau^{-N}L^2_{d\tau}}$, and the latter expression is bounded by 
		\[
		\big\|\langle\partial_{\tilde{\tau}}^2\rangle^{-2}\tilde{\lambda}_{\tilde{\tau}\tilde{\tau}}\big\|_{\tau^{-N}L^2_{d\tau}}.
		\]
		The second and sixth terms in \eqref{eq:E1mod} are handled similarly. 
	\end{proof}
	
	The preceding lemma needs to be complemented with the following one:
	\begin{lem}\label{lem:basicboundsfore_1modalphaterm} Letting $e_1^{\text{mod}, \tilde{\alpha}}$ denote the terms in $e_1^{\text{mod}}$ depending on $\tilde{\alpha}$, we have the bound
		\begin{align*}
			&\Big\|\int_0^\infty \xi^2\cdot S\big([\prod_{l=1}^j\mathcal{K}_l\circ S](G)\big)\rho(\xi)\,d\xi\Big\|_{\tau^{-N+}L^2_{d\tau}}\\&\hspace{5cm}\lesssim (\sqrt{N})^{-j}\cdot \big\|\tilde{\alpha}_{\tau}\big\|_{\tau^{-N}L^2_{d\tau}},\, G = \mathcal{F}\big(e_1^{\text{mod}, \tilde{\alpha}}\big), 
		\end{align*}
		where for each $l$ we have either $\mathcal{K}_l = \frac{\lambda_{\tau}}{\lambda}\cdot\mathcal{K}$ or $\mathcal{K}_l =\mathcal{K}_*$. Furthermore we can improve the inequality by replacing $\tau^{-N+}L^2_{d\tau}$ by $\tau^{-N-}L^2_{d\tau}$ provided there is at least one operator $\mathcal{K}_*$ present, the latter as defined just before Lemma~\ref{lem:concatenation2} .
	\end{lem}
	\begin{proof} One uses integration by parts with respect to the time variable in the right-most Schr\"odinger propagator $S$ to force the presence of either a factor $\frac{\tilde{\alpha}}{\tau}$ or $\tilde{\alpha}_{\tau}$, where one uses the fine structure of the terms in $e_1^{\text{mod}, \tilde{\alpha}}$ displayed in the completion of the proof of Lemma~\ref{lem:e1modgooderrors}. Then one repeats the proof of Lemma~\ref{lem:concatenation1}.
		
	\end{proof}
	
	In order to control the temporal high-frequency terms contributing to the evolution of $\tilde{\kappa}_1, \kappa_2$, the following lemmas shall be useful:
	\begin{lem}\label{lem:hightemprfreqyzdotW1} Recalling $y_z$ as defined in \eqref{eq:yzdfn}, we have under the assumption \eqref{eq:zedcomprefined} 
		\begin{align*}
			&\Big\|\langle\xi\partial_{\xi}\rangle^{1+\delta_0}\mathcal{F}\Big(Q^{(\tilde{\tau})}_{>\tau^{\frac12+}}y_z\cdot W\Big)\Big\|_{\tau^{-N-1-}L^2_{d\tau}L^2_{\rho(\xi)\,d\xi}}\\&\lesssim 
			\big\|z_{nres}\big\|_{S} + \big\|\langle\partial_{\tilde{\tau}}^2\rangle^{-2}\tilde{\lambda}_{\tilde{\tau}\tilde{\tau}}\big\|_{\tau^{-N}L^2_{d\tau}} + \big\|(\tilde{\kappa}_1,\kappa_2)\big\|_{\tau^{-N}L^2_{d\tau}}
		\end{align*}
	\end{lem} 
	\begin{proof} Observe that we have the following inequality:
		\begin{align*}
			\big\|Q^{(\tilde{\tau})}_{>\tau^{\frac12+}}f(\tau,\xi)\big\|_{\tau^{-N-1-}L^2_{d\tau}L^2_{\rho(\xi)\,d\xi}}\lesssim \big\|\partial_{\tilde{\tau}}^2f\big\|_{\tau^{-N}L^2_{d\tau}L^2_{\rho(\xi)\,d\xi}}
		\end{align*}
		Applying $\partial_{\tilde{\tau}}^2$ to $y_z$, as expressed by the Duhamel parametrix \eqref{eq:wavepropagator}, \eqref{eq:nflatfourierrepresent}, one either gains a factor $\xi^2$ or replaces $Q^{(\tilde{\tau})}_{>\tau^{\frac12+}}y_z\cdot W$ by 
		\[
		2Q^{(\tilde{\tau})}_{>\tau^{\frac12+}}\triangle \lambda^2\Re(W\bar{z})\cdot W. 
		\]
		The desired bound then follows from combining the inequalities
		\begin{align*}
			&\big\|\nabla\big(\langle R\rangle\triangle (W\bar{z})\big)\big\|_{\tau^{-N}L^2_{d\tau}L^2_{R^3\,dR}}\lesssim \big\|z_{nres}\big\|_{S} + \big\|\langle\partial_{\tilde{\tau}}^2\rangle^{-2}\tilde{\lambda}_{\tilde{\tau}\tilde{\tau}}\big\|_{\tau^{-N}L^2_{d\tau}} + \big\|(\tilde{\kappa}_1,\kappa_2)\big\|_{\tau^{-N}L^2_{d\tau}}\\
			&\big\|\triangle (W\bar{z})\big\|_{\tau^{-N}L^2_{d\tau}L^2_{R^3\,dR}}\lesssim \big\|z_{nres}\big\|_{S} + \big\|\langle\partial_{\tilde{\tau}}^2\rangle^{-2}\tilde{\lambda}_{\tilde{\tau}\tilde{\tau}}\big\|_{\tau^{-N}L^2_{d\tau}} + \big\|(\tilde{\kappa}_1,\kappa_2)\big\|_{\tau^{-N}L^2_{d\tau}}\\
		\end{align*}
		combined with Lemma~\ref{lem:wavebasicinhomstructure}.  
	\end{proof}
	
	The following lemma deals with the component $y_{\tilde{\lambda}}^{mod}$ of $y$, as with the source term $E_2^{mod}$ (recall \eqref{eq:E2mod}, \eqref{eq:ylamndatildemod}): 
	\begin{lem}\label{lem:ytildelambdaE2} We have the estimate 
		\begin{align*}
			\big\|\lambda^{-4}\langle\nabla_R\rangle^{10}E_2^{mod}\big\|_{\tau^{-N}L^2_{d\tau}L^{1+}_{R^3\,dR}\cap L^2_{R^3\,dR}}\lesssim \big\|\tilde{\lambda}_{\tilde{\tau}\tilde{\tau}}\big\|_{\tau^{-N}L^2_{d\tau}} + \big\|\frac{\tilde{\lambda}_{\tilde{\tau}}}{\tilde{\tau}}\big\|_{\tau^{-N}L^2_{d\tau}}
		\end{align*}
		where $1+ = 1+\frac{10}{\nu}$. 
		Furthermore, we can bound 
		\begin{align*}
			\big\|y_{\tilde{\lambda}}^{mod}\big\|_{\tau^{-N}L^2_{d\tau}\langle R\rangle^{1+\delta_0}L^2_{R^3\,dR}} + \big\|y_{\tilde{\lambda}}^{mod}\big\|_{\tau^{-(N-1)}L^2_{d\tau}L^2_{R^3\,dR}}\lesssim \big\|\langle\partial_{\tilde{\tau}}^2\rangle^{-1}\tilde{\lambda}_{\tilde{\tau}\tilde{\tau}}\big\|_{\tau^{-N}L^2_{d\tau}} 
		\end{align*}
		Finally we have the bound 
		\begin{align*}
			&\big\|\lambda^{-4}\langle\nabla_R\rangle^{10}Q^{(\tilde{\tau})}_{>\sqrt{\gamma}^{-1}}P_{<\gamma^{-\frac14}}\partial_{\tilde{\tau}}^{-2}\big(E_2^{mod} - 2\tilde{\lambda}_{tt}\lambda^2\Lambda W\cdot W\big)\big\|_{\tau^{-N}L^2_{d\tau}L^{1+}_{R^3\,dR}\cap L^2_{R^3\,dR}}\\&\ll_{\tau_*} \big\|\langle\partial_{\tilde{\tau}}^2\rangle^{-1}\tilde{\lambda}_{\tilde{\tau}\tilde{\tau}}\big\|_{\tau^{-N}L^2_{d\tau}} 
		\end{align*}
	\end{lem}
	\begin{proof} {\it{First estimate}}. Taking advantage of the construction \eqref{eq:psinmodulated} as well as Lemma~\ref{lem:approxsolasymptotics3}, we can estimate the contribution of the first two terms in \eqref{eq:E2mod} directly by 
		\begin{align*}
			\big\|\lambda^{-4}\big(-\tilde{\lambda}_{tt}\cdot \partial_{\tilde{\lambda}}n_*^{(\tilde{\lambda})} - \tilde{\lambda}_t\cdot \partial_t\big(\partial_{\tilde{\lambda}}n_*^{(\tilde{\lambda})}\big)\big)\big\|_{\tau^{-N}L^2_{d\tau}L^{1+\frac{10}{\nu}}_{R^3\,dR}\cap L^2_{R^3\,dR}}\lesssim \big\|\tilde{\lambda}_{\tilde{\tau}\tilde{\tau}}\big\|_{\tau^{-N}L^2_{d\tau}} + \big\|\frac{\tilde{\lambda}_{\tilde{\tau}}}{\tilde{\tau}}\big\|_{\tau^{-N}L^2_{d\tau}}.
		\end{align*}
		The remaining terms forming $E_2^{mod}$ are handled similarly. For example for the fifth term, we have 
		\begin{align*}
			\big|\lambda^{-4}\cdot 2\partial_r(\chi_3)\cdot \partial_r\big(n_*^{(\tilde{\lambda})} - n_*\big)\big|\lesssim \tilde{\chi}_3\cdot R^{-2}\cdot \big|\tilde{\lambda} - 1\big|\cdot R^{-4+\frac{2}{\nu}}
		\end{align*}
		where the smooth cutoff $\tilde{\chi}_3$ localizes to the region $R\sim \tau^{\frac12-}$, and hence we have the bound 
		\begin{align*}
			\Big|\tilde{\chi}_3\cdot R^{-2}\cdot \big|\tilde{\lambda} - 1\big|\Big|\lesssim \tau^{-1+}\cdot \big|\tilde{\lambda} - 1\big|,
		\end{align*}
		whence 
		\begin{align*}
			\Big\|\tilde{\chi}_3\cdot R^{-2}\cdot \big|\tilde{\lambda} - 1\big|\Big\|_{\tau^{-N}L^2_{d\tau}}\lesssim \big\|\tilde{\lambda}_{\tilde{\tau}\tilde{\tau}}\big\|_{\tau^{-N}L^2_{d\tau}}. 
		\end{align*}
		This together with $\big\| \langle R\rangle^{-4+\frac{2}{\nu}}\big\|_{L^{1+\frac{10}{\nu}}_{R^3\,dR}}\lesssim 1$ implies the desired bound for the fifth term. The remaining terms in $E_2^{mod}$ are handled similarly, as is the estimate for the norm $L^2_{R^3\,dR}$.
		\\
		{\it{Second estimate}}.  This is essentially a direct consequence of the preceding bound and Lemma~\ref{lem:wavebasicinhom}, except that we need to replace $\big\|\tilde{\lambda}_{\tilde{\tau}\tilde{\tau}}\big\|_{\tau^{-N}L^2_{d\tau}}$ by $\big\|\langle\partial_{\tilde{\tau}}^2\rangle^{-1}\tilde{\lambda}_{\tilde{\tau}\tilde{\tau}}\big\|_{\tau^{-N}L^2_{d\tau}} $. This follows from the fact that all terms forming $E_2^{mod}$ are localized with respect to spatial frequency up to rapidly decaying tails, and hence $\Box^{-1}$ acts like $\langle\partial_{\tilde{\tau}}^2\rangle^{-1}$ in the high frequency regime. 
		\\
		The final estimate is also a consequence of this observation and the fact that we have removed the second derivatives of $\tilde{\lambda}$. 
	\end{proof}
	
	We shall also require a more detailed structural result in the high temporal frequency regime, whose proof follows from the frequency localization of $y_{\tilde{\lambda}}$ around spatial frequency $\sim 1$ up to rapidly decaying tails: 
	\begin{lem}\label{lem:ytildelambdamodhightempfreq} Let $0<\gamma\ll 1$. We can write 
		\begin{align*}
			Q_{>\gamma^{-1}}^{(\tilde{\tau})}\big(\lambda^{-2}y_{\tilde{\lambda}}\cdot W\big) = 2Q_{>\gamma^{-1}}^{(\tilde{\tau})}(\tilde{\lambda})\Lambda W\cdot W^2 +  H, 
		\end{align*}
		where the distorted Fourier coefficient $h: = \chi_{\xi\lesssim 1}\langle H, \phi(R;\xi)\rangle $ satisfies the estimate 
		\begin{align*}
			\big\|\langle\xi\partial_{\xi}\rangle^{2+}h\big\|_{\tau^{-N}L^2_{d\tau}L^p_{\rho(\xi)\,d\xi}}\ll_{\tau_*}\big\|\langle\partial_{\tilde{\tau}}\rangle^{-2}\tilde{\lambda}_{\tilde{\tau}\tilde{\tau}}\big\|_{\tau^{-N}L^2_{d\tau}},\,2\leq p\leq \infty.
		\end{align*}
		On the other hand, the high frequency part $g: = \chi_{\xi\gtrsim 1}\langle H, \phi(R;\xi)\rangle $ satisfies the final conclusion of Lemma~\ref{lem:yzWbound3}.\\
		Furthermore, we have the estimate (here $P_{<\sqrt{\gamma}^{-1}}$ denotes a standard Littlewood Paley multiplier)
		\begin{align*}
			\big\|\langle\xi\partial_{\xi}\rangle^{2+}\mathcal{F}\big(P_{<\sqrt{\gamma}^{-1}}Q_{>\gamma^{-1}}^{(\tilde{\tau})}(\lambda^{-2}y_{\tilde{\lambda}}\cdot W)\big)\big\|_{\tau^{-N}L^2_{d\tau}L^2_{\rho(\xi)\,d\xi}\cap L^\infty_{\rho(\xi)\,d\xi}}\lesssim \big\|\langle\partial_{\tilde{\tau}}\rangle^{-2}\tilde{\lambda}_{\tilde{\tau}\tilde{\tau}}\big\|_{\tau^{-N}L^2_{d\tau}}.
		\end{align*}
		Finally, we have the estimate (without added temporal frequency localization)
		\begin{align*}
			&\Big\|\langle\xi\partial_{\xi}\rangle^{1+\delta_0}\mathcal{F}\Big(\lambda^{-2}y_{\tilde{\lambda}}\cdot W\Big)\Big\|_{\tau^{-N+}L^2_{d\tau}(L^2_{\rho(\xi)\,d\xi}\cap L^\infty_{d\xi})}\lesssim \big\|\langle\partial_{\tilde{\tau}}\rangle^{-2}\tilde{\lambda}_{\tilde{\tau}\tilde{\tau}}\big\|_{\tau^{-N}L^2_{d\tau}},\\
			&\Big\|\langle\xi\partial_{\xi}\rangle^{1+\delta_0}\mathcal{F}\Big(\partial_{\tau}\big(\lambda^{-2}y_{Q^{(\tilde{\tau})}_{<\tau^{\frac12+}}\tilde{\lambda}}\cdot W\big)\Big)\Big\|_{\tau^{-N}L^2_{d\tau}(L^2_{\rho(\xi)\,d\xi}\cap L^\infty_{d\xi})}\ll_{\tau_*}\big\|\langle\partial_{\tilde{\tau}}\rangle^{-2}\tilde{\lambda}_{\tilde{\tau}\tilde{\tau}}\big\|_{\tau^{-N}L^2_{d\tau}}.
		\end{align*}
		The same bounds apply if we include an operator $\langle\nabla^4\rangle$ in front of $\lambda^{-2}y_{\tilde{\lambda}}\cdot W$. 
		We also have a small spatial frequency gain: 
		\begin{align*}
			&\Big\|\chi_{\xi<\tau^{-\delta_1}}\langle\xi\partial_{\xi}\rangle^{1+\delta_0}\mathcal{F}\Big(\lambda^{-2}y_{\tilde{\lambda}}\cdot W\Big)\Big\|_{\tau^{-N-}L^2_{d\tau}(L^2_{\rho(\xi)\,d\xi}\cap L^M_{d\xi})}\lesssim \big\|\langle\partial_{\tilde{\tau}}\rangle^{-2}\tilde{\lambda}_{\tilde{\tau}\tilde{\tau}}\big\|_{\tau^{-N}L^2_{d\tau}},\,M<\infty
		\end{align*}
	\end{lem}
	The second bound follows as usual by using that 
	\[
	P_{<\sqrt{\gamma}^{-1}}Q_{>\gamma^{-1}}^{(\tilde{\tau})}\Box^{-1}\sim P_{<\sqrt{\gamma}^{-1}}Q_{>\gamma^{-1}}^{(\tilde{\tau})}\partial_{\tilde{\tau}}^{-2},
	\]
	recalling \eqref{eq:ytildelamba}, and using \eqref{eq: transference1} as well as Lemma~\ref{lem:KboundsKST} to deal with the effect of $\langle \xi\partial_{\xi}\rangle^{2+}$. The third and fourth bound are similar to Corollary~\ref{cor:yzW}. The last bound follows by first estimating the Fourier coefficient in $L^\infty_{d\xi}$ and using Holder's inequality.

	\subsection{Basic facts on the approximate solution}\label{subsec:approxsolnbasics}
	
	We summarise here some asymptotic bounds for corrections used to build the approximate solution $(\psi_*, n_*)$. These follow from Lemma 2.29, Corollary 3.33 and Corollary 4.3 in \cite{KrSchm}. 
	\begin{lem}\label{lem:approxsolasymptotics1} Interpreting the component $\psi_*$ of the approximate solution as function of $(\tau, R)$, we can write (with $\lambda = \lambda(\tau)\sim \tau^{\frac12+\frac{1}{4\nu}}$)
		\begin{align*}
			\psi_*(\tau, R) = e^{i\alpha(t)}\cdot \lambda\cdot u_*(\tau, R), \,\alpha(t) = \alpha_0\log t,
		\end{align*}
		where we can set 
		\[
		u_*(\tau, R) = W(R) + g(\tau, R), 
		\]
		and the correction function $g$ satisfies the bounds 
		\begin{align*}
			\Big\|\partial_{\tau}^{l_1}\partial_R^{l_2}g(\tau, R)\Big\|_{L^\infty_{dR}}\lesssim_{l_1, l_2}\lambda^{-1}\cdot (\tau^{\frac12-\frac{1}{4\nu}})^{-l_2}\cdot (\tau^{1-\frac{1}{2\nu}})^{-l_1},\,l_1+l_2\ll \nu. 
		\end{align*}
		Furthermore, we can decompose $g = g_1 + g_2$, where we can refine the preceding estimate as follows:
		\begin{align*}
			&\Big\|\partial_{\tau}^{l_1}\partial_R^{l_2}g_1(\tau, R)\Big\|_{L^\infty_{dR}}\lesssim_{l_1, l_2}\lambda^{-1}\cdot (\tau^{\frac12-\frac{1}{4\nu}})^{-l_2}\cdot (\tau^{1-\frac{1}{2\nu}})^{-l_1}\chi_{R\lesssim \lambda},\\
			& \Big\|\partial_{\tau}^{l_1}\partial_R^{l_2}g_2(\tau, R)\Big\|_{L^\infty_{dR}}\lesssim_{l_1, l_2}\lambda^{-1}\cdot(\tau^{\frac12-\frac{1}{4\nu}})^{-l_2}\cdot (\tau^{1-\frac{1}{2\nu}})^{-l_1}\cdot \langle \frac{R}{\lambda}\rangle^{-c\nu},\,c>0,\,,l_1+l_2\ll \nu.
		\end{align*}
		The function $g(\tau, R)$ satisfies the better bounds 
		\begin{align*}
			\big|g(\tau, R)\big|\lesssim \tau^{-1}\log R
		\end{align*}
		in the region $R\lesssim \tau^{\frac12-}$. More precisely, in this region we can split
		\begin{align*}
			g(\tau,R) = g_5(\tau, R) + g_6(\tau, R),
		\end{align*}
		where $g_5$ is purely imaginary and satisfies the preceding bound, while we have the better bound. 
		\[
		\big|g_6(\tau, R)\big|\lesssim \frac{\log R\cdot R^2}{\tau^2}
		\]
		Finally, we note that we may assumer $g(\tau, 0) = O(\tau^{-M})$ for $M\gg N$.  
	\end{lem}
	\begin{rem}\label{rem:approxsolasymptotics1} We recall that the construction of $g$ involves four regions, with the inner two regions contained in the inner Schr\"odinger region $r\lesssim t^{\frac12+\epsilon}$. This translates to 
		\begin{align*}
			R = \lambda(t)\cdot r = t^{-\frac12-\nu}\cdot r \lesssim t^{\epsilon - \nu}\sim \tau^{\frac12 - \frac{\epsilon}{2\nu}}. 
		\end{align*}
	\end{rem}
	We also have the following 
	\begin{lem}\label{lem:approxsolasymptotics2}
		In the context of the preceding lemma, we can decompose 
		\begin{align*}
			g = g_3 + g_4
		\end{align*}
		where we have the estimates (for suitable $\zeta>0$)
		\begin{align*}
			&\big|\partial_R^{l+1}\big(\chi_{R\gtrsim\tau^{\frac12-}}\cdot g_3(\tau, R)\big)\big|\lesssim \lambda^{-1}\cdot R^{-1}\cdot \big(\frac{R}{\tau^{\frac12}}\big)^{-\zeta}\cdot (\tau^{\frac12-\frac{1}{4\nu}})^{-l} ,\,\zeta>0,\,l\geq 0\\
			&\big|\partial_R^lg_4(\tau, R)\big|\lesssim \lambda^{-1}R^{-l},\,l\geq 0. 
		\end{align*}
		In particular, we have the bound 
		\begin{align*}
			\big\|\partial_R\big(\chi_{R\gtrsim\tau^{\frac12-}}\cdot g_3(\tau, R)\big)\big\|_{L^2_{R^3\,dR}}\lesssim \tau^{0-}. 
		\end{align*}
	\end{lem}
	\begin{proof} We only need to consider the terms in the expansion of $g$ which do not have symbol behavior with respect to $R$, which constitute $g_3$. These are the terms with an exponential term $e^{m\frac{ir^2}{4t}} = e^{imc(\nu)\cdot \frac{R^2}{4\tau}}$, recalling that $t\cdot\lambda^2\sim \tau$.
		Each such exponential term comes with at least one factor $t$ and a further factor bounded by $\big(\frac{R}{\tau^{\frac12}}\big)^{-\zeta}$ with symbol behavior, and we have 
		\begin{align*}
			t\cdot\partial_R\big(e^{imc(\nu)\cdot \frac{R^2}{4\tau}}\big)\sim \tau^{-\frac{1}{2\nu}}\cdot \frac{R}{\tau}\cdot e^{imc(\nu)\cdot \frac{R^2}{4\tau}}.
		\end{align*}
		Given that $R\lesssim \lambda\sim \tau^{\frac12+\frac{1}{4\nu}}$ on the support of these functions, we have 
		\begin{align*}
			\big|\tau^{-\frac{1}{2\nu}}\cdot \frac{R}{\tau}\cdot e^{imc(\nu)\cdot \frac{R^2}{4\tau}}\big|\lesssim \tau^{-\frac12 - \frac{1}{4\nu}}\lesssim R^{-1}, 
		\end{align*}
		as desired. 
		The second inequality is an easy consequence of the first one since $R\lesssim\lambda$ on the support. 
	\end{proof}
	
	\begin{lem}\label{lem:approxsolasymptotics3} We can write 
		\begin{align*} 
			\lambda^{-2}n_* = W^2 + h(\tau, R), 
		\end{align*}
		where we have the bounds 
		\begin{align*}
			\big|\langle R\partial_R\rangle\big(\chi_{R\lesssim\tau^{\frac12 - \frac{1}{4\nu}}}h\big)\big|\lesssim \tau^{-1+\frac{1}{2\nu}}R^{-2}\log R,\,\big|\langle R\partial_R\rangle\big(\chi_{R\gtrsim\tau^{\frac12 - \frac{1}{4\nu}}}h\big)\big|\lesssim \lambda^{-2}\sim \tau^{-1-\frac{1}{2\nu}}. 
		\end{align*}
		We also have symbol behavior with respect to $\tilde{\tau}$ provided we apply $\ll \nu$ many derivatives. 
	\end{lem}
	
	\subsection{Basic estimates on modulation errors}
	
	Here we derive some basic estimates on the terms $E_1^{\text{mod}}, E_2^{\text{mod}}$, as defined in \eqref{eq:E1mod}, \eqref{eq:E2mod}. 
	\begin{lem}\label{lem:E2modbound1} We have the estimate
		\begin{align*}
			&\big\|\lambda^{-4}\nabla^{-(2-)}Q^{(\tilde{\tau})}_{<1}E_2^{\text{mod}}\big\|_{\tau^{-N}L^2_{d\tau}L^2_{R^3\,dR}} + \big\|\lambda^{-4}\nabla^{-(1-)}\big(\langle R\rangle Q^{(\tilde{\tau})}_{<1}E_2^{\text{mod}}\big)\big\|_{\tau^{-N}L^2_{d\tau}L^2_{R^3\,dR}}\\&\lesssim \big\|\langle\partial_{\tilde{\tau}}^2\rangle^{-1}\partial_{\tilde{\tau}}^2\tilde{\lambda}\big\|_{\tau^{-N}L^2_{d\tau}}\\
		\end{align*}
		which implies the wave propagator bound 
		\begin{align*}
			\big\|\langle R\partial_R\rangle\big(\chi_{R\gtrsim\tau^{\frac12-}}\lambda^{-2}\Box^{-1}Q^{(\tilde{\tau})}_{<1}E_2^{\text{mod}}\big)\big\|_{\tau^{-N+}L^2_{d\tau}L^2_{R^3\,dR}}\lesssim \big\|\langle\partial_{\tilde{\tau}}^2\rangle^{-1}\partial_{\tilde{\tau}}^2\tilde{\lambda}\big\|_{\tau^{-N}L^2_{d\tau}}
		\end{align*}
		Furthermore we have the bound
		\begin{align*}
			\big\|\langle R\partial_R\rangle\big(\chi_{R\gtrsim\tau^{\frac12-}}\lambda^{-2}\Box^{-1}Q^{(\tilde{\tau})}_{\geq 1}E_2^{\text{mod}}\big)\big\|_{\tau^{-N+}L^2_{d\tau}L^2_{R^3\,dR}}\lesssim \big\|\langle\partial_{\tilde{\tau}}^2\rangle^{-1}\partial_{\tilde{\tau}}^2\tilde{\lambda}\big\|_{\tau^{-N}L^2_{d\tau}}\\
		\end{align*}
	\end{lem}
	\begin{proof} The first bound of the lemma is a straightforward application of Sobolev's lemma applied to the various terms forming $E_2^{\text{mod}}$ in \eqref{eq:E2mod}, and simple bounds for $n_*^{(\tilde{\lambda})}$, $\psi^{(\tilde{\lambda})}$. Thus for the leading part of $\tilde{\lambda}_{tt}\cdot\partial_{\tilde{\lambda}}n_*^{(\tilde{\lambda})}$ we obtain the bound 
		\begin{align*}
			\big\|\nabla^{-(2-)}Q^{(\tilde{\tau})}_{<1}\big(\tilde{\lambda}_{\tilde{\tau}\tilde{\tau}}\Lambda W\cdot W\big)\big\|_{\tau^{-N}L^2_{d\tau}L^2_{R^3\,dR}}&\lesssim  \big\|\langle\partial_{\tilde{\tau}}^2\rangle^{-1}\partial_{\tilde{\tau}}^2\tilde{\lambda}\big\|_{\tau^{-N}L^2_{d\tau}}\cdot \big\|\Lambda W\cdot W\big\|_{L^{1+}_{R^3\,dR}}\\
			&\lesssim  \big\|\langle\partial_{\tilde{\tau}}^2\rangle^{-1}\partial_{\tilde{\tau}}^2\tilde{\lambda}\big\|_{\tau^{-N}L^2_{d\tau}}.
		\end{align*}
		Furthermore, we have 
		\begin{align*}
			\big\|\lambda^{-4}\nabla^{-(2-)}Q^{(\tilde{\tau})}_{<1}\big(\Box \chi_3\cdot (n_*^{(\tilde{\lambda})} - n_*) \big)\big\|_{\tau^{-N}L^2_{d\tau}L^2_{R^3\,dR}}&\lesssim \big\|\tilde{\tau}^{-2}\tilde{\lambda}\big\|_{\tau^{-N}L^2_{d\tau}}\cdot\big\|\lambda^{-2}\frac{(n_*^{(\tilde{\lambda})} - n_*)}{\tilde{\lambda}}\big\|_{L^{1+}_{R^3\,dR}}\\
			&\lesssim  \big\|\langle\partial_{\tilde{\tau}}^2\rangle^{-1}\partial_{\tilde{\tau}}^2\tilde{\lambda}\big\|_{\tau^{-N}L^2_{d\tau}}.
		\end{align*}
		The estimates with $\nabla^{-(2-)}$ replaced by $\nabla^{-(1-)}\langle R\rangle$ are simple consequences of Holder's and Sobolev's inequality. 
		The contributions from the remaining terms of $E_2^{\text{mod}}$ are estimated similarly, and the second estimate of the lemma is then a straightforward application of Lemma~\ref{lem:refinedwavepropagatorwithphysicallocalization} and its proof, as well as interpolation. To obtain the final bound of the lemma, we  use the propagator in \eqref{eq:wavepropagator} and perform integration by parts twice with respect to the variable $\tilde{\sigma}$ for those terms involving $\tilde{\lambda}_{\tilde{\sigma}\tilde{\sigma}} \sim \lambda^{-2}\tilde{\lambda}_{tt}$. Then we repeat application of Lemma~\ref{lem:refinedwavepropagatorwithphysicallocalization} and its proof.
	\end{proof}
	
	Concerning the remaining modulation error $E_1^{\text{mod}}$, we have the following lemma, which is useful for the high temporal frequency regime:
	\begin{lem}\label{lem:E1hightempfreq} Recalling \eqref{eq:E1mod}, \eqref{eq:e1moddef} and letting $0<\gamma = \gamma(\tau_*)$, $\lim_{\tau_*\rightarrow\infty}\gamma(\tau_*) = 0$, we can write
		\begin{align*}
			&Q_{>\gamma^{-1}}^{(\tilde{\tau})}\big(\int_\tau^\infty S(\tau,\sigma,\xi)\cdot \mathcal{F}\big(e_1^{\text{mod}}\big)(\sigma, \frac{\lambda(\tau)}{\lambda(\sigma)}\xi)\,d\sigma\big)\\& = \lambda^2\cdot \partial_{\tau}g_1 + Q_{>\gamma^{-1}}^{(\tilde{\tau})}\big(\int_\tau^\infty S(\tau,\sigma,\xi)\cdot g_2(\sigma,\xi)\,d\sigma\big), 
		\end{align*}
		where we have the estimates ($2\leq p\leq\infty$)
		\begin{align*}
			&\big\|g_1\big\|_{\tau^{-N}L^2_{d\tau}L^p_{\rho(\xi)\,d\xi}}\ll_{\tau_*} \big\|\tilde{\alpha}_{\tau}\big\|_{\log^{-1}(\tau)\cdot\tau^{-N}L^2_{d\tau}} + \big\|\langle\partial_{\tilde{\tau}}^2\rangle^{-2}\tilde{\lambda}_{\tilde{\tau}\tilde{\tau}}\big\|_{\tau^{-N}L^2_{d\tau}},\\
			&\big\|g_2\big\|_{\tau^{-N}L^2_{d\tau}L^p_{\rho(\xi)\,d\xi}}\lesssim \big\|\tilde{\alpha}_{\tau}\big\|_{\log^{-1}(\tau)\cdot\tau^{-N}L^2_{d\tau}} + \big\|\langle\partial_{\tilde{\tau}}^2\rangle^{-2}\tilde{\lambda}_{\tilde{\tau}\tilde{\tau}}\big\|_{\tau^{-N}L^2_{d\tau}}.
		\end{align*}
	\end{lem} 
	\begin{proof} It is straightforward to check that the sum of the first, second, sixth, seventh and eight terms in \eqref{eq:E1mod} can be incorporated into $g_2$, leading in fact to a term in the better spaces $\tau^{-N-}L^2_{d\tau}L^p_{\rho(\xi)\,d\xi}$, $2\leq p\leq \infty$. For the contribution of the sum of the third and fourth terms in \eqref{eq:E1mod}, we use the identity
		\begin{align*}
			i\triangle(\chi_1)\tilde{\alpha}\cdot \psi_*^{(\tilde{\lambda})} + 2i\partial_r(\chi_1)\cdot\tilde{\alpha}\cdot \partial_r \psi_*^{(\tilde{\lambda})} = -i\tilde{\alpha}\cdot\mathcal{L}\big(\chi_1\psi_*^{(\tilde{\lambda})}\big) + i\chi_1\tilde{\alpha}\mathcal{L}\big(\psi_*^{(\tilde{\lambda})}-W\big).
		\end{align*}
		Using the asymptotics in Lemma~\ref{lem:approxsolasymptotics1}, se see that (for a suitable constant $c_1$)
		\[
		i\chi_1\tilde{\alpha}\mathcal{L}\big(\psi_*^{(\tilde{\lambda})}-W\big) = i\chi_1\frac{\tilde{\alpha}}{\tau}\cdot \big(c_1W(R) + O(\frac{\log R}{\langle R\rangle^4})\big),
		\]
		and the asymptotics of the Fourier basis $\phi(R;\xi)$ in the subsection~\ref{subsec:basicfourier} show that the Fourier coefficients of these expressions are indeed in $\tau^{-N}L^2_{d\tau}L^p_{\rho(\xi)\,d\xi},\,2\leq p\leq \infty$. As for the first term on the right, the operator $-\mathcal{L}$ results in a factor $\xi^2$ for the Fourier coefficient, and performing integration by parts with respect to $\sigma$ for this contribution, we arrive at the boundary term 
		\begin{equation}\label{eq:bdrytermlargemodulation}
			Q_{>\gamma^{-1}}^{(\tilde{\tau})}\big(-\tilde{\alpha}\cdot \mathcal{F}\big(\chi_1 W\big)(\tau,\xi)\big),
		\end{equation}
		and up to a term cancelling the contribution of the fifth term in \eqref{eq:E1mod} to the $\sigma$-integral, we arrive at the schematically written terms 
		\begin{align*}
			Q_{>\gamma^{-1}}^{(\tilde{\tau})}\big(\int_\tau^\infty S(\tau,\sigma,\xi)\cdot h(\sigma, \frac{\lambda(\tau)}{\lambda(\sigma)}\xi)\,d\sigma, \,h\in\{ \frac{\tilde{\alpha}}{\sigma}\mathcal{F}\big(\tilde{\chi}_1W\big),  \frac{\tilde{\alpha}}{\sigma}(\xi\partial_{\xi})\mathcal{F}\big(\chi_1W\big)\}.
		\end{align*}
		The functions $h$ here are of type $g_2$, while for the boundary term \eqref{eq:bdrytermlargemodulation}, it suffices to use the identity $Q_{>\gamma^{-1}}^{(\tilde{\tau})} = \frac{\partial\tau}{\partial\tilde{\tau}}\cdot\partial_{\tau}\big(\partial_{\tilde{\tau}}^{-1}Q_{>\gamma^{-1}}^{(\tilde{\tau})}\big)$  twice, together with the observation that $\frac{\partial\tau}{\partial\tilde{\tau}}\sim \lambda$, to identify it with $\lambda^2\partial_{\tau}g_1$. 
	\end{proof}
	
	\subsection{Numerical non-degeneracy assumptions}\label{subsec:numerics}
	
	\subsubsection{Absence of resonance/root mode at the edge of spectrum}
	
	We make
	\begin{itemize}
		\item {\bf{(S1)}} The operator $\mathcal{L}_*$ from subsection~\ref{subsec:Lstarbasics} has neither a root mode nor a resonance at the origin $\xi = 0$. 
	\end{itemize}

	\subsubsection{Assumptions concerning the function $v_{\hat{\tau}}$ given by \eqref{eq:vtauhatiteration}.}
	
	We make 
	\begin{itemize}
		\item {\bf{(A1)}} The function $\mathcal{F}_{|\tau|}(K_{main}v_{\hat{\tau}})$ does not vanish identically. If 
		\[
		\mathcal{F}(K_{main}v_{\hat{\tau}_1})(|\hat{\tau}_1|) = \langle K_{main}v_{\hat{\tau}_1}, \phi_{|\hat{\tau}_1|}\rangle_{L^2_{R^3\,dR}} = 0, 
		\]
		then $\langle K_{main}v_{\hat{\tau}_1}, \theta_{|\hat{\tau}_1|}\rangle_{L^2_{R^3\,dR}}\neq 0$. 
	\end{itemize}
	We observe that the function $\mathcal{F}_{|\tau|}(K_{main}v_{\hat{\tau}})$ is analytic, hence its vanishing set will be discrete. Since we only use this assumptions for $|\hat{\tau}|\in [0, \hat{\tau}_*]$ for some absolute constant $\hat{\tau}_*$, the preceding in effect amounts to a condition for finitely many points.  
	
	\subsubsection{Assumptions pertaining to one dimensional projection operators}
	\begin{itemize}
		\item {\bf{(B1)}} We have the relation 
		\begin{align*}
			\int_0^\infty (\frac{W}{2} + \frac{\Lambda W}{16})\cdot \triangle^{-1}\big(\Lambda W\cdot W\big)\cdot W R^3\,dR\neq 0. 
		\end{align*}
		\item {\bf{(B2)}} We have the relation $\frac{2}{\alpha_{**}}\cdot \int_0^\infty \psi W\cdot \triangle(W^2)R^3\,dR\neq 1$, where $\psi = 2\Lambda W + 16W$ and $\alpha_{**}$ is given by \eqref{eq:alphastarstar} with $c_2 = \frac12$, \eqref{eq:alphastar}, and \eqref{eq:cstardef}.
		\item {\bf{(B3)}} The range of the complex valued function $ \beta_*(\hat{\tilde{\tau}})\cdot\langle \tilde{g}(\hat{\tilde{\tau}}, \cdot), W^2(\cdot)\rangle$ omits the value $1$ for $\hat{\tilde{\tau}}\in \R_+$. 
	\end{itemize}
	
	\subsubsection{Assumptions pertaining to the modulation parameter control}
	\begin{itemize}
		\item {\bf{(C1)}} We have the non-vanishing relation $\int_0^\infty \triangle^{-1}(\Lambda W\cdot W)\cdot W^2R^3\,dR\neq 0$. 
		\item {\bf{(C2)}} We have the non-vanishing relation $\Re(\beta_2(\hat{\tau}_*))\neq 0$ where $\beta_2$ is as in the proof of Lemma~\ref{lem:Fouriertransform2} and $\hat{\tau}_*$ is as in Lemma~\ref{lem:FourierNV2}.
		\item {\bf{(C3)}} We have relation \eqref{eq:alphastarstar} where $c_2 = \frac12$ and $\alpha_*$ is given by \eqref{eq:alphastar} with $c_*$ given by \eqref{eq:cstardef}.
	\end{itemize}
	
	We note that ${\bf{(C1)}}$ is a direct consequence of a scaling argument: We observe that (with $\big|\nabla\big| = \sqrt{-\triangle}$)
	\begin{align*}
	\int_0^\infty \triangle^{-1}(\Lambda W\cdot W)\cdot W^2R^3\,dR = -\int_0^\infty \big|\nabla\big|^{-1}(\Lambda W\cdot W)\cdot \big|\nabla\big|^{-1}\big(W^2\big)R^3\,dR,
	\end{align*}
	and further 
	\begin{align*}
	&4\cdot \int_0^\infty \big|\nabla\big|^{-1}(\Lambda W\cdot W)\cdot \big|\nabla\big|^{-1}\big(W^2\big)R^3\,dR\\
	& = \partial_{\tilde{\lambda}}\Big(\int_0^\infty \big|\nabla\big|^{-1}(W_{\tilde{\lambda}}^2)\cdot \big|\nabla\big|^{-1}\big(W_{\tilde{\lambda}}^2\big)R^3\,dR\Big)\big|_{\tilde{\lambda} = 1}\\
	& = -2\cdot \int_0^\infty \big|\nabla\big|^{-1}(W^2)\cdot \big|\nabla\big|^{-1}\big(W^2\big)R^3\,dR\\
	&\neq 0. 
	\end{align*}
	
	The other numerical assumptions will be verified in \cite{KSchmNum}.

\end{document}